%% file: main.tex
\documentclass[12pt]{amsart}
\usepackage[pdftex,lmargin=1in,rmargin=1in,tmargin=1in,bmargin=1in]{geometry}
\usepackage{amssymb}
\usepackage{amsmath}
\usepackage{array}
\usepackage{bm}
\usepackage{bbm}
\usepackage{comment} 
\usepackage{enumitem}
\usepackage{multirow}
\usepackage{stackrel}
\usepackage{stmaryrd}
\usepackage{todonotes} 
\usepackage[cmtip,all]{xy}
\usepackage[colorlinks=true,linkcolor=blue,citecolor=blue,urlcolor=blue,backref=page]{hyperref}
\ifpdf
\usepackage{tikz}
\usetikzlibrary{matrix,arrows}

\newif\ifshow 
\showfalse 

\ifshow
  \includecomment{wrap}
\else
  \excludecomment{wrap} 
\fi

\tikzset{cdlabel/.style={above,sloped,
    execute at begin node=$\scriptstyle,execute at end node=$}}
\tikzset{algarrow/.style={->, thick}}
\tikzset{blgarrow/.style={->, thick}}
\tikzset{clgarrow/.style={->, thick}}
\tikzset{tensoralgarrow/.style={double, double equal sign distance, -implies}}
\tikzset{tensorblgarrow/.style={double, double equal sign distance, -implies}}
\tikzset{tensorclgarrow/.style={double, double equal sign distance, -implies}}
\tikzset{tensorelgarrow/.style={double, double equal sign distance, -implies}}
\tikzset{modarrow/.style={->, dashed}}
\tikzset{Amodar/.style={->, dashed}}
\tikzset{Dmodar/.style={->, dashed}}
\tikzset{DAmodar/.style={->, dashed}}

\usepackage{tikz-cd}

\input{macros}
\input{defs}

\copyrightinfo{2009}{American Mathematical Society}

\title[Seiberg-Witten Floer K-Theory and Cyclic Group Actions]{Seiberg-Witten Floer K-Theory and Cyclic Group Actions on Spin Four-Manifolds with Boundary}

\author{Ian Montague}
\address{Department of Mathematics, Boston College, Chestnut Hill, MA 02467}
\email{montagui@bc.edu}

\begin{document}

\input{abstract}

\maketitle

\setcounter{tocdepth}{3}

\tableofcontents

\input{intro}
\input{group_actions}
\input{k_theory}
\input{k_invariants}

\input{stable_k_invariants}
\input{stable_homotopy_type}
\input{kappa_invariants}
\input{calculations_knot_invariants}

\input{applications}

\appendix

\clearpage
\input{number_theory}
\input{tables}

\bibliographystyle{alpha}
\clearpage
\bibliography{refs}

\end{document}

%% file: macros.tex
\newread\testin

\graphicspath{{draws/}{mpdraws/}}
\makeatletter
\def\input@path{{}{draws/}}
\makeatother

\makeatletter
\newcommand\mi@kern[1]{%
  \settowidth\@tempdima{$\mi@obj^{#1}$}
  \kern-\@tempdima
  #1
  \settowidth\@tempdima{$\mi@obj$}
  \kern\@tempdima
}

\newtoks\mi@toksp
\newtoks\mi@toksb
\DeclareRobustCommand{\manyindices}[5]{
  \def\mi@obj{#5}
  \mi@toksp\expandafter{\mi@kern{#2}}
  \mi@toksb\expandafter{\mi@kern{#1}}
  \@mathmeasure4\textstyle{#5_{#1}^{#2}}
  \@mathmeasure6\textstyle{#5_{#3}^{#4}}
  \dimen0-\wd6 \advance\dimen0\wd4
  \@mathmeasure8\textstyle{\hphantom{{}_{#1}^{#2}}#5^{\the\mi@toksp#4}_{\the\mi@toksb#3}}
  \hbox to \dimen0{}{\kern-\dimen0\box8}
}
\makeatother 

\usepackage{ifpdf}
\ifpdf
  
\else
  
\fi


%% file: defs.tex
\newtheorem{theorem}{Theorem}[section]
\newtheorem{corollary}[theorem]{Corollary}
\newtheorem{lemma}[theorem]{Lemma}
\newtheorem{proposition}[theorem]{Proposition}
\newtheorem{conjecture}[theorem]{Conjecture}
\newtheorem{question}[theorem]{Question}

\theoremstyle{definition}
\newtheorem{definition}[theorem]{Definition}
\newtheorem{example}[theorem]{Example}

\theoremstyle{remark}
\newtheorem{remark}[theorem]{Remark}
\newtheorem{fact}[theorem]{Fact}
\newtheorem{observation}[theorem]{Observation}

\numberwithin{equation}{section}

\newenvironment{psmallmatrix}
  {\left(\begin{smallmatrix}}
  {\end{smallmatrix}\right)}

\newcommand{\twopartdef}[4]
{
	\left\{
		\begin{array}{ll}
			#1 & \mbox{if } #2 \\
			#3 & \mbox{if } #4
		\end{array}
	\right.
}

\newcommand{\threepartdef}[6]
{
	\left\{
		\begin{array}{ll}
			#1 & \mbox{if } #2 \\
			#3 & \mbox{if } #4 \\
            #5 & \mbox{if } #6
		\end{array}
	\right.
}

\newcommand{\bb}[1]{\mathbb{#1}}

\newcommand{\cal}[1]{\mathcal{#1}}
\renewcommand{\frak}[1]{\mathfrak{#1}}
\newcommand{\ZZ}{\bb{Z}}
\newcommand{\QQ}{\bb{Q}}
\newcommand{\RR}{\bb{R}}
\newcommand{\CC}{\bb{C}}

\DeclareMathOperator{\im}{im}
\DeclareMathOperator{\tr}{tr}
\DeclareMathOperator{\Aut}{Aut}

\DeclareMathOperator{\Hom}{Hom}

\newcommand{\spinc}{\text{Spin}^{c}}
\newcommand{\wt}[1]{\widetilde{#1}}
\newcommand{\wh}[1]{\widehat{#1}}
\newcommand{\ol}[1]{\overline{#1}}
\newcommand{\ul}[1]{\underline{#1}}

\newcommand{\del}{\partial}

\newcommand{\<}{\langle}
\renewcommand{\>}{\rangle}

\newcommand{\mbf}[1]{\mathbf{#1}}

\makeatletter
\def\widebreve{\mathpalette\wide@breve}
\def\wide@breve#1#2{\sbox\z@{$#1#2$}%
     \mathop{\vbox{\m@th\ialign{##\crcr
\kern0.08em\brevefill#1{0.8\wd\z@}\crcr\noalign{\nointerlineskip}%
                    $\hss#1#2\hss$\crcr}}}\limits}
\def\brevefill#1#2{$\m@th\sbox\tw@{$#1($}%
  \hss\resizebox{#2}{\wd\tw@}{\rotatebox[origin=c]{90}{\upshape(}}\hss$}
\makeatletter

\newcommand{\B}{\mathcal{B}}
\newcommand{\C}{\mathcal{C}}
\newcommand{\CCC}{\frak{C}}
\newcommand{\D}{\mathcal{D}}

\newcommand{\DDD}{\frak{D}}

\newcommand{\EEE}{\frak{E}}

\newcommand{\G}{\mathcal{G}}

\renewcommand{\H}{\mathcal{H}}
\newcommand{\HH}{\mathbb{H}}
\newcommand{\I}{\mathcal{I}}

\newcommand{\III}{\mathfrak{I}}
\newcommand{\K}{\mathcal{K}}

\newcommand{\LLL}{\mathfrak{L}}
\newcommand{\M}{\mathcal{M}}
\newcommand{\N}{\mathcal{N}}
\newcommand{\NN}{\mathbb{N}}

\renewcommand{\P}{\mathcal{P}}

\newcommand{\Q}{\mathcal{Q}}
\newcommand{\R}{\mathcal{R}}
\renewcommand{\S}{\mathcal{S}}
\renewcommand{\SS}{\mathbb{S}}
\newcommand{\SSS}{\mathfrak{S}}
\newcommand{\T}{\mathcal{T}}

\newcommand{\U}{\mathcal{U}}

\newcommand{\VV}{\mathbb{V}}

\newcommand{\X}{\mathcal{X}}

\renewcommand{\backslash}{\char`\\}

\DeclareMathOperator{\ind}{ind}
\DeclareMathOperator{\id}{id}

\DeclareMathOperator{\ev}{ev}
\DeclareMathOperator{\odd}{odd}
\DeclareMathOperator{\sym}{sym}

\DeclareMathOperator{\Pin}{Pin}
\DeclareMathOperator{\DSWF}{DSWF}
\DeclareMathOperator{\SWF}{SWF}

\DeclareMathOperator{\Fr}{Fr}
\DeclareMathOperator{\SF}{SF}
\DeclareMathOperator{\res}{res}
\DeclareMathOperator{\pt}{pt}
\DeclareMathOperator{\can}{can}

\DeclareMathOperator{\loc}{loc}
\DeclareMathOperator{\nrm}{nrm}
\DeclareMathOperator{\spn}{span}
\DeclareMathOperator{\nt}{nt}

\usepackage{slashed}
\usepackage{mathtools}
\newcommand{\fsl}[1]{\ensuremath{\mathrlap{\!\not{\phantom{#1}}}#1}}
\newcommand{\dirac}{\fsl{\partial}}
\newcommand{\Dirac}{\slashed{D}}

\newcommand{\mbfa}{\mbf{a}}
\newcommand{\Amid}{\mbf{A}}

\newcommand{\mbfb}{\mbf{b}}
\newcommand{\Bmid}{\mbf{B}}

\newcommand{\mbfh}{\mbf{h}}

\newcommand{\mbfk}{\mbf{k}}

\newcommand{\mbfn}{\mbf{n}}

\newcommand{\mbfr}{\mbf{r}}

\newcommand{\mbfs}{\mbf{s}}

\newcommand{\mbft}{\mbf{t}}

\newcommand{\mbfu}{\mbf{u}}
\newcommand{\Umid}{\mbf{U}}

\newcommand{\mbfv}{\mbf{v}}

\newcommand{\mbfx}{\mbf{x}}

\newcommand{\mbfw}{\mbf{w}}

\DeclareMathOperator{\st}{st}

\DeclareMathOperator{\coker}{coker}
\DeclareMathOperator{\sign}{sign}
\DeclareMathOperator{\Spin}{Spin}

\DeclareMathOperator{\CSD}{CSD}
\DeclareMathOperator{\Map}{Map}

\DeclareMathOperator{\KMT}{KMT}
\DeclareMathOperator{\gC}{gC}
\DeclareMathOperator{\irr}{irr}
\DeclareMathOperator{\rot}{rot}
\DeclareMathOperator{\topp}{top}
\DeclareMathOperator{\colim}{colim}

\newcommand{\PS}{\mathcal{PS}}
\newcommand{\SWFM}{\mathcal{SWFM}}

\newcommand{\SWFS}{\mathcal{SWFS}}

\newcommand{\SRH}{\mathcal{SRH}}
\newcommand{\LSWFS}{\mathcal{LSWFS}}

\makeatletter
\def\hlinewd#1{
\noalign{\ifnum0=`}\fi\hrule \@height #1 \futurelet
\reserved@a\@xhline}
\makeatother

%% file: abstract.tex
\begin{abstract}
Given a spin rational homology sphere $Y$ equipped with a $\ZZ/m$-action preserving the spin structure, we use the Seiberg--Witten equations to define equivariant refinements of the invariant $\kappa(Y)$ from \cite{Man14}, which take the form of a finite subset of elements in a lattice constructed from the representation ring of a twisted product of $\Pin(2)$ and $\ZZ/m$. The main theorems consist of equivariant relative 10/8-ths type inequalities for spin equivariant cobordisms between rational homology spheres. We provide applications to knot concordance, give obstructions to extending cyclic group actions to spin fillings, and via taking branched covers we obtain genus bounds for knots in punctured 4-manifolds. In some cases, these bounds are strong enough to determine the relative genus for a large class of knots within certain homology classes in $\CC P^{2}\#\CC P^{2}$, $S^{2}\times S^{2}\# S^{2}\times S^{2}$, $\CC P^{2}\# S^{2}\times S^{2}$, and homotopy $K3$ surfaces.
\end{abstract}

%% file: intro.tex
\section{Introduction}
\label{sec:intro}

\subsection{Overview}
\label{subsec:intro_overview}

In this article we define a package of invariants for spin rational homology spheres equipped with cyclic group actions, as well as equivariant relative 10/8-ths type inequalities for equivariant spin fillings of such manifolds. The construction of these invariants goes through an application of equivariant $K$-theory to a version of the Seiberg--Witten Floer stable homotopy type which takes the cyclic group action into account.

The main theorems of this paper are given by two equivariant relative 10/8-ths inequalities for equivariant spin 4-manifolds with boundary. The first is an inequality which decomposes Manolescu's relative 10/8-ths inequality into its eigenspace components, and the second is a Bryan-type inequality for odd-type $2^{r}$-fold actions.

As applications, we give homological constraints on extending cyclic group actions over spin fillings, some of which are in terms of the equivariant $\eta$-invariants of the Dirac operator on the bounding 3-manifold. We also obtain results on knot concordance and obtain genus bounds for knots in the boundaries of punctured 4-manifolds --- in particular we find examples for which these bounds are sharp and thus determine the relative genus.

\subsection{Background}
\label{subsec:intro_background}

A fundamental question in 4-manifold topology concerns the following \emph{geography} problem: 

\begin{question}
Which symmetric bilinear forms can be realized as the intersection form of a smooth, spin 4-manifold $W$ (closed or with boundary)?
\end{question}

For $W$ closed, it is a fact that the intersection form must be even and unimodular, and by Rokhlin's Theorem (\cite{Rok52}) the signature of $W$ must be divisible by 16. In 1982, Matsumoto \cite{Mat82} posited
the $11/8$-ths conjecture, which says that
\[b_{2}(W)\geq \tfrac{11}{8}|\sigma(W)|.\]
Furuta \cite{Fur01} proved the following inequality for closed oriented indefinite spin 4-manifolds:
\[b_{2}(W)\geq \tfrac{10}{8}|\sigma(W)|+2,\]
sometimes referred to as Furuta’s $10/8$-ths inequality. Recently, Hopkins-Lin-Shi-Xu were able to give a refinement of Furuta's inequality depending on the value of $\sigma(W)/16$ modulo 8 (\cite{HLSX22}). As a corollary, they showed that
\[b_{2}(W)\geq \tfrac{10}{8}|\sigma(W)|+4\]
if $|\sigma(W)|\geq 32$. 

We can rephrase the above inequalities as follows: by switching the orientation if necessary, we can assume that $\sigma(W)\le 0$. Writing $p=-\frac{1}{8}\sigma(W)$, $q=b_{2}^{+}(W)$, one can rewrite the various inequalities given above as follows:
\begin{align*}
    &11/8\text{-ths Conjecture:} & &q\geq \tfrac{3}{2}p, \\
    &10/8\text{-ths Theorem:} & &q\geq p+1, \\
    &\text{\cite{HLSX22}}: & &q\geq p+2\qquad\text{if }p\geq 4.
\end{align*}

An interesting generalization of the above is the question of whether there exists an analogous inequality between $p$ and $q$ for $W$ a spin 4-manifold with fixed boundary $\del W=Y$. Given a rational homology 3-sphere $Y$ equipped with a spin structure $\frak{s}$, Manolescu \cite{Man14} defined a numerical invariant $\kappa(Y,\frak{s})\in\QQ$ and proved that for any indefinite spin 4-manifold $(W,\frak{t})$ with $\del W = Y$, $\frak{t}|_{\del W}=\frak{s}$, and $p,q$
as above, the following relative $10/8$-ths inequality holds:
\begin{equation}
    q+\kappa(Y,\frak{s})\geq p+1. \label{eq:manolescus_inequality}
\end{equation}
Furthermore, he showed that $\kappa(S^{3})=0$, and so his inequality implies Furuta’s $10/8$-ths theorem as a corollary.

\subsection{Equivariant \texorpdfstring{$\kappa$}{κ}-invariants}
\label{subsec:intro_equivariant_kappa_invariants}

Let $(Y,\frak{s})$ be a spin rational homology sphere, and let $\sigma:Y\to Y$ be a diffeomorphism of order $m\geq 2$ which preserves the spin structure $\frak{s}$, i.e., lifts to a symmetry of the spinor bundle. In this setting, we can construct a \emph{spin lift} $\wh{\sigma}$ of $\sigma$, which comes in two flavors -- either \emph{even} or \emph{odd} type (see Section \ref{sec:group_actions} for more details), which we call the \emph{parity} of $\wh{\sigma}$. We call such a triple $(Y,\frak{s},\wh{\sigma})$ a \emph{$\ZZ_{m}$-equivariant spin rational homology sphere}, where $\ZZ_{m}:=\ZZ/m\ZZ$. Define the groups
\begin{align*}
    &G^{\ev}_{m}:=\Pin(2)\times\ZZ_{m}, & &G^{\odd}_{m}:=\Pin(2)\times_{\ZZ_{2}}\ZZ_{2m}, 
\end{align*}
where in the latter group we mod out by the diagonal $\ZZ_{2}$ subgroup. One of the main results in this paper is the construction of a $G^{*}_{m}$-equivariant Seiberg--Witten Floer stable homotopy type associated to $(Y,\frak{s},\wh{\sigma})$, where $\ast=\ev$ or $\odd$ depending on the parity of $\wh{\sigma}$:

\begin{theorem}
Associated to any triple $(Y,\frak{s},\wh{\sigma})$, there exists a well-defined metric-independent $G^{*}_{m}$-spectrum class $\SWF(Y,\frak{s},\wh{\sigma})$ which reduces to the $\Pin(2)$-equivariant spectrum class \\ $\SWF(Y,\frak{s})$ defined in \cite{Man16} under the corresponding restriction map. In particular, for any $\ZZ_{m}$-equivariant metric $g$ on $Y$ there exists a well-defined equivariant correction term $n(Y,\frak{s},\wh{\sigma},g)\in\QQ[\ZZ_{2m}]$ whose variation under one-parameter families of equivariant metrics agrees with the equivariant spectral flow of the Dirac operator on $Y$.
\end{theorem}

We are then able to extract numerical invariants from the $G^{*}_{m}$-equivariant K-theory of $\SWF(Y,\frak{s},\wh{\sigma})$, which serve as equivariant analogues of the invariant $\kappa(Y,\frak{s})$ defined by \\ Manolescu \cite{Man14}. 

These invariants come in a somewhat peculiar form --- to the representation ring $R(G^{*}_{m})$ we associate a certain poset $\Q^{m}_{*}$ arising as a quotient of $\QQ^{m}$ endowed with the standard product partial order.
More precisely, $\Q^{m}_{*}=(\Q^{m}_{*},\preceq,+,|\cdot|)$ has the structure of a \emph{$\QQ$-graded additive poset}, i.e., $(\Q^{m}_{*},+)$ is an additive monoid endowed with a partial order $\preceq$ which is compatible with $+$ in a suitable sense, and an additive poset homomorphism $|\cdot|:(\Q^{m}_{*},\preceq,+)\to(\QQ,\le,+)$ referred to as the \emph{$\QQ$-grading} on $\Q^{m}_{*}$.

In the case were $m=p^{r}$ is a prime power, we can determine explicitly the relations defining this quotient lattice, which arise from identities involving units in $\ZZ[e^{2\pi i/p^{r}}]$ (see Appendix \ref{sec:number_theory}). From the image of the restriction map on $K$-theory to the $S^{1}$-fixed point set of $\SWF(Y,\frak{s},\wh{\sigma})$, we extract a semi-infinite sub-poset $I\subset\Q^{m}_{*}$ whose collection of minima 
\[\K(Y,\frak{s},\sigma)=\min(I)\subset \Q^{m}_{*}\]
constitutes an invariant of the triple $(Y,\frak{s},\sigma)$. We call this set of minima the \emph{set of equivariant $\kappa$-invariants} of $(Y,\frak{s},\sigma)$, some of whose properties are contained in the following theorem:

\begin{theorem}
\label{theorem:intro_properties_equivariant_kappa_invariants}
Let $(Y,\frak{s},\sigma)$ be a $\ZZ_{m}$-equivariant spin rational homology sphere. We can associate to $(Y,\frak{s},\sigma)$ a finite subset $\K(Y,\frak{s},\sigma)\subset(\Q^{m}_{*},\succeq)$ which satisfies the following properties:
\begin{enumerate}
    \item Conjugation invariance: For any orientation-preserving diffeomorphism $f:Y\to Y$ which preserves $\frak{s}$, we have that:
    \[\K(Y,\frak{s},f^{-1}\circ\sigma\circ f)=\K(Y,\frak{s},\sigma).\]
    \item Orientation reversal: For any $\vec{\kappa}\in\K(Y,\frak{s},\sigma)$ and $\vec{\kappa}'\in\K(-Y,\frak{s},\sigma)$, where $-Y$ denotes the orientation-reverse of $Y$, we have that:
    \[\vec{\kappa}+\vec{\kappa}'\succeq [\vec{0}],\]
    where $[\vec{0}]\in\Q^{m}_{*}$ denotes the equivalence class of the zero vector $\vec{0}\in\QQ^{m}$.
    \item Equivariant homology cobordism: Suppose there exists a $\ZZ_{m}$-equivariant spin rational homology cobordism $(W,\frak{t},\tau)$ from $(Y_{0},\frak{s}_{0},\sigma_{0})$ to $(Y_{1},\frak{s}_{1},\sigma_{1})$. Then:
    \[\K(Y_{0},\frak{s}_{0},\sigma_{0})=\K(Y_{1},\frak{s}_{1},\sigma_{1})\]
    as subsets of $\Q^{m}_{*}$.
    \item Comparison with Manolescu's invariants: For any $\vec{\kappa}\in\K(Y,\frak{s},\sigma)$ we have that
    \[|\vec{\kappa}|\geq\kappa(Y,\frak{s}),\]
    where $|\cdot|:\Q^{m}_{*}\to\QQ$ denotes the $\QQ$-grading.
\end{enumerate}
\end{theorem}

From the subset $\K(Y,\frak{s},\sigma)\subset\Q^{*}_{m}$ we also extract two more invariants called the \emph{lower} and \emph{upper} equivariant $\kappa$-invariants of $(Y,\frak{s},\sigma)$:
\begin{align*}
    &\vec{\ul{\kappa}}(Y,\frak{s},\sigma)\in\wh{\Q}^{m}_{*}, &     &\vec{\ol{\kappa}}(Y,\frak{s},\sigma)\in\wh{\Q}^{m}_{*},
\end{align*}
where $\wh{\Q}^{*}_{m}=\Q^{m}_{*}\cup\{+\infty\}$. These invariants are the meet and join, respectively, of $\K(Y,\frak{s},\sigma)$ as a finite subset of $(\wh{\Q}^{m}_{*},\preceq)$, with the convention that if $\K(Y,\frak{s},\sigma)=\emptyset$, then
\[\vec{\ul{\kappa}}(Y,\frak{s},\sigma)=\vec{\ol{\kappa}}(Y,\frak{s},\sigma)=+\infty.\]
%

\subsection{Equivariant Relative 10/8-ths Inequalities}
\label{subsec:intro_equivariant_10/8ths}

In this article, we also derive equivariant generalizations of Manolescu's relative $10/8$-ths inequality (\ref{eq:manolescus_inequality}). Before presenting our equivariant versions of the above inequality, we need to introduce the equivariant analogues of $b_{2}^{+}(W)$ and $\sigma(W)$ which appear in our story. In order to do this, we first establish some notation:

Suppose $(W,\frak{t},\wh{\tau})$ is a compact $\ZZ_{m}$-equivariant spin 4-manifold. For each $0\le k\le m-1$, let $b_{2,\CC}^{+}(W,\tau)_{k}\geq 0$ denote the dimension (over $\CC$) of the $e^{2\pi ik/m}$-eigenspace of the induced action of $\tau$ on $H^{2}_{+}(W,\CC)$. Throughout this paper, we fix the following basis for $\QQ^{m}$:
\[\QQ^{m}=\text{span}_{\QQ}\{\vec{e}_{0},\dots,\vec{e}_{m-1}\}.\]
We then define
\[\vec{b}_{2}^{+}(W,\tau):=(b_{2}^{+}(W,\tau)_{0},\dots,b_{2}^{+}(W,\tau)_{m-1})\in\ZZ_{\geq 0}^{m}\subset\QQ^{m}\]
to be the vector consisting of entries equal to the dimensions of the various eigenspaces of $H^{2}_{+}(W,\CC)$.

The role of $\sigma(W)$ is played by a somewhat more mysterious character --- in particular, the invariant that arises is not quite related to the equivariant signature as one might suspect. Rather, it is a topological invariant assembled from terms in the $G$-Spin theorem coming from fixed-point sets of iterates of $\tau$. Given a compact $\ZZ_{m}$-equivariant spin 4-manifold $(W,\frak{t},\tau)$, for each $\ell=0,\dots,m-1$ we define an invariant $\SSS_{\ell}(W,\frak{t},\tau)\in\QQ$ (see Definition \ref{def:frak_S} for the full definition) which we can assemble into the following vector:
\[\vec{\SSS}(W,\frak{t},\tau)=(\SSS_{0}(W,\frak{t},\tau),\dots,\SSS_{m-1}(W,\frak{t},\tau))\in\QQ^{m}.\]
This invariant satisfies the following properties:
\begin{enumerate}
    \item $|\vec{\SSS}(W,\frak{t},\tau))|=\sigma(W)$.
    \item If $W$ is closed, then each $\SSS_{\ell}(W,\frak{t},\tau)\in\ZZ$ can be expressed in terms of the dimensions of eigenspaces of the action of $\wh{\tau}$ on $\ind\Dirac^{+}_{W}\in R(\ZZ_{2m})$.
\end{enumerate}
We are now ready to state our equivariant relative 10/8-ths inequalities for fillings:

\begin{theorem}
\label{theorem:intro_equivariant_filling}
Let $m\geq 2$ be an integer. Suppose $(W,\frak{t},\wh{\tau})$ is a compact $\ZZ_{m}$-equivariant spin 4-manifold with $b_{1}(W)=0$, and boundary a $\ZZ_{m}$-equivariant spin rational homology sphere $(Y,\frak{s},\wh{\sigma})$. Let
\begin{align*}
&\frak{p}=-\tfrac{1}{8}\vec{\SSS}(W,\frak{t},\tau), & &\frak{q}=\vec{b}_{2}^{+}(W,\tau),
\end{align*}
and for each $\vec{v}\in\QQ^{m}$ let $[\vec{v}]$ denote its equivalence class in the quotient lattice $\Q^{m}_{*}$. Then for each $\vec{\kappa}\in\K(Y,\frak{s},\sigma)$, the following inequality holds:
\begin{equation*}
    [\frak{q}]+\vec{\kappa}\succeq[\frak{p}]+\left\{
	\begin{array}{ll}
		[\vec{e}_{0}] & \mbox{if } b_{2}^{+}(W,\tau)_{0}\geq 1, \\
        {[\vec{0}]} & \mbox{if } b_{2}^{+}(W,\tau)_{0}=0.
	\end{array}
\right.
\end{equation*}
In particular, we have that:
\begin{equation*}
    [\frak{q}]+\vec{\ul{\kappa}}(Y,\frak{s},\sigma)\succeq[\frak{p}]+\left\{
	\begin{array}{ll}
		[\vec{e}_{0}] & \mbox{if } b_{2}^{+}(W,\tau)_{0}\geq 1, \\
        {[\vec{0}]} & \mbox{if } b_{2}^{+}(W,\tau)_{0}=0.
	\end{array}
\right.
\end{equation*}
\end{theorem}

We also have a similar inequality for cobordisms --- see Theorem \ref{theorem:equivariant_cobordism} for the full statement. The above inequalities can be potentially difficult to use in practice, especially since they all lie within the context of the non-standard lattice $\Q^{m}_{*}$. However in the case where $m=p^{r}$ is an odd prime power, we can extract more tractable inequalities by ``splitting" these inequalities into $\ZZ_{p^{r}}$-invariant and non-invariant parts. In particular, we can extract a finite subset
\[\K^{\pi}(Y,\frak{s},\sigma)\subset\QQ^{2},\]
as well as invariants
\begin{align*}
	&\ul{\kappa}_{0}(Y,\frak{s},\sigma), & &\ol{\kappa}_{0}(Y,\frak{s},\sigma), & &\ul{\kappa}_{1}(Y,\frak{s},\sigma), & &\ol{\kappa}_{1}(Y,\frak{s},\sigma)
\end{align*}
which lie in $\wh{\QQ}=\QQ\cup\{+\infty\}$. As a consequence of Theorem \ref{theorem:intro_equivariant_filling} we obtain the following:

\begin{theorem}
\label{theorem:intro_equivariant_cobordism_odd_prime_powers}
Let $p^{r}$ be an odd prime power, let $(Y,\frak{s},\sigma)$ be a $\ZZ_{p^{r}}$-equivariant spin rational homology sphere, and let $(W,\frak{t},\tau)$ be a compact $\ZZ_{p^{r}}$-equivariant spin filling of $(Y,\frak{s},\sigma)$ with $b_{1}(W)=0$. Then
\begin{align*}
	&b_{2}^{+}(W,\tau)_{0}+\kappa_{0}\geq -\tfrac{1}{8}\SSS(W,\frak{t},\tau)_{0}+\left\{
	\begin{array}{ll}
		1 & \mbox{if } b_{2}^{+}(W,\tau)_{0}\geq 0 \\
		0 & \mbox{otherwise}
	\end{array}
	\right. \\
		&b_{2}^{+}(W)-b_{2}^{+}(W,\tau)_{0}+\kappa_{1}\geq -\tfrac{1}{8}\big(\sigma(W)-\SSS(W,\frak{t},\tau)_{0}\big)
\end{align*}
for all $(\kappa_{0},\kappa_{1})\in\K^{\pi}(Y,\frak{s},\sigma)\subset\QQ^{2}$. In particular, the following inequalities hold:
\begin{align*}
	&b_{2}^{+}(W,\tau)_{0}+\ul{\kappa}_{0}(Y,\frak{s},\sigma)\geq -\tfrac{1}{8}\SSS(W,\frak{t},\tau)_{0}+\left\{
	\begin{array}{ll}
		1 & \mbox{if } b_{2}^{+}(W,\tau)_{0}\geq 0, \\
		0 & \mbox{otherwise,}
	\end{array}
	\right. \\
		&b_{2}^{+}(W)-b_{2}^{+}(W,\tau)_{0}+\ul{\kappa}_{1}(Y,\frak{s},\sigma)\geq -\tfrac{1}{8}\big(\sigma(W)-\SSS(W,\frak{t},\tau)_{0}\big).
\end{align*}
\end{theorem}

\begin{remark}
If $(W,\frak{t},\wh{\tau})$ is a closed $\ZZ_{p^{r}}$-equivariant spin 4-manifold with $\wh{\tau}$ a spin lift of even type, then
\[-\tfrac{1}{8}\SSS(W,\frak{t},\tau)_{0}=\ind_{\CC}(\Dirac^{+})^{\ZZ_{p^{r}}},\]
where $\ind_{\CC}(\Dirac^{+})^{\ZZ_{p^{r}}}$ denotes the dimension of the $\ZZ^{p^{r}}$-fixed subspace of the index of the Dirac operator on $W$. It follows that the inequalities in Theorem \ref{theorem:intro_equivariant_cobordism_odd_prime_powers} are natural generalizations of the inequalities considered by Fang (\cite{Fang01}) and Kim (\cite{Kim08}).
\end{remark}

\subsection{\texorpdfstring{$2^{r}$}{2r}-fold Actions}
\label{subsec:intro_2_r_actions}

In the case where $m=2^{r}$ is a power of two, a $\ZZ_{2^{r}}$-equivariant refinement of Furuta's theorem in the closed 4-manifold case was established independently by Bryan (\cite{Bry97}), Fang (\cite{Fang01}), and Kim (\cite{Kim00}). We have the following analogue of Bryan and Fang's theorem for odd-type $\ZZ_{2^{r}}$-equivariant spin 4-manifolds with boundary:

\begin{theorem}
\label{theorem:intro_odd_2_r_filling}
Let $(Y,\frak{s},\wh{\sigma})$ be a $\ZZ_{2^{r}}$-equivariant spin rational homology sphere, let $(W,\frak{t},\wh{\tau})$ be a compact $\ZZ_{2^{r}}$-equivariant spin filling of $(Y,\frak{s},\wh{\sigma})$ with $b_{1}(W)=0$, and let $p=-\frac{1}{8}\sigma(W)$, $q=b_{2}^{+}(W)$. Then
\[q+|\vec{\kappa}|\geq p+r+1\qquad\text{ for all }\vec{\kappa}\in\K(Y,\frak{s},\sigma),\]
provided certain conditions on $Y$ and $W$ are satisfied.
\end{theorem}

\begin{remark}
The hypotheses for Theorem \ref{theorem:intro_odd_2_r_filling} are quite technical, so we provide the full statement as Theorem \ref{theorem:odd_2_r_filling}.
\end{remark}

In the case of odd-type involutions, we are able to prove a slightly stronger inequality than the one implied by Theorem \ref{theorem:intro_odd_2_r_filling}. We first remark on the structure of our equivariant $\kappa$-invariants in the $G^{\odd}_{2}$-setting:

The $\QQ$-grading on $\Q^{2}_{\odd}$ in fact induces an isomorphism of additive posets
\[|\cdot|:(\Q^{2}_{\odd},\preceq,+)\xrightarrow{\cong}(\QQ,\le,+).\]
In particular for any $\ZZ_{2}$-equivariant spin rational homology sphere $(Y,\frak{s},\iota)$ of odd type, the invariant $\K(Y,\frak{s},\iota)\subset\Q^{2}_{\odd}$ always consists of a single element $\vec{\kappa}$, from which it follows that 
\[\vec{\ul{\kappa}}(Y,\frak{s},\iota)=\vec{\ol{\kappa}}(Y,\frak{s},\iota)=\vec{\kappa}\in\Q^{2}_{\odd}\cong\QQ.\]
We therefore define the invariant
\[\wt{\kappa}(Y,\frak{s},\iota):=|\vec{\ul{\kappa}}(Y,\frak{s},\iota)|=|\vec{\ol{\kappa}}(Y,\frak{s},\iota)|\in\QQ.\]
Note that property (4) in Theorem \ref{theorem:intro_properties_equivariant_kappa_invariants} implies that $\wt{\kappa}(Y,\frak{s},\iota)\geq\kappa(Y,\frak{s})$. With a little more work we can show the following:
\[\wt{\kappa}(Y,\frak{s},\iota)=\kappa(Y,\frak{s}) \text{ or }\kappa(Y,\frak{s})+2.\]
We say that $(Y,\frak{s},\iota)$ is \emph{$\SWF$-$\Pin(2)$-surjective} if $\wt{\kappa}(Y,\frak{s},\iota)=\kappa(Y,\frak{s})$. In fact, for all examples calculated thus far this property holds. Due to a lack of any counterexamples, we make the following conjecture:

\begin{conjecture}
\label{conj:intro_Pin(2)-surjective}
All $\ZZ_{2}$-equivariant spin rational homology spheres of odd type are $\SWF$-$\Pin(2)$-surjective.
\end{conjecture}

The following theorem provides a small amount of evidence for Conjecture \ref{conj:intro_Pin(2)-surjective}: 

\begin{theorem}
\label{theorem:intro_seifert_fibered_spaces}
Let $(Y,\frak{s})$ be a spin rational homology Seifert-fibered space of negative fibration in the sense of \cite{MOY} (see also \cite{Stoff20}, Section 5.1), and with at most four singular fibers. Suppose $\iota:Y\to Y$ is an odd-type involution preserving $\frak{s}$. Then
\[\wt{\kappa}(Y,\frak{s},\iota)=\kappa(Y,\frak{s}).\]
\end{theorem}

Recently, Konno--Miyazawa--Taniguchi (\cite{KMT}) extended an inequality of Kato (\cite{Kat22}) to the setting of odd-type involutions on spin 4-manifolds with boundary. Associated to an odd type $\ZZ_{2}$-equivariant spin rational homology sphere $(Y,\frak{s},\iota)$, they constructed an invariant $\kappa_{\KMT}(Y,\frak{s},\iota)\in\QQ$ such that if $(W,\frak{t},\tau)$ is a $\ZZ_{2}$-equivariant spin filling of $(Y,\frak{s},\iota)$ with $b_{1}(W)=0$, then:
\begin{equation}
\label{eq:intro_KMT}
    b_{2}^{+}(W,\tau)_{1}\geq -\tfrac{1}{16}\sigma(W)-\kappa_{KMT}(Y,\frak{s},\iota).
\end{equation}

We will also need to introduce one more definition. Recall that
\[G^{\odd}_{2}=\Pin(2)\times_{\ZZ_{2}}\ZZ_{4}.\]
Let $\Pin(2)=S^{1}\cup jS^{1}\subset SU(2)$ as per usual, and let $\mu$ be a fixed generator of $\ZZ_{4}\subset G^{\odd}_{2}$. The element $j\mu$ generates a $\ZZ_{2}$-subgroup of $G^{\odd}_{2}$, and therefore defines an involution on any $G^{\odd}_{2}$-space $\X$, whose fixed-point set we denote by $\X^{\<j\mu\>}$. We say that $(Y,\frak{s},\iota)$ is \emph{locally $\SWF$-$\<j\mu\>$-spherical} if for any spin lift $\wh{\iota}$ of $\iota$ there exists a $G^{\odd}_{2}$-spectrum class $\X$ such that:
\begin{enumerate}
    \item $\X^{\<j\mu\>}$ is homotopy equivalent to a sphere.
    \item There exist $G^{\odd}_{2}$-equivariant stable maps 
    \[\SWF(Y,\frak{s},\wh{\iota})\stackrel[g]{f}{\rightleftarrows}\X\]
    which induce $G^{\odd}_{2}$-equivariant stable homotopy equivalences on the $S^{1}$-fixed point sets.
\end{enumerate}

With this in mind, we can now state our relative 10/8-ths inequality for odd-type involutions:

\begin{theorem}
\label{theorem:intro_odd_2_filling}
Let $(Y,\frak{s},\wh{\iota})$ be a $\ZZ_{2}$-equivariant spin rational homology sphere with $\wh{\iota}$ of odd type, and let $(W,\frak{t},\wh{\tau})$ be a compact, $\ZZ_{2}$-equivariant spin filling of $(Y,\frak{s},\wh{\iota})$ with $b_{1}(W)=0$. Let
\begin{align*}
	&p=-\tfrac{1}{8}\sigma(W), & &q=b_{2}^{+}(W), & &q_{0}=b_{2}^{+}(W)_{0}, & &q_{1}=b_{2}^{+}(W)_{1},
\end{align*}
and suppose that:
\begin{enumerate}
    \item $q_{0},q_{1}\geq 1$.
    \item $(Y,\frak{s},\iota)$ is locally $\SWF$-$\<j\mu\>$-spherical.
\end{enumerate}
Then:
\begingroup
\renewcommand{\arraystretch}{1.5} 
\[q\geq p-\wt{\kappa}(Y,\frak{s},\iota)+\left\{
		\begin{array}{ll}
			4 & \mbox{if } q_{0},q_{1}\text{ both even, and }q_{1}\neq p-2\kappa_{\KMT}(Y,\frak{s},\iota), \\
			3 & \mbox{if } q_{0}\text{ odd},\,q_{1}\text{ even, and }q_{1}\neq p-2\kappa_{\KMT}(Y,\frak{s},\iota),\text{ or} \\
            & \mbox{if } q_{0}\text{ even, }q_{1}\text{ odd, and }q_{1}\neq p-2\kappa_{\KMT}(Y,\frak{s},\iota)-1, \\
            2 & \mbox{if } q_{0}, q_{1}\text{ both odd}.
		\end{array}
	\right.\]
\endgroup
\end{theorem}

One can replace the locally $\SWF$-$\<j\mu\>$-spherical condition on $(Y,\frak{s},\iota)$ with a weaker condition concerning the $RO(\ZZ_{4})$-graded equivariant homotopy groups of $\SWF(Y,\frak{s},\wh{\iota})^{\<j\mu\>}$ with respect to the residual $\<j\>\cong\ZZ_{4}$-action --- see Theorem \ref{theorem:odd_2_filling} for the general statement.

\begin{remark}
One of the reasons that the invariant $\kappa_{\KMT}(Y,\frak{s},\iota)$ defined in \cite{KMT} makes an appearance in Theorem \ref{theorem:intro_odd_2_filling} is that the $\<j\mu\>$-fixed point set of the $G^{\odd}_{2}$-spectrum \\ $\SWF(Y,\frak{s},\wh{\iota})$ with the residual $\<j\>\cong\ZZ_{4}$-action is equivalent to ``one-half" of the spectrum constructed in \cite{KMT}. In particular, our notion of $(Y,\frak{s},\wh{\iota})$ being locally $\SWF$-$\<j\mu\>$-spherical is related to the local $\DSWF$-spherical condition considered in their paper (see Section \ref{subsubsec:stable_fixed_point_sets} for more on the relationship between these two notions).
\end{remark}

We note here that while most of the examples we consider in this paper are locally $\SWF$-$\<j\mu\>$-spherical, not all of them are:

\begin{example}
Let $Y$ be the Brieskorn sphere $\Sigma(2,3,12n-1)$ equipped with its unique spin structure $\frak{s}$, and let $\rho_{2}:Y\to Y$ be the covering involution which realizes $Y$ as the double-branched cover over the torus knot $T(3,12n-1)$. Then $(Y,\frak{s},\rho_{2})$ is locally $\SWF$-$\<j\mu\>$-spherical. However if $\iota_{c}:Y \to Y$ denotes the covering involution which realizes $Y$ as the double-branched cover over the Montesinos knot $k(2,3,12n-1)$, then $(Y,\frak{s},\iota_{c})$ is \emph{not} locally $\SWF$-$\<j\mu\>$-spherical. (Some further examples can be found in Table \ref{table:knots}.)
\end{example}

As a consequence of Theorem \ref{theorem:intro_odd_2_filling} we obtain a slight improvement of Bryan and Fang's inequality for closed 4-manifolds:

\begin{corollary}
\label{cor:intro_odd_2_closed}
Let $(W,\frak{t},\wh{\tau})$ be a closed $\ZZ_{2}$-equivariant spin 4-manifold of odd type with $b_{1}(W)=0$, and let $p,q,q_{0},q_{1}$ be as in Theorem \ref{theorem:intro_odd_2_filling}. Furthermore, suppose that $q_{0},q_{1}\geq 1$. Then:
\begingroup
\renewcommand{\arraystretch}{1.5} 
\[q\geq p+\left\{
		\begin{array}{ll}
			4 & \mbox{if } q_{0},q_{1}\text{ both even, and }q_{1}\neq p, \\
			3 & \mbox{if } q_{0}\text{ odd},\,q_{1}\text{ even, and }q_{1}\neq p,\text{ or} \\
            & \mbox{if } q_{0}\text{ even, }q_{1}\text{ odd, and }q_{1}\neq p-1, \\
            2 & \mbox{if } q_{0}, q_{1}\text{ both odd}.
		\end{array}
	\right.\]
\endgroup
\end{corollary}

One can show using Corollary \ref{cor:intro_odd_2_closed}, for example, that the connected sum of two $K3$ surfaces cannot arise as a double branched cover over a manifold with $b_{2}^{+}=4$.

\subsection{Knot Concordance Invariants}
\label{subsec:intro_knot_concordance}

Let $K\subset S^{3}$ be an oriented knot and let $m=p^{r}$ be a prime power. It is a standard theorem in topology that the $p^{r}$-fold branched cover $\Sigma_{p^{r}}(K)$ is a rational homology sphere. By a theorem of \cite{GRS08}, there exists a distinguished spin structure $\frak{s}_{0}$ on $\Sigma_{m}(K)$ which is invariant under the canonical $p^{r}$-fold covering action $\sigma:\Sigma_{p^{r}}(K)\to\Sigma_{p^{r}}(K)$. We define the \emph{set of $p^{r}$-fold equivariant $\kappa$-invariants} of $K$ to be
\[\K_{p^{r}}(K):=\K(\Sigma_{p^{r}}(K),\frak{s}_{0},\sigma)\subset\Q^{p^{r}}_{*},\]
as well the \emph{upper and lower $p^{r}$-fold equivariant $\kappa$-invariants} of $K$ as follows:
\begin{align*}
	&\vec{\ol{\kappa}}_{p^{r}}(K):=\vec{\ol{\kappa}}(\Sigma_{p^{r}}(K),\frak{s}_{0},\sigma)\in\wh{\Q}^{p^{r}}_{*}, & &\vec{\ul{\kappa}}_{p^{r}}(K):=\vec{\ul{\kappa}}(\Sigma_{p^{r}}(K),\frak{s}_{0},\sigma)\in\wh{\Q}^{p^{r}}_{*}.
\end{align*}
Furthermore in the case where $p^{r}=2$ we define:
\[\wt{\kappa}(K):=\wt{\kappa}(\Sigma_{2}(K),\frak{s}_{0},\sigma).\]
The following theorem follows immediately from property (4) in Theorem \ref{theorem:intro_properties_equivariant_kappa_invariants}:

\begin{theorem}
\label{theorem:intro_concordance}
For any oriented knot $K\subset S^{3}$ and any prime power $p^{r}$, all of the $p^{r}$-fold equivariant $\kappa$-invariants are concordance invariants of $K$.
\end{theorem}

Unfortunately these concordance invariants are difficult to compute in practice. However, we can still indirectly obtain results by looking at properties of the Seiberg--Witten Floer spectrum classes of $p^{r}$-fold branched covers of knots. Consider the following definition:

\begin{definition}
\label{def:intro_locally_SWF_jmu_spherical}
Let $\C$ denote the smooth concordance group. We define $\LSWFS_{2}^{\<j\mu\>}\subset\C$ to be the subgroup generated by knots $K$ such that $(\Sigma_{2}(K),\frak{s}_{0},\sigma)$ is locally $\SWF$-$\<j\mu\>$-spherical.
\end{definition}

\begin{example}
Using results from \cite{KMT}, one can show that $\LSWFS_{2}^{\<j\mu\>}$ contains all torus knots, as well as knots whose double branched covers are \emph{minimal L-spaces} (in the sense of \cite{LinLipLS}). Furthermore, $\LSWFS_{2}^{\<j\mu\>}$ contains the subgroup generated by the families of knots $k(2,3,12n+1)$, $k(2,3,12n+5)$, where $k(p,q,r)$ denotes the Montesinos knot with double branched cover the Brieskorn sphere $\Sigma(p,q,r)$.
\end{example}

We have the following theorem, whose statement is reminiscent of (\cite{AKS20}, Theorem 1.6):

\begin{theorem}
\label{theorem:intro_montesinos_knots_not_concordant_to_LSWFS_2}
Let $K$ be a connected sum of knots of the form $k(2,3,12n-1)$, $k(2,3,12n-5)$, and their mirrors, such that the total number of prime factors of $K$ is odd. Then $K$ represents a non-zero element of the quotient $\C/\LSWFS_{2}^{\<j\mu\>}$.
\end{theorem}

The author expects that one could use the calculations and techniques developed by Alfieri-Kang-Stipcisz (\cite{AKS20}) and Dai-Hedden-Mallick (\cite{DHM22}) to prove a statement similar to Theorem \ref{theorem:intro_montesinos_knots_not_concordant_to_LSWFS_2} in the setting of Heegaard Floer homology.

\subsection{Applications to Spin Fillings}
\label{subsec:intro_spin_fillings}

Using our equivariant relative 10/8-ths inequalities, we can constrain the homological properties of cyclic group actions on spin fillings of rational homology spheres. We have the following result for odd-type involutions on spin manifolds bounded by certain homology Brieskorn spheres:

\begin{theorem}
\label{theorem:intro_extending_involution_constraints}
Let $W$ be a compact connected smooth oriented spin 4-manifold with $b_{1}(W)=0$, intersection form given by $p(-E_{8})\oplus qH$, and boundary $\del W=Y$ an integer homology sphere. 
\begin{enumerate}
    \item Suppose that $Y$ and the pair $(p,q)$ are given by one of the following:
    \begin{enumerate}
        \item $Y=\Sigma(2,3,12n-1)$ and $(p,q)=(2,2)$.
        \item $Y=\Sigma(2,3,12n-5)$, and $(p,q)=(1,2)$.
    \end{enumerate}
    Let $\iota$ be an odd-type involution on $Y$ which is isotopic to the identity. If $\iota$ extends to a smooth involution $\tau$ on $W$, then $b_{2}^{+}(W,\tau)_{0}=0$.
    \item Suppose that $Y=-\Sigma(2,3,12n+5)$, $(p,q)=(1,3)$, and let $\iota$ be an odd-type involution on $Y$ which is isotopic to the identity. If $\iota$ extends to a smooth involution $\tau$ on $W$, then $b_{2}^{+}(W,\tau)_{0}=1$.
    \item Suppose that $Y$ and the pair $(p,q)$ are given by one of the following:
    \begin{enumerate}
        \item $Y=\pm\Sigma(2,3,12n+1)$ and $(p,q)=(p,p+1)$, $p\geq 4$ even.
        \item $Y=\Sigma(2,3,12n+5)$, and $(p,q)=(p,p)$, $p\geq 3$ odd.
        \item $Y=-\Sigma(2,3,12n+5)$ and $(p,q)=(p,p+2)$, $p\geq 3$ odd.
    \end{enumerate}
    Let $\iota$ be any odd-type involution on $Y$. If $\iota$ extends to a smooth involution $\tau$ on $W$, then $b_{2}^{+}(W,\tau)_{0}=0$ or $1$.
\end{enumerate}
In all of the above cases, if $\iota$ is isotopic to the identity, then $\iota$ can extend to $W$ as a homologically trivial diffeomorphism.
\end{theorem}

\begin{remark}
Note that (3a) in the above theorem is a generalization of (\cite{KMT}, Corollary 5.5), where they consider the case $p=2$, i.e., in the case that $W$ has intersection form isomorphic to that of a K3 surface. Furthermore, they are able to exclude the case $b_{2}^{+}(W,\tau)_{0}=0$ by using (\cite{KonnoTaniguchi}, Theorem 1.2). 

While their result also includes the Brieskorn sphere $Y=-\Sigma(2,3,12n-1)$, by (\cite{Lin15}, Example 1.14) no such manifold $W$ with intersection form $p(-E_{8})\oplus (p+1)H$ is bounded by $Y$ for any $p\geq 2$.
\end{remark}

\begin{example}
The Brieskorn homology sphere $Y=\Sigma(2,3,7)$ is the boundary of the Milnor fiber $W=M(2,3,7)$ whose intersection form is given by $-E_{8}\oplus 2H$. Then for any odd-type involution $\iota$ which is isotopic to the identity on $Y$ (e.g., rotation by $\pi$ in the $S^{1}$-fibers), by (1b) in the above theorem any extension of $\iota$ to a smooth involution $\tau$ on $W$ must satisfy the property that $X=W/\tau$ is negative definite with boundary $Y/\iota\cong S^{3}$. By Donaldson's diagonalization theorem, it therefore follows that the quotient $X$ is homeomorphic to $\#^{N}\ol{\CC P}^{2}\setminus B^{4}$ for some $N\le 8$.
\end{example}

\subsection{Genus Bounds}
\label{subsec:intro_genus_bounds}

For any closed oriented 4-manifold $X$ and any homology class $A\in H_{2}(X;\ZZ)$, let $g(X,A)$ denote the minimal genus of a smooth embedded oriented surface $F\subset X$ representing the homology class $A$. The calculation of $g(X,A)$ for various pairs $(X,A)$ has a long history which is intertwined with the development of many of the important techniques used in four-dimensional topology, most notably the resolution of the Thom Conjecture (\cite{KMthomconj}) and its various generalizations (\cite{MST:thomconj}, \cite{OSz:sympthomconj}).

We will consider the following relative version of the minimal genus problem: Let $X$ be a closed oriented 4-manifold, let $K\subset S^{3}$ be an oriented knot, and let $A\in H_{2}(X;\ZZ)$ be a fixed 2-dimensional homology class. We define the \emph{$(X,A)$-genus of $K$}, denoted $g_{X,A}(K)$, to be the minimal genus over all properly embedded oriented surfaces $F\subset \mathring{X}:=X\setminus B^{4}$ such that $\del F=K\subset S^{3}$ and $[F]=A$.

Under favorable conditions, the double branched cover of $\mathring{X}$ over a surface $F$ as above is spin -- thus we can apply our relative 10/8-ths type inequality for odd-type involutions to obtain a lower bound for $g_{X,A}$:

\begin{theorem}
\label{theorem:intro_genus_bounds}
Let $X$ be a closed oriented 4-manifold with $b_{1}(X)=0$ and $b_{2}^{+}(X)\neq 0$. Furthermore, let $A\in H_{2}(X;\ZZ)$ be a two-dimensional homology class such that $2|A$ and $A/2\equiv w_{2}(X)\pmod{2}$. Suppose $K\subset S^{3}$ is a knot such that the pair $(\Sigma_{2}(K),\iota)$ is locally $\SWF$-$\<j\mu\>$-spherical, where $\iota:\Sigma_{2}(K)\to\Sigma_{2}(K)$ denotes the covering involution on the double branched cover of $K$. Finally, define
\[c(K,X):=b_{2}^{+}(X)+\wt{\kappa}(K)-2\kappa_{\KMT}(K).\]
Then the following inequality holds:
\begin{equation}
\label{eq:intro_genus_bounds}
    g_{X,A}(K)\geq -2b_{2}^{+}(X)-\tfrac{1}{4}\sigma(X)+\tfrac{5}{16}A^{2}-\tfrac{5}{8}\sigma(K)-\wt{\kappa}(K)+C,
\end{equation}
where:
\begin{equation}
    C=\left\{
		\begin{array}{ll}
			3 & \mbox{if }b_{2}^{+}(X)\text{ is even and }c(K,X)\geq 4, \\
			2 & \mbox{if }c(K,X)\geq 2, \\
			1 & \mbox{otherwise.}
		\end{array}
	\right.
\end{equation}
\end{theorem}

Next, note that we have the following upper bound on the $(X,A)$-genus of any knot $K\subset S^{3}$:
\begin{equation}
\label{eq:intro_upper_bound}
    g(X,A)+g_{4}(K)\geq g_{X,A}(K).
\end{equation}
Indeed, this follows from taking a connected sum of a closed surface representing $g(X,A)$, and a surface with boundary contained in $S^{3}\times[0,1]\subset X\setminus B^{4}$ representing $g_{4}(K)$.

There are plenty of examples of knots $K$ and pairs $(X,A)$ where this upper bound is not sharp, i.e., $g_{X,A}(K)< g(X,A)+g_{4}(K)$. For example, it was shown in (\cite{Nor69}, \cite{Suz69}) that every knot is slice in $S^{2}\times S^{2}$ and $\CC P^{2}\times\ol{\CC P}^{2}$ in some homology class. See also \cite{MMP20} for more examples.

However, there are some cases where the bound is sharp -- for example, it was shown in (\cite{BaragliaQP}, Corollary 1.3) that for $(X,A)=(K3,0)$ and any quasi-positive knot 
$K\subset S^{3}$, we have the equality $g_{K3,0}(K)=g_{4}(K)$. The following theorem adds to the list of cases where (\ref{eq:intro_upper_bound}) is sharp:

\begin{theorem}
\label{theorem:intro_X_A_genus_sharp}
Let $(X,A)$ be one of the following pairs, where $X$ is a closed oriented 4-manifold and $A\in H_{2}(X;\ZZ)$:

\begin{center}
\begin{tabular}{|c|c|}\hline
    $X$ & $A$ \\ \hline
    $S^{2}\times S^{2}\#S^{2}\times S^{2}$ & $((4,4),(4,4))$  \\ \hline
    \multirow{2}{*}{$\CC P^{2}\#\CC P^{2}$} & $(6,2)$\\ \cline{2-2}
    & $(6,6)$  \\ \hline
    \multirow{2}{*}{$S^{2}\times S^{2}\#\CC P^{2}$} & $((4,4),2)$ \\ \cline{2-2}
    & $((4,4),6)$ \\ \hline
    $hK3$ & $0$  \\ \hline
\end{tabular}
\end{center}
\smallskip
Here $X=hK3$ denotes any homotopy $K3$ surface. Furthermore, let $K\subset S^{3}$ be any knot such that:
\begin{enumerate}
    \item $K$ is smoothly concordant to a connected sum of knots $K_{1}\#\cdots\# K_{n}$ satisfying the following property: for each $i=1,\dots,n$, the double branched cover $\Sigma_{2}(K_{i})$ admits a $\ZZ_{2}$-equivariant metric $g$ such that $(\Sigma_{2}(K_{i}),g)$ admits no irreducible solutions to the Seiberg-Witten equations wiith respect to the invariant spin structure on $\Sigma_{2}(K_{i})$.
    \item $g_{4}(K)=-\frac{1}{2}\sigma(K)$.
\end{enumerate}
Then:
\begin{equation*}
    g_{X,A}(K)=g(X,A)+g_{4}(K).
\end{equation*}
In particular this folds for the following knots:
\begin{enumerate}
    \item Connected sums of quasi-positive two-bridge knots and $T(3,5)$.
    \item $9_{47}$, $9_{49}$, $10_{155}$, $m10_{156}$, $10_{160}$, and $10_{163}$.
\end{enumerate}
\end{theorem}

\begin{remark}
Note that all of the knots above are quasi-alternating, with the exception of $T(3,5)$. The knots $T(3,5)$, $T(2,2k+1)$, $9_{49}$, $10_{155}$ are quasi-positive, while $9_{47}$, $m10_{156}$, $10_{160}$, and $10_{163}$ are not. Therefore in the particular case of $(X,A)=(K3,0)$ there is some overlap between the class of knots considered in Theorem \ref{theorem:intro_X_A_genus_sharp} and those in (\cite{BaragliaQP}, Corollary 1.3).
\end{remark}

It is worth noting that for the 4-manifolds appearing in Theorem \ref{theorem:intro_X_A_genus_sharp} with the exception of the $K3$ surface, all the Seiberg--Witten and Bauer-Furuta invariants vanish --- in particular most of the usual methods to obtain genus bounds are inaccessible for these manifolds.

\subsection{Future Directions}
\label{subsec:intro_future_directions}

In order to explicitly compute equivariant $\kappa$-invariants in the case of odd-order cyclic group actions, one would need to compute equivariant $\eta$-invariants of the Dirac operator on 3-manifolds. While there are some computations of equivariant $\eta$-invariants of the odd signature operator (\cite{Anv:rho}), there are unfortunately no known computations in the literature for the Dirac operator at the time of writing, except for the 3-sphere (\cite{Deg01}). In upcoming work by the author (\cite{Mon:eta}), we plan to compute these equivariant $\eta$-invariants explicitly for Seifert-fibered spaces with respect to the cyclic group actions generated by rotations in the $S^{1}$-fibers, following the techniques used in \cite{Nic1}.

It would also be a worthwhile endeavor to construct $G^{*}_{m}$-equivariant analogues of the $\alpha$, $\beta$, $\gamma$ invariants from \cite{Man16} in the setting of equivariant homology, as well as equivariant analogues of the $\kappa o_{i}$ invariants from \cite{Lin15} in the setting of equivariant $KO$-theory.

\subsection{Organization}
\label{subsec:intro_organization}

In Section \ref{sec:group_actions} we give an overview of cyclic group actions and spin lifts. In Section \ref{sec:k_theory} we provide an overview of $G^{*}_{m}$-equivariant $K$-theory and introduce spaces of type $G^{*}_{m}$-$\SWF$ and $\CC$-$G^{*}_{m}$-$\SWF$. In Section \ref{sec:equivariant_k_invariants} we define equivariant $k$-invariants, the precursor to our equivariant $\kappa$-invariants which we apply to spaces of type $\CC$-$G^{*}_{m}$-$\SWF$. In Section \ref{sec:stable_equivariant_k_invariants} we introduce (stable) $G^{*}_{m}$-spectrum classes and $\CC$-$G^{*}_{m}$-spectrum classes and discuss how to stabilize our equivariant $k$-invariants to produce well-defined invariants of these spaces. In Section \ref{sec:stable_homotopy_type} we define the $G^{*}_{m}$-spectrum class $\SWF(Y,\frak{s},\wh{\sigma})$ --- along the way, we discuss the $G$-Spin theorem for 4-manifolds with boundary and introduce the ingredients involved in constructing the equivariant correction term $n(Y,\frak{s},\wh{\sigma},g)$. In Section \ref{sec:kappa_invariants} we define the equivariant $\kappa$-invariants and prove our equivariant relative 10/8ths inequalities. In Section \ref{sec:calculations_knot_invariants} we provide some calculations of our equivariant $\kappa$-invariants, and in Section \ref{sec:applications} we discuss topological applications. Appendix \ref{sec:number_theory} gives a proof of Proposition \ref{prop:monomials}, and Appendix \ref{sec:tables_appendix} features some tables referred to throughout the article.


\subsection{Acknowledgements}
\label{subsec:intro_acknowledgements}
I would like to express the utmost gratitude to my advisor Daniel Ruberman for his constant support and encouragement throughout this project, without which I would have been hopelessly lost. I would also like to thank Jianfeng Lin for suggesting this problem, as well as David Baraglia, Matthew Carr, Arun Debray, Anda Degeratu, Hokuto Konno, Jiakai Li, Rahul Krishna, Liviu Nicolaescu, Matthew Stoffregen, and Masaki Taniguchi for interesting and helpful conversations. This material is based upon work supported by the National Science Foundation under Grant No. DMS-1928930 while the author was in residence at the Simons Laufer Mathematical Science Institute (previously known as MSRI) in Berkeley, California, during the Fall 2022 semester. The author was also partially supported by NSF grant DMS-1811111.

\bigskip

%% file: group_actions.tex
\section{Cyclic Group Actions on Spin Manifolds}
\label{sec:group_actions}

In this section we clarify the concept of a spin cyclic group action. We start in Section \ref{subsec:cyclic_actions_spin_3_4_manifolds} by defining the notion of a spin lift of a cyclic group action on a spin 3- or 4-manifold, and explore the dichotomy between even and odd spin lifts via the Atiyah-Bott lemma. In Section \ref{subsec:equivariant_cobordisms} we define and explore properties of the spin $\ZZ_{m}$-equivariant cobordism group $\Omega_{3}^{\Spin,\ZZ_{m}}$, and we conclude in Section \ref{subsec:equivariant_connected_sums} by defining the notion of a spin $\ZZ_{m}$-equivariant connected sum of two manifolds.

\subsection{Cyclic Group Actions on Spin 3- and 4-Manifolds}
\label{subsec:cyclic_actions_spin_3_4_manifolds}

Let $n\geq 2$, and let $(M,\frak{s})$ be a compact connected oriented spin $n$-manifold equipped with an orientation-preserving self-diffeomorphism $\alpha:M\to M$ of order $m$ for some positive integer $m\geq 2$, such that $\alpha^{*}(\frak{s})=\frak{s}$, i.e., the spin structure $\frak{s}$ is $\alpha$-invariant. Note that this is equivalent to the condition that 
\[\alpha^{*}(\frak{s})-\frak{s}=0\in H^{1}(M,\ZZ_{2}).\]
Choose an $\alpha$-invariant metric $g$ on $M$, and consider the induced map $d\alpha:\Fr(M)\to\Fr(M)$ on the associated principal $SO(n)$-frame bundle $\Fr(M)$ of $M$, which is independent of the choice of $g$ up to $\ZZ_{m}$-equivariant $SO(n)$-bundle isomorphism. Let $P\to M$ be the principal $\Spin(n)$-bundle associated to $\frak{s}$, which double covers $\Fr(M)$. By invariance of $\frak{s}$ under $\alpha$, we can choose a lift of $\alpha$ to a smooth bundle automorphism $\wh{\alpha}:P\to P$ which double covers $d\alpha$. We call the pair $(\alpha,\wh{\alpha})$ a \emph{spin diffeomorphism}. Note that there are precisely two possible spin lifts --- if $\wh{\alpha}$ is one lift, then the other lift is given by $-\wh{\alpha}$ which acts on $P$ via $(-\wh{\alpha})(p)=-(\wh{\alpha}(p))$.

Although $\alpha$ is of order $m$, the lift $\wh{\alpha}$ may not necessarily be of order $m$. Because the induced map $d\alpha:\Fr(M)\to \Fr(M)$ on the $SO(n)$-frame bundle must satisfy $(d\alpha)^{m}=1$, it follows that $\wh{\alpha}^{m}$ is either equal to the identity or equal to the \emph{spin flip}, i.e., the non-trivial involution of $\text{Spin}(n)$ as a double-cover of $SO(n)$ when restricted to each fiber. We say that $\wh{\alpha}$ is \emph{even} if $\wh{\alpha}^{m}=1$, and \emph{odd} if $\wh{\alpha}^{m}$ is equal to the spin-flip. We will call this property the \emph{parity} of $\wh{\alpha}$, which is unrelated to the parity of $m$.

We will call a triple $(M,\frak{s},\wh{\alpha})$ as above a \emph{(connected) $\ZZ_{m}$-equivariant spin $n$-manifold}. Sometimes we will use the notation $(M,\frak{s},\alpha)$ to denote a triple as above where $\frak{s}$ is $\alpha$-invariant, but we have not fixed a particular spin lift of $\alpha$. Similarly for a fixed choice of $\alpha$-invariant metric $g$ on $M$, we refer to the quadruple $(M,\frak{s},\wh{\alpha},g)$ as a \emph{$\ZZ_{m}$-equivariant Riemannian spin $n$-manifold} (or $(M,\frak{s},\alpha,g)$ if we do not wish to pick a particular spin lift of $\alpha$).

We say that that two $\ZZ_{m}$-equivariant spin $n$-manifolds $(M,\frak{s},\alpha)$ and $(M',\frak{s}',\alpha')$ are \emph{$\ZZ_{m}$-equivariantly spin diffeomorphic} (or just \emph{equivariantly diffeomorphic}) if there exists a diffeomorphism $f:M\to M'$ such that $f^{*}(\frak{s}')\cong\frak{s}$ and $f\circ\alpha=\alpha'\circ f$. Fixing spin lifts $\wh{\alpha}$,$\wh{\alpha}'$ of $\alpha$, $\alpha'$, respectively, we say that $(M,\frak{t},\wh{\alpha})$ and $(M',\frak{s}',\wh{\alpha}')$ are \emph{strongly $\ZZ_{m}$-equivariantly spin diffeomorphic} if there exists a pair $(f,\wh{f})$ where $f$ is as above, and $\wh{f}:P\to P'$ is a map on the corresponding spin bundles which double covers the induced map $df:\Fr(M)\to \Fr(M')$ on frame bundles, such that $\wh{f}\circ\wh{\alpha}=\wh{\alpha'}\circ\wh{f}$. The notion of a (strong) $\ZZ_{m}$-equivariant spin isometry between $\ZZ_{m}$-equivariant Riemannian spin 4-manifolds is defined similarly.

In the special case where $\alpha$ is an involution we have the following lemma of Atiyah and Bott, which says that parity of any spin lift $\wh{\alpha}$ of $\alpha$ is determined by the fixed-point set $M^{\alpha}\subset M$ of $\alpha$:

\begin{proposition}[\cite{AtiyahBott2}]
\label{prop:atiyah_bott}
Let $(M,\frak{s},\wh{\alpha})$ be a $\ZZ_{2}$-equivariant spin $n$-manifold. Then:
\begin{enumerate}
    \item If $M^{\alpha}=\emptyset$, then $\wh{\alpha}$ is even $\iff$ the quotient manifold $M/\alpha$ admits a spin structure $\frak{s}'$ such that $\frak{s}'$ pulls back to $\frak{s}$ under the regular two-fold covering $\pi:M\to M/\alpha$.
    \item If $M^{\alpha}\neq\emptyset$, then:
    \begin{enumerate}
        \item $\wh{\alpha}$ is even $\iff \dim(M^{\alpha})\equiv\dim(M)\pmod{4}$.
        \item $\wh{\alpha}$ is odd $\iff \dim(M^{\alpha})\equiv\dim(M)+2\pmod{4}$.
\end{enumerate}
\end{enumerate}
\end{proposition}

By using Proposition \ref{prop:atiyah_bott}, we can classify the parities of spin lifts of $\tau$ for general $m$:

\begin{proposition}
\label{prop:classification_spin_lifts_n_manifolds}
Let $(M,\frak{s},\wh{\alpha})$ be a $\ZZ_{m}$-equivariant spin $n$-manifold. Then:
\begin{enumerate}
    \item Suppose $m$ is even, and let $\wh{\alpha}$ be any spin lift of $\alpha$. Then:
    \begin{enumerate}
        \item If $M^{\alpha^{m/2}}=\emptyset$, then $\wh{\alpha}$ is even $\iff$ the quotient manifold $M/\alpha^{m/2}$ admits a spin structure $\frak{s}'$ such that $\frak{s}'$ pulls back to $\frak{s}$ under the regular two-fold covering $\pi:M\to M/\alpha^{m/2}$.
        \item If $M^{\alpha^{m/2}}\neq\emptyset$, then:
        \begin{enumerate}
            \item $\wh{\alpha}$ is even $\iff \dim(M^{\alpha^{m/2}})\equiv\dim(M)\pmod{4}$.
            \item $\wh{\alpha}$ is odd $\iff \dim(M^{\alpha^{m/2}})\equiv\dim(M)+2\pmod{4}$.
        \end{enumerate}
    \end{enumerate}
    \item If $m$ is odd, then $\alpha$ admits precisely one even and one odd spin lift.
\end{enumerate}
\end{proposition}

\begin{proof}
Statement (1) follows from Proposition \ref{prop:atiyah_bott}. For (2), let $\wh{\alpha}$ be a spin lift of $\alpha$. Then the two spin lifts $\wh{\alpha}$ and $-\wh{\alpha}$ are of opposite parity, since $(-\wh{\alpha})^{m}=(-1)^{m}\wh{\alpha}^{m}=-\wh{\alpha}^{m}$.
\end{proof}

We have the following corollary in the case where $n=3$:

\begin{proposition}
\label{prop:classification_spin_lifts_3_manifolds}
Let $(Y,\frak{s},\wh{\sigma})$ be a $\ZZ_{m}$-equivariant spin $3$-manifold. Then:
\begin{enumerate}
    \item If $m$ is even and $\wh{\sigma}$ is any spin lift of $\sigma$, then:
    \begin{enumerate}
    	\item If $Y^{\sigma^{m/2}}=\emptyset$, then $\wh{\sigma}$ is even $\iff Y^{\sigma^{m/2}}=\emptyset$ and the quotient manifold $Y/\sigma^{m/2}$ admits a spin structure $\frak{s}'$ which pulls back to $\frak{s}$ under the covering $\pi:Y\to Y/\sigma^{m/2}$.
        \item If $Y^{\sigma^{m/2}}\neq\emptyset$, then $\wh{\sigma}$ is odd and $\dim(Y^{\sigma^{m/2}})=1$.
    \end{enumerate}
    \item If $m$ is odd, then $\sigma$ admits precisely one even and one odd spin lift.
\end{enumerate}
\end{proposition}

\begin{remark}
The above proposition implies that in the case where $\alpha$ is an involution, if $Y^{\alpha}$ is non-empty then any spin lift of $\alpha$ must be of odd type. However, it is not necessarily the case that all odd-type spin involutions on 3-manifolds must have non-empty fixed point set.

For example, consider the free involution $\iota:\RR P^{3}\to\RR P^{3}$ realizing $\RR P^{3}$ as a regular two-fold cover over the lens space $L(4,1)$, and let $\pi:\RR P^{3}\to L(4,1)$ denote the corresponding projection map. Note that 
\[H^{1}(\RR P^{3},\ZZ_{2})\cong H^{1}(L(4,1),\ZZ_{2})\cong\ZZ_{2},\]
and so $\RR P^{3}$ and $L(4,1)$ each admit precisely two spin structures. We observe that $\iota$ fixes both spin structures on $\RR P^{3}$, since it acts trivially on $\pi_{1}(\RR P^{3})\cong H_{1}(\RR P^{3})$. Indeed, if $\iota$ were to act non-trivially on $\pi_{1}(\RR P^{3})\cong\ZZ_{2}$, then the induced map $\pi_{1}(\RR P^{3})\to\pi_{1}(L(4,1))\cong\ZZ_{4}$ would be trivial, a contradiction.

Let $\alpha,\beta$ denote the generators of $H^{1}(\RR P^{3},\ZZ_{2})$ and $H^{1}(L(4,1),\ZZ_{2})$, respectively. Note that $\alpha$ can be identified with the function $f;\pi_{1}(\RR P^{3})\cong\ZZ_{2}\to\ZZ_{2}$ which sends $1\mapsto 1$, and similarly $\beta$ can be identified with the function $g:\pi_{1}(L(4,1))\cong\ZZ_{4}\to\ZZ_{2}$ which sends $1\mapsto 1$. It follows that the pullback map $\pi^{*}:H^{1}(L(4,1),\ZZ_{2})\to H^{1}(\RR P^{3},\ZZ_{2})$ is trivial, since the pullback $\pi^{*}g:\pi_{1}(\RR P^{3})\to\ZZ_{2}$ factors through the inclusion $\pi_{1}(\RR P^{3})\to\pi_{1}(L(4,1))$ which sends $1\mapsto 2$, and $g$ evaluates to zero on $2\in\pi_{1}(L(4,1))\cong\ZZ_{4}$. It follows that both spin structures on $L(4,1)$ pull back to the same spin structure $\frak{s}_{0}$ on $\RR P^{3}$. Letting $\frak{s}_{1}=\frak{s}_{0}+\alpha$, we see that any spin lift of $\alpha$ with respect to $\frak{s}_{1}$ must be of odd type.
\end{remark}

Next we consider the case of 4-manifolds:

\begin{proposition}
\label{prop:classification_spin_lifts_4_manifolds}
Let $(W,\frak{t},\wh{\tau})$ be a $\ZZ_{m}$-equivariant spin $4$-manifold. Then:
\begin{enumerate}
    \item Suppose $m$ is even, and let $\wh{\tau}$ be any spin lift of $\tau$. Then:
    \begin{enumerate}
        \item If $W^{\tau^{m/2}}=\emptyset$, then $\wh{\tau}$ is even $\iff$ the quotient manifold $M/\tau^{m/2}$ admits a spin structure $\frak{t}'$ such that $\frak{t}'$ pulls back to $\frak{t}$ under the regular two-fold covering $\pi:W\to W/\tau^{m/2}$.
        \item If $W^{\tau^{m/2}}\neq\emptyset$, then:
        \begin{enumerate}
            \item $\wh{\tau}$ is even $\iff \dim(W^{\tau^{m/2}})=0$.
            \item $\wh{\tau}$ is odd $\iff \dim(W^{\tau^{m/2}})=2$.
        \end{enumerate}
    \end{enumerate}
    \item If $m$ is odd, then $\tau$ admits precisely one even and one odd spin lift.
\end{enumerate}
\end{proposition}

\begin{remark}
The canonical example of an odd-type free involution	 on a spin 4-manifold is the involution on the K3 surface with quotient the Enriques surface, which does not admit a spin structure.
\end{remark}

We do not necessarily have to restrict ourselves to working with connected $\ZZ_{m}$-equivariant spin manifolds. Let $(M,\frak{s})$ be a spin $n$-manifold with $c$ connected components, and $\alpha:M\to M$ an $\frak{s}$-preserving self-diffeomorphism of order $m$ such that the orbit space of the $\ZZ_{m}$-action generated by $\alpha$ has $c'$ components. There are then precisely $2^{c'}$ possible spin lifts $\wh{\alpha}$ of $\alpha$, which essentially boil down to a choice of spin lift on each component of the orbit space. We then have the notion of the \emph{(generalized) parity} of a spin lift $\wh{\alpha}$, which is an assignment of either ``even'' or ``odd'' to each component of the orbit space. When disconnected manifolds arise, we will often restrict our attention to spin lifts of \emph{pure parity} (spin lifts which restrict to either all even or all odd spin lifts on the various components of $M$), as opposed to spin lifts of \emph{mixed parity} (those which restrict to a mixture of even and odd spin lifts on different components).

\bigskip
\subsection{Equivariant Cobordisms}
\label{subsec:equivariant_cobordisms}

For this section we will not assume our manifolds are necessarily connected. We begin with a few definitions:

\begin{definition}
Let $(Y,\frak{s},\wh{\sigma})$ be a $\ZZ_{m}$-equivariant spin 3-manifold.
\begin{itemize}
    \item We say that $(W,\frak{t},\wh{\tau})$ is a \emph{$\ZZ_{m}$-equivariant spin filling of $(Y,\frak{s},\wh{\sigma})$} if $(W,\frak{t},\wh{\tau})$ is a $\ZZ_{m}$-equivariant spin 4-manifold with boundary $\del W=Y$ such that $\frak{t}|_{Y}=\frak{s}$, and $\wh{\tau}|_{Y}=\wh{\sigma}$.
    \item We define the \emph{orientation reverse} of $(Y,\frak{s},\wh{\sigma})$ to be $-(Y,\frak{s},\wh{\sigma}):=(-Y,\frak{s},\wh{\sigma})$, where $-Y$ denotes the orientation reverse of $Y$. Here we conflate the spin structure $\frak{s}$ on $Y$ with its corresponding spin structure on $-Y$, and similarly for $\wh{\sigma}$.
\end{itemize}
Similarly, suppose $(Y_{0},\frak{s}_{0},\wh{\sigma}_{0})$ and  $(Y_{1},\frak{s}_{1},\wh{\sigma}_{1})$ are two $\ZZ_{m}$-equivariant spin 3-manifolds.
\begin{itemize}
    \item We define the \emph{disjoint union of $(Y_{0},\frak{s}_{0},\wh{\sigma}_{0})$ and $(Y_{1},\frak{s}_{1},\wh{\sigma}_{1})$} to be the $\ZZ_{m}$-equivariant spin 3-manifold
    \[(Y_{0},\frak{s}_{0},\wh{\sigma}_{0})\amalg(Y_{1},\frak{s}_{1},\wh{\sigma}_{1}):=(Y_{0}\amalg Y_{1},\frak{s}_{0}\amalg\frak{s}_{1},\wh{\sigma}_{0}\amalg\wh{\sigma}_{1}).\]
    \item A \emph{$\ZZ_{m}$-equivariant spin cobordism from $(Y_{0},\frak{s}_{0},\wh{\sigma}_{0})$ to $(Y_{1},\frak{s}_{1},\wh{\sigma}_{1})$} is a $\ZZ_{m}$-equivariant spin filling of $-(Y_{0},\frak{s}_{0},\wh{\sigma}_{0})\amalg(Y_{1},\frak{s}_{1},\wh{\sigma}_{1})$. We say that $(Y_{0},\frak{s}_{0},\wh{\sigma}_{0})$ and $(Y_{1},\frak{s}_{1},\wh{\sigma}_{1})$ are \emph{$\ZZ_{m}$-equivariantly spin cobordant} if there exists a $\ZZ_{m}$-equivariant spin cobordism between them.
\end{itemize}
\end{definition}

\begin{remark}
\label{remark:cobordism_parity}
Note that two \emph{connected} $\ZZ_{m}$-equivariant spin 3-manifolds $(Y_{0},\frak{s}_{0},\wh{\sigma}_{0})$, $(Y_{1},\frak{s}_{1},\wh{\sigma}_{1})$ are $\ZZ_{m}$-equivariantly spin cobordant only if the parities of $\wh{\sigma}_{0}$ and $\wh{\sigma}_{1}$ are equal. Furthermore, for any \emph{connected} equivariant cobordism $(W,\frak{t},\wh{\tau})$ between them, the parity of $\wh{\tau}$ must agree with the parities of $\wh{\sigma}_{0}$ and $\wh{\sigma}_{1}$.
\end{remark}

With these definitions in mind, we define the \emph{3-dimensional $\ZZ_{m}$-equivariant spin cobordism group} $\Omega^{\Spin,\ZZ_{m}}_{3}$ to be the set of $\ZZ_{m}$-equivariant spin 3-manifolds under the equivalence relation induced by $\ZZ_{m}$-equivariant spin cobordism, with addition given by disjoint union, identity given by the empty manifold $\emptyset$, and inverses given by orientation reversal. By Remark \ref{remark:cobordism_parity}, the group splits as a direct sum
\[\Omega^{\Spin,\ZZ_{m}}_{3}=\Omega^{\Spin,\ZZ_{m},\ev}_{3}\oplus\Omega^{\Spin,\ZZ_{m},\odd}_{3},\]
where $\Omega^{\Spin,\ZZ_{m},\ev}_{3}$, $\Omega^{\Spin,\ZZ_{m},\odd}_{3}$ denote  the subgroups generated by manifolds equipped with even and odd spin lifts, respectively. The rest of this section is devoted to proving the following proposition:

\begin{proposition}
\label{prop:equivariant_cobordism_group}
For each integer $m\geq 2$, the $\ZZ_{m}$-equivariant spin cobordism group $\Omega^{\Spin,\ZZ_{m}}_{3}$ is finite.
\end{proposition}

There are partial results in this direction (see \cite{Farsi}). The fact that we allow non-empty fixed-point sets in our definition of $\Omega^{\Spin,\ZZ_{m}}_{3}$ makes the issue of calculating this group explicitly a subtle one. However since we only wish to show that these groups are finite, we will take a more ad-hoc approach.

For our first step, we will show that any $\ZZ_{m}$-equivariant spin 3-manifold is $\ZZ_{m}$-equivariantly spin cobordant to one where the action is free. Indeed, let $Y$ be a $\ZZ_{m}$-equivariant spin 3-manifold, and let $L=\bigcup_{k=1}^{m-1}Y^{\sigma^{k}}$ be the union of the fixed-point sets of $\sigma^{k}$, $1\le k\le m-1$. For each component $K_{i}\subset L$ let $d_{i}|m$ be the minimal such divisor such that $K_{i}\subset Y^{\sigma^{d_{i}}}$, and let $\phi_{K_{i}}:\nu(K_{i})\xrightarrow{\cong}S^{1}\times D^{2}$ be an identification of an equivariant tubular neighborhood of $K_{i}$ with a fixed solid torus such that $\sigma^{d}$ acts on $\nu(K_{i})\cong S^{1}\times D^{2}$ via the identity on the $S^{1}$ factor and multiplication by $\omega_{d_{i}}^{k_{i}}=e^{2\pi i k_{i}/d_{i}}$ on the $D^{2}$ factor for some $1\le k_{i}\le d_{i}-1$, $(k_{i},d_{i})=1$. We can choose these framings coherently so that for each $1\le k\le m-1$, the projection of the action of $\sigma^{k}$ on $\nu(L)\cong\sqcup_{i\in I}S^{1}\times D^{2}$ onto $\sqcup_{i\in I}S^{1}$ permutes some subset of components, and fixes the rest point-wise -- denote these functions by $\ol{\sigma}^{k}$.

Attach a 0-framed 4-dimensional 2-handle $H_{i}$ to each $K_{i}$ via the identification $\phi_{i}$. By our choice of framings, we can extend $\sigma$ over $\cup_{i\in I}H_{i}\cong\cup_{i\in I}D^{2}\times D^{2}$ via the obvious extension of $\ol{\sigma}$ to $\cup_{i\in I}D^{2}$. By inspection, the group action preserves all of the spin structures involved. We then see that this extension of $\sigma$ over the $H_{i}$ produces a $\ZZ_{m}$-equivariant spin cobordism from $Y$ to the manifold $Y_{0}$ obtained by $0$-surgery on each component of $L$ with respect to the framings $\{\phi_{i}\}$, such that $\ZZ_{m}$ acts freely on $Y_{0}$. 

Now let $\Omega_{3}^{\Spin,\ZZ_{m},\text{free}}$ denote the spin cobordism group consisting of \emph{free} spin $\ZZ_{m}$-equivariant 3-manifolds and free spin equivariant cobordisms between them. By the above observation, it suffices to show the following:

\begin{proposition}
\label{prop:equivariant_cobordism_group_free}
For each integer $m\geq 2$, the free $\ZZ_{m}$-equivariant spin cobordism group $\Omega^{\Spin,\ZZ_{m},\text{free}}_{3}$ is finite.
\end{proposition}

\begin{proof}
Note that $\Omega^{\Spin,\ZZ_{m},\text{free}}_{3}$ splits as
\[\Omega^{\Spin,\ZZ_{m},\text{free}}_{3}=\Omega^{\Spin,\ZZ_{m},\text{free},\ev}_{3}\oplus\Omega^{\Spin,\ZZ_{m},\text{free},\odd}_{3}.\]
We treat the even case first. An argument of
Conner and Floyd (\cite{ConFlo79}) implies that $\Omega^{\Spin,\ZZ_{m},\text{free}}_{3}\cong\Omega^{\Spin}_{3}(B\ZZ_{m})$. There exists an associated Atiyah-Hirzebruch spectral sequence which takes the form 
\[E^{p,q}_{2}=H^{p}(B\ZZ_{m};\Omega^{\Spin}_{q})\Longrightarrow\Omega^{\Spin}_{p+q}(B\ZZ_{m})=E^{p,q}_{\infty}.\]
Since the $E_{2}$-terms are non-torsion if and only if $(p,q)=(0,0)$, we see at once that $\Omega^{\Spin}_{3}(B\ZZ_{m})$ is torsion, as desired. We leave the odd case as an exercise to the reader.
\end{proof}

\bigskip
\subsection{Equivariant Connected Sums}
\label{subsec:equivariant_connected_sums}

In this section we will describe the conditions and extra data we need in order to define the equivariant connected sum of two $\ZZ_{m}$-equivariant spin 3- or 4-manifolds.

Let $M_{0}$, $M_{1}$ be two compact oriented 3- or 4-manifolds, let $\alpha_{j}:M_{j}\to M_{j}$ be orientation-preserving diffeomorphisms of order $m$ for $j=0,1$, and suppose that $M_{j}^{\alpha_{j}}\neq \emptyset$ for each $j$. Choose basepoints $x_{j}\in M_{j}^{\alpha_{j}}\setminus\del M_{j}^{\alpha_{j}}$, and let $C_{j}\subset M_{j}^{\alpha_{j}}$ be the connected components containing $x_{j}$ for each $j=0,1$. If we choose orientations $\frak{o}(C_{j})$ of the $C_{j}$, using the orientations on the $M_{j}$ we get induced orientations on the vector spaces $N_{j}:=\nu(C_{j})|_{x_{j}}$, where $\nu(C_{j})$, $j=0,1$ denotes the normal bundle of $C_{j}$, which we can assume to be equivariant with respect to the $\ZZ_{m}$-action.

For any real $\ZZ_{m}$-representation $V$, let $D(V)$ denote the unit disk inside $V$ with boundary $S(V)$. Letting $\RR$ denote the trivial real $\ZZ_{m}$-representation of dimension 1, we can identify equivariant neighborhoods $\nu(x_{j})$ of the basepoints $x_{j}$ with $D(N_{j}\oplus\RR^{\dim(C_{j})})$ and $\del\nu(x_{j})$ with $S(N_{j}\oplus\RR^{\dim(C_{j})})$. Suppose there exists a \emph{orientation-reversing} isomorphism
\[\phi:S(N_{0}\oplus\RR^{\dim(C_{0})})\xrightarrow{\cong} S(N_{1}\oplus\RR^{\dim(C_{1})})\]
of $\ZZ_{m}$-equivariant representation spheres, which restricts to an orientation-reversing isomorphism
\[\phi_{\RR}:S(\RR^{\dim(C_{0})})\xrightarrow{\cong} S(\RR^{\dim(C_{1})}).\]
Letting
\begin{align*}
    &D_{1/2}(N_{j}\oplus\RR^{\dim(C_{j})}):=\{x\in N_{j}\oplus\RR^{\dim(C_{j})}:|x|\le 1/2\}, \\
    &S_{1/2}(N_{j}\oplus\RR^{\dim(C_{j})}):=\{x\in N_{j}\oplus\RR^{\dim(C_{j})}:|x|= 1/2\},
\end{align*}
as well as 
\[\phi_{1/2}:S_{1/2}(N_{0}\oplus\RR^{\dim(C_{0})})\xrightarrow{\cong} S_{1/2}(N_{1}\oplus\RR^{\dim(C_{1})})\]
the isomorphism induced by $\phi$, we can define
\[M_{0}\# M_{1}:= (M_{0}\setminus D_{1/2}(N_{0}\oplus\RR^{\dim(C_{0})}))\cup_{\phi_{1/2}}(M_{1}\setminus D_{1/2}(N_{1}\oplus\RR^{\dim(C_{1})}))\]
with $\ZZ_{m}$-action $\alpha^{\#}$ such that
\[\alpha^{\#}|_{M_{j}\setminus D_{1/2}(N_{j}\oplus\RR^{\dim(C_{j})})}=\alpha_{j}|_{M_{j}\setminus D_{1/2}(N_j\oplus\RR^{\dim(C_{j})})}\]
for each $j=0,1$. 

\begin{remark}
Note that up to $\ZZ_{m}$-equivariant diffeomorphism, the above equivariant connected sum construction depends only on the choices of components $C_{0}\subset M^{\alpha_{0}}$, $C_{1}\subset M^{\alpha_{1}}$ and the pair of orientations $\{\frak{o}(C_{0}),\frak{o}(C_{1})\}$ under the equivalence $\{\frak{o}(C_{0}),\frak{o}(C_{1})\}\equiv\{-\frak{o}(C_{0}),-\frak{o}(C_{1})\}$.
\end{remark}

Now suppose the $M_{j}$ are endowed with spin structures $\frak{s}_{j}$ which are preserved under $\alpha_{j}$ for each $j=0,1$, and for simplicity assume $M_{0}$, $M_{1}$ are connected. Then we also have an induced spin structure $\frak{s}^{\#}$ on $M_{0}\# M_{1}$, given by fixing a trivialization on the boundaries $\del(M_{j}\setminus D_{1/2}(N_{j}\oplus\RR^{\dim(C_{j})}))$, $j=0,1$, compatible with the map $\phi_{1/2}$ above. Given spin lifts $\wh{\alpha}_{j}$ of $\alpha_{j}$, $j=0,1$, precisely one of the spin lifts $\wh{\alpha}_{1}$ or $-\wh{\alpha}_{1}$ will glue up with $\wh{\alpha}_{0}$ to produce a globally-defined spin lift $\wh{\alpha}^{\#}$ of $\alpha^{\#}$. Therefore it only makes sense to define the connected sum
\[(M_{0}\# M_{1},\frak{s}^{\#},\alpha^{\#}):=(M_{0},\frak{s}_{0},\alpha_{0})\#(M_{1},\frak{s}_{1},\alpha_{1})\]
without fixing spin lifts, or alternatively, only fixing a spin lift of $\alpha_{0}$ or $\alpha_{1}$ but not both. The general case where $M_{0}$ and $M_{1}$ are possibly disconnected is similar, except if one fixes a spin lift $\wh{\alpha}_{0}$ of $\alpha_{0}$, then precisely half of the spin lifts of $\alpha_{1}$ will be compatible with $\wh{\alpha}_{0}$.

Next we analyze the possible cases that can arise, depending on the dimension of the manifolds and codimensions of the fixed-point sets.

Let $(Y_{j},\frak{s}_{j},\sigma_{j})$, $j=0,1$ be two $\ZZ_{m}$-equivariant spin 3-manifolds each with non-empty fixed-point set. Then the fixed-point sets are necessarily one-dimensional, and so $Y_{j}^{\sigma_{0}}=L_{j}$ for some links $L_{0}\subset Y_{0}$, $L_{1}\subset Y_{1}$. For each $j=0,1$ choose a link component $K_{j}\subset L_{j}$. Choosing basepoints $x_{j}\in K_{j}$ and orientations for the $K_{j}$ induces orientations on the two-dimensional real vector spaces $N_{j}=\nu(K_{j})|_{x_{j}}$. We can therefore identify $N_{0}$ and $N_{1}$ with one-dimensional complex $\ZZ_{m}$-representations, on which $d\sigma_{0}$, $d\sigma_{1}$ act by $e^{i\psi_{0}}$, $e^{i\psi_{1}}$, respectively, for some $\psi_{0},\psi_{1}\in[0,\pi)$. It follows that we can perform an equivariant connected sum along $x_{0},x_{1}$ if and only if $\psi_{0}=\psi_{1}$.

\begin{example}
\label{ex:connected_sum_branched_covers}
Let $K_{0},K_{1}\subset S^{3}$ be oriented knots, and let $\Sigma_{m}(K_{0}),\Sigma_{m}(K_{1})$ denote their corresponding $m$-fold cyclic branched covers. Let $\sigma_{0}$, $\sigma_{1}$, $\sigma$ denote the generators of the $\ZZ_{m}$-covering transformations $\Sigma_{m}(K_{0})$, $\Sigma_{m}(K_{1})$, and $\Sigma_{m}(K_{0}\# K_{1})$, respectively. Then for any $\ZZ_{m}$-invariant spin structures $\frak{s}_{0},\frak{s}_{1}, \frak{s}$ on $\Sigma_{m}(K_{0})$, $\Sigma_{m}(K_{1})$, and $\Sigma_{m}(K_{0}\# K_{1})$, respectively, there exists an (orientation-preserving) $\ZZ_{m}$-equivariant spin diffeomorphism
\[(\Sigma_{m}(K_{0}\# K_{1}),\frak{s},\sigma)\cong(\Sigma_{m}(K_{0})\#\Sigma_{m}(K_{1}),\frak{s}^{\#},\sigma^{\#}).\]
\end{example}

Now suppose $(W_{j},\frak{t}_{j},\tau_{j})$, $j=0,1$ are two $\ZZ_{m}$-equivariant spin 4-manifolds each with non-empty fixed-point set.
In general, the fixed-point set $W^{\tau_{j}}$ consists of a disjoint collection of points $p_{j,1},\dots,p_{j,r_{j}}$ and surfaces $\Sigma_{j,1},\dots,\Sigma_{j,s_{j}}$ for $j=0,1$. 

To take an equivariant connect sum along surface components of the fixed point sets, it suffices to choose components $\Sigma_{0,k_{0}}\in W^{\tau_{0}}$ and $\Sigma_{1,k_{1}}\in W^{\tau_{1}}$ and orientations on the $\Sigma_{j,k_{j}}$ for some $1\le k_{j}\le s_{j}$, $j=0,1$. As in the 3-dimensional case, by choosing basepoints $x_{j}\in \Sigma_{j,k_{j}}$ we can identify $N_{j}:=\nu(\Sigma_{j,k_{j}})|_{x_{j}}$ for $j=0,1$ with one-dimensional complex $\ZZ_{m}$-representations, whose actions are given by a set of angles $\psi_{0},\psi_{1}\in (0,\pi]$, and that we can perform an equivariant connected sum along $\Sigma_{0,k_{0}},\Sigma_{1,k_{1}}$ if and only if $\psi_{0}=\psi_{1}$.

To take an equivariant connect sum along two isolated fixed points: choose $p_{0,k_{0}}\in W^{\tau_{0}}$ and $p_{1,k_{1}}\in W^{\tau_{1}}$. We can identify $N_{j}:=\nu(p_{j,k_{j}})|_{x_{j}}$ for $j=0,1$ with two-dimensional complex $\ZZ_{m}$-representations, whose actions are determined by tuples of angles $(\alpha_{j},\beta_{j})\in(\RR/2\pi\ZZ)^{2}/\sim$, subject to the relations $(\alpha,\beta)\equiv(\beta,\alpha)$ and $(\alpha,\beta)\equiv(-\alpha,-\beta)$. In order to perform an equivariant connected sum, it is necessary and sufficient to have $(\alpha_{0},\beta_{0})\equiv(-\alpha_{1},\beta_{1})$.

One can also generalize the above construction to define the \emph{$\ZZ_{m}$-equivariant boundary connnected sum} of two $\ZZ_{m}$-equivariant spin 4-manifolds $(W_{0},\frak{t}_{0},\tau_{0})$, $(W_{1},\frak{t}_{1},\tau_{1})$ with boundaries $(Y_{0},\frak{s}_{0},\sigma_{0})$ and $(Y_{1},\frak{s}_{1},\sigma_{1})$, respectively, assuming that $Y_{0}^{\sigma_{0}}$ and $Y_{1}^{\sigma_{1}}$ are non-empty. We will usually denote such an equvariant boundary connect sum by $(W,\frak{t}^{\natural},\tau^{\natural})$, which again depends on choices of components of $Y_{0}^{\sigma_{0}}$ and $Y_{1}^{\sigma_{1}}$, as well as orientations of those components.

\begin{example}
\label{ex:connected_sum_branched_covers_cobordism}
Let $K_{0},K_{1}\subset S^{3}$ be oriented knots, and let $(\Sigma_{m}(K_{0}),\frak{s}_{0},\sigma_{0})$, $(\Sigma_{m}(K_{1}),\frak{s}_{1},\sigma_{1})$, and $(\Sigma_{m}(K_{0}\# K_{1}),\frak{s},\sigma)$ be as in Example \ref{ex:connected_sum_branched_covers}. Let $W'$ be the cylinder
\[W'=(\Sigma_{m}(K_{0})\amalg\Sigma_{m}(K_{1}))\times[0,1],\]
and let $(W,\frak{t}^{\natural},\tau^{\natural})$ be the $\ZZ_{m}$-equivariant spin 4-manifold obtained from $W'$ by taking the $\ZZ_{m}$-equivariant boundary connected sum along
\[(\Sigma_{m}(K_{0})\amalg\Sigma_{m}(K_{1}))\times\{1\}\subset\del W'.\]
Then $(W,\frak{t}^{\natural},\tau^{\natural})$ furnishes a $\ZZ_{m}$ equivariant spin cobordism from the disjoint union of $(\Sigma_{m}(K_{0}),\frak{s}_{0},\sigma_{0})$ and $(\Sigma_{m}(K_{1}),\frak{s}_{1},\sigma_{1})$ to the equivariant connected sum $(\Sigma_{m}(K_{0}\# K_{1}),\frak{s},\sigma)$.

More generally if $(Y_{0},\frak{s}_{0},\sigma_{0})$ and $(Y_{1},\frak{s}_{1},\sigma_{1})$ are $\ZZ_{m}$-equivariant spin 3-manifolds such that their connected sum $(Y_{0}\# Y_{1},\frak{s}^{\#},\sigma^{\#})$ is well-defined, a similar construction furnishes a $\ZZ_{m}$-equivariant spin cobordism from $(Y_{0},\frak{s}_{0},\sigma_{0})\sqcup(Y_{1},\frak{s}_{1},\sigma_{1})$ to $(Y_{0}\# Y_{1},\frak{s}^{\#},\sigma^{\#})$. If $Y_{0}$, $Y_{1}$ are rational homology spheres, then the aforementioned homology cobordism is a $\ZZ_{m}$-equivariant spin rational homology cobordism.
\end{example}

%% file: k_theory.tex
\section{\texorpdfstring{$G^{*}_{m}$}{G*m}-Equivariant K-Theory}
\label{sec:k_theory}

\bigskip
\subsection{Review of Equivariant K-Theory}
\label{subsec:equivariant_k_theory}

We start by reviewing some general facts about equivariant K-theory --- see \cite{Segal68} for more details.

Let $G$ be a compact topological group and let $X$ be a compact $G$-space. The group $K_{G}(X)$, called the \emph{equivariant (complex) K-theory} of $X$, is defined to be the Grothendieck group associated to $G$-equivariant complex vector bundles on $X$. When $X$ is a point, $R(G)=K_{G}(\pt)$ is the complex representation ring of $G$, and in general $K_{G}(X)$ is an algebra over $R(G)$.

\begin{fact}
\label{fact:continuous_map}
Any continuous map $f:X\to X'$ induces a map $f^{*}:K_{G}(X')\to K_{G}(X)$.
\end{fact}

\begin{fact}
\label{fact:subgroup}
For every subgroup $H\subset G$, there exists a functorial restriction map $\res^{G}_{H}:K_{G}(X)\to K_{H}(X)$.
\end{fact}

\begin{fact}
\label{fact:induction}
For every closed subgroup $H\subset G$ of finite index, there exists a functorial induction map $\ind^{G}_{H}:K_{H}(X)\to K_{G}(X)$, which for $X=\pt$ agrees with the usual induction map on representations.
\end{fact}

\begin{fact}
\label{fact:free_G_action}
If $X$ is a free $G$-space, then the pull-back map $\pi^{*}:K(X/G)\to K_{G}(X)$ is a ring homomorphism. More generally if $X$ is a free $G$-space and $H\subset G$ is a closed subgroup, then $K_{G}(X/H)\cong K_{H}(X/G)$.
\end{fact}

\begin{fact}
\label{fact:trivial_G_action}
If $G$ acts trivially on $X$, then the natural map $R(G)\otimes_{\ZZ} K(X)\to K_{G}(X)$ is an isomorphism of $R(G)$-algebras. More generally if $N\subset G$ is a closed normal subgroup such that $N\subset G$ acts trivially on $X$ and that the conjugation action of $G$ on $N$ is trivial, then there exists an isomorphism $R(N)\otimes_{\ZZ} K_{G/N}(X)\cong K_{G}(X)$.
\end{fact}

\begin{fact}
\label{fact:homogenous_spaces}
From Fact \ref{fact:free_G_action}, we have a ring  isomorphism $K_{G}(G)\cong K(\pt)\cong\ZZ$. More generally, if $H\subset G$ is a closed subgroup then $K_{G}(G/H)\cong R(H)$.
\end{fact}

Now suppose $X$ has a distinguished base point $\ast\in X$ which is fixed under the $G$-action. We define the \emph{reduced (complex) equivariant K-theory} of $X$, denoted $\wt{K}_{G}(X)$, to be the kernel of the map $i^{*}:K_{G}(X)\to K_{G}(\ast)$ induced by the inclusion $i:\ast\hookrightarrow X$.

\begin{fact}
\label{fact:free_G_action_base_point}
If the action of $G$ on $X$ is free away from the basepoint, then the pull-back map $\wt{K}(X/G)\to\wt{K}_{G}(X)$ is a ring isomorphism. More generally, if $G=G_{1}\oplus G_{2}$ and $G_{1}<G$ acts freely on $X$ away from $\ast\in X$, then we have a ring isomorphism $\wt{K}_{G_{2}}(X/G_{1})\cong\wt{K}_{G}(X)$.
\end{fact}

\begin{fact}
\label{fact:product_map}
There is a natural product map $\otimes:\wt{K}_{G}(X)\otimes\wt{K}_{G}(X')\to\wt{K}_{G}(X\wedge X')$.
\end{fact}

For any real $G$ representation $V$, we denote by $\Sigma^{V}X:=V^{+}\wedge X$ the (reduced) suspension of $X$ by $V$, with its induced $G$-action. If $V=n\RR=\RR^{n}$ is a trivial representation, we simply write $\Sigma^{n}X$ for $\Sigma^{n\RR}X$.

\begin{fact}
\label{fact:bott_isomorphism}
Suppose $V$ is a complex $G$-representation. Then there exists a functorial \emph{equivariant Bott periodicity} isomorphism $\wt{K}_{G}(X)\cong\wt{K}_{G}(\Sigma^{V}X)$, given by multiplication with a Bott class $b_{V}\in\wt{K}_{G}(V^{+})$ under the product map
\[\otimes:\wt{K}_{G}(V^{+})\otimes\wt{K}_{G}(X)\to\wt{K}_{G}(\Sigma^{V}X).\]
\end{fact}

\begin{fact}
\label{fact:euler_class}
Let $V$ be a complex representation. Then the composition of the Bott isomorphism $\wt{K}_{G}(X)\cong\wt{K}_{G}(\Sigma^{V}X)$ with the map $\wt{K}_{G}(\Sigma^{V}X)\to\wt{K}_{G}(X)$ induced by the inclusion $X\hookrightarrow\Sigma^{V}X$ is a map $\wt{K}_{G}(X)\to\wt{K}_{G}(X)$ given by multiplication with the \emph{K-theoretic Euler class}
\[\lambda_{-1}(V)=\sum_{k}(-1)^{k}[\Lambda^{k}(V)]\in R(G).\]
\end{fact}

Bott periodicity applied to the trivial complex representation $\CC\cong\RR^{2}$ gives an isomorphism $\wt{K}_{G}(X)\cong\wt{K}_{G}(\Sigma^{2}X)$. Hence for any $i\in\ZZ$, we can define the \emph{reduced equivariant K-cohomology groups} of $X$ by
\[\wt{K}_{G}^{i}(X):=\twopartdef{\wt{K}_{G}(X)}{i\text{ is even},}{\wt{K}_{G}(\Sigma X)}{i\text{ is odd}.}\]

\begin{fact}
\label{fact:long_exact_sequence}
If $A\subset X$ is a closed $G$-subspace containing $\ast\in X$, there is a long exact sequence:
\begin{equation}
\label{eq:long_exact_sequence_k_theory}
    \cdots\to\wt{K}^{i}_{G}(X\sqcup_{A}CA)\to\wt{K}_{G}^{i}(X)\to\wt{K}_{G}^{i}(A)\to\wt{K}_{G}^{i+1}(X\sqcup_{A}CA)\to\cdots
\end{equation}
where $CA$ denotes the (reduced) cone on $A$.
\end{fact}

\begin{fact}
\label{fact:wedge_sum}
There exists an isomorphism $\wt{K}_{G}(A\vee B)\cong\wt{K}_{G}(A)\oplus\wt{K}_{G}(B)$.
\end{fact}

\begin{definition}
\label{def:augmentation_ideal}
The \emph{augmentation ideal} $\frak{a}\subset R(G)$ is defined to be the kernel of the forgetful map (augmentation homomorphism) $R(G)\cong K_{G}(\pt)\to K(\pt)\cong\ZZ$, i.e., $\frak{a}$ consists of those virtual representations of (virtual) dimension $0$.
\end{definition}

\begin{fact}
\label{fact:nilpotent}
Suppose $X$ is a finite based $G$-CW complex and the $G$-action is free away from $\ast\in X$, then the elements of the augmentation ideal $\frak{a}\subset R(G)$ act nilpotently on $\wt{K}_{G}(X)\cong\wt{K}(X/G)$.
\end{fact}

One can also define the equivariant $K$-groups when $X$ is only locally compact, e.g., for the classifying bundle $EG$.

\begin{fact}
\label{fact:classifying_space_1}
The ring $K_{G}(EG)\cong K(BG)$ is isomorphic to $R(G)_{\frak{a}}^{\wedge}$, the completion of $R(G)$ at the augmentation ideal $\frak{a}$. The projection $EG\to\pt$ induces a map $K_{G}(\pt)\to K_{G}(EG)$, which corresponds to the natural map $R(G)\to R(G)_{\frak{a}}^{\wedge}$.
\end{fact}

\begin{fact}
\label{fact:classifying_space_2}
Let $X$ be a compact space with a free $G$-action, let $Q=X/G$, and let $\pi$ denote the projection $X\to\pt$. The induced map $\pi^{*}:K_{G}(\pt)\to K_{G}(X)$ can be identified with the composition
\[K_{G}(\pt)\cong R(G)\to R(G)_{\frak{a}}^{\wedge} \cong K(BG)\to K(Q)\]
where the map $K(BG)\to K(Q)$ is the one induced by the classifying map $Q\to BG$ for $X$.
\end{fact}

\bigskip
\subsection{Review of Pin(2)-Equivariant K-Theory}
\label{subsec:review_pin2_equivariant_k_theory}

If $\HH=\CC\oplus j\CC$ denotes the quaternions, recall that the group $\Pin(2)$ is defined to be $\Pin(2)=S^{1}\cup jS^{1}\subset\HH$. There is a short exact sequence
\[1\to S^{1}\to \Pin(2)\to \ZZ_{2}\to 1.\]
As in \cite{Man14}, we introduce notation for the following real representations of $\Pin(2)$:
\begin{itemize}
    \item the trivial representation $\RR$.
    \item the one-dimensional representation $\wt{\RR}$ on which $S^{1}\subset\Pin(2)$ acts trivially, and $j\in\Pin(2)$ acts by multiplication by $-1$.
    \item The quaternions $\HH$, acted on by $\Pin(2)$ via left multiplication.
\end{itemize}
Denote by $\wt{\CC}$ the complexification $\wt{\RR}\otimes_{\RR}\CC$, which is isomorphic to $\wt{\RR}^{2}$ as a real representation. Then the complex representation ring $R(\Pin(2))$ is generated by $\wt{c}=[\wt{\CC}]$ and $h=[\HH]$, subject to the relations $\wt{c}^{2}=1$ and $\wt{c}h=h$. In other words, we have the following presentation of $R(\Pin(2))$:
\[R(\Pin(2))=\ZZ[\tilde{c},h]/(\wt{c}^{2}-1,\tilde{c}h-h)\]
As in \cite{Man14}, we can make a change of basis as follows. Let $V$ be a complex $\Pin(2)$-representation, and consider the composition
\[\wt{K}_{\Pin(2)}(X)\to\wt{K}_{\Pin(2)}(\Sigma^{V}X)\to\wt{K}_{\Pin(2)}(X)\]
where the first map is the equivariant Bott periodicity isomorphism, and the second map is induced by the canonical inclusion $X\hookrightarrow\Sigma^{V}X$. Then this composition is given by multiplication with the $K$-theoretic Euler class
\[\lambda_{-1}(V):=\sum_{d=0}^{\text{dim}(V)}(-1)^{d}[\Lambda^{d}V]\in R(\Pin(2)).\]
Then under the coordinate change
\begin{align*}
    &w=\lambda_{-1}(\wt{c})=1-\tilde{c} & &z=\lambda_{-1}(h)=2-h
\end{align*}
we have the following alternate presentation of $R(\Pin(2))$ from \cite{Man14}:
\[R(\Pin(2))=\ZZ[w,z]/(w^{2}-2w,zw-2w).\]
%

\subsection{The Representation Ring \texorpdfstring{$R(G^{*}_{m})$}{R(G*m)}}
\label{subsec:complex_representation_ring}

As in the introduction, let $m\geq 2$ be a positive integer, and consider the groups
\begin{align*}
    &G^{\ev}_{m}=\Pin(2)\times\ZZ_{m} & &G^{\odd}_{m}=\Pin(2)\times_{\ZZ_{2}}\ZZ_{2m}
\end{align*}
where 
\[\Pin(2)\times_{\ZZ_{2}}\ZZ_{2m}=(\Pin(2)\times\ZZ_{2m})/\<(-1,\mu^{m})\>\]
denotes the quotient obtained by modding out the diagonal $\ZZ_{2}$-subgroup, and where $\mu$ is a fixed generator of $\ZZ_{2m}$. 

We will first compute the representation ring of $G^{\ev}_{m}$. Let $\gamma$ be a fixed generator of $\ZZ_{m}$, and let $\CC_{k}$ be the one-dimensional complex representation on which $\gamma$ acts by $\omega_{m}^{k}:=e^{2\pi ik/m}$ for $j=0,\dots,m-1$. Then the representation ring $R(\ZZ_{m})$ is generated by $\zeta:=[\CC_{1}]$, and
\[R(\ZZ_{m})=\ZZ[\zeta]/(\zeta^{m}-1).\]
It follows that
\[R(G^{\ev}_{m})\cong R(\Pin(2))\otimes R(\ZZ_{m})\cong \ZZ[\wt{c},h,\zeta]/(\wt{c}^{2}-1,\wt{c}h-h,\zeta^{m}-1).\]
We introduce notation for the following complex representations of $G^{\ev}_{m}=\Pin(2)\times\ZZ_{m}$:
\begin{itemize}
    \item the one-dimensional representations $\wt{\CC}_{k}:=\wt{\CC}\otimes\CC_{k}$ with $\zeta^{k}\wt{c}=[\wt{\CC}_{k}]$.
    \item the two-dimensional representations $\HH_{k}:=\HH\otimes_{\CC}\CC_{k}$ with $\zeta^{k}h=[\HH_{k}]$.
\end{itemize}
Write $R(\ZZ_{m})_{\geq 0}$ to denote the set of all elements $\mbfs=\sum_{k=0}^{m-1}s_{k}\zeta^{k}\in R(\ZZ_{m})$ with $s_{k}\geq 0$ for $k=0,\dots,m-1$. Given such an element $\mbfs\in R(\ZZ_{m})_{\geq 0}$, we will often use $\mbfs\wt{\CC}$ and $\mbfs\wt{\HH}$ to denote the representations
\begin{align*}
    &\mbfs\wt{\CC}:=\bigoplus_{k=0}^{m-1}\wt{\CC}_{k}^{s_{k}}, & &\mbfs\wt{\HH}:=\bigoplus_{k=0}^{m-1}\HH_{k}^{s_{k}}.
\end{align*}
We define the following variables, similar to the discussion above:
\begin{align*}
    &w_{k}:=\lambda_{-1}(\zeta^{k}\wt{c})=1-\zeta^{k}\wt{c} & &z_{k}:=\lambda_{-1}(\zeta^{k}h)=1-\zeta^{k}h+\zeta^{2k}
\end{align*}
for $k=0,\dots,m-1$. Note that $w_{0}=w$, and $z_{0}=z$. These variables satisfy the following relations:
\begin{equation}
\label{eq:G_ev_m_relations}
    \begin{aligned}
    &w_{0}^{2}=2w_{0}, & &\qquad\qquad(1-w_{k})(1-w_{\ell})=(1-w_{0})(1-w_{k+\ell}), \\
    &w_{0}z_{k}=w_{0}w_{2k}, & &\qquad\qquad z_{k}=1-(1-w_{0})(1-w_{k})(2-z_{0})+(1-w_{0})(1-w_{2k}).
    \end{aligned}
\end{equation}
\ignorespacesafterend
Here, we use the cyclic indexing convention $w_{am+k}=w_{k}$ for $a\in\ZZ$, and similarly for $z_{k}$.

\begin{lemma}
Let $\I_{m}^{\ev}$ be the ideal generated by the relations in (\ref{eq:G_ev_m_relations}). Then there exists an isomorphism of rings
\[\ZZ[w_{0},\dots,w_{m-1},z_{0},\dots,z_{m-1}]/\I_{m}^{\ev}\cong \ZZ[\wt{c},h,\zeta]/(\wt{c}^{2}-1,\wt{c}h-h,\zeta^{m}-1).\]
\end{lemma}

\begin{proof}
Let $I=\I_{m}^{\ev}\subset\ZZ[w_{0},\dots,w_{m-1},z_{0},\dots,z_{m-1}]$, let $J=(\wt{c}^{2}-1,\wt{c}h-h,\zeta^{m}-1)\subset\ZZ[\wt{c},h,\zeta]$, and let
\[f:\ZZ[w_{0},\dots,w_{m-1},z_{0},\dots,z_{m-1}]\to\ZZ[\wt{c},h,\zeta]\]
be the map which sends $w_{k}\mapsto 1-\zeta^{k}\wt{c}$ and $z_{k}\mapsto 1-\zeta^{k}h+\zeta^{2k}$.
Then:
\begin{align*}
    &f(w_{0}^{2}-2w_{0})=(1-\wt{c})^{2}-2(1-\wt{c})=1-2\wt{c}+\wt{c}^{2}-2+2\wt{c}=\wt{c}^{2}-1\in J, \\
    &f(w_{k}w_{\ell}-w_{k}-w_{\ell}+w_{0}+w_{k+\ell}-w_{0}w_{k+\ell}) \\
    &\qquad=(1-\zeta^{k}\wt{c})(1-\zeta^{\ell}\wt{c})-(1-\zeta^{k}\wt{c})-(1-\zeta^{\ell}\wt{c})+(1-\wt{c})+(1-\zeta^{k+\ell}\wt{c})-(1-\wt{c})(1-\zeta^{k+\ell}\wt{c}) \\
    &\qquad=(\zeta^{k}+\zeta^{\ell}-\zeta^{k+\ell}-1)\wt{c}-(\zeta^{k}+\zeta^{\ell}-\zeta^{k+\ell}-1)\wt{c}+\zeta^{k+\ell}\wt{c}^{2}-\zeta^{k+\ell}\wt{c}^{2}=0\in J, \\
    &f(w_{0}z_{k}-w_{0}w_{2k})=(1-\wt{c})(1-\zeta^{k}h+\zeta^{2k})-(1-\wt{c})(1-\zeta^{2k}\wt{c}) \\
    &\qquad=1-\zeta^{k}h+\zeta^{2k}-\wt{c}+\zeta^{k}\wt{c}h-\zeta^{2k}\wt{c}-1+\wt{c}+\zeta^{2k}\wt{c}-\zeta^{2k}\wt{c}^{2} \\
    &\qquad=\zeta^{k}(\wt{c}h-h)+\zeta^{2k}(1-\wt{c}^{2})\in J, \\
    &f(z_{k}+(1-w_{0})((1-w_{k})(2-z_{0})+(w_{2k}-1))-1)=1-\zeta^{k}h+\zeta^{2k}+\wt{c}(\zeta^{k}\wt{c}h-\zeta^{2k}\wt{c})-1 \\
    &\qquad=(\wt{c}^{2}-1)(\zeta^{k}h-\zeta^{2k})\in J,
\end{align*}
and hence $I\subset f^{-1}(J)$. Therefore $f$ descends to a map
\[\ol{f}:\ZZ[w_{0},\dots,w_{m-1},z_{0},\dots,z_{m-1}]/I\to\ZZ[\wt{c},h,\zeta]/J.\]
Furthermore, we see that $\ol{f}$ is surjective since $\ol{f}(1-w_{0})=\wt{c}$, $\ol{f}(2-z_{0})=h$, and $\ol{f}((1-w_{0})(1-w_{k}))=\wt{c}^{2}\zeta^{k}=\zeta^{k}$.

Next, consider the map
\[g:\ZZ[\wt{c},h,\zeta]\to\ZZ[w_{0},\dots,w_{m-1},z_{0},\dots,z_{m-1}]\]
which sends $\wt{c}\mapsto 1-w_{0}$, $h\mapsto 2-z_{0}$, $\zeta\mapsto (1-w_{0})(1-w_{1})$. Then:
\begin{align*}
    &g(\wt{c}^{2}-1)=(1-w_{0})^{2}-1 = 1-2w_{0}+w_{0}^{2}-1 = 2w_{0}-w_{0}^{2}\in I, \\
    &g(\wt{c}h-h)=(1-w_{0})(2-z_{0})-(2-z_{0})=2-z_{0}-2w_{0}+w_{0}z_{0}-2+z_{0}=w_{0}z_{0}-2w_{0} \\
    &\qquad=(w_{0}z_{0}-w_{0}^{2})+(w_{0}^{2}-2w_{0})\in I.
\end{align*}
To show that $g(\zeta^{m}-1)\in I$, we first show that
\[(1-w_{1})^{k}=(1-w_{0})^{k-1}(1-w_{k})+p_{k}(w_{0},\dots,w_{k})\]
for some $p_{k}(w_{0},\dots,w_{k})\in I$. By inspection it holds for $k=1$ with $p_{0}(w_{0})=0$, and assuming it holds for some $k\geq 1$, then
\[(1-w_{1})^{k+1}=(1-w_{0})^{k-1}(1-w_{1})(1-w_{k})+p_{k}(w_{0},\dots,w_{k})(1-w_{1})\]
\[=(1-w_{0})^{k-1}\Big((1-w_{0})(1-w_{k+1})+(1-w_{1})(1-w_{k})+(1-w_{0})(1-w_{k+1})\Big)+p_{k}(w_{0},\dots,w_{k})(1-w_{1})\]
\[=(1-w_{0})^{k}(1-w_{k+1})+(1-w_{0})^{k-1}\Big((1-w_{1})(1-w_{k})+(1-w_{0})(1-w_{k+1})\Big)+p_{k}(w_{0},\dots,w_{k})(1-w_{1})\]
\[=(1-w_{0})^{k}(1-w_{k+1})+p_{k+1}(w_{0},\dots,w_{k+1}),\]
where
\[p_{k+1}(w_{0},\dots,w_{k+1}):=(1-w_{0})^{k-1}\Big((1-w_{1})(1-w_{k})+(1-w_{0})(1-w_{k+1})\Big)+p_{k}(w_{0},\dots,w_{k})(1-w_{1}),\]
hence the claim is proved. In particular, we have that
\[(1-w_{1})^{m}=(1-w_{0})^{m}+p_{m}(w_{0},\dots,w_{m-1}),\]
where we use the cyclic indexing convention as above, so that $w_{m}=w_{0}$. Next, we show that for any $k\geq 0$,
\[(1-w_{0})^{2k}=1+q_{k}(w_{0})\]
for some $q_{k}(w_{0})\in I$. By inspection it holds for $k=0$ with $q_{0}(w_{0})=0$. Now suppose it holds for some $k\geq 0$. Then
\[(1-w_{0})^{2k+2}=(1-w_{0})^{2}+(1-w_{0})^{2}q_{k}(w_{0})\]
\[=1+w_{0}^{2}-2w_{0}+(1-w_{0})^{2}q_{k}(w_{0}) = 1 +q_{k+1}(w_{0}),\]
where
\[q_{k+1}(w_{0}):=w_{0}^{2}-2w_{0}+(1-w_{0})^{2}p_{k}(w_{0})\in I,\]
hence the claim is proved. 

With the above two claims in hand, we see that
\[g(\zeta^{m}-1)=(1-w_{0})^{m}(1-w_{1})^{m}-1=(1-w_{0})^{2m}-1+(1-w_{0})^{m}p_{m}(w_{0},\dots,w_{m-1})\]
\[=q_{m}(w_{0})+(1-w_{0})^{m}p_{m}(w_{0},\dots,w_{m-1})\in I,\]
hence $J\subset g^{-1}(I)$. It follows that $g$ descends to a map
\[\ol{g}:\ZZ[\wt{c},h,\zeta]/J\to\ZZ[w_{0},\dots,w_{m-1},z_{0},\dots,z_{m-1}]/I.\]
In fact, we see that $\ol{g}$ is surjective, since
\begin{align*}
    &\ol{g}(1-\zeta^{k}\wt{c})=1-(1-w_{0})^{k+1}(1-w_{1})^{k}=1-(1-w_{0})^{2k}(1-w_{k})=1-(1-w_{k})=w_{k}, \\
    &\ol{g}(1-\zeta^{k}h+\zeta^{2k})=1-(1-w_{0})^{k}(1-w_{1})^{k}(2-z_{0})+(1-w_{0})^{2k}(1-w_{1})^{2k} \\
    &\qquad=1-(1-w_{0})^{2k-1}(1-w_{k})(2-z_{0})+(1-w_{0})^{4k-1}(1-w_{2k}) \\
    &\qquad=1-(1-w_{0})(1-w_{k})(2-z_{0})+(1-w_{0})(1-w_{2k})=z_{k}
\end{align*}
for all $k=0,\dots,m-1$. Finally, by construction we see that $\ol{g}\ol{f}=\id$ and $\ol{f}\ol{g}=\id$. Thus the result follows.
\end{proof}

Next, we calculate the representation ring of $G^{\odd}_{m}$. Let $\ZZ_{2m}=\<\mu\>$, let $q\in\frac{1}{2}\ZZ$ be a \emph{half} integer, and let $\CC_{q}$ be the one-dimensional complex representation of $\ZZ_{2m}$ on which $\mu$ acts by $e^{2\pi iq/m}$. If we let $\xi:=[\CC_{1/2}]$, then we can write $R(\ZZ_{2m})=\ZZ[\xi]/(\xi^{2m}-1)$. As before, we have the following representations of $\Pin(2)\times\ZZ_{2m}$:
\begin{itemize}
    \item the 1-dimensional representations $\wt{\CC}_{q}:=\wt{\CC}\otimes\CC_{q}$ with $q\in\frac{1}{2}\ZZ$, $\xi^{2q}\wt{c}=[\wt{\CC}_{q}]$.
    \item the 2-dimensional representations $\HH_{q}:=\HH\otimes_{\CC}\CC_{q}$ with $q\in\frac{1}{2}\ZZ$, $\xi^{2q}h=[\HH_{q}]$.
\end{itemize}
Now note that a $\Pin(2)\times\ZZ_{2m}$ representation descends to a $G^{\odd}_{m}=\Pin(2)\times_{\ZZ_{2}}\ZZ_{2m}$ representation if and only if the actions of $-1\in\Pin(2)$ and $\mu^{m}\in\ZZ_{2m}$ coincide. In particular:
\begin{itemize}
    \item $\wt{\CC}_{q}$ descends to a $G^{\odd}_{m}$ representation if and only if $q\equiv 0\pmod{1}$, since in this case both $-1\in\Pin(2)$ and $\mu^{m}\in\ZZ_{2m}$ act trivially.
    \item $\HH_{q}$ descends to a $G^{\odd}_{m}$ representation if and only if $q\equiv \frac{1}{2}\pmod{1}$, since in this case both $-1\in\Pin(2)$ and $\mu^{m}\in\ZZ_{2m}$ both act by multiplication by $-1$.
\end{itemize}
It follows that
\[R(G^{\odd}_{m})\cong\ZZ[\wt{c},\xi h,\xi^{2}]/(\wt{c}^{2}-1,\wt{c}\xi h-\xi h,\xi^{2m}-1)\subset R(\Pin(2)\times\ZZ_{2m}).\]
It will be useful to introduce the following notation: define $R(\ZZ_{2m})^{\ev}$ to be the additive subgroup of elements $\mbfs=s_{k}\xi^{k}\in R(\ZZ_{2m})$ with $s_{k}=0$ for all $k$ odd, and similarly let $R(\ZZ_{2m})^{\odd}$ denote the additive subgroup of elements $\mbfs=s_{k}\xi^{k}\in R(\ZZ_{2m})$ with $s_{k}=0$ for all even $k$. Furthermore, define $R(\ZZ_{2m})^{\ev}_{\geq 0}=R(\ZZ_{2m})^{\ev}\cap R(\ZZ_{2m})_{\geq 0}$ and $R(\ZZ_{2m})^{\odd}_{\geq 0}=R(\ZZ_{2m})^{\odd}\cap R(\ZZ_{2m})_{\geq 0}$. Note that there is a canonical isomorphism $R(\ZZ_{m})\cong R(\ZZ_{2m})^{\ev}$ given by the correspndence $\zeta^{k}\mapsto\xi^{2k}$, and so we will oftentimes use this isomorphism freely and not distinguish between the two. Given elements 
\begin{align*}
    &\mbfs=\sum_{k=0}^{m-1}s_{k}\xi^{2k}\in R(\ZZ_{m})_{\geq 0}^{\ev} & &\mbft=\sum_{k=0}^{m-1}t_{k+1/2}\xi^{2k+1}\in R(\ZZ_{m})_{\geq 0}^{\odd},
\end{align*}
we will often use $\mbfs\wt{\CC}$ and $\mbft\wt{\HH}$ to denote the representations
\begin{align*}
    &\mbfs\wt{\CC}:=\bigoplus_{k=0}^{m-1}\wt{\CC}_{k}^{s_{k}}, & &\mbft\wt{\HH}:=\bigoplus_{k=0}^{m-1}\HH_{k+1/2}^{t_{k+1/2}}.
\end{align*}
We define the following variables, in analogy with the $G^{\ev}_{m}$ case:
\begin{align*}
    &w_{k}:=\lambda_{-1}(\xi^{2k}\wt{c})=1-\xi^{2k}\wt{c} & &z_{k+\frac{1}{2}}:=\lambda_{-1}(\xi^{2k+1}h)=1-\xi^{2k+1}h+\xi^{4k+2}
\end{align*}
for $k=0,\dots,m-1$. We leave the proof of the following lemma to the reader:

\begin{lemma}
Let 
\[\I_{m}^{\odd}\subset\ZZ[w_{0},\dots,w_{m-1},z_{\frac{1}{2}},z_{\frac{3}{2}},\dots,z_{m-\frac{1}{2}}]\]
be the ideal generated by the relations
\begin{align*}
    &w_{0}^{2}=2w_{0}, &
    &(1-w_{k})(1-w_{\ell})=(1-w_{0})(1-w_{k+\ell}), \\
    &w_{0}z_{k+\frac{1}{2}}=w_{0}w_{2k+1}, & &z_{k+\frac{1}{2}}=1-(1-w_{0})(1-w_{k})-(1-w_{0})^{2}(1-w_{1})(1-w_{k}) \\
    & & &\qquad\qquad+(1-w_{0})(1-w_{2k+1})+(1-w_{0})(1-w_{k})z_{\frac{1}{2}}.
\end{align*}
Then there exists an isomorphism of rings
\[R(G^{\odd}_{m})=\ZZ[\wt{c},\xi h, \xi^{2}]/(\wt{c}^{2}-1,\wt{c}\xi h-\xi h,\xi^{2m}-1)\cong \ZZ[w_{0},\dots,w_{m-1},z_{\frac{1}{2}},z_{\frac{3}{2}},\dots,z_{m-\frac{1}{2}}]/\I_{m}^{\odd}.\]
\end{lemma}

\begin{wrap}
We record here the restriction maps on representation rings associated to various subgroup inclusions:
\begin{itemize}
    \item Under the inclusion $\Pin(2)\subset G^{*}_{m}$, the corresponding restriction map is given by
    \begin{align*}
        \res^{G^{*}_{m}}_{\Pin(2)}:R(G^{*}_{m})&\to R(\Pin(2)) \\
        w_{k}&\mapsto w, \\
        z_{k}&\mapsto z.
    \end{align*}
    \item In the case where $\ast=\ev$, the restriction map corresponding to the subgroup $\ZZ_{m}\subset G^{\ev}_{m}$ is given by
    \begin{align*}
        \res^{G^{\ev}_{m}}_{\ZZ_{m}}:R(G^{\ev}_{m})&\to R(\ZZ_{m}) \\
        w_{k}&\mapsto 1-\zeta^{k}, \\
        z_{k}&\mapsto (1-\zeta^{k})^{2}.
    \end{align*}
    \item In the case where $\ast=\odd$, the restriction map corresponding to the subgroup $\ZZ_{2m}\subset G^{\odd}_{m}$ is given by
    \begin{align*}
        \res^{G^{\odd}_{m}}_{\ZZ_{2m}}:R(G^{\odd}_{m})&\to R(\ZZ_{2m}) \\
        w_{k}&\mapsto 1-\xi^{2k}, \\
        z_{k+\frac{1}{2}}&\mapsto (1-\xi^{2k+1})^{2}.
    \end{align*}
\end{itemize}
\end{wrap}

\bigskip
\subsection{The Representation Ring \texorpdfstring{$RO(\ZZ_{m})$}{RO(Zm)} and Real \texorpdfstring{$G^{*}_{m}$}{G*m}-representations}
\label{subsec:real_representation_ring}

We will also need to establish notation for certain \emph{real} $G^{*}_{m}$-representations for $\ast=\ev,\odd$. We first describe the ring $RO(\ZZ_{m})$. For the following, let $\gamma$ be a generator of $\ZZ_{m}$.
\begin{itemize}
    \item Let $\RR$ denote the trivial 1-dimensional representation.
    \item For $j=1,\dots,\lfloor\frac{m-1}{2}\rfloor$, let $\VV_{j}$ denote the irreducible 2-dimensional representation
    \[\gamma\mapsto\begin{pmatrix}
    \cos(\frac{2\pi k}{m}) & -\sin(\frac{2\pi k}{m}) \\
    \sin(\frac{2\pi k}{m}) & \cos(\frac{2\pi k}{m}) 
    \end{pmatrix}.\]
    \item If $m$ is even, let $\RR_{m/2}$ denote the 1-dimensional representation where $\gamma$ acts by multiplication by $-1$.
\end{itemize}
One can show that if $\rho:=[\RR_{m/2}]$, and $\nu_{k}=[\VV_{k}]$, then we have the following presentation of $RO(\ZZ_{m})$:
\[RO(\ZZ_{m})=\twopartdef{\ZZ[\rho,\nu_{1},\dots,\nu_{\frac{m}{2}-1}]/(\rho\nu_{k}-\nu_{\frac{m}{2}-k},\nu_{j}\nu_{k}-\nu_{j+k}-\nu_{j-k})}{m\text{ is even},}{\ZZ[\nu_{1},\dots,\nu_{\frac{m-1}{2}}]/(\nu_{j}\nu_{k}-\nu_{j+k}-\nu_{j-k})}{m\text{ is odd}.}\]
Here we use the indexing convention that if $\frac{m}{2}<k<m$, then $\nu_{k}:=\nu_{m-k}$, and if $am\le k<(a+1)m$ for $a\in\ZZ$, then $\nu_{k}:=\nu_{k-am}$.

Next we introduce notation for the following real representations of $G^{\ev}_{m}$:
\begin{itemize}
    \item the 1-dimensional representation $\wt{\RR}_{0}$, on which $S^{1},\ZZ_{m}\subset G^{\ev}_{m}$ act trivially, and $j$ acts by $-1$.
    \item the 2-dimensional representations $\wt{\VV}_{j}:=\wt{\RR}_{0}\otimes\VV_{j}$, $1\le j\le\lfloor\frac{m-1}{2}\rfloor$.
    \item the 1-dimensional representation $\wt{\RR}_{m/2}:=\wt{\RR}_{0}\otimes\RR_{m/2}$, if $m$ is even.
\end{itemize}
Write $RO(\ZZ_{m})_{\geq 0}$ to denote the set of all elements
\[\mbfr=r_{0}+(\sum_{j=0}^{\lfloor\frac{m-1}{2}\rfloor}r_{j}\nu_{j})+r_{m/2}\rho\in RO(\ZZ_{m})\]
with $r_{j}\geq 0$ for $j=0,\dots,m/2$. Given such an element $\mbfr\in RO(\ZZ_{m})_{\geq 0}$, we will often use $\mbfr\wt{\RR}$ to denote the real $G^{\ev}_{m}$-representation
\[\mbfr\wt{\RR}:=\wt{\RR}_{0}^{r_{0}}\oplus\Big(\bigoplus_{j=0}^{\lfloor\frac{m-1}{2}\rfloor}\wt{\VV}_{j}^{r_{j}}\Big)\oplus\wt{\RR}_{m/2}^{r_{m/2}}.\]
Next, consider the real representation ring 
\[RO(\ZZ_{2m})=\ZZ[\wh{\rho},\wh{\nu}_{1},\dots,\wh{\nu}_{m-1}]/(\wh{\rho}\wh{\nu}_{k}-\wh{\nu}_{m-k},\wh{\nu}_{j}\wh{\nu}_{k}-\wh{\nu}_{j+k}-\wh{\nu}_{j-k}).\]
Here we use the indexing convention that if $m<k<2m$, then $\wh{\nu}_{k}:=\wh{\nu}_{2m-k}$, and if $2am\le k<2(a+1)m$ for $a\in\ZZ$, then $\wh{\nu}_{k}:=\nu_{k-2am}$. We denote by $RO(\ZZ_{2m})^{\ev}$ the additive subgroup consisting of elements
\[\mbfr=r_{0}+\Big(\sum_{j=1}^{m-1}r_{j}\wh{\nu}_{j}\Big)+r_{m}\wh{\rho}\in RO(\ZZ_{2m})\]
such that $r_{j}=0$ for all $j$ odd. Again, one can see that there is a canonical isomorphism $RO(\ZZ_{m})\cong RO(\ZZ_{2m})^{\ev}$ given by the correspondence $\nu_{j}\mapsto\wh{\nu}_{2j}$, $\rho\mapsto\wh{\rho}$. As in the complex case, we will oftentimes not distinguish between $RO(\ZZ_{m})$ and $RO(\ZZ_{2m})$ unless necessary.

Similar to $G^{\ev}_{m}$, we have the real $G^{\odd}_{m}$-representations denoted by $\wt{\RR}_{0}$, $\wt{\VV}_{j}$ for all $1\le j\le\lfloor\frac{m-1}{2}\rfloor$, and $\wt{\RR}_{m/2}$ if $m$ is even. Let $RO(\ZZ_{2m})_{\geq 0}^{\ev}:=RO(\ZZ_{2m})^{\ev}\cap RO(\ZZ_{2m})_{\geq 0}$. Given an element 
\[\mbfr=r_{0}+\Big(\sum_{j=1}^{\lfloor\frac{m-1}{2}\rfloor}r_{j}\wh{\nu}_{2j}\Big)+r_{m/2}\wh{\rho}\in RO(\ZZ_{2m})_{\geq 0}^{\ev},\]
(with the convention that $r_{m/2}=0$ if $m$ is odd) we will often use $\mbfr\wt{\RR}$ to denote the real $G^{\odd}_{m}$-representation
\[\mbfr\wt{\RR}:=\wt{\RR}_{0}^{r_{0}}\oplus\Big(\bigoplus_{j=0}^{\lfloor\frac{m-1}{2}\rfloor}\wt{\VV}_{j}^{r_{j}}\Big)\oplus\wt{\RR}_{m/2}^{r_{m/2}}.\]
Finally, recall that for any group $G$, we have the \emph{complexification} map
\[c:RO(G)\to R(G)\]
which sends a real G-representation $V$ to its complexification $V_{\CC}=V\otimes\CC$. In the cases where $G=\ZZ_{m},\ZZ_{2m}$, one can show that
\begin{align*}
    &c(\nu_{k})=\zeta^{k}+\zeta^{m-k}, & &c(\wh{\nu}_{k})=\xi^{k}+\xi^{m-k}, \\ &c(\rho)=\zeta^{m/2}, & &c(\wh{\rho})=\xi^{m/2}.
\end{align*}
Note that the complexification map $c$ is functorial with respect to the canonical inclusions $RO(\ZZ_{m})\hookrightarrow R(G^{\ev}_{m})$ and $RO(\ZZ_{2m})^{\ev}\hookrightarrow R(G^{\odd}_{m})$. It follows that for any $\mbfr\in RO(\ZZ_{m})_{\geq 0}$ or $RO(\ZZ_{2m})_{\geq 0}^{\ev}$, we have that 
\[\mbfr\wt{\RR}\otimes\CC=c(\mbfr)\wt{\CC}\]
as complex $G^{*}_{m}$-representations. In particular, we have the following isomorphisms of complex $G^{*}_{m}$-representations: $\wt{\RR}_{0}\otimes\wt{\CC}\cong\CC_{0}$, $\wt{\VV}_{j}\otimes\CC\cong\wt{\CC}_{j}\oplus\wt{\CC}_{m-j}$ for all $j=1,\dots,\lfloor\frac{m-1}{2}\rfloor$, and $\wt{\RR}_{m/2}\otimes\CC\cong\wt{\CC}_{m/2}$ if $m$ is even.

\bigskip

%% file: k_invariants.tex
\section{Equivariant \texorpdfstring{$k$}{k}-invariants}
\label{sec:equivariant_k_invariants}

In this section we construct a package of \emph{equivariant $k$-invariants} associated to a special class of $G^{*}_{m}$-CW complexes we call \emph{spaces of type $\CC$-$G^{*}_{m}$-$\SWF$}, which are motivated by the construction of the $G^{*}_{m}$-equivariant Seiberg--Witten Floer stable homotopy type in Section \ref{sec:stable_homotopy_type}. Associated to the representation ring $R(G^{*}_{m})$ we associate an \emph{additive lattice} $\N^{m}$, and to a space $X$ of type $\CC$-$G^{*}_{m}$-$\SWF$ we extract a distinguished subset $\I(X)\subset\N^{m}$ from which these equivariant $k$-invariants are derived from.

In Section \ref{subsec:monomials} we analyze the structure of the representation ring $R(G^{*}_{m})$, and in Section \ref{subsec:additive_posets_additive_lattices} we construct the lattice $\N^{m}$. In Section \ref{subsec:definition_k_invariants}, after defining spaces of type $G^{*}_{m}$-$\SWF$ and $\CC$-$G^{*}_{m}$-$\SWF$ we define our equivariant analogues of Manolescu's $k$-invariants, as well as prove some properties about them. We then prove some further special properties of these invariants in the cases where $m=2^{r}$ and $m=p^{r}$ is an odd prime power in Sections \ref{subsec:2_r_actions} and \ref{subsec:p_r_actions}, respeectively.

\bigskip
\subsection{Monomials}
\label{subsec:monomials}

Let $m\geq 1$ be an integer, and consider the free commutative polynomial algebra $\ZZ[x_{0},\dots,x_{m-1}]$ on $m$ variables. We define $(X_{m},\cdot)\subset\ZZ[x_{0},\dots,x_{m-1}]$ to be the multiplicative monoid generated by the variables $x_{0},\dots,x_{m-1}$, whose elements we will often denote by $\mbfx^{\vec{a}}:=\prod_{k=0}^{m-1}x_{k}^{a_{k}}$, where $\vec{a}=(a_{0},\dots,a_{m-1})\in\NN^{m}=\ZZ_{\geq 0}^{m}$.

Next for $*\in\{\ev,\odd\}$, let $W_{m}^{*}\subset R(G^{*}_{m})$ denote the multiplicative monoid of elements which can be written in the form $w_{0}^{a_{0}}w_{1}^{a_{1}}\cdots w_{m-1}^{a_{m-1}}$ for some $a_{0},\dots,a_{m-1}\geq 0$. Again, we will write $\mbfw^{\vec{a}}:=\prod_{k=0}^{m-1}w_{k=0}^{m-1}$ to denote elements of $W_{m}^{*}$. It is not hard to see that $W^{\ev}_{m}\cong W^{\odd}_{m}$ as monoids, and so we will proceed to simply write $W_{m}$ to denote either $W^{\ev}_{m}$ or $W^{\odd}_{m}$.

Observe that there is a canonical surjection $\alpha_{m}:X_{m}\to W_{m}$ of multiplicative monoids which sends $\mbfx^{\vec{a}}\mapsto\mbfw^{\vec{a}}$. Given an element $\mbfw^{\vec{a}}\in W_{m}$, we define a \emph{presentation} of $\mbfw^{\vec{a}}$ to be an element of the inverse image $\alpha_{m}^{-1}(\mbfw^{\vec{a}})\subset X_{m}$, and refer to $\alpha_{m}^{-1}(\mbfw^{\vec{a}})$ as the \emph{set of presentations of} $\mbfw^{\vec{a}}$.

\begin{example}
	Consider the case where $m=2$, and without loss of generality consider the case $*=\ev$. One can show that the relations (\ref{eq:G_ev_m_relations}) imply that $w_{0}^{2}=2w_{0}$ and $w_{1}^{2}=2w_{1}$. From these relations, we can conclude that
	\begin{align*}
		&\alpha_{m}^{-1}(w_{0}^{a_{0}})=\{x_{0}^{a_{0}}\}, \\
		&\alpha_{m}^{-1}(w_{1}^{a_{1}})=\{x_{1}^{a_{1}}\}, \\
		&\alpha_{m}^{-1}(w_{0}^{a_{0}}w_{1}^{a_{1}})=\{x_{0}^{a_{0}+a_{1}-1}x_{1},x_{0}^{a_{0}+a_{1}-2}x_{1}^{2},\dots,x_{0}^{2}x_{1}^{a_{0}+a_{1}-2},x_{0}x_{1}^{a_{0}+a_{1}-1}\}\text{ for all }a_{0},a_{1}\geq 1.
	\end{align*}

\end{example}

For $k=0,\dots,m-1$, let $\vec{e}_{k}\in\NN^{m}$ be the vector with a 1 in the $k$th entry and zeroes in the all the other entries. We define $x_{0}X_{m}\subset X_{m}$ to be the semigroup consisting of elements of the form $\mbfx^{\vec{a}}$ where $\vec{a}=(a_{0},\dots,a_{m-1})$ is such that $a_{0}\geq 1$. Observe that $\mbfx^{\vec{b}}X_{m}$ is naturally a module over the monoid $X_{m}$, in the sense that if $\mbfx^{\vec{a}}\in x_{0}X_{m}$ and $\mbfx^{\vec{b}}\in X_{m}$, then $\mbfx^{\vec{a}}\cdot\mbfx^{\vec{b}}=\mbfx^{\vec{a}+\vec{b}}\in x_{0}X_{m}$. Similarly, we define $w_{0}W_{m}\subset W_{m}$ to be the semigroup consisting of elements of which can be expressed in the form $\mbfw^{\vec{a}}$ where $a_{0}\geq 1$. Again, we see that $w_{0}W_{m}$ is naturally a $W_{m}$-module, and hence an $X_{m}$-module via the surjection $\alpha_{m}:X_{m}\to W_{m}$. Note that $\alpha_{m}$ restricts to a surjection of semigroups $x_{0}X_{m}\to w_{0}W_{m}$ compatible with the $X_{m}$-module structure on both sides.

We have the following lemma concerning the presentations of elements in $w_{0}W_{m}$:

\begin{lemma}
Let $\mbfw^{\vec{a}}\in w_{0}W_{m}$. Then $\alpha_{m}^{-1}(\mbfw^{\vec{a}})\subset x_{0}X_{m}$. In other words, every presentation $\mbfx^{\vec{a'}}\in\alpha_{m}^{-1}(\mbfw^{\vec{a}})$ of $\mbfw^{\vec{a}}$ satisfies $a_{0}'\geq 1$.
\end{lemma}

\begin{proof}
Without loss of generality assume $*=\ev$. For $y\in R(G^{\ev}_{m})$, let $\tr_{g}(y)\in\CC$ denote the trace of the (virtual) representation $y$ at the element $g\in G^{\ev}_{m}$. In particular, note that $\tr_{\gamma}(1)=\tr_{\gamma}(\wt{c})=1$ and $\tr_{\gamma}(\zeta^{k})=\omega_{m}^{k}$ for all $k=0,\dots,m-1$, where $\omega_{m}=e^{2\pi i/k}\in\CC$. Hence for any $\vec{a}=(a_{0},\dots,a_{m-1})\in\NN^{m}$, we see that
\[\tr_{\gamma}(\mbfw^{\vec{a}})=\prod_{k=0}^{m-1}\tr_{\gamma}(1-\wt{c}\zeta^{k})^{a_{k}}=\prod_{k=0}^{m-1}(1-\omega_{m}^{k})^{a_{k}}=0\]
if and only if $a_{0}\geq 1$. In particular if $\vec{a}'=(a_{0}',\dots,a_{m-1}')\in\NN^{m}$ is any other vector such that $\mbfw^{\vec{a}'}=\mbfw^{\vec{a}}$ in $R(G^{\ev}_{m})$, then we must have $a_{0}'\geq 1$, as desired.
\end{proof}

Next, we have the following proposition, whose proof is given in Appendix \ref{sec:number_theory}:

\begin{proposition}
\label{prop:monomials}
Let $m=p^{r}$ be a prime power, and let $\vec{a},\vec{b}\in\NN^{m}$ with $a_{0},b_{0}\geq 1$. Then $\mbf{w}^{\mbf{a}}=\mbf{w}^{\mbf{b}}\in w_{0}W_{p^{r}}$ if and only if:
\begin{enumerate}
    \item \emph{if $p$ odd:}
    \begin{equation}
        \sum_{k=0}^{p^{r}-1}a_{k}=\sum_{k=0}^{p^{r}-1}b_{k},
    \end{equation}
    and for each $t\in\{0,\dots,r-1\}$, we have that:
    \begin{align}
        &\sum_{\ell=0}^{p^{r-t-1}-1}a_{\ell p^{t+1}}=\sum_{\ell=0}^{p^{r-t-1}-1}b_{\ell p^{t+1}},\\
        \begin{split}
        &\sum_{k=1}^{\frac{p^{t+1}-1}{2}}\sum_{\ell=0}^{p^{r-t-1}-1}k(a_{k+\ell p^{t+1}}-a_{-k-\ell p^{t+1}}) \\
        &\qquad\qquad\equiv\sum_{k=1}^{\frac{p^{t+1}-1}{2}}\sum_{\ell=0}^{p^{r-t-1}-1}k(b_{k+\ell p^{t+1}}-b_{-k-\ell p^{t+1}})\pmod{2p^{t+1}}, \text{ and}
        \end{split} \\
        \begin{split}
        &\sum_{\ell=0}^{p^{r-t-1}-1}\sum_{s=0}^{t}(a_{kp^{s}+p^{t+1}}+a_{-kp^{s}-p^{t+1}}-a_{(kp^{s}+p^{t+1})/2}-a_{(-kp^{s}-p^{t+1})/2}) \\
        &\qquad\qquad=\sum_{\ell=0}^{p^{r-t-1}-1}\sum_{s=0}^{t}(b_{kp^{s}+p^{t+1}}+b_{-kp^{s}-p^{t+1}}-b_{(kp^{s}+p^{t+1})/2}-b_{(-kp^{s}-p^{t+1})/2}) \\
        &\qquad\qquad\qquad\qquad\text{ for all }k=2,\dots,\frac{p^{t+1}-1}{2}\text{ with }(k,p)=1.
        \end{split}
    \end{align}
    Here we use the indexing convention that if $k$ is odd, then $k/2:=2^{-1}k$ where $2^{-1}\in\ZZ_{p^{r}}^{\times}$ is the unique inverse of $2$ in $\ZZ_{p^{r}}^{\times}$. 
    \item \emph{if $p=2$:}
    \begin{align}
        &\sum_{k=0}^{2^{r}-1}a_{k}=\sum_{k=0}^{2^{r}-1}b_{k}, \\
        &\sum_{k=0}^{2^{r-t-1}-1}a_{(2k+1)2^{t}}=0\iff\sum_{k=0}^{2^{r-t-1}-1}b_{(2k+1)2^{t}}=0\text{ for each }t=0,\dots,r-1,
    \end{align}
    and for each $t\in\{0,\dots,r-1\}$ such that
    \begin{equation}
        \sum_{k=0}^{2^{r-t-1}-1}a_{(2k+1)2^{t}}=\sum_{k=0}^{2^{r-t-1}-1}b_{(2k+1)2^{t}}=0,
    \end{equation}
    we have that:
    \begin{align}
        \begin{split}
        &\sum_{k=1}^{2^{t}-1}\sum_{\ell=0}^{2^{r-t-1}-1}k(a_{k+\ell 2^{t+1}}-a_{-k-\ell 2^{t+1}})\\
        &\qquad\qquad\equiv\sum_{k=1}^{2^{r-1}-1}\sum_{\ell=0}^{2^{r-t-1}-1}k(b_{k+\ell 2^{t+1}}-b_{-k-\ell 2^{t+1}})\pmod{2^{t+2}},\text{ and}
        \end{split} \\
        \begin{split}
        &\sum_{\ell=0}^{2^{r-t-1}-1}2a_{\ell 2^{t+1}}+\Big(\sum_{s=0}^{t-1}a_{k2^{s}+(2\ell+1)2^{t}}+a_{-k2^{s}-(2\ell+1)2^{t}}\Big) \\
        &\qquad\qquad=\sum_{\ell=0}^{2^{r-t-1}-1}2b_{\ell 2^{t+1}}+\Big(\sum_{s=0}^{t-1}b_{k2^{s}+(2\ell+1)2^{t}}+b_{-k2^{s}-(2\ell+1)2^{t}}\Big) \\
        &\qquad\qquad\qquad\qquad\text{ for all }k=1,\dots,2^{t}-1\text{ odd}.
        \end{split}
    \end{align}
\end{enumerate}
\end{proposition}

\begin{definition}
\label{def:normal_form_monomials}
Endow $X_{m}$ with the lexicographic ordering, with the convention that
\[x_{0}< x_{1} < \cdots < x_{m-1}.\]
We define the \emph{normal form} of a monomial $\mbfw^{\vec{a}}\in W_{m}$ to be the unique presentation
\[\mbfx^{\vec{a}_{\min}}\in \alpha_{m}^{-1}(\mbfw^{\vec{a}})\subset X_{m}\]
of $\mbfw^{\vec{a}}$ which is \emph{minimal} with respect to this lexicographical ordering.
\end{definition}

\begin{example}
\label{ex:w_0W_2}
For $m=2$, the relations from Proposition \ref{prop:monomials} are generated by the relation
\[w_{0}^{a}w_{1}^{b}=w_{0}^{a+b-1}w_{1}\qquad\text{ if }b\geq 1.\]
By ``trading'' all but one of the $w_{1}$'s over to $w_{0}$, it follows that every monomial in $w_{0}W_{2}$ has normal form $x_{0}^{a}x_{1}^{b}\in x_{0}X_{2}$ where $a\geq 1$, $b\in\{0,1\}$.
\end{example}

\begin{example}
\label{ex:w_0W_3}
For $m=3$, the relations from Proposition \ref{prop:monomials} are generated by the relation
\[w_{0}w_{1}^{3}=w_{0}w_{2}^{3}.\]
In analogy with the $m=2$ case, we can ``trade'' $w_{2}$'s over to $w_{1}$, but only three at a time. It follows that every monomial in $w_{0}W_{3}$ has a unique presentation in the normal form $x_{0}^{a}x_{1}^{b}x_{2}^{c}\in x_{0}X_{3}$ where $a\geq 1$, $b\geq 0$, and $c\in\{0,1,2\}$.
\end{example}

\begin{example}
\label{ex:w_0W_4}
For $m=4$, the relations from Proposition \ref{prop:monomials} are as follows:
\begin{align*}
	&w_{0}^{a}w_{2}^{c}=w_{0}^{a+c-1}w_{2}\qquad\text{ if }c\geq 1, \\
	&w_{0}w_{1}^{4}=w_{0}w_{3}^{4}, \\
	&w_{0}^{a}w_{1}^{b}w_{2}^{c}w_{3}^{d}=w_{0}^{a+b+c+d-2}w_{1}w_{2}\qquad\text{ if }b+d\geq 1, c\geq 1.
\end{align*}
Via a similar argument as in the $m=2$ and $3$ cases, we see that every monomial in $w_{0}W_{4}$ has a unique normal form that falls into one of the following three categories: 
\begin{align*}
	&x_{0}^{a}x_{2}^{b}, & &a\geq 1, b\in\{0,1\},\\
	&x_{0}^{a}x_{1}^{b}x_{3}^{d}, & &a\geq 1, b\geq 0, d\in\{0,1,2,3\}, \\
	&x_{0}^{a}x_{1}x_{2}, & &a\geq 1.
\end{align*}
\end{example}

\begin{example}
\label{ex:w_0W_5}
For $m=5$, the relations from Proposition \ref{prop:monomials} are generated by the relations
\begin{align*}
	&w_{0}w_{1}w_{2}^{2}=w_{0}w_{3}^{2}w_{4}, & &w_{0}w_{1}^{3}w_{2}^{1}=w_{0}w_{3}^{1}w_{4}^{3}, & &w_{0}w_{1}^{5}=w_{0}w_{4}^{5}, \\
	&w_{0}w_{1}^{2}w_{2}^{4}=w_{0}w_{3}^{4}w_{4}^{2}, & &w_{0}w_{1}^{4}w_{2}^{3}=w_{0}w_{3}^{3}w_{4}^{4}, & &w_{0}w_{2}^{5}=w_{0}w_{3}^{5}.
\end{align*}
By a similar argument as in the previous cases, every monomial in $w_{0}W_{5}$ has a unique normal form $x_{0}^{a}x_{1}^{b}x_{2}^{c}x_{3}^{d}x_{4}^{e}$, where $a\geq 1$, $b,c\geq 0$, and $0\le d,e\le 4$ are such that:
\[(d,e)\not\in\{(1,4),(2,1),(2,2),(2,3),(2,4),(3,4),(4,2),(4,3),(4,4)\}.\]
\end{example}

\begin{example}
\label{ex:w_0W_2^r}
Let $m=2^{r}$ for some $r\geq 3$. The full set of relations are difficult to calculate explicitly, even with Proposition \ref{prop:monomials} in hand. However we are able to deduce the following: Suppose $\mbfw^{\vec{a}}\in w_{0}W_{2^{r}}$ is such that for each $0\le s\le r$, there exists some $\ell_{s}$ such that $a_{(2\ell_{s}+1)2^{s}}>0$. Then $\mbfw^{\vec{a}}$ has normal form $\mbfx^{\vec{a}'}\in x_{0}X_{2^{r}}$, where $\vec{a}'=(|\vec{a}|-r)\vec{e}_{0}+\sum_{j=0}^{r-1}\vec{e}_{2^{j}}$.
\end{example}

\subsection{Additive Posets and Additive Lattices}
\label{subsec:additive_posets_additive_lattices}

The semigroups $x_{0}X_{m}$, $w_{0}W_{m}$ defined in the previous section come with some additional structure, best encapsulated within the notions of \emph{additive posets} and \emph{additive lattices}:

\begin{definition}
\label{def:additive_lattice}
An \emph{additive poset} is a triple $(P,\preceq,+)$, where:
\begin{enumerate}
    \item $(P,+)$ is a commutative monoid with identity element $0$.
    \item $(P,\preceq)$ is a poset.
    \item For every $a,b,c\in P$ with $a\preceq b$, we have that $a+c\preceq b+c$. 
\end{enumerate}
We say furthermore that $(P,\preceq,+)$ is an \emph{additive lattice} if the underlying poset $(P,\preceq)$ is a \emph{lattice}, i.e., every non-empty finite subset $S\subset P$ has a \emph{least upper bound} (or \emph{join}) $\vee S\in P$ and a \emph{greatest lower bound} (or \emph{meet}) $\wedge S\in P$. 

More precisely, $\vee S\in P$ is the unique element which satisfies $\vee S\succeq s$ for all $s\in S$, and if $x\in P$ is any element satisfying $x\succeq s$ for all $s\in S$, then $x\succeq\vee S$. Similarly, $\wedge S\in P$ is the unique element which satisfies $\wedge S\preceq s$ for all $s\in S$, and if $x\in P$ is any element satisfying $x\preceq s$ for all $s\in S$, then $x\preceq\vee S$. If $S=\{x,y\}$ has two elements, we write $x\vee y:=\vee S$, $x\wedge y:=\wedge S$.

An \emph{additive poset homomorphism} 
\[f:(P,\preceq,+)\to (Q,\preceq,+)\]
is a monoid homomorphism $(P,+)\to(Q,+)$ which is order-preserving, i.e., $a\preceq b$ implies $f(a)\preceq f(b)$ for all $a,b\in P$. If $P,Q$ are additive lattices, we say furthermore that $f$ is an \emph{additive lattice homomorphism} if $f$ satisfies the additional property that
\[f(\vee S)=\vee f(S)\text{ and }f(\wedge S)=\wedge f(S)\]
for all non-empty finite subsets $S\subset P$.
\end{definition}

It will be convenient for us to be able to define the least upper bound and greatest lower bound for empty subsets. To this, we introduce the following definition:

\begin{definition}
\label{def:bounded_completion}
Let $(L,\preceq,+)$ be an additive lattice. We define the \emph{completion} $(\wh{L},\preceq,+)$ of $(L,\preceq,+)$ to be the additive lattice with underlying set $\wh{L}:=L\cup\{+\infty\}$, and with partial order and addition such that $+\infty\in\wh{L}$ satisfies $a\preceq+\infty$ and $a+(+\infty)=(+\infty)+(+\infty)=+\infty$ for all $a\in L$. We then define
\[\vee\emptyset=\wedge\emptyset=+\infty.\]
\end{definition}

\begin{definition}
\label{def:T_grading}
Let $(P,\preceq,+)$ be an additive poset, and let $(T,\le,+)$ be an additive poset such that $\le$ is a \emph{total} order. We define a \emph{$T$-grading} on $(P,\preceq,+)$ to be an additive poset homomorphism $|\cdot|:(P,\preceq,+)\to(T,\le,+)$ such that:
\begin{equation}
\label{eq:T_grading}
    x\prec y\implies |x|<|y|\qquad\forall x,y\in P.
\end{equation}
We refer to the quadruple $(P,\preceq,+,|\cdot|)$ as a \emph{$T$-graded additive poset}. A \emph{homomorphism of $T$-graded additive posets}
is a homomorphism of additive posets which respects the $T$-grading.
\end{definition}

We will often take $T=\NN$, $\ZZ$ or $\QQ$. 

\begin{example}
\label{ex:additive_lattice_1}
Let $\NN=\ZZ_{\geq 0}$, let $m\geq 1$ be an integer, and consider the monoid $(\NN^{m},+)$ given by pointwise addition. Throughout this paper, we will often denote an element of $\NN^{m}$ by
\[\vec{k}=(k_{0},\dots,k_{m-1})\in\NN^{m}.\]
We can endow $\NN^{m}$ with the partial order $\preceq$ given by the product partial order induced by the total order on $\NN$, i.e.,
\[\vec{k}\preceq\vec{\ell}\iff k_{j}\le \ell_{j}\text{ for all }j=0,\dots,m-1.\]
Then $(\NN^{m},\preceq,+)$ is an additive lattice, since for any non-empty finite subset
\[S=\{\vec{k}^{1},\dots,\vec{k}^{n}\}\subset\NN^{m},\qquad\vec{k}^{i}=(k^{i}_{0},\dots,k^{i}_{m-1}),\qquad i=1,\dots, n,\]
the least upper bound and greatest lower bound of $S$ are given respectively by
\begin{align*}
	&\vee S=(\max\{k^{1}_{0},\dots,k^{n}_{0}\},\max\{k^{1}_{1},\dots,k^{n}_{1}\},\dots,\max\{k^{1}_{m-1},\dots,k^{n}_{m-1}\})\in\NN^{m}, \\
	&\wedge S=(\min\{k^{1}_{0},\dots,k^{n}_{0}\},\min\{k^{1}_{1},\dots,k^{n}_{1}\},\dots,\min\{k^{1}_{m-1},\dots,k^{n}_{m-1}\})\in\NN^{m}.
\end{align*}
Furthermore, $(\NN^{m},\preceq,+)$ has a natural $\NN$-grading given by $|(k_{0},\dots,k_{m-1})|:=k_{0}+\cdots+k_{m-1}$.
\end{example}

\begin{remark}
For an additive lattice $(L,\preceq,+)$, a $T$-grading $|\cdot|:(L,\preceq,+)\to(T,\le,+)$ is not required to be an additive lattice homomorphism, only an additive poset homomorphism. For example, the $\NN$-grading $|\cdot|:(\NN^{2},\preceq,+)\to(\NN,\le,+)$ is not an additive lattice homomorphism, since for example 
\[|\vee\{(1,0),(0,1)\}|=|(1,1)|=2\neq 1=\vee\{1\}=\vee\{|(1,0)|,|(0,1)|\}.\]
\end{remark}

Next we will need to construct quotients of additive posets. As an arbitrary equivalence relation on a poset will not necessarily produce a well-defined partial order on the corresponding quotient, we will need to impose some conditions on the equivalence relation to ensure this happens. The following definition was inspired by definitions from \cite{HS15}:

\begin{definition}
\label{def:quotient_posets}
Let $(P,\preceq,+,|\cdot|)$ be a $(T,\le,+)$-graded additive poset, and let $\sim$ be an equivalence relation on $P$ satisfying the following properties:
\begin{enumerate}[label=(C\arabic*),ref=C\arabic*]
    \item If $x\sim y$, then $|x|=|y|$. \label{condition:quotient_posets_grading}
    \item If $x\sim y$ and $x'\sim y'$, then $x+x'\sim y+y'$. \label{condition:quotient_posets_addition}
    \item If $x\preceq y$, then for any $x'\sim x$, there exists $y'\sim y$ such that $x'\preceq y'$. \label{condition:quotient_posets_homogenous}
\end{enumerate}
We define the \emph{quotient of $P$ by $\sim$} to be the poset with underlying set $\ol{P}=P/\sim$, and partial order, addition, and $T$-grading as follows:
\begin{enumerate}
    \item $[x]\preceq[y]$ if there exists $x',y'\in P$ with $[x']=[x]$ and $[y']=[y]$ such that $x'\preceq y'$.
    \item $[x]+[y]:=[x+y]$.
    \item $|[x]|:=|x|$.
\end{enumerate}
\end{definition}

\begin{proposition}
\label{prop:quotient_posets}
The quotient poset $(\ol{P},\preceq,+,|\cdot|)$ described in Definition \ref{def:quotient_posets} is well-defined. Moreover if $P$ is a lattice, then $\ol{P}$ is a lattice as well, assuming $\sim$ satisfies the additional property
\begin{enumerate}[label=(C4),ref=C4]
    \item If $x\sim x'$ and $y\sim y'$, then $x\vee y\sim x'\vee y'$ and $x\wedge y\sim x'\wedge y'$.
    \label{condition:quotient_posets_lattice}
\end{enumerate}
\end{proposition}

\begin{proof}
By \ref{condition:quotient_posets_addition} and \ref{condition:quotient_posets_homogenous}, the addition and $T$-grading are well-defined and compatible with the partial order, provided that the partial order itself is well-defined.

Reflexivity of the partial order is immediate. For symmetry: let $x,y\in P$ be such that $[x]\preceq[y]$ and $[y]\preceq[x]$. Then by \ref{condition:quotient_posets_homogenous} there exist $x',y'\in P$ with $[x']=[x]$ and $[y']=[y]$ such that $x\preceq y'$ and $y'\preceq x'$. If $x\prec y'$, then by \ref{condition:quotient_posets_grading} and Equation \ref{eq:T_grading} we would have
\[|x|<|y'|\le|x'|=|x|,\]
a contradiction. Hence $x=y'\sim y$, and so $[x]=[y]$. For transitivity: let $x,y,z\in P$ be such that $[x]\preceq[y]$ and $[y]\preceq[z]$. Then by \ref{condition:quotient_posets_homogenous} there exist $y',z'\in P$ with $[y']=[y]$ and $[z']=[z]$ such that $x\preceq y'$ and $y'\preceq z'$. Hence $x\preceq z'\sim z$, and so $[x]\preceq[z]$.

For the second claim, suppose that $P$ is a lattice and $\sim$ satisfies \ref{condition:quotient_posets_lattice}. Let $x,y\in P$, and let $z=x\vee y$. We claim that $[z]\in\ol{P}$ is the unique least upper bound of $[x]$ and $[y]$. Indeed, let $w\in P$ be such that $[w]\succeq[x]$ and $[w]\succeq[y]$. It suffices to show that $[w]\succeq[z]$. By the definition of the partial order on $\ol{P}$, there exist $x',y',w',w''\in P$ with $x'\sim x$, $y'\sim y$ and $w'\sim w''\sim w$ such that $w'\succeq x'$ and $w''\succeq y'$. By \ref{condition:quotient_posets_lattice} we have that
\begin{align*}
    &x'\sim x\implies w'=w'\vee x'\sim w'\vee x\implies w'\succeq x, \\
    &y'\sim y\implies w''=w''\vee y'\sim w''\vee y\implies w''\succeq y, \\
    &w''\sim w'\implies y=w''\wedge y\sim w'\wedge y\implies w'\succeq y.
\end{align*}
By the definition of $z=x\vee y$, we must have that $w'\succeq z$, and hence $[w]\succeq[z]$ in $\ol{P}$. The argument for the existence of unique greatest lower bounds is entirely analogous.
\end{proof}

Next, let $m\geq 2$ be an integer, let $\sim$ be the following equivalence relation on $(\NN^{m},\preceq,+,|\cdot|)$:
\[\vec{a}\sim\vec{b}\iff\mbf{w}^{\vec{a}+\vec{e}_{0}}=\mbf{w}^{\vec{b}+\vec{e}_{0}}\in R(G^{*}_{m}),\qquad\ast\in\{\ev,\odd\},\]
and let $\N^{m}:=\NN^{m}/\sim$.

\begin{proposition}
\label{prop:quotient_lattice_N}
$\N^{m}$ has a well-defined $\NN$-graded additive lattice structure coming from that of $\NN^{m}$.
\end{proposition}

\begin{proof}
By Proposition \ref{prop:quotient_posets}, it suffices to check that Conditions \ref{condition:quotient_posets_grading}--\ref{condition:quotient_posets_lattice} hold.

Let $\tr_{g}(x)\in\CC$ denote the trace of $x\in R(G^{*}_{m}))$ at $g\in G^{*}_{m}$. Note that since $\tr_{j}(w_{k})=2$ for all $k=0,\dots,m-1$, it follows that $\tr_{j}(\mbfw^{\vec{a}+\vec{e}_{0}})=2^{|\vec{a}|+1}$ for any $\vec{a}\in\NN^{m}$. Hence for any $\vec{a},\vec{b}\in\NN^{m}$ such that $\mbfw^{\vec{a}+\vec{e}_{0}}=\mbfw^{\vec{b}+\vec{e}_{0}}\in R(G^{*}_{m})$ we must have that $|\vec{a}|=|\vec{b}|$, showing that \ref{condition:quotient_posets_grading} holds.

Next, note that if $\vec{a},\vec{a}\,',\vec{b},\vec{b}'\in\NN^{m}$ are such that $\mbfw^{\vec{a}+\vec{e}_{0}}=\mbfw^{\vec{b}+\vec{e}_{0}}$ and $\mbfw^{\vec{a}'+\vec{e}_{0}}=\mbfw^{\vec{b}'+\vec{e}_{0}}$, then
\[2\mbfw^{\vec{a}+\vec{b}+\vec{e}_{0}}=\mbfw^{\vec{a}+\vec{b}+2\vec{e}_{0}}=\mbfw^{\vec{a}+\vec{e}_{0}}\mbfw^{\vec{b}+\vec{e}_{0}}=\mbfw^{\vec{a}'+\vec{e}_{0}}\mbfw^{\vec{b}'+\vec{e}_{0}}=\mbfw^{\vec{a}'+\vec{b}'+2\vec{e}_{0}}=2\mbfw^{\vec{a}'+\vec{b}'+\vec{e}_{0}},\]
where we used the fact that $w_{0}^{2}=2w_{0}\in R(G^{*}_{m})$. By inspection there is no 2-torsion in $R(G^{*}_{m})$, and so
\[\mbfw^{\vec{a}+\vec{b}+\vec{e}_{0}}=\mbfw^{\vec{a}'+\vec{b}'+\vec{e}_{0}}\implies\vec{a}+\vec{b}\sim\vec{a}'+\vec{b}',\]
hence \ref{condition:quotient_posets_addition} holds as well.

Next, suppose $\vec{a},\vec{b}\in\NN^{m}$ are such that $\vec{a}\preceq\vec{b}$, and let $\vec{a}'\in\NN^{m}$ be such that $\vec{a}\,'\sim\vec{a}$. Note that $\vec{b}-\vec{a}$ is a well-defined element of $\NN^{m}$. Hence by \ref{condition:quotient_posets_addition} we have that 
\[\vec{a}\,'=\vec{a}'+\vec{0}\preceq\vec{a}\,'+(\vec{b}-\vec{a})\sim\vec{a}+(\vec{b}-\vec{a})=\vec{b},\]
and so \ref{condition:quotient_posets_homogenous} holds as well. We leave the proof that \ref{condition:quotient_posets_lattice} holds to the interested reader.
\end{proof}

Next, consider the following definition:

\begin{definition}
Let $(P,\preceq_{P},+_{P})$ and $(M,\preceq_{M},+_{M})$ be additive posets. We say that $(M,\preceq_{M},+_{M})$ is a \emph{module} over $(P,\preceq_{P},+_{P})$ if:
\begin{enumerate}
    \item $(M,+_{M})$ is a module over $(P,+_{P})$ as additive monoids, i.e.:
    \begin{enumerate}
        \item For every $x\in M$ and $a\in P$, there is a well defined element $x+a\in M$.
        \item For every $x\in M$, $x+0_{P}=x$ where $0_{P}$ is the identity element in $0_{P}$.
        \item $x+(a+_{P}b)=(x+a)+b$ for all $x\in M$, $a,b\in P$.
    \end{enumerate}
    \item For every $x,y\in M$ with $x\preceq_{M} y$, we have that $x+a\preceq_{M} y+a$ for every $a\in P$. 
\end{enumerate}
If $P,M$ are $T$-graded, we furthermore say that $M$ is a \emph{graded module} over $P$ if $|x+a|=|x|+|a|$ for all $x\in M$, $a\in P$.
\end{definition}

\begin{example}
\label{ex:module_structure_projection_map}
Any $T$-graded additive poset $(P,\preceq,+,|\cdot|)$ is naturally a graded module over itself. More generally if $\ol{P}=P/\sim$ is obtained from $P$ as a quotient via an equivalence relation $\sim$ as in Definition \ref{def:quotient_posets}, then $(\ol{P},\preceq,+,|\cdot|)$ is a naturally module over $(P,\preceq,+,|\cdot|)$, with module structure given by $[x]+y:=[x]+[y]$ for all $x,y\in P$. In particular, $(\N^{m},\preceq,+,|\cdot|)$ is a graded module over $(\NN^{m},\preceq,+,|\cdot|)$ for any $m\geq 2$.
\end{example}

Although our definition of $(\N^{m},\preceq,+,|\cdot|)$ would suffice for our purposes, it can be difficult to work with equivalence classes in certain settings. It will therefore be useful to have an alternate presentation of $\N^{m}$ as follows: 

Recall from Definition \ref{def:normal_form_monomials} that each $\mbf{w}^{\vec{a}}\in W_{m}$ has a unique normal representative. Hence for each equivalence class $[\vec{a}]\in\N^{m}$, we have a unique representative $\nrm(\vec{a})\in\NN^{m}$ with $[\nrm(\vec{a})]=[\vec{a}]\in\N^{m}$ such that the monomial $\mbfx^{\nrm(\vec{a})+\vec{e}_{0}}\in\alpha_{m}^{-1}(\mbfw^{\vec{a}+\vec{e}_{0}})\subset x_{0}X_{m}$ is the unique normal representative of $\mbfw^{\vec{a}+\vec{e}_{0}}\in w_{0}W_{m}$. Via this observation, we have an embedding of sets
\begin{align*}
	e_{\N}:\N^{m}&\hookrightarrow\NN^{m} \\
	[\vec{a}]&\mapsto\nrm(\vec{a})
\end{align*}
which preserves the $\NN$-grading. This leads us to the following definition:

\begin{definition}
Let $\N^{m}_{\nrm}:=e_{\N}(\N^{m})\subset\NN^{m}$. We define a partial order and addition $\preceq_{\nrm}$, $+_{\nrm}$ on $\N^{m}_{\nrm}$ to be the \emph{push-forward} of the partial order and addition on $\N^{m}$ under the embedding $e_{\N}$. More precisely, suppose $\vec{a},\vec{b}\in\N^{m}_{\nrm}$. Then $\vec{a}\preceq_{\nrm}\vec{b}$ in $\N^{m}_{\nrm}$ if and only if $[\vec{a}]\preceq[\vec{b}]$ in $\N^{m}$, and $\vec{a}+_{\nrm}\vec{b}=\nrm(\vec{a}+\vec{b})$.

We define a $(\NN^{m},\preceq,+,|\cdot|)$-module structure on $(\N^{m}_{\nrm},\preceq_{\nrm},+_{\nrm},|\cdot|)$ to be the push-forward via $e_{\N}$ of the $(\NN^{m},\preceq,+,|\cdot|)$-module structure on $(\N^{m},\preceq,+,|\cdot|)$. Equivalently, for all $\vec{a}\in\N^{m}_{\nrm}$, $\vec{c}\in\NN^{m}$:
\[\vec{a}+\vec{c}=\vec{a}+_{\nrm}\nrm(\vec{c})=\nrm(\vec{a}+\vec{c}).\]
\end{definition}

We see that by construction, $(\N^{m},\preceq_{\nrm},+_{\nrm},|\cdot|)$ is \emph{isomorphic} to $(\N^{m},\preceq,+,|\cdot|)$ as an $\NN$-graded additive poset, presented as a subset of $\NN^{m}$ but with a deformed partial order and additive structure. 

\begin{example}
\label{ex:lattice_m_equals_2}
Let $m=2$. From Example \ref{ex:w_0W_2}, we see that $\N^{2}=\NN^{2}/\sim$, where the equivalence relation is given by 
\[(a,b)\sim(a+k,b-k)\text{ for all }b\geq 1,\;\;0\le k\le b-1.\]
It follows that the normal form of each $(a,b)\in\NN^{2}$ is given by
\begin{align*}
    &\nrm(a,0)=(a,0), & &\nrm(a,b)=(a+b-1,1)\text{ if }b\geq 1,
\end{align*}
and hence we can identify $\N^{2}_{\nrm}=\NN\times\{0,1\}\subset\NN^{2}$. The partial order $\preceq_{\nrm}$ on $\N^{2}_{\nrm}$ coincides with the restriction of the partial order $\preceq_{\NN^{2}}$ on $\NN^{2}$. Addition in $\N_{2}$ is given by
\[(a_{1},b_{1})+_{\nrm}(a_{2},b_{2})=\twopartdef{(a_{1}+a_{2},b_{1}+b_{2})}{b_{1}+b_{2}\le 1,}{(a_{1}+a_{2}+1,1)}{b_{1}+b_{2}= 2,}\]
and the $(\NN^{2},\preceq_{\NN^{2}},+)$-module structure is similarly given by
\[(a,b)+(k,\ell)=\twopartdef{(a+k,b+\ell)}{b+\ell\le 1,}{(a+k+b+\ell-1,1)}{b+\ell\geq 2.}\]
\end{example}

\begin{example}
\label{ex:lattice_m_equals_3}
Let $m=3$. From Example \ref{ex:w_0W_3}, we see that $\N^{3}=\NN^{3}/\sim$, where the equivalence relation is given by
\[(a,b,c)\sim(a,b+3k,c-3k)\text{ for all }-\tfrac{b}{3}\le k\le \tfrac{c}{3}.\]
It follows that the normal form of each $(a,b,c)\in\NN^{3}$ is given by
\[\nrm(a,b,c)=(a,b+3\lfloor\tfrac{c}{3}\rfloor,c-3\lfloor\tfrac{c}{3}\rfloor),\]
and hence we can identify $\N^{3}_{\nrm}=\NN^{2}\times\{0,1,2\}\subset\NN^{3}$. The partial order $\preceq_{\nrm}$ on $\N^{3}_{\nrm}$ does \emph{not} agree with the restriction of $\preceq_{\NN^{3}}$, since for example we have that
\[(0,0,2)\preceq_{\nrm}(0,3,0)\text{ in }\N^{3}_{\nrm},\]
but $(0,0,2)$ and $(0,3,0)$ are not comparable in $(\NN^{3},\preceq)$. Addition in $\N_{3}$ is given by
\[(a_{1},b_{1},c_{1})+_{\nrm}(a_{2},b_{2},c_{2})=(a_{1}+a_{2},b_{1}+b_{2}+3\lfloor\tfrac{c_{1}+c_{2}}{3}\rfloor,c_{1}+c_{2}-3\lfloor\tfrac{c_{1}+c_{2}}{3}\rfloor),\]
and the $(\NN^{3},\preceq_{\NN^{3}},+)$-module structure is described similarly.
\end{example}

Next, we will look at the minima of subsets of additive lattices.

\begin{definition}
\label{def:minimal_elements}
Let $(P,\preceq)$ be a poset, and let $A\subset P$. We define the \emph{set of minima of $A$}, denoted by $\min(A)\subset A$, to be the set of elements $a\in A$ such that if $a'\in A$ is any element with $a'\preceq a$, then $a'=a$.
\end{definition}

Note that in general, $\min(A)$ may be empty. Indeed, this is the case if $A=\emptyset$, or if $A$ is comprised of infinite chains $a_{0}\succ a_{1} \succ a_{2} \succ \cdots$ which extend infinitely down below. The following definition rules out this latter scenario:

\begin{definition}
\label{def:bounded_below}
A poset $(P,\preceq)$ is \emph{bounded below} if there exists an element $b\in P$ such that $a\succeq b$ for all $a\in P$. 
\end{definition}

If $(P,\preceq)$ is bounded below, then any nonempty subset $A\subset P$ must satisfy $\min(A)\neq\emptyset$.

\begin{example}
The posets $(\NN^{m},\preceq)$ and $(\N^{m}\preceq)$ are bounded below by $\vec{0}$ and $[\vec{0}]$, respectively.
\end{example}

\begin{definition}
\label{def:anti_chain_finite}
Let $(P,\preceq)$ be a poset. An \emph{anti-chain} is a non-empty subset of elements $A\subset P$ such that no two elements of $A$ are comparable under $\preceq$. A poset $(P,\preceq)$ is \emph{anti-chain finite} if it contains no infinite anti-chains, i.e., if $A\subset P$ is an infinite subset, then there exist $a\neq a'\in A$ such that $a\preceq a'$ or $a\succeq a'$.
\end{definition}

Note that the set of minima $\min(A)$ of any non-empty subset $A\subset P$ is necessarily an anti-chain. Hence if $(P,\preceq)$ is anti-chain finite, then any subset $A\subset P$ must satisfy $|\min(A)|<\infty$.

\begin{example}
One can show that $(\NN^{m},\preceq)$ is anti-chain finite as a consequence of either Dickson's Lemma (\cite{Dickson}) or the Hilbert basis theorem. From this it follows that $(\N^{m},\preceq)$ is anti-chain finite via the canonical surjection $\Pi:(\NN^{m},\preceq)\to(\N_{m},\preceq)$, as any anti-chain in $\N^{m}$ has a (non-canonical) lift to an anti-chain in $\NN^{m}$.
\end{example}

Many of the subsets of posets that we will consider in this article will be of a particular form, which we call \emph{upper-complete subsets}:

\begin{definition}
\label{def:upper_complete}
Let $(P,\preceq)$ be a poset, and let $A\subset P$ be a subset. We say that $A$ is \emph{upper-complete} if for each $a\in A$: if $a'\in P$ is such that $a'\succeq a$, then $a'\in A$ as well. Note that $\emptyset\subset P$ is vacuously upper-complete.
\end{definition}

The following lemma allows us to compare sets of minima of subsets:

\begin{lemma}
\label{lemma:comparing_minima_subsets}
Let $(L,\preceq)$ be a lattice, and let $A,B\subset L$ be two subsets such that $A\subset B$.
\begin{enumerate}
    \item The following statements are true:
    \begin{enumerate}
        \item For each $a\in\min(A)$:
        \begin{enumerate}
            \item $a\not\prec b$ for all $b\in\min(B)$, and
            \item there exists some $b\in\min(B)$ such that $a\succeq b$.
        \end{enumerate}
        \item $\vee\min(A)\succeq\wedge\min(B)$.
    \end{enumerate}
    \item Suppose $\min(A)$ consists of a single element $a\in A$. Then:
    \begin{enumerate}
        \item $a\not\prec b$ for all $b\in\min(B)$.
        \item There exists some $b\in\min(B)$ such that $a\succeq b$.
        \item $a\succeq\wedge\min(B)$.
    \end{enumerate}
    \item Suppose $\min(B)$ consists of a single element $b\in B$. Then:
    \begin{enumerate}
        \item $a\succeq b$ for all $a\in\min(A)$.
        \item $\wedge\min(A)\succeq b$.
    \end{enumerate}
\end{enumerate}

\end{lemma}

\begin{proof}
We will assume $A,B\neq\emptyset$, as otherwise the above conclusions are vacuous. For (1ai), if $a\in\min(A)$ and $b\in\min(B)$, then $a\prec b$ would contradict the minimality of $b\in\min(B)$. For (1aii), let $a\in\min(A)$ and define
\[B_{a}:=\{b\in B\;|\;b\preceq a\}\subset B.\]
Since $\min(A)\subset A\subset B$ it follows that $a\in B_{a}$, so in particular $B_{a}\neq\emptyset$. Furthermore, note that $\min(B_{a})\subset\min(B)$ by construction. Hence any element $b\in\min(B_{a})$ will suffice for part (1aii). Next, note that (1b) follows from (1aii). Indeed, if $a\in\min(A)$ and $b\in\min(B)$ are such that $a\succeq b$, then
\[\vee\min(A)\succeq a\succeq b\succeq\wedge\min(B).\]
Finally, one can check that conclusions (2a-c) and (3a) follow from (1a-b), and (3b) follows from (3a).
\end{proof}

\begin{definition}
\label{def:sums_posets}
Let $(P,\preceq,+)$ an additive poset. We introduce the following notation:
\begin{enumerate}
    \item Let $A\subset P$ be a subset and $x\in P$ an element. We define
    \[A+x:=\{a+x\;|\;a\in A\}\subset P.\]
    \item More generally, if $A,B\subset P$ are two subsets we define
    \[A+B:=\{a+b\;|\;a\in A, b\in B\}\subset P.\]
    \item For any subset $A\subset P$, we define
    \[2\cdot A:=\{2a\;|\;a\in A\},\]
    where for $a\in P$, $2a:=a+a$.
\end{enumerate}
\end{definition}

\begin{lemma}
\label{lemma:comparing_minima_sum_of_subsets}
Let $(L,\preceq,+)$ be an additive lattice, and let $A,B,C\subset L$ be non-empty subsets such that $A+B\subset C$. Then:
\begin{enumerate}
    \item The following statements are true:
    \begin{enumerate}
        \item For each $a\in\min(A)$, $b\in\min(B)$:
        \begin{enumerate}
            \item $a+b\not\prec c$ for every $c\in\min(C)$.
            \item There exists some $c\in\min(C)$ such that $a+b\succeq c$.
        \end{enumerate}
        \item $\vee\min(A)+\vee\min(B)\succeq\wedge\min(C)$.
    \end{enumerate}
    \item Suppose $\min(A)$ consists of a single element $a\in A$ and $\min(B)$ consists of a single element $b\in B$. Then:
    \begin{enumerate}
        \item $a+b\not\prec c$ for all $c\in\min(C)$.
        \item There exists some $c\in\min(C)$ such that $a+b\succeq c$.
        \item $a+b\succeq\wedge\min(C)$.
    \end{enumerate}
    \item Suppose $\min(C)$ consists of a single element $c\in C$. Then:
    \begin{enumerate}
        \item $a+b\succeq c$ for all $a\in\min(A)$, $b\in\min(B)$.
        \item $\wedge(\min(A)+\min(B))\succeq c$.
    \end{enumerate}
\end{enumerate}
\end{lemma}

\begin{proof}
The proofs of (1ai) and (1aii) are similar to those of the corresponding statements of Lemma \ref{lemma:comparing_minima_subsets}. For (1b), let $a\in\min(A)$, $b\in\min(B)$ and $c\in\min(C)$ be such that $a+b\succeq c$ as guaranteed by (1aii). Then:
\[\vee\min(A)+\vee\min(B)\succeq\vee\min(A)+b\succeq a+b\succeq c\succeq\wedge\min(C).\]
Next, we have that (2) follows from (2) of Lemma \ref{lemma:comparing_minima_subsets}, and the fact that in this case, $\min(A+B)=\min(A)+\min(B)$. Finally, we see that (3b) follows from (3a), which in turn follows from (1aii).
\end{proof}

\subsection{Equivariant \texorpdfstring{$k$}{k}-invariants}
\label{subsec:definition_k_invariants}

We now return to our study of $G^{*}_{m}$-equivariant $K$-theory. Consider the following definition:

\begin{definition}
A \emph{space of type $G^{*}_{m}$-$\SWF$ at level $\mbfr\in RO(\ZZ_{m})_{\geq 0}$} is a pointed finite $G^{*}_{m}$-CW complex $X$ such that:
\begin{enumerate}
    \item The $S^{1}$-fixed point set $X^{S^{1}}$ is $G^{*}_{m}$-homotopy equivalent to $(\mbfr\wt{\RR})^{+}$.
    \item The action of $\Pin(2)\subset G^{*}_{m}$ is free on the complement $X\setminus X^{S^{1}}$.
\end{enumerate}
\end{definition}

In order to define our invariants, we will need to consider the case where $X^{S^{1}}$ is a \emph{complex} representation sphere, so that we can use equivariant Bott periodicity to identify $\wt{K}_{G^{*}_{m}}(X^{S^{1}})\cong R(G^{*}_{m})$. This leads us to the following definition:

\begin{definition}
We say that a space of type $G^{*}_{m}$-$\SWF$ at level
\[\mbfr=r_{0}+\Big(\sum_{j=1}^{\lfloor\frac{m-1}{2}\rfloor}r_{j}\nu_{j}\Big)+r_{m/2}\rho\in RO(\ZZ_{m})_{\geq 0}\]
is at an \emph{even level} if $r_{j}$ is even for all $j=0,\dots,m/2$.
\end{definition}

If $X$ is a space of type $G^{*}_{m}$-$\SWF$ at an even level, from the isomorphisms $\wt{\RR}^{2}\cong\wt{\CC}$, $\wt{\VV}^{2}_{k}\cong\wt{\CC}_{k}\oplus\wt{\CC}_{m-k}$, and $\wt{\RR}_{m/2}^{2}\cong\wt{\CC}_{m/2}$ (if $m$ is even), we can conclude that $X^{S^{1}}$ is a complex representation sphere.

We will provide an alternate characterization of a spaces of type $G^{*}_{m}$-$\SWF$ at an even level. Let $R(\ZZ_{m})^{\sym}\subset R(\ZZ_{m})$ be the additive subgroup consisting of elements $\mbfs=\sum_{j=0}^{m-1}s_{j}\zeta^{j}$ such that $s_{k}=s_{m-k}$. We call these \emph{symmetric} (virtual) representations, which are precisely those representations which are in the image of the complexification map $c:RO(\ZZ_{m})\to R(\ZZ_{m})$. We define 
\begin{align*}
    &R(\ZZ_{2m})^{\sym,\ev}\subset R(\ZZ_{2m})^{\ev},  & &R(\ZZ_{m})_{\geq 0}^{\sym}:=R(\ZZ_{m})^{\sym}\cap R(\ZZ_{m})_{\geq 0}, \\
    &R(\ZZ_{2m})_{\geq 0}^{\sym,\ev}:=R(\ZZ_{m})^{\sym,\ev}\cap R(\ZZ_{m})_{\geq 0} & &
\end{align*}
similarly. As above, the correspondence $\zeta\mapsto\xi^{2}$ induces an isomorphism $R(\ZZ_{m})_{\geq 0}^{\sym}\cong R(\ZZ_{m})_{\geq 0}^{\sym,\ev}$, and so we will not distinguish between the two groups.

\begin{definition}
A \emph{space of type $\CC$-$G^{*}_{m}$-$\SWF$ at level $\mbfs\in R(\ZZ_{m})_{\geq 0}^{\sym}$} is a pointed finite $G^{*}_{m}$-CW complex $X$ such that:
\begin{enumerate}
    \item The $S^{1}$-fixed point set $X^{S^{1}}$ is $G^{*}_{m}$-homotopy equivalent to $(\mbfs\wt{\CC})^{+}$.
    \item The action of $\Pin(2)\subset G^{*}_{m}$ is free on the complement $X\setminus X^{S^{1}}$.
\end{enumerate}
\end{definition}

The following proposition establishes an equivalence between spaces of type $G^{*}_{m}$-$\SWF$ at an even level and spaces of type $\CC$-$G^{*}_{m}$-$\SWF$:

\begin{proposition}
A space of type $G^{*}_{m}$-$\SWF$ at even level $\mbfr\in RO(\ZZ_{m})_{\geq 0}$ is a space of type $\CC$-$G^{*}_{m}$-$\SWF$ at level $c(\frac{1}{2}\mbfr)\in R(\ZZ_{m})_{\geq 0}^{\sym}$. Conversely, a space of type $\CC$-$G^{*}_{m}$-$\SWF$ at level $\mbfs\in R(\ZZ_{m})_{\geq 0}^{\sym}$ is a space of type $G^{*}_{m}$-$\SWF$ at even level $r(\mbfs)\in RO(\ZZ_{m})_{\geq 0}$, where $r:R(\ZZ_{m})\to RO(\ZZ_{m})$ is the map which sends a complex representation to its underlying real representation.
\end{proposition}

\begin{proof}
Follows from the definitions.
\end{proof}

\begin{example}
If $X$ is a type of type $G^{*}_{m}$-$\SWF$ at level $\mbfr\in RO(\ZZ_{m})_{\geq 0}$, then the smash product $X\wedge X$ of two copies of $X$ endowed with the diagonal $G^{*}_{m}$-action is a space of type $\CC$-$G^{*}_{m}$-$\SWF$ at level $c(\mbfr)\in R(\ZZ_{m})_{\geq 0}^{\sym}$.
\end{example}

We can associate to any space $X$ of $\CC$-$G^{*}_{m}$-$\SWF$ a distinguished ideal in $R(G^{*}_{m})$:

\begin{definition}
\label{def:ideal}
Let $X$ be a space of type $\CC$-$G^{*}_{m}$-$\SWF$ at level $\mbfs$, and let $\iota:X^{S^{1}}\hookrightarrow X$ denote the inclusion map. We define $\III(X)\subset R(G^{*}_{m})$ to be the ideal with the property that the image of the induced map
\[\iota^{*}:\wt{K}_{G^{*}_{m}}(X)\to\wt{K}_{G^{*}_{m}}(X^{S^{1}})\]
is equal to $\III(X)\cdot b_{\mbfs\wt{\CC}}$.
\end{definition}

We are now ready to define our equivariant $k$-invariants:

\begin{definition}
\label{def:k_invariants}
Let $X$ be a space of type $\CC$-$G^{\ast}_{m}$-$\SWF$, and let $(\N^{m},\preceq,+)$ be the additive lattice defined in Section \ref{subsec:additive_posets_additive_lattices}.
\begin{enumerate}
    \item Define $I(X)\subset(\N^{m},\preceq,+)$ to be the projection onto $\N^{m}$ the subset of tuples
    \[\vec{k}=(k_{0},\dots,k_{m-1})\in\NN^{m}\]
    for which there exists $x\in\III(X)$ such that
    \[w_{0}x=w_{0}^{k_{0}+1}w_{1}^{k_{1}}\cdots w_{m-1}^{k_{m-1}}.\]
    \item Define $\mbfk(X):=\min(I(X))\subset\N^{m}$ to be the set of minima of $I(X)$ in the sense of Definition \ref{def:minimal_elements}, which we refer to as the \emph{set of equivariant $k$-invariants of $X$}.
\end{enumerate}
\end{definition}

The fact that $\III(X)\subset R(G^{*}_{m})$ is an ideal implies that $I(X)\subset\N^{m}$ is upper-complete in the sense of Definition \ref{def:upper_complete}. Furthermore, $\mbfk(X)\subset I(X)$ must be finite, as the poset $\N^{m}$ is anti-chain finite and bounded below. (See Definitions \ref{def:anti_chain_finite} and \ref{def:bounded_below}).

\begin{remark}
Note that $\mbfk(X)\neq\emptyset$ if and only if $I(X)\neq\emptyset$. We would be able to conclude that $I(X)\neq\emptyset$ for all spaces $X$ of type $\CC$-$G^{\ast}_{m}$-$\SWF$ if we had an analogue of (\cite{Man14}, Lemma 3.2), however, it is not clear whether such a result exists in the $G^{*}_{m}$-equivariant setting.
\end{remark}

Next, we describe the relationship between $\mbfk(X)$ and the invariant $k_{\Pin(2)}(X)$ defined in (\cite{Man14}, Definition 3.3):

\begin{lemma}
\label{lemma:comparison_of_k_invariants}
Let $X$ be a space of type $\wt{\CC}$-$G^{*}_{m}$-$\SWF$ (so that in particular $X$ is a space of type $\Pin(2)$-$\SWF$ at an even level). Then
\begin{equation}
\label{eq:comparison_of_k_invariants}
    k_{\Pin(2)}(X)\le\min\{|\vec{k}|:\vec{k}\in\mbfk(X)\}.
\end{equation}
Here we use the convention that if $\mbfk(X)=\emptyset$, then the right-hand side of (\ref{eq:comparison_of_k_invariants}) is equal to $+\infty$.
\end{lemma}

\begin{proof}
Let $\III_{\Pin(2)}(X)$ denote the ideal defined in \cite{Man14}, and define
\[I_{\Pin(2)}(X):=\{k\in\NN\;|\;\exists x\in\III_{\Pin(2)}(X)\text{ such that }wx=w^{k+1}\}\subset\NN.\]
Since the restriction map $\res:R(G^{*}_{m})\to R(\Pin(2))$ sends $w_{j}\mapsto w$, $z_{k}\mapsto z$, we see that the image of the grading map
\[|\cdot|:(\N^{m},\preceq,+)\to (\NN,\le,+)\]
restricted to $I(X)$ is contained in $I_{\Pin(2)}(X)$, whose minimal element is $k(X)$.
\end{proof}

\begin{definition}
Let $X$ be a space of type $\CC$-$G^{*}_{m}$-$\SWF$. We say that $X$ is \emph{$\Pin(2)$-surjective} if the ideal $\III(X)\subset R(G^{*}_{m})$ maps surjectively onto $\III_{\Pin(2)}(X)\subset R(\Pin(2))$ under the restriction map $\res:R(G^{*}_{m})\to R(\Pin(2))$, or equivalently if the inequality (\ref{eq:comparison_of_k_invariants}) is an equality.
\end{definition}

We also introduce two secondary $k$-invariants, which will prove useful in certain contexts:

\begin{definition}
\label{def:k_invariants_upper_lower}
Let $X$ be a space of type $\CC$-$G^{\ast}_{m}$-$\SWF$. We define $\vec{\ol{k}}(X)$ (respectively, $\vec{\ul{k}}(X)$) to be the \emph{least upper bound} (respectively, \emph{greatest lower bound}) of $\mbfk(X)$ as a finite subset of the completed lattice $(\wh{\N}^{m},\preceq)$, i.e.,
\begin{align*}
    &\vec{\ol{k}}(X):=\vee\mbfk(X), & &\vec{\ul{k}}(X):=\wedge\mbfk(X).
\end{align*}
Here we use the convention that if $\mbfk(X)=\emptyset$, then
\[\vec{\ol{k}}(X)=\vec{\ul{k}}(X)=+\infty\in\wh{\N}^{m}.\]
(See Definition \ref{def:bounded_completion}.)
\end{definition}

\begin{remark}
Note that $\vec{\ol{k}}(X)\in\mbfk(X)$ or $\vec{\ul{k}}(X)\in\mbfk(X)$ if and only if $\mbfk(X)=\{\vec{k}\}$ has a unique element, in which case $\vec{\ol{k}}(X)=\vec{\ul{k}}(X)=\vec{k}$. 
\end{remark}

\begin{example}
\label{ex:k_invariants_s_0}
The simplest example of a space of type $G^{*}_{m}$-$\SWF$ is $S^{0}$ with the trivial $G^{*}_{m}$-action. In this case, we have that $\III(S^{0})=(1)$. It follows that $I(S^{0})=\N^{m}$, and hence
\begin{align*}
	&\mbfk(S^{0})=\{[\vec{0}]\}, & &\vec{\ol{k}}(S^{0})=\vec{\ul{k}}(S^{0})=[\vec{0}].
\end{align*}
\end{example}

\begin{proposition}
\label{prop:k_invariants_suspensions}
If $X$ is a space of type $\CC$-$G^{*}_{m}$-$\SWF$, then
\begin{align*}
    &\III(\Sigma^{\mbfs\wt{\CC}}X)= \III(X)\text{ for any }\mbfs\in R(\ZZ_{m})_{\geq 0}^{\sym}, & &\III(\Sigma^{\HH_{k}}X)=z_{k}\cdot\III(X),
\end{align*}
where $k\equiv 0\pmod{1}$ if $\ast=\ev$ and $k\equiv\frac{1}{2}\pmod{1}$ if $\ast=\odd$. Consequently:
\begin{align*}
    &\mbfk(\Sigma^{\mbfs\wt{\CC}}X)=\mbfk(X)\text{ for any }\mbfs\in R(\ZZ_{m})_{\geq 0}^{\sym}, & &\mbfk(\Sigma^{\HH_{k}}X)= \mbfk(X)+\vec{e}_{2k},
\end{align*}
where we use the cyclic indexing convention $\vec{e}_{j+m}:=\vec{e}_{j}$.
\end{proposition}

\begin{proof}
The proof is essentially the same as that of (\cite{Man14}, Lemma 3.4).
\end{proof}

\begin{definition}
\label{def:more_lattice_notation}
We introduce here some additional notation:
\begin{enumerate}
    \item For $i=0,\dots,m-1$, recall that $\vec{e}_{i}\in\NN^{m}$ denotes the tuple with a $1$ in the $i$-th place and zeroes in the other entries. Define $\DDD^{\ev}:\NN^{m}\to\NN^{m}$ be the unique $\NN$-linear map which sends $e_{i}\mapsto e_{2i}$ for all $i=0,\dots,m-1$.
    \item We also introduce the space
    \[\NN_{1/2}^{m}:=\{(\ell_{1/2},\ell_{3/2},\dots,\ell_{m-1/2})\;|\;\ell_{j}\in\NN\},\]
    which is isomorphic to $\NN^{m}$ as an additive lattice, but indexed by half integers $j\in\frac{1}{2}\ZZ$ with $j\equiv\frac{1}{2}\pmod{1}$, $0<j<m$. For $i=0,\dots,m-1$, let $\vec{e}_{i+1/2}\in\NN^{m}_{1/2}$ denote the tuple with a $1$ in the $(i+\frac{1}{2})$-th place and zeroes in the other entries. Let $\DDD^{\odd}:\NN_{1/2}^{m}\to\NN^{m}$ be the unique $\NN$-linear map which sends $e_{i+1/2}\mapsto e_{2i+1}$, where we use the same cyclic indexing convention as above.
    \item Given an element
	\[\mbft=\twopartdef{\sum_{k=0}^{m-1}t_{k}\xi^{2k}\in R(\ZZ_{2m})_{\geq 0}^{\ev}}{*=\ev,}{\sum_{k=0}^{m-1}t_{k+1/2}\xi^{2k+1}\in R(\ZZ_{2m})_{\geq 0}^{\odd}}{*=\odd,}\]
	we will often use $\vec{\mbft}$ to denote the tuple
	\[\vec{\mbft}=\twopartdef{(t_{0},\dots,t_{m-1})\in\NN^{m}}{\ast=\ev,}{(t_{1/2},\dots,t_{m-1/2})\in\NN_{1/2}^{m}}{\ast=\odd.}\]
\end{enumerate}
\end{definition}

\begin{example}
\label{ex:k_invariants_representation_spheres}
Let $\ast\in\{\ev,\odd\}$. From Example \ref{ex:k_invariants_s_0} and Proposition \ref{prop:k_invariants_suspensions} we can deduce that for any $\mbfs\in R(\ZZ_{m})_{\geq 0}^{\sym}$ and any
\[\mbft=\twopartdef{\sum_{k=0}^{m-1}t_{k}\xi^{2k}\in R(\ZZ_{2m})_{\geq 0}^{\ev}}{*=\ev,}{\sum_{k=0}^{m-1}t_{k+1/2}\xi^{2k+1}\in R(\ZZ_{2m})_{\geq 0}^{\odd}}{*=\odd,}\]
we have that
\[\III\big((\mbfs\wt{\CC}\oplus\mbft\HH)^{+}\big)=\twopartdef{(z_{0}^{t_{0}}\cdots z_{m-1}^{t_{m-1}})\subset R(G^{\ev}_{m})}{*=\ev,}{(z_{1/2}^{t_{1/2}}\cdots z_{m-1/2}^{t_{m-1/2}})\subset R(G^{\odd}_{m})}{*=\odd,}\]
and consequently:
\begin{align*}
	&\mbfk\big((\mbfs\wt{\CC}\oplus\mbft\HH)^{+}\big)=\{[\DDD^{*}(\vec{\mbft})]\}, & &\vec{\ol{k}}\big((\mbfs\wt{\CC}\oplus\mbft\HH)^{+}\big)=\vec{\ul{k}}\big((\mbfs\wt{\CC}\oplus\mbft\HH)^{+}\big)=[\DDD^{*}(\vec{\mbft})].
\end{align*}
\end{example}

\begin{example}
\label{ex:ideal_trivial_Z_m_action}
If $X$ is a space of type $\CC$-$G^{\ev}_{m}$-$\SWF$ such that $\ZZ_{m} < G^{\ev}_{m}$ acts trivially on $X$, then $\wt{K}_{G^{\ev}_{m}}(X)\cong\wt{K}_{\Pin(2)}(X)\otimes_{\ZZ}R(\ZZ_{m})$ by Fact \ref{fact:trivial_G_action}. By analyzing the commutative diagram
\begin{center}
    \begin{tikzcd}
    \wt{K}_{G^{\ev}_{m}}(X) \arrow[r, "\iota^{*}"] \arrow[d, "\cong"]
    & \wt{K}_{G^{\ev}_{m}}(X^{S^{1}}) \arrow[d, "\cong"] \\
    \wt{K}_{\Pin(2)}(X)\otimes R(\ZZ_{m}) \arrow[r, "\iota^{*}\otimes \id" ]
    & \wt{K}_{\Pin(2)}(X^{S^{1}})\otimes R(\ZZ_{m}),
    \end{tikzcd}
\end{center}
we can conclude that
\[\III(X)\cong\III_{\Pin(2)}(X)\otimes R(\ZZ_{m}).\]
Therefore
\begin{align*}
	&\mbfk(X)=\{[k_{\Pin(2)}(X)\cdot\vec{e}_{0}]\}, & &\vec{\ol{k}}(X)=\vec{\ul{k}}(X)=[k_{\Pin(2)}(X)\cdot\vec{e}_{0}].
\end{align*}
On the other hand, if $X$ is a space of type $\CC$-$G^{\odd}_{m}$-$\SWF$ such that $\ZZ_{2m}<G^{\odd}_{m}$ acts trivially on $X$, then $X=X^{S^{1}}$. Hence $\III(X)=(1)$, and consequently $\mbfk(X)=\{[\vec{0}]\}$.

However, suppose that $m$ is odd, and let $X$ be a space of type $\CC$-$G^{\odd}_{m}$-$\SWF$ such that the action of $\mu\in\ZZ_{2m}<G^{\odd}_{m}$ coincides with the action of $-1\in S^{1}<G^{\odd}_{m}$. Then the $\ZZ_{m}$-subgroup $\<-\mu\><G^{\odd}_{m}$ acts trivially on $X$. By the same argument as above, we have that
\[\III(X)\cong\III_{\Pin(2)}(X)\otimes R(\ZZ_{m}),\]
under the embedding $\III_{\Pin(2)}(X)\hookrightarrow R(G^{\odd}_{m})$ induced by the embedding
\begin{align*}
    R(\Pin(2))&\hookrightarrow R(G^{\odd}_{m}) \\
    w&\mapsto w_{0}, \\
    z&\mapsto z_{m/2}.
\end{align*}
Hence
\begin{align*}
	&\mbfk(X)=\{[k_{\Pin(2)}(X)\cdot\vec{e}_{0}]\} & &\vec{\ol{k}}(X)=\vec{\ul{k}}(X)=[k_{\Pin(2)}(X)\cdot\vec{e}_{0}]
\end{align*}
as above.
\end{example}

\begin{example}
\label{ex:wedge_sum_free_G_space}
Let $X$ be a space of type $\CC$-$G^{*}_{m}$-$\SWF$, and let $X'$ be a $G^{*}_{m}$-space such that $X^{S^{1}}=\{\pt\}$. Then $X\vee X'$ is also a space of type $\CC$-$G^{*}_{m}$-$\SWF$, and $\III(X\vee X')=\III(X)$. Consequently $\mbfk(X\vee X')=\mbfk(X)$.
\end{example}

We now outline some properties of $\mbfk(X)$ analogous to those in the $\Pin(2)$-setting.

\begin{proposition}
\label{prop:k_invariants_stable_homotopy_equivalence}
Let $X$ and $X'$ be spaces of type $\CC$-$G^{*}_{m}$-$\SWF$, and suppose that there exists a based, $G^{*}_{m}$-equivariant homotopy equivalence from $\Sigma^{r\RR}X$ to $\Sigma^{r\RR}X'$ for some $r\geq 0$. Then $I(X)=I(X')$, and hence $\mbfk(X)=\mbfk(X')$.
\end{proposition}

\begin{proof}
This is implied by the argument given in the proof of (\cite{Man14}, Lemma 3.8), which applies in this situation without much change.
\end{proof}

\begin{proposition}
\label{prop:k_invariants_s1_fixed_point_homotopy_equivalence}
Let $X$ and $X'$ be spaces of type $\CC$-$G^{*}_{m}$-$\SWF$ at the same level $\mbfs\in R(\ZZ_{m})_{\geq 0}^{\sym}$, and suppose that $f:X\to X'$ is a $G^{*}_{m}$-equivariant map such that the induced map $f^{S^{1}}:X^{S^{1}}\to(X')^{S^{1}}$ on $S^{1}$-fixed point sets is a $G^{*}_{m}$-homotopy equivalence. Then $I(X)\supseteq I'(X)$. In particular:
\begin{enumerate}
    \item For each $\vec{k}'\in\mbfk_{G^{*}_{m}}(X')$:
    \begin{enumerate}
        \item $\vec{k}\not\succ\vec{k}'$ for all $\vec{k}\in\mbfk_{G^{*}_{m}}(X)$, and
        \item there exists some $\vec{k}\in\mbfk_{G^{*}_{m}}(X)$ such that $\vec{k}\preceq\vec{k}'$.
    \end{enumerate}
    \item $\vec{\ul{k}}(X)\preceq \vec{\ol{k}}(X')$.
\end{enumerate}
\end{proposition}

\begin{proof}
It suffices to show that $\III(X)\supseteq\III(X')$. Again the argument given in the proof of (\cite{Man14}, Lemma 3.9) applies here without much change.
\end{proof}

We have the following definition, inspired by (\cite{Stoff20}, Definition 2.7):

\begin{definition}
\label{def:unstable_local_equivalence}
Let $X$ and $X'$ be spaces of type $\CC$-$G^{*}_{m}$-$\SWF$. We say that $X_{1}$, $X_{2}$ are \emph{locally equivalent} if there exist $G^{*}_{m}$-equivariant maps
\[X\stackrel[g]{f}{\rightleftarrows}X'\]
such that the induced maps $f^{S^{1}}:X^{S^{1}}\to(X')^{S^{1}}$, $g^{S^{1}}:(X')^{S^{1}}\to X^{S^{1}}$ on the $S^{1}$-fixed point sets are $G^{*}_{m}$-equivariant homotopy equivalences.
\end{definition}

The following Corollary follows immediately from Proposition \ref{prop:k_invariants_s1_fixed_point_homotopy_equivalence}:

\begin{corollary}
\label{cor:k_invariants_local_equivalence}
Let $X$ and $X'$ be spaces of type $\CC$-$G^{*}_{m}$-$\SWF$ such that $X$ and $X'$ are locally equivalent. Then $I(X)=I(X')$ and $\mbfk(X)=\mbfk(X')$.
\end{corollary}

Before stating the following proposition, we introduce some notation. For any element $\mbfs=\sum_{j=0}^{m-1}s_{j}\zeta^{j}\in R(\ZZ_{m})_{\geq 0}^{\sym}$, we define
\[\vec{\mbfs}:=(s_{0},\dots,s_{m-1})\in\NN^{m}.\]
For any two such representations $\mbfs,\mbfs'\in R(\ZZ_{m})_{\geq 0}^{\sym}$, we write $\vec{\mbfs}\preceq\vec{\mbfs}'$ if $s_{j}\le s_{j}'$ for all $j=0,\dots,m-1$.

\begin{proposition}
\label{prop:k_invariants_pin(2)_fixed_point_homotopy_equivalence}
Let $X$ and $X'$ be spaces of type $\CC$-$G^{*}_{m}$-$\SWF$ at levels $\mbfs,\mbfs'\in R(\ZZ_{m})_{\geq 0}^{\sym}$, respectively, such that $\vec{\mbfs}\preceq\vec{\mbfs}\,'$. Suppose that $f:X\to X'$ is a $G^{*}_{m}$-equivariant map such that the induced map $f^{\Pin(2)}:X^{\Pin(2)}\to(X')^{\Pin(2)}$ on $\Pin(2)$-fixed point sets is a $G^{*}_{m}$-homotopy equivalence. Then:
\begin{enumerate}
    \item For each $\vec{k}'\in\mbfk_{G^{*}_{m}}(X')$:
    \begin{enumerate}
        \item $\vec{k}\not\succ\vec{k}'+(\vec{\mbfs}\,'-\vec{\mbfs})$ for all $\vec{k}\in\mbfk_{G^{*}_{m}}(X)$, and
        \item there exists some $\vec{k}\in\mbfk_{G^{*}_{m}}(X)$ such that $\vec{k}\preceq\vec{k}'+(\vec{\mbfs}\,'-\vec{\mbfs})$.
    \end{enumerate}
    \item $\vec{\ul{k}}(X)\preceq \vec{\ol{k}}(X')+(\vec{\mbfs}\,'-\vec{\mbfs})$.
\end{enumerate}
\end{proposition}

\begin{proof}
By Lemma \ref{lemma:comparing_minima_subsets} it suffices to show that
\[\III(X)\supseteq\mbfw^{\vec{\mbfs}\,'-\vec{\mbfs}}\cdot\III(X'),\]
so that
\[I(X)\supseteq I(X')+(\vec{\mbfs}\,'-\vec{\mbfs}).\]
Note that we have a commutative diagram
\begin{center}
    \begin{tikzcd}
    \wt{K}_{G^{*}_{m}}(X') \arrow[r, "f^{*}"] \arrow[d, "(\iota')^{*}"]
    & \wt{K}_{G^{*}_{m}}(X) \arrow[d, "\iota^{*}"] \\
    \wt{K}_{G^{*}_{m}}((X')^{S^{1}}) \arrow[r, "(f^{S^{1}})^{*}" ] \arrow[d,"\cdot\mbfw^{\vec{\mbfs}\,'}"]
    & \wt{K}_{G^{*}_{m}}(X^{S^{1}}) \arrow[d,"\cdot\mbfw^{\vec{\mbfs}}"] \\
    \wt{K}_{G^{*}_{m}}((X')^{\Pin(2)}) \arrow[r, "\cdot 1" ]
    & \wt{K}_{G^{*}_{m}}(X^{\Pin(2)}),
    \end{tikzcd}
\end{center}
where the bottom four groups are all isomorphic to $R(G^{*}_{m})$. By a similar argument as in the proof of (\cite{Man14}, Lemma 3.10), one can show that $(f^{S^{1}})^{*}$ must be multiplication by $\mbfw^{\vec{\mbfs}\,'-\vec{\mbfs}}\in R(G^{*}_{m})$, from which the result follows.
\end{proof}

In analogy with (\cite{Man14}, Definition 3.5), we make the following definition:

\begin{definition}
\label{def:K_G_split}
A space $X$ of type $\CC$-$G^{*}_{m}$-$\SWF$ is called \emph{$K_{G^{*}_{m}}$-split} if $\III(X)$ is a principal ideal generated by a single monomial in the $z_{k}$-variables.
\end{definition}

\begin{example}
By Example \ref{ex:k_invariants_representation_spheres} any $G^{*}_{m}$-representation sphere is $K_{G^{*}_{m}}$-split.
\end{example}

\begin{remark}
A space $X$ of type $\CC$-$G^{*}_{m}$-$\SWF$ which is $K_{G^{*}_{m}}$-split is $K_{\Pin(2)}$-split in the sense of \cite{Man14}, when considered as a $\Pin(2)$-space. However, the converse is not necessarily true.
\end{remark}

\begin{proposition}
\label{prop:k_invariants_kg_split}
Let $X$ and $X'$ be spaces of type $\CC$-$G^{*}_{m}$-$\SWF$ at levels $\mbfs,\mbfs'\in R(\ZZ_{m})_{\geq 0}^{\sym}$, respectively, such that $\vec{\mbfs}\preceq\vec{\mbfs}'$,  $s_{0}<s_{0}'$, and $X$ is $K_{G^{*}_{m}}$-split. Suppose that $f:X\to X'$ is a $G^{*}_{m}$-equivariant map whose $\Pin(2)$-fixed point set map is a $G^{*}_{m}$-homotopy equivalence. Then
In particular:
\begin{enumerate}
    \item For each $\vec{k}'\in\mbfk(X')$:
    \begin{enumerate}
        \item $\vec{k}+\vec{e}_{0}\not\succ\vec{k}'+(\vec{\mbfs}\,'-\vec{\mbfs})$ for all $\vec{k}\in\mbfk(X)$, and
        \item there exists some $\vec{k}\in\mbfk(X)$ such that $\vec{k}+\vec{e}_{0}\preceq\vec{k}'+(\vec{\mbfs}\,'-\vec{\mbfs})$.
    \end{enumerate}
    \item $\vec{\ul{k}}(X)+\vec{e}_{0}\preceq \vec{\ol{k}}(X')+(\vec{\mbfs}\,'-\vec{\mbfs})$.
\end{enumerate}
\end{proposition}

In order to prove Proposition \ref{prop:k_invariants_kg_split}, the following lemma will be helpful:

\begin{lemma}
\label{lemma:must_be_w_0}
Let $x,y\in R(G^{\ev}_{m})$, and suppose there exist integers $b_{0},\dots,b_{m-1}\geq 0$ such that
\begin{equation}
\label{eq:polynomial_lemma_even}
    w_{0}x=z_{0}^{b_{0}}z_{1}^{b_{1}}\cdots z_{m-1}^{b_{m-1}}y.
\end{equation}
Then $w_{0}|y$. Similarly, if $x,y\in R(G^{\odd}_{m})$ and $b_{1/2},\dots,b_{m-1/2}\geq 0$ are such that
\begin{equation}
\label{eq:polynomial_lemma_odd}
    w_{0}x=z_{0}^{b_{0}}z_{1/2}^{b_{1/2}}\cdots z_{m-1/2}^{b_{m-1/2}}y,
\end{equation}
then $w_{0}|y$.
\end{lemma}

\begin{proof}
We prove the $G^{\ev}_{m}$ case, as the $G^{\odd}_{m}$ case is similar. For any element $v\in R(G^{\ev}_{m})$ and any $g\in G^{\ev}_{m}$, let $\tr_{g}(v)\in\CC$ denote the trace of $v$ at the element $g$. Now let $\phi\in(0,2\pi)$ be an irrational multiple of $\pi$. A simple calculation shows that
\begin{align*}
    &\tr_{\gamma^{\ell}e^{i\phi}}(w_{k})=1-\omega_{m}^{k\ell}, &
    &\tr_{\gamma^{\ell}e^{i\phi}}(z_{k})=(e^{i\phi}-\omega_{m}^{k\ell})(e^{-i\phi}-\omega_{m}^{k\ell})
\end{align*}
for all $0\le\ell\le m-1$. Note that $1-\omega_{m}^{k\ell}=0$ if and only if $k\ell$ is a multiple of $m$, whereas $(e^{i\phi}-\omega_{m}^{k\ell})(e^{-i\phi}-\omega_{m}^{k\ell})$ can never be equal to zero for any $\ell$ by our assumption on $\phi$. In particular, $\tr_{\gamma^{\ell}e^{i\phi}}(w_{0})=0$ for all $\ell$. Since $a_{0}>0$, we see that that Equation \ref{eq:polynomial_lemma_even} is only satisfied if 
\begin{equation}
\label{eq:polynomial_lemma_proof_trace}
    \tr_{\gamma^{\ell}e^{i\phi}}(y)=0\text{ for all }0\le\ell\le m-1.
\end{equation}
Using the relations in $R(G^{\ev}_{m})$, we can write $y$ in the form
\[y=\alpha(\zeta)+\wt{\alpha}(\zeta)\wt{c}+\sum_{k=1}^{N}\alpha_{k}(\zeta)h^{k}\]
for some $N\geq 1$ and some $\alpha(\zeta)$, $\wt{\alpha}(\zeta)$, $\alpha_{k}(\zeta)\in R(\ZZ_{m})$. For each $\ell=0,\dots,m-1$, denote by $p_{\ell}(X)$ the polynomial
\[p_{\ell}(X):=\alpha_{0}(\omega_{m}^{\ell})+\wt{\alpha}_{0}(\omega_{m}^{\ell})+\sum_{k=1}^{N}\alpha_{k}(\omega_{m}^{\ell})X^{k}\in\QQ(\omega_{m})[X].\]
Equation \ref{eq:polynomial_lemma_proof_trace} implies that $p_{\ell}(e^{i\phi}+e^{-i\phi})=p_{\ell}(2\cos(\phi))=0$ for all $\ell=0,\dots,m-1$ and all $\phi\in(0,\phi)$ not a rational multiple of $2\pi$. In particular there exists some such $\phi\in(0,2\pi)$ such that $\cos(\phi)$ is transcendental over $\QQ(\omega_{m})$, and so it follows that $\alpha_{0}(\omega^{\ell}_{m})+\wt{\alpha}_{0}(\omega^{\ell}_{m})=0$ and $\alpha_{k}(\omega^{\ell}_{m})=0$ for all $k=1,\dots,N$ and all $0\le\ell\le m-1$. By (\cite{Bry97}, Lemma 3.5) it follows that $\wt{\alpha}_{0}(\zeta)\equiv -\alpha_{0}(\zeta)$ and $\alpha_{k}(\zeta)\equiv 0$ for all $k$. Thus we can write
\[y=\alpha_{0}(\zeta)-\alpha_{0}(\zeta)\wt{c}=(1-\wt{c})\alpha_{0}(\zeta)=w_{0}\alpha_{0}(\zeta),\]
and so $w_{0}|y$.
\end{proof}

\begin{proof}[Proof of Proposition \ref{prop:k_invariants_kg_split}]
It suffices to show that
\[w_{0}\cdot\III(X)\supseteq\mbfw^{\vec{\mbfs}\,'-\vec{\mbfs}}\cdot\III(X'),\]
so that
\[I(X)+\vec{e}_{0}\supseteq I(X')+(\vec{\mbfs}\,'-\vec{\mbfs}).\]
Without loss of generality assume $*=\ev$, as the argument for the $*=\odd$ case is essentially identical. Now let $x\in\III(X')$. From the proof of Proposition \ref{prop:k_invariants_pin(2)_fixed_point_homotopy_equivalence}, we see that $w_{0}^{s'_{0}-s_{0}}\cdots w_{m-1}^{s'_{m-1}-s_{m-1}}x\in\III(X)$. Since $X$ is $K_{G^{\ev}_{m}}$-split, there exist some $t_{0},\dots,t_{m-1}\geq 0$ such that $\III(X)=(z_{0}^{t_{0}}\cdots z_{m-1}^{t_{m-1}})$, and so
\[w_{0}^{s'_{0}-s_{0}}\cdots w_{m-1}^{s'_{m-1}-s_{m-1}}x=z_{0}^{t_{0}}\cdots z_{m-1}^{t_{m-1}}y\]
for some $y\in R(G^{\ev}_{m})$. But since $s'_{0}>s_{0}$, it follows from Lemma \ref{lemma:must_be_w_0} that $w_{0}|y$, from which the proposition follows.
\end{proof}

We provide some examples of non-$K_{G^{*}_{m}}$-split spaces:

\begin{example}
\label{ex:k_invariants_augmentation}
Let $Z$ be a finite $G^{*}_{m}$-complex such that $G^{*}_{m}$ acts freely on $Z$ with quotient $Q=Z/G^{*}_{m}$, and let $\wt{\Sigma}Z$ denote the \emph{unreduced} suspension of $Z$, with one of the two cone points being the basepoint. Then by Facts \ref{fact:long_exact_sequence} and \ref{fact:classifying_space_2}, the image of
\[\wt{K}_{G^{*}_{m}}(\wt{\Sigma}Z)\xrightarrow{\iota^{*}}\wt{K}_{G^{*}_{m}}((\wt{\Sigma}Z)^{S^{1}})\cong R(G^{*}_{m})\]
is equal to the kernel of the map $R(G^{*}_{m})\xrightarrow{\pi} K(Q)$.

In particular if $Z=G^{*}_{m}$ acting on itself by left multiplication, then $Q=\text{pt}$ and $R(G^{*}_{m})\xrightarrow{\pi} K(Q)\cong\ZZ$ is the augmentation map, hence 
\[\III(\wt{\Sigma}G^{*}_{m})=\frak{a}=\twopartdef{(w_{0},\dots,w_{m-1},z_{0},\dots,z_{m-1})}{*=\ev,}{(w_{0},\dots,w_{m-1},z_{1/2},\dots,z_{m-1/2})}{*=\odd,}\]
and therefore
\[\mbfk(\wt{\Sigma}G^{*}_{m})=\{[\vec{e}_{0}],\dots,[\vec{e}_{m-1}]\}\subset\N^{m}.\]
\end{example}

\begin{example}
\label{ex:k_invariants_quotients}
Let $Z$, $Q$ be as in Example \ref{ex:k_invariants_augmentation}, and suppose additionally that $Z$ has a free right $G^{*}_{m}$-action as well. Let $H\subset G^{*}_{m}$ be a finite subgroup, and let $Z':=Z/H$ denote the quotient of $Z$ by the action of $H$ on the right. We see that the image of
\[\iota^{*}:\wt{K}_{G^{*}_{m}}(\wt{\Sigma}Z')\to\wt{K}_{G^{*}_{m}}((\wt{\Sigma}Z')^{S^{1}})\]
is equal to the kernel of the connecting homomorphism
\[R(G^{*}_{m})\cong \wt{K}_{G^{*}_{m}}((\wt{\Sigma}Z')^{S^{1}})\xrightarrow{\delta}\wt{K}^{1}_{G^{*}_{m}}(\Sigma(Z'_{+}))\cong K_{G^{*}_{m}}(Z').\]
We can identify $\delta$ with the map $R(G^{*}_{m})\to K_{G^{*}_{m}}(Z')$ induced by the collapse map $Z'\to \pt$.
\end{example}

\begin{example}
\label{ex:cosets_cyclic}
Continuing from the previous example, suppose $Z=G^{*}_{m}$ with $G^{*}_{m}$-action given by left multiplication, so that $Z'=G^{*}_{m}/H$ is the left coset space. We can write the connecting homomorphism as a map
\[R(G^{*}_{m})\cong \wt{K}_{G^{*}_{m}}((\wt{\Sigma}(G^{*}_{m}/H))^{S^{1}})\xrightarrow{\delta} \wt{K}^{1}_{G^{*}_{m}}(\Sigma(G^{*}_{m}/H)_{+})\cong K_{G^{*}_{m}}(G^{*}_{m}/H)\cong R(H),\]
where the final isomorphism comes from Fact \ref{fact:homogenous_spaces}, and is given by restricting to the $H$-representation over the identity coset. Then $\delta$ can be identified with the restriction map $\res^{G^{*}_{m}}_{H}:R(G^{*}_{m})\to R(H)$, and hence
\[\III(\wt{\Sigma}(G^{*}_{m}/H))=\ker(\res^{G^{*}_{m}}_{H}).\]

Consider the following examples of specific subgroups $H\subset G^{*}_{m}$:
\begin{enumerate}
    \item Let $*=\ev$, let $\omega_{m}=e^{2\pi i/m}$, let $0\le a\le m-1$, and consider the subgroup $\<\gamma\omega_{m}^{-a}\>\cong\ZZ_{m}\subset G^{\ev}_{m}$. Define $Z_{a,m}$ to be the homogenous space
    \[Z_{a,m}:=G^{\ev}_{m}/\<\gamma\omega_{m}^{-a}\>.\]
    We can identify $Z_{a,m}$ with a copy of $\Pin(2)$, where $\Pin(2)\subset G^{\ev}_{m}$ acts by left multiplication, and $\gamma\in G^{\ev}_{m}$ acts on the right by $\omega_{m}^{a}$. From the above observation we have that:
    \begin{equation*}
        \III(\wt{\Sigma}Z_{a,m})=\ker\big(R(G^{\ev}_{m})\xrightarrow{\res}R(\<\gamma\omega_{m}^{-a}\>)\big).
    \end{equation*}
    Writing $R(\<\gamma\omega_{m}^{-a}\>)\cong R(\ZZ_{m})=\ZZ[\alpha]/(\alpha^{m}-1)$, one can check that
    \begin{align*}
        &\res(\zeta)=\alpha, & &\res(\wt{c})=1, &
        &\res(h)=\alpha^{a}+\alpha^{-a},
    \end{align*}
    and hence
    \begin{align*}
        &\res(w_{k})=1-\alpha^{k}, &
        &\res(z_{k})=(1-\alpha^{k+a})(1-\alpha^{k-a}).
    \end{align*}
    It follows that
    \begin{align*}
        &\III(\wt{\Sigma}Z_{a,m})=(w_{0},z_{a},z_{m-a}), & &\mbfk(\wt{\Sigma}Z_{a,m})=\{[\vec{e}_{0}],[\vec{e}_{2a}],[\vec{e}_{m-2a}]\}.
    \end{align*}
    \item Let $*=\odd$, let $a\in\frac{1}{2}\ZZ\setminus\ZZ$ be such that $\frac{1}{2}\le a\le m-\frac{1}{2}$, and consider the subgroup $\<\mu\omega_{m}^{-a}\>\cong\ZZ_{m}\subset G^{\odd}_{m}$. Define $Z_{a,m}$ to be the homogenous space
    \[Z_{a,m}:=G^{\odd}_{m}/\<\mu\omega_{m}^{-a}\>.\]
    Similarly as above, we can identify $Z_{a,m}$ with a copy of $\Pin(2)$, where $\Pin(2)\subset G^{\odd}_{m}$ acts by left multiplication, and $\mu$ acts on the right by $\omega_{m}^{a}$. Again we have that:
    \[\III(\wt{\Sigma}Z_{a,m})=\ker\Big(R(G^{\odd}_{m})\xrightarrow{\res}R(\<\mu\omega_{m}^{-a}\>)\Big).\]
    Writing $R(\<\mu\omega_{m}^{-a}\>)\cong R(\ZZ_{m})=\ZZ[\alpha]/(\alpha^{m}-1)$ as above, one can check that
    \begin{align*}
        &\res(\xi^{2})=\alpha, &
        &\res(\wt{c})=1, &
        &\res(\xi h)=\alpha^{\frac{1}{2}+a}+\alpha^{\frac{1}{2}-a},
    \end{align*}
    and hence
    \begin{align*}
        &\res(w_{k})=1-\alpha^{k}, &
        &\res(z_{k+1/2})=(1-\alpha^{k+a+\frac{1}{2}})(1-\alpha^{k-a+\frac{1}{2}}).
    \end{align*}
    Therefore
    \begin{align*}
        &\III(\wt{\Sigma}Z_{a,m})=(w_{0},z_{a},z_{m-a}), & &\mbfk(\wt{\Sigma}Z_{a,m})=\{[\vec{e}_{0}],[\vec{e}_{2a}],[\vec{e}_{m-2a}]\},
    \end{align*}
    as in the case $\ast=\ev$.
    
    In particular, consider the case where $m=2$ and $a=\frac{1}{2}$ or $\frac{3}{2}$, which correspond to the spaces $Z_{\frac{1}{2},2}$ and $Z_{\frac{3}{2},2}$, respectively. Then
    \begin{align*}
    	&\III(\wt{\Sigma}Z_{\frac{1}{2},2})=\III(\wt{\Sigma}Z_{\frac{3}{2},2})=(w_{0},z_{1/2},z_{3/2}), & &\mbfk(\wt{\Sigma}Z_{\frac{1}{2},2})=\mbfk(\wt{\Sigma}Z_{\frac{3}{2},2})=\{[\vec{e}_{0}],[\vec{e}_{1}]\}.
    \end{align*}
    \item Let $m=2$, $*=\odd$, and consider the subgroup $\<\mp\mu j\>\cong\ZZ_{2}\subset G^{\odd}_{2}$. Define $Z_{\pm j}$ to be the homogenous space
    \[Z_{\pm j}:=G^{\odd}_{2}/\<\mp\mu j\>.\]
    We can identify $Z_{\pm j}$ with a copy of $\Pin(2)$, where $\Pin(2)\subset G^{\odd}_{2}$ acts by left multiplication, and $\mu$ acts by multiplication by $\pm j$ on the right. Again we have that
    \[\III(\wt{\Sigma}Z_{\pm j})=\ker\Big(R(G^{\odd}_{2})\xrightarrow{\res}R(\<\mp\mu j\>)\Big).\]
    Writing $R(\<\mp\mu j\>)\cong R(\ZZ_{2})=\ZZ[\alpha]/(\alpha^{2}-1)$, one can show that
    \begin{align*}
        &\res(\xi^{2})=\res(\wt{c})=\alpha, &
        &\res(\xi h)=1+\alpha,
    \end{align*}
	and hence:
	\begin{align*}
    	&\res(w_{0})=1-\alpha, & &\res(w_{1})=\res(z_{1/2})=\res(z_{3/2})=0.
	\end{align*}
	Therefore 
        \begin{align*}
            &\III(\wt{\Sigma}Z_{\pm j})=(w_{1},z_{1/2},z_{3/2}), & &\mbfk(\wt{\Sigma}Z_{\pm j})=\{[\vec{e}_{1}]\}.
        \end{align*}
\end{enumerate}
\end{example}

\begin{example}
\label{ex:multiple_cosets_cyclic}
More generally, suppose that
\[Z'=\coprod_{k=1}^{n}G^{*}_{m}/H_{k}.\]
for some finite collection of subgroups $H_{1},\dots,H_{n}$. Then $K_{G^{*}_{m}}(Z')\cong\bigoplus_{k=1}^{n}R(H_{k})$, the map on $G^{*}_{m}$-equivariant K-theory induced by the collapse map $Z\to\pt$ is given by
\begin{align*}
    R(G^{*}_{m})&\to\bigoplus_{k=1}^{n}R(H_{k}) \\
    x&\mapsto\oplus_{k=1}^{n}\res^{G^{*}_{m}}_{H_{k}}(x),
\end{align*}
and hence
\[\III(\wt{\Sigma}Z')=\cap_{k=1}^{n}\ker(\res^{G^{*}_{m}}_{H_{k}}).\]
Consider the following examples:
\begin{enumerate}
    \item Let $\ast\in\{\ev,\odd\}$, and let $a_{1},\dots,a_{n}\in\frac{1}{2}\ZZ$ be such that, for all $k=1,\dots,n$: $0\le a_{k}\le m-\frac{1}{2}$, and $a_{k}\equiv 0\pmod{1}$ if $\ast=\ev$ and $a_{k}\equiv\frac{1}{2}\pmod{1}$ if $\ast=\odd$. Consider the $G^{\ev}_{m}$-space
    \[Z_{a_{1},\dots,a_{n};m}:=Z_{a_{1},m
    }\amalg\cdots\amalg Z_{a_{n},m
    },\]
    where $Z_{a,m}$ is as in Example \ref{ex:cosets_cyclic}. From the above calculations, we can conclude that
    \begin{align*}
        &\III(\wt{\Sigma}Z_{a_{1},\dots,a_{n};m})=(w_{0},z_{a},z_{m-a}) & &\mbfk(\wt{\Sigma}Z_{a_{1},\dots,a_{n};m})=\{[\vec{e}_{0}],[\vec{e}_{2a}],[\vec{e}_{m-2a}]\}
    \end{align*}
    if there exists $a\in\frac{1}{2}\ZZ$ such that $a_{k}\equiv\pm a\pmod{m}$ for all $k=1,\dots,n$, and otherwise
    \begin{align*}
        &\III(\wt{\Sigma}Z_{a_{1},\dots,a_{n};m})=(w_{0}), & &\mbfk(\wt{\Sigma}Z_{a_{1},\dots,a_{n};m})=\{[\vec{e}_{0}]\}.
    \end{align*}
    In particular if $\ast=\odd$ and $m=2$ (so that $a_{k}=\frac{1}{2}$ or $\frac{3}{2}$ for all $k=1,\dots,n$) then
    \begin{align*}
        &\III(\wt{\Sigma}Z_{a_{1},\dots,a_{n};2})=(w_{0},z_{1/2},z_{3/2}) & &\mbfk(\wt{\Sigma}Z_{a_{1},\dots,a_{n};2})=\{[\vec{e}_{0}],[\vec{e}_{1}]\}.
    \end{align*}
    \item Let $\varepsilon_{1},\dots,\varepsilon_{n}\in\{\pm 1\}$ and let
    \[Z_{\varepsilon_{1}j,\dots,\varepsilon_{n}j}:=Z_{\varepsilon_{1}j}\amalg\cdots\amalg Z_{\varepsilon_{n}j},\]
    where $Z_{\pm j}$ is as in Example \ref{ex:cosets_cyclic}. Then from the above calculation we have that
    \begin{align*}
        &\III(\wt{\Sigma}Z_{\varepsilon_{1}j,\dots,\varepsilon_{n}j})=(w_{1},z_{1/2},z_{3/2}), & &\mbfk(\wt{\Sigma}Z_{\varepsilon_{1}j,\dots,\varepsilon_{n}j})=\{[\vec{e}_{1}]\}.
    \end{align*}
\end{enumerate}
\end{example}

\begin{example}
\label{ex:torus_cyclic}
Consider the spaces
\begin{align*}
    &T_{m}^{\ev}:=S^{1}\times jS^{1}\times\ZZ_{m}\subset \HH\times\ZZ_{m}, & &T_{m}^{\odd}:=S^{1}\times jS^{1}\times_{\ZZ_{2}}\ZZ_{2m}\subset\HH\times_{\ZZ_{2}}\ZZ_{m},
\end{align*}
each of which are topologically the disjoint union of $m$ 2-tori, and are endowed with actions of $G^{\ev}_{m}$ and $G^{\odd}_{m}$, respectively, via restricting the natural left actions of $G^{\ev}_{m}$ and $G^{\odd}_{m}$ on $\HH\times\ZZ_{m}$ and $\HH\times_{\ZZ_{2}}\ZZ_{2m}$, respectively. 

We will consider the following quotients of $T^{*}_{m}$:
\begin{enumerate}
    \item Suppose $*=\ev$, let $0\le a\le m-1$, and define
    \[T_{a,m}:=T^{\ev}_{m}/\<\gamma\omega_{m}^{-a}\>.\]
    Then $T_{a,m}$ is topologically a single 2-torus with the action of $\Pin(2)\subset G^{\ev}_{m}$ given by the canonical left action of $\Pin(2)$ on $S^{1}\times jS^{1}\subset\HH$ as in (\cite{Man14}, Example 3.7), and the action of $\gamma$ coincides with the action of $\omega_{m}^{a}$ on the right. Note that the quotient
    \[Q=\Pin(2)\backslash T_{a,m}\]
    can be identified with $S^{1}$, on which we have a residual action on the left by the quotient group
    \[\Pin(2)\backslash G^{\ev}_{m}\cong\ZZ_{m}=\<[\gamma]\>.\]
    In order to determine the residual $\ZZ_{m}$-action on $Q$, note that the quotient map $T_{a,m}\to Q$ restricted to the submanifold
    \[\wt{Q}:=\{(1,je^{i\phi})\}\approx S^{1} \subset T_{a,m}\]
    induces a two-to-one covering map $f:\wt{Q}\to Q$. We endow $Q$ with coordinates $\{e^{i\phi}\}$ so that the map $f:\wt{Q}\to Q$ is given by $(1,je^{i\phi})\mapsto e^{2i\phi}$. We see that the action of $[\gamma]\in\Pin(2)\backslash G^{\ev}_{m}$ on $Q=\wt{Q}/\sim$ is given by
    \[[\gamma]\cdot[1,je^{i\phi}]=[\omega^{-a}_{m}\gamma]\cdot[1,je^{i\phi}]=[\omega^{-a}_{m}\omega^{a}_{m},\omega^{-a}_{m}je^{i\phi}\omega^{a}_{m}]=[1,j\omega_{m}^{2a}e^{i\phi}],\]
    or equivalently
    \[[\gamma]\cdot e^{i\phi}=\omega_{m}^{4a}e^{i\phi}.\]
    Therefore the order of the action of $[\gamma]$ on $Q$ is equal to $m/d$, where $d:=(4a,m)$, and so by Fact \ref{fact:trivial_G_action} we have that
    \[K_{G^{\ev}_{m}}(T_{a,m})\cong K_{\ZZ_{m}}(Q)\cong R(\ZZ_{d})\otimes K_{\ZZ_{m/d}}(Q)\cong R(\ZZ_{d})=R(\<[\gamma^{m/d}]\>),\]
    where the isomorphism $K_{\ZZ_{m/d}}(Q)\cong\ZZ$ follows from Fact \ref{fact:free_G_action} and the observation that $\ZZ_{m/d}\cong\ZZ_{m}/\ZZ_{d}$ acts freely on $Q$.

    Next let $V$ be a complex $G^{\ev}_{m}$-representation, let $q^{*}V\to T_{a,m}$ be the pull-back bundle induced by the collapse map $q:T_{a,m}\to\pt$, and let $E:=\Pin(2)\backslash q^{*}V$ be the bundle over $Q$ obtained by quotienting out by the left $\Pin(2)$-action. From our calculation above, in order to understand the image of $[V]$ under the connecting homomorphism 
    \[\delta:R(G^{\ev}_{m})\to K_{G^{\ev}_{m}}(T_{a,m})\cong R(\ZZ_{d}),\]
    it suffices to determine the induced action of $[\gamma^{m/d}]\in\ZZ_{m}$ on the fiber $E_{x}$ for some (or any) point $x\in Q$. Without loss of generality, we can take $x=1\in Q$. We can write any element of $E$ as an equivalence class $[e^{i\theta},je^{i\phi},v]$ for some $v\in V$, where
    \[[e^{i\theta},je^{i\phi},v]=[ge^{i\theta},gje^{i\phi},g\cdot v]\]
    for any $g\in\Pin(2)$. Note that any element $\eta\in E_{1}\subset E$ has a unique presentation of the form $\eta=[1,j,v]$. We see that
    \begin{align*}
    	[\gamma^{m/d}]\cdot[1,j,v]&=[\omega^{-am/d}_{m}\gamma^{m/d}]\cdot[1,j,v] \\
    	&=[1,j\omega^{2am/d}_{m},\omega^{-am/d}_{m}\gamma^{m/d}\cdot v]=[1,j,\omega^{-am/d}_{m}\gamma^{m/d}\cdot v]
    \end{align*}
    if $\frac{4a}{d}$ is even, and
    \begin{align*}
    	[\gamma^{m/d}]\cdot[1,j,v]&=[-j\omega^{am/d}_{m}\gamma^{m/d}]\cdot[1,j,v] \\
    	&=[1,-j\omega^{2am/d}_{m},-j\omega^{am/d}_{m}\gamma^{m/d}\cdot v]=[1,j,-j\omega^{am/d}_{m}\gamma^{m/d}\cdot v]
    \end{align*}
    if $\frac{4a}{d}$ is odd. From this calculation, we see that
    \[R(G^{\ev}_{m})\xrightarrow{\delta}K_{G^{\ev}_{m}}(T_{a,m})=\twopartdef{R(G^{\ev}_{m})\xrightarrow{\res}R(\<\omega_{m}^{-aam/d}\gamma^{m/d}\>)}{\frac{4a}{d}\text{ is even},}{R(G^{\ev}_{m})\xrightarrow{\res}R(\<-j\omega_{m}^{am/d}\gamma^{m/d}\>)}{\frac{4a}{d}\text{ is odd}.}\]
    First suppose $\frac{4a}{d}$ is even, and let $R(\<\omega_{m}^{-am/d}\gamma^{m/d}\>)\cong R(\ZZ_{d})=\ZZ[\alpha]/(\alpha^{d}-1)$. One can check that the restriction map $R(G^{\ev}_{m})\to R(\<\omega_{m}^{-am/d}\gamma^{m/d}\>)$ is given by
    \begin{align*}
    	&1\mapsto 1, & &\wt{c}\mapsto 1, & &w_{k}\mapsto 1-\alpha^{k}, \\	
    	&\zeta\mapsto\alpha, & &h\mapsto\alpha^{a}+\alpha^{-a}, & &z_{k}\mapsto(\alpha^{k}-\alpha^{a})(\alpha^{k}-\alpha^{-a}).
    \end{align*}
    It follows that
    \[\III(\wt{\Sigma}T_{a,m})=\big(\{w_{kd}\}_{k=0}^{\frac{m}{d}-1},\{z_{kd+a}\}_{k=0}^{\frac{m}{d}-1},\{z_{kd-a}\}_{k=0}^{\frac{m}{d}-1}\big),\qquad d=(4a,m),\;\tfrac{4a}{d}\text{ even}.\]
    Next suppose $\frac{4a}{d}$ is odd, and similarly let $R(\<-j\omega_{m}^{am/d}\gamma^{m/d}\>)\cong R(\ZZ_{d})=\ZZ[\alpha]/(\alpha^{d}-1)$. Note that since $d=(4a,m)$, this condition implies that $4|d$ and hence $4|m$. We see that the restriction map $R(G^{\ev}_{m})\to R(\<-j\omega_{m}^{am/d}\gamma^{m/d}\>)$ is given by
    \begin{align*}
    	&1\mapsto 1, & &\wt{c}\mapsto \alpha^{\frac{d}{2}}, & &w_{k}\mapsto 1-\alpha^{k+\frac{d}{2}}, \\	
    	&\zeta\mapsto\alpha, & &h\mapsto\alpha^{\frac{d}{4}}+\alpha^{-\frac{d}{4}}, & &z_{k}\mapsto(\alpha^{k}-\alpha^{\frac{d}{4}})(\alpha^{k}-\alpha^{-\frac{d}{4}}).
    \end{align*}
    In this case, it follows that
    \[\III(\wt{\Sigma}T_{a,m})=\big(\{w_{(2k+1)d/2}\}_{k=0}^{\frac{m}{d}-1},\{z_{(2k+1)d/4}\}_{k=0}^{\frac{2m}{d}-1}\big),\qquad d=(4a,m),\;\tfrac{4a}{d}\text{ odd}.\]
    In particular consider the case where $m=p$ is an odd prime. If $a=0$, then $d=(4p,p)=p$, and so the above calculation shows that
    \begin{align*}
    	&\III(\wt{\Sigma}T_{0,p})=(w_{0},z_{0}), & &\mbfk(\wt{\Sigma}T_{0,p})=\{[\vec{e}_{0}]\}.
    \end{align*}
    On the other hand if $1\le a\le p-1$, then $d=(4a,p)=1$ and the above calculation gives
    \begin{align*}
    	&\III(\wt{\Sigma}T_{a,p})=\big(\{w_{k}\}_{k=0}^{p-1},\{z_{k}\}_{k=0}^{p-1}\big)=\frak{a}, & &\mbfk(\wt{\Sigma}T_{a,p})=\{[\vec{e}_{k}]\;|\;k=0,\dots,p-1\}.
    \end{align*}
    \item Suppose $*=\odd$, let $a\in\frac{1}{2}\ZZ\setminus\ZZ$ be such that $\frac{1}{2}\le a\le m-\frac{1}{2}$, and define
    \[T_{a,m}:=T^{\odd}_{m}/\<\mu\omega_{m}^{-a}\>\]
    Again $T_{a,m}\approx S^{1}\times jS^{1}$ with the left action of $\Pin(2)\subset G^{\odd}_{m}$ as in the even case, and the action of $\mu$ coincides with the action of $\omega_{m}^{a}$ on the right. Let $Q,\wt{Q}$ be as in the $*=\ev$ case. By a similar argument as above, the action of $[\mu]\in\Pin(2)\backslash G^{\odd}_{m}\cong\ZZ_{m}$ on $Q$ is given by $[\mu]\cdot e^{i\phi}=\omega^{a}_{m}e^{i\phi}$. It follows that the action of $[\mu]$ on $Q$ is of order $m/d$ where $d:=(4a,m)$, and
    \[K_{G^{\odd}_{m}}(T_{a,m})\cong R([\mu^{m/d}])\cong R(\ZZ_{d}).\]

    A similar calculation as in the even case shows that
    \[R(G^{\odd}_{m})\xrightarrow{\delta}K_{G^{\odd}_{m}}(T_{a,m})=\twopartdef{R(G^{\odd}_{m})\xrightarrow{\res}R(\<\omega_{m}^{-am/d}\mu^{m/d}\>)}{\frac{4a}{d}\text{ is even},}{R(G^{\odd}_{m})\xrightarrow{\res}R(\<-j\omega_{m}^{am/d}\mu^{m/d}\>)}{\frac{4a}{d}\text{ is odd}.}\]
    If $\frac{4a}{d}$ is even, then the restriction map $R(G^{\odd}_{m})\xrightarrow{\res}R(\<\omega_{m}^{-am/d}\mu^{m/d}\>)\cong\ZZ[\alpha]/(\alpha^{d}-1)$ is given by
    \begin{align*}
    	&1\mapsto 1, & &\wt{c}\mapsto 1, & &w_{\ell}\mapsto 1-\alpha^{\ell}, \\	
    	&\xi^{2}\mapsto\alpha, & &\xi h\mapsto\alpha^{\frac{1}{2}+a}+\alpha^{\frac{1}{2}-a}, & &z_{k+\frac{1}{2}}\mapsto(1-\alpha^{k+a+\frac{1}{2}})(1-\alpha^{k-a+\frac{1}{2}}).
    \end{align*}
    It follows that
    \[\III(\wt{\Sigma}T_{a,m})=\big(\{w_{kd}\}_{k=0}^{\frac{m}{d}-1},\{z_{kd+\frac{a}{2}}\}_{k=0}^{\frac{m}{d}-1},\{z_{kd-\frac{a}{2}}\}_{k=0}^{\frac{m}{d}-1}\big),\qquad d=(4a,m),\;\tfrac{4a}{d}\text{ even}.\]
    Finally suppose $\frac{4a}{d}$ is odd. Note that this implies that $m$ is even and $d\equiv 2\pmod 4$. Then the restriction map $R(G^{\odd}_{m})\xrightarrow{\res}R(\<-j\omega_{m}^{am/d}\mu^{m/d}\>)\cong\ZZ[\alpha]/(\alpha^{d}-1)$ is given by
    \begin{align*}
    	&1\mapsto 1, & &\wt{c}\mapsto\alpha^{\frac{d}{2}}, & &w_{\ell}\mapsto 1-\alpha^{\ell+\frac{d}{2}}, \\	
    	&\xi^{2}\mapsto\alpha, & &\xi h\mapsto\alpha^{\frac{1}{2}+\frac{d}{4}}+\alpha^{\frac{1}{2}-\frac{d}{4}}, & &z_{k+\frac{1}{2}}\mapsto(1-\alpha^{k+\frac{d}{4}+\frac{1}{2}})(1-\alpha^{k-\frac{d}{4}+\frac{1}{2}}),
    \end{align*}
    and hence
    \[\III(\wt{\Sigma}T_{a,m})=\big(\{w_{(2k+1)d/2}\}_{k=0}^{\frac{m}{d}-1},\{z_{(2k+1)d/4}\}_{k=0}^{\frac{2m}{d}-1}\big),\qquad d=(4a,m),\;\tfrac{4a}{d}\text{ odd}.\]
    In particular, consider the following two sub-cases:
    \begin{enumerate}
        \item Suppose $m=2$ and $a=\frac{1}{2}$ or $\frac{3}{2}$. Note that in either case, we have that $d=(4a,2)=2$, and so $\frac{4a}{d}=2a$ is odd. From the calculation above, we see that
        \begin{align*}
            &\III(\wt{\Sigma}T_{\frac{1}{2},2})=\III(\wt{\Sigma}T_{\frac{3}{2},2})=(w_{1},z_{1/2},z_{3/2}), & &\mbfk(\wt{\Sigma}T_{\frac{1}{2},2})=\mbfk(\wt{\Sigma}T_{\frac{3}{2},2})=\{[\vec{e}_{1}]\}.
        \end{align*}
        \item Suppose $m=p$ is an odd prime. If $a=\frac{p}{2}$, then $d=(2p,p)=p$, and so the above calculation shows that
        \begin{align*}
    	&\III(\wt{\Sigma}T_{p/2,p})=(w_{0},z_{p/2}), &             &\mbfk(\wt{\Sigma}T_{p/2,p})=\{[\vec{e}_{0}]\}.
        \end{align*}
        On the other hand if $a\neq\frac{p}{2}$, then $d=(4a,p)=1$ and so
        \begin{align*}
    	&\III(\wt{\Sigma}T_{a,p})=\big(\{w_{k}\}_{k=0}^{p-1},\{z_{k+\frac{1}{2}}\}_{k=0}^{p-1}\big)=\frak{a}, & &\mbfk(\wt{\Sigma}T_{a,p})=\{[\vec{e}_{k}]\;|\;k=0,\dots,p-1\},
        \end{align*}
        as in the $\ast=\ev$ case.
    \end{enumerate}
    \item Suppose $*=\odd$, $m=2$, and define the $G^{\odd}_{2}$-space
    \[T_{\pm j}:=T^{\odd}_{2}/\<\mp\mu j\>.\]
    Again, $T_{\pm j}$ can be identified with a single copy of $S^{1}\times jS^{1}$, with $\Pin(2)$-action given by the canonical left action of $\Pin(2)$ on $S^{1}\times jS^{1}\subset\HH$, and with $\mu$ acting on $T_{j}$ via multiplication by $\pm j$ on the right. Let $Q=\Pin(2)\backslash G^{\odd}_{2}$, $\wt{Q}\subset T_{\pm j}$ be as above. We see that the action of $[\mu]\in\Pin(2)\backslash G^{\odd}_{2}\cong\ZZ_{2}$ on $Q=\wt{Q}/\sim$ is given by
    \[[\mu]\cdot[1,je^{i\phi}]=[\mp e^{i\phi}\mu]\cdot[1,je^{i\phi}]=[1,\mp je^{-i\phi}],\]
    or equivalently
    \[[\mu]\cdot e^{i\phi}=e^{-i\phi}.\]
    It follows that the induced action of $\mu$ on $Q\approx S^{1}$ coincides with the complex conjugation involution, with fixed points $\{\pm 1\}\subset Q$. Using arguments from \cite{ChoMas:00}, one can show that
	\[K_{\ZZ_{2}}(Q)\cong R(\ZZ_{2})\times_{\varepsilon} R(\ZZ_{2}):=\{(x,y)\in R(\ZZ_{2})\times R(\ZZ_{2})\;|\;\varepsilon(x)=\varepsilon(y)\},\]
	where $\varepsilon:R(\ZZ_{2})\to\ZZ$ denotes the augmentation map. Furthermore, if $E$ is a virtual $\ZZ_{2}$-equivariant vector bundle over $Q$, the two copies of $R(\ZZ_{2})$ can be identified with the virtual $\ZZ_{2}$-representations $E_{1}$, $E_{-1}$ over the fixed points $1,-1$, respectively.
	
	In order to determine the image of a virtual representation $V\in R(G^{\odd}_{2})$ under the connecting homomorphism 
	\[\delta:R(G^{\odd}_{2})\xrightarrow{q^{*}}K_{G^{\odd}_{2}}(T_{j})\cong K_{\ZZ_{2}}(Q),\]
	it suffices to choose an element from each of the preimages $f^{-1}(1),f^{-1}(-1)\subset\wt{Q}$, say $1,i\in\wt{Q}$, and look at the induced action of $\mu$ on representations above them. Given $v\in V$, we see that:
	\begin{align*}
	&[\mu]\cdot[1,j,v]=[\mp j\mu]\cdot[1,j,v]=[1,j,\mp j\mu\cdot v], \\
	&[\mu]\cdot[1,ji,v]=[\mp i\mu]\cdot[1,ji,v]=[1,ji,\mp i\mu\cdot v],
	\end{align*}
	and hence the connecting homomorphism can be identified with the homomorphism
	\begin{align*}
		R(G^{\odd}_{2})&\to R(\ZZ_{2})\times_{\varepsilon}R(\ZZ_{2}) \\
		[V]&\mapsto(\res^{G^{\odd}_{2}}_{\<\mp j\mu\>}([V]),\res^{G^{\odd}_{2}}_{\<\mp i\mu\>}([V]).
	\end{align*}
	Writing 
	\[R(\ZZ_{2})\times_{\varepsilon}R(\ZZ_{2})\subset R(\ZZ_{2})\oplus R(\ZZ_{2})=(\ZZ[\alpha_{1}]/(\alpha_{1}^{2}-1))\oplus(\ZZ[\alpha_{2}]/(\alpha_{2}^{2}-1)),\]
	we see that the above map is given by
    \begin{align*}
    	&1\mapsto (1,1), & &\wt{c}\mapsto(\alpha_{1},1), & &w_{0}\mapsto (1-\alpha_{1},0), & &w_{1}\mapsto(0,1-\alpha_{2}), \\	
    	&\xi^{2}\mapsto(\alpha_{1},\alpha_{2}), & &\xi h\mapsto(1+\alpha_{1},1+\alpha_{2}), & &z_{1/2}\mapsto(0,0), & &z_{3/2}\mapsto(0,0),
    \end{align*}
    and hence
    \begin{align*}
    &\III(\wt{\Sigma}T_{\pm j})=(z_{1/2},z_{3/2}), & &\mbfk(\wt{\Sigma}T_{\pm j})=\{[\vec{e}_{1}]\}.
    \end{align*}
\end{enumerate}
\end{example}

We next look at the behavior of our equivariant $k$-invariants under smash products:

\begin{lemma}
\label{lemma:k_invariants_products}
Suppose $X,X'$ are spaces of type $\CC$-$G^{*}_{m}$-$\SWF$. Then $\III(X)\cdot\III(X')\subset\III(X\wedge X')$. In particular:
\begin{enumerate}
    \item The following statements are true:
    \begin{enumerate}
        \item For any $\vec{k}\in \mbfk(X)$, $\vec{k}'\in \mbfk(X')$:
        \begin{enumerate}
            \item $\vec{k}+\vec{k}'\not\prec\vec{k}''$ for every $\vec{k}''\in\mbfk(X\wedge X')$.
            \item There exists some $\vec{k}''\in\mbfk(X\wedge X')$ such that $\vec{k}+\vec{k}'\succeq\vec{k}''$.
        \end{enumerate}
        \item $\vec{\ol{k}}(X)+\vec{\ol{k}}(X')\succeq\vec{\ul{k}}(X\wedge X')$.
    \end{enumerate}
    \item Suppose $X\wedge X'$ is $K_{G^{*}_{m}}$-split, and let $\vec{k}''$ be the unique element of $\mbfk(X\wedge X')$. Then:
    \begin{enumerate}
        \item For any $\vec{k}\in \mbfk(X)$, $\vec{k}'\in \mbfk(X')$, we have that $\vec{k}+\vec{k}'\succeq\vec{k}''$.
        \item In particular, $\vec{\ul{k}}(X)+\vec{\ul{k}}(X')\succeq \vec{k}''$.
    \end{enumerate}
\end{enumerate}
\end{lemma}

\begin{proof}
If $\iota:X^{S^{1}}\to X$, $\iota':(X')^{S^{1}}\to X'$ denote the inclusion maps, let
\[\iota\wedge\iota':X^{S^{1}}\wedge (X')^{S^{1}}\cong(X\wedge X')^{S^{1}}\to X\wedge X'\]
denote the inclusion of the $S^{1}$-fixed point set of $X\wedge X'$. Note that there exists a commutative diagram

\begin{center}
    \begin{tikzcd}
    \wt{K}_{G^{*}_{m}}(X)\otimes\wt{K}_{G^{*}_{m}}(X') \arrow[r, "\mu"] \arrow[d, "\iota^{*}\otimes\iota^{*}"]
    & \wt{K}_{G^{*}_{m}}(X\wedge X') \arrow[d, "(\iota\wedge\iota')^{*}"] \\
    \wt{K}_{G^{*}_{m}}(X^{S^{1}})\otimes\wt{K}_{G^{*}_{m}}((X')^{S^{1}}) \arrow[r, "\mu" ]
    & \wt{K}_{G^{*}_{m}}(X^{S^{1}}\wedge (X')^{S^{1}}),
    \end{tikzcd}
\end{center}

where $\mu$ denotes the external product map from Fact \ref{fact:product_map}. Hence $(\iota\wedge\iota')^{*}\circ\mu$ is equal to the composition
\[\wt{K}_{G^{*}_{m}}(X)\otimes\wt{K}_{G^{*}_{m}}(X')\xrightarrow{\iota^{*}\otimes(\iota')^{*}}\wt{K}_{G^{*}_{m}}(X^{S^{1}})\otimes\wt{K}_{G^{*}_{m}}((X')^{S^{1}})\xrightarrow{\mu}\wt{K}_{G^{*}_{m}}(X^{S^{1}}\wedge (X')^{S^{1}}).\]
Under the identifications 
\[\wt{K}_{G^{*}_{m}}(X^{S^{1}})\cong\wt{K}_{G^{*}_{m}}(X^{S^{1}}\wedge X^{S^{1}})\cong R(G^{*}_{m}),\]
one can show that the map  $\wt{K}_{G^{*}_{m}}(X^{S^{1}})\otimes\wt{K}_{G^{*}_{m}}(X^{S^{1}})\xrightarrow{\mu}\wt{K}_{G^{*}_{m}}(X^{S^{1}}\wedge X^{S^{1}})$ is equivalent to the multiplication map on $R(G^{*}_{m})$. Therefore the image of $\mu\circ(\iota^{*}\otimes(\iota')^{*})$ is precisely the set of elements of the form $ab\in R(G^{*}_{m})$, where $a,b\in\III(X)$. The proposition thus follows from the observation that $\im((\iota\wedge\iota')^{*}\circ\mu)\subset\im((\iota\wedge\iota')^{*})$.
\end{proof}

To conclude this section, we discuss the behavior of equivariant $k$-invariants under equivariant Spanier-Whitehead duality. Recall the following definition from \cite{Man14} (see also \cite{May96}, Section XVI.8):

\begin{definition}
Let $G$ be a compact Lie group, and let $V$ be a finite-dimensional $G$-representation. We say that two finite pointed $G$-spaces $X,X'$ are \emph{equivariantly $V$-dual} if there exist $G$-equivariant maps $\varepsilon:X'\wedge X\to V^{+}$ and $\eta:V^{+}\to X\wedge X'$ such that the following diagrams homotopy commute:

\begin{center}
\begin{tikzcd}
 V^{+}\wedge X \arrow[rd, "\tau"] \arrow[r, "\eta\wedge\id"] & X\wedge X'\wedge X \arrow[d, "\id\wedge\varepsilon"] \\
& X\wedge V^{+}
\end{tikzcd}\qquad\qquad\qquad
\begin{tikzcd}
 X'\wedge V^{+} \arrow[d, "\tau"] \arrow[r, "\id\wedge\eta"] & X'\wedge X\wedge X' \arrow[d, "\varepsilon\wedge\id"] \\
V^{+}\wedge X' \arrow[r, "-\id\wedge\id"] & V^{+}\wedge X',
\end{tikzcd}
\end{center}

where $\tau$ is the transposition map which swaps the two factors.
\end{definition}

The following is an immediate corollary of (\cite{Man14}, Lemma 3.12) and Lemma \ref{lemma:comparison_of_k_invariants}:

\begin{proposition}
\label{prop:k_invariants_duality}
Suppose $X$ and $X'$ are spaces of type $\CC$-$G^{*}_{m}$-$\SWF$ which are $G^{*}_{m}$-equivariantly $(\mbfu\wt{\CC}\oplus\mbft\HH)$-dual for some $\mbfu\in R(\ZZ_{m})_{\geq 0}^{\sym}$, $\mbft\in R(\ZZ_{2m})_{\geq 0}^{*}$. Then
\[\vec{k}+\vec{k}'\succeq [\DDD^{*}(\vec{\mbft})]\text{ for all }\vec{k}\in\mbfk(X),\, \vec{k}'\in\mbfk(X').\]
In particular:
\[\vec{\ul{k}}(X)+\vec{\ul{k}}(X')\succeq[\DDD^{*}(\vec{\mbft})].\]
\end{proposition}

\begin{proof}
By Lemma \ref{lemma:comparing_minima_sum_of_subsets}, Corollary \ref{cor:k_invariants_local_equivalence} and Lemma \ref{lemma:k_invariants_products}, it suffices to show that $X\wedge X'$ and $(\mbfu\wt{\CC}\oplus\mbft\HH)^{+}$ are locally equivalent. Let $\mbfs,\mbfs'$ denote the levels of $X,X'$, respectively, and for each divisor $d|m$ define the following subgroups:
\begin{align*}
&H^{\ev}_{m,d}:=S^{1}\times\ZZ_{d}\subset G^{\ev}_{m}, & &H^{\odd}_{m,d}:=S^{1}\times_{\ZZ_{2}}\ZZ_{2d}\subset G^{\odd}_{m}.
\end{align*}
Then for each $d|m$, the restrictions of the duality maps $\varepsilon$, $\eta$ to their $H^{*}_{m,d}$-fixed point sets induce a $(\sum_{j=0}^{d-1}u_{jm/d}\zeta^{jm/d})\wt{\CC}$-duality between $((\sum_{j=0}^{d-1}s_{jm/d}\zeta^{jm/d})\wt{\CC})^{+}$ and $((\sum_{j=0}^{d-1}s'_{jm/d}\zeta^{jm/d})\wt{\CC})^{+}$. In particular by taking $d=1$, we have that $\mbfs+\mbfs'=\mbfu$, and so we can think of $\varepsilon^{S^{1}}$, $\eta^{S^{1}}$ as $\ZZ_{2}\times\ZZ_{m}$-equivariant maps from $(\mbfu\wt{\CC})^{+}$ to itself. But since the maps $\varepsilon^{H^{*}_{m,d}}$, $\eta^{H^{*}_{m,d}}$ must induce non-equivariant homotopy equivalences for all $d|m$, it follows that $\varepsilon^{S^{1}}$, $\eta^{S^{1}}$ induce self homotopy equivalences $((\mbfu\wt{\CC})^{+})^{H}\simeq((\mbfu\wt{\CC})^{+})^{H}$ for all subgroups $H\subset\ZZ_{2}\times\ZZ_{m}$, and so $\varepsilon^{S^{1}}$, $\eta^{S^{1}}$ must be $G^{*}_{m}$-homotopy equivalences.
\end{proof}

We will make use of the following lemma (see \cite{May96}, Section XVI.8):

\begin{lemma}
\label{lemma:duality_embeddings}
Let $V$ be a $G$-representation and let $X$ be a $G$-space along with an embedding of $X$ into the unit sphere $S(V)$ of $V$. Then $\wt{\Sigma}X$ and $\wt{\Sigma}(S(V)\setminus X)$ are equivariantly $V$-dual.
\end{lemma}

\begin{example}
\label{ex:duality}
We exhibit some examples of $V$-dual spaces where the inequalities in Proposition \ref{prop:k_invariants_duality} are strict:
\begin{enumerate}
    \item Let $\ast\in\{\ev,\odd\}$, let $a\in\frac{1}{2}\ZZ$ be such that $0\le a\le m-\frac{1}{2}$ and $a\equiv 0\pmod{1}$ if $\ast=\ev$ and $a\equiv\frac{1}{2}\pmod{1}$ if $\ast=\odd$, and let $Z_{a,m}$, $T_{a,m}$ be the spaces considered in Examples \ref{ex:cosets_cyclic} and \ref{ex:torus_cyclic}. There exist canonical $G^{*}_{m}$-equivariant embeddings
    \begin{align*}
        &e_{Z}:Z_{a,m}\cong(S^{1}\times\{0\})\cup(\{0\}\times j S^{1})\hookrightarrow S(\HH_{a}) \\
        &e_{T}:T_{a,m}\cong \tfrac{1}{\sqrt{2}}(S^{1}\times j S^{1})\hookrightarrow S(\HH_{a})
    \end{align*}
    into the unit sphere $S(\HH_{a})$ of the $G^{*}_{m}$-representation $\HH_{a}$. We have the following explicit $G^{*}_{m}$-equivariant deformation retraction from $S(\HH_{a})\setminus Z_{a,m}$ to $\wt{\Sigma}T_{a,m}$:
    \begin{align*}
        d:(S(\HH_{a})\setminus Z_{a,m})\times[0,1]&\to S(\HH_{a})\setminus Z_{a,m} \\
        (re^{i\theta},j\sqrt{1-r^{2}}e^{i\phi},t)&\mapsto((r(1-t)+\tfrac{1}{\sqrt{2}}t)e^{i\theta},j(\sqrt{1-r^{2}}(1-t)+\tfrac{1}{\sqrt{2}}t)e^{i\phi}).
    \end{align*}
    By Lemma \ref{lemma:duality_embeddings} it follows that $\wt{\Sigma}Z_{a,m}$ is $\HH_{a}$-dual to $\wt{\Sigma}T_{a,m}$. From our calculations above, we have that
    \begin{align*}
        &\III(\HH_{a}^{+})=(z_{a}), \\
        &\III(\wt{\Sigma}Z_{a,m})=(w_{0},z_{a},z_{m-a}), \\
        &\III(\wt{\Sigma}T_{a,m})=\twopartdef{(w_{kd},z_{kd+a},z_{kd-a}\;|\;k=0,\dots,\tfrac{m}{d}-1)}{\frac{4a}{d}\text{ is even,}}{(w_{(2k+1)d/2},z_{(2k+1)d/4}\;|\;k=0,\dots,\tfrac{2m}{d}-1)}{\frac{4a}{d}\text{ is odd,}}
    \end{align*}
    where $d:=(4a,m)$. One can use the relations in $R(G^{*}_{m})$ to show that
    \[\III(\wt{\Sigma}Z_{a,m})\cdot\III(\wt{\Sigma}T_{a,m})\subsetneq\III(\HH_{a}^{+}).\]
    \item Let $m=2$, $*=\odd$, and consider the spaces $Z_{j}$, $T_{j}$ from Examples \ref{ex:cosets_cyclic}, \ref{ex:torus_cyclic}, respectively. There are $G^{\odd}_{2}$-equivariant embeddings $Z_{j},T_{j}\hookrightarrow S(\HH_{1/2})$ given as follows:
    \begin{align*}
        e_{Z_{j}}:Z_{j}&\hookrightarrow S(\HH_{1/2}) & e_{T_{j}}:T_{j}&\hookrightarrow S(\HH_{1/2}) \\
        e^{i\theta}&\mapsto(\tfrac{1}{\sqrt{2}}e^{i\theta},-\tfrac{1}{\sqrt{2}}jie^{-i\theta}), & (e^{i\theta},e^{i\phi})&\mapsto(\tfrac{1}{\sqrt{2}}(e^{i\theta}+ie^{-i\phi}),\tfrac{1}{\sqrt{2}}(j(-ie^{-i\theta}+e^{i\phi}))). \\
        je^{i\theta}&\mapsto(\tfrac{1}{\sqrt{2}}ie^{-i\theta},\tfrac{1}{\sqrt{2}}je^{i\theta}), & &
    \end{align*}
    As above, one can check that $S(\HH_{1/2})\setminus Z_{j}$ equivariantly deformation retracts onto $T_{j}$, and so $\wt{\Sigma}Z_{j}$ and $\wt{\Sigma}T_{j}$ are $\HH_{1/2}$-dual. Similarly, we have that $\wt{\Sigma}Z_{-j}$ and $\wt{\Sigma}T_{-j}$ are $\HH_{3/2}$-dual. From our calculations above we have that
    \begin{align*}
        &\III(\HH_{1/2}^{+})=(z_{1/2}), & &\III(\HH_{3/2}^{+})=(z_{3/2}), \\
        &\III(\wt{\Sigma}Z_{j})=\III(\wt{\Sigma}Z_{-j})=(w_{1},z_{1/2},z_{3/2}), &
        &\III(\wt{\Sigma}T_{j})=\III(\wt{\Sigma}T_{-j})=(z_{1/2},z_{3/2}),
    \end{align*}
    from which it follows that
    \begin{align*}
        &\III(\wt{\Sigma}Z_{j})\cdot\III(\wt{\Sigma}T_{j})\subsetneq\III(\HH_{1/2}^{+}), & &\III(\wt{\Sigma}Z_{-j})\cdot\III(\wt{\Sigma}T_{-j})\subsetneq\III(\HH_{3/2}^{+}).
    \end{align*}
\end{enumerate}
\end{example}

\begin{example}
\label{ex:multiple_cosets_duals}
Let $\ast\in\{\ev,\odd\}$, and let $a_{1},\dots,a_{n}\in\frac{1}{2}\ZZ$ be such that for each $k=1,\dots,n$: $0\le a_{k}\le m-\frac{1}{2}$ and $a_{k}\equiv 0\pmod{1}$ if $\ast=\ev$ and $a_{k}\equiv\frac{1}{2}\pmod{1}$ if $\ast=\odd$. 
Consider the $G^{*}_{m}$-space
\[X_{a_{1},\dots,a_{n};m}:=S(\oplus_{k=1}^{n}\HH_{a_{k}})\setminus e(Z_{a_{1},\dots,a_{n};m}),\]
where $Z_{a_{1},\dots,a_{n};m}$ is as in Example \ref{ex:multiple_cosets_cyclic}, and
\[e:Z_{a_{1},\dots,a_{n};m}\hookrightarrow S(\oplus_{k=1}^{n}\HH_{a_{k}}),\]
is the embedding which on each coset $Z_{a_{k},m}\subset Z_{a_{1},\dots,a_{n};m}$ restricts to the embedding into $S(\HH_{a_{k}})\subset S(\oplus_{k=1}^{n}\HH_{a_{k}})$ from Example \ref{ex:duality}. 

By Lemma \ref{lemma:duality_embeddings} we have that the unnreduced suspension $\wt{\Sigma}X_{a_{1},\dots,a_{n};m}$ is $\oplus_{k=1}^{n}\HH_{a_{k}}$-dual to $\wt{\Sigma}Z_{a_{1},\dots,a_{n};m}$. Note that in the case where $n=1$, there exists a $G^{*}_{m}$-equivariant deformation retraction from $X_{a;m}$ onto the space $T_{a,m}$ from Example \ref{ex:torus_cyclic}, and so this agrees with the situation outlined in Example \ref{ex:duality}.

We can estimate the ideal $\III(\wt{\Sigma}X_{a_{1},\dots,a_{n};m})$ by approximating it from ``above'' and ``below'', which allows us to obtain partial information about the equivariant $k$-invariants of $\wt{\Sigma}X_{a_{1},\dots,a_{n};m}$:

Let $X_{n}:=\res^{G^{*}_{m}}_{\Pin(2)}(X_{a_{1},\dots,a_{n};m})$ denote the underlying $\Pin(2)$-space of $X_{a_{1},\dots,a_{n};m}$. As duality is functorial with respect to restriction maps, it follows that $\wt{\Sigma}X_{n}$ is $\HH^{n}$-dual to $\wt{\Sigma}Z_{n}$, where $Z_{n}:=\res^{G^{*}_{m}}_{\Pin(2)}(Z_{a_{1},\dots,a_{n};m})$. Note that after suspending once, there exists a $\Pin(2)$-equivariant homotopy equivalence
\[\Sigma^{\RR}\wt{\Sigma}Z_{n}\simeq\Sigma^{\RR}(\wt{\Sigma}\Pin(2)\vee(\vee^{n-1}\Pin(2)_{+}).\]
Furthermore, the map
\[\wt{\Sigma}\Pin(2)\vee(\vee^{n-1}\Pin(2)_{+})\to\wt{\Sigma}\Pin(2)\]
which collapses the wedge-summand $\vee^{n-1}\Pin(2)_{+}$ to the basepoint induces a $\Pin(2)$-equivari-\\ ant local equivalence. Hence $\wt{\Sigma}Z_{n}\equiv_{\ell}\wt{\Sigma}\Pin(2)$. As local equivalence respects the operation of taking duals, we have that $\wt{\Sigma}X_{n}$ is locally equivalent to the $\HH^{n}$-dual of $\wt{\Sigma}\Pin(2)$. From (\cite{Man16}, Example 2.13) we have that $\wt{\Sigma}\Pin(2)$ is $\HH$-dual to the space $\wt{\Sigma}T$, where $T=S^{1}\times jS^{1}\subset\HH$. Therefore $\wt{\Sigma}X_{n}\equiv_{\ell}\Sigma^{(n-1)\HH}\wt{\Sigma}T$, and so by (\cite{Man14}, Lemma 3.4 and Example 3.7) we obtain
\[\III_{\Pin(2)}(\wt{\Sigma}X_{n})=\III_{\Pin(2)}(\Sigma^{(n-1)\HH}\wt{\Sigma}T)=z^{n-1}\cdot(w,z)=(w^{n},z^{n}).\]
As in the proof of Lemma \ref{lemma:comparison_of_k_invariants}, we have that $\res^{G^{*}_{m}}_{\Pin(2)}(\III(\wt{\Sigma}X_{a_{1},\dots,a_{n};m}))\subset\III_{\Pin(2)}(\wt{\Sigma}X_{n})$, and so
\[\III(\wt{\Sigma}X_{a_{1},\dots,a_{n};m})\subset(\res^{G^{*}_{m}}_{\Pin(2)})^{-1}((w^{n},z^{n}))=\frak{a}^{n},\]
where $\frak{a}\subset R(G^{*}_{m})$ denotes the augmentation ideal.

Next, note that
\[\wt{\Sigma}X_{a_{1},\dots,a_{n};m}\simeq\wedge_{k=1}^{n}X_{a_{k};m}\simeq\wedge_{k=1}^{n}T_{a_{k},m},\]
and so by Lemma \ref{lemma:k_invariants_products} we have that
\[\prod_{k=1}^{n}\III(T_{a_{k},m})\subset\III(\wt{\Sigma}X_{a_{1},\dots,a_{n};m}).\]
Hence altogether we have that
\begin{equation}
\label{eq:above_and_below}
    \III(T_{a_{k},m})\subset\III(\wt{\Sigma}X_{a_{1},\dots,a_{n};m})\subset\frak{a}^{n}.
\end{equation}
\begin{enumerate}
    \item Let $m=2$ and $*=\odd$, so that $a_{k}=\frac{1}{2}$ or $\frac{3}{2}$ for all $k=1,\dots,n$. From Equation \ref{eq:above_and_below} and the calculations from Example \ref{ex:torus_cyclic} we have that
    \[(w_{1},z_{1/2},z_{3/2})^{n}\subset\III(\wt{\Sigma}X_{a_{1},\dots,a_{n};2})\subset(w_{0},w_{1},z_{1/2},z_{3/2})^{n}.\]
    Although this does not determine $\mbfk(\wt{\Sigma}X_{a_{1},\dots,a_{n};2})$ explicitly, it follows that any element $\vec{k}\in\mbfk(\wt{\Sigma}X_{a_{1},\dots,a_{n};2})$ satisfies $|\vec{k}|=n$.
    \item Let $m=p$ be an odd prime, and let
    \[n_{0}:=\#\{1\le k\le n\;|\;2a_{k}\equiv 0\pmod{p}\}.\]
    Again from Equation \ref{eq:above_and_below} and Example \ref{ex:torus_cyclic} we have that
    \[\frak{a}_{0}^{n_{0}}\cdot\frak{a}^{n-n_{0}}\subset\III(\wt{\Sigma}X_{a_{1},\dots,a_{n};p})\subset\frak{a}^{n},\]
    where:
    \[\frak{a}_{0}:=\twopartdef{(w_{0},z_{0})}{\ast=\ev,}{(w_{0},z_{p/2})}{\ast=\odd.}\]
    We claim that in fact $\III(\wt{\Sigma}X_{a_{1},\dots,a_{n};p})=\frak{a}_{0}^{n_{0}}\cdot\frak{a}^{n-n_{0}}$. Indeed, let $1\le i_{1}<\cdots<i_{n_{0}}\le n$ be the subsequence satisfying $2a_{i_{k}}\equiv 0\pmod{p}$ for all $k=1,\dots,n_{0}$. Then the inclusion
    \[f:\wt{\Sigma}X_{a_{i_{1}},\dots,a_{i_{n_{0}}};p}\hookrightarrow\wt{\Sigma}X_{a_{1},\dots,a_{n};p}\]
    induces a commutative diagram
    \begin{center}
    \begin{tikzcd}
        \wt{K}_{G^{*}_{p}}(\wt{\Sigma}X_{a_{1},\dots,a_{n};p}) \arrow[r, "\iota^{*}"] \arrow[d, "f^{*}"]
        & \wt{K}_{G^{*}_{p}}(\wt{\Sigma}X_{a_{1},\dots,a_{n};p}^{S^{1}}) \arrow[d, "\cong"] \\
        \wt{K}_{G^{*}_{p}}(\wt{\Sigma}X_{a_{i_{1}},\dots,a_{i_{n_{0}}};p})\arrow[r, "\iota^{*}_{0}" ]
        & \wt{K}_{G^{*}_{p}}(\wt{\Sigma}X_{a_{i_{1}},\dots,a_{i_{n_{0}}};p}^{S^{1}}).
    \end{tikzcd}
    \end{center}
    By Examples \ref{ex:ideal_trivial_Z_m_action} and \ref{ex:torus_cyclic}, we conclude that
    \begin{align*}
    \III(\wt{\Sigma}X_{a_{1},\dots,a_{n};p})\subset\frak{a}_{0}^{n_{0}}&\implies \frak{a}^{n-n_{0}}\frak{a}^{n_{0}}\subset\III(\wt{\Sigma}X_{a_{1},\dots,a_{n};p})\subset\frak{a}^{n}\cap\frak{a}_{0}^{n_{0}}=\frak{a}^{n-n_{0}}\frak{a}^{n_{0}} \\
    &\implies \III(\wt{\Sigma}X_{a_{1},\dots,a_{n};p})=\frak{a}^{n-n_{0}}\frak{a}^{n_{0}}.
    \end{align*}
    Therefore:
    \[\mbfk(\wt{\Sigma}X_{a_{1},\dots,a_{n};p})=\{[\vec{v}]\;|\;\vec{v}\succeq n_{0}\vec{e}_{0},\;|\vec{v}|=n\}.\]
\end{enumerate}
\end{example}

\bigskip
\subsection{\texorpdfstring{$2^{r}$}{2r}-fold Actions}
\label{subsec:2_r_actions}

In this section, we consider the case where $m=2^{r}$ for some $r\geq 1$. We will first analyze the structure of the subset $I(X)\subset\N^{2^{r}}$ for a space $X$ of type $\CC$-$G^{*}_{2^{r}}$-$\SWF$.

\begin{lemma}
\label{lemma:2_r_exists_monomial}
Let $X$ be a space of type $\CC$-$G^{\ast}_{2^{r}}$-$\SWF$ for some $r\geq 1$. Then 
\[\Big[k_{\Pin(2)}(X)\cdot\vec{e}_{0}+\sum_{k=0}^{r-1}\vec{e}_{2^{k}}\Big]\in I(X).\]
In particular, $I(X)$ is non-empty.
\end{lemma}

\begin{proof}
Consider the following commutative diagram, whose rows consist of the low exact sequence from Fact \ref{fact:long_exact_sequence}, and whose vertical arrows denote the functorial induction map from Fact \ref{fact:induction}:
\begin{center}
\begin{tikzcd}
\cdots \arrow[r] & \wt{K}_{G^{*}_{2^{r}}}(X) \arrow[r, "\iota^{*}_{G^{*}_{2^{r}}}"] 
& \wt{K}_{G^{*}_{2^{r}}}(X^{S^{1}}) \arrow[r, "\delta_{G^{*}_{2^{r}}}"]  & \wt{K}^{1}_{G^{*}_{2^{r}}}(X/X^{S^{1}}) \arrow[r] & \cdots\\
\cdots \arrow[r] & \wt{K}_{\Pin(2)}(X) \arrow[r, "\iota^{*}_{\Pin(2)}"]\arrow[u, "\ind^{G^{*}_{2^{r}}}_{\Pin(2)}"]
& \wt{K}_{\Pin(2)}(X^{S^{1}}) \arrow[r, "\delta_{\Pin(2)}"] \arrow[u, "\ind^{G^{*}_{2^{r}}}_{\Pin(2)}"] & \wt{K}^{1}_{\Pin(2)}(X/X^{S^{1}}) \arrow[u, "\ind^{G^{*}_{2^{r}}}_{\Pin(2)}"] \arrow[r] & \cdots
\end{tikzcd}
\end{center}
Now let $x\in\III_{\Pin(2)}(X)$ be such that $wx=w^{k_{\Pin(2)}(X)+1}$. By commutativity of the right-hand square, we must have that $\ind_{\Pin(2)}^{G^{*}_{m}}(x)\in\III(X)$. It therefore suffices to show that
\begin{equation}
\label{eq:induction_proof}
	w_{0}\ind^{G^{*}_{2^{r}}}_{\Pin(2)}(x)=w_{0}^{k_{\Pin(2)}(X)+1}\prod_{k=0}^{r-1}w_{2^{k}}\in R(G^{*}_{2^{r}}).
\end{equation}
Observe that if $*=\ev$, then $\ind^{G^{\ev}_{2^{r}}}_{\Pin(2)}(y)=y\sum_{k=0}^{m-1}\zeta^{k}$ for all $y\in R(G^{\ev}_{m})$, and that if $*=\odd$, then
\begin{align*}
	&\ind^{G^{\odd}_{2^{r}}}_{\Pin(2)}(y)=y\sum_{j=0}^{2^{r}-1}\xi^{2j}\text{ for }y=1,\wt{c},\xi,\text{ or }h^{2k},\text{ and} & &\ind^{G^{\odd}_{2^{r}}}_{\Pin(2)}(h^{2k+1})=h^{2k+1}\sum_{j=0}^{2^{r}-1}\xi^{2j+1}.
\end{align*}
It follows that for all $\ell\geq 0$, we have that:
\begin{align*}
	&\ind^{G^{*}_{2^{r}}}_{\Pin(2)}(w^{\ell})=\twopartdef{w_{0}^{\ell}(1+\zeta+\cdots+\zeta^{2^{r}-1})}{*=\ev,}{w_{0}^{\ell}(1+\xi^{2}+\cdots+\xi^{2^{r+1}-2})}{*=\odd,}\\
	&\ind^{G^{*}_{2^{r}}}_{\Pin(2)}(z^{\ell})=\twopartdef{z_{0}^{\ell}(1+\zeta+\cdots+\zeta^{2^{r}-1})}{*=\ev,}{(2-\xi h)^{\ell}(1+\xi^{2}+\cdots+\xi^{2^{r+1}-2})}{*=\odd.}
\end{align*}
Note that 
\[w_{0}w_{0}^{\ell}=w_{0}z_{0}^{\ell}=w_{0}(2-\xi h)^{\ell}=w_{0}^{\ell+1}\]
for all $\ell\geq 0$. Therefore we see that:
\begin{equation}
\label{eq:induction_proof_2}
	w_{0}\ind^{G^{*}_{2^{r}}}_{\Pin(2)}(x)=\twopartdef{w_{0}^{k_{\Pin(2)}(X)+1}(1+\zeta+\cdots+\zeta^{2^{r}-1})}{*=\ev,}{w_{0}^{k_{\Pin(2)}(X)+1}(1+\xi^{2}+\cdots+\xi^{2^{r+1}-2})}{*=\odd.}
\end{equation}
Next, we show that the right side of Equation \ref{eq:induction_proof_2} is equal to the right-hand side of Equation \ref{eq:induction_proof}. It suffices to show that the traces of these expressions agree at all elements $g\in G^{*}_{2^{r}}$. Without loss of generality we treat the case $*=\ev$, as the $*=\odd$ case is entirely similar. Note that any element of $G^{\ev}_{2^{r}}$ is of the form $e^{i\phi}\gamma^{a}$ or $je^{i\phi}\gamma^{a}$ for some $a=0,\dots,2^{r}-1$, $\phi\in[0,2\pi)$. A quick calculation shows that:
\begin{align*}
    &\tr_{e^{i\phi}\gamma^{a}}(w_{k})=1-\omega_{2^{r}}^{ak}, & &\tr_{e^{i\phi}\gamma^{a}}(\zeta^{k})=\omega_{2^{r}}^{ak}, \\
    &\tr_{je^{i\phi}\gamma^{a}}(w_{k})=1+\omega_{2^{r}}^{ak}, & &\tr_{je^{i\phi}\gamma^{a}}(\zeta^{k})=\omega_{2^{r}}^{ak},
\end{align*}
for all $\phi\in[0,2\pi)$, and where $\omega_{2^{r}}=e^{\pi i/2^{r-1}}\in\CC$. Thus:
\begin{align*}
    &\tr_{e^{i\phi}\gamma^{a}}\Big(w_{0}^{k_{\Pin(2)}(X)+1}\sum_{\ell=0}^{2^{r}-1}\zeta^{\ell}\Big)=\tr_{e^{i\phi}\gamma^{a}}\Big(w_{0}^{k_{\Pin(2)}(X)+1}\prod_{k=0}^{r-1}w_{2^{k}}\Big)=0\text{ for all }a, \\
    &\tr_{je^{i\phi}\gamma^{a}}\Big(w_{0}^{k_{\Pin(2)}(X)+1}\sum_{\ell=0}^{2^{r}-1}\zeta^{\ell}\Big)=\tr_{je^{i\phi}\gamma^{a}}\Big(w_{0}^{k_{\Pin(2)}(X)+1}\prod_{k=0}^{r-1}w_{2^{k}}\Big)=\twopartdef{2^{k_{\Pin(2)}(X)+r+1}}{a=0,}{0}{a\neq 0.}
    \end{align*}
Therefore:
\[w_{0}^{k_{\Pin(2)}(X)+1}(1+\zeta+\cdots+\zeta^{2^{r}-1})=w_{0}^{k_{\Pin(2)}(X)+1}\prod_{k=0}^{r-1}w_{2^{k}}.\]
\end{proof}

The rest of this section is dedicated to proving a refinement of Proposition \ref{prop:k_invariants_kg_split} for spaces of type $\CC$-$G^{\odd}_{2^{r}}$-$\SWF$. In order to state the refinement we will need the following definitions:

\begin{definition}
\label{def:H_spherical}
Let $X$ be a space of type $\CC$-$G^{*}_{m}$-$\SWF$, and let $H\subset G^{*}_{m}$ be a subgroup. We say that $X$ is \emph{$H$-spherical (at level $d\in\NN$)} if the $H$-fixed point set $X^{H}\subset X$ is (non-equivariantly) homotopy equivalent to a sphere of dimension $d$.
\end{definition}

\begin{example}
Any space of $\CC$-$G^{*}_{m}$-$\SWF$ at level $\mbfs=\sum_{j=0}^{m-1}s_{j}\zeta^{j}$ is $S^{1}$-spherical at level $\sum_{j=0}^{m-1}2s_{j}$, and $\Pin(2)$-spherical at level 0.
\end{example}

\begin{definition}
\label{def:locally_H_spherical}
Let $X$ be a space of type $\CC$-$G^{*}_{m}$-$\SWF$, and let $H\subset G^{*}_{m}$ be a subgroup. We say that $X$ is \emph{locally $H$-spherical (at level $d\in\NN$)} if $X$ is locally equivalent to (in the sense of Definition \ref{def:unstable_local_equivalence}) a space $X'$ of type $\CC$-$G^{*}_{m}$-$\SWF$ which is $H$ spherical at level $d$.
\end{definition}

In the case where $m=2^{r}$ and $*=\odd$, we will need to consider the subgroup
\[\<j\mu^{2^{r}}\>\cong\ZZ_{2}\subset G^{\odd}_{2^{r}},\]
Note that if $X$ is a space of type $\CC$-$G^{\odd}_{2}$-$\SWF$ and $x\in X^{\<j\mu^{2^{r-1}}\>}$, then
\[j\mu^{2^{r-1}}\cdot(j\cdot x) = j\cdot(\mu^{2^{r-1}}j)\cdot x = j\cdot(j\mu^{2^{r-1}}\cdot x) = j\cdot x,\]
and so $j\cdot x\in X^{\<j\mu^{2^{r-1}}\>}$ as well. Hence the subgroup $\<j\>\cong\ZZ_{4}\subset G^{\odd}_{2}$ has a well-defined action on the fixed point set $X^{\<j\mu^{2^{r-1}}\>}$. It follows that $X^{\<j\mu^{2^{r-1}}\>}$ is an example of what we will call a space of type $\CC$-$\ZZ_{4}$-$\SWF$:

\begin{definition}[\cite{KMT}, Definition 3.1]
\label{def:C_Z_4_SWF}
Let $j\in\ZZ_{4}$ be a fixed generator, and let $\wt{\CC}$ be the one-dimensional complex $\ZZ_{4}$-representation on which $j$ acts by $-1$. A \emph{space $A$ of type $\CC$-$\ZZ_{4}$-$\SWF$ at level $s$} is a pointed finite $CW$-complex such that $A^{\ZZ_{2}}\simeq_{\ZZ_{4}}(\wt{\CC}^{s})^{+}$ for some $s\geq 0$, and $\ZZ_{4}$ acts freely on $A\setminus A^{\ZZ_{2}}$.
\end{definition}

Note that as $\ZZ_{4}=\<j\>$-spaces, we have the following isomorphisms:
\begin{align*}
    &\wt{\CC}_{2k}^{\<j\mu^{2^{r-1}}\>}=\{0\}, & &\wt{\CC}_{2k+1}^{\<j\mu^{2^{r-1}}\>}=\wt{\CC}, & &\HH_{2k+\frac{1}{2}}^{\<j\mu^{2^{r-1}}\>}=\CC_{3/2}, & &\HH_{2k+\frac{3}{2}}^{\<j\mu^{2^{r-1}}\>}=\CC_{1/2},
\end{align*}
where $0\le k\le 2^{r-1}-1$, and $\CC_{1/2}$, $\CC_{3/2}$ denote the one-dimensional complex representations on which $j$ acts by $i$ and $-i$, respectively. In particular given a space $X$ of type $\CC$-$G^{\odd}_{2^{r}}$-$\SWF$ at level $\mbfs=\sum_{k}^{2^{r}-1}s_{k}\zeta^{k}\in R(\ZZ_{2^{r}})_{\geq 0}^{\sym}$, we see that $X^{\<j\mu^{2^{r-1}}\>}$ endowed with the residual $\<j\>\cong\ZZ_{4}$-action is a space of type $\CC$-$\ZZ_{4}$-$\SWF$ at level $\sum_{k=0}^{2^{r-1}-1}s_{2k+1}$. 

Finally, we will need to make use of the $RO(\ZZ_{4})$-graded (unstable) homotopy groups of a space of type $\CC$-$\ZZ_{4}$-$\SWF$. Recall from Section \ref{subsec:real_representation_ring} that
\[RO(\ZZ_{4})=\ZZ[\rho,\nu]/(\rho\nu-\nu).\]
(Here, $\nu$ corresponds to $\nu_{1}$ in the notation of Section \ref{subsec:real_representation_ring}.) We will write $\RR$ for the trivial onee-dimensional representation, $\wt{\RR}$ with $\rho=[\wt{\RR}]$ for the irreducible one-dimensional real representation on which $j$ acts by $-1$, and $\VV$ with $\nu=[\VV]$ for the irreducible two-dimensional real representation on which $j$ acts by $\big(\begin{smallmatrix}
  0 & -1\\
  1 & 0
\end{smallmatrix}\big)$.

Given a space $A$ of type $\CC$-$\ZZ_{4}$-$\SWF$ and $r,s,t\in\NN$, we write
\[\pi^{\ZZ_{4}}_{r+s\rho+t\nu}(A):=[S^{r\RR+s\wt{\RR}+t\VV},A]_{\ZZ_{4}}\]
for the set of $\ZZ_{4}$-equivariant homotopy classes of based $\ZZ_{4}$-equivariant maps from the real $\ZZ_{4}$-representation sphere $S^{r\RR+s\wt{\RR}+t\VV}:=(\RR^{r}\oplus\wt{\RR}^{s}\oplus\VV^{t})^{+}$ to $A$, which is a group if $r+s+t>0$. Note that there is a natural restriction map 
\[\res^{\ZZ_{4}}_{1}:\pi^{\ZZ_{4}}_{r+s\rho+t\nu}(A)\to\pi_{r+s+2t}(A)\]
which ``forgets" the $\ZZ_{4}$-equivariant structure, and is a group homomorphism if $r+s+t>0$.

For the statement of the following theorem, we will mainly be interested in the image of $\res^{\ZZ_{4}}_{1}$ modulo torsion, i.e., the subgroup
\[\res^{\ZZ_{4}}_{1}(\pi^{\ZZ_{4}}_{r+s\rho+t\nu}(A)\otimes\QQ)\subset\pi_{r+s+2t}(A)\otimes\QQ.\]
If $r=s=t=0$, then tensoring with the rationals is not a well-defined operation, as $\pi^{\ZZ_{4}}_{0}(A)$ and $\pi_{0}(A)$ are not naturally $\ZZ$-modules. Instead, we will interpret the quantity
\[\res^{\ZZ_{4}}_{1}(\pi^{\ZZ_{4}}_{0}(A)\otimes\QQ):=\res^{\ZZ_{4}}_{1}(\pi^{\ZZ_{4}}_{0}(A))\subset\pi_{0}(A)\]
to be the set of path components of $A$ which intersect non-trivially with $A^{\ZZ_{4}}\simeq S^{0}\subset A$. In particular, $\res^{\ZZ_{4}}_{1}(\pi^{\ZZ_{4}}_{0}(A)\otimes\QQ)\neq 0$.

With this in mind we have the following proposition for spaces of type $\CC$-$G^{\odd}_{2^{r}}$-$\SWF$:

\begin{proposition}
\label{prop:k_invariants_2_r_odd}
Let $X$ be the $G^{\odd}_{2^{r}}$-representation sphere $X=(\mbfs\CC\oplus\mbft\HH)^{+}$ with
\begin{align*}
    &\mbfs=\sum_{j=0}^{2^{r}-1}s_{j}\xi^{2j}\in R(\ZZ_{2m})_{\geq 0}^{\sym,\ev}, & &\mbft=\sum_{k=0}^{2^{r}-1}t_{k+1/2}\xi^{2k+1}\in R(\ZZ_{2m})_{\geq 0}^{\odd},
\end{align*}
and let $X'$ be a space of type $\CC$-$G^{\odd}_{2^{r}}$-$\SWF$ at level $\mbfs'=\sum_{j=0}^{2^{r}-1}s'_{j}\xi^{2j}\in R(\ZZ_{2m})_{\geq 0}^{\sym,\ev}$ such that:
\begin{enumerate}
    \item $\vec{\mbfs}\preceq\vec{\mbfs}\,'$.
    \item $s_{0}<s'_{0}$.
    \item $\sum_{k=0}^{2^{r-a}-1}s_{2^{a}k}\neq \sum_{k=0}^{2^{r-a}-1}s'_{2^{a}k}$ for all $a=0,\dots,r-1$.
    \item $\sum_{k=0}^{2^{a}-1}s_{(2k+1)2^{r-a-1}}\neq\sum_{k=0}^{2^{a}-1}s'_{(2k+1)2^{r-a-1}}$ for all $a=0,\dots,r-2$.
    \item There exists some space $X''$ of type $\CC$-$G^{\odd}_{2^{r}}$-$\SWF$ locally equivalent to $X'$ such that
    \[\res^{\ZZ_{4}}_{1}\Big(\pi^{\ZZ_{4}}_{(\sum_{k=0}^{2^{r-1}-1}2s_{2k+1})\rho+(\sum_{k=0}^{2^{r}-1}t_{k+\frac{1}{2}})\nu}\big((X'_{r-1})^{\<j\mu^{2^{r-1}}\>}\big)\otimes\QQ\Big)=0.\]
\end{enumerate}
Suppose $f:X\to X'$ is a $G^{\odd}_{2^{r}}$-equivariant map such that the induced map $f^{\Pin(2)}$ on $\Pin(2)$-fixed point sets is a $G^{\odd}_{2^{r}}$-homotopy equivalence. Then
\begin{equation}
\label{eq:unstable_2_r_odd_inequality_lattice}
    \vec{k}+(\vec{\mbfs}\,'-\vec{\mbfs})\succeq[\DDD^{\odd}(\vec{\mbft})]+\vec{e}_{0}+\sum_{j=0}^{r-1}\vec{e}_{2^{j}}\qquad\text{ for all }\vec{k}\in\mbfk(X').
\end{equation}
In particular:
\begin{equation}
\label{eq:unstable_2_r_odd_inequality}
    |\vec{\mbfs}\,'-\vec{\mbfs}|\geq |\vec{\mbft}|-|\vec{k}|+r+1\qquad\text{ for all }\vec{k}\in\mbfk(X').
\end{equation}
Furthermore, (\ref{eq:unstable_2_r_odd_inequality_lattice}) and (\ref{eq:unstable_2_r_odd_inequality}) still hold if one replaces Condition $(5)$ above with the following:
\begin{enumerate}
    \item[(5')] $X'$ is locally $\<j\mu^{2^{r-1}}\>$-spherical at some level $d$, and
    \[\frac{1}{2}d\neq\sum_{k=0}^{2^{r-1}-1}s_{2k+1}+\sum_{k=0}^{2^{r}-1}t_{k+1/2}.\]
\end{enumerate}
\end{proposition}

As a corollary, we have the following statement for the case where $r=1$:

\begin{corollary}
\label{cor:k_invariants_2_odd}
Let $X$ be the $G^{\odd}_{2}$-representation sphere
\[X=(\wt{\CC}_{0}^{s_{0}}\oplus\wt{\CC}_{1}^{s_{1}}\oplus\HH_{1/2}^{t_{1/2}}\oplus\HH_{3/2}^{t_{3/2}})^{+},\]
and let $X'$ be a space of type $\CC$-$G^{\odd}_{2}$-$\SWF$ at level $\mbfs'=s'_{0}+s'_{1}\xi^{2}$ such that:
\begin{enumerate}
    \item $s_{0}<s'_{0}$ and $s_{1}<s'_{1}$.
    \item $X'$ is locally equivalent to a space $X''$ of type $\CC$-$G^{\odd}_{2}$-$\SWF$ such that
    \[\res^{\ZZ_{4}}_{1}\Big(\pi^{\ZZ_{4}}_{2s_{1}\rho+(t_{1/2}+t_{3/2})\nu}\big((X'')^{\<j\mu\>}\big)\otimes\QQ\Big)=0.\]
\end{enumerate}
Suppose $f:X\to X'$ is a $G^{\odd}_{2}$-equivariant map such that the induced map $f^{\Pin(2)}:X^{\Pin(2)}\to(X')^{\Pin(2)}$ is a $G^{\odd}_{2}$-homotopy equivalence. Then
\begin{equation}
\label{eq:unstable_2_odd_inequality}
    (s'_{0}-s_{0})+(s'_{1}-s_{1})\geq t_{1/2}+t_{3/2}-|\vec{k}|+2\qquad\text{ for all }\vec{k}\in\mbfk(X').
\end{equation}
Furthermore, (\ref{eq:unstable_2_odd_inequality}) still holds if one replaces Condition $(2)$ above with the following:
\begin{enumerate}
    \item[(2')] $X'$ is locally $\<j\mu\>$-spherical at some level $d$, and
    \[\tfrac{1}{2}d\neq s_{1}+t_{1/2}+t_{3/2}.\]
\end{enumerate}
\end{corollary}

Before we prove Proposition \ref{prop:k_invariants_2_r_odd}, we introduce some helpful lemmas:

\begin{lemma}
\label{lemma:equivalence_of_criteria}
Let $A$ be a space of type $\CC$-$\ZZ_{4}$-$\SWF$ such that $A$ is non-equivariantly homotopy equivalent to a sphere of dimension $d$. Then
\[\res^{\ZZ_{4}}_{1}\Big(\pi^{\ZZ_{4}}_{2s\rho+t\nu}(A)\otimes\QQ\Big)=0\]
if $d\neq 2s+2t$.
\end{lemma}

\begin{proof}
By (\cite{Spanier}, Theorem 9.9) we have that for all $a\geq 1$,
\begin{align*}
&\pi_{i}(S^{2a-1})\otimes\QQ=\twopartdef{\QQ}{i=2a-1,}{0}{\text{else},} &
&\pi_{i}(S^{2a})\otimes\QQ=\twopartdef{\QQ}{i=2a\text{ or }4a-1,}{0}{\text{else}.}
\end{align*}
Hence
\[\pi_{2s+2t}(S^{d})\otimes\QQ=\twopartdef{\QQ}{2s+2t=d,}{0}{2s+2t\neq d,}\]
from which the result follows.
\end{proof}

From the above lemma, we see that Condition $(5')$ of Theorem \ref{prop:k_invariants_2_r_odd} implies Condition $(5)$.

\begin{lemma}
\label{lemma:from_homotopy_to_k_theory}
Let $A$ be a finite $CW$-complex, let $s>0$ be an integer, and let $f:S^{2s}\to A$ be a map such that $[f]=0\in\pi_{2s}(A)\otimes\QQ$. Then the induced map on reduced $K$-theory $f^{*}:\wt{K}(A)\to\wt{K}(S^{2s})$ must be zero.
\end{lemma}

\begin{proof}
We can assume $2s>0$. The correspondence $[f]\mapsto f^{*}$ induces a homomorphism
\[\pi_{2s}(A)\to\Hom\big(\wt{K}(A),\wt{K}(S^{2s})\big).\]
Since $\wt{K}(S^{2s})\cong\ZZ$, any torsion element must be mapped to the zero homomorphism, i.e., the above correspondence factors through the map $\pi_{2s}(A)\to\pi_{2s}(A)\otimes\QQ$.
\end{proof}

The next lemma essentially follows from the series of lemmas used in the proof of (\cite{Bry97}, Theorem 1.2), whose proof we omit:

\begin{lemma}
\label{lemma:traces_2_r}
Let $y\in R(G^{\odd}_{2^{r}})$ be such that
\begin{enumerate}
    \item $\tr_{j}(y)\in\ZZ_{+}$.
    \item $\tr_{e^{i\phi}\mu^{\ell}}=0$ for each $\ell=0,\dots,2^{r}-1$ and each $\phi\in(0,2\pi)$ which is not a rational multiple of $2\pi$.
    \item $\tr_{j\mu^{\ell}}(y)=0$ for each $\ell=1,\dots,2^{r}-1$.
\end{enumerate}
Then for some $\lambda\in\ZZ_{+}$, we have that
\[y=\lambda w_{0}\prod_{a=0}^{r-1}w_{2^{a}}.\]
\end{lemma}

With these lemmas in hand, we are now finally able to prove Proposition \ref{prop:k_invariants_2_r_odd}:

\begin{proof}[Proof of Proposition \ref{prop:k_invariants_2_r_odd}]
Let $X''$ be as in Condition (5) of Proposition $\ref{prop:k_invariants_2_r_odd}$, and let
\begin{align*}
    &g:X'\to X'' & &h:X''\to X'
\end{align*}
be maps inducing local equivalences between $X'$ and $X''$. Define $\wt{f}:X\to X'$ to be the composition
\[X\xrightarrow{f}X'\xrightarrow{g}X''\xrightarrow{h}X'.\]
Since $X',X''$ are both at the same level $\mbfs'$ and $g,h$ induce $G^{\odd}_{2^{r}}$-homotopy equivalences on $\Pin(2)$-fixed point sets, it follows that $\wt{f}$ satisfies all of the same properties as $f$.

For the following, fix $\vec{k}=[(k_{0},\dots,k_{2^{r}-1})]\in\mbfk(X')$ and $x\in\wt{K}_{G^{\odd}_{2^{r}}}(X')$ such that
\[w_{0}\cdot(\iota')^{*}(x)=w_{0}^{k_{0}+1}\cdots w_{2^{r}-1}^{k_{2^{r}-1}}.\]
We claim that it suffices to show that
\begin{equation}
\label{eq:many_w_s}
    \wt{f}^{*}(x)=\lambda w_{0}\prod_{a=0}^{r-1}w_{2^{a}}\in\wt{K}_{G^{\odd}_{2^{r}}}(X)\cong R(G^{\odd}_{2^{r}})
\end{equation}
for some $\lambda\in\ZZ_{+}$. Indeed, consider the following commutative diagram:
\begin{center}
    \begin{tikzcd}[row sep=1em, column sep=1em]
    & \wt{K}_{\Pin(2)}(X')\arrow[dl,swap,"\wt{f}^{*}_{\Pin(2)}"]\arrow[dd,swap,near end,"(\iota')^{*}_{\Pin(2)}"] && \wt{K}_{G^{\odd}_{2^{r}}}(X')\arrow[ll,swap,"\res"]\arrow[dd,swap,near end,"(\iota')^{*}"]\arrow[dl,"\wt{f}^{*}"] \\
    \wt{K}_{\Pin(2)}(X)\arrow[dd,swap,"\iota^{*}_{\Pin(2)}"] && \wt{K}_{G^{\odd}_{2^{r}}}(X)\arrow[ll,swap,near start,"\res"]\arrow[dd,near end, "\iota^{*}"] & \\
    & \wt{K}_{\Pin(2)}((X')^{S^{1}})\arrow[dl,"(\wt{f}^{S^{1}})^{*}_{\Pin(2)}"] && \wt{K}_{G^{\odd}_{2^{r}}}((X')^{S^{1}})\arrow[ll,swap,near end,"\res"]\arrow[dl,"(\wt{f}^{S^{1}})^{*}"] \\
    \wt{K}_{\Pin(2)}(X^{S^{1}}) && \wt{K}_{G^{\odd}_{2^{r}}}(X^{S^{1}})\arrow[ll,"\res"] &
    \end{tikzcd}
\end{center}
We see that
\[w\cdot(\iota'_{\Pin(2)})^{*}(\res(x))=\res(w_{0}\cdot(\iota')^{*}(x))=w^{|\vec{k}|+1}=2^{|\vec{k}|}w\in R(\Pin(2)),\]
and Equation \ref{eq:many_w_s} implies that
\[\wt{f}^{*}_{\Pin(2)}(\res(x))=\res(\lambda w_{0}\prod_{a=0}^{r-1}w_{2^{a}})=\lambda w^{r+1}=\lambda 2^{r}w\in R(\Pin(2)).\]
As in the proof of (\cite{Man14}, Lemma 3.10) we have that 
\[((\wt{f}^{S^{1}})^{*}_{\Pin(2)}\circ(\iota')_{\Pin(2)}^{*}\circ\res)(x)=(w^{|\vec{s}'|-|\vec{s}|})\cdot(\iota'_{\Pin(2)})^{*}(\res(x))\]
\[=2^{|\vec{s}'|-|\vec{s}|-1}(w\cdot(\iota'_{\Pin(2)})^{*}(\res(x))=(2^{|\vec{s}'|-|\vec{s}|-1})(2^{|\vec{k}|}w)\]
\[=2^{|\vec{s}'|-|\vec{s}|+|\vec{k}|-1}w.\]
Since $\iota^{*}_{\Pin(2)}$ is given by multiplication by $z^{|\vec{t}|}\in R(\Pin(2))$, we have that
\[2^{|\vec{s}'|-|\vec{s}|+|\vec{k}|-1}w=((\wt{f}^{S^{1}})^{*}_{\Pin(2)}\circ(\iota')_{\Pin(2)}^{*}\circ\res)(x)\]
\[=(\iota^{*}_{\Pin(2)}\circ \wt{f}^{*}_{\Pin(2)}\circ\res)(x)=\iota^{*}_{\Pin(2)}(\lambda 2^{r}w)\]
\[=z^{|\vec{t}|}\cdot\lambda 2^{r}w=\lambda 2^{|\vec{t}|+r}w.\]
We therefore obtain the desired inequality
\begin{align*}
   &|\vec{s}'|-|\vec{s}|+|\vec{k}|-1\geq |\vec{t}|+r \\
\iff &|\vec{s}'|-|\vec{s}|\geq |\vec{t}|+r+1-|\vec{k}|.
\end{align*}

\bigskip

In order to show Equation \ref{eq:many_w_s}, by Lemma \ref{lemma:traces_2_r} it suffices to show that:
\begin{enumerate}
    \item $\tr_{j}(\wt{f}^{*}(x))\in\ZZ_{+}$.
    \item $\tr_{e^{i\phi}\mu^{(2\ell+1)2^{a}}}(\wt{f}^{*}(x))=0$ for all $a\in\{0,\dots,r-1\}$, for all $\ell\in\{0,\dots,2^{r-a}-1\}$, and for all $\phi\in(0,2\pi)$ not a rational multiple of $2\pi$.
    \item $\tr_{j\mu^{(2\ell+1)2^{a}}}(\wt{f}^{*}(x))=0$ for all $a\in\{0,\dots,r-2\}$ and $\ell\in \{0,\dots,2^{r-a-1}-1\}$.
    \item $\tr_{j\mu^{2^{r-1}}}(\wt{f}^{*}(x))=0$.
\end{enumerate}

\bigskip

To prove (1), note from the above diagram that
\[z^{|\vec{t}|}\cdot \wt{f}^{*}_{\Pin(2)}(\res(x))=((\iota^{*}_{\Pin(2)}\circ \wt{f}^{*}_{\Pin(2)}(\res(x))\]
\[=((\wt{f}^{S^{1}})^{*}_{\Pin(2)}\circ(\iota')^{*}_{\Pin(2)})(\res(x))=2^{|\vec{s}'|-|\vec{s}|+|\vec{k}|-1}w\in R(\Pin(2)).\]
As in the proof of (\cite{Man14}, Lemma 3.11), the only way this equation can be satisfied is if $\wt{f}^{*}_{\Pin(2)}(\res(x))=\lambda'w$ for some $\lambda'\in\ZZ_{+}$. Hence
\[\tr_{j}(\wt{f}^{*}(x))=\tr_{j}(\wt{f}_{\Pin(2)}^{*}(\res(x)))=\tr_{j}(\lambda'w)=2\lambda'\in\ZZ_{+}.\]

To prove (2), we restrict to the subgroup
\[G_{a}:=S^{1}\times_{\ZZ_{2}}\<\mu^{2^{a}}\>\subset G^{\odd}_{2},\]
where $a\in\{0,\dots,r-1\}$. The representation ring of $G_{a}$ can be identified with
\begin{align*}
	R(G_{a})&=\ZZ[\alpha^{2},\alpha\theta,\alpha\theta^{-1}]/(\alpha^{2^{r-a+1}}-1,(\alpha\theta)(\alpha\theta^{-1})-\alpha^{2}) \\
	&\subset\ZZ[\alpha,\theta,\theta^{-1}]/(\alpha^{2^{r-a+1}}-1,\theta\theta^{-1}-1)=R(S^{1})\otimes R(\ZZ_{2^{r-a+1}}),
\end{align*}
and one can show that the restriction map $\res:R(G^{\odd}_{2^{r}})\to R(G_{a})$ on the level of representation rings is given by
\begin{align*}
    \xi^{2}&\mapsto\alpha^{2} & w_{k}&\mapsto(1-\alpha^{2k})\\
    \wt{c}&\mapsto 1 & z_{k+\frac{1}{2}}&\mapsto(1-\theta\alpha^{2k+1})(1-\theta^{-1}\alpha^{2k+1}). \\
    \xi h&\mapsto \alpha(\theta+\theta^{-1}) &
\end{align*}
We denote by $\CC_{i,k}$ the $G_{a}$-representation corresponding to $\theta^{i}\alpha^{k}\in R(G_{a})$, where $i\in\ZZ$ and $k\in\{0,\dots, 2^{r-a+1}-1\}$ are such that $i\equiv k\pmod{2}$. Then we see that:
\begin{align*}
    &\res^{G^{\odd}_{2^{r}}}_{G_{a}}:\wt{\CC}_{k}\mapsto\CC_{0,2k}, & &\res^{G^{\odd}_{2^{r}}}_{G_{a}}:\HH_{k+\frac{1}{2}}\mapsto\CC_{1,2k+1}\oplus\CC_{-1,2k+1}.
\end{align*}
It follows that
\[X\simeq_{G_{a}}\Bigg(\bigoplus_{k=0}^{2^{r-a}-1}\Big(\CC_{0,2k}^{\sum_{\ell=0}^{2^{a}-1}s_{k+\ell 2^{r-a}}}\oplus\CC_{1,2k+1}^{\sum_{\ell=0}^{2^{a}-1}t_{k+\frac{1}{2}+\ell 2^{r-a}}}\oplus\CC_{-1,2k+1}^{\sum_{\ell=0}^{2^{a}-1}t_{k+\frac{1}{2}+\ell 2^{r-a}}}\Big)\Bigg)^{+},\]
and:
\begin{align*}
    &X^{G_{a}}\simeq_{G_{a}}\Big(\CC_{0,0}^{\sum_{\ell=0}^{2^{a}-1}s_{\ell 2^{r-a}}}\Big)^{+}, & &(X')^{G_{a}}\simeq_{G_{a}}\Big(\CC_{0,0}^{\sum_{\ell=0}^{2^{a}-1}s'_{\ell 2^{r-a}}}\Big)^{+}.
\end{align*}
Now let 
\begin{align*}
    &e_{a}:X^{G_{a}}\hookrightarrow X & &e_{a}':(X')^{G_{a}}\hookrightarrow X'
\end{align*}
denote the inclusions of the $G_{a}$-fixed point sets, and consider the following commutative diagram:
\begin{center}
    \begin{tikzcd}[row sep=3 em, column sep=9em]
    \wt{K}_{G^{\odd}_{2^{r}}}(X)\arrow[d,"\res"] & \wt{K}_{G^{\odd}_{2^{r}}}(X')\arrow[l,swap,"\wt{f}^{*}"]\arrow[d,"\res"] \\
    \wt{K}_{G_{a}}(X)\arrow[d,"e_{a}^{*}"] & \wt{K}_{G_{a}}(X')\arrow[l,swap,"\wt{f}_{G_{a}}^{*}"]\arrow[d,"(e_{a}')^{*}"]\\
    \wt{K}_{G_{a}}(X^{G_{a}})\arrow[d,"\cong"] & \wt{K}_{G_{a}}((X')^{G_{a}})\arrow[l,swap,"(\wt{f}^{G_{a}})_{G_{a}}^{*}"]\arrow[d,"\cong"] \\
    \wt{K}(X^{G_{a}})\otimes R(G_{a}) & \wt{K}\big((X')^{G_{a}}\big)\otimes R(G_{a}).\arrow[l,swap,"(\wt{f}^{G_{a}})^{*}\otimes\id"] 
    \end{tikzcd}
\end{center}

Here, $(\wt{f}^{G_{a}})^{*}$ is the map
\[(\wt{f}^{G_{a}})^{*}:\wt{K}(S^{\sum_{\ell=0}^{2^{a}-1}2s'_{\ell 2^{r-a}}})\cong\wt{K}\big((X')^{G_{a}}\big)\to\wt{K}(X^{G_{a}})\cong\wt{K}(S^{\sum_{\ell=0}^{2^{a}-1}2s_{\ell 2^{r-a}}})\]
induced by
\[\wt{f}^{G_{a}}:S^{\sum_{\ell=0}^{2^{a}-1}2s_{\ell 2^{r-a}}}\simeq X^{G_{a}}\to (X')^{G_{a}}\simeq S^{\sum_{\ell=0}^{2^{a}-1}2s_{\ell 2^{r-a}}}.\]
Since by assumption $\sum_{\ell=0}^{2^{r-a}-1}s'_{2^{a}\ell}\neq\sum_{\ell=0}^{2^{r-a}-1}s_{2^{a}\ell}$, it follows from Lemma \ref{lemma:from_homotopy_to_k_theory} that $(\wt{f}^{G_{a}})^{*}$ is zero, and hence the map
\[(\wt{f}^{G_{a}})^{*}_{G_{a}}:\wt{K}_{G_{a}}\big((X')^{G_{a}}\big)\to\wt{K}_{G_{a}}(X^{G_{a}})\]
must be zero as well. So by commutativity of the above diagram, we therefore must have that
\[(e_{a}^{*}\circ\res) (\wt{f}^{*}(x))=((\wt{f}^{G_{a}})_{G_{a}}^{*}\circ(e_{a}')^{*}\circ\res)(x)=0,\]
or in other words, $\res(\wt{f}^{*}(x))$ lies in the kernel of $e_{a}^{*}$. We can identify the map
\[e_{a}^{*}:R(G_{a})\cong\wt{K}_{G_{a}}(X)\to\wt{K}_{G_{a}}(X^{G_{a}})\cong R(G_{a})\]
as multiplication by the element
\[y_{e_{a}^{*}}:=\Big(\prod_{k=1}^{2^{r-a}-1}(1-\alpha^{2k})^{\sum_{\ell=0}^{2^{a}-1}s_{k+\ell 2^{r-a}}}\Big)\Big(\prod_{k=0}^{2^{r-a}-1}\big((1-\theta\alpha^{2k+1})(1-\theta^{-1}\alpha^{2k+1})\big)^{\sum_{\ell=0}^{2^{r-a}-1}t_{k+\frac{1}{2}+\ell 2^{r-a}}}\Big).\]
Since $\res(\wt{f}^{*}(x))\in\ker(e_{a}^{*})$, in particular we must have that
\[\tr_{h}(y_{e_{a}^{*}}\cdot\res(\wt{f}^{*}(x)))=0\]
for all $h\in G_{a}$. Now let $\phi\in(0,2\pi)$ be an irrational multiple of $2\pi$, and let $\ell\in\{0,\dots,2^{r-a}-1\}$. Then in particular
\[\tr_{e^{i\phi}\mu^{(2\ell+1)2^{a}}}(y_{e_{a}^{*}}\cdot\res(\wt{f}^{*}(x)))=0\]
for all such $\phi,\ell$. But note that
\begin{align*}
    &\tr_{e^{i\phi}\mu^{(2\ell+1)2^{a}}}(y_{e_{a}^{*}}) \\
    &\qquad\qquad=\Big(\prod_{k=1}^{2^{r-a}-1}(1-\omega_{2^{r-a+1}}^{2k(2\ell+1)})^{\sum_{\ell=0}^{2^{a}-1}s_{k+\ell 2^{r-a}}}\Big) \\
    &\qquad\qquad\qquad\qquad\cdot\Big(\prod_{k=0}^{2^{r-a}-1}\big((1-e^{i\phi}\omega_{2^{r-a+1}}^{(2k+1)(2\ell+1)})(1-e^{-i\phi}\omega_{2^{r-a+1}}^{(2k+1)(2\ell+1)})\big)^{\sum_{\ell=0}^{2^{r-a}-1}t_{k+\frac{1}{2}+\ell 2^{r-a}}}\Big) \\
    &\qquad\qquad\neq 0
\end{align*}
for all such $\phi,\ell$. Indeed,
\[\omega_{2^{r-a+1}}^{2k(2\ell+1)}=1\iff 2k(2\ell+1)\equiv 0\pmod{2^{r-a+1}} \iff k\equiv 0\pmod{2^{r-a}},\]
but such $k$ are excluded from the first term of the above product. Furthermore since $\phi$ is an irrational multiple of $2\pi$, we have that
\[e^{i\phi}\omega_{2^{r-a+1}}^{(2k+1)(2\ell+1)},e^{-i\phi}\omega_{2^{r-a+1}}^{(2k+1)(2\ell+1)}\neq 1\]
for any $k,\ell$, and so the second term of the product is non-zero as well. It follows then that
\[\tr_{e^{i\phi}\mu^{(2\ell+1)2^{a}}}(\wt{f}^{*}(x))=\tr_{e^{i\phi}\mu^{(2\ell+1)2^{a}}}(\res(\wt{f}^{*}(x)))=0\]
for all such $\phi$, $\ell$.

\bigskip

For (3), let $a\in\{0,\dots,r-2\}$, let $\ell\in\{0,\dots,2^{r-a-1}-1\}$ and consider what happens when we restrict to the subgroup
\[H_{a,\ell}:=\<j\mu^{(2\ell+1)2^{a}}\>\cong\ZZ_{2^{r-a+1}}\subset G^{\odd}_{2^{r}}.\]
The restriction map
\[\res:R(G^{\odd}_{2^{r}})\to R(H_{a,\ell})\cong R(\ZZ_{2^{r-a+1}})\cong\ZZ[\alpha]/(\alpha^{2^{r-a+1}}-1)\]
on the level of representation rings is given by
\begin{align*}
    \xi^{2} &\mapsto \alpha^{4\ell+2}, \\
    \wt{c} &\mapsto \alpha^{2^{r-a}}, \\
    \xi h &\mapsto \alpha^{2\ell+1+2^{r-a-1}}+\alpha^{2\ell+1-2^{r-a-1}}, \\
    w_{k} &\mapsto 1-\alpha^{2k(2\ell+1)+2^{r-a}}, \\
    z_{k+1/2} &\mapsto (1-\alpha^{(2k+1)(2\ell+1)+2^{r-a-1}})(1-\alpha^{(2k+1)(2\ell+1)-2^{r-a-1}}).
\end{align*}
Let $\CC_{k}$ denote the $H_{a,\ell}$-representation space corresponding to $\alpha^{k}\in R(H_{a,\ell})$. Then as an $H_{a,\ell}$-representation sphere, we see that
\[X\simeq_{H_{a,\ell}}\Bigg(\bigoplus_{k=0}^{2^{r-a}-1}\Big(\CC_{2k(2\ell+1)+2^{r-a}}^{\sum_{b=0}^{2^{a}-1}s_{k+b2^{r-a}}}\oplus\CC_{(2k+1)(2\ell+1)+2^{r-a-1}}^{\sum_{b=0}^{2^{a}-1}2t_{k+\frac{1}{2}+b2^{r-a}}}\Big)\Bigg)^{+}.\]
Note that
\[2k(2\ell+1)+2^{r-a}\equiv 0\pmod{2^{r-a+1}}\]
if and only if $k=(2c+1)2^{r-a-1}$ for some $c=0,\dots,2^{a}-1$, and:
\begin{align*}
    &(2k+1)(2\ell+1)+2^{r-a-1}\not\equiv 0\pmod{2^{r-a+1}}, \\
    &(2k+1)(2\ell+1)-2^{r-a-1}\not\equiv 0\pmod{2^{r-a+1}},
\end{align*}
for any $k$, since in particular $(2k+1)(2\ell+1)$ is odd and $2^{r-a-1}$ is even. Therefore the $H_{a,\ell}$-fixed point set of $X$ is given by
\[X^{H_{a,\ell}}=\Big(\CC_{0}^{\sum_{k=0}^{2^{a}-1}s_{(2k+1)2^{r-a-1}}}\Big)^{+}.\]
Similarly, note that since $-1=(j\mu^{(2\ell+1)2^{a}})^{2^{r-a}}\in H_{a,\ell}$, we have that $(X')^{H_{a,\ell}}\subset(X')^{\<-1\>}=(X')^{S^{1}}$. Hence 
\[(X')^{H_{a,\ell}}=\big((X')^{S^{1}}\big)^{H_{a,\ell}}=\Big(\CC_{0}^{\sum_{k=0}^{2^{a}-1}s'_{(2k+1)2^{r-a-1}}}\Big)^{+}.\]
Now let 
\begin{align*}
    &\varepsilon_{a}:X^{H_{a,\ell}}\hookrightarrow X & &\varepsilon'_{a}:(X')^{H_{a,\ell}}\hookrightarrow X'
\end{align*}
denote the inclusions of the $H_{a,\ell}$-fixed point-sets and consider the following commutative diagram:
\begin{center}
    \begin{tikzcd}[row sep=3 em, column sep=9em]
    \wt{K}_{G^{\odd}_{2^{r}}}(X)\arrow[d,"\res"] & \wt{K}_{G^{\odd}_{2^{r}}}(X')\arrow[l,swap,"\wt{f}^{*}"]\arrow[d,"\res"] \\
    \wt{K}_{H_{a,\ell}}(X)\arrow[d,"\varepsilon_{a}^{*}"] & \wt{K}_{H_{a,\ell}}(X')\arrow[l,swap,"\wt{f}_{H_{a,\ell}}^{*}"]\arrow[d,"(\varepsilon'_{a})^{*}"] \\
    \wt{K}_{H_{a,\ell}}(X^{H_{a,\ell}})\arrow[d,"\cong"] & \wt{K}_{H_{a,\ell}}\big((X')^{H_{a,\ell}}\big).\arrow[d,"\cong"]\arrow[l,swap,"(\wt{f}^{H_{a,\ell}})_{H_{a,\ell}}^{*}"] \\
    \wt{K}(X^{H_{a,\ell}})\otimes R(H_{a,\ell}) & \wt{K}\big((X')^{H_{a,\ell}}\big)\otimes R(H_{a,\ell}).\arrow[l,swap,"(\wt{f}^{H_{a,\ell}})^{*}\otimes\id"]
    \end{tikzcd}
\end{center}
Here,
\[(\wt{f}^{H_{a,\ell}})^{*}:\wt{K}\big((X')^{H_{a,\ell}}\big)\cong\wt{K}(S^{\sum_{k=0}^{2^{a}-1}2s'_{(2k+1)2^{r-a-1}}})\to\wt{K}(S^{\sum_{k=0}^{2^{a}-1}2s_{(2k+1)2^{r-a-1}}})\cong\wt{K}(X^{H_{a,\ell}})\]
is the map induced by
\[\wt{f}^{H_{a,\ell}}:X^{H_{a,\ell}}\simeq S^{\sum_{k=0}^{2^{a}-1}2s_{(2k+1)2^{r-a-1}}}\to S^{\sum_{k=0}^{2^{a}-1}2s'_{(2k+1)2^{r-a-1}}}\simeq(X')^{H_{a,\ell}}.\]
By our assumption that $\sum_{k=0}^{2^{a}-1}s_{(2k+1)2^{r-a-1}}\neq\sum_{k=0}^{2^{a}-1}s'_{(2k+1)2^{r-a-1}}$ and Lemma \ref{lemma:from_homotopy_to_k_theory} it follows that $(\wt{f}^{H_{a,\ell}})^{*}$ is zero, and hence the map
\[(\wt{f}^{H_{a,\ell}})_{H_{a,\ell}}^{*}:\wt{K}_{H_{a,\ell}}\big((X')^{H_{a,\ell}}\big)\to\wt{K}_{H_{a,\ell}}(X^{H_{a,\ell}})\]
is zero as well. By commutativity of the above diagram, we therefore must have that
\[(\varepsilon_{a}^{*}\circ\res) (\wt{f}^{*}(x))=\big((\wt{f}^{H_{a,\ell}})_{H_{a,\ell}}^{*}\circ(\varepsilon'_{a})^{*}\circ\res\big)(x)=0,\]
or in other words, $\res(\wt{f}^{*}(x))$ lies in the kernel of $\varepsilon_{a}^{*}$. We can identify the map
\[\varepsilon_{a}^{*}:R(H_{a,\ell})\cong\wt{K}_{H_{a,\ell}}(X)\to\wt{K}_{H_{a,\ell}}(X^{H_{a,\ell}})\cong R(H_{a,\ell})\]
as multiplication by the element
\begin{align*}
    y_{\varepsilon_{a}^{*}}&:=\Bigg(\prod_{k=1}^{2^{r-a}-1}(1-\alpha^{2k(2\ell+1)+2^{r-a}})^{\sum_{b=0}^{2^{a}-1}s_{k+b2^{r-a}}}\Bigg) \\
    &\qquad\qquad\cdot\Bigg(\prod_{k=1}^{2^{r-a}-1}\Big((1-\alpha^{(2k+1)(2\ell+1)+2^{r-a-1}})(1-\alpha^{(2k+1)(2\ell+1)-2^{r-a-1}})\Big)^{\sum_{b=0}^{2^{a}-1}t_{k+\frac{1}{2}+b2^{r-a}}}\Bigg).
\end{align*}
Since $\res(\wt{f}^{*}(x))\in\ker(\varepsilon_{a}^{*})$, in particular we must have that
\[\tr_{h}\big(y_{\varepsilon_{a}^{*}}\cdot\res(\wt{f}^{*}(x))\big)=0\]
for all $h\in H_{a,\ell}$. But note that
\begin{align*}
    &\tr_{j\mu^{(2\ell+1)2^{a}}}(y_{\varepsilon_{a}^{*}}) \\
    &\qquad\qquad=\Bigg(\prod_{k=1}^{2^{r-a}-1}(1-\omega_{2^{r-a+1}}^{2k(2\ell+1)+2^{r-a}})^{\sum_{b=0}^{2^{a}-1}s_{k+b2^{r-a}}}\Bigg) \\
    &\qquad\qquad\qquad\qquad\cdot\Bigg(\prod_{k=0}^{2^{r-a}-1}\Big((1-\omega_{2^{r-a+1}}^{(2k+1)(2\ell+1)+2^{r-a-1}})(1-\omega_{2^{r-a+1}}^{(2k+1)(2\ell+1)-2^{r-a-1}})\Big)^{\sum_{b=0}^{2^{r-a}-1}t_{k+\frac{1}{2}+b2^{r-a}}}\Bigg) \\
    &\qquad\qquad\neq 0.
\end{align*}
Indeed,
\[\omega_{2^{r-a+1}}^{2k(2\ell+1)+2^{r-a}}=1\iff 2k(2\ell+1)\equiv 2^{r-a}\pmod{2^{r-a+1}} \iff k\equiv 0\pmod{2^{r-a}},\]
but such $k$ are excluded from the first term of the above product. Furthermore 
\begin{align*}
    &\omega_{2^{r-a+1}}^{(2k+1)(2\ell+1)+2^{r-a-1}}=1\iff (2k+1)(2\ell+1)\equiv -2^{r-a-1}\pmod{2^{r-a+1}}, \\
    &\omega_{2^{r-a+1}}^{(2k+1)(2\ell+1)-2^{r-a-1}}=1\iff (2k+1)(2\ell+1)\equiv 2^{r-a-1}\pmod{2^{r-a+1}},
\end{align*}
neither of which can happen since we assumed $a\le r-2$, and so the second term of the product must be non-zero as well. It follows then that
\[\tr_{j\mu^{(2\ell+1)2^{a}}}(\wt{f}^{*}(x))=\tr_{j\mu^{(2\ell+1)2^{a}}}(\res(\wt{f}^{*}(x)))=0\]
for all $a\in\{0,\dots,r-2\}$ and $\ell\in\{0,\dots,2^{r-a-1}-1\}$.

For (4), consider the following subgroups of $G^{\odd}_{2^{r}}$:
\begin{align*}
&H_{r-1}:=\<j\mu^{2^{r-1}}\>\cong\ZZ_{2},	 &
&\<j\>\cong\ZZ_{4}, &
&H_{r-1}\times\<j\>\cong\<j\mu^{2^{r-1}},j\>\cong\ZZ_{2}\times\ZZ_{4}.	
\end{align*}
The restriction map
\[\res:R(G^{\odd}_{2^{r}})\to R(H_{r-1}\times\<j\>)=R(\ZZ_{2})\otimes R(\ZZ_{4})\cong\ZZ[\alpha,\beta]/(\alpha^{2}-1,\beta^{4}-1)\]
is given by
\begin{align*}
    \xi^{2} &\mapsto \alpha, & \wt{c} &\mapsto \alpha\beta^{2}, & \xi h &\mapsto\alpha\beta+\beta^{3}, \\
    w_{k} &\mapsto 1-\alpha^{k+1}\beta^{2}, & z_{k+\frac{1}{2}} &\mapsto (1-\alpha^{k+1}\beta)(1-\alpha^{k}\beta^{3}). & \xi^{2k+1}h &\mapsto\alpha^{k+1}\beta+\alpha^{k}\beta^{3}
\end{align*}
Let $\CC_{a,b}$ denote the $H_{r-1}\times\<j\>$-representation space corresponding to $\alpha^{a}\beta^{b}\in H_{r-1}\times\<j\>$. Then
\begin{align*}
	X\simeq_{H_{r-1}\times\<j\>}\Big(\CC_{0,2}^{\sum_{k=0}^{2^{r-1}-1}s_{2k+1}}&\oplus\CC_{0,1}^{\sum_{k=0}^{2^{r-1}-1}t_{2k+3/2}}\oplus\CC_{0,3}^{\sum_{k=0}^{2^{r-1}-1}t_{2k+1/2}} \\
	&\oplus\CC_{1,2}^{\sum_{k=0}^{2^{r-1}-1}s_{2k}}\oplus\CC_{1,1}^{\sum_{k=0}^{2^{r-1}-1}t_{2k+1/2}}\oplus\CC_{1,3}^{\sum_{k=0}^{2^{r-1}-1}t_{2k+3/2}}\Big)^{+}
\end{align*}
as a $(H_{r-1}\times\<j\>)$-representation sphere. Therefore the $H_{r-1}$-fixed point set of $X$ as a $(H_{r-1}\times\<j\>)$-representation sphere is given by
\[X^{H_{r-1}}\simeq_{H_{r-1}\times\<j\>}\Big(\CC_{0,2}^{\sum_{k=0}^{2^{r-1}-1}s_{2k+1}}\oplus\CC_{0,1}^{\sum_{k=0}^{2^{r-1}-1}t_{2k+3/2}}\oplus\CC_{0,3}^{\sum_{k=0}^{2^{r-1}-1}t_{2k+1/2}}\Big)^{+}.\]
Alternatively, we can express $X^{H_{r-1}}$ as the real $\<j\>\cong\ZZ_{4}$-representation sphere
\[X^{H_{r-1}}\simeq_{\<j\>}\Big(\wt{\RR}^{\sum_{k=0}^{2^{r-1}-1}2s_{2k+1}}\oplus\VV^{\sum_{k=0}^{2^{r}-1}t_{k+1/2}}\Big)^{+}.\]
Now let 
\begin{align*}
    &\varepsilon_{r-1}:X^{H_{r-1}}\hookrightarrow X &
    &\varepsilon'_{r-1}:(X')^{H_{r-1}}\hookrightarrow X'
\end{align*}
denote the inclusions of the $H_{r-1}$-fixed point-sets and consider the following commutative diagram:
\begin{center}
    \begin{tikzcd}[row sep=3 em, column sep=-2em]
    \wt{K}_{G^{\odd}_{2^{r}}}(X)\arrow[dd,"\res"]\arrow[dr,"\res"] & & \wt{K}_{G^{\odd}_{2^{r}}}(X')\arrow[ll,swap,"\wt{f}^{*}"]\arrow[dd,near start,"\res"]\arrow[dr,"\res"] & \\
    & \wt{K}_{H_{r-1}\times\<j\>}(X)\arrow[dd,near end,"\varepsilon_{r-1}^{*}"]\arrow[dl,"\res"] & & \wt{K}_{H_{r-1}\times\<j\>}(X')\arrow[dd,"(\varepsilon'_{r-1})^{*}"]\arrow[dl,"\res"]\arrow[ll,swap,near end,"\wt{f}_{H_{r-1}\times\<j\>}^{*}"]  \\
    \wt{K}_{H_{r-1}}(X)\arrow[dd,"\varepsilon_{r-1}^{*}"] & & \wt{K}_{H_{r-1}}(X')\arrow[dd,near end,"(\varepsilon'_{r-1})^{*}"]\arrow[ll,swap,near start,"\wt{f}_{H_{r-1}}^{*}"] & \\
    & \wt{K}_{H_{r-1}\times\<j\>}(X^{H_{r-1}})\arrow[dd,near end,"\cong"]\arrow[dl,swap,"\res"] & & \wt{K}_{H_{r-1}\times\<j\>}\big((X')^{H_{r-1}}\big)\arrow[dd,near end,"\cong"]\arrow[dl,"\res"]\arrow[ll,swap,"(\wt{f}^{H_{r-1}})_{H_{r-1}\times\<j\>}^{*}"] \\
    \wt{K}_{H_{r-1}}(X^{H_{r-1}})\arrow[dd,near end,"\cong"] & & \wt{K}_{H_{r-1}}((X')^{H_{r-1}})\arrow[dd,near end,"\cong"]\arrow[ll,swap,near start,"(\wt{f}^{H_{r-1}})_{H_{r-1}}^{*}"] & \\
    & \wt{K}_{\<j\>}(X^{H_{r-1}})\otimes R(H_{r-1})\arrow[dl,"\res\otimes\id"] & & \wt{K}_{\<j\>}((X')^{H_{r-1}})\otimes R(H_{r-1})\arrow[dl,"\res\otimes\id"]\arrow[ll,swap,"(\wt{f}^{H_{r-1}})^{*}_{\<j\>}
    \otimes\id"] \\
    \wt{K}(X^{H_{r-1}})\otimes R(H_{r-1}) & & \wt{K}((X')^{H_{r-1}})\otimes R(H_{r-1})\arrow[ll,swap,"(\wt{f}^{H_{r-1}})^{*}
    \otimes\id"] &    \end{tikzcd}
\end{center}
Here,
\begin{align*}
	&(\wt{f}^{H_{r-1}})^{*}_{\<j\>}:\wt{K}_{\<j\>}\big((X')^{H_{r-1}}\big)\to\wt{K}_{\<j\>}(X^{H_{r-1}}) \\
	&(\wt{f}^{H_{r-1}})^{*}:\wt{K}\big((X')^{H_{r-1}}\big)\to\wt{K}(X^{H_{r-1}})
\end{align*}
are the maps induced by
\[\wt{f}^{H_{r-1}}:X^{H_{r-1}}\simeq_{\<j\>}\Big(\wt{\RR}^{\sum_{k=0}^{2^{r-1}-1}2s_{2k+1}}\oplus\VV^{\sum_{k=0}^{2^{r}-1}t_{k+1/2}}\Big)^{+}\to (X')^{H_{r-1}}.\]
Recall that by construction, $\wt{f}$ factors as a composition $X\xrightarrow{g\circ f}X''\xrightarrow{h}X'$. Hence $\wt{f}^{H_{r-1}}$ factors through the map
\[(g\circ f)^{H_{r-1}}:X^{H_{r-1}}\simeq_{\<j\>}\Big(\wt{\RR}^{\sum_{k=0}^{2^{r-1}-1}2s_{2k+1}}\oplus\VV^{\sum_{k=0}^{2^{r}-1}t_{k+1/2}}\Big)^{+}\to (X'')^{H_{r-1}},\]
and $(\wt{f}^{H_{r-1}})^{*}_{\<j\>}$, $(\wt{f}^{H_{r-1}})^{*}$ factor through the maps
\begin{align*}
	&((g\circ f)^{H_{r-1}})^{*}_{\<j\>}:\wt{K}_{\<j\>}\big((X'')^{H_{r-1}}\big)\to\wt{K}_{\<j\>}(X^{H_{r-1}}), \\
	&((g\circ f)^{H_{r-1}})^{*}:\wt{K}\big((X'')^{H_{r-1}}\big)\to\wt{K}(X^{H_{r-1}}),
\end{align*}
respectively. Condition (5) is equivalent to the assertion that
\[\res^{\ZZ_{4}}_{1}\Big(\pi^{\ZZ_{4}}_{(\sum_{k=0}^{2^{r-1}-1}2s_{2k+1})\rho+(\sum_{k=0}^{2^{r}-1}t_{k+1/2})\nu}\big((X'')^{H_{r-1}}\big)\otimes\QQ\Big)=0.\]
Therefore by Lemma \ref{lemma:from_homotopy_to_k_theory} we have that $((g\circ f)^{H_{r-1}})^{*}=0$, and hence $\big(\wt{f}^{H_{r-1}}\big)_{H_{r-1}}^{*}=0$, which by commutativity of the above diagram implies that 
\[\res(\wt{f}^{*}(x))\in\ker\Big(\wt{K}_{H_{r-1}}(X)\xrightarrow{\varepsilon_{r-1}^{*}}\wt{K}_{H_{r-1}}(X^{H_{r-1}})\Big).\]

We can identify the map
\[\varepsilon_{r-1}^{*}:R(H_{r-1})\cong\wt{K}_{H_{r-1}}(X)\to\wt{K}_{H_{r-1}}(X^{H_{r-1}})\cong R(H_{r-1})\]
as multiplication by the element
\[y_{\varepsilon_{r-1}^{*}}:=(1-\alpha)^{(\sum_{k=0}^{2^{r-1}-1}s_{2k})+(\sum_{k=0}^{2^{r}-1}t_{k+1/2})}\in R(H_{r-1})=\ZZ[\alpha]/(\alpha^{2}-1).\]
Now note that since
\[\tr_{\<j\mu^{2^{r-1}}\>}(y_{\varepsilon_{r-1}^{*}})=2^{(\sum_{k=0}^{2^{r-1}-1}s_{2k})+(\sum_{k=0}^{2^{r}-1}t_{k+1/2})}\neq 0,\]
by a similar argument as above it follows that
\[\tr_{\<j\mu^{2^{r-1}}\>}(\wt{f}^{*}(x))=0,\]
which was what was to be proven.
\end{proof}

\bigskip
\subsection{\texorpdfstring{$p^{r}$}{pr}-fold Actions}
\label{subsec:p_r_actions}

For this section let $m=p^{r}$ be an odd prime power. In this setting we can extract more tractable invariants by projecting onto a $2$-dimensional sub-lattice, whose coordinates represent the ``trivial" and ``non-trivial" parts, respectively.

Let $\Pi:(\NN^{p^{r}},\preceq,+,|\cdot|)\to(\N^{p^{r}},\preceq,+,|\cdot|)$ denote the defining projection, and consider the surjection of $\NN$-graded additive posets
\begin{align*}
    \wt{\pi}:(\NN^{p^{r}},\preceq,+,|\cdot|)&\to(\NN^{2},\preceq,+,|\cdot|) \\
    (a_{0},\dots,a_{p^{r}-1})&\mapsto (a_{0},a_{1}+\cdots+a_{p^{r}-1}).
\end{align*}

\begin{proposition}
\label{prop:lattice_p^r_projection}
There exists a surjection of $\NN$-graded additive posets
\[\pi:(\N^{p^{r}},\preceq,+,|\cdot|)\to(\NN^{2},\preceq,+,|\cdot|)\]
such that the following diagram commutes:
\begin{center}
\begin{tikzcd}[column sep=small]
(\NN^{p^{r}},\preceq,+,|\cdot|) \arrow[rdd, swap, "\wt{\pi}"] \arrow[rr,"\Pi"] & & (\N^{p^{r}},\preceq,+,|\cdot|)\arrow[ldd, "\pi"] \\
& &  \\
& (\NN^{2},\preceq,+,|\cdot|). &
\end{tikzcd}
\end{center}
\end{proposition}

\begin{proof}
Proposition \ref{prop:monomials} implies that if 
\[\vec{a}=(a_{0},\dots,a_{p^{r}-1}),\;\vec{b}=(b_{0},\dots,b_{p^{r}-1})\in\NN^{p^{r}}\]
are such that $[\vec{a}]=[\vec{b}]\in \N^{p^{r}}$, then $a_{0}=b_{0}$. This implies that the zero-th coordinate of an element $[\vec{a}]\in\N^{p^{r}}$ is a well-defined quantity, as well as the sum of the remaining $p^{r}-1$ coordinates of $[\vec{a}]$.
\end{proof}

Given a space of type $\CC$-$G^{*}_{p^{r}}$-$\SWF$, we can therefore define the \emph{set of projected equivariant $k$-invariants}
\[\mbfk^{\pi}(X):=\pi(\mbfk(X))\subset\NN^{2},\]
as well as the corresponding \emph{upper} and \emph{lower} equivariant $k$-invariants
\begin{align*}
    &\vec{\ol{k}}\,^{\pi}(X)=(\ol{k}_{0}(X),\ol{k}_{1}(X)):=\vee\mbfk^{\pi}(X)\in\NN^{2}, & &\vec{\ul{k}}\,^{\pi}(X)=(\ul{k}_{0}(X),\ul{k}_{1}(X)):=\wedge\mbfk^{\pi}(X)\in\NN^{2}.
\end{align*}

\begin{proposition}
\label{prop:k_invariants_odd_prime_powers}
Let $p^{r}$ be an odd prime power and let $X$, $X'$ be spaces of type $\CC$-$G^{*}_{p^{r}}$-$\SWF$ at levels $\mbfs,\mbfs'\in R(\ZZ_{p^{r}})_{\geq 0}^{\sym}$, respectively. Suppose that $f:X\to X'$ is a $G^{*}_{p^{r}}$-equivariant map whose $\Pin(2)$-fixed point set is a $G^{*}_{p^{r}}$-homotopy equivalence. Then for all $(k_{0},k_{1})\in\mbfk^{\pi}(X)$:
\begin{enumerate}
    \item For each $(k'_{0},k'_{1})\in\mbfk^{\pi}(X')$ the following implications hold:
    \begin{align*}
        &k'_{0}+(s'_{0}-s_{0})\le k_{0}+\left\{
		  \begin{array}{ll}
			 1 & \mbox{if } X \text{ is }K_{G^{*}_{p^{r}}}\text{-split and }s_{0}<s'_{0} \\
			 0 & \mbox{otherwise}
		  \end{array}
		  \right. \\
        \implies &k'_{1}+\sum_{j=1}^{p^{r}-1}(s_{j}'-s_{j})\geq k_{1},\text{ and } \\
        &k'_{1}+\sum_{j=1}^{p^{r}-1}(s_{j}'-s_{j})\le k_{1} \\
        \implies &k'_{0}+(s'_{0}-s_{0})\geq k_{0}+\left\{
		\begin{array}{ll}
			1 & \mbox{if } X \text{ is }K_{G^{*}_{p^{r}}}\text{-split and }s_{0}<s'_{0}, \\
			0 & \mbox{otherwise.}
		\end{array}
		\right.
    \end{align*}
    \item There exists $(k'_{0},k'_{1})\in\mbfk^{\pi}(X')$ such that:
    \begin{align*}
	   &k'_{0}+(s'_{0}-s_{0})\geq k_{0}+\left\{
		  \begin{array}{ll}
			 1 & \mbox{if } X \text{ is }K_{G^{*}_{p^{r}}}\text{-split and }s_{0}<s'_{0}, \\
			 0 & \mbox{otherwise,}
		  \end{array}
		  \right. \\
	   &k'_{1}+\sum_{j=1}^{p^{r}-1}(s_{j}'-s_{j})\geq k_{1}.
    \end{align*}
\end{enumerate}
In particular:
\begin{align*}
	&\ol{k}_{0}(X')+(s'_{0}-s_{0})\geq\ul{k}_{0}(X)+\left\{
		\begin{array}{ll}
			1 & \mbox{if } X \text{ is }K_{G^{*}_{p^{r}}}\text{-split and }s_{0}<s'_{0}, \\
			0 & \mbox{otherwise,}
		\end{array}
		\right. \\
	&\ol{k}_{1}(X')+\sum_{j=1}^{p^{r}-1}(s_{j}'-s_{j})\geq\ul{k}_{1}(X).
\end{align*}
\end{proposition}

\begin{proof}
Follows from Propositions \ref{prop:k_invariants_pin(2)_fixed_point_homotopy_equivalence} and \ref{prop:k_invariants_kg_split} via the projection $\pi:\N^{p^{r}}\to\NN^{2}$.
\end{proof}

We conclude this section with the following example:

\begin{example}
\label{ex:odd_prime}
Let $p$ be an odd prime, and let $\ast\in\{\ev,\odd\}$.
\begin{enumerate}
    \item Let $a\in\frac{1}{2}\ZZ$ be such that $0\le a\le m-\frac{1}{2}$ and $a\equiv 0\pmod{1}$ if $\ast=\ev$ and $a\equiv \frac{1}{2}\pmod{1}$ if $\ast=\odd$. From the calculations in Example \ref{ex:cosets_cyclic}, we have that
    \[\mbfk(\wt{\Sigma}Z_{a,p})=\{[\vec{e}_{0}],[\vec{e}_{2a}],[\vec{e}_{p-2a}]\}.\]
    From this, it follows that
    \[\mbfk^{\pi}(\wt{\Sigma}Z_{a,p})=
        {\left\{\begin{array}{ll}
            \{(1,0)\} & \mbox{if } a=0,\frac{p}{2}, \\
            \{(1,0),(0,1)\} & \mbox{otherwise,}
	\end{array}
        \right.}\]
    and so:
    \begin{align*}
    &\vec{\ol{k}}\,^{\pi}(\wt{\Sigma}Z_{a,p})=
        {\left\{\begin{array}{ll}
            (1,0) & \mbox{if } a=0,\frac{p}{2}, \\
            (1,1) & \mbox{otherwise,}
	\end{array}
        \right.} & 
    &\vec{\ul{k}}\,^{\pi}(\wt{\Sigma}Z_{a,p})=
        {\left\{\begin{array}{ll}
            (1,0) & \mbox{if } a=0,\frac{p}{2}, \\
            (0,0) & \mbox{otherwise.}
	\end{array}
        \right.}
    \end{align*}
    \item More generally, suppose $a_{1},\dots,a_{n}\in\frac{1}{2}\ZZ$ is a sequence of half-integers such that for all $k=1,\dots,n$: $0\le a_{k}\le m-\frac{1}{2}$, and $a_{k}\equiv 0\pmod{1}$ if $\ast=\ev$ and $a_{k}\equiv \frac{1}{2}\pmod{1}$ if $\ast=\odd$. Then from Example \ref{ex:multiple_cosets_cyclic}, we have that
    \begin{align*}
        &\mbfk(\wt{\Sigma}Z_{a_{1},\dots,a_{n};p}) \\
        &={\left\{\begin{array}{ll}
            \{[\vec{e}_{0}],[\vec{e}_{2a}],[\vec{e}_{p-2a}]\} & \mbox{if }\exists a\in\frac{1}{2}\ZZ\text{ such that }a_{k}\equiv\pm a\pmod{p}\;\forall k=1,\dots,n, \\
            \{[\vec{e}_{0}]\} & \mbox{otherwise.}
	\end{array}
        \right.}
    \end{align*}
    Hence
    \begin{align*}
        &\mbfk^{\pi}(\wt{\Sigma}Z_{a_{1},\dots,a_{n};p}) \\
        &={\left\{\begin{array}{ll}
            \{(1,0),(0,1)\} & \mbox{if }\exists a\neq 0,\frac{p}{2}\text{ such that }a_{k}\equiv\pm a\pmod{p}\;\forall k=1,\dots,n,\;(\ddagger) \\
            \{(1,0)\} & \mbox{otherwise,}
	\end{array}
        \right.}
    \end{align*}
    and therefore:
    \begin{align*}
        &\vec{\ol{k}}\,^{\pi}(\wt{\Sigma}Z_{a_{1},\dots,a_{n};p})={\left\{\begin{array}{ll}
            (1,1) & \mbox{if }(\ddagger)\text{ holds,}\\
            (1,0) & \mbox{otherwise,}
	\end{array}
        \right.} &
        &\vec{\ul{k}}\,^{\pi}(\wt{\Sigma}Z_{a_{1},\dots,a_{n};p})={\left\{\begin{array}{ll}
            (0,0) & \mbox{if }(\ddagger)\text{ holds,}\\
            (1,0) & \mbox{otherwise.}
	\end{array}
        \right.}
    \end{align*}
    \item Let $a\in\frac{1}{2}\ZZ$ be as in (1). From the calculations in Example \ref{ex:torus_cyclic}, we deduce that
    \begin{align*}
        &\mbfk(\wt{\Sigma}T_{a,p})=
            {\left\{\begin{array}{ll}
                \{[\vec{e}_{0}]\} & \mbox{if } a=0,\frac{p}{2}, \\
                \{[\vec{e}_{k}]\;|\;k=0,\dots,p-1\} & \mbox{otherwise.}
	    \end{array}
            \right.}
    \end{align*}
    Hence
    \[\mbfk^{\pi}(\wt{\Sigma}T_{a,p})=\mbfk^{\pi}(\wt{\Sigma}Z_{a,p})=
        {\left\{\begin{array}{ll}
            \{(1,0)\} & \mbox{if } a=0,\frac{p}{2}, \\
            \{(1,0),(0,1)\} & \mbox{otherwise,}
	\end{array}
        \right.}\]
    from which it follows that
    \begin{align*}
        &\vec{\ol{k}}\,^{\pi}(\wt{\Sigma}T_{a,p})=\vec{\ol{k}}\,^{\pi}(\wt{\Sigma}Z_{a,p}), & &\vec{\ul{k}}\,^{\pi}(\wt{\Sigma}T_{a,p})=\vec{\ul{k}}\,^{\pi}(\wt{\Sigma}Z_{a,p}).
    \end{align*}
    \item Let $a_{1},\dots,a_{n}\in\frac{1}{2}\ZZ$ be as in (2). From Example \ref{ex:multiple_cosets_duals}, we have that
    \[\mbfk(\wt{\Sigma}X_{a_{1},\dots,a_{n};p})=\{[\vec{v}]\;|\;\vec{v}\succeq n_{0}\vec{e}_{0},\;|\vec{v}|=n\},\]
    where
    \[n_{0}:=\#\{1\le k\le n\;|\;2a_{k}\equiv 0\pmod{p}\}.\]
    Hence
    \[\mbfk^{\pi}(\wt{\Sigma}X_{a_{1},\dots,a_{n};p})=\{(k,n-k)\;|\;k=n_{0},\dots,n\},\]
    and therefore
    \begin{align*}
        &\vec{\ol{k}}\,^{\pi}(\wt{\Sigma}X_{a_{1},\dots,a_{n};p})=(n,n-n_{0}), & &\vec{\ol{k}}\,^{\pi}(\wt{\Sigma}X_{a_{1},\dots,a_{n};p})=(n_{0},0).
    \end{align*}
\end{enumerate}
\end{example}

\bigskip

%% file: stable_k_invariants.tex
\section{Stable Equivariant \texorpdfstring{$k$}{k}-Invariants}
\label{sec:stable_equivariant_k_invariants}

In this paper, the spaces we will ultimately be working with are not quite spaces of type $\CC$-$G^{*}_{m}$-$\SWF$, but rather formal desuspensions of spaces of this type by rational $G^{*}_{m}$--representations, i.e., by elements of $R(G^{*}_{m})\otimes\QQ$. Therefore we need to modify the definitions of the equivariant $k$-invariants given in the previous section in order to accomodate this framework.

In Section \ref{subsec:stable_equivalence} we formalize this process of rational desuspension in the $G^{*}_{m}$-equivariant setting, similar to the definitions given in \cite{Man16}, \cite{Man14}. In order to facilitate this we will use the notions of \emph{$G^{*}_{m}$-spectrum classes} and \emph{$\CC$-$G^{*}_{m}$-spectrum classes}. Finally in Section \ref{subsec:stable_k_invariants} we define the lattice $\Q^{m}_{*}$ and associate to a $\CC$-$G^{*}_{m}$-spectrum class a collection of \emph{stable} equivariant $k$-invariants which take values in $\Q^{m}_{*}$.

\bigskip
\subsection{\texorpdfstring{$G^{*}_{m}$}{G*m}- and \texorpdfstring{$\CC$-$G^{*}_{m}$}{C-G*m}- Spectrum Classes}
\label{subsec:stable_equivalence}

We first define the notion of a \emph{$G^{*}_{m}$-spectrum class}:

\begin{definition}
\label{def:stable_G_*_m_spectrum_class}
Let $\wt{\CCC}_{G^{*}_{m}}$ denote the set of triples of the form $(X,\mbfa,\mbfb)$, where:
\begin{enumerate}
    \item $X$ is a space of type $G^{*}_{m}$-$\SWF$,
    \item $\mbfa\in RO(\ZZ_{m})$,
    \item $\mbfb\in R(\ZZ_{2m})^{*}\otimes\QQ$.
\end{enumerate}
We say that $(X,\mbfa,\mbfb)$ is \emph{stably $G^{*}_{m}$-equivalent} to $(X',\mbfa',\mbfb')$ if $\mbfb-\mbfb'\in R(\ZZ_{2m})^{*}\subset R(\ZZ_{2m})^{*}\otimes\QQ$, and there exist
\begin{align*}
    &\Amid\in RO(\ZZ_{m})_{\geq 0}, & &\Bmid\in R(\ZZ_{2m})^{*}_{\geq 0}\otimes\QQ, & &r\geq 0,
\end{align*}
such that:
\begin{enumerate}
    \item $\Amid-\mbfa$, $\Amid-\mbfa'\in RO(\ZZ_{m})_{\geq 0}$.
    \item $\Bmid-\mbfb$, $\Bmid-\mbfb'\in R(\ZZ_{2m})^{*}_{\geq 0}$.
    \item There exists a $G^{*}_{m}$-homotopy equivalence
    \begin{equation*}
    \label{eq:stable_G_*_m_equivalence}
        \Sigma^{r\RR}\Sigma^{(\Amid-\mbfa)\wt{\RR}}\Sigma^{(\Bmid-\mbfb)\HH}X\simeq_{G^{*}_{m}}\Sigma^{r\RR}\Sigma^{(\Amid-\mbfa')\wt{\RR}}\Sigma^{(\Bmid-\mbfb')\HH}X'.
    \end{equation*}
\end{enumerate}
We denote by $\CCC_{G^{*}_{m}}:=\wt{\CCC}_{G^{*}_{m}}/\sim$ the set of stable equivalence classes of triples $(X,\mbfa,\mbfb)$, whose elements we refer to as \emph{$G^{*}_{m}$-spectrum classes}.
\end{definition}

Next, we introduce the concept of \emph{$\CC$-$G^{*}_{m}$-spectrum classes}, which are ultimately the types of spaces we will be applying our stable equivariant $k$-invariants to.

\begin{definition}
\label{def:stable_C_G_*_m_spectrum_class}
Denote by $\wt{\CCC}_{G^{*}_{m},\CC}$ the set of triples of the form $(X,\mbfa,\mbfb)$, where:
\begin{enumerate}
    \item $X$ is a space of type $\CC$-$G^{*}_{m}$-$\SWF$,
    \item $\mbfa\in R(\ZZ_{m})^{\sym}$,
    \item $\mbfb\in R(\ZZ_{2m})^{*}\otimes\QQ$.
\end{enumerate}
We say that $(X,\mbfa,\mbfb)$ is \emph{stably $\CC$-$G^{*}_{m}$-equivalent} to $(X',\mbfa',\mbfb')$ if $\mbfb-\mbfb'\in R(\ZZ_{2m})^{*}$, and there exist
\begin{align*}
    &\Amid\in R(\ZZ_{m})^{\sym}_{\geq 0}, & &\Bmid\in R(\ZZ_{2m})^{*}_{\geq 0}\otimes\QQ, & &r\geq 0,
\end{align*}
such that:
\begin{enumerate}
    \item $\Amid-\mbfa$, $\Amid-\mbfa'\in R(\ZZ_{m})^{\sym}_{\geq 0}$.
    \item $\Bmid-\mbfb$, $\Bmid-\mbfb'\in R(\ZZ_{2m})^{*}_{\geq 0}$.
    \item There exists a $G^{*}_{m}$-homotopy equivalence
    \begin{equation}
    \label{eq:stable_C_G_*_m_equivalence}
        \Sigma^{r\RR}\Sigma^{(\Amid-\mbfa)\wt{\CC}}\Sigma^{(\Bmid-\mbfb)\HH}X\xrightarrow{\simeq_{G^{*}_{m}}}\Sigma^{r\RR}\Sigma^{(\Amid-\mbfa')\wt{\CC}}\Sigma^{(\Bmid-\mbfb')\HH}X'.
    \end{equation}
\end{enumerate}
We denote by $\CCC_{G^{*}_{m},\CC}:=\wt{\CCC}_{G^{*}_{m},\CC}/\sim$ the set of stable $\CC$-$G^{*}_{m}$-equivalence classes of triples $(X,\mbfa,\mbfb)$, and refer to elements of $\CCC_{G^{*}_{m},\CC}$ as \emph{$\CC$-$G^{*}_{m}$-spectrum classes}.
\end{definition}

We can also make $\CCC_{G^{*}_{m},\CC}$ into a category as follows: for any $[(X,\mbfa,\mbfb)]$, $[(X',\mbfa',\mbfb')]\in\CCC_{G^{*}_{m},\CC}$, a map
\[f:[(X,\mbfa,b)]\to[(X',\mbfa',b')]\]
is a $G^{*}_{m}$-equivariant map as in Equation \ref{eq:stable_C_G_*_m_equivalence} with respect to any representatives $(X,\mbfa,\mbfb)$, $(X',\mbfa',\mbfb')\in\wt{\CCC}_{G^{*}_{m},\CC}$ of $[(X,\mbfa,\mbfb)]$, $[(X',\mbfa',\mbfb')]$, respectively, which is not necessarily a homotopy equivalence. (Note that such a map only exists if $\mbfb-\mbfb'\in R(\ZZ_{2m})^{*}$.) There is also a smash product on $\wt{\CCC}_{G^{*}_{m},\CC}$ given by
\[(X,\mbfa,\mbfb)\wedge(X',\mbfa',\mbfb'):=(X\wedge X',\mbfa+\mbfa',\mbfb+\mbfb'),\]
which descends to a monoidal product 
\[\wedge:\CCC_{G^{*}_{m}}\times\CCC_{G^{*}_{m}}\to\CCC_{G^{*}_{m}}\]
with identity $[(S^{0},0,0)]$, giving $\CCC_{G^{*}_{m},\CC}$ the structure of a monoidal category.

\smallskip

We also have the notion of \emph{$\CC$-$G^{*}_{m}$-local equivalence}:

\begin{definition}
\label{def:stable_G*m_local_equivalence}
We say that $[(X,\mbfa,\mbfb)],[(X',\mbfa',\mbfb')]\in\CCC_{G^{*}_{m},\CC}$ are \emph{$G^{*}_{m}$-locally equivalent} and write
\[[(X,\mbfa,\mbfb)]\equiv_{l}[(X',\mbfa',\mbfb')]\]
if there exist maps
\[[(X,\mbfa,\mbfb)]\stackrel[g]{f}{\rightleftarrows}[(X',\mbfa',\mbfb')]\]
which induce (stable) $G^{*}_{m}$-homotopy equivalences on the $S^{1}$-fixed point sets. We denote by $\LLL\EEE_{G^{*}_{m},\CC}=\CCC_{G^{*}_{m},\CC}/\sim$ the set of $\CC$-$G^{*}_{m}$-local equivalence classes of $\CC$-$G^{*}_{m}$-spectrum classes.
\end{definition}

We will sometimes write $[(X,\mbfa,\mbfb)]_{\loc}$ to denote the local equivalence class of $[(X,\mbfa,\mbfb)]\in\CCC_{G^{*}_{m},\CC}$, and write $[(X,\mbfa,\mbfb)]_{\loc}=[(X',\mbfa',\mbfb')]_{\loc}$ if $[(X,\mbfa,\mbfb)]$ and $[(X',\mbfa',\mbfb')]$ are locally equivalent.

As in the case of $\CCC_{G^{*}_{m},\CC}$, the smash product endows $\LLL\EEE_{G^{*}_{m},\CC}$ with the structure of a monoidal catgory, with identity given by $[(S^{0},0,0)]_{\loc}$.

We also comment on the existence of inverses in $\LLL\EEE_{G^{*}_{m},\CC}$: let $(X,\mbfa,\mbfb)\in\wt{\CCC}_{G^{*}_{m},\CC}$, and suppose that $X$ is equivariantly $\mbfs\wt{\CC}\oplus\mbft\HH$-dual to some space $Y$ of type $\CC$-$G^{*}_{m}$-$\SWF$. Then the smash product
\[(X,\mbfa,\mbfb)\wedge(Y,-\mbfa+\mbfs,-\mbfb+\mbft)\in\wt{\CCC}_{G^{*}_{m},\CC}\]
is locally $\CC$-$G^{*}_{m}$-equivalent to $(S^{0},0,0)$, and so $[(Y,-\mbfa+\mbfs,-\mbfb+\mbft)]_{\loc}$ provides an inverse to $[(X,\mbfa,\mbfb)]_{\loc}$ in $\LLL\EEE_{G^{*}_{m},\CC}$.

\begin{remark}
One can show that for every space $X$ of type $\CC$-$G^{*}_{m}$-$\SWF$, there exists some real $G^{*}_{m}$-representation $V$ such that $X$ is equivariantly $V$-dual to some space $Y$ of type $\CC$-$G^{*}_{m}$-$\SWF$. If we were working within a complete universe, this would imply that every spectrum class has an inverse up to local equivalence  --- however since we are only allowing de-suspensions by representations of the form $\mbfa\wt{\CC}$ and $\mbfb\HH$, this does not necessarily hold in the category $\CCC_{G^{*}_{m},\CC}$.
\end{remark}

\bigskip
\subsubsection{Fixed-Point Sets}
\label{subsubsec:stable_fixed_point_sets}

In this section we discuss $H$-fixed-point sets of $\CC$-$G^{*}_{m}$-spectrum classes, for various closed subgroups $H\subset G^{*}_{m}$.

\begin{definition}
\label{def:stable_s1_pin(2)_fixed_points}
Let $(X,\mbfa,\mbfb)\in\wt{\CCC}_{G^{*}_{m},\CC}$. We define:
\begin{align*}
    &(X,\mbfa,\mbfb)^{S^{1}}:=(X^{S^{1}},\mbfa,0)\in\wt{\CCC}_{G^{*}_{m},\CC}, & &(X,\mbfa,\mbfb)^{\Pin(2)}:=(X^{\Pin(2)},0,0)\in\wt{\CCC}_{G^{*}_{m},\CC}.
\end{align*}
Furthermore, suppose that $X$ is a space of type $\CC$-$G^{*}_{m}$-$\SWF$ at level $\mbfs\in R(\ZZ_{m})_{\geq 0}^{\sym}$. We define the \emph{level of $(X,\mbfa,\mbfb)$} to be $\mbfs-\mbfa\in R(\ZZ_{m})^{\sym}$.
\end{definition}

For closed subgroups $H\neq S^{1},\Pin(2)$, we cannot expect the $H$-fixed-point set of a space of type $\CC$-$G^{*}_{m}$-$\SWF$ to still be a $G^{*}_{m}$-equivariant space. However, its non-equivariant homotopy type is still well-defined. In order to encode this on the level of $\CC$-$G^{*}_{m}$-spectrum classes, we introduce the \emph{rational Spanier-Whitehead category}:

\begin{definition}
Let $\wt{\CCC}$ denote the set of pairs of the form $(X,q)$, where $X$ is a pointed finite CW complex, and $q\in\QQ$. We say that $(X,q)$ is stably equivalent to $(X',q')$ if $q-q'\in\ZZ$, and there exists $Q\in\QQ_{\geq 0}$ with $Q-q,Q-q'\in\NN$ such that there exists a (based) homotopy equivalence
\[\Sigma^{Q-q}X\to\Sigma^{Q-q'}X'.\]
We denote by $\CCC=\wt{\CCC}/\sim$ the set of such pairs under this equivalence relation, the \emph{rational Spanier-Whitehead category}.
\end{definition}

\begin{remark}
The terminology above follows from the observation that $\CCC$ is naturally identified with an infinite number of disjoint copies of the usual (non-equivariant) Spanier-Whitehead category, with the set of copies in one-to-one correspondence with $\QQ/\ZZ$.
\end{remark}

Fix a closed subgroup $H\subset G^{*}_{m}$. For any complex $G^{*}_{m}$-vector space $V$, we denote by $V^{H}$ the $H$-fixed-point subspace. Suppose $(X,\mbfa,\mbfb)\in\wt{\CCC}_{G^{*}_{m},\CC}$ with 
\begin{align*}
    &\mbfa=\twopartdef{\sum_{j=0}^{m-1}a_{j}\zeta^{j}\in R(\ZZ_{m})^{\sym}}{*=\ev,}{\sum_{j=0}^{m-1}a_{j}\xi^{2j}\in R(\ZZ_{2m})^{\sym,\ev}}{*=\odd,} \\
    &\mbfb=\twopartdef{\sum_{k=0}^{m-1}b_{k}\zeta^{k}\in R(\ZZ_{m})\otimes\QQ}{*=\ev,}{\sum_{k=0}^{m-1}b_{k+1/2}\xi^{2k+1}\in R(\ZZ_{2m})^{\odd}\otimes\QQ}{*=\odd.}
\end{align*}
Define
\begin{align*}
    &a_{j}^{H}:=a_{k}\cdot\dim_{\RR}(\wt{\CC}_{j}^{H})\in\ZZ, &
    &b_{k}^{H}:=b_{k}\dim_{\RR}(\HH_{k}^{H})\in\QQ, \\
    &\mbfa^{H}:=\sum_{j=0}^{m-1}a_{j}^{H}\in\ZZ, & &\mbfb^{H}:=\twopartdef{\sum_{k=0}^{m-1}b_{k}^{H}\in\QQ}{*=\ev,}{\sum_{k=0}^{m-1}b_{k+1/2}^{H}\in\QQ}{*=\odd.}
\end{align*}
We then define the \emph{H-fixed-point set of $(X,a,b)$} to be:
\[(X,\mbfa,\mbfb)^{H}:=(X^{H},\mbfa^{H}+\mbfb^{H})\in\wt{\CCC},\]
where $X^{H}$ is the $H$-fixed-point set of $X$ in the usual sense, treated as an ordinary finite CW-complex. One can show that this descends to a monoidal functor
\[(-)^{H}:\CCC_{G^{*}_{m},\CC}\to\CCC\]
from the category of $\CC$-$G^{*}_{m}$ spectrum classes to the rational Spanier-Whitehead category.

Now recall from Definition \ref{def:H_spherical} that a space $X$ of type $\CC$-$G^{*}_{m}$-$\SWF$ is $H$-spherical at some level $d\in\NN$ if $X^{H}\simeq S^{d}$.

\begin{definition}
\label{def:stable_H_spherical}
Let $[(X,\mbfa,\mbfb)]\in\CCC_{G^{*}_{m},\CC}$ be a $\CC$-$G^{*}_{m}$ spectrum class, and let $H\subset G^{*}_{m}$ be a closed subgroup. We say that $[(X,\mbfa,\mbfb)]$ is
\begin{enumerate}
    \item \emph{$H$-spherical at level $d\in\QQ$} if
    \[[(X,\mbfa,\mbfb)^{H}]=[(S^{d+q},q)]\in\CCC\]
    for some $q\in\QQ$ such that $d+q\in\NN$. 
    \item \emph{locally $H$-spherical at level $d\in\QQ$} if there exists $[(X',\mbfa',\mbfb')]\in\CCC_{G^{*}_{m},\CC}$ such that $[(X',\mbfa',\mbfb')]$ is $H$-spherical at level $d$ and 
    \[[(X,\mbfa,\mbfb)]_{\loc}=[(X',\mbfa',\mbfb')]_{\loc}.\]
\end{enumerate}
\end{definition}

Now consider the special case where $m=2^{r}$, $*=\odd$, and $H=\<j\mu^{2^{r-1}}\>\cong\ZZ_{2}\subset G^{\odd}_{2^{r}}$. Recall that for any $G^{\odd}_{2^{r}}$-space $X$, its set of $\<j\mu^{2^{r-1}}\>$-fixed points has a residual $\<j\>\cong\ZZ_{4}\subset G^{\odd}_{2}$ action, and that $X^{\<j\mu^{2^{r-1}}\>}$ is a space of type $\CC$-$\ZZ_{4}$-$\SWF$ in the sense of Definition \ref{def:C_Z_4_SWF}. The following definitions parallel (\cite{KMT}, Definition 3.13):

\begin{definition}
\label{def:stable_C_Z_4_spectrum_class}
Consider the following definitions:
\begin{enumerate}
	\item Let $\wt{\CCC}_{\ZZ_{4},\CC}$ denote the set of triples $(X,a,\mbfb)$ where $X$ of type $\CC$-$\ZZ_{4}$-$\SWF$, $a\in\ZZ$, and $\mbfb=b_{1/2}\xi+b_{3/2}\xi^{3}\in R(\ZZ_{4})^{\odd}\otimes\QQ$. We say that $(X,a,\mbfb)$ and $(X,a',\mbfb')$ are stably equivalent if $\mbfb-\mbfb'\in R(\ZZ_{4})^{\odd}$, and there exist
	\begin{align*}
    	&A\in\NN, & &\Bmid\in R(\ZZ_{4})^{\odd}\otimes\QQ, & &r\geq 0,
	\end{align*}
	such that $A-a,A-a'\geq 0$, $\Bmid-\mbfb,\Bmid-\mbfb'\in R(\ZZ_{4})_{\geq 0}^{\odd}$, and there exists a $\ZZ_{4}$-homotopy equivalence
	\begin{equation}
	\label{eq:stable_C_Z_4_equivalence}
		\Sigma^{r\RR}\Sigma^{(A-a)\wt{\CC}}\Sigma^{(\Bmid-\mbfb)\CC}X\xrightarrow{\simeq}\Sigma^{r\RR}\Sigma^{(A-a')\wt{\CC}}\Sigma^{(\Bmid-\mbfb')\CC}X'.
	\end{equation}
	we denote by $\CCC_{\ZZ_{4},\CC}=\wt{\CCC}_{\ZZ_{4},\CC}/\sim$ the set of stable equivalence classes of triples $(X,a,\mbfb)$, and refer to elements of $\CCC_{\ZZ_{4},\CC}$ as \emph{$\CC$-$\ZZ_{4}$-spectrum classes}. Morphisms in $\CCC_{\ZZ_{4},\CC}$ are given by $\ZZ_{4}$-equivariant maps as in (\ref{eq:stable_C_Z_4_equivalence}) which are not necessarily $\ZZ_{4}$-homotopy equivalences.
	\item We say that $[(X,a,\mbfb)],[(X',a',\mbfb')]\in\CCC_{\ZZ_{4},\CC}$ are \emph{$\ZZ_{4}$-locally equivalent} and write
	\begin{align*}
		&[(X,a,\mbfb)]\equiv_{l}[(X',a',\mbfb')], & &\text{ or } & &[(X,a,\mbfb)]_{\loc}=[(X',a',\mbfb')]_{\loc},
	\end{align*}
	if there exist morphisms
	\[[(X,a,\mbfb)]\stackrel[g]{f}{\rightleftarrows}[(X',a',\mbfb')]\]
	which induce $\ZZ_{4}$-homotopy equivalences on the $\ZZ_{2}$-fixed point sets. We write $\LLL\EEE_{\ZZ_{4},\CC}=\CCC_{\ZZ_{4},\CC}/\sim$ for the set of $\ZZ_{4}$-local equivalence classes of $\CC$-$\ZZ_{4}$-spectrum classes.
\end{enumerate}
\end{definition}

We see that if $(X,\mbfa,\mbfb)\in\wt{\CCC}_{G^{\odd}_{2^{r}},\CC}$ is such that
\begin{align*}
    &\mbfa=\sum_{k=0}^{2^{r}-1}a_{k}\zeta^{k}\in R(\ZZ_{2^{r}})^{\sym}, & &\mbfb=\sum_{k=0}^{2^{r}-1}b_{k+1/2}\xi^{2k+1}\in R(\ZZ_{2^{r+1}})^{\odd}\otimes\QQ,
\end{align*}
then 
\[(X,\mbfa,\mbfb)^{\<j\mu^{2^{r-1}}\>}=\Bigg(X^{\<j\mu^{2^{r-1}}\>},\sum_{k=0}^{2^{r-1}-1}a_{2k+1},\Big(\sum_{k=0}^{2^{r-1}-1}b_{2k+\frac{1}{2}}\Big)\xi+\Big(\sum_{k=0}^{2^{r-1}-1}b_{2k+\frac{3}{2}}\Big)\xi^{3}\Bigg)\in\wt{\CCC}_{\ZZ_{4},\CC}.\]
In the particular case where $r=1$, we can write
\[(X,\mbfa,\mbfb)^{\<j\mu\>}=(X^{\<j\mu\>},a_{1},\ol{\mbfb})\in\wt{\CCC}_{\ZZ_{4},\CC},\]
where $\ol{\mbfb}:=b_{3/2}\xi+b_{1/2}\xi^{3}$ denotes the \emph{conjugate} of $\mbfb=b_{1/2}\xi+b_{3/2}\xi^{3}\in R(\ZZ_{2^{r+1}})^{\odd}\otimes\QQ$.

Let $X$ be a space of type $\CC$-$\ZZ_{4}$-$\SWF$. Borrowing the notation from (\cite{KMT}, Definition 3.11), we define the \emph{double of $X$}, denoted by $D(X)$, to be the space of type $\CC$-$\ZZ_{4}$-$\SWF$ given by $D(X):=X\wedge X^{\dagger}$. Here $X^{\dagger}$ denotes the same topological space as $X$, but with $\ZZ_{4}$-action given by the $\ZZ_{4}$-action on $X$ precomposed with the automorphism of $\ZZ_{4}=\<j\>$ given by $j\mapsto -j$.

We can extend this to $\CC$-$\ZZ_{4}$-spectrum classes as follows: if $(X,a,\mbfb)\in\wt{\CCC}_{\ZZ_{4},\CC}$, we let $(X,a,\mbfb)^{\dagger}:=(X^{\dagger},a,\ol{\mbfb})$, and define 
\[D(X,a,\mbfb):=(X,a,\mbfb)\wedge(X,a,\mbfb)^{\dagger}=(D(X),2a,\mbfb+\ol{\mbfb}).\]
Note that $\mbfb+\ol{\mbfb}$ is \emph{symmetric}, in the sense that if $\mbfb=b_{1/2}\xi+b_{3/2}\xi^{3}$, then 
\[\mbfb+\ol{\mbfb}=(b_{1/2}+b_{3/2})(\xi+\xi^{3})\in R(\ZZ_{4})^{\sym,\odd}\otimes\QQ.\]
This leads us to the following definition, of which $D(X,a,\mbfb)$ is the prototypical example:

\begin{definition}
\label{def:stable_symmetric_C_Z_4_spectrum_class}
We say that $(X,a,\mbfb)\in\wt{\CCC}_{\ZZ_{4},\CC}$ is \emph{symmetric} if $\mbfb=b(\xi+\xi^{3})$ for some $b\in\QQ$, and denote by $\wt{\CCC}_{\ZZ_{4},\CC,\sym}\subset\wt{\CCC}_{\ZZ_{4},\CC}$ the set of such symmetric triples. We then define $\CCC_{\ZZ_{4},\CC,\sym}=\wt{\CCC}_{\ZZ_{4},\CC,\sym}/\sim\,\subset\CCC_{\ZZ_{4},\CC}$ to be the subset of \emph{symmetric $\CC$-$\ZZ_{4}$-spectrum classes}. Furthermore, we write $\LLL\EEE_{\ZZ_{4},\CC,\sym}=\CCC_{\ZZ_{4},\CC,\sym}/\sim$ for the set of $\ZZ_{4}$-local equivalence classes of symmetric $\CC$-$\ZZ_{4}$-spectrum classes.
\end{definition}

Given $(X,a,b(\xi+\xi^{3}))\in\wt{\CCC}_{\ZZ_{4},\CC,\sym}$, we will often drop the term $(\xi+\xi^{3})$ and denote such a triple by $(X,a,b)$ with $a\in\ZZ$, $b\in\QQ$, and similarly denote its corresponding spectrum class by $[(X,a,b)]\in\CCC_{\ZZ_{4},\CC,\sym}$. It is not hard to see that the category $\CCC_{\ZZ_{4},\CC,\sym}$ is essentially equivalent to the category $\CCC_{\iota}$ defined in (\cite{KMT}, Definition 3.13), the only difference being that our equivalence relation allows suspensions by the trivial representation $\RR$, while the one in \cite{KMT} does not. A similar observation holds for $\LLL\EEE_{\ZZ_{4},\CC,\sym}$ and the set $\cal{L}\cal{E}_{\ZZ_{4}}$ from \cite{KMT}. However this difference is not crucial, as the invariants defined in their paper are still well-defined for symmetric $\CC$-$\ZZ_{4}$-spectrum classes. In particular, given $[(X,a,b)]\in\CCC_{\ZZ_{4},\CC,\sym}$ we can consider the invariant $k_{\KMT}([(X,a,b)])\in\QQ$, where $k_{\KMT}:\CCC_{\iota}\to\QQ$ corresponds to the invariant denoted by $k$ in (\cite{KMT}, Lemma 3.16), which descends to a well-defined invariant $k_{\KMT}:\cal{L}\cal{E}_{\ZZ_{4}}\to\QQ$. Note that if $\X$, $\X'$ are $\CC$-$G^{\odd}_{2^{r}}$-spectrum classes such that $\X\equiv_{\ell}\X'$, then $\X^{\<j\mu^{2^{r-1}}\>}\equiv_{\ell}(\X')^{\<j\mu^{2^{r-1}}\>}$ as $\CC$-$\ZZ_{4}$-spectrum classes and $D(\X^{\<j\mu^{2^{r-1}}\>})\equiv_{\ell}D((\X')^{\<j\mu^{2^{r-1}}\>})$ as symmetric $\CC$-$\ZZ_{4}$-spectrum classes.

Finally we define stable homotopy groups in the setting of spectrum classes:

\begin{definition}
Suppose $\X=(X,q)\in\CCC$ is an element of the rational Spanier--Whitehead category. For $r\in\QQ$ we define the \emph{$r$-th stable homotopy group of $\X$} to be
\[\pi^{\st}_{r}(\X):=\twopartdef{\pi^{\st}_{r+q}(X)}{r+q\in\NN,}{0}{r+q\not\in\NN,}\]
where $\pi^{\st}_{r+q}(X)$ denotes the usual stable homotopy group of $X$ in degree $r+q\in\NN$.
\end{definition}

We also have the following $\ZZ_{4}$-equivariant analogue:

\begin{definition}
\label{def:Z_4_stable_homotopy_groups}
Suppose $\X=[(X,a,b_{1/2}\xi+b_{3/2}\xi^{3})]$ is a $\CC$-$\ZZ_{4}$-spectrum class. Given $r,s\in\ZZ$, $t\in\QQ$, we define the \emph{$\ZZ_{4}$-equivariant stable homotopy group of $\X$ in degree $r+s\rho+t\nu$} to be
\begin{align*}
	&\pi^{\st,\ZZ_{4}}_{r+s\rho+t\nu}(\X):= \\
	&\qquad\twopartdef{\colim_{R,S,T}\big[S^{(R+r)\RR+(S+2a+s)\wt{\RR}+(T+b_{1/2}+b_{3/2}+t)\VV},\Sigma^{R+S\wt{\RR}+T\VV}X\big]_{\ZZ_{4}}}{b_{1/2}+b_{3/2}+t\in\ZZ,}{0}{b_{1/2}+b_{3/2}+t\not\in\ZZ,}
\end{align*}
where the colimit is taken over all $R,S,T\in\NN$ such that the above expression is well-defined.
\end{definition}

As in the unstable setting there is a canonical restriction map 
\[\res^{\ZZ_{4}}_{1}:\pi^{\st,\ZZ_{4}}_{r+s\rho+t\nu}(\X)\to\pi^{\st}_{r+s+2t}(\X).\]
which ``forgets" the $\ZZ_{4}$-equivariant structure.

\bigskip
\subsection{Stable \texorpdfstring{$k$}{k}-Invariants}
\label{subsec:stable_k_invariants}

In this section, we define stable analogues of the equivariant $k$-invariants from Section \ref{sec:equivariant_k_invariants}. In particular, we will define an additive lattice $\Q^{m}_{*}$ and an invariant $\mbfk^{\st}(\X)\subset\Q^{m}_{*}$ associated to a $\CC$-$G^{*}_{m}$-spectrum class $\X=[(X,a,b)]\in\CCC_{G^{*}_{m},\CC}$ as in Definition \ref{def:stable_C_G_*_m_spectrum_class}. Roughly, the invariant $\mbfk^{\st}(\X)$ is obtained by formally inverting the behavior of $\mbfk(X)\subset\N^{m}$ under suspension by the representations $\{\HH_{k}\}$, akin to (\cite{Man14}, Section 4.2). 

On the level of lattices, one can think of $\Q^{m}_{*}$ as a simultaneous ``localization" and ``rationalization" of the lattice $\N^{m}$, analogous to the process of obtaining $\QQ$ from $\NN$. Recall from Proposition \ref{prop:k_invariants_suspensions} that
	\[\mbfk(\Sigma^{\mbft\HH}X)=\mbfk(X)+[\DDD^{*}(\vec{\mbft})]\]
	for any space $X$ of type $\CC$-$G^{*}_{m}$-$\SWF$ and any representation $\mbft\in R(\ZZ_{2m})_{\geq 0}^{*}$, and where $\DDD^{*}$ denotes one of the two doubling maps
\begin{align*}
&\DDD^{\ev}:\NN^{m}\to\NN^{m}	& &\DDD^{\odd}:\NN^{m}_{1/2}\to\NN^{m}
\end{align*}
from Definition \ref{def:more_lattice_notation}, depending on the parity of $*\in\{\ev,\odd\}$. It therefore suffices to ``localize'' $\N^{m}$ only with respect to the portion of lattice spanned by the image of $\Pi\circ\DDD^{*}$, where $\Pi:\NN^{m}\to\N^{m}$ denotes the defining projection. This gives rise to the following trichotomy depending on the parities of $m$ and $*$:
\begin{enumerate}
	\item If $m$ is odd, then $\im\DDD^{*}=\NN^{m}$ for either $*=\ev$ or $*=\odd$.
	\item If $m$ is even and $*=\ev$, then
	\[\DDD^{\ev}(\NN^{m})=\spn_{\NN}\{\vec{e}_{2j}\;|\;j=0,\dots,\tfrac{m}{2}-1\}\subsetneq\NN^{m}.\]
	\item If $m$ is even and $*=\odd$, then
	\[\DDD^{\odd}(\NN^{m}_{1/2})=\spn_{\NN}\{\vec{e}_{2j+1}\;|\;j=0,\dots,\tfrac{m}{2}-1\}\subsetneq\NN^{m}.\]
\end{enumerate}
We will proceed to construct the lattice $\Q^{m}_{*}$ in two steps: First, we define the \emph{stablized additive lattice} $\N^{m}_{\st,*}$, which arises as a quotient of $\N^{m}$ under the minimal amount of relations necessary to ensure that $\N^{m}_{\st,*}$ is stable in a suitable sense under the module action of $\im\DDD^{*}\subset\NN^{m}$. The lattice $\Q^{m}_{*}$ is then obtained by enlarging $\N^{m}_{\st,*}$ so that it admits an action by $\im\DDD^{*}_{\QQ}$, where $\im\DDD^{*}_{\QQ}\subset\QQ^{m}$ denotes the $\QQ$-span of the image of $\im\DDD^{*}\subset\NN^{m}$ under the canonical inclusion $\NN^{m}\hookrightarrow\QQ^{m}$. The construction is such that there exists a canonical factorization of additive lattices
\[\N^{m}\twoheadrightarrow\N^{m}_{\st,*}\hookrightarrow\Q^{m}_{*}.\]
Although defining $\N^{m}_{\st,*}$ is not strictly necessary for us to define $\Q^{m}_{*}$, it will guide our intuition for the behavior of $\N^{m}$ under this localization/rationalization process.

We will first construct the lattice $\N^{m}_{\st,*}$. Recall the monoids $X_{m},W_{m}$ and semigroups $x_{0}X_{m},w_{0}W_{m}$ from Section \ref{subsec:monomials}.

\begin{definition}
Define $(\N^{m}_{\st,*},\preceq,+,|\cdot|)$ to be the quotient of $(\N^{m},\preceq,+,|\cdot|)$ under the equivalence relation 
\[[\vec{a}]\sim[\vec{b}]\iff\exists\,\vec{c}\in\im(\DDD_{m}^{*})\subset\NN^{m}\text{ such that }[\vec{a}+\vec{c}]=[\vec{b}+\vec{c}]\in\N^{m}.\]
\end{definition}

We leave it to the reader to verify the following proposition, whose proof is similar to that of Proposition \ref{prop:quotient_lattice_N}:

\begin{proposition}
\label{prop:quotient_lattice_N_stable}
$(\N^{m}_{\st,*},\preceq,+,|\cdot|)$ is a well-defined $\NN$-graded additive lattice.
\end{proposition}

\begin{example}
\label{ex:pre_stable_invariants_lattice_m_equals_2}
Let $m=2$. In the case where $*=\ev$, we have that $\im\DDD^{*}=\<\vec{e}_{0}\>\subset\NN^{2}$. From Proposition \ref{prop:monomials} we have the map $\mbfw^{\vec{a}}\mapsto\mbfw^{\vec{a}+\vec{e}_{0}}$ is injective on $w_{0}W_{2}$, and so it follows that $\N^{2}_{\st,\ev}\cong\N^{2}$.

On the other hand if $*=\odd$, then $\im\DDD^{*}=\<\vec{e}_{1}\>\subset\NN^{2}$. From Proposition \ref{prop:monomials}, we have that $w_{0}^{a_{0}}w_{1}^{a_{1}}=w_{0}^{a_{0}+a_{1}-1}w_{1}\in w_{0}W_{2}$ for all $(a_{0},a_{1})\in\NN^{2}$ such that $a_{0},a_{1}\geq 1$. Hence the grading map
\[|\cdot|:(\N^{2}_{\st,\odd},\preceq,+)\to(\NN,\le,+)\]
induces an isomorphism of additive posets. In particular, we ave that
\[\N^{2}\cong\N^{2}_{\st,\ev}\not\cong\N^{2}_{\st,\odd}\cong\NN.\]
\end{example}

\begin{example}
\label{ex:pre_stable_invariants_lattice_m_equals_p^r}
If $m=p^{r}$ is an odd prime power, then $\N^{p^{r}}_{\st,*}\cong\N^{p^{r}}$ for either $*=\ev$ or $\odd$. Indeed, let $\vec{a},\vec{b}\in\NN^{p^{r}}$, and suppose there exists $\vec{c}\in\im\DDD^{*}\cong\NN^{p^{r}}$ be such that $\mbfw^{\vec{a}+\vec{c}+\vec{e}_{0}}=\mbfw^{\vec{b}+\vec{c}+\vec{e}_{0}}\in w_{0}W_{p^{r}}$. Then $\vec{a}+\vec{c}+\vec{e}_{0}$, $\vec{b}+\vec{c}+\vec{e}_{0}$ satisfy the relations given in Proposition \ref{prop:monomials}. By subtracting the terms corresponding to $\vec{c}$ on both sides of each of the linear relations, we see that $\mbfw^{\vec{a}+\vec{e}_{0}}=\mbfw^{\vec{b}+\vec{e}_{0}}\in w_{0}W_{p^{r}}$.
\end{example}

For the following, let
\[\Pi_{\st}:(\N^{m},\preceq,+,|\cdot|)\twoheadrightarrow(\N^{m}_{\st,*},\preceq,+,|\cdot|)\]
denote quotient map.

\begin{definition}
\label{def:pre_stable_k_invariants}
Let $X$ be a space of type $\CC$-$G^{*}_{m}$-$\SWF$. We define 
\[I^{\st}(X):=\Pi_{\st}(I(X))\subset\N^{m}_{\st,*},\]
as well as the \emph{set of stable equivariant $k$-invariants of $X$} to be the subset
\[\mbfk^{\st}(X):=\min(I^{\st}(X))\subset\N^{m}_{\st,*},\]
as well as the \emph{upper} and \emph{lower} equivariant $k$-invariants
\begin{align*}
    &\vec{\ol{k}}\,^{\st}(X):=\vee\mbfk^{\st}(X)\in\wh{\N}^{m}_{\st,*}, & &\vec{\ul{k}}\,^{\st}(X):=\wedge\mbfk^{\st}(X)\in\wh{\N}^{m}_{\st,*},
\end{align*}
where $\wh{\N}^{m}_{\st,*}=\N^{m}_{\st,*}\cup\{+\infty\}$ denotes the completion of $\N^{m}_{\st,*}$ as in Definition \ref{def:bounded_completion}.
\end{definition}

\begin{remark}
Note that $\vec{\ol{k}}\,^{\st}(X)\neq\Pi_{\st}(\vec{\ol{k}}(X))$ and $\vec{\ul{k}}\,^{\st}(X)\neq\Pi_{\st}(\vec{\ul{k}}(X))$ in general, as $\Pi_{\st}$ is not necessarily a lattice homomorphism.
\end{remark}

\begin{example}
\label{ex:cosets_torus_stable_k_invariants}
Let $m=2$ and $*=\odd$. Recall from Examples \ref{ex:cosets_cyclic} and \ref{ex:torus_cyclic} that for $X=Z_{a,2}$ or $T_{a,2}$, $a=\frac{1}{2}$ or $\frac{3}{2}$, we have that
\[\mbfk(X)=\{[\vec{e}_{0}],[\vec{e}_{1}]\}\subset\N^{2},\]
and for $X'=Z_{\pm j}$, or $T_{\pm j,c}$, we have that
\[\mbfk(X')=\{[\vec{e}_{1}]\}\subset\N^{2}.\]
Noting that $[\vec{e}_{0}]\neq[\vec{e}_{1}]\in\N^{2}\cong\NN\times\{0,1\}$, we see that $\mbfk(X)\neq\mbfk(X')$ as subsets of $\N^{2}$. However since $[\vec{e}_{0}]=[\vec{e}_{1}]$ in $\N^{2}_{\st,\odd}$, from Example \ref{ex:pre_stable_invariants_lattice_m_equals_2} it follows that
\[\mbfk^{\st}(X)=\mbfk^{\st}(X')=\{[\vec{e}_{0}]\}\subset\N^{2}_{\st,\odd}\cong\NN.\]
\end{example}

Next we will extend our stable equivariant $k$-invariants to invariants of $\CC$-$G^{*}_{m}$-spectrum classes, which take values in the lattice $\Q^{m}_{*}$, to be defined shortly, in which $\N^{m}_{\st,*}$ naturally embeds.

Consider the vector space $(\QQ^{m},+)$ endowed with the product partial order and $(\QQ,\le +)$-grading given by $(q_{0},\dots,q_{m-1})\mapsto q_{0}+\cdots+q_{m-1}$. We define the sublattice $\QQ^{m}_{\NN,*}$ to be the union $\QQ^{m}_{\NN,*}:=\im(\DDD^{*})_{\QQ}\cup\NN^{m}\subset\QQ^{m}$. More explicitly:
\begin{enumerate}
	\item For $m$ odd, and $*\in\{\ev,\odd\}$, we have $\QQ^{m}_{\NN,*}=\QQ^{m}$.
	\item For $m$ even and $*=\ev$, we have
	\[\QQ^{m}_{\NN,*}=\underbrace{(\QQ\times\NN)\times\cdots\times(\QQ\times\NN)}_{m/2}\subset\QQ^{m}.\]
	\item For $m$ even and $*=\odd$, we have
	\[\QQ^{m}_{\NN,*}=\underbrace{(\NN\times\QQ)\times\cdots\times(\NN\times\QQ)}_{m/2}\subset\QQ^{m}.\]
\end{enumerate}
Note that $\QQ^{m}_{\NN,*}$ inherits a partial order, addition and $\QQ$-grading from $\QQ^{m}$, but does not inherit a $\QQ^{m}$-module structure. However, it does still have the structure of a module over itself, as well as an $\NN^{m}$-module structure induced by the natural inclusion $\NN^{m}\hookrightarrow\QQ^{m}_{\NN,*}$.

Finally, we define the lattice $\Q^{m}_{*}$:

\begin{definition}
Define $(\Q^{m}_{*},\preceq,+,|\cdot|)$ to be the quotient of $(\QQ^{m}_{\NN,*},\preceq,+,|\cdot|)$ under the equivalence relation
\[\vec{a}\sim\vec{b}\iff\exists\,\vec{c}\in\im(\D^{*})_{\QQ}\text{ such that }\vec{a}+\vec{c},\,\vec{b}+\vec{c}\in\NN^{m}\text{ and }[\vec{a}+\vec{c}]=[\vec{b}+\vec{c}]\in\N^{m}.\]
\end{definition}

We leave it to the reader to verify the following propositions:

\begin{proposition}
$(\Q^{m}_{*},\preceq,+,|\cdot|)$ is a well-defined $\QQ$-graded additive lattice.
\end{proposition}

\begin{proposition}
\label{prop:map_N^m_to_Q^m_*}
The canonical inclusion of $\NN$-graded $\NN^{m}$-modules
\[(\NN^{m},\preceq_{\NN^{m}},+,|\cdot|)\hookrightarrow(\QQ^{m}_{\NN,*},\preceq_{\QQ^{m}},+,|\cdot|)\]
induces maps of $\NN$-graded $\NN^{m}$-modules
\[f_{\QQ}:(\N^{m}_{\st,*},\preceq,+,|\cdot|)\to(\Q^{m}_{*},\preceq,+,|\cdot|)\qquad e_{\QQ}:(\N^{m}_{\st,*},\preceq,+,|\cdot|)\to(\Q^{m}_{*},\preceq,+,|\cdot|)\]
such that:
\begin{enumerate}
    \item The map $e_{\QQ}$ is an embedding.
    \item The map $f_{\QQ}$ factors as the composition
    \begin{equation}
    \label{eq:map_N^m_to_Q^m_*}
	(\N^{m}_{\st,*},\preceq,+,|\cdot|)\xrightarrow{\Pi_{\st}}\mathrel{\mkern-14mu}\rightarrow(\N^{m}_{\st,*},\preceq,+,|\cdot|)\xhookrightarrow{e_{\QQ}}(\Q^{m}_{*},\preceq,+,|\cdot|).
    \end{equation}
\end{enumerate}
\end{proposition}

\begin{example}
\label{ex:stable_m=2}
Let $m=2$. For $*=\ev$, every element in $\Q^{2}_{\ev}$ has a unique representative $(q,b)\in\QQ^{2}_{\NN,\ev}=\QQ\times\NN$ with $b\in\{0,1\}$. Hence we have an identification of sets $\Q^{2}_{\ev}=\QQ\times\{0,1\}$.

For $*=\odd$, every element in $\Q^{2}_{\odd}$ has a unique representative of the form $(0,q)\in\QQ^{2}_{\NN,\odd}=\NN\times\QQ$, giving us an identification of sets $\Q^{2}_{\odd}=\QQ$. Moreover, the grading map
\[|\cdot|:(\QQ^{2}_{\odd},\preceq,+)\to(\QQ,\le,+)\]
induces an isomorphism of additive posets. Hence in particular $\Q^{2}_{\ev}\not\equiv\Q^{2}_{\odd}$.
\end{example}

\begin{example}
\label{ex:stable_m=3}
Let $m=3$, $\ast\in\{\ev,\odd\}$. From Example \ref{ex:w_0W_3}, every element in $\Q^{3}_{*}$ has a unique representative of the form $(a,b,c)\in\QQ^{3}_{\NN,*}=\QQ^{3}$ with $0\le c<3$. Hence we have an identification of sets $\Q^{3}_{*}=\QQ^{2}\times(\QQ\cap[0,3))$.
\end{example}

\begin{example}
\label{ex:stable_m=4}
Let $m=4$. For $\ast=\ev$, by Example \ref{ex:w_0W_4} we see that every element in $\Q^{4}_{\ev}$ has a unique representative of the form $(a,b,c,0)\in\QQ^{4}_{\ev}=\QQ\times\NN\times\QQ\times\NN$ with $b\in\{0,1\}$ and $0\le c<1$. Hence we have an identification of sets
\[\Q^{4}_{\ev}=\QQ\times\{0,1\}\times(\QQ\cap[0,1))\times\{0\}.\]

For $*=\odd$, again by Example \ref{ex:w_0W_4} every element in $\Q^{4}_{\ev}$ has a unique representative $(a,b,c,d)\in\QQ^{4}_{\odd}=\NN\times\QQ\times\NN\times\QQ$ which is written in one of the following two forms:
\begin{align*}
(a,b,0,d), & &0\le d < 4, \\
(a,b,1,d), & &0\le b,d <1.
\end{align*}
We therefore have an identification of sets
\[\Q^{4}_{\odd}=\big(\QQ^{2}\times\{0\}\times(\QQ\cap[0,4))\big)\cup\big(\QQ\times(\QQ\cap[0,1))\times\{1\}\times(\QQ\cap[0,1))\big).\]
\end{example}

\begin{example}
\label{ex:stable_m=5}
Let $m=5$, $\ast\in\{\ev,\odd\}$. From Example \ref{ex:w_0W_5} every element in $\Q^{5}_{*}$ has a unique representative of the form $(a,b,c,d,e)\in\QQ^{5}_{\NN,*}=\QQ^{5}$, where $a,b,c\in\QQ$, and $(d,e)$ lies in one of the following five (mutually exclusive) subsets of $\QQ^{2}$:
\begin{align*}
	&A_{0}:=\big(\QQ\cap[0,1)\big)\times\big(\QQ\cap[0,5)\big), &
	&A_{1}:=\big(\QQ\cap[1,2)\big)\times\big(\QQ\cap[0,3)\big), \\
	&A_{2}:=\big(\QQ\cap[2,3)\big)\times\big(\QQ\cap[0,1)\big), &
	&A_{3}:=\big(\QQ\cap[3,4)\big)\times\big(\QQ\cap[0,4)\big), \\
	&A_{4}:=\big(\QQ\cap[4,5)\big)\times\big(\QQ\cap[0,2)\big). & &
\end{align*}
Hence we have an identification of sets $\Q^{5}_{*}=\cup_{i=0}^{4}(\QQ^{3}\times A_{i})$.
\end{example}

We now proceed to define the stable $k$-invariants for $\CC$-$G^{*}_{m}$-spectrum classes. Let
\begin{align*}
&\DDD^{\ev}:\QQ^{m}\to\QQ^{m}_{\NN,*}	& &\DDD^{\odd}:\QQ^{m}_{1/2}\to\QQ^{m}_{\NN,*}
\end{align*}
be the natural $\QQ$-linear extensions of the maps $\DDD^{\ev},\DDD^{\odd}$ from Definition \ref{def:more_lattice_notation}.

\begin{definition}
\label{def:stable_k_invariants}
Let $(X,\mbfa,\mbfb)\in\wt{\CCC}_{G^{*}_{m},\CC}$. We define $I^{\st}(X,\mbfa,\mbfb)$ to be the subset of $\Q^{m}_{*}$ given by
\begin{equation*}
    I^{\st}(X,\mbfa,\mbfb):=e_{\QQ}(I^{\st}(X))-[\DDD^{*}(\vec{\mbfb})]=f_{\QQ}(I(X))-[\DDD^{*}(\vec{\mbfb})]\subset\Q^{m}_{*},
\end{equation*}
where $e_{\QQ},f_{\QQ}$ are the maps from Proposition \ref{prop:map_N^m_to_Q^m_*}. We then define the \emph{set of stable equivariant $k$-invariants of $(X,\mbfa,\mbfb)$} as follows:
\[\mbfk^{\st}(X,\mbfa,\mbfb):=\min(I^{\st}(X,\mbfa,\mbfb))=e_{\QQ}(\mbfk^{\st}(X))-[\DDD^{*}(\vec{\mbfb})]=f_{\QQ}(\mbfk(X))-[\DDD^{*}(\vec{\mbfb})]\subset\Q_{*}^{m}.\]
The \emph{upper} and \emph{lower} equivariant $k$-invariants of $(X,\mbfa,\mbfb)$ are defined to be
\begin{align*}
    &\vec{\ol{k}}\,^{\st}(X,\mbfa,\mbfb):=\vee\mbfk^{\st}(X,\mbfa,\mbfb)=e_{\QQ}(\vec{\ol{k}}\,^{\st}(X))-[\DDD^{*}(\vec{\mbfb})]\in\wh{\Q}^{m}_{*} \\
    &\vec{\ul{k}}\,^{\st}(X,\mbfa,\mbfb):=\wedge\mbfk^{\st}(X,\mbfa,\mbfb)=e_{\QQ}(\vec{\ul{k}}\,^{\st}(X))-[\DDD^{*}(\vec{\mbfb})]\in\wh{\Q}^{m}_{*}
\end{align*}
where $\wh{\Q}^{m}_{*}=\Q^{m}_{*}\cup\{+\infty\}$ denotes the completion of $\Q^{m}_{*}$ as in Definition \ref{def:bounded_completion}.
\end{definition}

Next, we show that the above invariants are well-defined invariants of $\CC$-$G^{*}_{m}$-spectrum classes:

\begin{proposition}
\label{prop:stable_k_invariants_stable_homotopy_equivalence}
Let $(X,\mbfa,\mbfb),(X',\mbfa',\mbfb')\in\wt{\CCC}_{G^{*}_{m},\CC}$ be such that $[(X,\mbfa,\mbfb)]=[(X',\mbfa',\mbfb')]$. Then
\[I^{\st}(X,\mbfa,\mbfb)=I^{\st}(X',\mbfa',\mbfb')\subset\Q^{m}_{*}.\]
\end{proposition}

\begin{proof}
By definition, we must have that $\mbfb-\mbfb'\in R(\ZZ_{2m})^{*}$, and there must exist $\Amid\in R(\ZZ_{m})^{\sym}_{\geq 0}$, $\Bmid\in R(\ZZ_{2m})_{\geq 0}^{*}\otimes\QQ$, and $r\geq 0$ such that: $\Amid-\mbfa$, $\Amid-\mbfa'\in R(\ZZ_{m})^{\sym}_{\geq 0}$, $\Bmid-\mbfb$, $\Bmid-\mbfb'\in R(\ZZ_{2m})_{\geq 0}^{*}$, and
\[\Sigma^{r\RR}\Sigma^{(\Amid-\mbfa)\wt{\CC}}\Sigma^{(\Bmid-\mbfb)\HH}X\simeq_{G^{*}_{m}}\Sigma^{r\RR}\Sigma^{(\Amid-\mbfa')\wt{\CC}}\Sigma^{(\Bmid-\mbfb')\HH}X'.\]
By Proposition \ref{prop:k_invariants_stable_homotopy_equivalence} and Example \ref{ex:k_invariants_representation_spheres} we must have that $I(\Sigma^{(\Bmid-\mbfb)\HH}X)=I(\Sigma^{(\Bmid-\mbfb')\HH}X')$, and thus
\begin{align*}
	I^{\st}(X,\mbfa,\mbfb)&=e_{\QQ}(I^{\st}(X))-[\DDD^{*}(\vec{\mbfb})]=f_{\QQ}(I(X))+[\DDD^{*}(\vec{\Bmid}-\vec{\mbfb})]-[\DDD^{*}(\vec{\Bmid})] \\
	&=f_{\QQ}(I(\Sigma^{(\Bmid-\mbfb)\HH}X))-[\DDD^{*}(\vec{\Bmid})]=f_{\QQ}(I(\Sigma^{(\Bmid-\mbfb')\HH}X'))-[\DDD^{*}(\vec{\Bmid})] \\
	&=f_{\QQ}(I(X'))+[\DDD^{*}(\vec{\Bmid}-\vec{\mbfb}')]-[\DDD^{*}(\vec{\Bmid})]=e_{\QQ}(I^{\st}(X'))-[\DDD^{*}(\vec{\mbfb}')] \\
	&=I^{\st}(X',\mbfa',\mbfb').
\end{align*}
\end{proof}

Next, we will show that the stable equivariant $k$-invariants satisfy many of the same properties as their unstable counterparts.

\begin{proposition}
\label{prop:stable_k_invariants_s1_fixed_point_homotopy_equivalence}
Let $\X,\X'\in\wt{\CCC}_{G^{*}_{m},\CC}$ be $\CC$-$G^{*}_{m}$-spectrum classes at the same level $\vec{s}$, and suppose there exists a morphism
\[f:\X\to\X'\]
such that the induced map on $S^{1}$-fixed point sets is a $G^{*}_{m}$-homotopy equivalence. Then:
\begin{enumerate}
    \item For each $\vec{k}'\in\mbfk^{\st}(\X')$:
    \begin{enumerate}
        \item $\vec{k}\not\succ\vec{k}'$ for all $\vec{k}\in\mbfk^{\st}(\X)$.
        \item There exists some $\vec{k}\in\mbfk^{\st}(\X)$ such that $\vec{k}\preceq\vec{k}'$.
    \end{enumerate}
    \item $\vec{\ul{k}}\,^{\st}(\X)\preceq \vec{\ol{k}}\,^{\st}(\X')$.
\end{enumerate}
\end{proposition}

\begin{proof}
Follows from Proposition \ref{prop:k_invariants_s1_fixed_point_homotopy_equivalence}.
\end{proof}

\begin{corollary}
\label{cor:stable_k_invariants_local_equivalence}
Suppose $\X,\X'$ are $\CC$-$G^{*}_{m}$-spectrum classes such that $[\X]_{\loc}=[\X']_{\loc}\in\LLL\EEE_{G^{*}_{m},\CC}$. Then $I^{\st}(\X)=I^{\st}(\X')$, and hence their corresponding equivariant $k$-invariants are all equal. (Compare with Corollary \ref{cor:k_invariants_local_equivalence}.)
\end{corollary}

\begin{proposition}
\label{prop:stable_k_invariants_pin(2)_fixed_point_homotopy_equivalence}
Let $\X,\X'$ be $\CC$-$G^{*}_{m}$-spectrum classes at levels $\mbfs$ and $\mbfs'$, respectively, such that $\vec{\mbfs}\preceq\vec{\mbfs}\,'$. Suppose there exists a morphism
\[f:\X\to\X'\]
such that the induced map on $\Pin(2)$-fixed point sets is a $G^{*}_{p^{r}}$-homotopy equivalence. Then:
\begin{enumerate}
    \item For each $\vec{k}'\in\mbfk^{\st}(\X')$:
    \begin{enumerate}
        \item $\vec{k}\not\succ\vec{k}'+(\vec{\mbfs}\,'-\vec{\mbfs})$ for all $\vec{k}\in\mbfk^{\st}(\X)$.
        \item There exists some $\vec{k}\in\mbfk^{\st}(\X)$ such that $\vec{k}\preceq\vec{k}'+(\vec{\mbfs}\,'-\vec{\mbfs})$.
    \end{enumerate}
    \item $\vec{\ul{k}}\,^{\st}(\X)\preceq \vec{\ol{k}}\,^{\st}(\X')+(\vec{\mbfs}\,'-\vec{\mbfs})$.
\end{enumerate}
\end{proposition}

\begin{proof}
Follows from Proposition \ref{prop:k_invariants_pin(2)_fixed_point_homotopy_equivalence}.
\end{proof}

\begin{definition}
\label{def:stable_K_G_split}
    Let $\X$ be a $\CC$-$G^{*}_{m}$-spectrum class. We say that $\X$ is \emph{$K_{G^{*}_{m}}$-split} if there exists a representative $(X,\mbfa,\mbfb)\in\wt{\CCC}_{G^{*}_{m},\CC}$ with $[(X,\mbfa,\mbfb)]=\X$ such that $X$ is $K_{G^{*}_{m}}$-split.
\end{definition}

\begin{proposition}
\label{prop:stable_k_invariants_kg_split}
Let $\X,\X'\in\CCC_{G^{*}_{m},\CC}$ be $\CC$-$G^{*}_{m}$-spectrum classes at levels $\vec{\mbfs}$ and $\vec{\mbfs}\,'$, respectively, such that $\vec{\mbfs}\prec\vec{\mbfs}\,'$, $s_{0}<s'_{0}$, and $\X$ is $K_{G^{*}_{m}}$-split. Suppose there exists a morphism
\[f:\X\to\X'\]
such that the induced map on $\Pin(2)$-fixed point sets is a $G^{*}_{m}$-homotopy equivalence. Then:
\begin{enumerate}
    \item For each $\vec{k}'\in\mbfk^{\st}(\X')$:
    \begin{enumerate}
        \item $\vec{k}+[\vec{e}_{0}]\not\succ\vec{k}'+(\vec{\mbfs}\,'-\vec{\mbfs})$ for all $\vec{k}\in\mbfk^{\st}(\X)$.
        \item There exists some $\vec{k}\in\mbfk^{\st}(\X)$ such that $\vec{k}+[\vec{e}_{0}]\preceq\vec{k}'+(\vec{\mbfs}\,'-\vec{\mbfs})$.
    \end{enumerate}
    \item $\vec{\ul{k}}\,^{\st}(\X)+[\vec{e}_{0}]\preceq \vec{\ol{k}}\,^{\st}(\X')+(\vec{\mbfs}\,'-\vec{\mbfs})$.
\end{enumerate}
\end{proposition}

\begin{proof}
Follows from Proposition \ref{prop:k_invariants_kg_split}.
\end{proof}

\begin{proposition}
\label{prop:stable_k_invariants_duality}
Let $\X,\X'$ be $\CC$-$G^{*}_{m}$-spectrum classes at levels $\vec{\mbfs}$ and $\vec{\mbfs}'$, respectively, and suppose that $\X,\X'$ are $G^{*}_{m}$-equivariantly $[(S^{0},\mbfs,\mbft)]$-dual for some $\mbfs\in R(\ZZ_{m})^{\sym}$, $\mbft\in R(\ZZ_{m})^{*}\otimes\QQ$. Then
\[\vec{k}+\vec{k}'\succeq [\DDD^{*}(\vec{\mbft})]\text{ for all }\vec{k}\in\mbfk^{\st}(\X),\, \vec{k}'\in\mbfk^{\st}(\X').\]
In particular:
\[\vec{\ul{k}}\,^{\st}(\X)+\vec{\ul{k}}\,^{\st}(\X')\geq [\DDD^{*}(\vec{\mbft})].\]
\end{proposition}

\begin{proof}
Follows from Proposition \ref{prop:k_invariants_duality}.
\end{proof}

\begin{proposition}
\label{prop:stable_k_invariants_2_r_odd}
Let $r\geq 1$ be an integer, let $\mbfs\in R(\ZZ_{2^{r}})^{\sym}$, $\mbft\in R(\ZZ_{2^{r+1}})^{\odd}\otimes\QQ$, let $\X'$ be a $\CC$-$G^{\odd}_{2^{r}}$-spectrum class at level $\mbfs'\in R(\ZZ_{2^{r}})^{\sym}$, and suppose there exists a morphism
\[f:[(S^{0},-\mbfs,-\mbft)]\to\X'\]
such that the induced map on $\Pin(2)$-fixed point sets is a $G^{\odd}_{2^{r}}$-homotopy equivalence. Furthermore, suppose that:
\begin{enumerate}
    \item $\vec{\mbfs}\preceq\vec{\mbfs}\,'$.
    \item $s_{0}<s'_{0}$.
    \item $\sum_{k=0}^{2^{r-a}-1}s_{2^{a}k}<\sum_{k=0}^{2^{r-a}-1}s'_{2^{a}k}$ for all $a=0,\dots,r-1$.
    \item $\sum_{j=0}^{2^{a}-1}s_{(2k+1)2^{r-a-1}}<\sum_{j=0}^{2^{a}-1}s'_{(2k+1)2^{r-a-1}}$ for all $a=0,\dots,r-2$.
    \item There exists a $\CC$-$G^{\odd}_{2^{r}}$-spectrum class $\X''$ with $\X''\equiv_{\ell}\X'$ such that
    \[\res^{\ZZ_{4}}_{1}\Big(\pi^{\st,\ZZ_{4}}_{2(\sum_{k=0}^{2^{r-1}-1}s_{2k+1})\rho+(\sum_{k=0}^{2^{r}-1}t_{k+\frac{1}{2}})\nu}\big((\X'')^{\<j\mu^{2^{r-1}}\>}\big)\otimes\QQ\Big)=0,\]
    where $\pi^{\st,\ZZ_{4}}$ and $\res^{\ZZ_{4}}_{1}$ are as in Definition \ref{def:Z_4_stable_homotopy_groups} and the subsequent discussion.
\end{enumerate}
Then
\begin{equation}
\label{eq:stable_2_r_odd_1}
	\vec{k}+(\vec{\mbfs}\,'-\vec{\mbfs})\succeq\Big[\DDD^{\odd}(\vec{\mbft})+\vec{e}_{0}+\sum_{j=0}^{r-1}\vec{e}_{2^{j}}\Big]\in\Q^{2^{r}}_{\odd}\qquad\text{ for all }\vec{k}'\in\mbfk^{\st}(\X').
\end{equation}
In particular:
\begin{equation}
\label{eq:stable_2_r_odd_2}
	|\vec{\mbfs}\,'-\vec{\mbfs}|\geq|\vec{\mbft}|-|\vec{\ol{k}}\,^{\st}(\X')|+r+1\qquad\text{ for all }\vec{k}'\in\mbfk^{\st}(\X').
\end{equation}
Furthermore, (\ref{eq:stable_2_r_odd_1}) and (\ref{eq:stable_2_r_odd_2}) still hold if one replaces Condition (5) above with the following condition:
\begin{enumerate}
    \item[(5')] $\X'$ is locally $\<j\mu^{2^{r-1}}\>$-spherical at some level $d\in\QQ$, and
    \[\sum_{k=0}^{2^{r-1}-1}s_{2k+1}+\sum_{k=0}^{2^{r}-1}t_{k+\frac{1}{2}}\neq \tfrac{1}{2}d.\]
\end{enumerate}
\end{proposition}

\begin{proof}
Follows from Proposition \ref{prop:k_invariants_2_r_odd}.
\end{proof}

Recall that in the case where $r=1$, the lattice $\Q^{2}_{\odd}$ is isomorphic to $\QQ$. From this and Lemma \ref{lemma:2_r_exists_monomial} it follows that for any $\CC$-$G^{\odd}_{2}$-spectrum class $\X$, we have that
\[\mbfk^{\st}(\X)=\{\vec{k}\}\]
 consists of a single element $\vec{k}\in\Q^{2}_{\odd}$. We therefore define
\begin{equation}
	\wt{k}^{\st}(\X):=|\vec{k}|\in\QQ.
\end{equation}
The following proposition relates $\wt{k}^{\st}(\X)$ with the invariant $k_{\Pin(2)}(\X)\in\QQ$ from (\cite{Man14}, Lemma 4.3):

\begin{proposition}
\label{prop:stable_k_tilde_2_odd_inequality}
For any $\CC$-$G^{\odd}_{2}$ spectrum class $\X$ the following inequality holds:
\[k_{\Pin(2)}(\X)\le \wt{k}^{\st}(\X)\le k_{\Pin(2)}(\X)+1.\]
\end{proposition}

\begin{proof}
Follows from Lemma \ref{lemma:2_r_exists_monomial}.
\end{proof}

\begin{example}
\label{ex:k_tilde_multiple_cosets_duals}
Let $\wt{\Sigma}Z_{a_{1},\dots,a_{n};2}$ and $\wt{\Sigma}X_{a_{1},\dots,a_{n};2}$ be the $G^{\odd}_{2}$-spaces considered in Examples \ref{ex:multiple_cosets_cyclic} and \ref{ex:multiple_cosets_duals}, respectively, with each $a_{k}=\frac{1}{2}$ or $\frac{3}{2}$ for $k=1,\dots,n$. Then our previous calculations imply that
\begin{align*}
    &\wt{k}^{\st}(\wt{\Sigma}Z_{a_{1},\dots,a_{n};2})=1, &
    &\wt{k}^{\st}(\wt{\Sigma}X_{a_{1},\dots,a_{n};2})=n.
\end{align*}
Let $Z_{n}=\res^{G^{\odd}_{2}}_{\Pin(2)}(Z_{a_{1},\dots,a_{n};2})$, which consists of $n$ disjoint copies of $\Pin(2)$. Hence the connecting homomorphism in the long exact sequence from Fact \ref{fact:long_exact_sequence} is of the form
\[\wt{K}_{\Pin(2)}((\wt{\Sigma}Z_{n})^{S^{1}})\approx R(\Pin(2))\xrightarrow{\overbrace{\varepsilon\oplus\cdots\oplus\varepsilon}^{n}}\ZZ^{n}\approx\wt{K}_{\Pin(2)}(Z_{n}),\]
where $\varepsilon:R(\Pin(2))\to\ZZ$ denotes the augmentation homomorphism. Therefore
\begin{align*}
    &\III_{\Pin(2)}(\wt{\Sigma}Z_{n})=(w,z), & &k_{\Pin(2)}(\wt{\Sigma}Z_{n})=1.
\end{align*}
Similarly let $X_{n}=\res^{G^{\odd}_{2}}_{\Pin(2)}(X_{a_{1},\dots,a_{n};2})$. In Example \ref{ex:multiple_cosets_duals}, it was shown that
\begin{align*}
    &\III_{\Pin(2)}(\wt{\Sigma}X_{n})=(w^{n},z^{n}), & &k_{\Pin(2)}(\wt{\Sigma}X_{n})=n.
\end{align*}
Hence
\begin{align*}
    &\wt{k}^{\st}(\wt{\Sigma}Z_{a_{1},\dots,a_{n};2})=k_{\Pin(2)}(\wt{\Sigma}Z_{n})=1, &
    &\wt{k}^{\st}(\wt{\Sigma}X_{a_{1},\dots,a_{n};2})=k_{\Pin(2)}(\wt{\Sigma}X_{n})=n.
\end{align*}
Similarly for any $\varepsilon_{1},\dots,\varepsilon_{n}\in\{\pm 1\}$ we have that
\begin{align*}
    &\wt{k}^{\st}(\wt{\Sigma}Z_{\varepsilon_{1}j,\dots,\varepsilon_{n}j})=k_{\Pin(2)}(\wt{\Sigma}Z_{n})=1, &
    &\wt{k}^{\st}(\wt{\Sigma}X_{\varepsilon_{1}j,\dots,\varepsilon_{n}j})=k_{\Pin(2)}(\wt{\Sigma}X_{n})=n.
\end{align*}
\end{example}

We have the following corollary of Proposition \ref{prop:stable_k_invariants_2_r_odd} in the case $r=1$:

\begin{corollary}
\label{cor:stable_k_invariants_2_odd}
Let
\begin{align*}
    &\mbfs=s_{0}+s_{1}\zeta\in R(\ZZ_{2}), & &\mbft=t_{\frac{1}{2}}\xi+t_{\frac{3}{2}}\xi^{3}\in R(\ZZ_{4})^{\odd}\otimes\QQ,
\end{align*}
let $\X'$ be a $\CC$-$G^{\odd}_{2}$-spectrum class at level $\mbfs'=s'_{0}+s'_{1}\zeta\in R(\ZZ_{2})$, and suppose that there exists a morphism
\[f:[(S^{0},-\mbfs,-\mbft)]\to\X'\]
such that the induced map on $\Pin(2)$-fixed point sets is a $G^{\odd}_{2}$-homotopy equivalence. Furthermore, suppose that:
\begin{enumerate}
    \item $s_{0}<s'_{0}$ and $s_{1}<s'_{1}$.
    \item There exists a $\CC$-$G^{\odd}_{2^{r}}$-spectrum class $\X''$ with $\X''\equiv_{\ell}\X'$ such that
    \[\res^{\ZZ_{4}}_{1}\Big(\pi^{\st,\ZZ_{4}}_{2s_{1}\rho+(t_{1/2}+t_{3/2})\nu}\big((\X'')^{\<j\mu\>}\big)\otimes\QQ\Big)=0.\]
\end{enumerate}
Then:
\begin{equation}
\label{eq:stable_2_odd_eq}
(s'_{0}-s_{0})+(s'_{1}-s_{1})\geq t_{\frac{1}{2}}+t_{\frac{3}{2}}-\wt{k}^{\st}(\X')+2.
\end{equation}
Furthermore, one can replace Condition (2) above with the following condition:
\begin{enumerate}
	\item[(2')] $\X'$ is locally $\<j\mu\>$-spherical at some level $d\in\QQ$, and
	\[s_{1}+t_{\frac{1}{2}}+t_{\frac{3}{2}}\neq\tfrac{1}{2}d.\]
\end{enumerate}
\end{corollary}

Finally let $m=p^{r}$ be an odd prime power. Recall from Proposition \ref{prop:lattice_p^r_projection} that we have a commutative diagram of $\NN$-graded additive posets
\begin{center}
\begin{tikzcd}[column sep=small]
(\NN^{p^{r}},\preceq,+,|\cdot|) \arrow[rdd, swap, "\wt{\pi}"] \arrow[rr,"\Pi"] & & (\N^{p^{r}},\preceq,+,|\cdot|)\arrow[ldd, "\pi"] \\
& &  \\
& (\NN^{2},\preceq,+,|\cdot|), &
\end{tikzcd}
\end{center}
where $\wt{\pi}$ is the projection
\begin{align*}
    \wt{\pi}:(\NN^{p^{r}},\preceq,+,|\cdot|)&\to(\NN^{2},\preceq,+,|\cdot|) \\
    (a_{0},\dots,a_{p^{r}-1})&\mapsto (a_{0},a_{1}+\cdots+a_{p^{r}-1}).
\end{align*}
The following proposition follows from the observation in Example \ref{ex:pre_stable_invariants_lattice_m_equals_p^r}:

\begin{proposition}
\label{prop:stable_lattice_p^r_projection}
Let $\wt{\pi}$ denote the projection
\begin{align*}
    \wt{\pi}:(\QQ^{p^{r}},\preceq,+,|\cdot|)&\to(\QQ^{2},\preceq,+,|\cdot|) \\
    (a_{0},\dots,a_{p^{r}-1})&\mapsto (a_{0},a_{1}+\cdots+a_{p^{r}-1}).
\end{align*}
Then for $\ast\in\{\ev,\odd\}$ there exists a surjection of $\QQ$-graded additive posets
\[\pi:(\Q^{p^{r}}_{*},\preceq,+,|\cdot|)\to(\QQ^{2},\preceq,+,|\cdot|)\]
making the following diagram commute:
\begin{center}
\begin{tikzcd}[column sep=small]
(\QQ^{p^{r}},\preceq,+,|\cdot|) \arrow[rdd, swap, "\wt{\pi}"] \arrow[rr,"\Pi"] & & (\Q_{*}^{p^{r}},\preceq,+,|\cdot|)\arrow[ldd, "\pi"] \\
& &  \\
& (\QQ^{2},\preceq,+,|\cdot|). &
\end{tikzcd}
\end{center}
\end{proposition}

Given a $\CC$-$G^{*}_{m}$-spectrum class $\X=(X,\mbfa,\mbfb)$ we define the \emph{set of projected stable equivariant $k$-invariants}
\[\mbfk^{\st,\pi}(\X):=\pi(\mbfk^{\st}(\X))\subset\QQ^{2},\]
as well as the corresponding projected \emph{upper} and \emph{lower} equivariant $k$-invariants
\begin{align*}
    &\vec{\ol{k}}\,^{\st,\pi}(\X)=(\ol{k}\,^{\st}_{0}(\X),\ol{k}\,^{\st}_{1}(\X)):=\vee\mbfk^{\st,\pi}(\X)=\vee\mbfk^{\pi}(X)-\pi(\DDD^{*}(\vec{\mbfb}))\in\QQ^{2}, \\
    &\vec{\ul{k}}\,^{\st,\pi}(\X)=(\ul{k}_{0}^{\st}(\X),\ul{k}_{1}^{\st}(\X)):=\wedge\mbfk^{\st,\pi}(\X)=\wedge\mbfk^{\pi}(X)-\pi(\DDD^{*}(\vec{\mbfb}))\in\QQ^{2}.
\end{align*}

\begin{proposition}
\label{prop:stable_k_invariants_odd_prime_powers}
Let $p^{r}$ be an odd prime power and let $\X$,$\X'$ be $\CC$-$G^{*}_{p^{r}}$-spectrum classes at levels $\mbfs,\mbfs'\in R(\ZZ_{p^{r}})^{\sym}$, respectively. Suppose that $f:X\to X'$ is a morphism whose $\Pin(2)$-fixed point set is a $G^{*}_{p^{r}}$-homotopy equivalence. Then:
\begin{enumerate}
    \item For all $(k_{0},k_{1})\in\mbfk^{\st,\pi}(\X)$:
    \begin{enumerate}
        \item For each $(k'_{0},k'_{1})\in\mbfk^{\st,\pi}(\X')$ the following implications hold:
        \begin{align*}
            &k'_{0}+(s'_{0}-s_{0})\le k_{0}+\left\{
		  \begin{array}{ll}
			 1 & \mbox{if } \X \text{ is }K_{G^{*}_{p^{r}}}\text{-split and }s_{0}<s'_{0} \\
			 0 & \mbox{otherwise}
		  \end{array}
		  \right. \\
            \implies &k'_{1}+\sum_{j=1}^{p^{r}-1}(s_{j}'-s_{j})\geq k_{1},\text{ and } \\
            &k'_{1}+\sum_{j=1}^{p^{r}-1}(s_{j}'-s_{j})\le k_{1} \\
            \implies &k'_{0}+(s'_{0}-s_{0})\geq k_{0}+\left\{
		\begin{array}{ll}
			1 & \mbox{if } \X \text{ is }K_{G^{*}_{p^{r}}}\text{-split and }s_{0}<s'_{0}, \\
			0 & \mbox{otherwise.}
		\end{array}
		\right.
        \end{align*}
        \item There exists $(k'_{0},k'_{1})\in\mbfk^{\st,\pi}(\X')$ such that:
        \begin{align*}
	   &k'_{0}+(s'_{0}-s_{0})\geq k_{0}+\left\{
		  \begin{array}{ll}
			 1 & \mbox{if } \X \text{ is }K_{G^{*}_{p^{r}}}\text{-split and }s_{0}<s'_{0}, \\
			 0 & \mbox{otherwise,}
		  \end{array}
		  \right. \\
	   &k'_{1}+\sum_{j=1}^{p^{r}-1}(s_{j}'-s_{j})\geq k_{1}.
        \end{align*}
    \end{enumerate}
    \item The following inequalities hold:
    \begin{align*}
	&\ol{k}\,^{\st}_{0}(\X')+(s'_{0}-s_{0})\geq\ul{k}\,^{\st}_{0}          (\X)+\left\{
		\begin{array}{ll}
			1 & \mbox{if } \X \text{ is }K_{G^{*}_{p^{r}}}\text{-split and }s_{0}<s'_{0}, \\
			0 & \mbox{otherwise,}
		\end{array}
		\right. \\
	&\ol{k}\,^{\st}_{1}(\X')+\sum_{j=1}^{p^{r}-1}(s_{j}'-s_{j})\geq\ul{k}\,^{\st}_{1}(\X).
    \end{align*}
\end{enumerate}
\end{proposition}

\begin{proof}
Follows from Proposition \ref{prop:k_invariants_odd_prime_powers}.
\end{proof}

\bigskip

%% file: stable_homotopy_type.tex
\section{A \texorpdfstring{$G^{*}_{m}$}{G*m}-Equivariant Seiberg-Witten Floer Stable Homotopy Type}
\label{sec:stable_homotopy_type}

Let $m\geq 2$ be an integer. In this section we define a metric-independent $G^{*}_{m}$-equivariant Seiberg-Witten stable homotopy type $\SWF(Y,\frak{s},\wh{\sigma})$ associated to any $\ZZ_{m}$-equivariant spin rational homology sphere $(Y,\frak{s},\wh{\sigma})$, generalizing the $\Pin(2)$-equivariant spectrum $\SWF(Y,\frak{s})$ defined in \cite{Man16}.

In Section \ref{subsec:finite_dimensional_approximation}, we give a brief sketch of the construction of the Seiberg--Witten--Floer spectrum from \cite{Man03}, \cite{Man16}, pointing out the extra modifications to accomodate the extra invariance. In Section \ref{subsec:correction_term} we define the equivariant correction term required to ensure metric independence of the Floer spectrum, which we define in Section \ref{subsec:sw_floer_spectrum_class}.

\subsection{Finite-Dimensional Approximation}
\label{subsec:finite_dimensional_approximation}

Let $(Y,\frak{s},\wh{\sigma},g)$ be a $\ZZ_{m}$-equivariant Riemannian spin rational homology sphere of either even or odd type, as in Section \ref{sec:group_actions}. For this section, we let $G^{*}_{m}$ denote either $G^{ev}_{m}$ or $G^{odd}_{m}$ depending on whether $\wh{\sigma}$ is an even or odd spin lift. Let
\[\C(Y,\frak{s}):=i\Omega^{1}(Y)\oplus\Gamma(\SS)\]
denote the Seiberg-Witten configuration space associated to $(Y,\frak{s})$, and let
\[V:=i\Omega^{1}_{C}(Y)\oplus\Gamma(\SS)\subset \C(Y,\frak{s})\]
denote the global Coloumb slice, where 
\[\Omega^{1}_{C}(Y):=\{a\in\Omega^{1}(Y)\;|\;d^{*}a=0\}.\]
Recall our notation for generators $\<\gamma\>=\ZZ_{m}<G^{\ev}_{m}$ and $\<\mu\>=\ZZ_{2m}<G^{\odd}_{m}$. We define an action of $G^{*}_{m}$ on $\C(Y,\frak{s})$ via
\begin{align*}
    &e^{i\theta}\cdot(a,\phi)=(a,e^{i\theta}\phi), & &\gamma\cdot(a,\phi)=(\sigma^{*}(a),\wh{\sigma}^{*}(\phi))\text{ if }\ast=\ev, \\
    &j\cdot(a,\phi)=(-a,j\phi), & &\mu\cdot(a,\phi)=(\sigma^{*}(a),\wh{\sigma}^{*}(\phi))\text{ if }\ast=\odd,
\end{align*}
which descends to a $G^{*}_{m}$-action on $V$. In this setting, the Chern-Simons-Dirac functional $CSD:\C(Y,\frak{s})\to\RR$ given by
\[CSD(a,\phi)=\frac{1}{2}\Bigg(\int_{Y}\<\phi,\dirac\phi+\rho(a)\phi\>\;dvol_{g}-\int_{Y}a\wedge da\Bigg)\]
is $G^{*}_{m}$-equivariant as well as the restriction of its gradient to $V$ (with respect to a suitable metric). We have a $G^{*}_{m}$-equivariant decomposition $\nabla CSD=\ell+c:V\to V$, where $\ell=(\ast d,\dirac)$ is a self-adjoint elliptic operator. 

For $\nu<0$, $\lambda>0$, we denote by $V_{\nu}^{\lambda}$ the finite-dimensional subspace of $V$ spanned by the eigenvectors of $\ell$ with eigenvalues in the interval $(\nu,\lambda]$. By analyzing the inverse image of the restriction map $\res^{G^{*}_{m}}_{\Pin(2)}:R(G^{*}_{m})\to R(\Pin(2))$, we see that as a $G^{*}_{m}$-representation, $V_{\nu}^{\lambda}$ splits as a direct sum of copies of $\wt{\RR}$, $\wt{\VV}_{j}$ for $j=1,\dots, \lfloor\frac{m-1}{2}\rfloor$, $\wt{\RR}_{m/2}$ (if $m$ is even), and $\HH_{k}$ for $k=0,\dots,m-1$ if $*=\ev$, or $\HH_{k+1/2}$ for $k=0,\dots,m-1$ if $*=\odd$. We write this decomposition as
\begin{equation}
\label{eq:v_nu_lambda_decomposition}
    V_{\nu}^{\lambda}=\twopartdef{V_{\nu}^{\lambda}(\wt{\RR}_{0})\oplus\bigoplus_{j=1}^{\lfloor\frac{m-1}{2}\rfloor}V_{\nu}^{\lambda}(\wt{\VV}_{j})\oplus V_{\nu}^{\lambda}(\wt{\RR}_{m/2})\oplus\bigoplus_{k=0}^{m-1}V_{\nu}^{\lambda}(\HH_{k})}{*=\ev,}{V_{\nu}^{\lambda}(\wt{\RR}_{0})\oplus\bigoplus_{j=1}^{\lfloor\frac{m-1}{2}\rfloor}V_{\nu}^{\lambda}(\wt{\VV}_{j})\oplus V_{\nu}^{\lambda}(\wt{\RR}_{m/2})\oplus\bigoplus_{k=0}^{m-1}V_{\nu}^{\lambda}(\HH_{k+1/2})}{*=\odd,}
\end{equation}
where we use the convention that $V_{\nu}^{\lambda}(\wt{\RR}_{m/2})=\{0\}$ if $m$ is odd.

In fact, it will suffice to use $\nu=-\lambda$ for our purposes. Consider the gradient flow of the restriction $CSD|_{V_{-\lambda}^{\lambda}}$, which we view as a finite dimensional approximation to the Seiberg-Witten flow. Pick $R>>0$ independent of $\lambda$ such that all the finite energy Seiberg-Witten flow lines are inside the ball $B(R)$ in a suitable Sobolev completion of $V$. The trajectories of the approximate Seiberg-Witten flow on $V_{-\lambda}^{\lambda}$ that stay inside $B(R)$ form an isolated invariant set, and hence we can construct an associated $G^{*}_{m}$-equivariant Conley index
\begin{equation}
\label{eq:conley_index}
    I_{\lambda}=I_{G^{*}_{m}}(S^{\lambda}_{-\lambda},\phi^{\lambda}_{-\lambda}),
\end{equation}
which is an invariant of the tuple $(Y,\frak{s},\wh{\sigma},g,\lambda)$ up to $G^{*}_{m}$-homotopy equivalence. If we formally desuspend the Conley index by a copy of $(V_{-\lambda}^{0})^{+}$ (thought of as a $G^{*}_{m}$-representation sphere), we obtain a $G^{*}_{m}$-equivariant stable homotopy type independent of the eigenvalue cut-off $\lambda$, which we denote by $\text{SWF}(Y,\frak{s},\wh{\sigma},g)$ (see Section \ref{subsec:sw_floer_spectrum_class} for a more precise definition).

\bigskip
\subsection{Revisiting the Correction Term}
\label{subsec:correction_term}

In order to obtain a stable homotopy type independent of $g$, we need to revisit Manolescu's correction term $n(Y,\frak{s},g)$, and adapt it to the $G^{*}_{m}$-equivariant setting -- in particular we will define an \emph{equivariant correction term} $n(Y,\frak{s},\wh{\sigma},g)$. Before defining such a correction term, we will start off with a discussion of the $G$-Spin theorem for $\ZZ_{m}$-equivariant spin 4-manifolds.

\bigskip
\subsubsection{G-Spin Theorem}
\label{subsubsec:g_spin_theorem}

For this section, let $(W,\frak{t},\wh{\tau},g_{W})$ be a compact connected $\ZZ_{m}$-equivariant Riemannian spin 4-manifold.

First suppose $\wh{\tau}$ is of \emph{even} type, and let $\gamma\in\ZZ_{m}$ be a fixed generator. Then $\wh{\tau}$ induces a $\ZZ_{m}$ action on spinors via
\[\gamma\cdot\phi:=\wh{\tau}_{\#}\phi,\;\;\;\;\phi\in\Gamma(\SS^{\pm}_{W}).\]
By equivariance of $g_{W}$, this action descends to $\ZZ_{m}$-actions on the spaces of \emph{harmonic spinors} $\wt{\H}^{+}:=\ker(\Dirac^{+}_{W})$ and $\wt{\H}^{-}:=\ker(\Dirac^{-}_{W})\cong\coker(\Dirac^{+}_{W})$. We can therefore define the equivariant index
\[\text{Spin}(W,\frak{t},\wh{\tau},g_{W}):=[\wt{\H}^{+}]-[\wt{\H}^{-}]\]
as an element of the complex representation ring $R(\ZZ_{m})$. By taking traces at various elements of $\ZZ_{m}$, we obtain the corresponding set of characters
\[\text{Spin}^{\gamma^{k}}(W,\frak{t},\wh{\tau},g_{W}):=\text{tr}(\wh{\tau}^{k}_{\#}|\wt{\H}^{+})-\text{tr}(\wh{\tau}^{k}_{\#}|\wt{\H}^{-})\in\ZZ[\omega_{m}],\;\;\;\;k=0,\dots,m-1,\]
where $\omega_{m}=e^{2\pi i/m}\in\CC$.

Next suppose $\wh{\tau}$ is of \emph{odd} type, and let $\mu\in\ZZ_{2m}$ be a fixed generator. Then analogously to the even case we obtain an equivariant index
\[\text{Spin}(W,\frak{t},\wh{\tau},g_{W}):=[\wt{\H}^{+}]-[\wt{\H}^{-}]\in R(\ZZ_{2m})\]
as a complex $\ZZ_{2m}$-representation, as well as corresponding characters
\[\text{Spin}^{\mu^{k}}(W,\frak{t},\wh{\tau},g_{W}):=\text{tr}((\wh{\tau}^{k})_{\#}|\wt{\H}^{+})-\text{tr}((\wh{\tau}^{k})_{\#}|\wt{\H}^{-})\in\ZZ[\omega_{2m}],\;\;\;\;k=0,\dots,2m-1.\]

In order to put the even and odd cases on an equal footing, for the even case we will recast $\text{Spin}(W,\frak{t},\wh{\tau},g_{W})$ as a $\ZZ_{2m}=\<\mu\>$-representation by factoring through the canonical quotient map $\ZZ_{2m}\to\ZZ_{m}$ which sends $\mu\mapsto\gamma$. It follows that for the even case, the equation $\wh{\tau}^{m}=1$ implies that
\begin{align*}
    &\text{Spin}(W,\frak{t},\wh{\tau},g_{W})\in R(\ZZ_{2m})^{\ev},\\
    &\text{Spin}^{\mu^{k}}(W,\frak{t},\wh{\tau},g_{W})=\text{Spin}^{\mu^{m+k}}(W,\frak{t},\wh{\tau},g_{W})\qquad\text{for all }k=0,\dots,m-1.
\end{align*}
In the odd case, the equation $\wh{\tau}^{m}=-1$ implies that
\begin{align*}
    &\text{Spin}(W,\frak{t},\wh{\tau},g_{W})\in R(\ZZ_{2m})^{\odd},\\
    &\text{Spin}^{\mu^{k}}(W,\frak{t},\wh{\tau},g_{W})=-\text{Spin}^{\mu^{m+k}}(W,\frak{t},\wh{\tau},g_{W})\qquad\text{for all }k=0,\dots,m-1.
\end{align*}
Note that in both the even and odd cases, the action of $\wh{\tau}_{\#}$ commutes with the quaternionic structure on $\Gamma(\SS^{\pm})$. It follows that $\text{Spin}(W,\frak{t},\wh{\tau},g_{W})$ is a \emph{spin} representation, i.e., it comes from an element of the quaternionic representation ring $R\text{Sp}(\ZZ_{2m})$. One can show that this implies that
\begin{align*}
&\text{Spin}(W,\frak{t},\wh{\tau},g_{W})\in R(\ZZ_{2m})^{\sym,*}\subset R(\ZZ_{2m})^{*}, \\
&\text{Spin}^{\mu^{k}}(W,\frak{t},\wh{\tau},g_{W})=\text{Spin}^{\mu^{2m-k}}(W,\frak{t},\wh{\tau},g_{W})\qquad\text{for all }k=1,\dots,m-1.
\end{align*}

For the moment, suppose that $W$ is \emph{closed}, with $\wh{\tau}$ either even or odd. The equivariant Atiyah-Singer index theorem (\cite{AS3}) provides a formula for $\text{Spin}^{\mu^{k}}(W,\frak{t},\wh{\tau})=\text{Spin}^{\mu^{k}}(W,\frak{t},\wh{\tau},g_{W})$, $k\neq 0,m$, which is independent of the metric $g_{W}$ and depends only on the $\tau^{k}$-fixed-point set $W^{\tau^{k}}\subset W$. Let
\begin{align*}
    &p_{k,1},\dots,p_{k,m_{k}}, & &\Sigma_{k,1},\dots,\Sigma_{k,n_{k}},
\end{align*}
be enumerations of the dimension 0 and dimension 2 components of $W^{\tau^{k}}$, respectively. For each $p_{k,i}$, let $\alpha_{k,i},\beta_{k,i}\in\RR/2\pi\ZZ$ be two (non-zero) angles by which $\tau^{k}$ acts on an equivariant neighborhood $\nu(p_{k,i})\cong T_{p_{k,i}}W$ by $\begin{psmallmatrix}e^{i\alpha_{k,i}} & 0\\0 & e^{i\beta_{k,i}}\end{psmallmatrix}$ with respect to some local complex basis. The pair $(\alpha_{k,i},\beta_{k,i})$ is well-defined up to reordering and the equivalence relation $(\alpha_{k,i},\beta_{k,i})\equiv(-\alpha_{k,i},-\beta_{k,i})$. Similarly, let $\psi_{k,j}\in\RR/2\pi\ZZ$ be the angle by which $\tau^{k}$ acts fiberwise on $\nu(\Sigma_{k,j})$ by $e^{i\psi_{k,j}}$ with respect to some local (complex) basis, well-defined up to the equivalence relation $\psi_{k,j}\equiv-\psi_{k,j}$. Then:
\begin{equation}
    \Spin^{\mu^{k}}(W,\frak{t},\wh{\tau})=-\frac{1}{4}\Big(\sum_{i=1}^{m_{k}}\varepsilon_{k,i}R(\alpha_{k,i},\beta_{k,i})+\sum_{j=1}^{n_{k}}\varepsilon_{k,j}'[\Sigma_{k,j}]^{2}S(\psi_{k,j})\Big),\qquad k\neq 0,m,
\end{equation}
where:
\begin{enumerate}
    \item $R(\alpha_{k,i},\beta_{k,i}):=\csc(\tfrac{\alpha_{k,i}}{2})\csc(\tfrac{\beta_{k,i}}{2})$.
    \item $S(\psi_{k,j}):=\cot(\tfrac{\psi_{k,j}}{2})\csc(\tfrac{\psi_{k,j}}{2})$.
    \item $\varepsilon_{k,i},\varepsilon_{k,j}'\in\{\pm 1\}$ are signs which depend in a subtle manner on the particular component and the choices of angles $\alpha_{k,i},\beta_{k,i},\psi_{k,j}$.
\end{enumerate}
For $k=0,m$ we have that:
\begin{align}
\begin{split}
    &\Spin^{\mu^{0}}(W,\frak{t},\wh{\tau})=\ind_{\CC}(\Dirac_{W}^{+})=-\tfrac{1}{8}\sigma(W) \\ &\Spin^{\mu^{m}}(W,\frak{t},\wh{\tau})=\varepsilon\ind_{\CC}(\Dirac_{W}^{+})=-\varepsilon\tfrac{1}{8}\sigma(W),
\end{split}
\end{align}
where $\varepsilon=1$ or $-1$ if $\wh{\tau}$ is an even or odd spin lift, respectively.

\smallskip

Next suppose that $(W,\frak{t},\wh{\tau},g_{W})$ is a $\ZZ_{m}$-equivariant Riemannian spin 4-manifold with non-empty boundary $(Y,\frak{s},\wh{\sigma},g)$, and suppose that there exists a neighborhood of $\del W$ equivariantly isometric to a product $Y\times[0,1]$. Let $\Gamma(\SS)$ be the space of spinors on $Y$, and let $\dirac:\Gamma(\SS)\to\Gamma(\SS)$ denote the corresponding Dirac operator on $Y$. Here we use the convention from (\cite{Man03}, \cite{Man14}, \cite{Man16}) that if $\rho:TY\to\frak{su}(\SS)$ denotes the Clifford multiplication map, then $\rho(e_{1})\rho(e_{2})\rho(e_{3})=1$ for any orthonormal frame $\{e_{i}\}$ on $TY$. 

As in the 4-dimensional case, the map $\wh{\sigma}$ induces a $\ZZ_{2m}$-action $\Gamma(\SS)$, and the operator $\dirac$ is equivariant with respect to this action. The \emph{equivariant eta-invariant associated to $\dirac$ at $\mu^{k}\in\ZZ_{2m}$}, denoted by $\eta^{\mu^{k}}_{\dirac,\wh{\sigma},g}$, is the value at $s=0$ of the meromorphic continuation of the function
\[\eta^{\mu^{k}}_{\dirac,\wh{\sigma},g}(s)=\sum_{\lambda\neq 0}\frac{\sign(\lambda)\tr((\wh{\sigma}^{k})_{\#}|V_{\lambda})}{\lambda^{s}}\in\CC,\]
where $\tr((\wh{\sigma}^{k})_{\#}|V_{\lambda})$ denotes the trace of the induced action of $\wh{\sigma}^{k}$ on the $\lambda$-eigenspace $V_{\lambda}\subset\Gamma(\SS)$, and the sum is taken over all non-zero eigenvalues of $\dirac$. We also have a corresponding \emph{reduced equivariant eta invariant} $\ol{\eta}^{\mu^{k}}_{\dirac,\wh{\sigma},g}$, defined by
\[\ol{\eta}^{\mu^{k}}_{\dirac,\wh{\sigma},g}:=\frac{\eta^{\mu^{k}}_{\dirac,\wh{\sigma},g}-k^{\mu^{k}}_{\dirac,\wh{\sigma},g}}{2}\in\CC,\]
where 
\[k^{\mu^{k}}_{\dirac,\wh{\sigma},g}:=\tr((\wh{\sigma}^{k})_{\#}|\ker(\dirac))\in\CC\]
denotes the trace of the induced action of $\wh{\sigma}^{k}$ on the kernel of $\dirac$.

Now let $\mu^{k}\in\ZZ_{2m}$ with $k\neq 0,m$ and let $\{p_{k,i}\}_{i=1}^{m_{k}}$,$\{\Sigma_{k,j}\}_{j=1}^{n_{k}}$, $\{(\alpha_{k,i},\beta_{k,i})\}_{i=1}^{m_{k}}$, and $\{\psi_{k,j}\}_{j=1}^{n_{k}}$ be as above. There is also an equivariant analogue of the Atiyah-Patodi-Singer index theorem due to Donnelly (\cite{Don78}). When applied to the Dirac operator $\Dirac_{W}^{+}$, Donnelly's theorem states that (for $k\neq 0,m$):
\begin{align}
\label{eq:equivariant_APS}
\begin{split}
&\Spin^{\mu^{k}}(W,\frak{t},\wh{\tau},g_{W})= \\
&\qquad\qquad\qquad\ol{\eta}^{\mu^{k}}_{\dirac,\wh{\sigma},g}-\frac{1}{4}\Bigg(\sum_{i=1}^{m_{k}}\varepsilon_{k,i}R(\alpha_{k,i},\beta_{k,i})+\sum_{j=1}^{n_{k}}\int_{\Sigma_{k,j}}\varepsilon_{k,j}'S(\psi_{k,j})e(\nu(\Sigma_{k,j});g_{W})\Bigg),
\end{split}
\end{align}
where 
\[e(\nu(\Sigma_{k,j});g_{W})=\text{Pf}(F_{\nabla^{\nu}})\in\Omega^{2}(\Sigma_{k,j})\]
denotes the Chern-Weil form associated to the Euler class of the normal bundle of $\Sigma_{k,j}$, with $\nabla^{\nu}$ denoting the connection on $\nu(\Sigma_{k,j})$ induced by the Levi-Cevita connection $\nabla^{\text{LC}}$ corresponding to the metric $g_{W}$. For $k=0,m$ we have that:
\begin{align}
\begin{split}
    &\Spin^{\mu^{0}}(W,\frak{t},\wh{\tau},g_{W})=\ind_{\CC}(\Dirac_{W}^{+})=\ol{\eta}_{\dirac,g}-\int_{W}\tfrac{1}{24}p_{1}(W;g_{W}) \\ &\Spin^{\mu^{m}}(W,\frak{t},\wh{\tau},g_{W})=\varepsilon\ind_{\CC}(\Dirac_{W}^{+})=\varepsilon\ol{\eta}_{\dirac,g}-\varepsilon\int_{W}\tfrac{1}{24}p_{1}(W;g_{W}),
\end{split}
\end{align}
where:
\begin{enumerate}
    \item $p_{1}(W;g_{W})=-\frac{1}{8\pi^{2}}\text{tr}(F_{\nabla^{\text{LC}}}\wedge F_{\nabla^{\text{LC}}})$ denotes the Chern-Weil form associated to the first Pontryagin class of $W$.
    \item $\varepsilon\in\{\pm 1\}$ is as in the closed case.
    \item $\ol{\eta}_{\dirac,g}$ is the \emph{reduced (non-equivariant) eta invariant} given by
    \[\ol{\eta}_{\dirac,g}=\frac{\eta_{\dirac,g}-k_{\dirac,g}}{2}\in\RR,\]
    where:
    \[\eta_{\dirac,g}:=\sum_{\lambda\neq 0}\frac{\sign(\lambda)\dim_{\CC}(V_{\lambda})}{\lambda^{s}}\in\RR\]
    is the \emph{(un-reduced, non-equivariant) eta invariant} associated to $(Y,\frak{s},g)$, and
    \[k_{\dirac,g}:=\dim_{\CC}(\ker(\dirac))\in\NN\]
    denotes the dimension of the kernel of $\dirac$.
\end{enumerate}
By examination of the defining formulas, one can see that
\begin{align*}
    &\ol{\eta}^{\mu^{0}}_{\dirac,\wh{\sigma},g}=\ol{\eta}_{\dirac,g}, & &\ol{\eta}^{\mu^{m}}_{\dirac,\wh{\sigma},g}=\varepsilon\ol{\eta}_{\dirac,g}.
\end{align*}

\smallskip

We also have a variation formula for $\ol{\eta}^{\mu^{k}}_{\dirac,\wh{\sigma},g}$ under changes of metric. Let $g_{0},g_{1}$ be $\ZZ_{m}$-equivariant metrics on $Y$ and suppose $\{g_{s}\}$ is a one-parameter family of equivariant metrics interpolating between $g_{0}$ and $g_{1}$, which is constant near the ends. Fix an enumeration
\[K_{k,1},K_{k,2},\dots,K_{k,\ell_{k}}\]
of the (necessarily 1-dimensional) components of the fixed-point set $Y^{\sigma^{k}}\subset Y$ for $k\neq 0,m$. Furthermore, for each $j=1,\dots,\ell_{k}$ let $\psi_{k,j}\in \RR/2\pi\ZZ$ be the angle such that $\sigma^{k}$ acts fiberwise on $\nu(K_{k,j})$ via $e^{i\psi_{k,j}}$ with respect to some choice of local complex basis, again well-defined up to the equivalence relation $\psi_{k,j}\equiv-\psi_{k,j}$. Applying Equation \ref{eq:equivariant_APS} to $Y\times[0,1]$, equipped with the metric $\wh{g}_{s}$ such that $\wh{g}_{s}|_{Y\times\{s\}}=g_{s}$, we have that
\begin{equation}
\label{eq:variation_equivariant_eta_invariant}
    \ol{\eta}^{\mu^{k}}_{\dirac,\wh{\sigma},g_{1}}-\ol{\eta}^{\mu^{k}}_{\dirac,\wh{\sigma},g_{0}}=\SF^{\mu^{k}}(\{\dirac_{s}\})+\frac{1}{4}\sum_{j=1}^{\ell}\int_{K_{k,j}\times[0,1]}\varepsilon_{k,j}'S(\psi_{k,j})e(\nu(K_{k,j}\times[0,1]);\wh{g}_{s}),
\end{equation}
where $\SF^{\mu^{k}}(\{\dirac_{s}\})$ denotes the (trace) equivariant spectral flow at $\mu^{k}\in\ZZ_{2m}$ of the one-parameter family of operators $\{\dirac_{s}\}_{s\in[0,1]}$ (see \cite{LimWang}).

For $k=0,m$, the variation formulas are given by
\begin{align}
\begin{split}
    &\ol{\eta}^{\mu^{0}}_{\dirac,\wh{\sigma},g_{1}}-\ol{\eta}^{\mu^{0}}_{\dirac,\wh{\sigma},g_{0}}=\SF(\{\dirac_{s}\})+\int_{W}\tfrac{1}{24}p_{1}(W;\wh{g}_{s}), \\
    &\ol{\eta}^{\mu^{m}}_{\dirac,\wh{\sigma},g_{1}}-\ol{\eta}^{\mu^{m}}_{\dirac,\wh{\sigma},g_{0}}=\varepsilon\SF(\{\dirac_{s}\})+\varepsilon\int_{W}\tfrac{1}{24}p_{1}(W;\wh{g}_{s}),
\end{split}
\end{align}
where $\SF(\{\dirac_{s}\})$ denotes the ordinary (non-equivariant) spectral flow of $\{\dirac_{s}\}_{s\in[0,1]}$.

\bigskip
\subsubsection{Definition of the Equivariant Correction Term}
\label{subsubsec:definition_equivariant_correction_term}

We now begin our discussion of the correction term. Recall (\cite{Man16},\cite{Man14}) that for a spin rational homology sphere $(Y,\frak{s})$ equipped with a metric $g$, the correction term is defined to be
\[n(Y,\frak{s},g)=\ind_{\CC}(\Dirac_{W}^{+})+\tfrac{1}{8}\sigma(W)\in\tfrac{1}{8}\ZZ\subset\QQ.\]
where $(W,\frak{t})$ is any compact spin 4-manifold with boundary $(Y,\frak{s})$, equipped with a Riemannian metric $g_{W}$ isometric to $ds^{2}+g$ near the boundary, and $\Dirac_{W}^{+}:\Gamma(\bb{S}^{+}_{W})\to\Gamma(\bb{S}^{-}_{W})$ is the corresponding Dirac operator on $W$. 

Now let $\ol{\eta}_{\dirac,g}$ be the reduced, non-equivariant eta invariant of the Dirac operator as in Section \ref{subsubsec:g_spin_theorem}, and let $\eta_{sign,g}$ be the eta-invariant of the odd signature operator on $(Y,\frak{s},g)$. By the Atiyah-Patodi-Singer index theorem (\cite{APS1}), we have that
\begin{align*}
    &\ind_{\CC}(D^{+})=-\tfrac{1}{24}\int_{W}p_{1}(W;g_{W})+\ol{\eta}_{\dirac,g}  ,& &\tfrac{1}{8}\sigma(W)=\tfrac{1}{24}\int_{W}p_{1}(W;g_{W})-\tfrac{1}{8}\eta_{sign,g},
\end{align*}
and hence
\[n(Y,\frak{s},g)=\ol{\eta}_{\dirac,g}-\tfrac{1}{8}\eta_{sign,g}.\]
It follows that $n(Y,\frak{s},g)$ is well-defined and independent of the choice of spin filling $(W,\frak{t})$. Furthermore, if $g_{0}$ and $g_{1}$ are metrics on $Y$, then for any path of metrics $\{g_{s}\}_{s\in[0,1]}$ interpolating between $g_{0}$ and $g_{1}$, we have that
\[n(Y,\frak{s},g_{1})-n(Y,\frak{s},g_{0})=\text{SF}(\{\dirac_{s}\}),\]
where $\text{SF}(\{\dirac_{s}\})$ denotes the spectral flow of the family of Dirac operators associated to the path $\{g_{s}\}$.

In order to define the equivariant correction term, we will need to define the \emph{torsion} $t(L,g)$ of a framed link $L$ inside a Riemannian 3-manifold $(Y,g)$ (see \cite{Yos85} for more details). 

Let $\nabla^{\text{fr}}$ be the $SO(3)$-connection on the $SO(3)$-frame bundle $\text{Fr}(Y)\to Y$ induced by the Levi-Cevita connection on $(Y,g)$, and let 
$\theta=(\theta_{ij})\in\Omega^{1}(Y;\frak{so}(3))$ be the connection one-form associated to $\nabla^{\text{fr}}$. Given a framed, oriented link $L\subset Y$, we can trivialize $TY|_L$ by setting at each point $x\in L$:
\begin{itemize}
    \item $e_{1}(x)$ to be the unit tangent vector to $L$, with direction determined by the given orientation.
    \item $e_{2}(x)$ to be the unit vector pointing in the direction of the framing.
    \item $e_{3}(x)=e_{1}(x)\times e_{2}(x)$.
\end{itemize}
This trivialization then provides a section $\phi:L\to\text{Fr}(Y)$, and we define
\[t(L,g,\alpha):=-\int_{L}\phi^{*}\theta_{23},\]
which we call the \emph{torsion of $L$ with respect to $(g,\alpha)$}. Note that for any two framings $\alpha_{0}$, $\alpha_{1}$, we have that 
\[t(L,g,\alpha_{1})-t(L,g,\alpha_{0})\in 2\pi\ZZ.\]

Now if $Y$ is a rational homology sphere, then any link $L\subset Y$ is rationally null-homologous. We will use the following fact, which guarantees that any such $L$ has a canonical framing (see \cite{MT18},\cite{Raoux20}):

\begin{fact}
Let $K\subset Y$ be a rationally null-homologous knot in a 3-manifold $Y$. Then there exists a unique choice of longitude $\lambda_{\can}$ for $K$ (called the \emph{canonical longitude}) such that
\[[\del F]=c(d\lambda_{\can}+r\mu)\]
for any rational Seifert surface $F$ for $K$, where $c$ is such that $[\del F]=c\gamma$ for some primitive element $\gamma\in H_{1}(\del\nu(K))$ (called the \emph{complexity} of $K$), and $0\le r < d$. Furthermore if $K$ is null-homologous, then $\lambda_{\can}$ agrees with the usual Seifert framing. 
\end{fact}

In general, for a rationally null-homologous link $L$ we define the \emph{canonical framing} on $L$ to be the unique framing which restricts to $\lambda_{\can}$ on each component $K\subset L$. With this in mind, given a Riemannian rational homology sphere $(Y,g)$ and a link $L$, we will use the convention that $t(L,g)$ denotes the torsion computed with respect to the canonical framing on $L$.

We now state our definition of the equivariant correction term:

\begin{definition}
\label{def:equivariant_correction_term}
Let $(Y,\frak{s},\wh{\sigma},g)$ be a $\ZZ_{m}$-equivariant Riemannian spin rational homology sphere. For each $k=1,\dots,2m-1$, $k\neq m$, suppose that the fixed point set of $\sigma^{k}:Y\to Y$ is given by $Y^{\sigma^{k}}=K_{k,1}\cup\cdots\cup K_{k,\ell_{k}}$, and that $\sigma^{k}$ acts on $\nu(K_{k,j})$ via rotation by $\psi_{k,j}\in\RR/2\pi\ZZ$ with respect to some local complex basis. For each $k=0,\dots,2m-1$, define $n^{\mu^{k}}(Y,\frak{s},\wh{\sigma},g)\in\CC$ as follows:
\begin{equation}
    n^{\mu^{k}}(Y,\frak{s},\wh{\sigma},g):=\threepartdef{\ol{\eta}_{\dirac,g}-\tfrac{1}{8}\eta_{sign,g}}{k=0,}{\varepsilon(\ol{\eta}_{\dirac,g}-\tfrac{1}{8}\eta_{sign,g})}{k=m,}{\ol{\eta}^{\mu^{k}}_{\dirac,\wh{\sigma},g}+\frac{1}{8\pi}\sum_{j=1}^{n_{k}}\varepsilon_{k,j}S(\psi_{k,j})t(K_{k,j},g)}{k\neq 0,m,}
\end{equation}
where:
\begin{enumerate}
    \item $\varepsilon=1$ if $\wh{\sigma}$ is of even type and $\varepsilon=-1$ if $\wh{\sigma}$ is of odd type.
    \item $\varepsilon_{k,j}\in\{\pm 1\}$ are the signs as in the $G$-Spin theorem, depending in a subtle manner on the spin lift $\wh{\sigma}$ and the choice of angles $\{\psi_{k,j}\}$.
    \item $S(\psi_{k,j}):=\cot(\tfrac{\psi_{k,j}}{2})\csc(\tfrac{\psi_{k,j}}{2})$.
    \item $t(K_{k,j},g)=t(K_{k,j},g,\lambda_{\can})$ is the torsion of $K_{k,j}\subset Y^{\sigma^{k}}$ with respect to its canonical longitude $\lambda_{\can}$.
\end{enumerate}
Finally, we define the \emph{equivariant correction term} $n(Y,\frak{s},\wh{\sigma},g)$ to be representation
\[n(Y,\frak{s},\wh{\sigma},g):=\frac{1}{2m}\sum_{j=0}^{2m-1}\Big(\sum_{k=0}^{2m-1}n^{\mu^{k}}(Y,\frak{s},\wh{\sigma},g)\cdot\omega_{2m}^{-jk}\Big)\xi^{j}\in R(\ZZ_{2m})\otimes\CC.\]
\end{definition}

We will devote the rest of this section to prove the following theorem:

\begin{theorem}
\label{theorem:properties_equivariant_correction_term}
The equivariant correction term satisfies the following properties:
\begin{enumerate}
	\item $n(Y,\frak{s},\wh{\sigma},g)\in R(\ZZ_{2m})^{*}\otimes\QQ$.
    \item Under the augmentation map
    \[\alpha:R(\ZZ_{2m})^{*}\otimes\QQ\to\QQ\]
    $n(Y,\frak{s},\wh{\sigma},g)$ is sent to $n(Y,\frak{s},g)$.
    \item For any two equivariant metrics $g_{0},g_{1}$ and any path of equivariant metrics $\{g_{s}\}$ interpolating between $g_{0}$ and $g_{1}$,
    \[n(Y,\frak{s},\wh{\sigma},g_{1})-n(Y,\frak{s},\wh{\sigma},g_{0})= \text{SF}^{\ZZ_{2m}}(\{\dirac_{s}\})\in R(\ZZ_{2m})^{*},\]
    where $\text{SF}^{\ZZ_{2m}}(\{\dirac_{s}\})$ denotes the \emph{(representation-theoretic) equivariant spectral flow} of the family of operators $\{\dirac_{s}\}$, whose character at each $\mu^{k}\in\ZZ_{2m}$ is given by the quantity $\text{SF}^{\mu^{k}}(\{\dirac_{s}\})$ appeearing in Equation \ref{eq:variation_equivariant_eta_invariant}.
\end{enumerate}
\end{theorem}

Rather than prove these properties directly, it will be helpful to work instead with the set of characters of $n(Y,\frak{s},\wh{\sigma},g)$. In general, given an element $\mbfa(\xi)=\sum_{j=0}^{m-1}a_{k}\xi^{j}\in R(\ZZ_{2m})\otimes\CC$ we define its character to be the function
\begin{align*}
\chi_{\mbfa(\xi)}:\ZZ_{2m}&\to\CC \\	
\mu^{k}&\mapsto \mbfa(\omega_{2m}^{k})
\end{align*}
which assigns to an element $\mu^{k}\in\ZZ_{2m}$ the (generalized) trace of $\mbfa(\xi)$ at $\mu^{k}$. Note that any such representation $\mbfa(\xi)\in R(\ZZ_{2m})\otimes\CC$ can be recovered via the orthogonality relations from its set of characters:
\[\mbfa(\xi)=\frac{1}{2m}\sum_{j=0}^{2m-1}\Big(\sum_{k=0}^{2m-1}\chi_{\mbfa(\xi)}(\mu^{k})\cdot\omega_{2m}^{-jk}\Big)\xi^{j}\in R(\ZZ_{2m})\otimes\CC.\]

We see immediately from Definition \ref{def:equivariant_correction_term} that the characters of $n(Y,\frak{s},\wh{\sigma},g)$ are given by
\[\chi_{n(Y,\frak{s},\wh{\sigma},g)}(\mu^{k})=n^{\mu^{k}}(Y,\frak{s},\wh{\sigma},g)\]
for each $k=0,\dots,2m-1$. One can show that Theorem \ref{theorem:properties_equivariant_correction_term} is equivalent to the following proposition:

\begin{proposition}
\label{prop:properties_equivariant_correction_term_characters}
The characters $n^{\mu^{k}}(Y,\frak{s},\wh{\sigma},g)$ satisfy the following properties:
\begin{enumerate}
    \item $\sum_{k=0}^{2m-1}n^{\mu^{k}}(Y,\frak{s},\wh{\sigma},g)\cdot\omega_{2m}^{-jk}\in\QQ$ for all $j=0,\dots,2m-1$.
  	\item $n^{\mu^{k}}(Y,\frak{s},\wh{\sigma},g)=\varepsilon n^{\mu^{m+k}}(Y,\frak{s},\wh{\sigma},g)$ for all $k=0,\dots,m-1$, where $\varepsilon=1$ if $\wh{\sigma}$ is of even type and $\varepsilon=-1$ if $\wh{\sigma}$ is of odd type.
    \item $n^{\mu^{0}}(Y,\frak{s},\wh{\sigma},g)=n(Y,\frak{s},g)$.
    \item For any two $\sigma$-equivariant metrics $g_{0},g_{1}$ and any path of equivariant metrics $\{g_{s}\}$ interpolating between $g_{0}$ and $g_{1}$,
    \[n^{\mu^{k}}(Y,\frak{s},\wh{\sigma},g_{1})-n^{\mu^{k}}(Y,\frak{s},\wh{\sigma},g_{0})=\SF^{\mu^{k}}(\{\dirac_{s}\}).\]
\end{enumerate}
\end{proposition}

Indeed, (1) is equivalent to the assertion that $n(Y,\frak{s},\wh{\sigma},g)\in R(\ZZ_{2m})\otimes\QQ$, (2) is equivalent to the assertion that $n(Y,\frak{s},\wh{\sigma},g)\in R(\ZZ_{2m})^{*}\otimes\CC$, (3) is equivalent to Condition (2) in Theorem \ref{theorem:properties_equivariant_correction_term}, and (4) is equivalent to Condition (3) in Theorem \ref{theorem:properties_equivariant_correction_term}.

Of course, Condition (3) of Proposition \ref{prop:properties_equivariant_correction_term_characters} follows by definition of $n^{\mu^{0}}(Y,\frak{s},\wh{\sigma},g)$. We will next prove Condition (2) of Proposition \ref{prop:properties_equivariant_correction_term_characters}:

\begin{lemma}
	For all $k=0,\dots,m-1$, we have that $n^{\mu^{k}}(Y,\frak{s},\wh{\sigma},g)=\varepsilon n^{\mu^{m+k}}(Y,\frak{s},\wh{\sigma},g)$,  where $\varepsilon=1$ if $\wh{\sigma}$ is of even type and $\varepsilon=-1$ if $\wh{\sigma}$ is of odd type.
\end{lemma}

\begin{proof}
For $k=0$ this is by construction. For $1\le k\le m-1$, it suffices to show the following:
\begin{enumerate}
	\item $\ol{\eta}_{\dirac,\wh{\sigma},g}^{\mu^{m+k}}=\varepsilon\ol{\eta}_{\dirac,\wh{\sigma},g}^{\mu^{k}}$.
	\item $\varepsilon_{k,j}=\varepsilon\cdot\varepsilon_{m+k,j}(\wh{\sigma})$.
\end{enumerate}
But (1) follows from the fact that $\wh{\sigma}^{m}$ acts on spinors by $\varepsilon$, and (2) follows from Donnelly's Theorem applied to $Y\times[0,1]$.
\end{proof}

Next, we will look at Condition (4) of Proposition \ref{prop:properties_equivariant_correction_term_characters}. The following lemma will be useful:

\begin{lemma}
\label{lemma:variation_torsion}
Let $L\subset Y$ be a link, and let $g_{0}$, $g_{1}$ be metrics on $Y$ such that $L$ is totally geodesic with respect to $g_{0}$ and $g_{1}$. Then for any fixed framing $\alpha$ on $L$ and any smooth path of metrics $\{g_{s}\}_{s\in[0,1]}$ interpolating between $g_{0}$ and $g_{1}$ such that $L$ remains totally geodesic for all $s\in[0,1]$, we have that
\begin{equation}
\label{eq:variation_torsion}
	t(L,g_{1},\alpha)-t(L,g_{0},\alpha)=-2\pi\int_{L\times[0,1]}e(\nu(L\times[0,1]);\hat{g}_{s}),
\end{equation}
where $e(\nu(L\times[0,1]);\hat{g}_{s})$ is as in Equation \ref{eq:variation_equivariant_eta_invariant}.
\end{lemma}

\begin{proof}
It suffices to assume $L=K$ is connected, since both sides of Equation \ref{eq:variation_torsion} are additive on connected components. First, note that since $K$ is totally geodesic with respect to the metric $g_{s}$ for all $s\in[0,1]$, we have that the form $\theta(g_{s})\in\Omega^{1}(Y;\frak{so}(3))$ can be written as
\[\theta(g_{s})=
\begin{pmatrix}
0 & 0 & 0\\
0 & 0 & \theta_{23}(g_{s}) \\
0 & -\theta_{23}(g_{s}) & 0
\end{pmatrix}.\]

Next, note that the framing on $K$ extends to a framing on $K\times[0,1]\subset Y\times[0,1]$, and hence a section $\wh{\phi}:K\times[0,1]\to\Fr(Y\times[0,1])$. Since $K$ is assumed to be totally geodesic for all $s\in[0,1]$, we have that the annulus $K\times[0,1]$ is totally geodesic as a submanifold of $(Y\times[0,1],\hat{g}_{s})$. In particular, the one parameter family of connection one-forms $\{\theta(g_{s})\}_{s\in[0,1]}$ assembles into a connection one-form $\wh{\theta}\in\Omega^{1}(Y\times[0,1];\frak{so}(4))$ which can be locally written as
\[\wh{\theta}=
\begin{pmatrix}
0 & 0 & 0 & 0\\
0 & 0 & \wh{\theta}_{23}& 0 \\
0 & -\wh{\theta}_{23} & 0 & 0 \\
0 & 0 & 0 & 0
\end{pmatrix}.\]
Here, $\wh{\theta}_{23}\in\Omega^{1}(Y\times[0,1])$ is given by $\wh{\theta}_{23}|_{Y\times\{s\}}=\theta_{23}(g_{s})$ for all $s\in[0,1]$. By Stokes' Theorem we see that
\[\int_{K\times[0,1]}d(\wh{\phi}^{*}\wh{\theta}_{23})=\int_{K}\phi^{*}\theta_{23}(g_{1})-\int_{K}\phi^{*}\theta_{23}(g_{0})=-\big(t(K,g_{1},\alpha)-t(K,g_{0},\alpha)\big).\]

Next, consider the 2-dimensional vector bundle $N(K\times[0,1])\to K\times[0,1]$, and let $\wh{g}^{\nu}_{s}$ be the metric on $N(K\times[0,1])$ induced by $\wh{g}_{s}$. We define $\wh{\theta}^{\nu}\in\Omega^{1}(K\times[0,1],\frak{so}(2))$ to be $SO(2)$-valued connection one-form induced by $\wh{g}^{\nu}_{s}$. Our assumption that $K\times[0,1]$ is totally geodesic implies that
\[\wh{\theta}^{\nu}=\begin{pmatrix}
0 & \wh{\theta}_{23} \\
-\wh{\theta}_{23} & 0
\end{pmatrix}.\]
We see that the corresponding curvature 2-form is given by
\[\Omega^{\nu}=d\wh{\theta}^{\nu}+\wh{\theta}^{\nu}\wedge\wh{\theta}^{\nu}=\begin{pmatrix}
0 & d\wh{\theta}_{23} \\
-d\wh{\theta}_{23} & 0
\end{pmatrix},\]
and thus
\begin{align*}
	&\int_{K\times[0,1]}e(\nu(K\times[0,1]);\wh{g}_{s})=\frac{1}{2\pi}\int_{K\times[0,1]}\text{Pfaff}(\Omega^{\nu}) \\
	&\qquad\qquad=\frac{1}{2\pi}\int_{K\times[0,1]}d\wh{\theta}_{23}=-\frac{1}{2\pi}\big(t(K,g_{1},\alpha)-t(K,g_{0},\alpha)\big).
\end{align*}
\end{proof}

\begin{proposition}
\label{prop:variation_of_equivariant_metric}
For each $k=0,\dots,2m-1$ and any two $\sigma$-equivariant metrics $g_{0}$ and $g_{1}$,
\[n^{\mu^{k}}(Y,\frak{s},\wh{\sigma},g_{1})-n^{\mu^{k}}(Y,\frak{s},\wh{\sigma},g_{0})=\SF^{\mu^{k}}(\{\dirac_{s}\})\]
for any smooth path of equivariant metrics $\{g_{s}\}_{s\in[0,1]}$ interpolating between $g_{0}$ and $g_{1}$.
\end{proposition}

\begin{proof}
The $k=0,m$ cases follow from the variation formula for the non-equivariant correction term (\cite{Man03}). For $k=1,\dots,2m-1$, $k\neq m$, this follows from Equation \ref{eq:variation_equivariant_eta_invariant} and Lemma \ref{lemma:variation_torsion}.
\end{proof}

It remains to prove Condition (1) of Proposition \ref{prop:properties_equivariant_correction_term_characters}. In order to do this, we recast the equivariant correction term in terms of equivariant spin fillings, in analogy with the equality
\begin{equation}
\label{eq:equality_non_equivariant_correection_term}
	n(Y,\frak{s},g)=\ind_{\CC}(\Dirac_{W}^{+})+\tfrac{1}{8}\sigma(W)=\ol{\eta}_{\dirac,g}-\tfrac{1}{8}\eta_{\sign,g}.
\end{equation}
In our setting, the role of $\ind_{\CC}(\Dirac_{W}^{+})$ will be played by the equivariant index $\Spin(W,\frak{t},\wh{\tau},g_{W})$, and the role of $\sigma(W)$ will be played by the following quantity, which was alluded to in the introduction:

\begin{definition}
\label{def:S_invariant}
Let $(W,\frak{t},\wh{\tau})$ be a compact $\ZZ_{m}$-equivariant Riemannian spin 4-manifold such that if $\del W\neq\emptyset$, then $b_{1}(\del W)=0$. We define the \emph{$\S$-invariant} of $(W,\frak{t},\wh{\tau})$ to be the representation
\begin{equation}
    \S(W,\frak{t},\wh{\tau}):=\frac{1}{2m}\sum_{\ell=0}^{2m-1}\Big(\sum_{k=0}^{2m-1}\S^{\mu^{k}}(W,\frak{t},\wh{\tau})\omega_{2m}^{-\ell k}\Big)\xi^{\ell}\in R(\ZZ_{2m})\otimes\CC,
\end{equation}
where:
\[\S^{\mu^{k}}(W,\frak{t},\wh{\tau}):=\left\{
		\begin{array}{ll}
			\sigma(W) & \mbox{if } k=0, \\
			\varepsilon\sigma(W) & \mbox{if } k=m, \\
            2\Big(\sum_{i=1}^{m_{k}}\varepsilon_{k,i}R(\alpha_{k,i},\beta_{k,i})+\sum_{j=1}^{n_{k}}\varepsilon_{k,j}'S(\psi_{k,j})[\Sigma_{k,j}]^{2}\Big) & \mbox{otherwise.}
            \end{array}
    \right.\]
Here, $m_{k}$, $n_{k}$, $R(\alpha_{k,i},\beta_{k,i})$, $S(\psi_{k,j})$, $\varepsilon_{k,i}$, $\varepsilon_{k,j}'$, and $\Sigma_{k,i}$ are as in the $G$-Spin Theorem (Section \ref{subsubsec:g_spin_theorem}), $\varepsilon=\pm 1$ depending on whether $\wh{\tau}$ is an even or odd spin lift, and if $\del W\neq\emptyset$, then $[\Sigma_{k,i}]^{2}$ denotes the self-intersection of $\Sigma_{k,i}$ with respect to the canonical framing.
\end{definition}

We state some useful properties of the $\S$-invariant:

\begin{proposition}
\label{prop:properties_S_invariant}
Let $(W,\frak{t},\wh{\tau})$ be as in Definition \ref{def:S_invariant}. Then:
\begin{enumerate}
    \item $\S(W,\frak{t},\wh{\tau})$ is sent to $\sigma(W)$ under the augmentation map
    \[\alpha:R(\ZZ_{2m})\otimes\CC\to\CC.\]
    \item If $W$ is closed, then
    \[-\tfrac{1}{8}\S(W,\frak{t},\wh{\tau})=\Spin(W,\frak{t},\wh{\tau}).\]
\end{enumerate}
\end{proposition}

\begin{proof}
(1) follows from the identity
\[\alpha(\S(W,\frak{t},\wh{\tau}))=\S^{\mu^{0}}(W,\frak{t},\wh{\tau})=\sigma(W),\]
and (2) follows from the $G$-spin theorem.
\end{proof}

We have the following proposition:

\begin{proposition}
\label{prop:correction_term_spin_filling}
Let $(Y,\frak{s},\wh{\sigma},g)$ be a $\ZZ_{m}$-equivariant Riemannian spin rational homology sphere, and suppose that $(Y,\frak{s},\wh{\sigma},g)$ admits a $\ZZ_{m}$-equivariant Riemannian spin filling $(W,\frak{t},\wh{\tau},g_{W})$. Then
\[n^{\mu^{k}}(Y,\frak{s},\wh{\sigma},g)=\Spin^{\mu^{k}}(W,\frak{t},\wh{\tau},g_{W})+\tfrac{1}{8}\S^{\mu^{k}}(W,\frak{t},\wh{\tau})\in\CC\]
for all $k=0,\dots,2m-1$, and consequently
\[n(Y,\frak{s},\wh{\sigma},g)=\Spin(W,\frak{t},\wh{\tau},g_{W})+\tfrac{1}{8}\S(W,\frak{t},\wh{\tau})\in R(\ZZ_{2m})\otimes\CC.\]
\end{proposition}

\begin{proof}
The cases $k=0,m$ follow from Equation \ref{eq:equality_non_equivariant_correection_term}. Now suppose $1\le k\le 2m-1$, $k\neq m$. Using Equation \ref{eq:equivariant_APS}, we see that:
\begin{align*}
    n^{\mu^{k}}(Y,\frak{s},\wh{\sigma},g)&=\Spin^{\mu^{k}}(W,\frak{t},\hat{\tau},g_{W})+\frac{1}{4}\sum_{i=1}^{m_{k}}\varepsilon_{k,i}R(\alpha_{k,i},\beta_{k,i}) \\
    &\qquad+\frac{1}{4}\sum_{j=1}^{n_{k}}\varepsilon_{k,j}'S(\psi_{k,j})\Big(\int_{\Sigma_{k,j}}e(\nu(\Sigma_{k,j});g_{W})+\sum_{K_{k,\ell}\subset\del\Sigma_{k,j}}\tfrac{1}{2\pi}t(K_{k,\ell},g)\Big).
\end{align*}
The proposition then follows from the equality
\[\int_{\Sigma_{k,j}}e(\nu(\Sigma_{k,j});g_{W})+\sum_{K_{k,\ell}\subset\del\Sigma_{k,j}}\tfrac{1}{2\pi}t(K_{k,\ell},g)=[\Sigma_{k,j}]^{2},\]
which in turn follows from an application of the Chern-Gauss-Bonnet Theorem to the normal bundle of $\Sigma_{k,j}$.
\end{proof}

\begin{remark}
\label{remark:spin_fillings_disjoint_unions}
Note that the correction term $n(Y,\frak{s},\wh{\sigma},g)$ can naturally be extended to disjoint unions of $\ZZ_{m}$-equivariant rational homology spheres, and that it is additive under disjoint unions. Using the fact that $\Omega^{\Spin,\ZZ_{m}}_{3}$ is torsion, we can extend the providence of Proposition \ref{prop:correction_term_spin_filling} as follows: let $K\geq 1$ be the order of $(Y,\frak{s},\wh{\sigma})$ in $\Omega^{\Spin,\ZZ_{m}}_{3}$. Then Proposition \ref{prop:correction_term_spin_filling} implies that
\[n(Y,\frak{s},\wh{\sigma},g)=\frac{1}{K}\Big(\Spin(W,\frak{t},\wh{\tau},g_{W})+\tfrac{1}{8}\S(W,\frak{t},\wh{\tau})\Big),\]
where $(W,\frak{t},\wh{\tau},g_{W})$ is a $\ZZ_{m}$-eequivariant Riemannian spin filling of the disjoint union of $K$ copies of $(Y,\frak{s},\wh{\sigma})$.
\end{remark}

We are now ready to prove Condition (1) of Proposition \ref{prop:properties_equivariant_correction_term_characters}:

\begin{proposition}
\label{prop:integrality}
Let $(Y,\frak{s},\wh{\sigma},g)$ be a $\ZZ_{m}$-equivariant Riemannian spin rational homology sphere. Then $\sum_{k=0}^{2m-1}n^{\mu^{k}}(Y,\frak{s},\wh{\sigma},g)\cdot\omega_{2m}^{-jk}\in\QQ$ for all $j=0,\dots 2m-1$.
\end{proposition}

\begin{proof}
By Remark \ref{remark:spin_fillings_disjoint_unions}, we can without loss of generality assume that $(Y,\frak{s},\wh{\sigma},g)$ admits a $\ZZ_{m}$-eequivariant Riemannian spin filling $(W,\frak{t},\wh{\tau},g_{W})$. Note that since $\Spin(W,\frak{t},\wh{\tau},g_{W})\in R(\ZZ_{2m})$, it suffices to show that $\S(W,\frak{t},\wh{\tau})\in R(\ZZ_{2m})\otimes\QQ$.

By attaching 4-dimensional 0-framed 2-handles to $-Y$ along $L=\cup_{k=1}^{m-1}Y^{\sigma^{k}}$, we obtain a spin cobordism $W'$ from $-Y$ to the manifold $-Y_{0}$ obtained by performing 0-surgery on the fixed point set of $\sigma^{k}$, with orientation opposite to that of $Y$. As in Section \ref{subsec:equivariant_cobordisms}, there is a natural extension of $\wh{\sigma}$ to a spin $\ZZ_{m}$-action $\wh{\tau}'$ on $W'$ which restricts to a free action on $-Y_{0}$. The fixed point sets of $\{(\wh{\tau}\cup_{\wh{\sigma}}\wh{\tau}')^{k}\}$ on $W\cup_{-Y} W'$ can be identified with the fixed point sets of $\{\tau^{k}\}$ on $W$, but with the boundary components capped off by disks. Futhermore, the self-intersection of all the surface components of the fixed point set of $W\cup_{-Y}W'$ is equal to the self-intersection of the corresponding components of the fixed point set of $W$.

Since the free $\ZZ_{m}$-equivariant spin cobordism group $\Omega^{\Spin,\ZZ_{m},\text{free}}_{3}$ is torsion, there exists some $K\geq 1$ such that the disjoint union $\sqcup^{K}(-Y_{0})$ of $K$ copies of $-Y_{0}$ admits a free $\ZZ_{m}$-equivariant spin filling. Applying (2) of Proposition \ref{prop:properties_S_invariant} to the closed manifold
\[X=(\sqcup^{K}(W\cup_{-Y}W'))\cup_{\sqcup^{K}(-Y_{0})}W'',\]
we have an identification
\[\Spin(X,\frak{t}_{X},\wh{\tau}_{X})=-\tfrac{1}{8}\S(X,\frak{t}_{X},\wh{\tau}_{X})=-\tfrac{K}{8}\cdot\S(W,\frak{t},\wh{\tau}),\]
where $\frak{t}_{X}$ and $\wh{\tau}_{X}$ are the spin structure and spin $\ZZ_{m}$-action constructed above. Since $\Spin(X,\frak{t}_{X},\wh{\tau}_{X})\in R(\ZZ_{2m})$, the proposition is thus proven.
\end{proof}

Next we will prove some properties about the equivariant correction term:

\begin{proposition}
\label{prop:correction_term_orientation_reversal}
Let $(Y,\frak{s},\wh{\sigma},g)$ be a $\ZZ_{m}$-equivariant Riemannian spin rational homology sphere, and let $(-Y,\frak{s},\wh{\sigma},g)$ denote its orientation reverse. Then
\begin{equation}
\label{eq:correction_term_orientation_reversal}
    n(Y,\frak{s},\wh{\sigma},g)+n(-Y,\frak{s},\wh{\sigma},g)=-k_{\dirac,\wh{\sigma},g},
\end{equation}
where $k_{\dirac,\wh{\sigma},g}=\ker(\dirac)\in R(\ZZ_{2m})^{*}_{\geq 0}$ denotes the kernel of the Dirac operator on $Y$, thought of as a complex $\ZZ_{2m}$-representation.
\end{proposition}

\begin{proof}
It suffices to show that
\begin{equation}
\label{eq:trace_correction_term_orientation_reversal}
    n^{\mu^{k}}(Y,\frak{s},\wh{\sigma},g)+n^{\mu^{k}}(-Y,\frak{s},\wh{\sigma},g)=-k^{\mu^{k}}_{\dirac,\wh{\sigma},g}
\end{equation}
for each $k=0,\dots,2m-1$. The cases $k=0,m$ follow from the proof of (\cite{Man16}, Proposition 3.8), so it suffices to consider the cases $k=1,\dots,2m-1$, $k\neq m$. If we write
\begin{align*}
    &n^{\mu^{k}}(Y,\frak{s},\wh{\sigma},g)=\frac{1}{2}\Big(\eta^{\mu^{k}}_{\dirac,\wh{\sigma},g}(Y)-k^{\mu^{k}}_{\dirac,\wh{\sigma},g}(Y)\Big)+\frac{1}{8\pi}\sum_{j=1}^{n_{k}}\varepsilon_{k,j}(Y)S(\psi_{k,j})t(K_{k,j},g,Y), \\
    &n^{\mu^{k}}(-Y,\frak{s},\wh{\sigma},g)=\frac{1}{2}\Big(\eta^{\mu^{k}}_{\dirac,\wh{\sigma},g}(-Y)-k^{\mu^{k}}_{\dirac,\wh{\sigma},g}(-Y)\Big)+\frac{1}{8\pi}\sum_{j=1}^{n_{k}}\varepsilon_{k,j}(-Y)S(\psi_{k,j})t(K_{k,j},g,-Y),
\end{align*}
one can check that $\eta^{\mu^{k}}_{\dirac,\wh{\sigma},g}(-Y)=-\eta^{\mu^{k}}_{\dirac,\wh{\sigma},g}(Y)$, $k^{\mu^{k}}_{\dirac,\wh{\sigma},g}(-Y)=k^{\mu^{k}}_{\dirac,\wh{\sigma},g}(Y)$, $\varepsilon_{k,j}(-Y)=\varepsilon_{k,j}(Y)$, and $t(K_{k,j},g,-Y)=-t(K_{k,j},g,Y)$, from which Equation \ref{eq:trace_correction_term_orientation_reversal} follows.
\end{proof}

\begin{proposition}
\label{prop:correction_term_change_of_spin_lift}
Let $(Y,\frak{s},\wh{\sigma},g)$ be a $\ZZ_{m}$-equivariant Riemannian spin rational homology sphere, and let $-\wh{\sigma}$ denote the opposite spin lift of $\wh{\sigma}$. Then
\[n(Y,\frak{s},-\wh{\sigma},g)=\xi^{m}n(Y,\frak{s},\wh{\sigma},g)\in R(\ZZ_{2m})\otimes\QQ.\]
Similarly let $(W,\frak{t},\wh{\tau},g_{W})$ be a $\ZZ_{m}$-equivariant Riemannian spin filling of $(Y,\frak{s},\wh{\sigma},g)$. Then
\begin{align*}
&\Spin(W,\frak{t},-\wh{\tau},g_{W})=\xi^{m}\Spin(W,\frak{t},\wh{\tau},g_{W}), & &\S(W,\frak{t},-\wh{\tau})=\xi^{m}\S(W,\frak{t},\wh{\tau}).	
\end{align*}
\end{proposition}

\begin{proof}
Note that $\Spin(W,\frak{t},-\wh{\tau},g_{W})=\xi^{m}\Spin(W,\frak{t},\wh{\tau},g_{W})$ follows from the fact that $(-\wh{\tau})(\phi)=-(\wh{\tau}(\phi))$ for all $\phi\in\Gamma(\SS_{W}^{+})$. In view of Proposition \ref{prop:correction_term_spin_filling} it suffices to show that $\S(W,\frak{t},-\wh{\tau})=\xi^{m}\S(W,\frak{t},\wh{\tau})$, or equivalently that
\[\S^{\mu^{k}}(W,\frak{t},-\wh{\tau})=(-1)^{k}\S^{\mu^{k}}(W,\frak{t},\wh{\tau}).\]
Since $\Spin^{\mu^{k}}(W,\frak{t},-\wh{\tau},g_{W})=(-1)^{k}\Spin^{\mu^{k}}(W,\frak{t},\wh{\tau},g_{W})$, by the $G$-spin theorem this implies that 
\begin{align*}
	&\varepsilon_{k,i}(-\wh{\tau})=(-1)^{k}\varepsilon_{k,i}(\wh{\tau}), & &\varepsilon'_{k,j}(-\wh{\tau})=(-1)^{k}\varepsilon'_{k,j}(\wh{\tau}),
\end{align*}
for all $i,j,k$, from which the result follows.
\end{proof}

We introduce some additional notation which we will use throughout the paper.

\begin{definition}
\label{def:vector_versions}
Let $m\geq 2$ be an integer, and let $\R^{\mu^{0}},\R^{\mu^{1}},\cdots,\R^{\mu^{2m-1}}\in\CC$ be such that
\[\R=\frac{1}{2m}\sum_{j=0}^{2m-1}\Big(\sum_{k=0}^{2m-1}\R^{\mu^{k}}\cdot\omega^{-jk}\Big)\xi^{j}\in R(\ZZ_{2m})^{*}\otimes\QQ\]
for $*\in\{\ev,\odd\}$. Define $\vec{\R}$ to be the vector 
\[\vec{\R}:=\twopartdef{(\R_{0},\R_{1},\dots,\R_{m-1})\in\QQ^{m}}{*=\ev,}
{(\R_{1/2},\R_{3/2},\dots,\R_{m-1/2})\in\QQ^{m}_{1/2}}{*=\odd,}\]
where for each $j=0,\dots,m-1$:
\begin{align*}
    &\R_{j}:=\tfrac{1}{2m}\sum_{k=0}^{2m-1}\R^{\mu^{k}}\cdot\omega_{2m}^{-2jk}& &(\text{if }\ast=\ev), \\
    &\R_{j+\frac{1}{2}}:=\tfrac{1}{2m}\sum_{k=0}^{2m-1}\R^{\mu^{k}}\cdot\omega_{2m}^{-(2j+1)k}& &(\text{if }\ast=\odd).
\end{align*}
\end{definition}

In the particular cases where $\R$ is equal to one of the rational virtual representations $n(Y,\frak{s},\wh{\sigma},g)$, $\Spin(W,\frak{t},\wh{\tau},g_{W})$, or $\S(W,\frak{t},\wh{\tau})$, defined in Definition \ref{def:equivariant_correction_term}, Section \ref{subsubsec:g_spin_theorem}, and Definition \ref{def:S_invariant}, respectively, we obtain corresponding vectors $\vec{n}(Y,\frak{s},\wh{\sigma},g)$, $\overset{\longrightarrow}{\Spin}(W,\frak{t},\wh{\tau},g_{W})$, and $\vec{\S}(W,\frak{t},\wh{\tau})$. Note that by construction, we have that
\[\vec{n}(Y,\frak{s},\wh{\sigma},g)=\overset{\longrightarrow}{\Spin}(W,\frak{t},\wh{\tau},g_{W})+\tfrac{1}{8}\vec{\S}(W,\frak{t},\wh{\tau}).\]

The following corollary follows immediately from Proposition \ref{prop:correction_term_change_of_spin_lift}:

\begin{corollary}
\label{cor:correction_term_change_of_spin_lift_vector}
Let $(Y,\frak{s},\wh{\sigma},g)$ be a $\ZZ_{m}$-equivariant Riemannian spin rational homology sphere, and let $-\wh{\sigma}$ denote the opposite spin lift of $\wh{\sigma}$. Then
\[n(Y,\frak{s},-\wh{\sigma},g)_{j}=n(Y,\frak{s},\wh{\sigma},g)_{j+m/2},\]
where we use the standard cyclic indexing convention. Consequently:
\[\DDD^{*}(\vec{n}(Y,\frak{s},\wh{\sigma},g))=\DDD^{*}(\vec{n}(Y,\frak{s},-\wh{\sigma},g))\in \QQ^{m}.\]
The corresponding equalities hold for $\overset{\longrightarrow}{\Spin}(W,\frak{t},\wh{\tau},g_{W})$ and $\vec{\S}(W,\frak{t},\wh{\tau})$ with respect to replacing $\wh{\tau}$ by $-\wh{\tau}$.
\end{corollary}

\begin{definition}
\label{def:frak_S}
Let $(W,\frak{t},\tau)$ be a $\ZZ_{m}$-equivariant spin 4-manifold with $b_{1}(\del W)=0$. We define
\[\vec{\SSS}(W,\frak{t},\tau):=\DDD^{*}(\vec{\S}(W,\frak{t},\wh{\tau}))\in\QQ^{m}\]
where $\wh{\tau}$ is any spin lift of $\tau$, and $\vec{\S}(W,\frak{t},\wh{\tau})$ is the vector (see Definition \ref{def:vector_versions}) corresponding to the representation $\S(W,\frak{t},\wh{\tau})\in R(\ZZ_{2m})^{*}\otimes\QQ$ from Definition \ref{def:S_invariant}. For each $\ell=0,\dots,m-1$ we define $\SSS(W,\frak{t},\tau)_{\ell}\in\QQ$ to be the $\ell$-th component of $\vec{\SSS}(W,\frak{t},\tau)$. More precisely, we have that:
\begin{enumerate}
	\item If $m$ is even and $\wh{\tau}$ is an even spin lift, then
	\[\SSS(W,\frak{t},\tau)_{\ell}=\twopartdef{\S(W,\frak{t},\wh{\tau})_{\frac{\ell}{2}}+\S(W,\frak{t},\wh{\tau})_{\frac{\ell+m}{2}}}{\ell\text{ even},}{0}{\ell\text{ odd}.}\]
	\item If $m$ is even and $\wh{\tau}$ is an odd spin lift, then
	\[\SSS(W,\frak{t},\tau)_{\ell}=\twopartdef{0}{\ell\text{ even},}{\S(W,\frak{t},\wh{\tau})_{\frac{\ell}{2}}+\S(W,\frak{t},\wh{\tau})_{\frac{\ell+m}{2}}}{\ell\text{ odd}.}\]
	\item If $m$ is odd and $\wh{\tau}$ is an even spin lift, then
	\[\SSS(W,\frak{t},\tau)_{\ell}=\twopartdef{\S(W,\frak{t},\wh{\tau})_{\frac{\ell}{2}}}{\ell\text{ even},}{\S(W,\frak{t},\wh{\tau})_{\frac{\ell+m}{2}}}{\ell\text{ odd}.}\]
	\item If $m$ is odd and $\wh{\tau}$ is an odd spin lift, then
	\[\SSS(W,\frak{t},\tau)_{\ell}=\twopartdef{\S(W,\frak{t},\wh{\tau})_{\frac{\ell+m}{2}}}{\ell\text{ even},}{\S(W,\frak{t},\wh{\tau})_{\frac{\ell}{2}}}{\ell\text{ odd}.}\]
\end{enumerate}
Here we use the cyclic indexing convention as per usual. By Corollary \ref{cor:correction_term_change_of_spin_lift_vector} each $\SSS(W,\frak{t},\tau)_{\ell}\in\QQ$ is independent of the choice of spin lift.
\end{definition}

Next, we will discuss an alternate choice of notation for the invariants defined above in the case of even spin lifts, which may be more useful in certain contexts. More precisely, let $m\geq 2$ be an integer and let $\R\in R(\ZZ_{2m})^{\ev}\otimes\QQ$ be a rational virtual representation with corresponding characters $\R^{\mu^{0}},\R^{\mu^{1}},\cdots,\R^{\mu^{2m-1}}\in\CC$ so that
\[\R=\frac{1}{2m}\sum_{j=0}^{2m-1}\Big(\sum_{k=0}^{2m-1}\R^{\mu^{k}}\cdot\omega^{-jk}\Big)\xi^{j}.\]
for $*\in\{\ev,\odd\}$. Using the isomorphism $R(\ZZ_{2m})^{\ev}\cong R(\ZZ_{m})$, we can alternatively view $\R$ as the (rational, virtual) $\ZZ_{m}$-representation
\[\R=\frac{1}{m}\sum_{j=0}^{m-1}\Big(\sum_{k=0}^{m-1}\R^{\gamma^{k}}\cdot\omega^{-jk}\Big)\zeta^{j}\in R(\ZZ_{m})\otimes\QQ,\]
where $\R^{\gamma^{k}}:=\R^{\mu^{k}}=\R^{\mu^{k+m}}$ for all $k=0,\dots,m-1$. Here, we are considering $\R^{\gamma^{k}}$ as the character of the $\ZZ_{m}$-representation $\R$ at $\gamma^{k}\in\<\gamma\>=\ZZ_{m}$. This alternate notation will sometimes be used in the cases where
\[\R=n(Y,\frak{s},\wh{\sigma},g),\;\Spin(W,\frak{t},\wh{\tau},g_{W}),\text{ or }\S(W,\frak{t},\wh{\tau})\]
in the case of even spin lifts, but we will also freely use this notation in other situations, e.g., equivariant eta-invariants. For example, for each $k=0,\dots,m-1$, we can alternatively write the character of $n(Y,\frak{s},\wh{\sigma},g)$ at $\gamma^{k}\in\ZZ_{m}$ as follows:
\begin{equation}
\label{eq:correction_term_even}
    n^{\gamma^{k}}(Y,\frak{s},\wh{\sigma},g):=\twopartdef{\ol{\eta}_{\dirac,g}-\tfrac{1}{8}\eta_{sign,g}}{k=0,}{\ol{\eta}^{\gamma^{k}}_{\dirac,\wh{\sigma},g}+\frac{1}{8\pi}\sum_{j=1}^{n_{k}}\varepsilon_{k,j}S(\psi_{k,j})t(K_{k,j},g)}{k\neq 0.}
\end{equation}
We invite the reader to recast all of the above material in the setting of even spin lifts using this alternate notation, if so desired.

We conclude this section with the following proposition, which points out a simplification of the equivariant correction term in the case of involutions:

\begin{proposition}
\label{prop:correction_term_involutions}
Let $(Y,\frak{s},\wh{\iota},g)$ be a $\ZZ_{2}$-equivariant Riemannian spin rational-homology sphere.
\begin{enumerate}
    \item If $\wh{\iota}$ is of even type, then
    \begin{align*}
        &n(Y,\frak{s},\wh{\iota},g)_{0}=\tfrac{1}{2}(n(Y,\frak{s},g)+\ol{\eta}^{\gamma}_{\dirac,\wh{\iota},g}), &
        &n(Y,\frak{s},\wh{\iota},g)_{1}=\tfrac{1}{2}(n(Y,\frak{s},g)-\ol{\eta}^{\gamma}_{\dirac,\wh{\iota},g}).
    \end{align*}
    \item If $\wh{\iota}$ is of odd type, then
    \begin{align*}
        &n(Y,\frak{s},\wh{\iota},g)_{\frac{1}{2}}=\tfrac{1}{2}(n(Y,\frak{s},g)-i\ol{\eta}^{\mu}_{\dirac,\wh{\iota},g}), &
        &n(Y,\frak{s},\wh{\iota},g)_{\frac{3}{2}}=\tfrac{1}{2}(n(Y,\frak{s},g)+i\ol{\eta}^{\mu}_{\dirac,\wh{\iota},g}).
    \end{align*}
\end{enumerate}
\end{proposition}

\begin{proof}
If $\wh{\iota}$ is even, this follows from the fact that $Y^{\iota}=\emptyset$. If $\wh{\iota}$ is odd, let $Y^{\iota}=K_{1}\cup\cdots\cup K_{n}$. Recall that
\[n^{\mu}(Y,\frak{s},\wh{\iota},g)=-n^{\mu^{3}}(Y,\frak{s},\wh{\iota},g)=\ol{\eta}^{\mu}_{\dirac,\wh{\iota},g}+\frac{1}{8\pi}\sum_{j=1}^{n}\varepsilon_{j}S(\psi_{j})\tau(K_{j},g),\]
where $\psi_{j}\in\RR/2\pi\ZZ$ is the angle by which $\iota$ acts by $e^{i\psi_{j}}$ in a tubular neighborhood of $K_{j}$, and $S(\psi_{j})=\cot(\tfrac{\psi_{j}}{2})\csc(\tfrac{\psi_{j}}{2})$. But since $\iota^{2}=\text{id}$, we must have that $\psi_{j}=\pi$ for all $j=1,\dots,n$. Therefore
\[S(\psi_{j})=\cot(\tfrac{\pi}{2})\csc(\tfrac{\pi}{2})=0\]
for all $j$, and
\[n^{\mu}(Y,\frak{s},\wh{\iota},g)=-n^{\mu^{3}}(Y,\frak{s},\wh{\iota},g)=\ol{\eta}^{\mu}_{\dirac,\wh{\iota},g}.\]
\end{proof}

\subsection{The Seiberg-Witten Floer Spectrum Class}
\label{subsec:sw_floer_spectrum_class}

In this section, we define a metric-independent $\CC$-$G^{*}_{m}$-spectrum class $\text{SWF}(Y,\frak{s},\wh{\sigma})\in\CCC_{G^{*}_{m},\CC}$ associated to a $\ZZ_{m}$-equivariant spin rational homology sphere $(Y,\frak{s},\wh{\sigma})$. 

Let $(Y,\frak{s},\wh{\sigma},g)$ be a $\ZZ_{m}$-equivariant Riemannian spin rational homology sphere. Fix an eigenvalue cut-off $\lambda>>0$, and let $I_{\lambda}$ be the $G^{*}_{m}$-equivariant Conley index as in Equation \ref{eq:conley_index}. Roughly, the metric-dependent stable homotopy type $\SWF(Y,\frak{s},\wh{\sigma},g)$ is given by the desuspension
\[\SWF(Y,\frak{s},\wh{\sigma})=\Sigma^{-V^{0}_{-\lambda}}I_{\lambda},\]
and the metric-independent stable homotopy type $\SWF(Y,\frak{s},\wh{\sigma})$ is given by the further desuspension
\[\SWF(Y,\frak{s},\wh{\sigma})=\Sigma^{-\frac{1}{2}n(Y,\frak{s},\wh{\sigma},g)\HH}\Sigma^{-V^{0}_{-\lambda}}I_{\lambda}.\]
To be more precise, from Equation \ref{eq:v_nu_lambda_decomposition} we have a $G^{*}_{m}$-equivariant decomposition
\[V^{0}_{-\lambda}=\big(\mbfv_{-\lambda}^{0}(\RR)\cdot\wt{\RR}\big)\oplus\big(\mbfv_{-\lambda}^{0}(\HH)\cdot\HH\big)\]
for some representations 
\begin{align*}
	&\mbfv_{-\lambda}^{0}(\RR)\in RO(\ZZ_{m})_{\geq 0}, & &\mbfv_{-\lambda}^{0}(\HH)\in R(\ZZ_{2m})^{*}_{\geq 0}.
\end{align*}
The following observation can be deduced by a similar method as in the proof of (\cite{Man16}, Lemma 3.6), via perturbing the CSD functional by a $\ZZ_{m}$-equivariant imaginary-valued one-form $\omega\in i\Omega^{1}_{\ZZ_{m}}(Y)$:

\begin{observation}
\label{observation:level_conley_index}
The Conley index $I_{\lambda}$ is a space of type $G^{*}_{m}$-$\SWF$ at level $\mbfv_{-\lambda}^{0}(\RR)\in RO(\ZZ_{m})_{\geq 0}$. 
\end{observation}

Next, let $c^{\CC}_{\RR}:RO(\ZZ_{m})\to R(\ZZ_{m})^{\sym}$ denote the complexification map, and define
\[\mbfv_{-\lambda}^{0}(\CC):=c^{\CC}_{\RR}(\mbfv_{-\lambda}^{0}(\RR))\in R(\ZZ_{m})_{\geq 0}^{\sym}.\]
If $\mbfv_{-\lambda}^{0}(\CC)_{j}$ denotes the coefficient of $\zeta^{j}$ in $\mbfv_{-\lambda}^{0}(\CC)$, let $\mbfu=\sum_{j=0}^{m-1}u_{j}\zeta^{j}\in R(\ZZ_{m})_{\geq 0}^{\sym}$ be the representation with coefficients given by
\[u_{j}:=\twopartdef{0}{\mbfv_{-\lambda}^{0}(\CC)_{j}\equiv 0\pmod{2},}{1}{\mbfv_{-\lambda}^{0}(\CC)_{j}\equiv 1\pmod{2}.}\]
Then by construction, the coefficients of the sum $\mbfv_{-\lambda}^{0}(\CC)+\mbfu\in R(\ZZ_{m})_{\geq 0}^{\sym}$ are even, so that $\frac{1}{2}(\mbfv_{-\lambda}^{0}(\CC)+\mbfu)\in R(\ZZ_{m})_{\geq 0}^{\sym}$. Finally, define the real representation $\Umid$ by
\[\Umid:=(c_{\RR}^{\CC})^{-1}(\mbfu)\in RO(\ZZ_{m})_{\geq 0}.\]

From Observation \ref{observation:level_conley_index}, it follows that the suspension $\Sigma^{\Umid\wt{\RR}}I_{\lambda}$ is a space of type $G^{*}_{m}$-$\SWF$ at even level $\mbfv_{-\lambda}^{0}(\RR)+\Umid\in RO(\ZZ_{m})_{\geq 0}$, and hence a space of type $\CC$-$G^{*}_{m}$-$\SWF$ at level $\frac{1}{2}(\mbfv_{-\lambda}^{0}(\CC)+\mbfu)\in R(\ZZ_{m})_{\geq 0}^{\sym}$.

\begin{definition}
We define the \emph{$(g,\lambda)$-dependent $G^{*}_{m}$-equivariant Seiberg-Witten Floer stable homotopy type} associated to $(Y,\frak{s},\wh{\sigma},g,\lambda)$ to be the $G^{*}_{m}$-spectrum class
\[\SWF(Y,\frak{s},\wh{\sigma},g,\lambda):=\big[(I_{\lambda},0,0)\big]\in\CCC_{G^{*}_{m}},\]
the \emph{metric-dependent $G^{*}_{m}$-equivariant Seiberg-Witten Floer stable homotopy type} associated to $(Y,\frak{s},\wh{\sigma},g)$ to be the $\CC$-$G^{*}_{m}$-spectrum class
\[\SWF(Y,\frak{s},\wh{\sigma},g):=\Big[\big(\Sigma^{\Umid\wt{\RR}}I_{\lambda},\tfrac{1}{2}(\mbfv_{-\lambda}^{0}(\CC)+\mbfu),\mbfv_{-\lambda}^{0}(\HH)\big)\Big]\in\CCC_{G^{*}_{m},\CC},\]
and the \emph{(metric-independent) $G^{*}_{m}$-equivariant Seiberg-Witten Floer stable homotopy type} associated to $(Y,\frak{s},\wh{\sigma})$ to be the $\CC$-$G^{*}_{m}$-spectrum class
\[\SWF(Y,\frak{s},\wh{\sigma}):=\Big[\big(\Sigma^{\Umid\wt{\RR}}I_{\lambda},\tfrac{1}{2}(\mbfv_{-\lambda}^{0}(\CC)+\mbfu),\mbfv_{-\lambda}^{0}(\HH)+\tfrac{1}{2}n(Y,\frak{s},\wh{\sigma},g)\big)\Big]\in\CCC_{G^{*}_{m},\CC}.\]
\end{definition}

\begin{remark}
	Note that by Observation \ref{observation:level_conley_index}, $\SWF(Y,\frak{s},\wh{\sigma},g)$ and $\SWF(Y,\frak{s},\wh{\sigma})$ are both spectrum classes at level $\mbf{0}\in R(\ZZ_{m})^{\sym}$.
\end{remark}

\begin{proposition}
The spectrum class $\SWF(Y,\frak{s},\wh{\sigma},g)$ is independent of the eigenvalue cut-off $\lambda$, and the spectrum class $\SWF(Y,\frak{s},\wh{\sigma})$ is independent of $\lambda$ and the metric $g$.
\end{proposition}

\begin{proof}
The first statement is clear. The second statement follows from Proposition \ref{prop:variation_of_equivariant_metric}, and the fact that the equivariant spectral flow of the linearization of the $CSD$ functional is precisely equal to the equivariant spectral flow of the Dirac operator.
\end{proof}

We conclude this section with the following proposition:

\begin{proposition}
\label{prop:floer_spectrum_duality}
Let $(Y,\frak{s},\wh{\sigma},g)$ be a $\ZZ_{m}$-equivariant spin rational homology sphere. Then $\SWF(Y,\frak{s},\wh{\sigma},g)$ and $\SWF(-Y,\frak{s},\wh{\sigma},g)$ are $[(S^{0},0,0)]$-dual.
\end{proposition}

\begin{proof}
As in the proof of (\cite{Man16}, Proposition 3.8), one can adapt the argument of \cite{Cor00} to show that the Conley indices of the flow and its inverse are $G^{*}_{m}$-equivariantly $V_{-\lambda}^{\lambda}$-dual to each other. The result then follows from Proposition \ref{prop:correction_term_orientation_reversal} and the fact that
\begin{align*}
	&\mbfv_{-\lambda}^{\lambda}(\RR)=\mbfv_{-\lambda}^{0}(\RR)+\ol{\mbfv}_{-\lambda}^{0}(\RR)\in RO(\ZZ_{m})_{\geq 0}, \\
	&\mbfv_{-\lambda}^{\lambda}(\HH)+k_{\dirac,\wh{\sigma},g}=\mbfv_{-\lambda}^{0}(\HH)+\ol{\mbfv}_{-\lambda}^{0}(\HH)\in R(\ZZ_{2m})_{\geq 0}^{*},
\end{align*}
where $\ol{\mbfv}_{-\lambda}^{0}(\RR)$ and $\ol{\mbfv}_{-\lambda}^{0}(\HH)$ are the representations corresponding to $\ol{V}_{-\lambda}^{0}$, with $\ol{V}$ denoting the Coloumb slice for $-Y$.
\end{proof}

\bigskip

%% file: kappa_invariants.tex
\section{Equivariant \texorpdfstring{$\kappa$}{κ}-Invariants and Equivariant Relative 10/8-ths Inequalities}
\label{sec:kappa_invariants}

In this section we define a package of \emph{equivariant $\kappa$-invariants} for $\ZZ_{m}$-equivariant spin rational homology spheres. In particular to a triple $(Y,\frak{s},\sigma)$ we associate two subsets $\K(Y,\frak{s},\sigma)$, $\K^{\wedge}(Y,\frak{s},\sigma)$ of the lattice $\Q^{m}_{*}$ defined in Section \ref{subsec:stable_k_invariants}.

In Section \ref{subsec:equivariant_kappa_invariants} we define our equivariant $\kappa$-invariants. In Section \ref{subsec:relative_bauer_furuta}, we review Manolescu's construction of the relative Bauer--Furuta invariants arising from the Seiberg--Witten equations on 
4-manifolds with boundary (\cite{Man03},\cite{Man16}, corrected in \cite{Kha15}), and in Section \ref{subsec:G_*_m_cobordisms}, we analyze the behavior of these cobordism maps in the $G^{*}_{m}$-equivariant setting. Finally in Section \ref{subsec:main_theorems}, we state and prove our relative equivariant 10/8ths inequalities.

\bigskip
\subsection{Equivariant \texorpdfstring{$\kappa$}{κ}-invariants}
\label{subsec:equivariant_kappa_invariants}

Recall that in \cite{Man14}, Manolescu defined the invariant $\kappa(Y,\frak{s})$ to be double the $k$-invariant of the Seiberg-Witten Floer spectrum class $\SWF(Y,\frak{s})$. There are essentially two different ways of ``doubling'' $\mbfk^{\st}(\SWF(Y,\frak{s},\wh{\sigma}))$, the first by multiplying the elements of $\mbfk^{\st}(\SWF(Y,\frak{s},\wh{\sigma}))\subset\Q^{m}_{*}$ by $2$, and the other is to consider the equivariant $k$-invariants of the ``double" of $\SWF(Y,\frak{s},\wh{\sigma})$. This leads us to the following definition:

\begin{definition}
\label{def:equivariant_kappa_invariants}
Let $(Y,\frak{s},\wh{\sigma})$ be a $\ZZ_{m}$-equivariant spin rational homology sphere. We define the \emph{equivariant $\kappa$-invariants of $(Y,\frak{s},\wh{\sigma})$} as follows:
\begin{enumerate}
    \item Define
    \[\K(Y,\frak{s},\wh{\sigma}):=2\cdot\mbfk^{\st}\big(\SWF(Y,\frak{s},\wh{\sigma})\big)\subset\Q^{m}_{*},\]
    and $\vec{\ol{\kappa}}(Y,\frak{s},\wh{\sigma})$, $\vec{\ul{\kappa}}(Y,\frak{s},\wh{\sigma})$ to be the least upper bound and greatest lower bound, respectively, of $\K(Y,\frak{s},\wh{\sigma})$ as a subset of $\Q^{m}_{*}$.
    \item Define 
    \[\K^{\wedge}(Y,\frak{s},\wh{\sigma}):=\mbfk^{\st}\big(\SWF(Y,\frak{s},\wh{\sigma})\wedge\SWF(Y,\frak{s},\wh{\sigma})\big)\subset\Q^{m}_{*},\]
    and $\vec{\ol{\kappa}}\,^{\wedge}(Y,\frak{s},\wh{\sigma})$, $\vec{\ul{\kappa}}\,^{\wedge}(Y,\frak{s},\wh{\sigma})$ to be the least upper bound and greatest lower bound, respectively, of $\K^{\wedge}(Y,\frak{s},\wh{\sigma})$ as a subset of $\Q^{m}_{*}$.
    \item For $(Y,\frak{s},\wh{\iota})$ an odd-type $\ZZ_{2}$-equivariant rational homology sphere, we define the invariant $\wt{\kappa}(Y,\frak{s},\wh{\iota})\in\QQ$ to be the unique element of
    \[\K(Y,\frak{s},\wh{\iota})\subset(\Q^{2}_{\odd},\preceq,+,|\cdot|)\xrightarrow[\cong]{|\cdot|}(\QQ,\le,+,|\cdot|).\]
    \item For $m=p^{r}$ an odd prime power, we define
    \begin{align*}
        &\K^{\pi}(Y,\frak{s},\wh{\sigma}):=\pi(\K(Y,\frak{s},\wh{\sigma}))\subset\QQ^{2}, & &\K^{\wedge,\pi}(Y,\frak{s},\wh{\sigma}):=\pi(\K^{\wedge}(Y,\frak{s},\wh{\sigma}))\subset\QQ^{2}
    \end{align*}
    where
    \[\pi:(\Q^{p^{r}}_{*},\preceq,+,|\cdot|)\to(\QQ^{2},\preceq,+,|\cdot|)\]
    is the projection map from Propostion \ref{prop:stable_lattice_p^r_projection}. Furthermore we define
    \begin{align*}
        &\vec{\ol{\kappa}}^{\pi}(Y,\frak{s},\wh{\sigma})=(\ol{\kappa}_{0}^{\pi}(Y,\frak{s},\wh{\sigma}),\ol{\kappa}_{1}^{\pi}(Y,\frak{s},\wh{\sigma})):=\vee\K^{\pi}(Y,\frak{s},\wh{\sigma})\in\QQ^{2} \\
        &\vec{\ul{\kappa}}^{\pi}(Y,\frak{s},\wh{\sigma})=(\ul{\kappa}_{0}^{\pi}(Y,\frak{s},\wh{\sigma}),\ul{\kappa}_{1}^{\pi}(Y,\frak{s},\wh{\sigma})):=\wedge\K^{\pi}(Y,\frak{s},\wh{\sigma})\in\QQ^{2} \\
        &\vec{\ol{\kappa}}^{\wedge,\pi}(Y,\frak{s},\wh{\sigma})=(\ol{\kappa}_{0}^{\wedge,\pi}(Y,\frak{s},\wh{\sigma}),\ol{\kappa}_{1}^{\wedge,\pi}(Y,\frak{s},\wh{\sigma})):=\vee\K^{\wedge,\pi}(Y,\frak{s},\wh{\sigma})\in\QQ^{2} \\ &\vec{\ul{\kappa}}^{\wedge,\pi}(Y,\frak{s},\wh{\sigma})=(\ul{\kappa}_{0}^{\wedge,\pi}(Y,\frak{s},\wh{\sigma}),\ul{\kappa}_{1}^{\wedge,\pi}(Y,\frak{s},\wh{\sigma})):=\wedge\K^{\wedge,\pi}(Y,\frak{s},\wh{\sigma})\in\QQ^{2}
    \end{align*}
    to be the corresponding least upper bounds and greatest lower bounds of these subsets.
\end{enumerate}
\end{definition}

It turns out that these equivariant $\kappa$-invariants do not depend on the spin lift $\wh{\sigma}$:

\begin{proposition}
\label{prop:independent_of_spin_lift}
The invariants in Definition \ref{def:equivariant_kappa_invariants} are all independent of the choice of spin lift $\wh{\sigma}$ of $\sigma:Y\to Y$.
\end{proposition}

\begin{proof}
It suffices to show that $\mbfk^{\st}\big(\SWF(Y,\frak{s},\wh{\sigma})\big)$ is independent of the spin lift. Fix an equivariant metric $g$ and an eigenvalue cut-off $\lambda$. Note that every element of $\mbfk^{\st}\big(\SWF(Y,\frak{s},\wh{\sigma})\big)$ is of the form
\[\vec{k}-\big[\DDD^{*}\big(\vec{\mbfv}\,^{0}_{-\lambda}(\HH)_{\wh{\sigma}}\big)\big]-\tfrac{1}{2}\big[\DDD^{*}\big(n(Y,\frak{s},\wh{\sigma},g)\big)\big]\in\Q^{m}_{*}\]
for some 
\[\vec{k}\in\min\Big(\Pi_{\st}\big(\mbfk(\Sigma^{\Umid\wt{\RR}}I_{\lambda,\wh{\sigma}})\big)\Big)\subset\N^{m}_{\st,*},\]
where:
\begin{enumerate}
	\item $\vec{\mbfv}\,^{0}_{-\lambda}(\HH)_{\wh{\sigma}}\in\NN^{m}$ or $\NN^{m}_{1/2}$ denotes the vector corresponding to the representation $\mbfv_{-\lambda}^{0}(\HH)_{\wh{\sigma}}\in R(\ZZ_{2m})_{\geq 0}^{*}$ as in Section \ref{subsec:sw_floer_spectrum_class}, defined with respect to the spin lift $\wh{\sigma}$.
	\item $I_{\lambda,\wh{\sigma}}$ denotes the associated Conley index defined with respect to $\wh{\sigma}$.
	\item $\Pi_{\st}:\N^{m}\to\N^{m}_{\st,*}$ is the projection map from Section \ref{subsec:stable_k_invariants}.
	\item $\Umid\in RO(\ZZ_{m})_{\geq 0}$ is as in Section \ref{subsec:sw_floer_spectrum_class}.
\end{enumerate}
By Corollary \ref{cor:correction_term_change_of_spin_lift_vector}, we have that 
\[\DDD^{*}\big(n(Y,\frak{s},\wh{\sigma},g)\big)=\DDD^{*}\big(n(Y,\frak{s},-\wh{\sigma},g)\big),\]
and so it suffices to show that:
\begin{enumerate}
	\item $\DDD^{*}\big(\vec{\mbfv}\,^{0}_{-\lambda}(\HH)_{\wh{\sigma}}\big)=\DDD^{*}\big(\vec{\mbfv}\,^{0}_{-\lambda}(\HH)_{-\wh{\sigma}}\big)$ as elements of $\NN^{m}$, and
	\item $I(\Sigma^{\Umid\wt{\RR}}I_{\lambda,\wh{\sigma}})=I(\Sigma^{\Umid\wt{\RR}}I_{\lambda,-\wh{\sigma}})$ as subsets of $\N^{m}$.
\end{enumerate}

For (1), the fact that $-\wh{\sigma}$ acts by $-1$ times the action of $\wh{\sigma}$ on the Seiberg-Witten configuration space implies that $\mbfv_{-\lambda}^{0}(\HH)_{-\wh{\sigma}}=\xi^{m}\mbfv_{-\lambda}^{0}(\HH)_{\wh{\sigma}}$, and hence $\DDD^{*}\big(\vec{\mbfv}\,^{0}_{-\lambda}(\HH)_{\wh{\sigma}}\big)=\DDD^{*}\big(\vec{\mbfv}\,^{0}_{-\lambda}(\HH)_{-\wh{\sigma}}\big)$.

This leaves us to consider (2). First suppose that $m$ is even. Then we have automorphisms
\begin{align*}
    &\alpha^{\ev}:G^{\ev}_{m}\xrightarrow{\cong} G^{\ev}_{m} &
    &\alpha^{\odd}:G^{\odd}_{m}\xrightarrow{\cong} G^{\odd}_{m}
\end{align*}
which restrict to the identity on $\Pin(2)\subset G^{*}_{m}$, and send $\gamma\mapsto-\gamma$ and $\mu\mapsto-\mu$, respectively. If $m$ is odd, we have isomorphisms
\begin{align*}
    &\alpha^{\ev\to\odd}:G^{\ev}_{m}\xrightarrow{\cong} G^{\odd}_{m} & &\alpha^{\odd\to\ev}:G^{\odd}_{m}\xrightarrow{\cong} G^{\ev}_{m}
\end{align*}
which are equal to the identity on $\Pin(2)\subset G^{*}_{m}$, and send $\gamma\mapsto-\mu$ and $\mu\mapsto-\gamma$, respectively. In either case, we denote the relevant automorphisms/isomorphisms by a single map $\alpha:G^{*}_{m}\to G^{*}_{m}$, which induces automorphisms/isomorphisms of the real and complex representation rings of $G^{*}_{m}$:
\begin{align*}
    &\wt{\alpha}_{\RR}:RO(G^{*}_{m})\xrightarrow{\cong}RO(G^{*}_{m}), & &\wt{\alpha}_{\CC}:R(G^{*}_{m})\xrightarrow{\cong}R(G^{*}_{m}).
\end{align*}

Note that $\wt{\alpha}_{\RR},\wh{\alpha}_{\CC}$ act as the identity on representations which are fixed by $S^{1}\subset G^{*}_{m}$ --- in particular $\wt{\alpha}_{\RR}$ acts by the identity on the subset $RO(\ZZ_{m})\wt{\RR}\subset RO(G^{*}_{m})$, so that $\wt{\alpha}_{\RR}(\Umid\wt{\RR})=\Umid\wt{\RR}$. Furthermore, one can check that $\wh{\alpha}_{\CC}$ sends $w_{j}\mapsto w_{j}$ and $z_{k}\mapsto z_{k+\frac{m}{2}}$.

Next observe that the $G^{*}_{m}$-action induced by $-\wh{\sigma}$ on the Seiberg-Witten configuration space is precisely equal to the $G^{*}_{m}$-action induced by $\wh{\sigma}$, precomposed by $\alpha$. We therefore have a canonical $G^{*}_{m}$-homotopy equivalence
\[f:\Sigma^{\Umid\wt{\RR}}I_{\lambda,\wh{\sigma}}\xrightarrow{\simeq}\Sigma^{\Umid\wt{\RR}}I_{\lambda,-\wh{\sigma}}\]
which covers $\alpha$, as well as an induced isomorphism
\[f^{*}:\wt{K}_{G^{*}_{m}}(\Sigma^{\Umid\wt{\RR}}I_{\lambda,-\wh{\sigma}})\xrightarrow{\cong}\wt{K}_{G^{*}_{m}}(\Sigma^{\Umid\wt{\RR}}I_{\lambda,\wh{\sigma}})\]
which covers $\wt{\alpha}$. By analyzing the commutative diagram
\begin{center}
    \begin{tikzcd}
    \wt{K}_{G^{*}_{m}}(\Sigma^{\Umid\wt{\RR}}I_{\lambda,-\wh{\sigma}}) \arrow[r, "\iota^{*}"] \arrow[d, "f^{*}"]
    & \wt{K}_{G^{*}_{m}}(\Sigma^{\Umid\wt{\RR}}(I_{\lambda,-\wh{\sigma}})^{S^{1}})\cong R(G^{*}_{m}) \arrow[d, "(f^{S^{1}})^{*}=\,\wt{\alpha}"] \\
    \wt{K}_{G^{*}_{m}}(\Sigma^{\Umid\wt{\RR}}I_{\lambda,\wh{\sigma}}) \arrow[r, "\iota^{*}" ]
    & \wt{K}_{G^{*}_{m}}(\Sigma^{\Umid\wt{\RR}}(I_{\lambda,\wh{\sigma}})^{S^{1}})\cong R(G^{*}_{m})
    \end{tikzcd}
\end{center}
we see that
\[\III(\Sigma^{\Umid\wt{\RR}}I_{\lambda,\wh{\sigma}})=\wt{\alpha}\big(\III(\Sigma^{\Umid\wt{\RR}}I_{\lambda,-\wh{\sigma}})\big)\subset R(G^{*}_{m}).\]
But since
\[w_{0}z_{k+\frac{m}{2}}=w_{0}w_{2k+m}=w_{0}w_{2k}=w_{0}z_{k}\in R(G^{*}_{m}),\]
we see that $I(\Sigma^{\Umid\wt{\RR}}I_{\lambda,\wh{\sigma}})=I(\Sigma^{\Umid\wt{\RR}}I_{\lambda,-\wh{\sigma}})$, as desired.
\end{proof}

We will henceforth drop the choice of spin lift from our notation for the equivariant $\kappa$-invariants. These invariants satisfy the following properties:

\begin{theorem}
\label{theorem:properties_equivariant_kappa_invariants}
Let $(Y,\frak{s},\sigma)$ be a $\ZZ_{m}$-equivariant spin rational homology sphere.
\begin{enumerate}
    \item For any orientation-preserving diffeomorphism $f:Y\to Y$ which preserves $\frak{s}$, we have that:
    \begin{align*}
        &\K(Y,\frak{s},f^{-1}\circ\sigma\circ f)=\K(Y,\frak{s},\sigma), & 
        &\K^{\wedge}(Y,\frak{s},f^{-1}\circ\sigma\circ f)=\K^{\wedge}(Y,\frak{s},\sigma).
    \end{align*}
    \item For any $\vec{\kappa}\in\K(Y,\frak{s},\sigma)$ and $\vec{\kappa}'\in\K(-Y,\frak{s},\sigma)$, where $-Y$ denotes the orientation-reverse of $Y$, we have that:
    \[\vec{\kappa}+\vec{\kappa}'\succeq [\vec{0}].\]
    An analogous inequality holds for elements of $\K^{\wedge}(Y,\frak{s},\sigma)$ and $\K^{\wedge}(-Y,\frak{s},\sigma)$.
    \item For any $\vec{\kappa}\in\K(Y,\frak{s},\sigma)$ and any $\vec{\kappa}^{\wedge}\in\K^{\wedge}(Y,\frak{s},\sigma)$ we have that
    \begin{align*}
        &|\vec{\kappa}|\geq\kappa(Y,\frak{s}), & &|\vec{\kappa}^{\wedge}|\geq\tfrac{1}{2}\kappa(Y\# Y,\frak{s}\#\frak{s}).
    \end{align*}
\end{enumerate}
\end{theorem}

\begin{proof}
For (1), fix an equivariant metric $g$ on $Y$ and spin lifts $\wh{\sigma}$, $\wh{f}$ of $\sigma$ and $f$, respectively. Then $\wh{f}$ induces a $G^{*}_{m}$-equivariant homeomorphism of configuration spaces 
\[\C(Y,\frak{s},\wh{\sigma},g)\xrightarrow{\cong}\C(Y,\frak{s},\wh{f}^{-1}\circ\wh{\sigma}\circ\wh{f},f^{*}g),\]
and hence a $G^{*}_{m}$-homotopy equivalence of Conley indices
\[I_{\lambda}(Y,\frak{s},\wh{\sigma},g)\xrightarrow{\simeq}I_{\lambda}(Y,\frak{s},\wh{f}^{-1}\circ\wh{\sigma}\circ\wh{f},f^{*}g),\]
from which it follows that
\[\SWF(Y,\frak{s},\wh{\sigma})=SWF(Y,\frak{s},\wh{f}^{-1}\circ\wh{\sigma}\circ\wh{f})\]
as $\CC$-$G^{*}_{m}$-spectrum classes. Statement (2) follows from
Propositions \ref{prop:stable_k_invariants_duality} and \ref{prop:floer_spectrum_duality}, and (3) follows from Lemma \ref{lemma:comparison_of_k_invariants}.
\end{proof}

We have the following comparison lemma between $\K(Y,\frak{s},\sigma)$ and $\K^{\wedge}(Y,\frak{s},\sigma)$:

\begin{lemma}
\label{lemma:comparison_normal_wedge_invariants}
Let $(Y,\frak{s},\wh{\sigma})$ be a $\ZZ_{m}$-equivariant spin rational homology sphere. Then: 
\begin{enumerate}
    \item For each $\kappa\in\K(Y,\frak{s},\sigma)$:
    \begin{enumerate}
        \item $\vec{\kappa}\not\prec\vec{\kappa}^{\wedge}$ for all $\vec{\kappa}^{\wedge}\in\K^{\wedge}(Y,\frak{s},\sigma)$.
        \item There exists $\vec{\kappa}^{\wedge}\in\K^{\wedge}(Y,\frak{s},\sigma)$ such that $\vec{\kappa}\succeq\vec{\kappa}^{\wedge}$.
    \end{enumerate}
    \item In particular, $\vec{\ol{\kappa}}(Y,\frak{s},\sigma)\succeq\vec{\ul{\kappa}}^{\wedge}(Y,\frak{s},\sigma)$.
\end{enumerate}
\end{lemma}

\begin{proof}
Follows from Lemmas \ref{lemma:comparing_minima_subsets}, \ref{lemma:comparing_minima_sum_of_subsets} and \ref{lemma:k_invariants_products}.
\end{proof}

\begin{definition}
\label{def:floer_K_G_split}
Let $(Y,\frak{s},\sigma)$ be a $\ZZ_{m}$-equivariant spin rational homology sphere.
\begin{enumerate}
    \item We say $(Y,\frak{s},\sigma)$ is \emph{Floer $K_{G^{*}_{m}}$-split} if the $\CC$-$G^{*}_{m}$-spectrum class $\SWF(Y,\frak{s},\wh{\sigma})$ is $K_{G^{*}_{m}}$-split in the sense of Definition \ref{def:stable_K_G_split} for any choice of spin lift $\wh{\sigma}$ of $\sigma$.
    \item We say $(Y,\frak{s},\sigma)$ is \emph{Floer $\wedge^{2}$-$K_{G^{*}_{m}}$-split} if the $\CC$-$G^{*}_{m}$-spectrum class 
    \[\wedge^{2}\SWF(Y,\frak{s},\wh{\sigma}):=\SWF(Y,\frak{s},\wh{\sigma})\wedge\SWF(Y,\frak{s},\wh{\sigma})\]
    is $K_{G^{*}_{m}}$-split for any choice of spin lift $\wh{\sigma}$ of $\sigma$.
\end{enumerate}
\end{definition}

\begin{remark}
The proof of Proposition \ref{prop:independent_of_spin_lift} implies that $\SWF(Y,\frak{s},\wh{\sigma})$ is $K_{G^{*}_{m}}$-split if and only if $\SWF(Y,\frak{s},-\wh{\sigma})$ is $K_{G^{*}_{m}}$-split, i.e., it suffices to verify the $K_{G^{*}_{m}}$-splitness property for only one choice of spin lift $\wh{\sigma}$ of $\sigma$. Furthermore, note that if $(Y,\frak{s},\sigma)$ is Floer $K_{G^{*}_{m}}$-split (respectively, Floer $\wedge^{2}$-$K_{G^{*}_{m}}$-split), then $\K(Y,\frak{s},\sigma)$ (respectively, $\K^{\wedge}(Y,\frak{s},\sigma)$) consists of a single element.
\end{remark}

Finally, we have the following proposition in the case of odd-type involutions:

\begin{proposition}
\label{prop:kappa_kappa_tilde}
Let $(Y,\frak{s},\iota)$ be an odd-type $\ZZ_{2}$-equivariant spin rational homology sphere. Then
\[\wt{\kappa}(Y,\frak{s},\iota)=\kappa(Y,\frak{s},\iota)\text{ or }\kappa(Y,\frak{s},\iota)+2.\]
\end{proposition}

\begin{proof}
Follows from Proposition \ref{prop:stable_k_tilde_2_odd_inequality}.
\end{proof}

\bigskip
\subsection{Relative Bauer-Furuta Invariants}
\label{subsec:relative_bauer_furuta}

Let $m\geq 2$ be an integer, let $(Y,\frak{s},\wh{\sigma})$ be a $\ZZ_{m}$-equivariant spin rational homology sphere, and let $(W,\frak{t},\wh{\tau})$ be an equivariant spin filling of $(Y,\frak{s},\wh{\sigma})$ such that $b_{1}(W)=0$. Pick an equivariant metric $g$ on $Y$ and an equivariant metric $g_{W}$ on $W$ so that the the boundary has a collar neighborhood isometric to $[0,1]\times Y$. Denote by $\SS_{W}=\SS^{+}_{W}\oplus \SS^{-}_{W}$ the spinor bundle of $W$ and $\SS$ the spinor bundle of $Y$.

As outlined in \cite{Kha15}, given any 1-form $A\in\Omega^{1}(W)$, the inclusion $\del W\hookrightarrow W$ induces a decomposition
\[A|_{\del W}=\mbft A+\mbfn A\]
of the restriction of $A$ to $\del W$ into its tangential and normal components. We then define the space of 1-forms satisfying the \emph{double Coulomb condition} to be
\[\Omega^{1}_{CC}(W):=\{A\in\Omega^{1}(W)\;|\;d^{*}A=0, d^{*}(\mbft A)=0, \int_{Y}\mbft(\ast A)=0.\}\]
As in the three-dimensional case, the action of $\tau$ on $W$ induces an action of $G^{*}_{m}=G^{\ev}_{m}$ (if $\wh{\tau}$ is an even spin lift) or $G^{\odd}_{m}$ (if $\wh{\tau}$ is an odd spin lift) on $i\Omega^{k}(W)\oplus\Gamma(\SS^{\pm}_{W})$ via
\begin{align*}
    &e^{i\theta}\cdot(A,\Phi)=(A,e^{i\theta}\Phi), \\
    &j\cdot(a,\phi)=(-A,j\Phi),
\end{align*}
and
\begin{align*}
    &\gamma\cdot(A,\Phi)=(\sigma^{*}(A),\wh{\sigma}^{*}(\Phi))\text{ if }\ast=\ev, \\
    &\mu\cdot(A,\Phi)=(\sigma^{*}(A),\wh{\sigma}^{*}(\Phi))\text{ if }\ast=\odd.
\end{align*}
In particular, this action descends to a $G^{*}_{m}$-action on $i\Omega^{1}_{CC}(W)\oplus\Gamma(\SS^{+}_{W})$ because $g_{W}$ is isometric to a product near the boundary, and similarly descends to a $G^{*}_{m}$-action on $i\Omega^{2}_{+}(W)\oplus\Gamma(\SS^{-}_{W})$.

Let $r$ denote the restriction map
\[r:i\Omega^{1}_{CC}(W)\oplus\Gamma(\SS^{+}_{W})\to i\Omega^{1}_{C}(Y)\oplus\Gamma(\SS)=V\]
from the double Coloumb slice of $W$ to the Coloumb slice of $Y$. Combining this with the Seiberg-Witten map, we obtain a map
\[\wt{SW}=SW\oplus r:i\Omega^{1}_{CC}(W)\oplus\Gamma(\SS^{+}_{W})\to \big(i\Omega^{2}_{+}(W)\oplus\Gamma(\SS^{-}_{W})\big)\oplus\big( i\Omega^{1}_{C}(Y)\oplus\Gamma(\SS)\big)\]
\[(A,\Phi)\mapsto(d^{+}A-\rho^{-1}((\Phi\Phi^{*})_{0}),\,\Dirac^{+}_{W}\Phi+\rho(A)\Phi)\oplus r(A,\Phi).\]
Unfortunately the linearization of $\wt{SW}$ is not Fredholm. To remedy this, we fix an eigenvalue cut-off $\nu>>0$ and consider the modified map
\[\wt{SW}^{\nu}=SW\oplus (\Pi^{\nu}\circ r):i\Omega^{1}_{CC}(W)\oplus\Gamma(\SS_{W}^{+})\to i\Omega^{2}_{+}(W)\oplus\Gamma(\SS_{W}^{-})\oplus V^{\nu}_{-\infty}\]
whose linearization is Fredholm, where $\Pi^{\nu}:i\Omega^{1}_{C}(Y)\oplus\Gamma(\SS)\to 
V^{\nu}_{-\infty}$ denotes the canonical projection map. We can write $\wt{SW}^{\nu}=\wh{D}\oplus(\Pi^{\nu}\circ r) + Q$, where
\[\wh{D}:i\Omega^{1}_{CC}(W)\oplus\Gamma(\SS^{+}_{W})\to i\Omega_{+}^{2}(W)\oplus\Gamma(\SS^{-}_{W})\]
\[(A,\Phi)\mapsto(d^{+}A,D_{W}^{+}\Phi)\]
and
\[Q:i\Omega^{1}_{CC}(W)\oplus\Gamma(\SS^{+}_{W})\to i\Omega_{+}^{2}(W)\oplus\Gamma(\SS^{-}_{W})\]
\[(A,\Phi)\mapsto(-\rho^{-1}((\Phi\Phi^{*})_{0}),\rho(A)\Phi)\]
is a quadratic map with nice compactness properties.

For the following, let $\U_{W}:=i\Omega^{1}_{CC}(W)\oplus\Gamma(\SS^{+}_{W})$ and $\U'_{W}:=i\Omega^{2}_{+}(W)\oplus\Gamma(\SS^{-}_{W})$, so that we can write $\wt{SW}^{\nu}$ as a map $\wt{SW}^{\nu}:\U_{W}\to\U'_{W}\oplus V_{-\infty}^{\nu}$. Since the linear map $\wh{D}\oplus(\Pi^{\nu}\circ r)$ is Fredholm, in particular its cokernel is finite. Pick a finite-dimensional subspace $U'\subset\U'_{W}$ and an eigenvalue $\lambda<<0$ such that the finite-dimensional subspace $U'\oplus V_{\lambda}^{\nu}\subset\U'_{W}\oplus V_{-\infty}^{\nu}$ contains $\text{coker}(\wh{D}\oplus(\Pi^{\nu}\circ r))$. Next, let
\[U:=(\wh{D}\oplus(\Pi^{\nu}\circ r))^{-1}(U'\oplus V_{\lambda}^{\nu})\subset\U_{W}\]
and consider the corresponding projected map
\[\pi_{U'\oplus V_{\lambda}^{\nu}}\circ\wt{SW}^{\nu}|_{U}:U\to U'\oplus V_{\lambda}^{\nu}\]
between the finite dimensional subspaces. The nice compactness properties of $\wt{SW}^{\nu}$ imply that if $R_{2}>0$ is chosen so that $B(R_{2},V_{\lambda}^{\nu})$ is an isolating neighborhood for the compressed Seiberg-Witten flow on $V_{\lambda}^{\nu}$, then there exists $R_{1}>0$ such that the above map descends to a map
\[\pi_{U'\oplus V_{\lambda}^{\nu}}\circ\wt{SW}^{\nu}|_{B(R_{1},U)}:B(R_{1},U)\to V\oplus B(R_{2},V_{\lambda}^{\nu}).\]
If $U'$ and $-\lambda$ are chosen large enough then this induces a based map
\begin{equation}
\label{eq:relative_bauer_furuta_filling}
    \psi_{U',\nu,\lambda}:S^{U}\to S^{U'}\wedge I_{\lambda}^{\nu}
\end{equation}
from the one-point compactification of $U$ to a suspension of the Conley index $I_{\lambda}^{\nu}$. One can show that all of the above spaces have an induced $G^{*}_{m}$-action, and that all of the above maps are equivariant with respect to this action as well, owing to the equivariance of $g_{W}$.

If $(W,\frak{t},\wh{\tau})$ is an equivariant cobordism from $(Y_{0},\frak{s}_{0},\wh{\sigma}_{0})$ to $(Y_{1},\frak{s}_{1},\wh{\sigma}_{1})$, via duality we obtain a map
\begin{equation}
\label{eq:relative_bauer_furuta_cobordism}
    \psi_{U',\nu,\lambda}:S^{U}\wedge(I_{0})_{\lambda}^{\nu}\to S^{U'}\wedge(I_{1})_{\lambda}^{\nu},
\end{equation}
where $(I_{0})_{\lambda}^{\nu}$, $(I_{1})_{\lambda}^{\nu}$ are Conley indices for $Y_{0}$, $Y_{1}$, respectively.

\bigskip
\subsection{\texorpdfstring{$G^{*}_{m}$}{G*m}-Equivariant Cobordism Maps}
\label{subsec:G_*_m_cobordisms}

Suppose $(W,\frak{t},\wh{\tau})$ is a $\ZZ_{m}$-equivariant spin 4-manifold with $b_{1}(W)=0$. Let $\H^{2}_{+}(W,i\RR)$ denote the space of imaginary-valued harmonic self-dual 2-forms, with $\tau$ acting by pull-back and $j$ acting by $\pm 1$. Considered as a real $G^{*}_{m}$-representation, it can be written in the form 
\[\H^{2}_{+}(W,i\RR)=b_{2,\RR}^{+}(W,\tau)\cdot[\wt{\RR}]\in RO(G^{*}_{m}),\]
where $b_{2,\RR}^{+}(W,\tau)\in RO(\ZZ_{m})_{\geq 0}$ denotes the real $\ZZ_{m}$-representation
\begin{equation}
    b_{2,\RR}^{+}(W,\tau)=b_{2,\RR}^{+}(W,\tau)_{0}+\Big(\sum_{j=1}^{\lfloor\frac{m-1}{2}\rfloor}\tfrac{1}{2}b_{2,\RR}^{+}(W,\tau)_{j}\cdot\nu_{j}\Big)+b_{2,\RR}^{+}(W,\tau)_{m/2}\cdot\rho,
\end{equation}
and where $b_{2,\RR}^{+}(W,\tau)_{j}$ denotes the $\RR$-dimension of the $(\omega_{m}^{j}+\omega_{m}^{-j})$-eigenspace of the induced action of $\tau$ on $\H^{2}_{+}(W,i\RR)$. Here we take the convention thaat $b_{2}^{+}(W,\tau)_{m/2}=0$ if $m$ is odd.

Similarly, let $\H^{2}_{+}(W,\CC)$ denote the space of complex-valued harmonic self-dual 2-forms, with $G^{*}_{m}$-action as above. Observe that
\[\H^{2}_{+}(W,\CC)=c(\H^{2}_{+}(W,i\RR))=b_{2}^{+}(W,\tau)\cdot[\wt{\CC}]\in R(G^{*}_{m}),\]
where $c:RO(G)\to R(G)$ denotes the complexification map, and
\[b_{2}^{+}(W,\tau):=c(b_{2}^{+}(W,\tau))\in R(\ZZ_{m})_{\geq 0}.\]
More precisely, we can write
\[b_{2}^{+}(W,\tau)=\sum_{k=0}^{m-1}b_{2}^{+}(W,\tau)_{k}\cdot \zeta^{k},\]
where:
\begin{align*}
    &b_{2}^{+}(W,\tau)_{0}=b_{2,\RR}^{+}(W,\tau)_{0}, \\
    &b_{2}^{+}(W,\tau)_{k}=b_{2}^{+}(W,\tau)_{m-k}=\tfrac{1}{2}b_{2,\RR}^{+}(W,\tau)_{k}\text{ for all }1\le k\le \lfloor\tfrac{m-1}{2}\rfloor,\text{ and} \\
    &b_{2}^{+}(W,\tau)_{m/2}=b_{2,\RR}^{+}(W,\tau)_{m/2}\text{ if }m\text{ is even}.
\end{align*}
Equivalently, $b_{2}^{+}(W,\tau)_{k}$ is the complex dimension of the $\omega_{m}^{k}$-eigenspace of the induced action of $\tau$ on $\H^{2}_{+}(W,\CC)$, as described in the introduction.

We now return to our analysis of cobordism maps. Suppose $(W,\frak{t},\wh{\tau})$ is a $\ZZ_{m}$-equivariant spin cobordism from $(Y_{0},\frak{s}_{0},\wh{\sigma}_{0})$ to $(Y_{1},\frak{s}_{1},\wh{\sigma}_{1})$. Then the map from Equation \ref{eq:relative_bauer_furuta_cobordism} takes the form of a based $G^{*}_{m}$-equivariant map
\begin{equation}
\label{eq:relative_bauer_furuta_cobordism_real}
    \psi_{U',\nu,\lambda}:S^{\mbfr\wt{\RR}+\mbfh\HH}\wedge(I_{0})_{\lambda}^{\nu}\to S^{\mbfr'\wt{\RR}+\mbfh'\HH}\wedge(I_{1})_{\lambda}^{\nu},
\end{equation}
where:
\begin{align}
    \mbfr-\mbfr'&=\mbfv_{\lambda}^{0}(\RR)-b_{2,\RR}^{+}(W,\tau)\in RO(\ZZ_{m}),\label{eq:real_cobordism_r} \\
\begin{split}
    \mbfh-\mbfh'&=\mbfv_{\lambda}^{0}(\HH)+\tfrac{1}{2}\Spin(W,\frak{t},\wh{\tau},g_{W}) \\
    &=\mbfv_{\lambda}^{0}(\HH)+\tfrac{1}{2}\big(n(Y_{1},\frak{s}_{1},\wh{\sigma}_{1},g_{1})-n(Y_{0},\frak{s}_{0},\wh{\sigma}_{0},g_{0})\big)-\tfrac{1}{16}\S(W,\frak{t},\wh{\tau})\in R(\ZZ_{2m})^{*}.\label{eq:real_cobordism_h}
\end{split}
\end{align}

In the special case where each $b_{2}^{+}(W,\tau)_{j}$ is an even integer, by suspending by copies of $\wt{\RR}$, $\wt{\VV}_{j}$ and $\wt{\RR}_{m/2}$ if necessary we can assume both the domain and target of Equation \ref{eq:relative_bauer_furuta_cobordism_real} are spaces of type $\CC$-$G^{*}_{m}$-$\SWF$. Let us denote by $\vec{b}_{2}^{+}(W,\tau)$ the following vector:
\[\vec{b}_{2}^{+}(W,\tau):=(b_{2}^{+}(W,\tau)_{0},\dots,b_{2}^{+}(W,\tau)_{m-1})\in\NN^{m}.\]
Using Equations \ref{eq:real_cobordism_r} and \ref{eq:real_cobordism_h}, we make the following observation:

\begin{observation}
\label{observation:real_cobordism}
Suppose each $b_{2}^{+}(W,\tau)_{j}$ is an even integer. Then $\psi_{U',\nu,\lambda}$ can be interpreted as a morphism
\[f:[(S^{0},\tfrac{1}{2}b_{2}^{+}(W,\tau),\tfrac{1}{16}\S(W,\frak{t},\wh{\tau}))]\wedge\SWF(Y_{0},\frak{s}_{0},\wh{\sigma}_{0})\to\SWF(Y_{1},\frak{s}_{1},\wh{\sigma}_{1})\]
of $\CC$-$G^{*}_{m}$-spectrum classes. In particular:
\begin{enumerate}
    \item The difference in levels between the domain and codomain is given by
    \[-\tfrac{1}{2}b_{2}^{+}(W,\tau)\in R(\ZZ_{m})^{\sym}.\]
    \item The equivariant $k$-invariants of the domain are given by
    \[\mbfk^{\st}(\SWF(Y_{0},\frak{s}_{0},\wh{\sigma}_{0}))-\tfrac{1}{16}\big[\vec{\SSS}(W,\frak{t},\wh{\tau})\big]\subset\Q^{m}_{*}.\]
    \item The equivariant $k$-invariants of the codomain are given by
    \[\mbfk^{\st}(\SWF(Y_{1},\frak{s}_{1},\wh{\sigma}_{1}))\subset\Q^{m}_{*}.\]
\end{enumerate}
\end{observation}

Next, consider the following "complexified" or "doubled" version of Equation \ref{eq:relative_bauer_furuta_cobordism_real}:
\begin{equation}
\label{eq:relative_bauer_furuta_cobordism_complex}
    \psi_{U',\nu,\lambda,\CC}:=\psi_{U',\nu,\lambda}\wedge\psi_{U',\nu,\lambda}:S^{\mbfs\wt{\CC}+2\mbfh\HH}\wedge(\wedge^{2}(I_{0})_{\lambda}^{\nu})\to S^{\mbfs'\wt{\CC}+2\mbfh'\HH}\wedge(\wedge^{2}(I_{1})_{\lambda}^{\nu}),
\end{equation}
where $\mbfs=c(\mbfr)$, $\mbfs'=c(\mbfr')$ with $\mbfr,\mbfr'$ as in Equation \ref{eq:real_cobordism_r}, and $\wedge^{2}X:=X\wedge X$ denotes the two-fold smash product of a space $A$. We can therefore write
\begin{equation}
\label{eq:complex_cobordism_s}
    \mbfs-\mbfs'=\mbfv_{\lambda}^{0}(\CC)-b_{2}^{+}(W,\tau)\in R(\ZZ_{m})^{\sym},
\end{equation}
where $\mbfv_{\lambda}^{0}(\CC)=c(\mbfv_{\lambda}^{0}(\RR))\in R(\ZZ_{m})^{\sym}$ as in Section \ref{subsec:sw_floer_spectrum_class}. Using Equations \ref{eq:real_cobordism_h} and \ref{eq:complex_cobordism_s}, we have the following observation:

\begin{observation}
\label{observation:complex_cobordism}
The map $\psi_{U',\nu,\lambda,\CC}$ can be interpreted as a morphism
\[f_{\CC}:[(S^{0},b_{2}^{+}(W,\tau),\tfrac{1}{8}\S(W,\frak{t},\wh{\tau}))]\wedge(\wedge^{2}\SWF(Y_{0},\frak{s}_{0},\wh{\sigma}_{0}))\to\wedge^{2}\SWF(Y_{1},\frak{s}_{1},\wh{\sigma}_{1})\]
of $\CC$-$G^{*}_{m}$-spectrum classes. In particular:
\begin{enumerate}
    \item The difference in levels between the domain and codomain is given by
    \[-b_{2}^{+}(W,\tau)\in R(\ZZ_{m})^{\sym}.\]
    \item The equivariant $k$-invariants of the domain are given by
    \[\K^{\wedge}(Y_{0},\frak{s}_{0},\wh{\sigma}_{0})-\tfrac{1}{8}\big[\vec{\SSS}(W,\frak{t},\wh{\tau})\big]\subset\Q^{m}_{*}.\]
    \item The equivariant $k$-invariants of the codomain are given by
    \[\K^{\wedge}(Y_{1},\frak{s}_{1},\wh{\sigma}_{1})\subset\Q^{m}_{*}.\]
\end{enumerate}
\end{observation}

Next, suppose $(Y,\frak{s},\wh{\sigma})$ is a $\ZZ_{m}$-equivariant spin rational homology sphere and $(W,\frak{t},\wh{\tau})$ is an equivariant spin filling of $(Y,\frak{s},\wh{\sigma})$ with $b_{1}(W)=0$. 

\begin{observation}
\label{observation:filling}
Observe the following:
\begin{enumerate}
    \item Suppose each $b_{2}^{+}(W,\tau)_{j}$ is an even integer. Then the corresponding relative Bauer-Furuta map can be interpreted as a morphism
    \[f:[(S^{0},\tfrac{1}{2}b_{2}^{+}(W,\tau),\tfrac{1}{16}\S(W,\frak{t},\wh{\tau}))]\to\SWF(Y,\frak{s},\wh{\sigma})\]
    of $\CC$-$G^{*}_{m}$-spectrum classes, such that:
    \begin{enumerate}
        \item The difference in levels between the domain and codomain is given by
        \[-\tfrac{1}{2}b_{2}^{+}(W,\tau)\in R(\ZZ_{m})^{\sym}.\]
        \item The equivariant $k$-invariants of the domain are given by
        \[\big\{-\tfrac{1}{16}\big[\vec{\SSS}(W,\frak{t},\wh{\tau})\big]\big\}\subset\Q^{m}_{*}.\]
        \item The equivariant $k$-invariants of the codomain are given by
        \[\mbfk^{\st}\big(\SWF(Y,\frak{s},\wh{\sigma})\big)\subset\Q^{m}_{*}.\]
    \end{enumerate}
    \item The corresponding complexified relative Bauer-Furuta map can be interpreted as a morphism
    \[f_{\CC}:[(S^{0},b_{2}^{+}(W,\tau),\tfrac{1}{8}\S(W,\frak{t},\wh{\tau}))]\to\wedge^{2}\SWF(Y,\frak{s},\wh{\sigma})\]
    of $\CC$-$G^{*}_{m}$-spectrum classes, such that:
    \begin{enumerate}
        \item The difference in levels between the domain and codomain is given by
        \[-b_{2}^{+}(W,\tau)\in R(\ZZ_{m})^{\sym}.\]
        \item The equivariant $k$-invariants of the domain are given by
        \[\big\{-\tfrac{1}{8}\big[\vec{\SSS}(W,\frak{t},\wh{\tau})\big]\big\}\subset\Q^{m}_{*}.\]
        \item The equivariant $k$-invariants of the codomain are given by
        \[\K^{\wedge}(Y,\frak{s},\sigma)\subset\Q^{m}_{*}.\]
    \end{enumerate}
\end{enumerate}
\end{observation}

Finally, we observe what happens in the case where $(W,\frak{t},\wh{\tau})$ is a \emph{closed} $\ZZ_{m}$-equivariant spin 4-manifold with $b_{1}(W)=0$:

\begin{observation}
\label{observation:closed}
Observe the following:
\begin{enumerate}
    \item Suppose each $b_{2}^{+}(W,\tau)_{j}$ is an even integer. The corresponding Bauer-Furuta map can be interpreted as a morphism
    \[f:[(S^{0},\tfrac{1}{2}b_{2}^{+}(W,\tau),\tfrac{1}{16}\S(W,\frak{t},\wh{\tau}))]\to[(S^{0},0,0)]\]
    of $\CC$-$G^{*}_{m}$-spectrum classes, such that:
    \begin{enumerate}
        \item The difference in levels between the domain and codomain is given by
        \[-\tfrac{1}{2}b_{2}^{+}(W,\tau)\in R(\ZZ_{m})^{\sym}.\]
        \item The equivariant $k$-invariants of the domain are given by
        \[\big\{-\tfrac{1}{16}\big[\vec{\SSS}(W,\frak{t},\wh{\tau})\big]\big\}\subset\Q^{m}_{*}.\]
        \item The equivariant $k$-invariants of the codomain are given by
        \[\{[\vec{0}]\}\subset\Q^{m}_{*}.\]
    \end{enumerate}
    \item The corresponding complexified Bauer-Furuta map can be interpreted as a morphism
    \[f_{\CC}:[(S^{0},b_{2}^{+}(W,\tau),\tfrac{1}{8}\S(W,\frak{t},\wh{\tau}))]\to[(S^{0},0,0)]\]
    of $\CC$-$G^{*}_{m}$-spectrum classes, such that:
    \begin{enumerate}
        \item The difference in levels between the domain and codomain is given by
        \[-b_{2}^{+}(W,\tau)\in R(\ZZ_{m})^{\sym}.\]
        \item The equivariant $k$-invariants of the domain are given by
        \[\big\{-\tfrac{1}{8}\big[\vec{\SSS}(W,\frak{t},\wh{\tau})\big]\big\}\subset\Q^{m}_{*}.\]
        \item The equivariant $k$-invariants of the codomain are given by
        \[\{[\vec{0}]\}\subset\Q^{m}_{*}.\]
    \end{enumerate}
\end{enumerate}
\end{observation}

\bigskip
\subsection{Main Theorems}
\label{subsec:main_theorems}

Before stating our results for the equivariant $\kappa$-invariants, we will restate Manolescu's results in the $\Pin(2)$-equivariant setting.

\begin{theorem}[\cite{Man14}, Theorems 1.1, 1.4, Corollary 1.5]
\label{theorem:manolescu}
Let $(Y_{0},\frak{s}_{0})$ and $(Y_{1},\frak{s}_{1})$ be spin rational homology spheres, let $(W,\frak{t})$ be a spin cobordism from $(Y_{0},\frak{s}_{0})$ to $(Y_{1},\frak{s}_{1})$, and let
\begin{align*}
    &p=-\tfrac{1}{8}\sigma(W), & &q=b_{2}^{+}(W).
\end{align*}
Then
\begin{equation}
q+\kappa(Y_{1},\frak{s}_{1})\geq p+\kappa(Y_{0},\frak{s}_{0})+C, 
\end{equation}
where:
\begingroup
\renewcommand{\arraystretch}{1.3} 
\begin{equation}
C=\left\{
	\begin{array}{ll}
            2 & \mbox{if } b_{2}^{+}(W)\text{ is even,}\geq 2,\text{ and }(Y_{0},\frak{s}_{0})\text{ is }K_{\Pin(2)}\text{-split}, \\
            1 & \mbox{if } b_{2}^{+}(W)\text{ is odd and }(Y_{0},\frak{s}_{0})\text{ is }K_{\Pin(2)}\text{-split}, \\
            0 & \mbox{if } b_{2}^{+}(W)=0,\text{ or } \\
            & \mbox{if } b_{2}^{+}(W)\text{ is even,}\geq 2,\text{ and }(Y_{0},\frak{s}_{0})\text{ is not }K_{\Pin(2)}\text{-split}, \\
            -1 & \mbox{if } b_{2}^{+}(W)\text{ is odd and }(Y_{0},\frak{s}_{0})\text{ is not }K_{\Pin(2)}\text{-split}.
	\end{array}
	\right.
\end{equation}
\endgroup
In particular, if $(W,\frak{t})$ is a smooth, compact, spin 4-manifold with boundary a spin rational homology sphere $(Y,\frak{s})$, then:
\begingroup
\renewcommand{\arraystretch}{1.3} 
\begin{equation}
q+\kappa(Y,\frak{s})\geq p+\left\{
		\begin{array}{ll}
            2 & \mbox{if } q\text{ is even,} \geq 2, \\
            1 & \mbox{if } q\text{ is odd,} \\
            0 & \mbox{if } q=0.
		\end{array}
	\right.
\end{equation}
\endgroup
\end{theorem}

We now state our $\ZZ_{m}$-equivariant analogues of Manolescu's inequalities:

\begin{theorem}
\label{theorem:equivariant_homology_cobordism}
Suppose $(Y_{0},\frak{s}_{0},\wh{\sigma}_{0})$, $(Y_{1},\frak{s}_{1},\wh{\sigma}_{1})$ are $\ZZ_{m}$-equivariant rational homology cobordant $\ZZ_{m}$-equivariant spin rational homology spheres. Then:
\[\SWF(Y_{0},\frak{s}_{0},\wh{\sigma}_{0})\equiv_{l}\SWF(Y_{1},\frak{s}_{1},\wh{\sigma}_{1}).\]
In particular:
\begin{align*}
&\K(Y_{0},\frak{s}_{0},\sigma_{0})=\K(Y_{1},\frak{s}_{1},\sigma_{1}), & &\K^{\wedge}(Y_{0},\frak{s}_{0},\sigma_{0})=\K^{\wedge}(Y_{1},\frak{s}_{1},\sigma_{1}).
\end{align*}
\end{theorem}

\begin{proof}
Let $(W,\frak{t},\wh{\tau})$ be a $\ZZ_{m}$-equivariant spin rational homology cobordism from $(Y_{0},\frak{s}_{0},\wh{\sigma}_{0})$ to  $(Y_{1},\frak{s}_{1},\wh{\sigma}_{1})$. Note that any such cobordism must satisfy $\vec{b}_{2}^{+}(W,\tau)=-\tfrac{1}{8}\vec{\SSS}(W,\frak{t},\tau)=\vec{0}$. Hence by Observation \ref{observation:real_cobordism}, the cobordism $(W,\frak{t},\wh{\tau})$ along with its inverse induce morphisms
\begin{align*}
    &f:\SWF(Y_{0},\frak{s}_{0},\wh{\sigma}_{0})\to\SWF(Y_{1},\frak{s}_{1},\wh{\sigma}_{1}), & &g:\SWF(Y_{1},\frak{s}_{1},\wh{\sigma}_{1})\to\SWF(Y_{0},\frak{s}_{0},\wh{\sigma}_{0}),
\end{align*}
which induce $G^{*}_{m}$-homotopy equivalences on their $S^{1}$-fixed point sets. It follows that
\[\SWF(Y_{0},\frak{s}_{0},\wh{\sigma}_{0})\equiv_{l}\SWF(Y_{1},\frak{s}_{1},\wh{\sigma}_{1}),\]
and hence
\[\wedge^{2}\SWF(Y_{0},\frak{s}_{0},\wh{\sigma}_{0})\equiv_{l}\wedge^{2}\SWF(Y_{1},\frak{s}_{1},\wh{\sigma}_{1}).\]
The result then follows by Corollary \ref{cor:stable_k_invariants_local_equivalence}.
\end{proof}

Note that this implies the following corollary:

\begin{corollary}
\label{cor:connected_sums}
Let $(Y_{0},\frak{s}_{0},\wh{\sigma}_{0})$ and $(Y_{1},\frak{s}_{1},\wh{\sigma}_{1})$ and be $\ZZ_{m}$-equivariant spin rational homology spheres, and let $(Y_{0}\# Y_{1},\frak{s}_{0}\#\frak{s}_{1},\wh{\sigma_{0}\#\sigma_{1}})$ be an equivariant connected sum of $Y_{0}$, $Y_{1}$ as in Section \ref{subsec:equivariant_connected_sums}, assuming it is well-defined. Then
\[\SWF(Y_{0}\# Y_{1},\frak{s}_{0}\#\frak{s}_{1},\wh{\sigma_{0}\#\sigma_{1}})\equiv_{l}\SWF(Y_{0},\frak{s}_{0},\wh{\sigma}_{0})\wedge\SWF(Y_{1},\frak{s}_{1},\wh{\sigma}_{1}).\]
\end{corollary}

\begin{proof}
Follows from Theorem \ref{theorem:equivariant_homology_cobordism} and the fact that $(Y_{0},\frak{s}_{0},\wh{\sigma}_{0})\sqcup(Y_{1},\frak{s}_{1},\wh{\sigma}_{1})$ is $\ZZ_{m}$-equivariantly spin cobordant to $(Y_{0}\# Y_{1},\frak{s}_{0}\#\frak{s}_{1},\wh{\sigma_{0}\#\sigma_{1}})$, as in Example \ref{ex:connected_sum_branched_covers_cobordism}.
\end{proof}

\begin{theorem}
\label{theorem:equivariant_cobordism}
Let $(Y_{0},\frak{s}_{0},\wh{\sigma}_{0})$, $(Y_{1},\frak{s}_{1},\wh{\sigma}_{1})$ be $\ZZ_{m}$-equivariant spin rational homology spheres, and suppose that $(W,\frak{t},\wh{\tau})$ is a smooth spin $\ZZ_{m}$-equivariant cobordism from $(Y_{0},\frak{s}_{0},\wh{\sigma}_{0})$ to $(Y_{1},\frak{s}_{1},\wh{\sigma}_{1})$ with $b_{1}(W)=0$. Furthermore, let
\begin{align*}
    &\vec{\frak{p}}:=-\tfrac{1}{8}\vec{\SSS}(W,\frak{t},\tau), & &\vec{\frak{q}}:=\vec{b}_{2}^{+}(W,\tau).
\end{align*}
\begin{enumerate}
    \item The following statements hold:
    \begin{enumerate}
        \item For each $\vec{\kappa}_{1}\in\K(Y_{1},\frak{s}_{1},\sigma_{1})$, we have that:
        \begin{enumerate}
            \item For each $\vec{\kappa}^{\wedge}_{0}\in\K^{\wedge}(Y_{0},\frak{s}_{0},\sigma_{0})$ we have that
            \[[\vec{\frak{q}}]+\vec{\kappa}_{1}\not\prec[\vec{\frak{p}}]+\vec{\kappa}^{\wedge}_{0}.\]
            \item There exists some $\vec{\kappa}^{\wedge}_{0}\in\K^{\wedge}(Y_{0},\frak{s}_{0},\sigma_{0})$ such that
            \[[\vec{\frak{q}}]+\vec{\kappa}_{1}\succeq[\vec{\frak{p}}]+\vec{\kappa}^{\wedge}_{0}.\]
        \end{enumerate}
        \item The following inequality holds:
        \[[\vec{\frak{q}}]+\vec{\ol{\kappa}}(Y_{1},\frak{s}_{1},\sigma_{1})\succeq[\vec{\frak{p}}]+\vec{\ul{\kappa}}^{\wedge}(Y_{0},\frak{s}_{0},\sigma_{0}).\]
    \end{enumerate}
    \item Suppose that $(Y_{0},\frak{s}_{0},\wh{\sigma}_{0})$ is Floer $\wedge^{2}$-$K_{G^{*}_{m}}$-split (in the sense of Definition \ref{def:floer_K_G_split}), let $\vec{\kappa}_{0}^{\wedge}\in\Q^{m}_{*}$ denote the unique element of $\K^{\wedge}(Y_{0},\frak{s}_{0},\sigma_{0})$, and let
    \[\vec{C}=\left\{
	\begin{array}{ll}
        \vec{e}_{0} & \mbox{if } b_{2}^{+}(W,\tau)_{0}\geq 1, \\
        \vec{0} & \mbox{if } b_{2}^{+}(W,\tau)_{0}=0.
	\end{array}\right.\]
    Then:
    \begin{enumerate}
        \item For each $\vec{\kappa}_{1}\in\K(Y_{1},\frak{s}_{1},\sigma_{1})$, we have that
        \[[\vec{\frak{q}}]+\vec{\kappa}_{1}\succeq[\vec{\frak{p}}]+\vec{\kappa}_{0}^{\wedge}+[\vec{C}].\]
        \item In particular, the following inequality holds:
        \[[\vec{\frak{q}}]+\vec{\ul{\kappa}}(Y_{1},\frak{s}_{1},\sigma_{1})\succeq[\vec{\frak{p}}]+\vec{\kappa}_{0}^{\wedge}+[\vec{C}].\]
    \end{enumerate}
    \item Suppose that:
    \begin{enumerate}[label=(R\arabic*),ref=R\arabic*]
        \item $b_{2}^{+}(W,\tau)_{j}$ is even for each $j=1,\dots,m-1$, \label{cond:even_b_2_^+_j_cobordism}
        \item There exists $x\in W^{\tau}$ such that $\tau$ acts semi-freely near $x$. \label{cond:even_b_2_^+_j_cobordism_fix}
    \end{enumerate}
    and let
    \[\vec{C}=\left\{
	\begin{array}{ll}
        \vec{0} & \mbox{if } b_{2}^{+}(W,\tau)_{0}\text{ is even}, \\
        -\vec{e}_{0} & \mbox{if } b_{2}^{+}(W,\tau)_{0}\text{ is odd}.
	\end{array}\right.\]
    Then:
    \begin{enumerate}
        \item For each $\vec{\kappa}_{1}\in\K(Y_{1},\frak{s}_{1},\sigma_{1})$, we have that:
        \begin{enumerate}
            \item For each $\vec{\kappa}_{0}\in\K(Y_{0},\frak{s}_{0},\sigma_{0})$ we have that
            \[[\vec{\frak{q}}]+\vec{\kappa}_{1}\not\prec[\vec{\frak{p}}]+\vec{\kappa}_{0}+[\vec{C}].\]
            \item There exists some $\vec{\kappa}_{0}\in\K(Y_{0},\frak{s}_{0},\sigma_{0})$ such that
            \[[\vec{\frak{q}}]+\vec{\kappa}_{1}\succeq[\vec{\frak{p}}]+\vec{\kappa}_{0}+[\vec{C}].\]
        \end{enumerate}
        \item The following inequality holds:
        \[[\vec{\frak{q}}]+\vec{\ol{\kappa}}(Y_{1},\frak{s}_{1},\sigma_{1})\succeq[\vec{\frak{p}}]+\vec{\ul{\kappa}}(Y_{0},\frak{s}_{0},\sigma_{0})+[\vec{C}].\]
    \end{enumerate}
    \item Suppose that Conditions \ref{cond:even_b_2_^+_j_cobordism}-\ref{cond:even_b_2_^+_j_cobordism_fix} hold, and furthermore suppose that $(Y_{0},\frak{s}_{0},\wh{\sigma}_{0})$ is Floer $K_{G^{*}_{m}}$-split. Let $\vec{\kappa}_{0}\in\Q^{m}_{*}$ denote the unique element of $\K(Y_{0},\frak{s}_{0},\sigma_{0})$, and let
    \[\vec{C}=\left\{
	\begin{array}{ll}
        2\vec{e}_{0} & \mbox{if } b_{2}^{+}(W,\tau)_{0}\text{ is even},\;\geq 2, \\
        \vec{e}_{0} & \mbox{if } b_{2}^{+}(W,\tau)_{0}\text{ is odd}, \\
        \vec{0} & \mbox{if } b_{2}^{+}(W,\tau)_{0}=0.
	\end{array}\right.\]
    Then:
    \begin{enumerate}
        \item For each $\vec{\kappa}_{1}\in\K(Y_{1},\frak{s}_{1},\sigma_{1})$, we have that
        \[[\vec{\frak{q}}]+\vec{\kappa}_{1}\succeq[\vec{\frak{p}}]+\vec{\kappa}_{0}+[\vec{C}].\]
        \item In particular, the following inequality holds:
        \[[\vec{\frak{q}}]+\vec{\ul{\kappa}}(Y_{1},\frak{s}_{1},\sigma_{1})\succeq[\vec{\frak{p}}]+\vec{\kappa}_{0}+[\vec{C}].\]
    \end{enumerate}
\end{enumerate}
\end{theorem}

Before we prove Theorem \ref{theorem:equivariant_cobordism}, we will need the following construction:
\begin{enumerate}
    \item Given $a\in\ZZ_{m}^{\times}$, let $\gamma_{(m;a)}:S^{2}\to S^{2}$ denote the $\ZZ_{m}$-action $(z,r)\mapsto(\omega_{m}^{a}z,r)$, where:
    \[S^{2}=\{(z,r)\in\CC\times\RR\;|\;|z|^{2}+r^{2}=1\}.\]
    \item Given $(a,b)\in(\ZZ_{m}^{\times})^{2}$, let $\tau_{(m;a,b)}:S^{2}\times S^{2}\to S^{2}\times S^{2}$ denote the homologically trivial pseudofree $\ZZ_{m}$-action $\tau_{(m;a,b)}:=\gamma_{(m;a)}\times\gamma_{(m;b)}$ with four fixed points.
    \item Given $a\in\ZZ_{m}^{\times}$, let $\tau_{(m;a)}:S^{2}\times S^{2}\to S^{2}\times S^{2}$ denote the homologically trivial $\ZZ_{m}$-action $\tau_{(m;a)}:=\gamma_{(m;a)}\times\id_{S^{2}\times S^{2}}$, with fixed-point set $S^{0}\times S^{2}\subset S^{2}\times S^{2}$.
\end{enumerate}

\begin{proof}[Proof of Theorem \ref{theorem:equivariant_cobordism}]
For (1) and (2), from Observation \ref{observation:complex_cobordism} we have a morphism
\begin{align*}
    &f_{\CC}:[(S^{0},b_{2}^{+}(W,\tau),\tfrac{1}{8}\S(W,\frak{t},\wh{\tau}))]\wedge(\wedge^{2}\SWF(Y_{0},\frak{s}_{0},\wh{\sigma}_{0}))\to\wedge^{2}\SWF(Y_{1},\frak{s}_{1},\wh{\sigma}_{1})
\end{align*}
which induces a $G^{*}_{p^{r}}$-homotopy equivalence on $\Pin(2)$-fixed point sets. The result then follows from Propositions \ref{prop:stable_k_invariants_pin(2)_fixed_point_homotopy_equivalence} and \ref{prop:stable_k_invariants_kg_split}, and Lemmas \ref{lemma:comparing_minima_subsets} and \ref{lemma:comparison_normal_wedge_invariants}.

For (3) and (4), suppose $b_{2}^{+}(W,\tau)_{j}\equiv 0\pmod{2}$ for all $j=0,\dots,m-1$. From Observation \ref{observation:real_cobordism} we have a morphism
\begin{align*}
    &f:[(S^{0},\tfrac{1}{2}b_{2}^{+}(W,\tau),\tfrac{1}{16}\S(W,\frak{t},\wh{\tau}))]\wedge\SWF(Y_{0},\frak{s}_{0},\wh{\sigma}_{0})\to\SWF(Y_{1},\frak{s}_{1},\wh{\sigma}_{1})
\end{align*}
which induces a $G^{*}_{p^{r}}$-homotopy equivalence on $\Pin(2)$-fixed point sets. In this case, the result follows from the same propositions and lemmas as above. 

Now suppose that $b_{2}^{+}(W,\tau)_{0}$ is odd, $b_{2}^{+}(W,\tau)_{j}\equiv 0\pmod{2}$ for all $j=1,\dots,m-1$, and that there exists $x\in W^{\tau}$ such that $\tau$ acts pseudo-freely near $x$.

If $x$ is an isolated fixed point, there exists $(a,b)\in(\ZZ_{m}^{\times})^{2}$ such that the equivariant connected sum
\[(W\# S^{2}\times S^{2},\frak{t}\#\frak{t}_{0},\wh{\tau\#\tau}{(m;a,b)})\]
is well-defined, where $\frak{t}_{0}$ denotes the unique spin structure on $S^{2}\times S^{2}$. Similarly if the dimension of $W^{\tau}$ near $x$ is equal to $2$, there exists $a\in\ZZ_{m}^{\times}$ such that the equivariant connected sum
\[(W\# S^{2}\times S^{2},\frak{t}\#\frak{t}_{0},\wh{\tau\#\tau}{(m;a)})\]
is well-defined. In either case, let $(W\# S^{2}\times S^{2},\frak{t}',\wh{\tau}')$ denote the resulting $\ZZ_{m}$-equivariant spin 4-manifold, which satisfies $b_{2}(W\# S^{2}\times S^{2},\tau')^{+}\equiv 0\pmod{2}$ for all $j=0,\dots,m-1$. Then we can apply the above inequality to $(W\# S^{2}\times S^{2},\frak{t}',\wh{\tau}')$, from which the result follows.
\end{proof}

\begin{theorem}
\label{theorem:equivariant_filling}
Suppose $(W,\frak{t},\wh{\tau})$ is a smooth, compact, $\ZZ_{m}$-equivariant spin 4-manifold with boundary a $\ZZ_{m}$-equivariant spin rational homology sphere $(Y,\frak{s},\wh{\sigma})$ satisfying $b_{1}(W)=0$. Furthermore, let $\vec{\frak{p}}$ and $\vec{\frak{q}}$ be as in Theorem \ref{theorem:equivariant_cobordism}, and let
\[\vec{C}=\left\{
	\begin{array}{ll}
        2\vec{e}_{0} & \mbox{if } b_{2}^{+}(W,\tau)_{0}\text{ is even},\;\geq 2,\text{ and \ref{eq:even_b_2^+_j} holds}, \\
        \vec{e}_{0} & \mbox{if } b_{2}^{+}(W,\tau)_{0}\text{ is even},\;\geq 2,\text{ and \ref{eq:even_b_2^+_j} does not hold, or} \\
        & \mbox{if } b_{2}^{+}(W,\tau)_{0}\text{ is odd}, \\
        \vec{0} & \mbox{if } b_{2}^{+}(W,\tau)_{0}=0,
	\end{array}\right.\]
where:
\begin{equation}
\label{eq:even_b_2^+_j}
b_{2}^{+}(W,\tau)_{j}\text{ is even for all }j=1,\dots,m-1. \tag{$\dagger$}
\end{equation}
Then:
\begin{enumerate}
    \item For each $\vec{\kappa}\in\K(Y,\frak{s},\sigma)$, we have that
    \[[\vec{\frak{q}}]+\vec{\kappa}\succeq[\vec{\frak{p}}]+[\vec{C}].\]
    \item In particular, the following inequality holds:
    \[[\vec{\frak{q}}]+\vec{\ul{\kappa}}(Y,\frak{s},\sigma)\succeq[\vec{\frak{p}}]+[\vec{C}].\]
\end{enumerate}
\end{theorem}

\begin{proof}
The proof is similar to the proof of Theorem \ref{theorem:equivariant_cobordism}, but using Observation \ref{observation:filling} instead of Observations \ref{observation:real_cobordism} and \ref{observation:complex_cobordism}.
\end{proof}

Finally, we have the corresponding statement in the closed case:

\begin{theorem}
\label{theorem:equivariant_closed}
Suppose $(W,\frak{t},\wh{\tau})$ is a smooth, closed, $\ZZ_{m}$-equivariant spin 4-manifold with $b_{1}(W)=0$, let $\vec{\frak{p}}$, $\vec{\frak{q}}$ be as in Theorems \ref{theorem:equivariant_cobordism} and \ref{theorem:equivariant_filling}, and let $\vec{C}$ be as in Theorem \ref{theorem:equivariant_filling}.
Then
\[[\vec{\frak{q}}]\succeq[\vec{\frak{p}}]+[\vec{C}].\]
\end{theorem}

\begin{proof}
Follows from Observation \ref{observation:closed}.
\end{proof}

Next, we state our Bryan-type inequalities for odd-type $2^{r}$-fold actions:

\begin{theorem}
\label{theorem:odd_2_r_filling}
Let $(Y,\frak{s},\wh{\sigma})$ be a $\ZZ_{2^{r}}$-equivariant spin rational homology sphere with $\wh{\sigma}$ of odd type. Suppose $(W,\frak{t},\wh{\tau})$ is a smooth, compact, $\ZZ_{2^{r}}$-equivariant spin filling of $(Y,\frak{s},\wh{\sigma})$ with $b_{1}(W)=0$. Let
\begin{align*}
    &p=-\sigma(W)/8, & &q=b_{2}^{+}(W), & &q_{j}=b_{2}^{+}(W,\tau)_{j},\;j=0,\dots,2^{r}-1,
\end{align*}
and suppose that:
\begin{enumerate}
    \item $q_{0}>0$,
    \item $\sum_{k=0}^{2^{r-a}-1}q_{2^{a}k}>0$ for each $a=0,\dots,r-1$.
    \item $\sum_{k=0}^{2^{a}-1}q_{(2k+1)2^{r-a-1}}>0$ for all $a\in\{0,\dots,r-2\}$.
    \item There exists some $\CC$-$G^{\odd}_{2^{r}}$-spectrum class $\X$ locally equivalent to $\wedge^{2}\SWF(Y,\frak{s},\wh{\sigma})$ such that
    \[\res^{\ZZ_{4}}_{1}\Big(\pi^{\st,\ZZ_{4}}_{-2(\sum_{k=0}^{2^{r-1}-1}q_{2k+1})\rho+p\nu}\big(\X^{\<j\mu^{2^{r-1}}\>}\big)\otimes\QQ\Big)=0.\]
\end{enumerate}
Then
\[q\geq p+r+1-|\vec{\kappa}|\qquad\text{ for any }\vec{\kappa}\in\K(Y,\frak{s},\sigma).\]
In particular, (4) holds if
\begin{enumerate}
    \item[(4')] $(Y,\frak{s},\wh{\sigma})$ is locally $\SWF$-$\<j\mu^{2^{r-1}}\>$-spherical, and
    \[\sum_{k=0}^{2^{r-1}-1}q_{2k+1}\neq p-2\kappa_{\KMT}(Y,\frak{s},\sigma^{2^{r-1}}),\]
\end{enumerate}
where $\kappa_{\KMT}(Y,\frak{s},\sigma^{2^{r-1}})$ denotes the invariant defined in \cite{KMT} with respect to the odd-type involution $\sigma^{2^{r-1}}$ on $Y$.
\end{theorem}

Before proving Theorem \ref{theorem:odd_2_r_filling}, we will make use of the following lemma:

\begin{lemma}
\label{lemma:KMT}
Let $(Y,\frak{s},\wh{\iota})$ be a $\ZZ_{2}$-equivariant spin rational homology sphere of odd type, and suppose $(Y,\frak{s},\wh{\iota})$ is locally $\SWF$-$\<j\mu\>$-spherical at level $d\in\QQ$. Then
\[\tfrac{1}{2}d=\kappa_{\KMT}(Y,\frak{s},\iota).\]
\end{lemma}

\begin{proof}
Since $\SWF(Y,\frak{s},\wh{\iota})$ is a $G^{\odd}_{2}$-spectrum class at level $0$, our assumption implies that
\[\SWF(Y,\frak{s},\wh{\iota})^{\<j\mu\>}\equiv_{\ell}[(S^{0},0,-\tfrac{1}{2}d_{1}\xi-\tfrac{1}{2}d_{2}\xi^{3})]\in\CCC_{\ZZ_{4},\CC}\]
for some $d_{1},d_{2}\in\QQ$ with $d_{1}+d_{2}=d$. Writing
\[\DSWF(Y,\frak{s},\wh{\iota}):=D(\SWF(Y,\frak{s},\wh{\iota})^{\<j\mu\>})=\SWF(Y,\frak{s},\wh{\iota})^{\<j\mu\>}\wedge(\SWF(Y,\frak{s},\wh{\iota})^{\<j\mu\>})^{\dagger},\]
we see that
\[\DSWF(Y,\frak{s},\wh{\iota})\equiv_{\ell}[(S^{0},0,-\tfrac{1}{2}d)]\in\CCC_{\ZZ_{4},\CC,\sym},\]
and hence
\[\kappa_{\KMT}(Y,\frak{s},\iota)=k_{\KMT}(S^{0},0,-\tfrac{1}{2}d) = k(S^{0})+\tfrac{1}{2}d = \tfrac{1}{2}d.\]
\end{proof}

\begin{proof}[Proof of Theorem \ref{theorem:odd_2_r_filling}]
Follows from Lemma \ref{lemma:KMT}, Proposition \ref{prop:stable_k_invariants_2_r_odd} and Lemma \ref{lemma:comparison_normal_wedge_invariants} applied to Observation \ref{observation:filling}.
\end{proof}

In the case where $r=1$, we can get slightly better inequalities than the one implied by Theorem \ref{theorem:odd_2_r_filling}, depending on the parities of $q_{0}$ and $q_{1}$:

\begin{theorem}
\label{theorem:odd_2_filling}
Let $(Y,\frak{s},\wh{\iota})$ be a $\ZZ_{2}$-equivariant spin rational homology sphere with $\wh{\iota}$ of odd type, and suppose $(W,\frak{t},\wh{\tau})$ is a smooth, compact, $\ZZ_{2}$-equivariant spin filling of $(Y,\frak{s},\wh{\iota})$ such that $b_{1}(W)=0$ and $W^{\tau}\neq\emptyset$. Let
\begin{align*}
    &p=-\sigma(W)/8, & &q=b_{2}^{+}(W), & &q_{0}=b_{2}^{+}(W,\tau)_{0}, & &q_{1}=b_{2}^{+}(W,\tau)_{1},
\end{align*}
and suppose that $q_{0},q_{1}>0$. Then
\[q\geq p-\wt{\kappa}(Y,\frak{s},\iota)+C,\]
where:
\begingroup
\renewcommand{\arraystretch}{1.5} 
\[C=\left\{
		\begin{array}{ll}
			4 & \mbox{if } q_{0},q_{1}\text{ both even, and there exists }\X\in\CCC_{G^{\odd}_{2},\CC}\text{ with }\X\equiv_{\ell}\SWF(Y,\frak{s},\wh{\iota}) \\
			&\qquad\text{ such that }\res^{\ZZ_{4}}_{1}\Big(\pi^{\st,\ZZ_{4}}_{-q_{1}\rho+\frac{p}{2}\nu}\big(\X^{\<j\mu\>}\big)\otimes\QQ\Big)=0, \\
			3 & \mbox{if } q_{0}\text{ odd},\,q_{1}\text{ even, and there exists }\X\in\CCC_{G^{\odd}_{2},\CC}\text{ with }\X\equiv_{\ell}\SWF(Y,\frak{s},\wh{\iota}) \\
			&\qquad\text{ such that }\res^{\ZZ_{4}}_{1}\Big(\pi^{\st,\ZZ_{4}}_{-q_{1}\rho+\frac{p}{2}\nu}\big(\X^{\<j\mu\>}\big)\otimes\QQ\Big)=0,\text{ or} \\
            & \mbox{if } q_{0}\text{ even, }q_{1}\text{ odd, and there exists }\X\in\CCC_{G^{\odd}_{2},\CC}\text{ with }\X\equiv_{\ell}\SWF(Y,\frak{s},\wh{\iota}) \\
            &\qquad\text{ such that }\res^{\ZZ_{4}}_{1}\Big(\pi^{\st,\ZZ_{4}}_{-(q_{1}+1)\rho+\frac{p}{2}\nu}\big(\X^{\<j\mu\>}\big)\otimes\QQ\Big)=0, \\
            2 & \mbox{if } q_{0}, q_{1}\text{ both odd, or }\\
            & \mbox{if }q_{0}, q_{1}\text{ are of any parity, and there exists }\X\in\CCC_{G^{\odd}_{2},\CC}\text{ with }\X\equiv_{\ell} \\
            &\qquad\wedge^{2}\SWF(Y,\frak{s},\wh{\iota})\text{ such that }\res^{\ZZ_{4}}_{1}\Big(\pi^{\st,\ZZ_{4}}_{-2q_{1}\rho+p\nu}\big(\X^{\<j\mu\>}\big)\otimes\QQ\Big)=0.
         \end{array}
	\right.\]
\endgroup
In particular if $(Y,\frak{s},\wh{\iota})$ is locally $\SWF$-$\<j\mu\>$-spherical, then $C=C'$, where:

\begingroup
\renewcommand{\arraystretch}{1.5} 
\[C'=\left\{
		\begin{array}{ll}
			4 & \mbox{if } q_{0},q_{1}\text{ both even, and }q_{1}\neq p-2\kappa_{\KMT}(Y,\frak{s},\iota), \\
			3 & \mbox{if } q_{0}\text{ odd},\,q_{1}\text{ even, and }q_{1}\neq p-2\kappa_{\KMT}(Y,\frak{s},\iota),\text{ or} \\
            & \mbox{if } q_{0}\text{ even, }q_{1}\text{ odd, and }q_{1}\neq p-2\kappa_{\KMT}(Y,\frak{s},\iota)-1, \\
            2 & \mbox{if } q_{0}, q_{1}\text{ both odd, or }\\
            & \mbox{if }q_{0}, q_{1}\text{ are of any parity, and }q_{1}\neq p-2\kappa_{\KMT}(Y,\frak{s},\iota).
		\end{array}
	\right.\]
\endgroup
\end{theorem}

The following lemma will be useful for the proof of the above theorem:

\begin{lemma}
\label{lemma:connected_sum_S2_times_S2}
There exist odd-type involutions $\tau_{0}$ and $\tau_{1}$ on $S^{2}\times S^{2}$ with non-empty fixed point sets, such that
\begin{align*}
    &b_{2}^{+}(S^{2}\times S^{2},\tau_{0})_{0} = 1, & &b_{2}^{+}(S^{2}\times S^{2},\tau_{0})_{1} = 0, \\
    &b_{2}^{+}(S^{2}\times S^{2},\tau_{1})_{0} = 0, & &b_{2}^{+}(S^{2}\times S^{2},\tau_{1})_{1} = 1.
\end{align*}
\end{lemma}

\begin{proof}
Let $f:S^{2}\to S^{2}$ be the orientation preserving involution induced by $z\mapsto -z$ on $\CC$, which performs a rotation of $\pi$ about the axis which goes through $0,\infty\in S^{2}\cong\CC^{+}$, and whose fixed point set is precisely $\{0,\infty\}$. Let $g:S^{2}\to S^{2}$ be the orientation-reversing involution induced by $z\mapsto 1/\ol{z}$ on $\CC$, which reflects the northern and southern hemispheres, and whose fixed point set is the unit circle $S^{1}\subset S^{2}$. We then define $\tau_{j}:S^{2}\times S^{2}\to S^{2}\times S^{2}$, $j=0,1$ by
\begin{align*}
    &\tau_{0}:=f\times\id_{S^{2}}, & &\tau_{1}:=g\times g.
\end{align*}
We see that the fixed point set of $\tau_{0}$ is the disjoint union of the two 2-spheres
\[(S^{2}\times S^{2})^{\tau_{0}}=\{0\}\times S^{2},\{\infty\}\times S^{2}\subset S^{2}\times S^{2},\]
and the fixed point set of $\eta_{1}$ is the 2-torus 
\[(S^{2}\times S^{2})^{\tau_{1}}=S^{1}\times S^{1}\subset S^{2}\times S^{2}.\]
Therefore $\tau_{0}$ and $\tau_{1}$ have non-empty fixed point sets of codimension 2, and hence admit odd-type spin lifts with respect to the unique spin structure $\frak{t}_{0}$ on $S^{2}\times S^{2}$. 

One can see that $\tau_{0}$ is isotopic to the identity via the map $S^{2}\times I\to S^{2}\times I$ induced by $z\mapsto e^{2\pi it}$, and hence acts trivially on homology. Finally since $\tau_{1}=g\times g$ where $g_{*}([S^{2}])=-[S^{2}]$, it follows that $\tau_{1}$ acts by $-1$ on $H_{2}(S^{2}\times S^{2})\cong H_{2}(S^{2})\otimes H_{2}(S^{2})$.
\end{proof}

\begin{remark}
Note that $\tau_{0}$ from Lemma \ref{lemma:connected_sum_S2_times_S2} coincides with the map $\tau_{(2;1)}$ from the discussion preceding the proof of Theorem \ref{theorem:equivariant_cobordism}.
\end{remark}

We are now ready to prove Theorem \ref{theorem:odd_2_filling}:

\begin{proof}[Proof of Theorem \ref{theorem:odd_2_filling}]
First suppose that $q_{0},q_{1}$ are both even and nonzero. By Observation \ref{observation:filling}, we have a morphism
\[f:[(S^{0},\tfrac{1}{2}b_{2}^{+}(W,\tau),\tfrac{1}{16}\S(W,\frak{t},\wh{\tau}))]\to\SWF(Y,\frak{s},\wh{\sigma}).\]
Combining Proposition \ref{cor:stable_k_invariants_2_odd} with Lemma \ref{lemma:KMT}, it follows that
\begin{align}
\label{eq:0_1_both_even}
	&\tfrac{1}{2}q\geq \tfrac{1}{2}p-|\wt{k}\,^{\st}(\SWF(Y,\frak{s},\wh{\iota}))|+2 & &\iff q\geq p-\wt{\kappa}(Y,\frak{s},\iota)+4,
\end{align}
provided that there exists $\X\in\CCC_{G^{\odd}_{2},\CC}$ with $\X\equiv_{\ell}\SWF(Y,\frak{s},\wh{\iota})$ such that
\[\res^{\ZZ_{4}}_{1}\Big(\pi^{\st,\ZZ_{4}}_{-q_{1}\rho+\frac{p}{2}\nu}\big(\X^{\<j\mu\>}\big)\otimes\QQ\Big)=0.\]
If one of $q_{0}$, $q_{1}$ is odd, we can replace $W$ with spin equivariant connect sums of $W$ and $S^{2}\times S^{2}$. More precisely, we choose connected components $\Sigma_{W}\subset W^{\tau}$ and $\Sigma_{j}\subset(S^{2}\times S^{2})^{\tau_{j}}$, $j=0,1$, where $\tau_{0}$, $\tau_{1}$ are as in Lemma \ref{lemma:connected_sum_S2_times_S2}, as well as orientations on $\Sigma_{W}$, $\Sigma_{0}$, $\Sigma_{1}$. With these choices fixed, as in Section \ref{subsec:equivariant_connected_sums} we can define the following equivariant connect sums:
\begin{align*}
    &(W\# S^{2}\times S^{2},\frak{t}\#\frak{t}_{0},\wh{\tau\#\tau_{0}}), &
    &(W\# S^{2}\times S^{2},\frak{t}\#\frak{t}_{0},\wh{\tau\#\tau_{1}}),
\end{align*}
where $\frak{t}_{0}$ denotes the unique spin structure on $S^{2}\times S^{2}$.

For example, suppose $q_{0}$ is odd and $q_{1}$ is even, $q_{1}\neq 0$. Then by Lemma \ref{lemma:connected_sum_S2_times_S2}, we have that
\begin{align*}
    &b_{2}^{+}(W\# S^{2}\times S^{2},\tau\#\tau_{0})_{0}=b_{2}^{+}(W,\tau)_{0}+1 & &b_{2}^{+}(W\# S^{2}\times S^{2},\tau\#\tau_{0})_{1}=b_{2}^{+}(W,\tau)_{1}, \\
    &b_{2}^{+}(W\# S^{2}\times S^{2}) = b_{2}^{+}(W)+1, & &\sigma(W\# S^{2}\times S^{2}) = \sigma(W).
\end{align*}
Hence we can apply inequality \ref{eq:0_1_both_even} to the triple $(W\# S^{2}\times S^{2},\frak{t}\#\frak{t}_{0},\wh{\tau\#\tau_{0}})$, and obtain
\begin{equation*}
    q\geq p-\wt{\kappa}(Y,\frak{s},\iota)+3,
\end{equation*}
provided that there exists $\X\in\CCC_{G^{\odd}_{2},\CC}$ with $\X\equiv_{\ell}\SWF(Y,\frak{s},\wh{\iota})$ such that
\[\res^{\ZZ_{4}}_{1}\Big(\pi^{\st,\ZZ_{4}}_{-q_{1}\rho+\frac{p}{2}\nu}\big(\X^{\<j\mu\>}\big)\otimes\QQ\Big)=0.\]
Similarly, suppose $q_{0}$ is even, $q_{0}\neq 0$, and $q_{1}$ is odd. By replacing $(W,\frak{t},\wh{\tau})$ with $(W\# S^{2}\times S^{2},\frak{t}\#\frak{t}_{0},\wh{\tau\#\tau_{1}})$, we obtain
\begin{equation*}
    q\geq p-\wt{\kappa}(Y,\frak{s},\iota)+3
\end{equation*}
provided that there exists $\X\in\CCC_{G^{\odd}_{2},\CC}$ with $\X\equiv_{\ell}\SWF(Y,\frak{s},\wh{\iota})$ such that
\[\res^{\ZZ_{4}}_{1}\Big(\pi^{\st,\ZZ_{4}}_{-(q_{1}+1)\rho+\frac{p}{2}\nu}\big(\X^{\<j\mu\>}\big)\otimes\QQ\Big)=0.\]
The other two cases follow from Theorem \ref{theorem:manolescu} and Theorem \ref{theorem:odd_2_filling}, respectively. The statement in the locally $\SWF$-$\<j\mu\>$-spherical setting follows from Lemma \ref{lemma:KMT}.
\end{proof}

We now turn towards the case of odd prime powers. The next three theorems follow from Theorems \ref{theorem:equivariant_cobordism} -- \ref{theorem:equivariant_closed} and Proposition \ref{prop:stable_k_invariants_odd_prime_powers}:

\begin{theorem}
\label{theorem:equivariant_cobordism_p^r}
Let $p^{r}$ be an odd prime power, let $(Y_{0},\frak{s}_{0},\wh{\sigma}_{0})$, $(Y_{1},\frak{s}_{1},\wh{\sigma}_{1})$ be $\ZZ_{p^{r}}$-equivariant spin rational homology spheres, and suppose that $(W,\frak{t},\wh{\tau})$ is a smooth spin $\ZZ_{p^{r}}$-equivariant cobordism from $(Y_{0},\frak{s}_{0},\wh{\sigma}_{0})$ to $(Y_{1},\frak{s}_{1},\wh{\sigma}_{1})$ with $b_{1}(W)=0$.
\begin{enumerate}
    \item The following statements hold:
    \begin{enumerate}
        \item For each $\vec{\kappa}_{1}=(\kappa_{0}^{1},\kappa_{1}^{1})\in\K^{\pi}(Y_{1},\frak{s}_{1},\sigma_{1})$, we have that:
        \begin{enumerate}
            \item For each $\vec{\kappa}^{\wedge}_{0}=(\kappa_{0}^{\wedge,0},\kappa_{1}^{\wedge,0})\in\K^{\wedge,\pi}(Y_{0},\frak{s}_{0},\sigma_{0})$, the following implications hold:
            \begin{align*}
                &b_{2}^{+}(W,\tau)_{0}+\kappa_{0}^{1}\le-\tfrac{1}{8}\SSS(W,\tau)_{0}+\kappa_{0}^{\wedge,0}\\
                \implies &b_{2}^{+}(X)-b_{2}^{+}(W,\tau)_{0}+\kappa_{1}^{1}\geq -\tfrac{1}{8}\sigma(W)+\tfrac{1}{8}\SSS(W,\tau)_{0}+\kappa_{1}^{\wedge,0},\text{ and } \\
                &b_{2}^{+}(X)-b_{2}^{+}(W,\tau)_{0}+\kappa_{1}^{1}\le -\tfrac{1}{8}\sigma(W)+\tfrac{1}{8}\SSS(W,\tau)_{0}+\kappa_{1}^{\wedge,0} \\
                \implies &b_{2}^{+}(W,\tau)_{0}+\kappa_{0}^{1}\geq-\tfrac{1}{8}\SSS(W,\tau)_{0}+\kappa_{0}^{\wedge,0}.
            \end{align*}
            \item There exists some $\vec{\kappa}^{\wedge}_{0}=(\kappa_{0}^{\wedge,0},\kappa_{1}^{\wedge,0})\in\K^{\wedge,\pi}(Y_{0},\frak{s}_{0},\sigma_{0})$ such that
            \begin{align*}
                &b_{2}^{+}(W,\tau)_{0}+\kappa_{0}^{1}\geq-\tfrac{1}{8}\SSS(W,\tau)_{0}+\kappa_{0}^{\wedge,0},\\
                &b_{2}^{+}(X)-b_{2}^{+}(W,\tau)_{0}+\kappa_{1}^{1}\geq -\tfrac{1}{8}\sigma(W)+\tfrac{1}{8}\SSS(W,\tau)_{0}+\kappa_{1}^{\wedge,0}.
            \end{align*}
        \end{enumerate}
        \item In particular, the following inequalities hold:
        \begin{align*}
            &b_{2}^{+}(W,\tau)_{0}+\ol{\kappa}_{0}^{\pi}(Y_{1},\frak{s}_{1},\sigma_{1})\geq-\tfrac{1}{8}\SSS(W,\tau)_{0}+\ul{\kappa}_{0}^{\wedge,\pi}(Y_{0},\frak{s}_{0},\sigma_{0}),\\
            &b_{2}^{+}(X)-b_{2}^{+}(W,\tau)_{0}+\ol{\kappa}_{1}^{\pi}(Y_{1},\frak{s}_{1},\sigma_{1})\geq -\tfrac{1}{8}\sigma(W)+\tfrac{1}{8}\SSS(W,\tau)_{0}+\ul{\kappa}_{1}^{\wedge,\pi}(Y_{0},\frak{s}_{0},\sigma_{0}).
        \end{align*}
    \end{enumerate}
    \item Suppose that $(Y_{0},\frak{s}_{0},\wh{\sigma}_{0})$ is Floer $\wedge^{2}$-$K_{G^{*}_{p^{r}}}$-split, let $\vec{\kappa}_{0}^{\wedge}=(\kappa_{0}^{\wedge,0},\kappa_{1}^{\wedge,0})\in\QQ^{2}$ denote the unique element of $\K^{\wedge,\pi}(Y_{0},\frak{s}_{0},\sigma_{0})$, and let
    \[C=\left\{
	\begin{array}{ll}
        1 & \mbox{if } b_{2}^{+}(W,\tau)_{0}\geq 1, \\
        0 & \mbox{if } b_{2}^{+}(W,\tau)_{0}=0.
	\end{array}\right.\]
    Then:
    \begin{enumerate}
        \item For each $\vec{\kappa}_{1}=(\kappa_{0}^{1},\kappa_{1}^{1})\in\K^{\pi}(Y_{1},\frak{s}_{1},\sigma_{1})$ we have that:
        \begin{align*}
            &b_{2}^{+}(W,\tau)_{0}+\kappa_{0}^{1}\geq-\tfrac{1}{8}\SSS(W,\tau)_{0}+\kappa_{0}^{\wedge,0}+C,\\
            &b_{2}^{+}(X)-b_{2}^{+}(W,\tau)_{0}+\kappa_{1}^{1}\geq -\tfrac{1}{8}\sigma(W)+\tfrac{1}{8}\SSS(W,\tau)_{0}+\kappa_{1}^{\wedge,0}.
        \end{align*}
        \item In particular, the following inequalies hold:
        \begin{align*}
            &b_{2}^{+}(W,\tau)_{0}+\ul{\kappa}_{0}^{\pi}(Y_{1},\frak{s}_{1},\sigma_{1})\geq-\tfrac{1}{8}\SSS(W,\tau)_{0}+\kappa_{0}^{\wedge,0}+C,\\
            &b_{2}^{+}(X)-b_{2}^{+}(W,\tau)_{0}+\ul{\kappa}_{1}^{\pi}(Y_{1},\frak{s}_{1},\sigma_{1})\geq -\tfrac{1}{8}\sigma(W)+\tfrac{1}{8}\SSS(W,\tau)_{0}+\kappa_{1}^{\wedge,0}.
        \end{align*}
    \end{enumerate}
    \item Suppose Conditions \ref{cond:even_b_2_^+_j_cobordism}-\ref{cond:even_b_2_^+_j_cobordism_fix} hold, and let
    \[C=\left\{
	\begin{array}{ll}
        0 & \mbox{if } b_{2}^{+}(W,\tau)_{0}\text{ is even}, \\
        -1 & \mbox{if } b_{2}^{+}(W,\tau)_{0}\text{ is odd}.
	\end{array}\right.\]
    Then:
    \begin{enumerate}
        \item For each $\vec{\kappa}_{1}=(\kappa_{0}^{1},\kappa_{1}^{1})\in\K^{\pi}(Y_{1},\frak{s}_{1},\sigma_{1})$, we have that:
        \begin{enumerate}
            \item For each $\vec{\kappa}_{0}=(\kappa_{0}^{0},\kappa_{1}^{0})\in\K^{\pi}(Y_{0},\frak{s}_{0},\sigma_{0})$, the following implications hold:
            \begin{align*}
                &b_{2}^{+}(W,\tau)_{0}+\kappa_{0}^{1}\le-\tfrac{1}{8}\SSS(W,\tau)_{0}+\kappa_{0}^{0}+C\\
                \implies &b_{2}^{+}(X)-b_{2}^{+}(W,\tau)_{0}+\kappa_{1}^{1}\geq -\tfrac{1}{8}\sigma(W)+\tfrac{1}{8}\SSS(W,\tau)_{0}+\kappa_{1}^{0},\text{ and } \\
                &b_{2}^{+}(X)-b_{2}^{+}(W,\tau)_{0}+\kappa_{1}^{1}\le -\tfrac{1}{8}\sigma(W)+\tfrac{1}{8}\SSS(W,\tau)_{0}+\kappa_{1}^{0} \\
                \implies &b_{2}^{+}(W,\tau)_{0}+\kappa_{0}^{1}\geq-\tfrac{1}{8}\SSS(W,\tau)_{0}+\kappa_{0}^{0}+C.
            \end{align*}
            \item There exists some $\vec{\kappa}_{0}=(\kappa_{0}^{0},\kappa_{1}^{0})\in\K^{\pi}(Y_{0},\frak{s}_{0},\sigma_{0})$ such that
            \begin{align*}
                &b_{2}^{+}(W,\tau)_{0}+\kappa_{0}^{1}\geq-\tfrac{1}{8}\SSS(W,\tau)_{0}+\kappa_{0}^{0}+C,\\
                &b_{2}^{+}(X)-b_{2}^{+}(W,\tau)_{0}+\kappa_{1}^{1}\geq -\tfrac{1}{8}\sigma(W)+\tfrac{1}{8}\SSS(W,\tau)_{0}+\kappa_{1}^{0}.
            \end{align*}
        \end{enumerate}
        \item In particular, the following inequalities hold:
        \begin{align*}
            &b_{2}^{+}(W,\tau)_{0}+\ol{\kappa}_{0}^{\pi}(Y_{1},\frak{s}_{1},\sigma_{1})\geq-\tfrac{1}{8}\SSS(W,\tau)_{0}+\ul{\kappa}_{0}^{\pi}(Y_{0},\frak{s}_{0},\sigma_{0})+C,\\
            &b_{2}^{+}(X)-b_{2}^{+}(W,\tau)_{0}+\ol{\kappa}_{1}^{\pi}(Y_{1},\frak{s}_{1},\sigma_{1})\geq -\tfrac{1}{8}\sigma(W)+\tfrac{1}{8}\SSS(W,\tau)_{0}+\ul{\kappa}_{1}^{\pi}(Y_{0},\frak{s}_{0},\sigma_{0}).
        \end{align*}
    \end{enumerate}
    \item Suppose that Conditions \ref{cond:even_b_2_^+_j_cobordism}-\ref{cond:even_b_2_^+_j_cobordism_fix} hold, and furthermore suppose that $(Y_{0},\frak{s}_{0},\wh{\sigma}_{0})$ is Floer $K_{G^{*}_{p^{r}}}$-split. Let $\vec{\kappa}_{0}=(\kappa_{0}^{0},\kappa_{1}^{0})\in\QQ^{2}$ denote the unique element of $\K^{\pi}(Y_{0},\frak{s}_{0},\sigma_{0})$, and let
    \[C=\left\{
	\begin{array}{ll}
        2 & \mbox{if } b_{2}^{+}(W,\tau)_{0}\text{ is even},\;\geq 2, \\
        1 & \mbox{if } b_{2}^{+}(W,\tau)_{0}\text{ is odd}, \\
        0 & \mbox{if } b_{2}^{+}(W,\tau)_{0}=0.
	\end{array}\right.\]
    Then:
    \begin{enumerate}
        \item For each $\vec{\kappa}_{1}=(\kappa_{0}^{1},\kappa_{1}^{1})\in\K^{\pi}(Y_{1},\frak{s}_{1},\sigma_{1})$ we have that:
        \begin{align*}
            &b_{2}^{+}(W,\tau)_{0}+\kappa_{0}^{1}\geq-\tfrac{1}{8}\SSS(W,\tau)_{0}+\kappa_{0}^{0}+C,\\
            &b_{2}^{+}(X)-b_{2}^{+}(W,\tau)_{0}+\kappa_{1}^{1}\geq -\tfrac{1}{8}\sigma(W)+\tfrac{1}{8}\SSS(W,\tau)_{0}+\kappa_{1}^{0}.
        \end{align*}
        \item In particular, the following inequalities hold:
        \begin{align*}
            &b_{2}^{+}(W,\tau)_{0}+\ul{\kappa}_{0}^{\pi}(Y_{1},\frak{s}_{1},\sigma_{1})\geq-\tfrac{1}{8}\SSS(W,\tau)_{0}+\kappa_{0}^{0}+C,\\
            &b_{2}^{+}(X)-b_{2}^{+}(W,\tau)_{0}+\ul{\kappa}_{1}^{\pi}(Y_{1},\frak{s}_{1},\sigma_{1})\geq -\tfrac{1}{8}\sigma(W)+\tfrac{1}{8}\SSS(W,\tau)_{0}+\kappa_{1}^{0}.
        \end{align*}
    \end{enumerate}
\end{enumerate}
\end{theorem}

\begin{theorem}
\label{theorem:equivariant_filling_p^r}
Let $p^{r}$ be an odd prime power, let $(Y,\frak{s},\sigma)$ be a $\ZZ_{p^{r}}$-equivariant spin rational homology sphere, let $(W,\frak{t},\tau)$ be a smooth $\ZZ_{p^{r}}$-equivariant spin filling of $Y$ such that $b_{1}(W)=0$, and let
\[C=\left\{
	\begin{array}{ll}
        2 & \mbox{if } b_{2}^{+}(W,\tau)_{0}\text{ is even},\;\geq 2,\text{ and \ref{eq:even_b_2^+_j_p^r} holds}, \\
        1 & \mbox{if } b_{2}^{+}(W,\tau)_{0}\text{ is even},\;\geq 2,\text{ and \ref{eq:even_b_2^+_j_p^r} does not hold, or} \\
        & \mbox{if } b_{2}^{+}(W,\tau)_{0}\text{ is odd}, \\
        0 & \mbox{if } b_{2}^{+}(W,\tau)_{0}=0,
	\end{array}\right.\]
where:
\begin{equation}
\label{eq:even_b_2^+_j_p^r}
b_{2}^{+}(W,\tau)_{j}\text{ is even for all }j=1,\dots,p^{r}-1. \tag{$\dagger\dagger$}
\end{equation}
Then:
\begin{enumerate}
    \item For each $\vec{\kappa}=(\kappa_{0},\kappa_{1})\in\K^{\pi}(Y,\frak{s},\sigma)$ the following inequalities hold:
    \begin{align*}
        &b_{2}^{+}(W,\tau)_{0}+\kappa_{0}\geq-\tfrac{1}{8}\SSS(W,\tau)_{0}+C,\\
        &b_{2}^{+}(W)-b_{2}^{+}(W,\tau)_{0}+\kappa_{1}\geq -\tfrac{1}{8}\sigma(W)+\tfrac{1}{8}\SSS(W,\tau)_{0}.
    \end{align*}
    \item In particular:
    \begin{align*}
        &b_{2}^{+}(W,\tau)_{0}+\vec{\ul{\kappa}}_{0}^{\pi}(Y,\frak{s},\sigma)\geq-\tfrac{1}{8}\SSS(W,\tau)_{0}+C,\\
        &b_{2}^{+}(W)-b_{2}^{+}(X,\tau)_{0}+\vec{\ul{\kappa}}_{1}^{\pi}(Y,\frak{s},\sigma)\geq -\tfrac{1}{8}\sigma(W)+\tfrac{1}{8}\SSS(W,\tau)_{0}.
    \end{align*}
\end{enumerate}
\end{theorem}

\begin{theorem}
\label{theorem:equivariant_closed_p^r}
Let $p^{r}$ be an odd prime power. Suppose $(W,\frak{t},\wh{\tau})$ is a smooth, closed, $\ZZ_{p^{r}}$-equivariant spin 4-manifold with $b_{1}(W)=0$, and let $C$ be as in Theorem \ref{theorem:equivariant_filling_p^r}.
Then
\begin{align*}
    &b_{2}^{+}(W,\tau)_{0}\geq-\tfrac{1}{8}\SSS(W,\tau)_{0}+C,\\
    &b_{2}^{+}(W)-b_{2}^{+}(X,\tau)_{0}\geq -\tfrac{1}{8}\sigma(W)+\tfrac{1}{8}\SSS(W,\tau)_{0}.
\end{align*}
\end{theorem}

\bigskip

%% file: calculations_knot_invariants.tex
\section{Calculations and Knot Invariants}
\label{sec:calculations_knot_invariants}

In this section we calculate the $G^{*}_{m}$-equivariant Seiberg--Witten Floer stable homotopy types of some $\ZZ_{m}$-equivariant spin rational homology spheres, as well as their corresponding equivariant $\kappa$-invariants. We also discuss the construction of knot concordance invariants via taking cyclic branched covers.

\subsection{Seiberg-Witten Floer Minimal Spaces}
\label{subsec:SWF_minimal}

We start with the following definition:

\begin{definition}
\label{def:SWF_minimal}
Let $(Y,\frak{s})$ be a spin rational homology sphere. We say that $(Y,\frak{s})$ is \emph{Seiberg--Witten Floer minimal} (or just \emph{SWF-minimal}) if there exists a metric $g$ on $Y$ such that $(Y,\frak{s},g)$ admits no irreducible solutions to the Seiberg--Witten equations.

If $(Y,\frak{s},\wh{\sigma})$ is a $\ZZ_{m}$-equivariant spin rational homology sphere, we say that $(Y,\frak{s},\wh{\sigma})$ is \emph{equivariantly SWF-minimal} if $(Y,\frak{s})$ is SWF-minimal with respect to a $\ZZ_{m}$-equivariant metric $g$.
\end{definition}

\begin{example}
Suppose $(Y,\frak{s},\wh{\sigma})$ is a $\ZZ_{m}$-equivariant spin rational homology sphere which admits a $\ZZ_{m}$-equivariant metric $g$ of positive scalar curvature. Then $(Y,\frak{s},\wh{\sigma})$ is equivariantly SWF-minimal (see \cite{KMbook}, p.448).
\end{example}

The notion of an SWF-minimal space is related to that of a \emph{minimal L-space}, which was first coined by Lin-Lipnowski (\cite{LinLipLS}). A minimal L-space is a rational homology sphere $Y$ which admits a metric $g$ such that $(Y,\frak{s},g)$ admits no irreducible Seiberg-Witten solutions for \emph{any} $\spinc$-structure $\frak{s}$ on $Y$. In particular if $Y$ is a minimal L-space, then $(Y,\frak{s})$ is SWF-minimal for every spin structure $\frak{s}$ on $Y$.

It is an open question whether every (Heegaard-Floer) L-space is a minimal L-space. It is known that all elliptic manifolds and the Hantzsche-Wendt manifold $\mathsf{HW}$ (the unique flat rational homology sphere) are minimal L-spaces. More recently, Lin and Lin-Lipnowski have shown that this class includes a number of small hyperbolic 3-manifolds (\cite{LinLipLS}), the Seifert-Weber dodecahedral space $\mathsf{SW}$ (\cite{LinLipSW}), and all $\mathsf{Solv}$ rational homology spheres (\cite{LinSolv}).

\begin{proposition}
\label{prop:SWF_minimal}
If $(Y,\frak{s},\wh{\sigma})$ is an equivariantly SWF-minimal $\ZZ_{m}$-equivariant spin rational homology sphere, then
\begin{equation}
\label{eq:SWF_minimal_spectrum}
    \SWF(Y,\frak{s},\wh{\sigma})=[(S^{0},0,\tfrac{1}{2}n(Y,\frak{s},\wh{\sigma},g))]\in\CCC_{G^{*}_{m},\CC}.
\end{equation}
Consequently:
\begin{equation}
\label{eq:SWF_minimal_kappa}
    \K(Y,\frak{s},\wh{\sigma})=\K^{\wedge}(Y,\frak{s},\wh{\sigma})=\big\{-\big[\DDD^{*}(\vec{n}(Y,\frak{s},\wh{\sigma},g))\big]\big\}\subset\Q^{m}_{*}.
\end{equation}
In particular, $(Y,\frak{s},\wh{\sigma})$ is $\SWF$-$\Pin(2)$-surjective.
\end{proposition}

\begin{proof}
This essentially follows from the argument for the corresponding statement in the $S^{1}$-equivariant setting as in \cite{Man03}.
\end{proof}

The following corollary follows immediately from Propositions \ref{prop:correction_term_involutions} and \ref{prop:SWF_minimal}:

\begin{corollary}
\label{cor:SWF_minimal_involutions}
If $(Y,\frak{s},\wh{\iota})$ is an equivariant SWF-minimal $\ZZ_{2}$-equivariant spin rational homology sphere, then:
\[\K(Y,\frak{s},\wh{\iota})=\K^{\wedge}(Y,\frak{s},\wh{\iota})=\{[(\kappa(Y,\frak{s}),0)]\}\subset\Q^{2}_{*}.\]
In particular, if $\wh{\iota}$ is of odd type then $\wt{\kappa}(Y,\frak{s},\iota)=\kappa(Y,\frak{s})$.
\end{corollary}

\begin{example}
\label{ex:SWF_minimal_double_branched_cover}
Let $(Y,\frak{s},\iota)$ be an $\SWF$-minimal $\ZZ_{2}$-equivariant spin rational homology sphere of odd type such that $\iota:Y\to Y$ exhibits $Y$ as a double branched cover over a knot $K\subset S^{3}$. Then
\[\wt{\kappa}(Y,\frak{s},\iota)=\kappa(Y,\frak{s})=-n(Y,\frak{s},g)=\delta(Y,\frak{s})=-h(Y,\frak{s})=-\tfrac{1}{8}\sigma(K),\]
where $\sigma(K)$ denotes the signature of $K$. The first and second equalities follow from Proposition \ref{prop:SWF_minimal}, the third equality follows from the definition of $\delta(Y,\frak{s})$ from (\cite{Man16}, Section 3.7), the fourth equality follows from (\cite{LidMan}, Corollary 1.2.3), and the final equality follows from the monopole Lefschetz formula of Lin--Ruberman--Saveliev (\cite{LRS18Froyshov}, Theorem A).
\end{example}

\begin{example}
\label{ex:lens_spaces_involutions}
Let $p,q\geq 1$ be relatively prime integers with $p$ odd, and consider the lens space $L(p,q)$, which can be realized as the link of a complex singularity in $\CC^{2}$. The metric $g$ given by the restriction of the standard metric on $\CC^{2}$ to $L(p,q)\subset\CC^{2}$ is the ``standard metric'' on $L(p,q)$, which has positive scalar curvature. The complex conjugation map on $\CC^{2}$ induces an involution $\iota_{c}:L(p,q)\to L(p,q)$, which realizes $L(p,q)$ as the double branched cover over the two-bridge knot $K(p,q)\subset S^{3}$, and preserves the metric $g$. Furthermore since $p$ is odd, $L(p,q)$ admits a unique spin structure which is necessarily preserved by $\iota$. It then follows from Example \ref{ex:SWF_minimal_double_branched_cover} that 
\begin{equation}
    \wt{\kappa}(L(p,q),\iota_{c})=\kappa(L(p,q))=-\tfrac{1}{8}\sigma(K(p,q)).
\end{equation}
\end{example}

\subsection{Seifert-Fibered Spaces}
\label{subsec:seifert_fibered_spaces}

Let $\pi:Y\to\Sigma$ be a Seifert-fibered rational homology sphere, and let $i\eta\in\Omega^{1}(Y;i\RR)$ denote the connection form of the circle bundle. Recall from \cite{MOY} that any constant curvature orbifold metric $g_{\Sigma}$ on $\Sigma$ induces a metric
\[g=\eta^{2}+\pi^{*}(g_{\Sigma})\]
on $Y$, which we call the \emph{Seifert metric}. Note that $g$ is well-defined up to the choice of $g_{\Sigma}$ and the length of the circle fibers. The Levi-Cevita connection on $\Sigma$ induces a (not-necessarily torsion-free) connection $\nabla^{\infty}$ on $Y$ which is compatible with $g$ and respects the splitting $TY\cong \RR\oplus\pi^{*}(T\Sigma)$. We will refer to $\nabla^{\infty}$ as the \emph{reducible connection} on $Y$, following the terminology of \cite{MOY} (also referred to as the \emph{adiabatic connection} in \cite{Nic1}).

Although the $G^{*}_{m}$-equvariant Seiberg-Witten stable homotopy type is defined above using the Levi-Cevita connection, one can re-define everything with respect to a not necessarily torsion-free metric compatible connection. In particular, if $\sigma:Y\to Y$ is an order $m$ diffeomorphism which preserves $g$ (and hence $\nabla^{\infty}$), $\frak{s}$ is a $\sigma$-invariant spin structure on $Y$, and $\wh{\sigma}$ is a spin lift of $\sigma$, then we can define the \emph{$(g,\nabla^{\infty})$-dependent $G^{*}_{m}$-equivariant stable homotopy type}
\[\SWF(Y,\frak{s},\wh{\sigma},g,\nabla^{\infty})\in\CCC_{G^{*}_{m},\CC}\]
using the Dirac operator $\dirac^{\infty}$ defined with respect to $\nabla^{\infty}$. Moreover, we can define a corresponding equivariant correction term 
\[n(Y,\frak{s},\wh{\sigma},g,\nabla^{\infty})\in R(\ZZ_{2m})^{\sym,*}\otimes\QQ\]
using the reduced equivariant eta invariants of $\dirac^{\circ}$ and torsion terms $\{t(K_{k,j},g,\nabla^{\infty})\}$ defined with respect to $\nabla^{\infty}$ in place of the Levi-Cevita connection $\nabla^{LC}$. The fact that the desuspension of $\SWF(Y,\frak{s},\wh{\sigma},g,\nabla^{\infty})$ by $\frac{1}{2}n(Y,\frak{s},\wh{\sigma},g,\nabla^{\infty})\HH$ is $G^{*}_{m}$-stably equivalent to the metric-independent spectrum class $\SWF(Y,\frak{s},\wh{\sigma})$ follows from a spectral flow argument interpolating between the Dirac operators defined with respect to $\nabla^{\infty}$ and $\nabla^{LC}$.

In this section we will discuss two cyclic group actions on $Y$, the \emph{rotation action $\rho_{m}$}, and the \emph{complex conjugation involution $\iota_{c}$}.

The first of these actions $\rho_{m}$ is defined to be the order $m$ diffeomorphism induced by a $1/m$-ths rotation of the fibers. More precisely, if $\rho:S^{1}\times Y\to Y$ denotes the canonical fixed-point-free $S^{1}$-action given by a continuous rotation of the $S^{1}$-fibers, then $\rho_{m}:Y\to Y$ is given by
\[\rho_{m}(y)=\rho(e^{2\pi i/m},y)\qquad\text{for all }y\in Y.\]
In particular, this description shows that $\rho_{m}$ is isotopic to the identity.

Let $\frak{s}$ be a spin structure on $Y$ with corresponding principal $\Spin(3)$-bundle $P\to Y$. We can define the notion of a spin lift $\wh{\rho}$ of $\rho$ as in the cyclic group action case. More precisely, we take $\wh{\rho}:S^{1}\times P\to P$ to be an $S^{1}$-action on $P$ which makes the following diagram commute:
\begin{center}
    \begin{tikzcd}
        S^{1}\times P \arrow[r,"\wh{\rho}"] \arrow[d, "f\times\pi"] & P \arrow[d, "\pi"] \\
        S^{1}\times\Fr(Y) \arrow[r, "\rho"] & \Fr(Y)
    \end{tikzcd}
\end{center}
Here either $f=\id_{S^{1}}$, in which case we call $\wh{\rho}$ an \emph{even} spin lift, or $f:S^{1}\to S^{1}$ is the double-covering map, in which case we say $\wh{\rho}$ is an \emph{odd} spin lift. In contrast to the case of spin lifts of cyclic groups, a spin lift of $\rho$ always exists, and is unique (see \cite{AH70}). Note that the restriction of the spin lift $\wh{\rho}$ to any finite cyclic subgroup induces a distinguished spin lift $\wh{\rho}_{m}$ or $\rho_{m}$, whose parity agrees with that of $\wh{\rho}$.

The following lemma tells us how to determine the parity of $\wh{\rho}$:

\begin{lemma}
\label{lemma:S_1_actions_seifert_fibered}
    Let $(Y,\frak{s})$ be a spin Seifert-fibered rational homology sphere, and suppose the underlying orbifold surface of $Y$ is given by $S^{2}(\alpha_{1},\dots,\alpha_{n})$, the orbifold with underlying topological space $S^{2}$ and $n$ orbifold points of orders $\alpha_{1},\dots,\alpha_{n}$. Then the unique spin lift $\wh{\rho}$ of $\rho$ is of even type if and only if all of the $\alpha_{i}$ are odd, and is of odd type otherwise.
\end{lemma}

\begin{proof}
Suppose at least one of the $\alpha_{i}$ is even. Then the involution $\rho_{2}$ must have non-empty fixed-point set corresponding to the even-order exceptional fibers. Hence by the Atiyah--Bott Lemma (Proposition \ref{prop:atiyah_bott}), $\wh{\rho}_{2}$ and therefore $\wh{\rho}$ must be of odd type. 

Conversely, suppose $\alpha_{1},\dots,\alpha_{n}$ are all odd. Then as in \cite{Nic1}, the underlying orbifold surface admits a spin structure $\frak{s}_{\Sigma}$ which pulls back to the spin structure $\frak{s}$ on $Y$. Again restricting to the involution $\rho_{2}$, we see that the quotient $Y/\rho_{2}$ admits the structure of a Seifert fibration over the same orbifold surface, from which $\frak{s}_{\Sigma}$ pulls back again to a spin structure $\frak{s}'$ on $Y/\rho_{2}$. As this spin structure pulls back to $\frak{s}$ under the quotient $Y\to Y/\rho_{2}$, by the Atiyah--Bott Lemma we must have that $\wh{\rho}_{2}$ and hence $\wh{\rho}$ is of even type.
\end{proof}

From now on, we will let $\wh{\rho}_{m}$ denote the distinguished spin lift of $\rho_{m}$ induced by $\wh{\rho}$.

In the case where $Y$ is an integer homology Seifert-fibered space, we also have an odd-type involution
\[\iota_{c}:Y\to Y\]
called the \emph{complex conjugation involution} defined as follows: writing $Y=\Sigma(\alpha_{1},\dots,\alpha_{n})$, we can express $Y$ as the link of an isolated singularity of a complex variety in $\CC^{n}$. Then $\iota_{c}$ is defined to be the restriction of complex conjugation on $\CC^{n}$ to $Y$. Alternatively, let $\Sigma=S^{2}(\alpha_{1},\dots,\alpha_{n})$ denote the associated orbifold surface with orbifold points $x_{1},\dots,x_{n}\in\Sigma$. If we isotope the $x_{i}$ to the equator of $S^{2}$, then $\iota_{c}$ can be expressed as a composition $f\circ g$ of two orientation-reversing involutions, where $f$ is the diffeomorphism induced by reflection across the equator in $\Sigma$, and $g$ is the diffeomorphism induced by reflection in the fibers.

In contrast to the case of $\rho_{m}$, the complex conjugation involution $\iota_{c}$ does not come with a distinguished choice of spin lift. From now on, we will make an arbitrary fixed choice of spin lift $\wh{\iota}_{c}$.

One can show that both $\rho_{m}$ and $\iota_{c}$ preserve $g$ (and hence $\nabla^{\infty}$) as well as each spin structure on $Y$. Furthermore in the special case where $Y$ is a Brieskorn sphere, any cyclic group action on $Y$ is conjugate to either $\rho_{m}$ or $\iota_{c}$ (see \cite{AH21}).

We will see how $\wh{\rho}_{m}$ and $\wh{\iota}_{c}$ act on the associated $G^{*}_{m}$-equivariant Seiberg--Witten Floer spectrum classes over a series of propositions, each building off of one another. First we analyze how these diffeomorphisms act on the (unbased) Seiberg-Witten moduli space of critical points and flows between them.

Let $(Y,\frak{s})$ be a spin rational homology sphere, and let $\M$ denote the (unbased) moduli space of solutions to the Seiberg-Witten equations on $(Y,\frak{s})$ with respect to $\nabla^{\infty}$, i.e., $\M=\B/\G$ where $\B\subset\C(Y,\frak{s})$ denotes the kernel of the functional $\CSD$ defined with respect to $\nabla^{\infty}$, and $\G=\Map(Y,S^{1})$ denotes the full, unbased gauge group. We can write
\[\M=\{\Theta\}\cup\M^{\irr},\]
where $\Theta\in\M$ denotes the reducible, and $\M^{\irr}\subset\M$ is the set of irreducible Seiberg--Witten solutions. There is a ``charge conjugation'' involution $\frak{c}:\M\to\M$ which acts trivially on the reducible $\Theta$, and acts freely on $\M^{\irr}$. For $x,y\in\M$ we denote by $\M(x,y)$ the unbased, unparametrized space of trajectories from $x$ to $y$, on which $\ZZ_{2}=\<\frak{c}\>$ acts freely assuming $\M(x,y)$ is non-empty and positive-dimensional.

Next, let $\wt{\M}$ denote the \emph{based} moduli space of solutions to the Seiberg--Witten equations on $(Y,\frak{s})$ with respect to $\nabla^{\infty}$, and $\wt{\M}^{\irr}\subset\wt{\M}$ the irreducible locus. There is a residual $S^{1}$-action on $\wt{\M}$ whose quotient can be identified with $\M$. The corresponding quotient map $\Pi:\wt{\M}\to\M$ restricts to an $S^{1}$-fibration $\Pi^{\irr}:\wt{\M}^{\irr}\to\M^{\irr}$ on the irreducible locus, and there is a unique reducible $\wt{\Theta}\in\wt{\M}$ which is naturally identified with $\Theta\in\M$ under $\Pi$. Furthermore, the charge conjugation involution $\frak{c}:\M\to\M$ has a canonical lift to an order $4$ action whose square coincides with multiplication by $-1\in S^{1}$. Together with the $S^{1}$-action this induces a $\Pin(2)$-action on $\wt{\M}$ such that multiplication by $j\in\Pin(2)$ is given by the lift of $\frak{c}$.

A similar observation holds for trajectories.

An order $m$ action $\sigma:Y\to Y$ along with a spin lift $\wh{\sigma}$ combines with the $\Pin(2)$-action to give a $G^{*}_{m}$-action on $\wt{\M}$ and $\cup_{x,y}\wt{\M}(x,y)$, and induces a residual $(\ZZ_{2}\times\ZZ_{m})$-action on $\M$ and $\cup_{x,y}\M(x,y)$, respectively. In particular, we have the following commutative diagrams, where 
\[q:G^{*}_{m}\to G^{*}_{m}/S^{1}\cong \ZZ_{2}\times\ZZ_{m}\]
denotes the canonical quotient map:

\begin{center}
    \begin{tikzcd}
        G^{*}_{m}\times \wt{\M} \arrow[r] \arrow[d, "q\times\Pi"] & \wt{\M} \arrow[d, "\Pi"] & & G^{*}_{m}\times \cup_{x,y}\wt{\M}(x,y) \arrow[r] \arrow[d, "q\times\Pi"] & \cup_{x,y}\wt{\M}(x,y) \arrow[d, "\Pi"] \\
        (\ZZ_{2}\times\ZZ_{m})\times\M \arrow[r] & \M, & & (\ZZ_{2}\times\ZZ_{m})\times\cup_{x,y}\M(x,y) \arrow[r] & \cup_{x,y}\M(x,y).
    \end{tikzcd}
\end{center}

In the case where $(Y,\frak{s})$ is a spin Seifert-fibered rational homology sphere, we have the following description of the unbased moduli spaces $\M$ and $\cup_{x,y}\M(x,y)$ due to \cite{MOY}: Let $K_{\Sigma}$ denote the canonical bundle over the orbifold surface $\Sigma$, and let $E_{0}\to\Sigma$ be a line bundle such that the spin structure $\frak{s}$ is represented by the bundle $\pi^{*}(E_{0})\oplus \pi^{*}(K_{\Sigma})\otimes \pi^{*}(E_{0})\to Y$. By (\cite{MOY}, Theorem 5.19), we can identify
\[\M=\{\Theta\}\cup\coprod_{E}\Big(\C^{+}(E)\amalg\C^{-}(E)\Big),\]
where $\C^{+}(E)\approx\C^{-}(E)$ are two isomorphic copies of the moduli space of effective orbifold divisors in $E$, and where we take the disjoint union over all isomorphism classes of line bundles $E\to\Sigma$ such that
\begin{enumerate}
    \item $0\le \deg(E)<\tfrac{1}{2}\deg(K_{\Sigma})$, and
    \item $\pi^{*}(E)\cong\pi^{*}(E_{0})$. 
\end{enumerate}
The charge conjugation involution $\frak{c}:\M\to\M$ acts trivially on the reducible $\Theta$ and sends each component $\C^{+}(E)$ homeomorphically onto $\C^{-}(E)$ and vice-versa via the Hodge star operator.

On the level of trajectories, \cite{MOY} gives a description in terms of certain divisors on the resolution $\wh{R}$ of a ruled surface $R$ associated to $\Sigma$.

We first determine the actions of $\wh{\rho}_{m}$ and $\wh{\iota}_{c}$ on the unbased Seiberg--Witten moduli spaces:

\begin{lemma}
\label{lemma:unbased_moduli_space_seifert_fibered}
Let $(Y,\frak{s})$ be a spin Seifert-fibered rational homology sphere, let $\rho_{m}$, $\iota_{c}$ be the diffeomorphisms described above with corresponding spin lifts $\wh{\rho}_{m}$ and $\wh{\iota}_{c}$. Then $\wh{\rho}_{m}$ acts by the identity on $\M$ and $\bigcup_{x,y}\M(x,y)$, and the action of $\wh{\iota}_{c}$ on $\M$ and $\bigcup_{x,y}\M(x,y)$ agrees with the charge conjugation involution $\frak{c}$.
\end{lemma}

\begin{proof}
As noted in \cite{MOY}, the critical points are all represented by pairs $(a,\phi)$ which are invariant in the fiber direction (up to gauge), and so it follows that the induced action of $\wh{\rho}_{m}$ on $\M$ is equal to the identity. 
 
To show that the action of $\wh{\iota}_{c}$ on $\M$ agrees with $\frak{c}$, let $f:\Sigma\to\Sigma$ denote the orientation-reversing diffeomorphism induced by reflection across the equator $Q\subset\Sigma$. After isotoping if necessary, we can choose a great circle $P\subset\Sigma$ perpendicular to the equator which induces a decomposition $\Sigma=\Sigma_{+}\cup_{P}\Sigma_{-}$ such that $|\Sigma_{+}|\approx|\Sigma_{-}|\approx D^{2}$, and all of the orbifold points $x_{1},\dots,x_{n}$ lie on $\Sigma_{+}\cap Q$. Any complex line bundle $E\to\Sigma$ is then determined by a clutching function $\varphi$ on $P$ and a tuple of integers $(\gamma_{1},\dots,\gamma_{n})$ associated to the orbifold points $(x_{1},\dots,x_{n})$. We see that the reflection $f$ sends $\varphi\mapsto\varphi^{-1}$ and $\gamma_{i}\mapsto-\gamma_{i}$ for all $i=1,\dots,n$, and it follows that $f^{*}(E)\approx E^{-1}$. If $E$ is a holomorphic line bundle, then the holomorphic structure on $E$ gets sent under $f^{*}$ to an anti-holomorphic structure on $f^{*}(E)\approx E^{-1}$, and furthermore we see that holomorphic sections of $E$ are sent via $f^{*}$ to the corresponding anti-holomorphic sections of $E^{-1}$ -- this agrees precisely with the description of $\frak{c}$ in terms of the Hodge star operator.

To see this on the level of trajectories between critical points, we use the framework laid out in \cite{Bal03} which identifies Seiberg--Witten solutions on a 4-manifold with a fixed-point free $S^{1}$-action with the solutions on the quotient 3-orbifold. Although Baldridge only considers closed 4-manifolds, many of the results carry over to the open 4-manifold $Y\times[0,1]$. In particular, (\cite{Bal03}, Theorem B) implies that all solutions on $Y\times[0,1]$ are fiber-wise invariant up to gauge, hence $\wh{\rho}_{m}$ acts trivially on trajectories between critical points. Furthermore, a one-parameter version of the argument above shows that the action of $\wh{\iota}_{c}$ agrees with $\frak{c}$ on trajectories, as well.
\end{proof}

Next, we will determine how to lift the $\<\frak{c}\>\times\<\wh{\rho}_{m}\>\cong\ZZ_{2}\times\ZZ_{m}$- and $\<\frak{c}\>\times\<\wh{\iota}_{c}\>\cong\ZZ_{2}\times\ZZ_{2}$-actions on $\M$ and $\cup_{x,y}\M(x,y)$ to $G^{*}_{m}$- and $G^{\odd}_{2}$-actions, respectively, on $\wt{\M}$ and $\cup_{x,y}\wt{\M}(x,y)$. 

On the question of lifting the action of $\wh{\rho}_{m}$, we will use the fact that it arises as the restriction of the continuous $S^{1}$ action $\wh{\rho}$. Note that each irreducible $x\in\M^{\irr}$ lifts to a circle of irreducible solutions $\Pi^{-1}(x)\subset\wt{\M}^{\irr}$, and that the induced action of $\wh{\rho}$ rotates $\Pi^{-1}(x)$ at a certain rate with respect to already extant $S^{1}$-action. We encapsulate this within the following definition:

\begin{definition}
\label{def:rotation_numbers}
Let $x\in\M^{\irr}$ be an irreducible solution. We define the \emph{rotation number} $\rot(x)\in\frac{1}{2}\ZZ$ of $x$ as follows:
\begin{enumerate}
    \item If $\wh{\rho}$ is of even type: $\rot(x)\in\ZZ$ is such that the induced action of $\wh{\rho}$ on $\Pi^{-1}(x)\subset\wt{\M}^{\irr}$ is given by
    \begin{align*}
        \wh{\rho}_{*}:S^{1}\times\Pi^{-1}(x)&\to\Pi^{-1}(x) \\
        (e^{it},\wt{x})&\mapsto e^{i\rot(x)t}\cdot\wt{x}.
    \end{align*}
    \item If $\wh{\rho}$ is of odd type: $\rot(x)\in\frac{1}{2}\ZZ\setminus\ZZ$ is such that the induced action of $\wh{\rho}$ on $\Pi^{-1}(x)\subset\wt{\M}^{\irr}$ is given by
    \begin{align*}
        \wh{\rho}_{*}:\wh{S}^{1}\times\Pi^{-1}(x)&\to\Pi^{-1}(x) \\
        (e^{it},\wt{x})&\mapsto e^{2i\rot(x)t}\cdot\wt{x},
    \end{align*}
    where $\wh{S}^{1}$ denotes the double cover of $S^{1}$, to signal that this $\wh{S}^{1}$-action double covers the $S^{1}$-action on $Y$ itself.
\end{enumerate}
\end{definition}

We leave the following lemma as an exercise to the reader:

\begin{lemma}
The rotation number satisfies the following properties:
\begin{enumerate}
    \item $\rot(\frak{c}(x))=-\rot(x)$.
    \item If $x,y\in\M^{\irr}$ lie in the same connected component then $\rot(x)=\rot(y)$.
    \item If $x,y\in\M^{\irr}$ are such that $\M(x,y)\neq 0$, then $\rot(x)=\rot(y)$.
\end{enumerate}
\end{lemma}

In view of the above lemma, for a connected component $C\subset\M^{\irr}$ we will sometimes denote by $\rot(C)$ the rotation number of any $x\in C$. Furthermore if $\wt{x}\in\wt{\M}^{\irr}$ we denote by $\rot(\wt{x}):=\rot(\Pi(x))$, and similarly for connected components.

The problem of lifting the induced action of $\wh{\iota}_{c}$ to the based moduli space(s) is somewhat more subtle. Although Lemma \ref{lemma:unbased_moduli_space_seifert_fibered} implies that $(\wh{\iota}_{c})_{*}$ agrees with $\frak{c}$ on the unbased moduli space(s), it cannot be the case that $(\wh{\iota}_{c})_{*}$ agrees with the action of $j$ or $-j$ on the based moduli space(s). Indeed $(\wh{\iota}_{c})_{*}$ must commute with the $S^{1}$-action, whereas $j$ anti-commutes with $i\in S^{1}$. It turns out that the relations in $G^{\odd}_{2}$ imply that the action of $(\wh{\iota}_{c})_{*}$ must agree with the action of $j$ or $-j$, composed with a map modelled on the complex conjugation action on $S^{1}$. This leads us to the following definition:

\begin{definition}
\label{def:S_1_equivariant_Real_structure}
Let $X$ be a finite $S^{1}$-CW complex such that on each connected component, $S^{1}$ acts either freely or trivially. An \emph{$S^{1}$-equivariant Real structure} on $X$ consists of an involution $c:X\to X$ called the \emph{conjugation involution} such that
\begin{enumerate}
    \item The $S^{1}$-action and the involution $c$ combine into an $O(2)$-action on $X$.
    \item Let $O\subset X$ be an $O(2)$-orbit. Then:
    \begin{enumerate}
        \item If $S^{1}$ acts trivially on $O$: $O\simeq\{\pt\}$.
        \item If $S^{1}$ acts freely on $O$: $O$ is $O(2)$-equivariantly homotopy equivalent to $S^{1}$ with its canonical $O(2)$-action, where $c$ corresponds to multiplication by $(\begin{smallmatrix}1  & 0 \\ 0 & -1\end{smallmatrix})\in O(2)$.
    \end{enumerate}
\end{enumerate}
\end{definition}

Note that every space $X$ as in Definition \ref{def:S_1_equivariant_Real_structure} admits a unique $S^{1}$-equivariant Real structure up to $O(2)$-equivariant homotopy equivalence.

We encapsulate the preceding discussion in the following proposition:

\begin{proposition}
\label{prop:based_moduli_space_seifert_fibered}
Let $(Y,\frak{s})$ be a spin Seifert-fibered rational homology sphere. Let $C=C_{+}\amalg C_{-}\subset\wt{\M}^{\irr}$ be two connected components of irreducible solutions which are interchanged by the action of $j\in\Pin(2)$.
\begin{enumerate}
    \item Let $m\geq 2$ be an integer and let $\rho_{m}:Y\to Y$ be the order $m$ diffeomorphism induced by rotation of the fibers as above, with distinguished spin lift $\wh{\rho}_{m}$. Then:
    \begin{enumerate}
        \item If $\wh{\rho}_{m}$ is an even spin lift: the induced action of $\wh{\rho}_{m}$ on $C_{+}$, $\wt{\M}(C_{+},\wt{\Theta})$, and $\wt{\M}(\wt{\Theta},C_{+})$ coincides with $\omega_{m}^{\rot(C_{+})}\in S^{1}$, and on $C_{-}$, $\wt{\M}(C_{-},\wt{\Theta})$, $\wt{\M}(\wt{\Theta},C_{-})$, coincides with $\omega_{m}^{\rot(C_{-})}=\omega_{m}^{-\rot(C_{+})}\in S^{1}$. Similarly for any other such pair of connected components $C'=C'_{+}\amalg C'_{-}\subset\wt{\M}^{\irr}_{sw}$ such that at least one of $\wt{\M}(C,C')$, $\wt{\M}(C',C)$ is non-empty, then $\rot(C'_{\pm})=\rot(C_{\pm})$, and the induced action of $\wh{\rho}_{m}$ on $C'_{+}$, $\wt{\M}(C_{+},C'_{+})$, and $\wt{\M}(C'_{+},C_{+})$ coincides with $\omega_{m}^{\rot(C_{+})}\in S^{1}$, and on $C'_{-}$, $\wt{\M}(C_{-},C'_{-})$ and $\wt{\M}(C'_{-},C_{-})$ coincides with $\omega_{m}^{\rot(C_{-})}=\omega_{m}^{-\rot(C_{+})}\in S^{1}$.
        \item If $\wh{\rho}_{m}$ is an odd spin lift: the induced action of $\wh{\rho}_{m}$ on $C_{+}$, $\wt{\M}(C_{+},\wt{\Theta})$, and $\wt{\M}(\wt{\Theta},C_{+})$ coincides with $\omega_{2m}^{2\rot(C_{+})}\in S^{1}$, and on $C_{-}$, $\wt{\M}(C_{-},\wt{\Theta})$, $\wt{\M}(\wt{\Theta},C_{-})$, coincides with $\omega_{2m}^{2\rot(C_{-})}=\omega_{2m}^{-2\rot(C_{+})}\in S^{1}$. Similarly for any other such pair of connected components $C'=C'_{+}\amalg C'_{-}\subset\wt{\M}^{\irr}_{sw}$ such that at least one of $\wt{\M}(C,C')$, $\wt{\M}(C',C)$ is non-empty, then $\rot(C')=\rot(C)$, and the induced action of $\wh{\rho}_{m}$ on $C'_{+}$, $\wt{\M}(C_{+},C'_{+})$, and $\wt{\M}(C'_{+},C_{+})$ coincides with the action of $\omega_{2m}^{2\rot(C_{+})}\in S^{1}$, and on $C'_{-}$, $\wt{\M}(C_{-},C'_{-})$ and $\wt{\M}(C'_{-},C_{-})$ coincides with the action of $\omega_{2m}^{2\rot(C_{-})}=\omega_{2m}^{-2\rot(C_{+})}\in S^{1}$.
    \end{enumerate}
    \item Suppose $Y$ is a Seifert-fibered homology sphere, let $\iota_{c}:Y\to Y$ be the complex conjugation involution as above, and let $\wh{\iota}_{c}$ be a spin lift of $\iota_{c}$. Then:
    \begin{enumerate}
        \item The spaces $C$, $\wt{\M}(C,\wt{\Theta})$, and $\wt{\M}(\wt{\Theta},C)$ admit $S^{1}$-equivariant Real structures as in Definition \ref{def:S_1_equivariant_Real_structure}, and there exists some $\varepsilon(C)\in\{\pm 1\}$ such that the induced action of $\wh{\iota}$ on $C$, $\wt{\M}(C,\wt{\Theta})$, and $\wt{\M}(\wt{\Theta},C)$ coincides with the action of $\varepsilon(C)\cdot j\in \Pin(2)$, followed by the conjugation involution $c$.
        \item For any other such pair of connected components $C'=C'_{+}\amalg C'_{-}\subset\wt{\M}^{\irr}$ such that at least one of $\wt{\M}(C,C')$, $\wt{\M}(C',C)$ is non-empty, we have that $\varepsilon(C')=\varepsilon(C)\in\{\pm 1\}$, and the spaces $C'$, $\wt{\M}(C,C')$, $\wt{\M}(C',C)$ admit $S^{1}$-equivariant Real structures such that the induced action of $\wh{\iota}_{c}$ on these spaces coincides with the action of $\varepsilon(C)\cdot j\in \Pin(2)$, followed by the conjugation involution $c$.
    \end{enumerate}
\end{enumerate}
\end{proposition}

We now proceed to relate these observations to the associated Seiberg--Witten Floer spectrum classes. The following proposition is a an equivariant generalization of Stoffregen's obsrvation that the Floer spectra of spin Seifert-fibered rational homology spheres of negative fibration are \emph{$j$-split} (\cite{Stoff20}, Lemma 5.3):

\begin{proposition}
\label{prop:floer_spectrum_seifert_fibered}
Let $(Y,\frak{s},g)$ be a spin Seifert-fibered rational homology sphere of negative fibration with at most four singular fibers, equipped with the Seifert metric $g$.
\begin{enumerate}
    \item Let $m\geq 2$, let $\rho_{m}:Y\to Y$ be the order $m$ diffeomorphism induced by rotation of the fibers as above, and let $\wh{\rho}_{m}$ be a spin lift of $\rho_{m}$. Then there exists a space $X$ of type $\CC$-$G^{*}_{m}$-$\SWF$ such that
    \[\SWF(Y,\frak{s},\wh{\rho}_{m},g,\nabla^{\infty})=[(X,\mbfa,\mbfb)],\]
    and $G^{*}_{m}$-spaces $X_{1},\dots,X_{n}$ such that
    \begin{equation}
    \label{eq:rho_m_split}
        X/X^{S^{1}}=\big(X_{1}\vee\cdots\vee X_{k}\big)\vee\big(jX_{1}\vee\cdots\vee jX_{n}\big),
    \end{equation}
    where for each $i=1,\dots,n$, $X_{i}$ and $jX_{i}$ are two copies of the same space, interchanged by the action of $j\in\Pin(2)$. Moreover:
    \begin{enumerate}
        \item If $\wh{\rho}_{m}$ is an even spin lift: there exist some integers $k_{1},\dots,k_{n}$ with $0\le k_{i}\le m-1$ such that for each $i=1,\dots,n$, the induced action of $\wh{\rho}_{m}$ on $X_{i}$ coincides with $\omega_{m}^{k_{i}}\in S^{1}$, and the induced action of $\wh{\rho}_{m}$ on $jX_{i}$ coincides with $\omega_{m}^{-k_{i}}\in S^{1}$.
        \item If $\wh{\rho}_{m}$ is an odd spin lift: there exist some odd integers $k_{1},\dots,k_{n}$ with $1\le k_{i}\le 2m-1$ such that for each $i=1\dots,n$, the induced action of $\wh{\rho}_{m}$ on $X_{i}$ coincides with $\omega_{2m}^{k_{i}}\in S^{1}$, and the induced action of $\wh{\rho}_{m}$ on $jX_{i}$ coincides with $\omega_{2m}^{-k_{i}}\in S^{1}$.
    \end{enumerate}
    \item Let $\iota_{c}:Y\to Y$ be the odd-type involution induced by complex conjugation as above, and let $\wh{\iota}_{c}$ be a spin lift of $\iota_{c}$. Then there exists a a space $X$ of type $\CC$-$G^{\odd}_{2}$-$\SWF$ such that
    \[\SWF(Y,\frak{s},\wh{\iota}_{c},g,\nabla^{\infty})=[(X,\mbfa,\mbfb)],\]
    and $G^{\odd}_{2}$ spaces $X_{+},X_{-}$ such that
    \begin{equation}
    \label{eq:iota_c_split}
    X/X^{S^{1}}=\big(X_{+}\vee jX_{+}\big)\vee\big(X_{-}\vee jX_{-}\big),
    \end{equation}
    where:
    \begin{enumerate}
        \item $X_{+}$ and $jX_{+}$ are two copies of the same space interchanged by the action of $j\in\Pin(2)$, and similarly for $X_{-}$ and $jX_{-}$.
        \item As $S^{1}$-CW-complexes, the spaces $X_{+}$ and $X_{-}$ are built solely of free cells of the form $S^{1}\times D^{n}$, and are endowed with \emph{complex conjugation actions} $c:X_{\pm}\to X_{\pm}$ which on each cell restricts the usual complex conjugation involution on the $S^{1}$-factor:
        \begin{align*}
            S^{1}\times D^{n}&\to S^{1}\times D^{n}\\
            (e^{i\theta},re^{i\phi})&\mapsto (e^{-i\theta},re^{i\phi}).
        \end{align*}
        \item  The induced action of $\wh{\iota}_{c}$ on $X_{+}\vee jX_{+}$ coincides with the action of $j\in \Pin(2)$, followed by complex conjugation, and the induced action of $\wh{\iota}_{c}$ on $X_{-}\vee jX_{-}$ coincides with the action of $-j\in \Pin(2)$, followed by complex conjugation.
    \end{enumerate}
\end{enumerate}
\end{proposition}

\begin{proof}
In the case where $Y$ has at most four singular fibers, it is shown in \cite{MOY} that all of the critical points are isolated and non-degenerate, and that all of the flows between critical points are Morse-Bott. Furthermore, the fact that $Y$ is of negative fibration implies that $\SWF(Y,\frak{s},\wh{\sigma},g,\nabla^{\infty})$ for $\wh{\sigma}=\wh{\rho}_{m}$ or $\wh{\iota}_{c}$ is $j$-split.

Let $\C(Y,\frak{s})$ be the Seiberg-Witten configuration space, let $\T_{0}$ denote the $L^{2}$ completion of the tangent space to $\C(Y,\frak{s})$, and let 
\[\nabla CSD:\C(Y,\frak{s})\to \T_{0}\]
denote the gradient of the Chern-Simons-Dirac functional defined with respect to $(g,\nabla^{\infty})$. If $V\subset\C(Y,\frak{s})$ denotes the global Coloumb slice, and $\T_{0}^{\gC}$ denotes the $L^{2}$-completion of the tangent bundle to $V$, we have an induced map
\[(\nabla CSD)^{\gC}:V\to \T_{0}^{\gC}\]
In \cite{LidMan}, the authors proved that there is a one-to-one correspondence between finite-energy trajectories of $(\nabla CSD)^{\gC}$ and finite-energy trajectories of $\nabla CSD$ modulo the based gauge group. Now for an eigenvalue $\lambda>>0$ of the linearization of $(\nabla CSD)^{\gC}$, let
\[(\nabla CSD)^{\gC}_{\lambda}:V^{\lambda}\to\T_{0,\lambda}^{\gC}\]
denote the finite dimensional approximation corresponding to $\lambda$, and let $T_{\lambda}$ be the isolated invariant set consisting of all critical points of $(\nabla CSD)^{\gC}_{\lambda}$ and finite-energy flows between them. We denote by $I_{\lambda}=I(T_{\lambda})$ the associated Conley index. For each $\omega>0$, let:
\begin{itemize}
    \item $T_{\lambda}^{>\omega,\irr}$ be the set of irreducible critical points $x$ with $CSD(x)>\omega$, together with all points on the flows between critical points of this type, and $I_{\lambda}^{>\omega,\irr}=I(T_{\lambda}^{>\omega,\irr})$ the associated Conley index.
    \item $T_{\lambda}^{\le\omega}$ be the set of all critical points $x$ with $CSD(x)\le\omega$, not necessarily irreducible, together with all points on the flows between critical points of this type, and $I_{\lambda}^{\le\omega}=I(T_{\lambda}^{\le\omega})$ the associated Conley index.
\end{itemize}

As in \cite{Man03} we have the following attractor-repeller coexact sequence:
\begin{equation}
\label{eq:coexact_sequence}
	I_{\lambda}^{\le\omega}\to I_{\lambda}\to I_{\lambda}^{>\omega,\irr}\to\Sigma I_{\lambda}^{\le\omega}\to\cdots.
\end{equation}
We will proceed by induction on the cut-off $\omega>0$. First consider the induced action of $\wh{\rho}_{m}$. We will show that there exists a decomposition
\begin{align}
\begin{split}
\label{eq:rho_m_split_omega}
	&I_{\lambda}^{\le\omega}/(I_{\lambda}^{\le\omega})^{S^{1}}=\big(X_{1}\vee\cdots\vee X_{k}\big)\vee\big(jX_{1}\vee\cdots\vee jX_{n}\big)
\end{split}
\end{align}
as in Equation \ref{eq:rho_m_split} for each $\omega\geq 0$. Note that Equation \ref{eq:rho_m_split_omega} holds for $\omega=0$, since the only critical point with $\omega=0$ is the unique reducible $\Theta$. Now suppose Equation \ref{eq:rho_m_split_omega} holds for some fixed $\omega_{0}\geq 0$. Fix some collection of critical points $x_{1}^{+},\dots,x_{\ell}^{+},x_{1}^{-},\dots,x_{\ell}^{-}$ which satisfy $CSD=\omega'_{0}>\omega_{0}$, and are minimal among all critical points satisfying $CSD>\omega_{0}$. The coexact sequence (\ref{eq:coexact_sequence}) and Proposition \ref{prop:based_moduli_space_seifert_fibered} implies that $I_{\lambda}^{\le\omega'_{0}}$ is obtained by attaching $\ell$ cells to $I_{\lambda}^{\le\omega_{0}}$ corresponding to the pairs $(x_{k}^{+},x_{k}^{-})$, each of which are of the form $G^{*}_{m}/H_{k}\times D^{\ind(x^{\pm}_{k})}$ with $H_{k}=\<\gamma\omega_{m}^{-a(k)}\>$ for some $0\le a(k)\le m-1$ if $*=\ev$, or $H_{k}=\<\mu\omega_{2m}^{-a(k)}\>$ for some $1\le a(k)\le 2m-1$ odd if $*=\odd$. We see immediately that the splitting (\ref{eq:rho_m_split_omega}) holds for $\omega=\omega'_{0}$.

The argument for the splitting (\ref{eq:iota_c_split}) is entirely analogous and left as an exercise to the reader.
\end{proof}

\begin{remark}
One could try to extend Proposition \ref{prop:floer_spectrum_seifert_fibered} to all Seifert-fibered rational homology spheres of negative fibration as follows: by defining a class of $G^{*}_{m}$-equivariant analogues of tame admissible perturbations $\frak{q}$ in the sense of \cite{KMbook}, one could define perturbed analogues $\SWF(Y,\frak{s},\wh{\sigma},g,\frak{q})$, and show that the $G^{*}_{m}$-stable equivalence class of $\SWF(Y,\frak{s},\wh{\sigma},g,\frak{q})$ is independent of the choice of equivariant perturbation $\frak{q}$. However, showing the existence and genericity of (tame, admissible) $G^{*}_{m}$-equivariant perturbations would be a rather delicate problem, and may not even be true. Although this has been done by Lin \cite{LinFthesis} in the $\Pin(2)$-equivariant case, it is not clear to the author whether this could be done in the $G^{*}_{m}$-equivariant case, as the issue of equivariant transversality is a notoriously difficult one.
\end{remark}

\begin{corollary}
\label{cor:seifert_fibered_Pin(2)_surjective}
Let $(Y,\frak{s})$ be a spin Seifert-fibered rational homology sphere of negative fibration with at most four singular fibers. Then $(Y,\frak{s},\rho_{m})$ for any $m$ and $(Y,\frak{s},\iota_{c})$ are $\Pin(2)$-surjective. In particular,
\[\wt{\kappa}(Y,\frak{s},\rho_{2})=\wt{\kappa}(Y,\frak{s},\iota_{c})=\kappa(Y,\frak{s}).\]
\end{corollary}

\begin{proof}
Let $\sigma$ denote either $\rho_{m}$ or $\iota_{c}$, let $\wh{\sigma}$ be a spin lift of $\sigma$, and let
\[I_{\lambda}=\SWF(Y,\frak{s},\wh{\iota}_{c},g,\nabla^{\infty},\lambda)\]
denote the $G^{*}_{m}$-equivariant Conley index corresponding to a fixed eigenvalue cut-off $\lambda>>0$. We can assume that $I_{\lambda}$ is a space of type $\CC$-$G^{*}_{m}$-$\SWF$ via suspending by real $G^{*}_{m}$-representations if necessary, as in the definition of the $\CC$-$G^{*}_{m}$-Seiberg--Witten Floer spectrum class.

It suffices to show that for any virtual $\Pin(2)$-equivariant bundle $E\in\wt{K}_{\Pin(2)}(I_{\lambda})$ of dimension 0, we can extend the $\Pin(2)$-action on $E$ to a $G^{*}_{m}$-action. First note that we can always extend the $\Pin(2)$ action to a $G^{*}_{m}$-action over the restriction of $E$ to the $S^{1}$-fixed point set $I_{\lambda}^{S^{1}}$. Indeed since $I_{\lambda}^{S^{1}}$ is $G^{*}_{m}$-homotopy equivalent to a complex $G^{*}_{m}$-representation sphere, this follows from the fact that the restriction map $\res:R(G^{*}_{m})\to R(\Pin(2))$ is surjective.

We first consider the question of how to extend $\wh{\rho}_{m}$ over all of $E\to I_{\lambda}$. Using the decomposition in Proposition \ref{prop:floer_spectrum_seifert_fibered}, on each wedge summand $Z=X_{i}$ or $jX_{i}$ of $I_{\lambda}/I_{\lambda}^{S^{1}}$ the action of $\wh{\rho}_{m}$ is contained in the action of $S^{1}$ on $Z$. We can therefore use the already extant $S^{1}$-action on $E|_{Z}$ to define a lift of $\wh{\rho}_{m}$ to an action on $E|_{Z}$ for each wedge summand $Z$. Hence this gives us a well-defined global lift of the $G^{*}_{m}$ action to $E\to I_{\lambda}$.

The extension of $\wh{\iota}_{c}$ to all of $E$ is similar, with the only subtlely being the complex conjugation action. But since the virtual bundle $E$ is of dimension 0, it can be represented by the difference of two even-dimensional bundles, on each of which complex conjugation acts as a complex-linear map.
\end{proof}

\begin{proof}[Proof of Theorem \ref{theorem:intro_seifert_fibered_spaces}]
Note that any odd-type involution on a Seifert-fibered homology sphere which is isotopic to the identity is conjugate to $\rho_{2}$. The result then follows from Property (1) of Theorem \ref{theorem:properties_equivariant_kappa_invariants} and Corollary \ref{cor:seifert_fibered_Pin(2)_surjective}.
\end{proof}

\bigskip
\subsection{Cyclic Group Actions on a Family of Brieskorn Spheres}
\label{subsec:brieskorn_spheres}

We now proceed to compute explicitly the $G^{*}_{m}$-equivariant Seiberg-Witten Floer spectrum classes of the family of Brieskorn spheres $\pm\Sigma(2,3,6n\pm 1)$ with respect to the actions $\rho_{m}$ and $\iota_{c}$ considered in Section \ref{subsec:seifert_fibered_spaces}.

There is a certain amount of ambiguity in these calculations --- for the $\ZZ_{m}$-action $\rho_{m}$ the ambiguity amounts to the collection of rotation numbers associated to each pair of irreducibles interchanged by $j$ as in Definition \ref{def:rotation_numbers}, and for the odd-type involution $\iota_{c}$ the ambiguity amounts to a certain assignment of signs $\varepsilon\in\{\pm 1\}$ to each pair of irreducibles. Fortunately in the case of the odd-type involutions $\rho_{2}$ and $\iota_{c}$ the invariants $\wt{\kappa}$ and $\kappa_{KMT}$ do not depend on these signs. 

In the case of $\rho_{m}$, $m\neq 2$, the aforementioned ambiguity means that we are unable to calculate the full set of equivariant $\kappa$-invariants associated to $(\pm\Sigma(2,3,k),\rho_{m})$ on the nose. However in the case where $m=p$ is an odd prime, we are able to extract some partial information about the set $\K_{\Pi}(Y,\rho_{p})\subset\QQ^{2}$.

\subsubsection{The Action \texorpdfstring{$\rho_{m}$}{rhom}}
\label{subsubsec:brieskorn_spheres_cyclic_actions}

Let $Y=\pm\Sigma(2,3,6n\pm 1)$, and let $\wh{\rho}_{m}$ be the spin lift of $\rho_{m}$ induced by $\wh{\rho}$, which by Lemma \ref{lemma:S_1_actions_seifert_fibered} must be of \emph{odd} type. We have the following proposition:

\begin{proposition}
\label{prop:calculation_spectrum_2_3_k_rho_m}
Let $m\geq 2$. For $Y=\pm\Sigma(2,3,12n\pm 1)$ or $\pm\Sigma(2,3,12n\pm 5)$, let $\M^{\irr}=\{x_{1,+},x_{1,-},\dots,x_{n,+},x_{n,-}\}$ be an enumeration of the $n$ pairs of irreducible Seiberg--Witten solutions on $Y$, and for each $k=1,\dots,n$, let $\frac{1}{2}\le r_{k}\le m-\frac{1}{2}$ be the unique half-integer satisfying $r_{k}\equiv\rot(x_{k,+})\pmod{m}$.
\begin{enumerate}
    \item Let $Y=\Sigma(2,3,12n-1)$ or $\Sigma(2,3,12n-5)$. Then the $G^{\odd}_{m}$-equivariant Seiberg--Witten Floer spectrum class of $(Y,\wh{\rho}_{m})$ is given by
    \[\SWF(Y,\wh{\rho}_{m})=\Big[\Big(\wt{\Sigma}Z_{r_{1},\dots,r_{n};m},0,\tfrac{1}{2}n(Y,\wh{\rho}_{m},g,\nabla^{\infty})\Big)\Big],\]
    and the spectrum class of $(-Y,\wh{\rho}_{m})$ is given by
    \[\SWF(-Y,\wh{\rho}_{m})=\Big[\Big(\wt{\Sigma}X_{r_{1},\dots,r_{n};m},0,-\tfrac{1}{2}n(Y,\wh{\rho}_{m},g,\nabla^{\infty})+\sum_{k=1}^{n}\xi^{2r_{k}}\Big)\Big].\]
    \item Let $Y=\Sigma(2,3,12n+1)$ or $\Sigma(2,3,12n+5)$. Then the $G^{\odd}_{m}$-equivariant Seiberg--Witten Floer spectrum class of $(Y,\wh{\rho}_{m})$ is given by
    \begin{align*}
        \SWF(Y,\wh{\rho}_{m})&=\Big[\Big(S^{0}\vee\bigvee_{k=1}^{n}\Sigma^{-1}(Z_{r_{k},m})_{+},0,\tfrac{1}{2}n(Y,\wh{\rho}_{m},g,\nabla^{\infty})\Big)\Big] \\
        &:=\Big[\Big(S^{\sum_{k=1}^{n}\HH_{r_{k}}}\vee\bigvee_{k=1}^{n}\Sigma^{3\RR+\sum_{\ell\neq k}\HH_{r_{\ell}}}(Z_{r_{k},m})_{+},0,\tfrac{1}{2}n(Y,\wh{\rho}_{m},g,\nabla^{\infty})+\sum_{k=1}^{n}\xi^{2r_{k}}\Big)\Big],
    \end{align*}
    and the spectrum class of $(-Y,\wh{\rho}_{m})$ is given by
    \[\SWF(-Y,\wh{\rho}_{m})=\Big[\Big(S^{0}\vee\bigvee_{k=1}^{n}(Z_{r_{k},m})_{+},0,-\tfrac{1}{2}n(Y,\wh{\rho}_{m},g,\nabla^{\infty})\Big)\Big].\]
    Consequently, on the level of local equivalence we have that
    \[\SWF(\pm Y,\wh{\rho}_{m})\equiv_{\ell}\Big[\Big(S^{0},0,\pm\tfrac{1}{2}n(Y,\wh{\rho}_{m},g,\nabla^{\infty})\Big)\Big].\]
\end{enumerate}
\end{proposition}

In order to prove Proposition \ref{prop:calculation_spectrum_2_3_k_rho_m}, we will make use of the following special case of the tom Dieck splitting theorem:

\begin{proposition}[\cite{tomDieck75}, \cite{LMSM86}]
\label{prop:tom_dieck_splitting}
Let $G$ be a compact Lie group. For $H<G$ a subgroup let $W_{G}H=N_{G}H/H$ denote its Weyl group where $N_{G}H$ denotes the normalizer of $H$ in $G$, and let $\C$ denote the set of conjugacy classes of subgroups of $G$. Then for any $k\in\ZZ$, there exists an isomorphism
\[\pi_{k}^{\st,G}(S^{0})\cong\bigoplus_{[H]\in\C}\pi_{k}^{\st}(\Sigma^{\text{Ad}W_{G}H}BW_{G}H_{+}),\]
where $\text{Ad}W_{G}H$ denotes the adjoint representation of $W_{G}H$ on its Lie algebra.
\end{proposition}

We will make use of the following special cases:
\begin{enumerate}
    \item Suppose $G=\ZZ_{m}$. Then Proposition \ref{prop:tom_dieck_splitting} implies that
    \[\pi_{k}^{\st,\ZZ_{m}}(S^{0})\cong\pi_{k}^{\st}(S^{0})\oplus\bigoplus_{1\le k\le m-1}\pi_{k}^{\st}((B\ZZ_{(k,m)})_{+}).\]
    In particular, we have that
    \begin{align*}
        &\pi_{-1}^{\st,\ZZ_{m}}(S^{0})=0, & &\pi_{0}^{\st,\ZZ_{m}}(S^{0})=\ZZ^{m},
    \end{align*}
    where the latter group has a preferred generating set $\ZZ^{m}=\ZZ\<\gamma_{0},\dots,\gamma_{m-1}\>$ in one-to-one correspondence with the elements of $\ZZ_{m}$.
    \item Suppose $G=S^{1}$. Then in particular, Proposition \ref{prop:tom_dieck_splitting} implies that
    \[\pi_{0}^{\st,S^{1}}(S^{0})\xrightarrow[\cong]{\res^{S^{1}}_{e}}\pi_{0}^{\st}(S^{0})=\ZZ.\]
\end{enumerate}

\begin{proof}[Proof of Proposition \ref{prop:calculation_spectrum_2_3_k_rho_m}]
We proceed on a case-by-case basis:
\begin{enumerate}[label=Case \arabic*:,leftmargin=.61in]
    \item $Y=\Sigma(2,3,12n-1)$ or $\Sigma(2,3,12n-5)$:
    
    In this case, the irreducibles are all at the same (Morse) degree, with the reducible one degree lower. For each $k$ let $C_{k,\pm}\subset\wt{\M}^{\irr}$ denote the circle of irreducibles corresponding to $x_{k,\pm}\in\M^{\irr}$. By Proposition \ref{prop:based_moduli_space_seifert_fibered}, each pair of irreducibles $\{C_{k,+},C_{k,-}\}$ can be identified with the $G^{\odd}_{m}$-cell
    \[Z_{r_{k},m}=G^{\odd}_{m}/\<\omega^{-r_{k}}_{m}\mu\>\]
    from Example \ref{ex:cosets_cyclic}.

    First, note that there are no flows between irreducibles. Indeed, such a map would be determined by an element of
    \[\{\Sigma^{-1}(Z_{r_{k},m})_{+},(Z_{r_{\ell},m})_{+}\}_{G^{\odd}_{m}}\cong\pi_{-1}^{\st,G^{\odd}_{m}}\big((Z_{r_{\ell},m})_{+}^{\<\omega^{-r_{k}}_{m}\mu\>}\big).\]
    If $r_{k}\not\equiv r_{\ell}\pmod{m}$, then $(Z_{r_{\ell},m})_{+}^{\<\omega^{-r_{k}}_{m}\mu\>}\simeq\pt$, and so
    \[\pi_{-1}^{\st,G^{\odd}_{m}}\big((Z_{r_{\ell},m})_{+}^{\<\omega^{-r_{k}}_{m}\mu\>}\big)=\pi_{-1}^{\st,G^{\odd}_{m}}(\pt)=0.\]
    On the other hand if $r_{k}\equiv r_{\ell}\pmod{m}$, then
    \[\pi_{-1}^{\st,G^{\odd}_{m}}\big((Z_{r_{\ell},m})_{+}^{\<\omega^{-r_{k}}_{m}\mu\>}\big)=\pi_{-1}^{\st,G^{\odd}_{m}}\big((Z_{r_{\ell},m})_{+}\big)\cong\pi_{-1}^{\st,\<\omega^{-r_{k}}_{m}\mu\>}(S^{0})\cong\pi_{-1}^{\st,\ZZ_{m}}(S^{0})=0,\]
    and so any attaching map between irreducible cells must be trivial. The spectrum class can therefore be constructed by (stably) attaching the $n$ $G^{\odd}_{m}$-cells $\{\Sigma(Z_{r_{k},m})_{+}\}_{k=1}^{n}$ to a trivial cell $S^{0}$. The attaching map for each cell is determined by a stable homotopy class in
    \begin{equation*}
    \label{eq:Z_m_attaching_map}
        \{(Z_{r_{k},m})_{+},S^{0}\}_{G^{\odd}_{m}}\cong \{S^{0},S^{0}\}_{\<\omega_{m}^{-r_{k}}\mu\>}\cong\{S^{0},S^{0}\}_{\ZZ_{m}}=\pi_{0}^{\st,\ZZ_{m}}(S^{0})=\ZZ^{m}.
    \end{equation*}

    In order to determine the map more precisely, recall that $\rho_{m}$ is obtained as a restriction of the $S^{1}$-action $\rho$ on $Y$. It follows that the $G^{\odd}_{m}$-action on all of the spaces considered in this context extend to an action of the larger group  $G^{\odd}_{S^{1}}:=\Pin(2)\times_{\ZZ_{2}}S^{1}$. From the analysis in Section \ref{subsec:seifert_fibered_spaces}, we can conclude that the irreducibles correspond to $G^{\odd}_{S^{1}}$-cells of the form
    \[Z_{r_{k},S^{1}}:=G^{\odd}_{S^{1}}/H_{r_{k},S_{1}},\]
    where $H_{r_{k},S_{1}}$ denotes the subgroup
    \[H_{r_{k},S_{1}}:=\<[(e^{-2\pi i r_{k}t},e^{2\pi i t})]\;|\;0\le t<2\pi\><G^{\odd}_{S^{1}}.\]
    The attaching maps of these $G^{\odd}_{S^{1}}$-cells to the reducible are each determined by an element in
    \[\{(Z_{r_{k},S^{1}})_{+},S^{0}\}_{G^{\odd}_{S^{1}}}\cong \{S^{0},S^{0}\}_{H_{r_{k},S_{1}}}\cong\{S^{0},S^{0}\}_{S^{1}}=\pi_{0}^{\st,S^{1}}(S^{0})=\ZZ.\]
    Hence the attaching map of the $G^{\odd}_{m}$-cell $Z_{r_{k},m}$ must lie in the $\ZZ$-summand
    \[\res^{S^{1}}_{\ZZ_{m}}(\pi_{0}^{\st,S^{1}}(S^{0}))=\ZZ\<\gamma_{0}\><\ZZ\<\gamma_{0},\dots,\gamma_{m-1}\>=\pi_{0}^{\st,\ZZ_{m}}(S^{0})\]
    corresponding to the trivial element of $\ZZ_{m}$.
    
    Therefore, altogether the attaching maps for the $n$ pairs of irreducibles are determined by an element of $\ZZ^{n}$ as in the $\Pin(2)$-equivariant setting (\cite{Man14}, Section 5.2). In \cite{MOY} it was shown that there is a unique flowline from each irreducible to the reducible solution, and hence we can assume that this element is $(\pm 1,\dots,\pm 1)\in\ZZ^{n}$. As the spectrum class depends only on the divisibility of this element, we can assume that the attaching map is given by $(1,\dots,1)\in\ZZ^{n}$. Hence a model for the spectrum class is given by 
    \[\wt{\Sigma}Z_{r_{1},\dots,r_{n};m}=\wt{\Sigma}\Big(Z_{r_{1},m}\amalg\cdots\amalg Z_{r_{1},m}\Big)\]
    (de-)suspended by the appropriate equivariant correction term, as claimed.
    
    \item $Y=-\Sigma(2,3,12n-1)$ or $-\Sigma(2,3,12n-5)$:

    Follows from Proposition \ref{prop:floer_spectrum_duality} and the calculations in Example \ref{ex:multiple_cosets_duals}.

    \item $Y=\Sigma(2,3,12n+1)$ or $\Sigma(2,3,12n+5)$:

    In this case, the irreducibles are all at the same degree, but all lie one degree lower than the reducible. As in Case 1, there cannot be any flows between irreducibles. Furthermore the attaching maps from the irreducibles to the reducible must be trivial as in the $\Pin(2)$-equivariant setting, which implies the given presentation of the $\SWF$ spectrum class.
    
    \item $Y=-\Sigma(2,3,12n+1)$ or $-\Sigma(2,3,12n+5)$:

    This follows from Proposition \ref{prop:floer_spectrum_duality} and the Wirthm\"uller isomorphism (\cite{Wir75}, \cite{LMSM86}), which shows that $(Z_{r_{k},m})_{+}$ and $\Sigma^{3}(Z_{r_{k},m})_{+}$ are $\HH_{r_{k}}$-dual for each $k=1,\dots,n$.
\end{enumerate}
\end{proof}

The above proposition implies that in the case where $Y=\pm\Sigma(2,3,12n+1)$ or $Y=\pm\Sigma(2,3,12n+5)$, the set of equivariant $\kappa$-invariants of $(Y,\rho_{m})$ are completely determined by the equivariant correction term $n(Y,\wh{\rho}_{m},g,\nabla^{\infty})$:

\begin{corollary}
\label{cor:equivariant_kappa_invariants_2_3_12n+1_12n+5}
Let $m\geq 2$ be an integer and let $Y=\pm\Sigma(2,3,12n+1)$ or $\pm\Sigma(2,3,12n+5)$. Then $(Y,\rho_{m})$ is both Floer $K_{G^{*}_{m}}$-split and Floer $\wedge^{2}$-$K_{G^{*}_{m}}$-split, and the equivariant $\kappa$-invariants of $(Y,\rho_{m})$ are given by
    \[\K(Y,\rho_{m})=\K^{\wedge}(Y,\rho_{m})=\big\{[\DDD^{\odd}(\vec{n}(Y,\wh{\rho}_{m},g,\nabla^{\infty}))]\big\}\subset\Q^{m}_{\odd}.\]
\end{corollary}

We will next focus on the case where $Y=\pm\Sigma(2,3,12n-1)$ or $\pm\Sigma(2,3,12n-5)$. In the case of the involution $\rho_{2}$, we have the following result:

\begin{proposition}
\label{prop:kappa_tilde_rho_2}
Let $Y=\pm\Sigma(2,3,6n\pm 1)$. Then $\wt{\kappa}(Y,\rho_{2})=\kappa(Y)$.
\end{proposition}

\begin{proof}
The case $Y=\pm\Sigma(2,3,12n+1)$ or $\pm\Sigma(2,3,12n+5)$ follows from Corollary \ref{cor:equivariant_kappa_invariants_2_3_12n+1_12n+5}, and the case $Y=\pm\Sigma(2,3,12n-1)$ or $\pm\Sigma(2,3,12n-5)$ follows from Proposition \ref{prop:calculation_spectrum_2_3_k_rho_m} and Example \ref{ex:k_tilde_multiple_cosets_duals}.
\end{proof}

We also have the following partial calculation in the case where $m=p$ is an odd prime, depending on the rotation numbers $r_{1},\dots,r_{n}$:

\begin{proposition}
\label{prop:kappa_invariants_12n-1_12n-5_odd_primes}
Let $m=p$ be an odd prime, let $Y=\Sigma(2,3,12n-1)$ or $\Sigma(2,3,12n-5)$, let $r_{1},\dots,r_{n}$ denote the corresponding set of rotation numbers as in Proposition \ref{prop:calculation_spectrum_2_3_k_rho_m}, and let
\[\vec{n}(Y,p):=\DDD^{\odd}(\vec{n}(Y,\wh{\rho}_{p},g,\nabla^{\infty}))\in\QQ^{p},\qquad\vec{n}^{\pi}(Y,p):=\pi([\vec{n}(Y,p)])\in\QQ^{2}.\]
Then
\begin{align*}
    &\K(Y,\rho_{p}) \\
    &\qquad={
    \left\{
    \begin{array}{ll}
    \{[2\vec{e}_{0}],[2\vec{e}_{2r}],[2\vec{e}_{p-2r}]\}-[\vec{n}(Y,p)] & \mbox{if }\exists\,r\in\frac{1}{2}\ZZ\setminus\ZZ\text{ such that} \\
    & \qquad r_{i}\equiv\pm r\pmod{p}\;\;\forall i=1,\dots,n,\;\\
    \{[2\vec{e}_{0}]\}-[\vec{n}(Y,p)] & \mbox{otherwise,}
    \end{array}
    \right.} \\
    &\K(-Y,\rho_{p})=\Big\{2[\vec{a}]+[\vec{n}(Y,p)]\;\Big|\;\vec{a}\succeq(0,n_{1},\dots,n_{p-1}),\;|\vec{a}|=n\Big\},
\end{align*}
where
\[n_{j}:=\#\{1\le k\le n\;|\;2r_{k}\equiv j\pmod{p}\},\qquad 0\le j\le p-1.\]
In particular:
\begin{align*}
    &\K^{\pi}(Y,\rho_{p}) \\
    &\qquad={
    \left\{
    \begin{array}{ll}
    \{(2,0),(0,2)\}-\vec{n}^{\pi}(Y,p) & \mbox{if }\exists\,r\in\frac{1}{2}\ZZ\setminus\ZZ, r\neq\frac{p}{2}\text{ such that} \\
    & \qquad r_{i}\equiv\pm r\pmod{p}\;\;\forall i=1,\dots,n,\\
    \{(2,0)\}-\vec{n}^{\pi}(Y,p) & \mbox{otherwise,}
    \end{array}
    \right.} \\
    &\K^{\pi}(-Y,\rho_{p})=\Big\{(2k,-2k)+\vec{n}^{\pi}(Y,p)\;\Big|\;0\le k\le n-n_{0}\Big\}.
\end{align*}
\end{proposition}

\begin{proof}
Follows from Proposition \ref{prop:calculation_spectrum_2_3_k_rho_m} and the calculations in Example \ref{ex:odd_prime}.
\end{proof}

\subsubsection{The Involution \texorpdfstring{$\iota_{c}$}{ιc}}
\label{subsubsec:brieskorn_spheres_iota_c}

Next we calculate the $G^{\odd}_{2}$-equivariant Seiberg--Witten Floer spectrum classes associated the odd-type involution $\iota_{c}$ on $Y=\pm\Sigma(2,3,6n\pm 1)$:

\begin{proposition}
\label{prop:calculation_spectrum_2_3_k_iota_c}
For $Y=\pm\Sigma(2,3,12n\pm 1)$ or $\pm\Sigma(2,3,12n\pm 5)$, and let $\M^{\irr}=\{x_{1,+},x_{1,-},\dots,x_{n,+},x_{n,-}\}$ be an enumeration of the $n$ pairs of irreducible Seiberg--Witten solutions on $Y$. Let $\wh{\iota}_{c}$ be a spin lift of $\iota_{c}$, and for each $k=1,\dots,n$, let $\varepsilon_{k}\in\{\pm 1\}$ be as in Proposition \ref{prop:based_moduli_space_seifert_fibered}. Furthermore, let $a_{+}$ and $a_{-}$ denote the number of $+1$'s (respectively, $-1$'s) appearing among $\varepsilon_{1},\dots,\varepsilon_{n}$.
\begin{enumerate}
    \item Let $Y=\Sigma(2,3,12n-1)$ or $\Sigma(2,3,12n-5)$. Then the $G^{\odd}_{2}$-equivariant Seiberg--Witten Floer spectrum class of $(Y,\wh{\iota}_{c})$ is given by
    \[\SWF(Y,\wh{\iota}_{c})=\Big[\Big(\wt{\Sigma}Z_{\varepsilon_{1}j,\dots,\varepsilon_{n}j},0,\tfrac{1}{4}n(Y,g,\nabla^{\infty})(\xi+\xi^{3})\Big)\Big],\]
    and the spectrum class of $(-Y,\wh{\iota}_{c})$ is given by
    \[\SWF(-Y,\wh{\iota}_{c})=\Big[\Big(\wt{\Sigma}X_{\varepsilon_{1}j,\dots,\varepsilon_{n}j},0,-\tfrac{1}{4}n(Y,g,\nabla^{\infty})(\xi+\xi^{3})+(a_{+}\xi+a_{-}\xi^{3})\Big)\Big].\]
    \item Let $Y=\Sigma(2,3,12n+1)$ or $\Sigma(2,3,12n+5)$. Then the $G^{\odd}_{2}$-equivariant Seiberg--Witten Floer spectrum class of $(Y,\wh{\iota}_{c})$ is given by
    \begin{align*}
        &\SWF(Y,\wh{\iota}_{c}) \\
        &\qquad=\Big[\Big(S^{0}\vee\bigvee_{k=1}^{n}\Sigma^{-1}(Z_{\varepsilon_{k}j})_{+},0,\tfrac{1}{4}n(Y,g,\nabla^{\infty})(\xi+\xi^{3})\Big)\Big] \\
        &\qquad:=\Big[\Big(S^{a_{+}\HH_{1/2}+a_{-}\HH_{3/2}}\vee\bigvee_{\substack{k=1 \\ \varepsilon_{k}=+1}}^{n}\Sigma^{3\RR+(a_{+}-1)\HH_{1/2}+a_{-}\HH_{3/2}}(Z_{\varepsilon_{k}j})_{+}\vee \\
        &\qquad\qquad\bigvee_{\substack{k=1 \\ \varepsilon_{k}=-1}}^{n}\Sigma^{3\RR+a_{+}\HH_{1/2}+(a_{-}-1)\HH_{3/2}}(Z_{\varepsilon_{k}j})_{+},0,\tfrac{1}{4}n(Y,g,\nabla^{\infty})(\xi+\xi^{3})+(a_{+}\xi+a_{-}\xi^{3})\Big)\Big],
    \end{align*}
    and the spectrum class of $(-Y,\wh{\iota}_{c})$ is given by
    \[\SWF(-Y,\wh{\iota}_{c})=\Big[\Big(S^{0}\vee\bigvee_{k=1}^{n}(Z_{\varepsilon_{k}j})_{+},0,-\tfrac{1}{4}n(Y,g,\nabla^{\infty})(\xi+\xi^{3})\Big)\Big].\]
    Consequently, on the level of local equivalence we have that
    \[\SWF(\pm Y,\wh{\iota}_{c})\equiv_{\ell}\Big[\Big(S^{0},0,\pm\tfrac{1}{4}n(Y,g,\nabla^{\infty})(\xi+\xi^{3})\Big)\Big].\]
\end{enumerate}
\end{proposition}

\begin{proof}
The proof is much the same as the proof of Proposition \ref{prop:calculation_spectrum_2_3_k_iota_c}, except the $G^{\odd}_{2}$-cell corresponding to the $k$-th pair of irreducibles $\{x_{k,+},x_{k,-}\}$ is given by $Z_{\varepsilon_{k}j}=G^{\odd}_{2}/\<-\varepsilon_{k}j\mu\>$. The only difference is the determination of the attaching maps from the irreducibles to the reducible cell $S^{0}$, each of which are given by an element of
\begin{equation*}
\label{eq:iota_c_attaching_map}
    \{(Z_{\varepsilon_{k}j})_{+},S^{0}\}_{G^{\odd}_{2}}\cong \{S^{0},S^{0}\}_{\<-\varepsilon_{k}j\mu\>}\cong\{S^{0},S^{0}\}_{\ZZ_{2}}=\pi_{0}^{\st,\ZZ_{2}}(S^{0})=\ZZ^{2}=\ZZ\<\gamma_{0},\gamma_{1}\>.
\end{equation*}
But by similar arguments as in the proof of Proposition \ref{prop:floer_spectrum_seifert_fibered}, the attaching map must lie in the summand $\ZZ\<\gamma_{0}\><\ZZ\<\gamma_{0},\gamma_{1}\>$ corresponding to the trivial element of $\ZZ_{2}$. The rest of the argument proceeds as in the case of $\rho_{m}$.

Finally, the fact that
\[\tfrac{1}{2}n(Y,\wh{\iota}_{c},g,\nabla^{\infty})=\tfrac{1}{4}n(Y,g,\nabla^{\infty})(\xi+\xi^{3})\in R(\ZZ_{4})^{\odd,\sym}\otimes\QQ\]
follows from Proposition \ref{prop:correction_term_involutions}.
\end{proof}

\begin{proposition}
\label{prop:kappa_tilde_iota_c}
Let $Y=\pm\Sigma(2,3,6n\pm 1)$. Then $\wt{\kappa}(Y,\iota_{c})=\kappa(Y)$.
\end{proposition}

\begin{proof}
Follows from Proposition \ref{prop:calculation_spectrum_2_3_k_iota_c} and Example \ref{ex:k_tilde_multiple_cosets_duals}.
\end{proof}

\bigskip
\subsubsection{\texorpdfstring{$\<j\mu\>$}{<jμ>}-fixed points for \texorpdfstring{$\rho_{2}$}{rho2} and \texorpdfstring{$\iota_{c}$}{ιc}}
\label{subsubsec:jmu_fixed_points}

In this section we will calculate the $\<j\mu\>$-fixed point sets of the spectrum classes associated to the odd-type involutions $\rho_{2}$ and $\iota_{c}$ on $Y=\pm\Sigma(2,3,6n\pm 1)$ as $\ZZ_{4}$-equivariant spaces under the residual action of $j$, and use these to determine the doubled Seiberg--Witten Floer spectrum $\DSWF$ from \cite{KMT}. Afterwards, we then proceed to calculate $\kappa_{\KMT}(Y,\iota_{c})$, as well as $\kappa_{\KMT}$ of equivariant connected sums of Brieskorn spheres belonging to this family.

The next lemma essentially follows from (\cite{KMT}, Theorem 3.58):

\begin{lemma}
\label{lemma:brieskorn_rho_2_jmu_fixed_point_sets_rho_2}
Let $Y=\pm\Sigma(2,3,6n\pm 1)$ and let $\rho_{2}:Y\to Y$ be the odd-type involution as above. Then for any spin lift $\wh{\rho}_{2}$ of $\rho_{2}$,
\[\SWF(Y,\wh{\rho}_{2})^{\<j\mu\>}=\big[\big(S^{0},0,\tfrac{1}{4}n(Y,g,\nabla^{\infty})(\xi+\xi^{3})\big)\big]\in\CCC_{\ZZ_{4},\CC}.\]
In particular, $(Y,\rho_{2})$ is $\<j\mu\>$-spherical. Consequently,
\[\DSWF(Y,\wh{\rho}_{2})^{\<j\mu\>}=\big[\big(S^{0},0,\tfrac{1}{2}n(Y,g,\nabla^{\infty})\big)\big]\in\CCC_{\ZZ_{4},\CC,\sym},\]
and
\[\kappa_{\KMT}(Y,\rho_{2})=-\tfrac{1}{2}n(Y,g,\nabla^{\infty})=-\tfrac{1}{2}\ol{\mu}(Y),\]
where $\ol{\mu}(Y)$ denotes the Neumann-Siebenmann invariant of $Y$.
\end{lemma}

\begin{remark}
The proof of (\cite{KMT}, Theorem 3.58) applies to all Seifert-fibered rational homology spheres, and only uses the fact that $\rho_{2}$ is isotopic to the identity. However, one can show the above lemma directly from our calculations of the spectrum class $\SWF(Y,\wh{\rho}_{2})$ by using the fact that $j\mu\in G^{\odd}_{2}$ acts freely on the spaces $\wt{\Sigma}Z_{a_{1},\dots,a_{n};2}$ and $\wt{\Sigma}X_{a_{1},\dots,a_{n};2}$ away from the $S^{1}$-fixed point set for any $a_{1},\dots,a_{n}\in\{\frac{1}{2},\frac{3}{2}\}$.
\end{remark}

In contrast to $\rho_{2}$, the involution $\iota_{c}$ is never isotopic to the identity, except in the exceptional case $Y=\pm\Sigma(2,3,5)$. We shall proceed to calculate the $\ZZ_{4}$-equivariant spectrum classes $\SWF(Y,\wh{\iota}_{c})^{\<j\mu\>}$ for $Y=\pm\Sigma(2,3,6n\pm 1)$, as well as their corresponding doubles
\[\DSWF(Y,\wh{\iota}_{c})^{\<j\mu\>}:=\big(\SWF(Y,\wh{\iota}_{c})^{\<j\mu\>}\big)\wedge\big(\SWF(Y,\wh{\iota}_{c})^{\<j\mu\>}\big)^{\dagger}.\]
In contrast to the result of Lemma \ref{lemma:brieskorn_rho_2_jmu_fixed_point_sets_rho_2}, these spectrum classes have a more interesting structure.

We will first look at the $\<j\mu\>$-fixed point sets of the model space 
\[\wt{\Sigma}Z_{\varepsilon_{1}j,\dots,\varepsilon_{n}j}=\wt{\Sigma}\Big(Z_{\varepsilon_{1}j}\amalg\cdots\amalg Z_{\varepsilon_{n}j}\Big)\]
from Example \ref{ex:multiple_cosets_cyclic}. Recall that the action of $\mu\in G^{\odd}_{2}$ on $Z_{j}=G^{\odd}_{2}/\<-j\mu\>$ coincides with multiplication by $j$ on the right, and similarly the action of $\mu$ on $Z_{-j}=G^{\odd}_{2}/\<j\mu\>$ coincides with multiplication by $-j$ on the right. Hence the action of $j\mu$ on $Z_{j}$ is given by $x\mapsto jxj$, and on $Z_{-j}$ is given by $x\mapsto -jxj$. We see that under the canonical identifications of $Z_{\pm j}$ with $\Pin(2)$ as $\Pin(2)$-spaces, we have that
\begin{align*}
    &Z_{j}^{\<j\mu\>}=\{i,ji,-i,-ji\}, & &Z_{-j}^{\<j\mu\>}=\{1,j,-1,-j\}.
\end{align*}
Hence under the residual $\<j\>\cong\ZZ_{4}$-actions we have identifications $Z_{j}^{\<j\mu\>}\cong Z_{-j}^{\<j\mu\>}\cong \ZZ_{4}$. This implies that $(\wt{\Sigma}Z_{j})^{\<j\mu\>}\cong(\wt{\Sigma}Z_{j})^{\<j\mu\>}\cong\wt{\Sigma}\ZZ_{4}$, and more generally we have that
\[(\wt{\Sigma}Z_{(\varepsilon_{1}j,\dots,\varepsilon_{n}j)})^{\<j\mu\>}\cong\wt{\Sigma}(\amalg_{n}\ZZ_{4}),\]
where $\amalg_{n}\ZZ_{4}$ denotes the disjoint union of $n$ copies of $\ZZ_{4}$.

\begin{lemma}
\label{lemma:self_dual_Z_4}
The $\ZZ_{4}$-space $\wt{\Sigma}(\amalg_{n}\ZZ_{4})$ is both equivariantly $\CC_{1/2}$-self-dual and $\CC_{3/2}$-self-dual.
\end{lemma}

\begin{proof}
As real representation spaces we have that $\CC_{1/2}\cong\CC_{3/2}\cong\VV_{1/2}$, so it suffices to show that $\wt{\Sigma}(\amalg_{n}\ZZ_{4})$ is equivariantly $\VV_{1/2}$-self-dual.

Note that we have a canonical embedding of $\amalg_{n}\ZZ_{4}\hookrightarrow S(\VV_{1/2})\cong S^{1}$ as the $4n$-th roots of unity, whose complement in $S(\VV_{1/2})$ equivariantly deformation retracts onto $\amalg_{n}\ZZ_{4}$. The result then follows from Lemma \ref{lemma:duality_embeddings}.
\end{proof}

We leave the proof of the following lemma as an exercise to the reader:

\begin{lemma}
\label{lemma:Z_4_spaces_homotopy_equivalent}
The spaces $\wt{\Sigma}(\amalg_{n}\ZZ_{4})$ and $\wt{\Sigma}\ZZ_{4}\vee\bigvee^{n-1}\Sigma(\ZZ_{4})_{+}$ are $\ZZ_{4}$-equivariantly stably homotopy equivalent.
\end{lemma}

This leads us to the following proposition:

\begin{proposition}
\label{prop:jmu_fixed_points_spectrum_classes}
The following statements are true:
\begin{enumerate}
    \item Let $Y=\Sigma(2,3,12n-1)$ or $\Sigma(2,3,12n-5)$. Then for any spin lift $\wh{\iota}_{c}$ of $\iota_{c}$, the $\<j\mu\>$-fixed point spectrum classes are given by
    \begin{align*}
        &\SWF(Y,\wh{\iota}_{c})^{\<j\mu\>}=\Big[\Big(\wt{\Sigma}\ZZ_{4}\vee\vee^{n-1}\Sigma(\ZZ_{4})_{+},0,\tfrac{1}{4}n(Y,g,\nabla^{\infty})(\xi+\xi^{3})\Big)\Big], \\
        &\SWF(-Y,\wh{\iota}_{c})^{\<j\mu\>}=\Big[\Big(\wt{\Sigma}\ZZ_{4}\vee\vee^{n-1}\Sigma(\ZZ_{4})_{+},0,-\tfrac{1}{4}n(Y,g,\nabla^{\infty})(\xi+\xi^{3})+\xi\Big)\Big],
    \end{align*}
    and the corresponding doubled spectrum classes are given by
    \begin{align*}
        &\DSWF(Y,\wh{\iota}_{c})^{\<j\mu\>}=\Big[\Big(\VV_{1/2}^{+}\vee(\vee^{2n-2}\wt{\Sigma}\ZZ_{4}\wedge\Sigma(\ZZ_{4})_{+}) \\
        &\qquad\qquad\qquad\qquad\qquad\qquad\vee(\vee^{(n-1)^{2}}\wedge^{2}\Sigma(\ZZ_{4})_{+})),0,\tfrac{1}{2}n(Y,g,\nabla^{\infty})\Big)\Big], \\
        &\DSWF(-Y,\wh{\iota}_{c})^{\<j\mu\>}=\Big[\Big(\VV_{1/2}^{+}\vee(\vee^{2n-2}\wt{\Sigma}\ZZ_{4}\wedge\Sigma(\ZZ_{4})_{+}) \\
        &\qquad\qquad\qquad\qquad\qquad\qquad\vee(\vee^{(n-1)^{2}}\wedge^{2}\Sigma(\ZZ_{4})_{+})),0,-\tfrac{1}{2}n(Y,g,\nabla^{\infty})+1\Big)\Big], \\
    \end{align*}
    as elements of $\CCC_{\ZZ_{4},\CC,\sym}$. Consequently on the level of local equivalence we have that
    \begin{align*}
        &\SWF(Y,\wh{\iota}_{c})^{\<j\mu\>}\equiv_{\ell}\Big[\Big(\wt{\Sigma}\ZZ_{4},0,\tfrac{1}{4}n(Y,g,\nabla^{\infty})(\xi+\xi^{3})\Big)\Big], \\
        &\SWF(-Y,\wh{\iota}_{c})^{\<j\mu\>}\equiv_{\ell}\Big[\Big(\wt{\Sigma}\ZZ_{4},0,-\tfrac{1}{4}n(Y,g,\nabla^{\infty})(\xi+\xi^{3})+\xi\Big)\Big],
    \end{align*}
    and
    \begin{align*}
        &\DSWF(Y,\wh{\iota}_{c})^{\<j\mu\>}\equiv_{\ell}\Big[\Big(\VV_{1/2}^{+},0,\tfrac{1}{2}n(Y,g,\nabla^{\infty})\Big)\Big], \\
        &\DSWF(-Y,\wh{\iota}_{c})^{\<j\mu\>}\equiv_{\ell}\Big[\Big(\VV_{1/2}^{+},0,-\tfrac{1}{2}n(Y,g,\nabla^{\infty})+1\Big)\Big]. \\
    \end{align*}
    \item Let $Y=\Sigma(2,3,12n+1)$ or $\Sigma(2,3,12n+5)$. Then for any spin lift $\wh{\iota}_{c}$ of $\iota_{c}$, the $\<j\mu\>$-fixed point spectrum classes are given by
    \begin{align*}
        &\SWF(Y,\wh{\iota}_{c})^{\<j\mu\>}=\Big[\Big(\CC_{1/2}^{+}\vee(\vee^{n}\Sigma(\ZZ_{4})_{+}),0,\tfrac{1}{4}n(Y,g,\nabla^{\infty})(\xi+\xi^{3})+\xi\Big)\Big], \\
        &\SWF(-Y,\wh{\iota}_{c})^{\<j\mu\>}=\Big[\Big(S^{0}\vee(\vee^{n}(\ZZ_{4})_{+}),0,-\tfrac{1}{4}n(Y,g,\nabla^{\infty})(\xi+\xi^{3})\Big)\Big],
    \end{align*}
    and the corresponding doubled spectrum classes are given by
    \begin{align*}
        &\DSWF(Y,\wh{\iota}_{c})^{\<j\mu\>}=\Big[\Big((\CC_{1/2}\oplus\CC_{3/2})^{+}\vee(\vee^{2n}\Sigma^{\CC_{1/2}}\Sigma(\ZZ_{4})_{+}) \\
        &\qquad\qquad\qquad\qquad\qquad\qquad\vee(\vee^{n^{2}}(\wedge^{2}\Sigma(\ZZ_{4})_{+})),0,\tfrac{1}{2}n(Y,g,\nabla^{\infty})+1\Big)\Big], \\
        &\DSWF(-Y,\wh{\iota}_{c})^{\<j\mu\>}=\Big[\Big(S^{0}\vee(\vee^{2n}(\ZZ_{4})_{+}) \\
        &\qquad\qquad\qquad\qquad\qquad\qquad\vee(\vee^{n^{2}}(\ZZ_{4}\times\ZZ_{4})_{+}),0,-\tfrac{1}{2}n(Y,g,\nabla^{\infty})\Big)\Big].
    \end{align*}
    Consequently on the level of local equivalence we have that
    \begin{align*}
        &\SWF(\pm Y,\wh{\iota}_{c})^{\<j\mu\>}\equiv_{\ell}\Big[\Big(S^{0},0,\pm\tfrac{1}{4}n( Y,g,\nabla^{\infty})(\xi+\xi^{3})\Big)\Big], \\
        &\DSWF(\pm Y,\wh{\iota}_{c})^{\<j\mu\>}\equiv_{\ell}\Big[\Big(S^{0},0,\pm\tfrac{1}{2}n( Y,g,\nabla^{\infty})\Big)\Big].
    \end{align*}
\end{enumerate}
\end{proposition}

\begin{proof}
Follows from Proposition \ref{prop:calculation_spectrum_2_3_k_iota_c}, and Lemmas \ref{lemma:self_dual_Z_4} and \ref{lemma:Z_4_spaces_homotopy_equivalent}.
\end{proof}

From the above proposition we can conclude the following:

\begin{proposition}
\label{prop:kappa_KMT_complex_conjugation}
Let $Y=\pm\Sigma(2,3,6n\pm 1)$. Then $\kappa_{\KMT}(Y,\iota_{c})=\frac{1}{2}\kappa(Y)$.
\end{proposition}

\begin{proof}
We first show that $k_{\KMT}(\VV_{1/2})^{+})=1$. Indeed, let $R(\ZZ_{4})=\ZZ[t]/(t^{4}-1)$ and let
\begin{align*}
	&z=1-t=\lambda_{-1}(\CC_{1/2}), & &w=1-t^{2}=\lambda_{-1}(\wt{\CC}), &
	&w+z-wz=1-t^{3}=\lambda_{-1}(\CC_{3/2}),
\end{align*}
as in (\cite{KMT}, Section 3.1). (Here we use $\CC_{1/2}$ and $\CC_{3/2}$ instead of $\CC_{+}$ and $\CC_{-}$.) Note that the two different complex structures on $\VV_{1/2}$ corresponding to $\CC_{1/2}$ and $\CC_{3/2}$, respectively, induce two distinct isomorphisms 
\begin{align*}
	&f_{1/2}:\wt{K}_{\ZZ_{4}}\big(\VV_{1/2}^{+}\big)\xrightarrow{\cong}R(\ZZ_{4}) & &f_{3/2}:\wt{K}_{\ZZ_{4}}\big((\VV_{1/2}^{+}\big)\xrightarrow{\cong}R(\ZZ_{4})
\end{align*}
such that $f_{3/2}^{-1}\circ f_{1/2}$ sends $z\mapsto w+z-wz$ and vice-versa. We can think of these maps as coming from two distinct Bott elements $b_{\CC_{1/2}},b_{\CC_{3/2}}\in\wt{K}_{\ZZ_{4}}\big(\VV_{1/2}^{+}\big)$. Depending on our identification, we either have that
\[\III\big(\VV_{1/2}^{+}\big)=(z)\text{ or }(w+z-wz).\]
In either case, we have that 
\[k_{\KMT}\big(\VV_{1/2}^{+}\big)=\min\{k\geq 0\;|\;\exists x\in\III(\VV_{1/2}^{+}),\;wx=2^{k}w\}=1,\]
as claimed. Hence by the calculations of the $\DSWF$ local equivalence classes in Proposition \ref{prop:jmu_fixed_points_spectrum_classes}, we obtain
\begin{align*}
    &\kappa_{\KMT}(\Sigma(2,3,12n-1),\iota_{c})=1=\tfrac{1}{2}\kappa(\Sigma(2,3,12n-1)), \\
    &\kappa_{\KMT}(-\Sigma(2,3,12n-1),\iota_{c})=0=\tfrac{1}{2}\kappa(-\Sigma(2,3,12n-1)), \\
    &\kappa_{\KMT}(\Sigma(2,3,12n-5),\iota_{c})=\tfrac{1}{2}=\tfrac{1}{2}\kappa(\Sigma(2,3,12n-5)), \\
    &\kappa_{\KMT}(-\Sigma(2,3,12n-5),\iota_{c})=\tfrac{1}{2}=\tfrac{1}{2}\kappa(-\Sigma(2,3,12n-5), \\
    &\kappa_{\KMT}(\Sigma(2,3,12n+1),\iota_{c})=0=\tfrac{1}{2}\kappa(\Sigma(2,3,12n+1)), \\
    &\kappa_{\KMT}(-\Sigma(2,3,12n+1),\iota_{c})=0=\tfrac{1}{2}\kappa(-\Sigma(2,3,12n+1)), \\
    &\kappa_{\KMT}(\Sigma(2,3,12n+5),\iota_{c})=\tfrac{1}{2}=\tfrac{1}{2}\kappa(\Sigma(2,3,12n+5)), \\
    &\kappa_{\KMT}(-\Sigma(2,3,12n+5),\iota_{c})=-\tfrac{1}{2}=\tfrac{1}{2}\kappa(-\Sigma(2,3,12n+5)), \\
\end{align*}
as desired.
\end{proof}

\begin{proposition}
\label{prop:not_jmu_spherical}
Let $Y=\Sigma(2,3,12n-1)$ or $\Sigma(2,3,12n-5)$. Then neither $(Y,\iota_{c})$ nor $(-Y,\iota_{c})$ are locally $\SWF$-$\<j\mu\>$-spherical.
\end{proposition}

Before we prove Proposition \ref{prop:not_jmu_spherical}, the following lemma will be useful:

\begin{lemma}
\label{lemma:must_be_isomorphic}
Let $V_{0}$, $V_{1}$ be two $\ZZ_{4}$-representations such that the representation spheres $V_{0}^{+}$, $V_{1}^{+}$ are spaces of type $\CC$-$\ZZ_{4}$-$\SWF$. Then $[V_{0}^{+}]_{\loc}=[V_{1}^{+}]_{\loc}$ if and only if $V_{0}\cong V_{1}$.
\end{lemma}

\begin{proof}
As $V_{0}^{+}$, $V_{1}^{+}$ are spaces of type $\CC$-$\ZZ_{4}$-$\SWF$ and $(V_{0}^{+})^{\ZZ_{2}}\simeq_{\ZZ_{2}}(V_{1}^{+})^{\ZZ_{2}}$, it follows that there exists some $p,q,r\geq 0$ such that
\begin{align*}
	&V_{0}\cong\wt{\RR}^{2p}\oplus\VV_{1/2}^{q}, & &V_{1}\cong\wt{\RR}^{2p}\oplus\VV_{1/2}^{r}.
\end{align*}
Suppose $q<r$. Using the fact that $k_{\KMT}(\VV_{1/2}^{+})=1$ and the fact that the invariant $k_{\KMT}$ respects local equivalence classes, we must have that $r=q+1$. But by a result of Crabbe (\cite{Crabb:Z/4}), the image of the $\ZZ_{4}$-fixed point homomorphism
\[\pi^{\st,\ZZ_{4}}_{\nu}(S^{0})\to\pi^{\st}_{0}(S^{0})\cong\ZZ\]
is contained in $4\ZZ\subset\ZZ$, implying that no $\ZZ_{4}$-equivariant map
\[f:\big(\RR^{A}\oplus\wt{\RR}^{B+2p}\oplus\VV_{1/2}^{C+q+1}\big)^{+}\to\big(\RR^{A}\oplus\wt{\RR}^{B+2p}\oplus\VV_{1/2}^{C+q}\big)^{+}\]
for $A,B,C>>0$ sufficiently large can induce a homotopy equivalence on $\ZZ_{2}$-fixed points. Hence by symmetry we must have that $q=r$, implying that $V_{0}\cong V_{1}$.
\end{proof}

\begin{proof}[Proof of Proposition \ref{prop:not_jmu_spherical}]
Suppose $[\wt{\Sigma}\ZZ_{4}]_{\loc}=[V^{+}]_{\loc}$ for some $\ZZ_{4}$-representation sphere $V^{+}$. Then by Lemma \ref{lemma:self_dual_Z_4} we have that 
\[[(V^{2})^{+}]_{\loc}=[\wedge^{2}\wt{\Sigma}\ZZ_{4}]_{\loc}=[\VV_{1/2}^{+}]_{\loc}.\]
By Lemma \ref{lemma:must_be_isomorphic} we must have that $\VV_{1/2}\cong V^{2}$, a contradiction since $\VV_{1/2}$ is irreducible. Finally since $\SWF(\pm Y,\wh{\iota}_{c})^{\<j\mu\>}$ is $\ZZ_{4}$-locally equivalent to a (de)suspension of $\wt{\Sigma}\ZZ_{4}$ for $Y=\Sigma(2,3,12n-1)$ or $\Sigma(2,3,12n-5)$, the result follows.
\end{proof}

\begin{proposition}
\label{prop:RO(Z_4)_graded_homotopy}
Let $Y=\Sigma(2,3,12n-1)$ or $\Sigma(2,3,12n-5)$. Then there exists $\X\in\CCC_{G^{\odd}_{2},\CC}$ locally equivalent to $\SWF(\pm Y,\wh{\iota}_{c})$ such that
\[\res^{\ZZ_{4}}_{1}\big(\pi^{\st,\ZZ_{4}}_{s\rho+t\nu}(\X^{\<j\mu\>})\otimes\QQ\big)=0\qquad\text{ for all }s\in\ZZ,t\in\QQ.\]
\end{proposition}

\begin{proof}
From our calculations above it suffices to show that 
\[\res^{\ZZ_{4}}_{1}\big(\pi^{\st,\ZZ_{4}}_{s\rho+t\nu}(\wt{\Sigma}\ZZ_{4})\otimes\QQ\big)=0\qquad\text{ for all }s\in\ZZ,t\in\ZZ.\]
Note that we have a non-equivariant homotopy equivalence $\wt{\Sigma}\ZZ_{4}\simeq \vee^{3}S^{1}$. Since $\pi^{\st}_{*}(\vee^{3}S^{1})\otimes\QQ\neq 0$ if and only if $*=1$, then the only possible degree in which the above map could be non-zero is $(s,t)=(1,0)$. But by inspection any $\ZZ_{4}$-equivariant map $\wt{\RR}^{+}\to\wt{\Sigma}\ZZ_{4}$ must be null-homotopic, thus the claim follows.
\end{proof}

We assemble all of the calculations from this section combined with the results from \cite{KMT} into the following proposition:

\begin{proposition}
\label{prop:brieskorn_spheres_involutions_properties}
Let $Y=\pm\Sigma(2,3,6n\pm 1)$ and $\rho_{2}$, $\iota_{c}$ be as above. Then
\begin{align*}
    &\wt{\kappa}(Y,\rho_{2})=\wt{\kappa}(Y,\iota_{c})=\kappa(Y), &
    &\kappa_{\KMT}(Y,\rho_{2})=-\tfrac{1}{2}\ol{\mu}(Y), & &\kappa_{\KMT}(Y,\iota_{c})=\tfrac{1}{2}\kappa(Y).
\end{align*}
Furthermore:
\begin{enumerate}
    \item For $Y=\pm\Sigma(2,3,12n-1)$ or $\pm\Sigma(2,3,12n-5)$, the pairs $(Y,\rho_{2})$ and $(Y,\iota_{c})$ are not locally $\SWF$-spherical.
    \item For $Y=\pm\Sigma(2,3,12n+1)$ or $\pm\Sigma(2,3,12n+5)$, the pairs $(Y,\rho_{2})$ and $(Y,\iota_{c})$ are locally $\SWF$-spherical but not $\SWF$-spherical for $n\geq 1$.
    \item For $Y=\pm\Sigma(2,3,6n\pm 1)$, the pair $(Y,\rho_{2})$ is $\SWF$-$\<j\mu\>$-spherical.
    \item For $Y=\pm\Sigma(2,3,12n-1)$ or $\pm\Sigma(2,3,12n-5)$, the pair $(Y,\iota_{c})$ is not locally $\SWF$-$\<j\mu\>$-spherical.
    \item For $Y=\pm\Sigma(2,3,12n+1)$ or $\pm\Sigma(2,3,12n+5)$, the pair $(Y,\iota_{c})$ is locally $\SWF$-$\<j\mu\>$-spherical but not $\SWF$-$\<j\mu\>$-spherical for $n\geq 1$.
\end{enumerate}
\end{proposition}

\subsection{Classes of Equivariant Spin Rational Homology Spheres}
\label{subsec:classes_equivariant_homology_spheres}

\begin{definition}
\label{def:classes_equivariant_manifolds}
For $\ast\in\{\ev,\odd\}$ and $m\geq 2$ a prime power, let $\SRH_{m,*}$ be the set of all $\ZZ_{m}$-equivariant spin rational homology spheres $(Y,\frak{s},\wh{\sigma})$ with $\wh{\sigma}$ of $\ast$ type, up to spin equivariant diiffeomorphism. We define the following subsets of $\SRH_{m,*}$:
\begin{enumerate}
    \item Let $\PS_{m,*}$ be set of triples $(Y,\frak{s},\wh{\sigma})$ which are $\Pin(2)$-surjective.
    \item Let $\SWFM_{m,*}$ to be the set of triples $(Y,\frak{s},\wh{\sigma})$ such that $Y$ admits a $\sigma$-equivariant metric $g$ such that $(Y,\frak{s},g)$ admits no irreducible Seiberg-Witten solutions.
    \item Let $\SWFM_{m,*}^{\#}$ be the closure of $\SWFM_{m,*}$ under spin $\ZZ_{m}$-equivariant connected sums.
    \item Let $\SWFS_{m,*}$ be the set of triples $(Y,\frak{s},\wh{\sigma})$ which are $\SWF$-spherical.
    \item Let $\SWFS_{m,*}^{\#}$ be the closure of $\SWFS_{m,*}$ under spin $\ZZ_{m}$-equivariant connected sums.
    \item Let $\LSWFS_{m,*}$ be the set of triples $(Y,\frak{s},\wh{\sigma})$ which are \emph{locally} $\SWF$-spherical.
\end{enumerate}
Additionally, for $H\subset G^{*}_{m}$ a closed subgroup:
\begin{enumerate}
    \item[(7)] Let $\SWFS_{m,*}^{H}$ be the set of triples $(Y,\frak{s},\wh{\sigma})$ which are $\SWF$-$H$-spherical.
    \item[(8)] Let $\SWFS_{m,*}^{\#,H}$ be the closure of $\SWFM_{m,*}^{H}$ under spin $\ZZ_{m}$-equivariant connected sums.
    \item[(9)] Let $\LSWFS_{m,*}^{H}$ be the set of triples $(Y,\frak{s},\wh{\sigma})$ which are \emph{locally} $\SWF$-$H$-spherical.
\end{enumerate}
\end{definition}

\begin{proposition}
\label{prop:inclusions_classes_equivariant_manifolds}
Let $\ast\in\{\ev,\odd\}$. For each $m\geq 2$ and each closed subgroup $H\subset G^{*}_{m}$, we have inclusions
\begin{center}
\begin{tikzcd}[remember picture]
    \SWFS_{m,*}^{H} &[-15pt] \SWFS_{m,*}^{\#,H} &[-15pt]  \LSWFS_{m,*}^{H} &[-15pt] \\[-10pt]
    \SWFS_{m,*}&[-15pt] \SWFS_{m,*}^{\#} &[-15pt] \LSWFS_{m,*} &[-15pt] \PS_{m}. \\[-10pt]
    \SWFM_{m,*} &[-15pt] \SWFM_{m,*}^{\#} &[-15pt] &[-15pt]
\end{tikzcd}
\begin{tikzpicture}[overlay,remember picture]
\path (\tikzcdmatrixname-1-1) to node[midway,sloped]{$\subset_{1}$}
(\tikzcdmatrixname-1-2);
\path (\tikzcdmatrixname-1-2) to node[midway,sloped]{$\subset_{2}$}
(\tikzcdmatrixname-1-3);
\path (\tikzcdmatrixname-2-1) to node[midway,sloped]{$\subset_{3}$}
(\tikzcdmatrixname-1-1);
\path (\tikzcdmatrixname-2-2) to node[midway,sloped]{$\subset_{4}$}
(\tikzcdmatrixname-1-2);
\path (\tikzcdmatrixname-2-3) to node[midway,sloped]{$\subset_{5}$}
(\tikzcdmatrixname-1-3);
\path (\tikzcdmatrixname-2-1) to node[midway,sloped]{$\subset_{6}$}
(\tikzcdmatrixname-2-2);
\path (\tikzcdmatrixname-2-2) to node[midway,sloped]{$\subset_{7}$}
(\tikzcdmatrixname-2-3);
\path (\tikzcdmatrixname-2-3) to node[midway,sloped]{$\subset_{8}$}
(\tikzcdmatrixname-2-4);
\path (\tikzcdmatrixname-3-1) to node[midway,sloped]{$\subset_{9}$}
(\tikzcdmatrixname-2-1);
\path (\tikzcdmatrixname-3-2) to node[midway,sloped]{$\subset_{10}$}
(\tikzcdmatrixname-2-2);
\path (\tikzcdmatrixname-3-1) to node[midway,sloped]{$\subset_{11}$}
(\tikzcdmatrixname-3-2);
\end{tikzpicture}
\end{center}
Furthermore in the case where $m=2$, $\ast=\odd$, and $H=\<j\mu\>$, the inclusions $\subset_{2}$, $\subset_{3}$, $\subset_{4}$, $\subset_{5}$, $\subset_{7}$, and $\subset_{8}$ are strict.
\end{proposition}

\begin{proof}
The existence of the above inclusions is clear. The final claim follows directly from Proposition \ref{prop:brieskorn_spheres_involutions_properties}.
\end{proof}

\begin{remark}
It is not clear to the author whether the inclusions $\subset_{1}$, $\subset_{6}$, $\subset_{11}$ are strict -- one would need a careful analysis of the behavior of the Seiberg-Witten-Floer moduli space under equivariant connected sum, which has not appeared in the literature, even in the non-equivariant setting. To show that the inclusions $\subset_{9}$, $\subset_{10}$ are strict, it would suffice to find a $\ZZ_{2}$-equivariant odd-type Riemannian spin rational homology sphere whose associated Seiberg-Witten moduli space consists of irreducibles which lie in a single spherical family. To the author's knowledge, it is not clear that this possibility can be ruled out.
\end{remark}

\smallskip
\subsection{Knot Invariants}
\label{subsec:knot_invariants}

Let $K\subset S^{3}$ be an oriented knot, and $p^{r}$ a prime power. It is a standard result in topology that the $p^{r}$-fold branched cover $\Sigma_{p^{r}}(K)$ is a $\ZZ_{p}$-homology sphere, hence in particular a rational homology sphere. Let $\sigma:\Sigma_{p^{r}}(K)\to\Sigma_{p^{r}}(K)$ denote the generator of the $p^{r}$-fold covering transformation determined by the orientation on $K$. By (\cite{GRS08}, Lemma 2.1) there is a canonical $\sigma$-invariant spin structure on $\Sigma_{p^{r}}(K)$, which we will henceforth denote by $\frak{s}_{0}$. We can therefore apply all of our constructions thus far to the triple $(\Sigma_{p^{r}}(K),\frak{s}_{0},\sigma)$ to obtain a family of invariants associated to $K\subset S^{3}$:

\begin{definition}
\label{def:knot_invariants}
Let $K\subset S^{3}$ be an oriented knot. We define the \emph{$p^{r}$-fold equivariant $\kappa$-invariants of $K$} as follows:
\begin{align*}
    &\K_{p^{r}}(K):=\K(\Sigma_{p^{r}}(K),\frak{s}_{0},\sigma), & &\K^{\wedge}_{p^{r}}(K):=\K^{\wedge}(\Sigma_{p^{r}}(K),\frak{s}_{0},\sigma), \\
    &\vec{\ul{\kappa}}_{p^{r}}(K):=\vec{\ul{\kappa}}(\Sigma_{p^{r}}(K),\frak{s}_{0},\sigma), & &\vec{\ul{\kappa}}_{p^{r}}^{\wedge}(K):=\vec{\ul{\kappa}}^{\wedge}(\Sigma_{p^{r}}(K),\frak{s}_{0},\sigma), \\
    &\vec{\ol{\kappa}}_{p^{r}}(K):=\vec{\ol{\kappa}}(\Sigma_{p^{r}}(K),\frak{s}_{0},\sigma), & &\vec{\ol{\kappa}}_{p^{r}}^{\wedge}(K):=\vec{\ol{\kappa}}^{\wedge}(\Sigma_{p^{r}}(K),\frak{s}_{0},\sigma).
\end{align*}
In the case where $p^{r}=2$, we define 
\[\wt{\kappa}(K):=\wt{\kappa}(\Sigma_{2}(K),\frak{s}_{0},\sigma),\]
and in the case where $p^{r}$ is odd, we define
\begin{align*}
	&\ul{\kappa}_{p^{r}}(K)_{0}:=\ul{\kappa}(\Sigma_{p^{r}}(K),\frak{s}_{0},\sigma)_{0}, & &\ul{\kappa}^{\wedge}_{p^{r}}(K)_{0}:=\ul{\kappa}^{\wedge}(\Sigma_{p^{r}}(K),\frak{s}_{0},\sigma)_{0}, \\
	&\ol{\kappa}_{p^{r}}(K)_{0}:=\ol{\kappa}(\Sigma_{p^{r}}(K),\frak{s}_{0},\sigma)_{0}, & &\ol{\kappa}^{\wedge}_{p^{r}}(K)_{0}:=\ol{\kappa}^{\wedge}(\Sigma_{p^{r}}(K),\frak{s}_{0},\sigma)_{0}. \\
	&\ul{\kappa}_{p^{r}}(K)_{\nt}:=\ul{\kappa}(\Sigma_{p^{r}}(K),\frak{s}_{0},\sigma)_{\nt}, & &\ul{\kappa}^{\wedge}_{p^{r}}(K)_{\nt}:=\ul{\kappa}^{\wedge}(\Sigma_{p^{r}}(K),\frak{s}_{0},\sigma)_{\nt}, \\
	&\ol{\kappa}_{p^{r}}(K)_{\nt}:=\ol{\kappa}(\Sigma_{p^{r}}(K),\frak{s}_{0},\sigma)_{\nt}, & &\ol{\kappa}^{\wedge}_{p^{r}}(K)_{\nt}:=\ol{\kappa}^{\wedge}(\Sigma_{p^{r}}(K),\frak{s}_{0},\sigma)_{\nt}.
\end{align*}
\end{definition}

\begin{proposition}
\label{prop:knot_concordance_local_equivalence}
Suppose $K,K'\subset S^{3}$ are smoothly concordant knots. Let $p^{r}$ be a prime power, and let $(\Sigma_{p^{r}}(K),\frak{s}_{0},\sigma)$, $(\Sigma_{p^{r}}(K'),\frak{s}'_{0},\sigma')$ denote the corresponding $p^{r}$-fold branched covers. Then there exist spin lifts $\wh{\sigma}$, $\wh{\sigma}'$ of $\sigma$, $\sigma'$ respectively such that
\[\big[\SWF(\Sigma_{p^{r}}(K),\frak{s}_{0},\wh{\sigma})\big]_{\loc}=\big[\SWF(\Sigma_{p^{r}}(K'),\frak{s}'_{0},\wh{\sigma}')\big]_{\loc}\in\LLL\EEE_{G^{*}_{p^{r}},\CC}.\]
Consequently, all of the equivariant $\kappa$-invariants introduced in Definition \ref{def:knot_invariants} are knot concordance invariants.
\end{proposition}

\begin{proof}
This follows from Proposition \ref{theorem:equivariant_homology_cobordism}, and the fact that if $F\subset S^{3}\times[0,1]$ is a concordance from $K$ to $K'$, then the $p^{r}$-fold branched cover of $S^{3}\times[0,1]$ over $F$ is a $\ZZ_{p^{r}}$-equivariant $\ZZ_{p}$-homology cobordism $(W,\tau)$ from $(\Sigma_{p^{r}}(K),\sigma)$ to $(\Sigma_{p^{r}}(K'),\sigma')$. Moreover, $W$ carries a unique invariant spin structure $\frak{t}_{0}$ which restricts to $\frak{s}_{0}$, $\frak{s}'_{0}$ on either side of the cobordism, hence the claim follows.
\end{proof}

\begin{remark}
Given a knot $K\subset Y$ for $Y$ an integer homology sphere, one can also define corresponding invariants $[\SWF_{p^{r}}(Y,K)]_{\loc}$, $\K_{p^{r}}(Y,K)$, $\K_{p^{r}}^{\wedge}(Y,K)$, etc., which are invariants of the \emph{homology concordance class} of $(Y,K)$. (See \cite{Zhou21}, \cite{DHST:homcon} for more information on the notion of homology concordance.)
\end{remark}

The following follows from the above proposition and Corollary \ref{cor:connected_sums}:

\begin{corollary}
\label{cor:knot_concordance_local_equivalence_p^{r}}
Let $p^{r}$ be an odd prime power. The correspondence 
\[K\mapsto[\SWF(\Sigma_{p^{r}}(K),\frak{s}_{0},\wh{\sigma})]_{\loc}\]
where $\wh{\sigma}$ is the unique even spin lift of $\sigma$ induces a group homomorphism 
\[\LLL:\C\to\LLL\EEE_{G^{\ev}_{p^{r}},\CC}\]
where $\C$ denotes the smooth concordance group of knots.	
\end{corollary}

\begin{remark}
We would like to define such a homomorphism in the case where $p^{r}=2^{r}$ is a power of two. However, both spin lifts of the covering transformation are of odd type, and it is impossible to pick a coherent choice of distinguished spin lifts for all knots in $S^{3}$. One could consider the unordered pair of local equivalence classes of Seiberg--Witten Floer spectrum classes corresponding to the two spin lifts, but it seems difficult to construct a well-defined group using this framework.
\end{remark}

We next define certain classes of knots in $S^{3}$, which will be helpful for calculations of our equivariant $\kappa$-invariants:

\begin{definition}
\label{def:classes_knots}
Let $\K$ be the set of all oriented knots $K\subset S^{3}$ up to isotopy. We define the following subsets of $\K$:
\begin{enumerate}
    \item Let $\PS_{p^{r}}$ be the set of knots $K$ such that $(\Sigma_{p^{r}}(K),\frak{s}_{0},\pm\wh{\sigma})\in\PS_{p^{r},*}$. We call such knots \emph{$\SWF$-$\Pin(2)$-surjective}.
    \item Let $\SWFM_{p^{r}}$ be the set of knots $K$ such that $(\Sigma_{p^{r}}(K),\frak{s}_{0},\pm\wh{\sigma})\in\SWFM_{p^{r},*}$. We call such knots \emph{$\SWF$-minimal}.
    \item Let $\SWFM_{p^{r}}^{\#}$ be the set of knots $K$ such that $(\Sigma_{p^{r}}(K),\frak{s}_{0},\pm\wh{\sigma})\in\SWFM_{p^{r},*}^{\#}$.
    \item Let $\SWFS_{p^{r}}$ be the set of knots $K$ such that $(\Sigma_{p^{r}}(K),\frak{s}_{0},\pm\wh{\sigma})\in\SWFS_{p^{r},*}$. We call such knots \emph{$\SWF$-spherical}.
    \item Let $\SWFS_{p^{r}}^{\#}$ be the set of knots $K$ such that $(\Sigma_{p^{r}}(K),\frak{s}_{0},\pm\wh{\sigma})\in\SWFS_{p^{r},*}^{\#}$.
    \item Let $\LSWFS_{p^{r}}$ be the set of knots $K$ such that $(\Sigma_{p^{r}}(K),\frak{s}_{0},\pm\wh{\sigma})\in\LSWFS_{p^{r},*}$. We call such knots \emph{locally $\SWF$-spherical}.
\end{enumerate}
For $H\subset G^{*}_{p^{r}}$ a closed subgroup:
\begin{enumerate}
    \item[(7)] Let $\SWFS_{p^{r}}^{H}$ be the set of knots $K$ such that $(\Sigma_{p^{r}}(K),\frak{s}_{0},\pm\wh{\sigma})\in\SWFS_{p^{r},*}^{H}$. We call such knots \emph{$\SWF$-$H$-spherical}.
    \item[(8)] Let $\SWFS_{p^{r}}^{\#,H}$ be the set of knots $K$ such that $(\Sigma_{p^{r}}(K),\frak{s}_{0},\pm\wh{\sigma})\in\SWFS_{p^{r},*}^{\#,H}$.
    \item[(9)] Let $\LSWFS_{p^{r}}^{H}$ be the set of knots $K$ such that $(\Sigma_{p^{r}}(K),\frak{s}_{0},\pm\wh{\sigma})\in\LSWFS_{p^{r},*}^{H}$. We call such knots \emph{locally $\SWF$-$H$-spherical}.
\end{enumerate}
\end{definition}

The following proposition follows immediately from Proposition \ref{prop:inclusions_classes_equivariant_manifolds}:

\begin{proposition}
\label{prop:inclusions_classes_knots}
For each prime power $p^{r}$ and each closed subgroup $H\subset G^{*}_{p^{r}}$, we have inclusions
\begin{center}
\begin{tikzcd}[remember picture]
    \SWFS_{p^{r}}^{H} &[-15pt] \SWFS_{p^{r}}^{\#,H} &[-15pt]  \LSWFS_{p^{r}}^{H} &[-15pt] \\[-10pt]
    \SWFS_{p^{r}}&[-15pt] \SWFS_{p^{r}}^{\#} &[-15pt] \LSWFS_{p^{r}} &[-15pt] \PS_{p^{r}}. \\[-10pt]
    \SWFM_{p^{r}} &[-15pt] \SWFM_{p^{r}}^{\#} &[-15pt] &[-15pt]
\end{tikzcd}
\begin{tikzpicture}[overlay,remember picture]
\path (\tikzcdmatrixname-1-1) to node[midway,sloped]{$\subset_{1}$}
(\tikzcdmatrixname-1-2);
\path (\tikzcdmatrixname-1-2) to node[midway,sloped]{$\subset_{2}$}
(\tikzcdmatrixname-1-3);
\path (\tikzcdmatrixname-2-1) to node[midway,sloped]{$\subset_{3}$}
(\tikzcdmatrixname-1-1);
\path (\tikzcdmatrixname-2-2) to node[midway,sloped]{$\subset_{4}$}
(\tikzcdmatrixname-1-2);
\path (\tikzcdmatrixname-2-3) to node[midway,sloped]{$\subset_{5}$}
(\tikzcdmatrixname-1-3);
\path (\tikzcdmatrixname-2-1) to node[midway,sloped]{$\subset_{6}$}
(\tikzcdmatrixname-2-2);
\path (\tikzcdmatrixname-2-2) to node[midway,sloped]{$\subset_{7}$}
(\tikzcdmatrixname-2-3);
\path (\tikzcdmatrixname-2-3) to node[midway,sloped]{$\subset_{8}$}
(\tikzcdmatrixname-2-4);
\path (\tikzcdmatrixname-3-1) to node[midway,sloped]{$\subset_{9}$}
(\tikzcdmatrixname-2-1);
\path (\tikzcdmatrixname-3-2) to node[midway,sloped]{$\subset_{10}$}
(\tikzcdmatrixname-2-2);
\path (\tikzcdmatrixname-3-1) to node[midway,sloped]{$\subset_{11}$}
(\tikzcdmatrixname-3-2);
\end{tikzpicture}
\end{center}
Furthermore in the case where $p^{r}=2$ and $H=\<j\mu\>$, the inclusions $\subset_{2}$, $\subset_{3}$, $\subset_{4}$, $\subset_{5}$, $\subset_{7}$, and $\subset_{8}$ are strict.
\end{proposition}

Note that if $K\in\LSWFS_{p^{r}}$, then so is any knot concordant to $K$. Moreover, if $K,K'\in\LSWFS_{p^{r}}$, then $K\# K'\in\LSWFS_{p^{r}}$ as well. It follows that for each prime power $p^{r}$, $\LSWFS_{p^{r}}$ descends to a well-defined subgroup of the smooth concordance group $\C$, which we shall also denote by $\LSWFS_{p^{r}}\subset\C$. A similar observation applies to $\LSWFS_{p^{r}}^{H}$ for any closed subgroup $H\subset G^{*}_{p^{r}}$.

In the case where $p^{r}=2$, it will also be helpful to compare $\wt{\kappa}(K)$ with the corresponding invariants from \cite{Man14} and \cite{KMT}. Define
\begin{align*}
    &\kappa(K):=\kappa(\Sigma_{2}(K)), & &\kappa_{\KMT}(K):=\kappa_{\KMT}(\Sigma_{2}(K),\iota),
\end{align*}
where each of these invariants are calculated with respect to the unique spin structure on the double branched cover $\Sigma_{2}(K)$, and $\iota:\Sigma_{2}(K)\to\Sigma_{2}(K)$ denotes the covering involution. Note that by Proposition \ref{prop:kappa_kappa_tilde}, we have that
\[\wt{\kappa}(K)=\kappa(K)\text{ or }\kappa(K)+2\]
for any $K\subset S^{3}$, and that $\wt{\kappa}(K)=\kappa(K)$ if $K\in\P\S_{2}$.

It will be helpful for us to define the following additional class of knots:

\begin{definition}
\label{def:concordant_to_SWFM_knots}
Let 
\[\SWFM_{2}^{\#,\C}:=\{K\subset S^{3}\;|\;K\text{ is smoothly concordant to a knot in }\SWFM_{2}^{\#}\}\subset\K,\]
i.e. $K\in\SWFM_{2}^{\#,\C}$ if and only if $K$ is concordant to a connected sum of knots $K_{1}\#\cdots\# K_{m}$ such that the double branched cover of each $K_{i}$ admits a $\ZZ_{2}$-equivariant metric $g_{i}$ such that $(\Sigma_{2}(K_{i}),\iota_{i},g_{i})$ has no irreducible Seiberg-Witten solutions.
\end{definition}

\begin{example}
\label{ex:SWFM_knots}
Let $K\in\SWFM_{2}^{\#,\C}$. Then $K\in\PS_{2}$, and furthermore by Example \ref{ex:SWF_minimal_double_branched_cover} (and a similar argument for the invariant $\kappa_{\KMT}$) we have that:
\begin{align*}
    &\wt{\kappa}(K)=\kappa(K)=-\tfrac{1}{8}\sigma(K), & &\kappa_{\KMT}(K)=-\tfrac{1}{16}\sigma(K).
\end{align*}
Hence in particular the above invariants are additive on the subclass of knots $\SWFM_{2}^{\#,\C}\subset\K$. As noted implicitly in \cite{KMT}, $\SWFM_{2}^{\#,\C}$ contains connected sums of all of the following knots: two-bridge knots, $T(3,5)$, $9_{47}$, $9_{49}$, $10_{155}$, $10_{156}$, $10_{160}$, $10_{163}$, $K11n92$, $K11n118$. Indeed, the two-bridge knot $K(p,q)$ has double branched cover $L(p,q)$, which admits a metric of positive scalar curvature, and similarly the double branched cover of $T(3,5)$ is the Poincar\'e homology sphere $\Sigma(2,3,5)$. The remaining knots listed above have double branched covers homeomorphic to the hyperbolic 3-manifolds $m007(3,2)$, $m003(-3,1)$, $m003(-4,3)$, $m003(-5,3)$, $m007(1,2)$, $m003(-4,1)$, $m006(-3,2)$, and $m007(4,1)$ in the Hodgson-Weeks census, as identified in (\cite{BalSiv21}, Table 7), and these manifolds were shown to admit no irreducible Seiberg-Witten solutions in \cite{LinLipLS}.
\end{example}

\begin{example}
The torus knot $T(2,2k+1)$ for $k\geq 1$ can be identified with the two-bridge knot $K(2k+1,1)$, with double branched cover $L(2k+1,1)$. Hence $T(2,2k+1)\in\SWFM_{2}$, and so
\begin{align*}
    &\wt{\kappa}(T(2,2k+1))=\kappa(T(2,2k+1))=-\tfrac{1}{8}\sigma(T(2,2k+1))=\tfrac{k}{4}, \\
    &\kappa_{\KMT}(T(2,2k+1))=-\tfrac{1}{16}\sigma(T(2,2k+1))=\tfrac{k}{8}.
\end{align*}
\end{example}

\begin{example}
\label{ex:LSWFS_knots}
Let $K\in\LSWFS_{2}$ be locally $\SWF$-spherical. This includes all knots in $\SWFM^{\#,\C}_{2}$ as in Example \ref{ex:SWFM_knots}, as well as the families of knots $T(3,12n+1)$, $P(-2,3,12n+1)$, $T(3,12n+5)$, $P(-2,3,12n+5)$, and their respective mirrors (see Table \ref{table:knots}). One can show that
\[\wt{\kappa}(K)=\kappa(K)=2\kappa_{\KMT}(K)\]
for any $K\in\LSWFS_{2}$. Indeed, the first equality follows from the fact that $\LSWFS_{2}\subset\P\S_{2}$, and the second equality follows from the fact that if $K\in\LSWFS_{2}$, then there exist some $b_{1/2},b_{3/2}\in\QQ$ such that
\[\big[\SWF(\Sigma_{2}(K),\wh{\iota})\big]_{\loc}=\big[(S^{0},0,b_{1/2}\xi+b_{3/2}\xi^{3})\big]_{\loc}\in\LLL\EEE_{G^{\odd}_{2},\CC}\]
for some spin lift $\wh{\iota}$ of the covering involution $\iota$, and hence
\[\big[\DSWF(\Sigma_{2}(K),\wh{\iota})^{\<j\mu\>}\big]_{\loc}=\big[(S^{0},0,b_{1/2}+b_{3/2})\big]_{\loc}\in\LLL\EEE_{\ZZ_{4},\CC,\sym}.\]
From these presentations we see that
\begin{align*}
	&\wt{\kappa}(K)=\kappa(K)=2(b_{1/2}+b_{3/2}), & &\kappa_{\KMT}(K)=b_{1/2}+b_{3/2}.
\end{align*}
Note that the formulas from Example \ref{ex:SWFM_knots} do not necessarily hold for general knots in $\LSWFS_{2}$. For example,
\begin{align*}
    &\wt{\kappa}(T(3,13))=\kappa(T(3,13))=0\neq 2=-\tfrac{1}{8}\sigma(T(3,13)), \\
    &\kappa_{\KMT}(T(3,13))=0\neq 1=-\tfrac{1}{16}\sigma(T(3,13)).
\end{align*}

Furthermore for any two knots $K,K'\in\LSWFS_{2}$ the following formula holds:
\[\wt{\kappa}(K\# K')=\wt{\kappa}(K)+\wt{\kappa}(K').\]
In other words, the invariants $\wt{\kappa}$, $\kappa$, $\kappa_{\KMT}$ each descend to group homomorphisms
\[[\LSWFS_{2}]\;\to\QQ,\]
where $[\LSWFS_{2}]$ denotes the group of knots in $\LSWFS_{2}$ under connected sum, modulo the relation given by knot concordance. 
\end{example}

\begin{example}
In general, the invariants $\wt{\kappa}$, $\kappa$, $\kappa_{\KMT}$ are not additive concordance invariants. For example, 
\[a(k(2,3,11)\#\ol{k(2,3,11)})\neq a(k(2,3,11))+a(\ol{k(2,3,11)})\]
for $a=\wt{\kappa}$, $\kappa$, or $\kappa_{\KMT}$. Furthermore we do not necesssarily have $\wt{\kappa}(K)=2\kappa_{\KMT}(K)$, since for example
\begin{align*}
	&\wt{\kappa}(T(3,11))=\kappa(T(3,11))=2, & &\kappa_{\KMT}(T(3,11))=0.
\end{align*}
\end{example}

%% file: applications.tex
\section{Applications}
\label{sec:applications}

In this section we discuss various applications of our equivariant relative 10/8-ths inequalities from Section \ref{subsec:main_theorems}.

\bigskip
\subsection{Constraints on Spin Cyclic Group Actions}
\label{subsec:constraints_group_actions}

In this section we prove Theorem \ref{theorem:intro_extending_involution_constraints} from the introduction. We will need the following two results:

\begin{proposition}[\cite{KMT}, Theorem 1.13]
\label{prop:KMT_inequality}
Let $a_{1},\dots,a_{n}$ be pairwise coprime natural numbers with $a_{1}$ an even number. Set $Y=\Sigma(a_{1},\dots,a_{n})$, and let $\iota:Y\to Y$ be the odd-type involution given by a rotation of $\pi$ in the fibers. Let $W$ be a compact connected smooth oriented spin 4-manifold bounded by $Y$ with
$b_{1}(W)=0$ and intersection form given by $p(-E_{8})\oplus qH$. Then:
\begin{enumerate}
    \item The involution $\iota$ cannot extend to $W$ as a smooth involution $\tau$ so that
    \begin{equation*}
        \tfrac{1}{2}p>b_{2}^{+}(W,\tau)_{1}+\kappa_{\KMT}(Y).
    \end{equation*}
    \item Suppose that $p\neq-\ol{\mu}(Y)$. Then $\iota$ cannot extend to $W$ as a homologically trivial smooth involution, while $\iota$ can extend to $W$ as a homologically trivial diffeomorphism.
\end{enumerate}
\end{proposition}

\begin{proposition}[\cite{KonnoTaniguchi}, Theorem 1.2]
Let $(Y,\frak{s})$ be a spin rational homology 3-sphere and let $(W,\frak{t})$ be a compact spin filling of $(Y,\frak{s})$ with $b_{1}(W)=0$. Let $B$ be a compact topological space and
\[(W,\frak{t})\to E \to B\]
a smooth $\Aut((W,\frak{t}),\del)$-bundle. Then:
\begin{enumerate}
    \item If $w_{b_{2}^{+}(W)}(H_{2}^{+}(E))\neq 0$, then $\gamma(Y,\frak{s})\geq p$.
    \item If $w_{b_{2}^{+}(W)-1}(H_{2}^{+}(E))\neq 0$ and $b_{2}^{+}(W)\geq 1$, then $\beta(Y,\frak{s})\geq p$.
    \item If $w_{b_{2}^{+}(W)-2}(H_{2}^{+}(E))\neq 0$ and $b_{2}^{+}(W)\geq 2$, then $\alpha(Y,\frak{s})\geq p$.
\end{enumerate}
\end{proposition}

\begin{corollary}
\label{cor:q1_even}
Let $(Y,\frak{s},\wh{\tau})$ be a $\ZZ_{2}$-equivariant spin rational homology 3-sphere such that $\wh{\iota}$ is of odd type, and the underlying involution $\iota:Y\to Y$ is isotopic to the identity. Let $(W,\frak{t},\wh{\tau})$ be a compact $\ZZ_{2}$-equivariant spin filling of $(Y,\frak{s},\wh{\iota})$. Then:
\begin{enumerate}
    \item if $q=1$ and $\gamma(Y,\frak{s})<p$, or
    \item if $q=2$ and $\beta(Y,\frak{s})< p$, or
    \item if $q=3$ and $\alpha(Y,\frak{s})< p$,
\end{enumerate}
then $q_{1}$ must be even.
\end{corollary}

We are now ready to prove Theorem \ref{theorem:intro_extending_involution_constraints}:

\begin{proof}[Proof of Theorem \ref{theorem:intro_extending_involution_constraints}]
By a result of McCullough--Soma (\cite{MS13}) any odd-type involution $\iota$ on a hyperbolic Brieskorn sphere must be conjugate to $\rho_{2}$ if $\iota$ is isotopic to the identity, or to the complex conjugation involution $\iota_{c}$ otherwise.

We will essentially proceed on a case-by-case basis, using Theorem \ref{theorem:odd_2_filling} and our calculations from Section \ref{subsec:brieskorn_spheres}, along with the quoted propositions above.

For (1a), let $W$ be a spin filling of $\Sigma(2,3,12n-1)$ with intersection form $-2E_{8}\oplus 2H$. Note that $q_{0}=0,1,$ or $2$. If $\rho_{2}$ extends to an involution $\tau$ on $W$ then by Proposition \ref{prop:KMT_inequality} we must have that
\[q_{1}\geq \frac{p}{2}-\kappa_{\KMT}(\Sigma(2,3,12n-1),\rho_{2})=1-0=1,\]
hence $q_{0}\neq 2$. Note that since $q=2$ and
\[\beta(\Sigma(2,3,12n-1))=0<2=p,\]
by Corollary \ref{cor:q1_even}, we must have that $q_{1}$ is even. Hence $q_{0}=0$. The proof of (1b) is similar.

For (2), let $W$ be a spin filling of $-\Sigma(2,3,12n+5)$ with intersection form $-E_{8}\oplus 3H$, let $\iota:Y\to Y$ be an odd-type involution conjugate to either $\rho_{2}$ or $\iota_{c}$, and suppose $\iota$ extends to an involution $\tau$ on $W$. Then by Theorem \ref{theorem:odd_2_filling} we have that
\[3\geq 1-\wt{\kappa}(Y,\frak{s},\iota)+C=2+C.\]
Note that the possibility that $q_{0}=3$ is excluded by Proposition \ref{prop:KMT_inequality}. Furthrmore note that if $q_{0},q_{1}\neq 0$, then $C\geq 2$ unless
\[q_{1}=p-2\kappa_{\KMT}(Y,\frak{s},\iota)=1-2(-\tfrac{1}{2})=2,\]
and so we must have $q_{0}\neq 2$. Finally note that since $q=3$ and
\[\alpha(-\Sigma(2,3,12n+5))=-1<1=p,\]
by Corollary \ref{cor:q1_even} we must have that $q_{1}$ is even, and therefore $q_{0}$ must be odd. Thus we can rule out $q_{0}=0$, and so we must have $q_{0}=1$.

For (3a), let $Y=\pm\Sigma(2,3,12n+1)$, let $W$ be a spin filliing of $Y$ with intersection form $p(-E_{8})\oplus(p+1)H$, $p\geq 4$ even, let $\iota:Y\to Y$ be an odd-type involution conjugate to either $\rho_{2}$ or $\iota_{c}$, and suppose $\iota$ extends to an involution $\tau$ on $W$. Then by Theorem \ref{theorem:odd_2_filling}, we have that
\[q=p+1\geq p-\wt{\kappa}(Y,\frak{s},\iota)+C=p+C,\]
and so $C\le 1$. Using the fact that $\kappa_{\KMT}(Y,\frak{s},\iota)=0$, we must have that either
\begin{enumerate}
	\item $(q_{0},q_{1})=(0,p+1)$,
	\item $(q_{0},q_{1})=(p+1,0)$, or
	\item $(q_{0},q_{1})=(1,p)$.
\end{enumerate}
Case (2) can be ruled out by Proposition \ref{prop:KMT_inequality}, and so we must have either $q_{0}=0$ or $1$.

For (3b), let $Y=-\Sigma(2,3,12n-5)$ or $\Sigma(2,3,12n+5)$ and let $W$ be a spin filling of $Y$ with intersection form $p(-E_{8})\oplus pH$ with $p\geq 3$, odd. By Theorem \ref{theorem:odd_2_filling} if $\rho_{2}$ extends to an involution $\tau$ on $W$ then
\[p=q\geq p-\wt{\kappa}(Y)+C,\]
where $C$ is as in the statement of the theorem. As $\kappa(Y)=1$, it follows that we must have that $C\le 1$. Using the fact that $q=p$ is odd, it follows that either:
\begin{enumerate}
	\item $q_{0}=0$,
	\item $q_{1}=0$, or
	\item $q_{1}=p-2\kappa_{\KMT}(Y,\rho_{2})$.
\end{enumerate}
Using the fact that $\kappa_{\KMT}(Y,\rho_{2})=\frac{1}{2}$, we see that Case (3) is equivalent to $q_{0}=1$. We can rule out Case (2), since by Proposition \ref{prop:KMT_inequality} we must have that
\[q_{1}\geq\frac{p}{2}-\kappa_{KMT}(Y,\rho_{2})=\frac{p-1}{2}>0\]
by our assumption that $p\geq 3$. Hence $q_{0}=0$ or $1$.

Finally for (3c), let $W$ be a spin filling of $-\Sigma(2,3,12n+5)$ with intersection form $p(-E_{8})\oplus(p+2)H$ with $p\geq 3$ odd, let $\iota:Y\to Y$ be an odd-type involution conjugate to either $\rho_{2}$ or $\iota_{c}$, and suppose $\iota$ extends to an involution $\tau$ on $W$. Then by Theorem \ref{theorem:odd_2_filling}, we have that
\[q=p+2\geq p-\wt{\kappa}(Y,\frak{s},\iota)+C=p+1+C,\]
and so $C\le 1$. Using the fact that $\kappa_{\KMT}(Y,\frak{s},\iota)=-\frac{1}{2}$, we must have that either:
\begin{enumerate}
	\item $(q_{0},q_{1})=(0,p+2)$,
	\item $(q_{0},q_{1})=(p+2,0)$, or
	\item $(q_{0},q_{1})=(1,p+1)$.
\end{enumerate}
Again Case (2) can be ruled out by Proposition \ref{prop:KMT_inequality}, and so we must have that either $q_{0}=0$ or $1$.
\end{proof}

\subsection{Constraints on Equivariant Cobordisms}
\label{subsec:constraints_equivariant_cobordisms}

In this section, we give a proof of Theorem \ref{theorem:intro_montesinos_knots_not_concordant_to_LSWFS_2} from the introduction.

\begin{proof}[Proof of Theorem \ref{theorem:intro_montesinos_knots_not_concordant_to_LSWFS_2}]
Let $K$ be as in the statement of the theorem. It suffices to show that the spectrum class $\SWF(\Sigma_{2}(K),\frak{s}_{0},\wh{\iota})$ is not locally $\<j\mu\>$-spherical for any spin lift $\wh{\iota}$ of the covering involution $\iota:\Sigma_{2}(K)\to\Sigma_{2}(K)$. But by Corollary \ref{cor:connected_sums} and the calculations from Section \ref{subsubsec:jmu_fixed_points}, it follows that the $\<j\mu\>$-fixed point spectrum class $\SWF(\Sigma_{2}(K),\frak{s}_{0},\wh{\iota})^{\<j\mu\>}\in\CCC_{\ZZ_{4},\CC}$ is locally equivalent to a $\CC$-$\ZZ_{4}$-spectrum class of the form
\[\big[\Sigma^{e\VV_{1/2}}\wt{\Sigma}\ZZ_{4},0,b_{1/2}\xi+b_{3/2}\xi^{3}\big]\in\CCC_{\ZZ_{4},\CC},\]
for some $e\in\{0,1\}$, $b_{1/2},b_{3/2}\in\QQ$. But as in the proof of Proposition \ref{prop:not_jmu_spherical}, $\wt{\Sigma}\ZZ_{4}$ cannot be $\ZZ_{4}$-locally equivalent to a sphere, from which the result follows.
\end{proof}

\subsection{Genus Bounds}
\label{subsec:genus_bounds}

\begin{definition}
Let $X$ be a closed oriented 4-manifold, let $K\subset S^{3}$ be an oriented knot, and let $A\in H_{2}(X;\ZZ)$ be a fixed 2-dimensional homology class. We define the \emph{$(X,A)$-genus of $K$}, denoted $g_{X,A}(K)$, to be the minimal genus over all properly embedded oriented surfaces $F\subset X\setminus B^{4}$ such that $\del F=K\subset S^{3}$ and $[F]=A$.
\end{definition}

Consider the following Lemma from \cite{KMT}:

\begin{lemma}[\cite{KMT}, Lemma 4.2]
\label{lemma:double_branched_cover_spin}
Let $K\subset S^{3}$ be an oriented knot, let $X$ be a closed oriented smooth 4-manifold with $b_{1}(X)=0$, $b_{2}^{+}(X)\neq 0$, and let $\mathring{X}=W\setminus B^{4}$. Suppose $F$ is a compact smooth properly embedded surface in $\mathring{X}$ such that:
\begin{itemize}
    \item $\del F=K\subset S^{3}=\del\mathring{X}$.
    \item $[F]\in H_{2}(\mathring{X},\del\mathring{X};\ZZ)\cong H_{2}(X;\ZZ)$ is divisible by 2.
    \item $\frac{1}{2}[F]\equiv\text{PD}(w_{2}(X))\pmod{2}$,
\end{itemize}
where $\text{PD}(w_{2}(X))$ denotes the Poincar\'e dual of $w_{2}(X)\in H^{2}(X,\ZZ_{2})$. Then the double branched cover $W:=\Sigma(\mathring{X},F)$ of $\mathring{X}$ over $F$ is spin, with a distinguished spin structure $\frak{t}$ which is invariant under the covering involution $\tau$ on $W$. Furthermore, if
\begin{align*}
    &p=-\sigma(W)/8, & &q=b_{2}^{+}(W), & &q_{0}=b_{2}^{+}(W,\tau)_{0}, & &q_{1}=b_{2}^{+}(W,\tau)_{1},
\end{align*}
Then $q_{0},q_{1}\neq 0$, and:
\begin{align*}
    &p=-\tfrac{1}{4}\sigma(X)+\tfrac{1}{16}[F]^{2}-\tfrac{1}{8}\sigma(K), \\
    &q=2b_{2}^{+}(X)+g(F)-\tfrac{1}{4}[F]^{2}+\tfrac{1}{2}\sigma(K), \\
    &q_{0}=b_{2}^{+}(X), \\
    &q_{1}=b_{2}^{+}(X)+g(F)-\tfrac{1}{4}[F]^{2}+\tfrac{1}{2}\sigma(K).
\end{align*}
\end{lemma}

With this in mind, we have the following theorem:

\begin{theorem}
\label{theorem:genus_bounds}
Let $X$ be a closed oriented 4-manifold with $b_{1}(X)=0$, $b_{2}^{+}(X)\neq 0$, let $A\in H_{2}(X)$ be such that $2|A$ and $A/2\equiv w_{2}(X)\pmod{2}$, let $K\in\LSWFS^{\<j\mu\>}$, and let
\[c(K,X):=b_{2}^{+}(X)+\wt{\kappa}(K)-2\kappa_{\KMT}(K).\]
Then:
\begin{equation}
\label{eq:genus_bounds}
    g_{X,A}(K)\geq -2b_{2}^{+}(X)-\tfrac{1}{4}\sigma(X)+\tfrac{5}{16}A^{2}-\tfrac{5}{8}\sigma(K)-\wt{\kappa}(K)+C,
\end{equation}
where:
\begin{equation}
    C=\left\{
		\begin{array}{ll}
			3 & \mbox{if }b_{2}^{+}(X)\text{ is even and }c(K,X)\geq 4, \\
			2 & \mbox{if }c(K,X)\geq 2, \\
			1 & \mbox{otherwise.}
		\end{array}
	\right.
\end{equation}
\end{theorem}

\begin{proof}
Note that by Lemma \ref{lemma:double_branched_cover_spin}, inequality (\ref{eq:genus_bounds}) is equivalent to the inequality
\begin{equation}
    q\geq p-\wt{\kappa}(K)+C.
\end{equation}
To see how this follows from Theorem \ref{theorem:odd_2_filling}, we split into two cases depending on the parity of $q_{0}$:

\bigskip

\begin{enumerate}
    \item[\emph{Case 1}:] $q_{0}=b_{2}^{+}(X)$ even.
    
    \bigskip
    
    If $q_{1}$ is even, then $q\geq p-\wt{\kappa}(K)+4$ if $q_{1}\neq p-2\kappa_{\KMT}(K)$, and if $q_{1}$ is odd, then $q\geq p-\wt{\kappa}(K)+3$ if $q_{1}\neq p-2\kappa_{\KMT}(K)-1$. Hence in either case, 
    \begin{equation}
    \label{eq:proof_genus_bounds_b2+even}
        q\geq p-\wt{\kappa}(K)+3,
    \end{equation}
    unless
    \begin{align}
        &q_{1}=p-2\kappa_{\KMT}(K),\text{ or} \label{eq:proof_genus_bounds_kappa}\\
        &q_{1}=p-2\kappa_{\KMT}(K)-1. \label{eq:proof_genus_bounds_kappa-1}
    \end{align}
    Note that Equations \ref{eq:proof_genus_bounds_kappa} and \ref{eq:proof_genus_bounds_kappa-1} are equivalent to the following two equations, respectively:
    \begin{align}
        &g_{X,A}(K)= -2b_{2}^{+}(X)-\tfrac{1}{4}\sigma(X)+\tfrac{5}{16}A^{2}-\tfrac{5}{8}\sigma(K)-\wt{\kappa}(K)+c(K,X), \label{eq:proof_genus_bounds_c(K,X)}\\
        &g_{X,A}(K)= -2b_{2}^{+}(X)-\tfrac{1}{4}\sigma(X)+\tfrac{5}{16}A^{2}-\tfrac{5}{8}\sigma(K)-\wt{\kappa}(K)+c(K,X)-1. \label{eq:proof_genus_bounds_c(K,X)-1}
    \end{align}
    So if $c(K,X)\geq 4$, then if one Equations \ref{eq:proof_genus_bounds_kappa} and \ref{eq:proof_genus_bounds_kappa-1} holds, then inequality (\ref{eq:proof_genus_bounds_b2+even}) is still satisfied.
    
    If $c(K,X)\geq 2$, we use the fact that
    \begin{equation}
    \label{eq:proof_genus_bounds_b2+even_2}
        q\geq p-\wt{\kappa}(K)+2
    \end{equation}
    unless $q_{1}=p-2\kappa_{\KMT}(K)$, in which case we still have that
    \begin{align*}
        g_{X,A}(K)&= -2b_{2}^{+}(X)-\tfrac{1}{4}\sigma(X)+\tfrac{5}{16}A^{2}-\tfrac{5}{8}\sigma(K)-\wt{\kappa}(K)+c(K,X) \\
        &\geq -2b_{2}^{+}(X)-\tfrac{1}{4}\sigma(X)+\tfrac{5}{16}A^{2}-\tfrac{5}{8}\sigma(K)-\wt{\kappa}(K)+2,
    \end{align*}
    as desired.
    \item[\emph{Case 2}:] $q_{0}=b_{2}^{+}(X)$ odd. 
    
    \bigskip
    
    If $q_{1}$ is even, then $q\geq p-\wt{\kappa}(K)+3$ if $q_{1}\neq p-2\kappa_{\KMT}(K)$, and if $q_{1}$ is odd, then $q\geq p-\wt{\kappa}(K)+2$. Hence in either case, 
    \begin{equation}
    \label{eq:proof_genus_bounds_b2+odd}
        q\geq p-\wt{\kappa}(K)+2,
    \end{equation}
    unless
    \begin{equation}
    \label{eq:proof_genus_bounds_kappa_b2+odd}
        q_{1}=p-2\kappa_{\KMT}(K).
    \end{equation}
    Note that as in Case 1, Equation \ref{eq:proof_genus_bounds_kappa_b2+odd} is equivalent to
    \begin{equation}
    \label{eq:proof_genus_bounds_c(K,X)_b2+odd}
        g_{X,A}(K)= -2b_{2}^{+}(X)-\tfrac{1}{4}\sigma(X)+\tfrac{5}{16}A^{2}-\tfrac{5}{8}\sigma(K)-\wt{\kappa}(K)+c(K,X).
    \end{equation}
    Hence if $c(K,X)\geq 2$, then inequality (\ref{eq:proof_genus_bounds_b2+odd}) holds, as desired.
\end{enumerate}
\end{proof}

Recall from Example \ref{ex:SWFM_knots} that
\begin{align*}
	&\wt{\kappa}(K)=\kappa(K)=-\tfrac{1}{8}\sigma(K), & &\kappa_{\KMT}=-\tfrac{1}{16}\sigma(K),
\end{align*}
for any knot $K\in\SWFM_{2}^{\#,\C}$. We have the following corollary of Theorem \ref{theorem:genus_bounds}:

\begin{corollary}
\label{cor:genus_bounds_SWFM}
Let $X$ be a closed oriented 4-manifold with $b_{1}(X)=0$, $b_{2}^{+}(X)\neq 0$, let $A\in H_{2}(X)$ be such that $2|A$ and $A/2\equiv w_{2}(X)\pmod{2}$, and suppose $K\in\SWFM_{2}^{\#,\C}$. Then:
\[g_{X,A}(K)\geq -2b_{2}^{+}(X)-\tfrac{1}{4}\sigma(X)+\tfrac{5}{16}A^{2}-\tfrac{1}{2}\sigma(K)+C,\]
where:
\[C=\left\{
		\begin{array}{ll}
			3 & \mbox{if }b_{2}^{+}(X)\text{ is even}, b_{2}^{+}(X)\geq 4, \\
			2 & \mbox{if }b_{2}^{+}(X)\text{ is odd}, b_{2}^{+}(X)\geq 3,\text{ or} \\
		    & \mbox{if }b_{2}^{+}(X)=2, \\
			1 & \mbox{if }b_{2}^{+}(X)=1.
		\end{array}
	\right.\]
\end{corollary}

\begin{example}
Consider the case where $X=\#^{n}S^{2}\times S^{2}$ for some $n\geq 1$, and let $A=((a_{1},b_{1}),\dots,(a_{n},b_{n}))$, where $a_{k}\equiv b_{k}\equiv 0\pmod{4}$ for all $k=1,\dots, n$. Then for any $K\in\SWFM_{2}^{\#,\C}$, we have that
\begin{equation}
\label{eq:genus_bounds_SWF-spherical_S2xS2}
    g_{X,A}(K)\geq -2n+\tfrac{5}{8}\sum_{k=1}^{n}a_{k}b_{k}-\tfrac{1}{2}\sigma(K)+C,
\end{equation}
where
\begin{equation}
    C=\left\{
		\begin{array}{ll}
			3 & \mbox{if }n\text{ is even}, n\geq 4, \\
			2 & \mbox{if }n\text{ is odd}, n\geq 3,\text{ or} \\
		    & \mbox{if }n=2, \\
			1 & \mbox{if }n=1.
		\end{array}
	\right.
\end{equation}
\end{example}

\begin{example}
Suppose $X=(\#^{m}\CC P^{2})\#(\#^{n}\ol{\CC P}^{2})$ for some $m\geq 1$, $n\geq 0$, and let $A=((a_{1},\dots,a_{m}),(b_{1},\dots,b_{n}))$, where $a_{k}\equiv b_{\ell}\equiv 2\pmod{4}$ for all $k,\ell$. Then for any $K\in\SWFM_{2}^{\#,\C}$, we have that
\begin{equation}
\label{eq:genus_bounds_SWF-spherical_CP2_CP2bar}
    g_{X,A}(K)\geq -\tfrac{9}{4}m+\tfrac{1}{4}n+\tfrac{5}{16}\Big(\sum_{k=1}^{m}a_{k}^{2}-\sum_{\ell=1}^{n}b_{\ell}^{2}\Big)-\tfrac{1}{2}\sigma(K)+C,
\end{equation}
where
\begin{equation}
    C=\left\{
		\begin{array}{ll}
			3 & \mbox{if }m\text{ is even}, m\geq 4, \\
			2 & \mbox{if }m\text{ is odd}, m\geq 3,\text{ or} \\
		    & \mbox{if }m=2, \\
			1 & \mbox{if }m=1.
		\end{array}
	\right.
\end{equation}
\end{example}

\begin{example}
Consider the case where $X=\#^{n}K3$ for some $n\geq 1$, and let $A=0$. Then for any $K\in\SWFM_{2}^{\#}$, we have that
\begin{equation}
\label{eq:genus_bounds_SWF-spherical_K3}
    g_{X,0}(K)\geq -2n-\tfrac{1}{2}\sigma(K)+\left\{
		\begin{array}{ll}
			3 & \mbox{if }n\text{ is even},\\
			2 & \mbox{if }n\text{ is odd}.
		\end{array}
	\right.
\end{equation}
For example, in the case where $n=1$ and $K=T(2,2k+1)$ is a two-bridge torus knot, we have that
\[g_{K3,0}(T(2,2k+1))=k=g_{4}(T(2,2k+1)),\]
which agrees with (\cite{BaragliaQP}, Corollary 1.3).
\end{example}

For any closed oriented 4-manifold $X$ and any homology class $A\in H_{2}(X;\ZZ)$, let
\[g(X,A):=\min\{\text{genus}(F)\,|\,F\hookrightarrow X, [F]=A\},\]
and recall the upper bound
\begin{equation}
\label{eq:upper_bound}
	g_{X,A}(K)\le g(X,A)+g_{4}(K)
\end{equation}
for the $(X,A)$-genus of $K$ as in the introduction. We can rephrase Theorem \ref{theorem:intro_X_A_genus_sharp} from the introduction as follows:

\begin{theorem}
\label{theorem:X_A_genus_sharp}
Let $(X,A)$ be one of the following pairs, where $X$ is a closed oriented 4-manifold and $A\in H_{2}(X;\ZZ)$:
\begin{center}
\begin{tabular}{|c|c|}\hline
    $X$ & $A$ \\ \hline
    $S^{2}\times S^{2}\#S^{2}\times S^{2}$ & $((4,4),(4,4))$  \\ \hline
    \multirow{2}{*}{$\CC P^{2}\#\CC P^{2}$} & $(6,2)$\\ \cline{2-2}
    & $(6,6)$  \\ \hline
    \multirow{2}{*}{$S^{2}\times S^{2}\#\CC P^{2}$} & $((4,4),2)$ \\ \cline{2-2}
    & $((4,4),6)$ \\ \hline
    $hK3$ & $0$  \\ \hline
\end{tabular}
\end{center}
Here $X=hK3$ denotes any homotopy $K3$ surface. Furthermore, let $K\in\SWFM^{\#,\C}_{2}$ be such that $g_{4}(K)=-\frac{1}{2}\sigma(K)$. Then:
\begin{equation}
    g_{X,A}(K)=g(X,A)+g_{4}(K).
\end{equation}
\end{theorem}

\begin{proof}
The values of of $g(X,A)$ for the pairs $(X,A)$ listed in the above table are given as follows:
\begin{center}
\begin{tabular}{|c|c|c|}\hline
    $X$ & $A$ & $g(X,A)$  \\ \hline
    $S^{2}\times S^{2}\#S^{2}\times S^{2}$ & $((4,4),(4,4))$ & $18$ \\ \hline
    \multirow{2}{*}{$\CC P^{2}\#\CC P^{2}$} & $(6,2)$ & $10$ \\ \cline{2-3}
    & $(6,6)$ & $20$ \\ \hline
    \multirow{2}{*}{$S^{2}\times S^{2}\#\CC P^{2}$} & $((4,4),2)$ & $9$ \\ \cline{2-3}
    & $((4,4),6)$ & $19$ \\ \hline
    $hK3$ & $0$ & $0$ \\ \hline
\end{tabular}
\end{center}

Indeed the first five of these cases follow from an application of Bryan's inequality (\cite{Bry97}, Theorem 1.6), and the resolution of the minimal genus problem for $S^{2}\times S^{2}$ (\cite{Rub96}) and $\CC P^{2}$ (\cite{KMthomconj}). It suffices to check that in each of these cases, the lower bound obtained from Theorem \ref{theorem:genus_bounds} is equal to $g(X,A)+g_{4}(K)$.

Using Corollary \ref{cor:genus_bounds_SWFM} and the assumption that $g_{4}(K)=-\frac{1}{2}\sigma(K)$, we can check that $g_{X,A}(K)\geq g(X,A)+g_{4}(K)$ for each of the six cases in the theorem:

\begin{enumerate}
    \item[Case 1.] $(X,A)=(S^{2}\times S^{2}\# S^{2}\times S^{2},((4,4),(4,4)))$:
    \begin{align*}
        g_{X,A}(K)&\geq-2b_{2}^{+}(X)-\tfrac{1}{4}\sigma(X)+\tfrac{5}{16}A^{2}-\tfrac{1}{2}\sigma(K)+C \\
        &=-4+\tfrac{5}{16}(64)-\tfrac{1}{2}\sigma(K)+2 \\
        &=18-\tfrac{1}{2}\sigma(K) \\
        &=g(X,A)+g_{4}(K).
    \end{align*}
    
    \item[Case 2.] $(X,A)=(\CC P^{2}\#\CC P^{2},(6,2))$:
    \begin{align*}
        g_{X,A}(K)&\geq-2b_{2}^{+}(X)-\tfrac{1}{4}\sigma(X)+\tfrac{5}{16}A^{2}-\tfrac{1}{2}\sigma(K)+C \\
        &=-\tfrac{9}{2}+\tfrac{5}{16}(40)-\tfrac{1}{2}\sigma(K)+2 \\
        &=10-\tfrac{1}{2}\sigma(K) \\
        &=g(X,A)+g_{4}(K).
    \end{align*}
    \item[Case 3.] $(X,A)=(\CC P^{2}\#\CC P^{2},(6,6))$:
    \begin{align*}
        g_{X,A}(K)&\geq-2b_{2}^{+}(X)-\tfrac{1}{4}\sigma(X)+\tfrac{5}{16}A^{2}-\tfrac{1}{2}\sigma(K)+C \\
        &=-\tfrac{9}{2}+\tfrac{5}{16}(72)-\tfrac{1}{2}\sigma(K)+2 \\
        &=20-\tfrac{1}{2}\sigma(K) \\
        &=g(X,A)+g_{4}(K).
    \end{align*}
    \item[Case 4.] $(X,A)=(S^{2}\times S^{2}\#\CC P^{2},((4,4),2))$:
    \begin{align*}
        g_{X,A}(K)&\geq-2b_{2}^{+}(X)-\tfrac{1}{4}\sigma(X)+\tfrac{5}{16}A^{2}-\tfrac{1}{2}\sigma(K)+C \\
        &=-\tfrac{17}{4}+\tfrac{5}{16}(36)-\tfrac{1}{2}\sigma(K)+2 \\
        &=9-\tfrac{1}{2}\sigma(K) \\
        &=g(X,A)+g_{4}(K).
    \end{align*}
    \item[Case 5.] $(X,A)=(S^{2}\times S^{2}\#\CC P^{2},((4,4),6))$:
    \begin{align*}
        g_{X,A}(K)&\geq-2b_{2}^{+}(X)-\tfrac{1}{4}\sigma(X)+\tfrac{5}{16}A^{2}-\tfrac{1}{2}\sigma(K)+C \\
        &=-\tfrac{17}{4}+\tfrac{5}{16}(68)-\tfrac{1}{2}\sigma(K)+2 \\
        &=19-\tfrac{1}{2}\sigma(K) \\
        &=g(X,A)+g_{4}(K).
    \end{align*}
    \item[Case 6.] $(X,A)=(hK3,0)$:
    \begin{align*}
        g_{X,A}(K)&\geq-2b_{2}^{+}(X)-\tfrac{1}{4}\sigma(X)+\tfrac{5}{16}A^{2}-\tfrac{1}{2}\sigma(K)+C \\
        &=-6+4+0-\tfrac{1}{2}\sigma(K)+2 \\
        &=0-\tfrac{1}{2}\sigma(K) \\
        &=g(X,A)+g_{4}(K).
    \end{align*}
\end{enumerate}
\end{proof}

We conclude this section by comparing the relative genus bound from Theorem \ref{theorem:genus_bounds} with other bounds from the literature. The first one gives a lower bound for the \emph{topological} $(X,A)$-genus $g_{X,A}^{\topp}(K)$ defined as follows:

\begin{definition}
Let $X$ be a closed oriented topological 4-manifold, let $K\subset S^{3}$ be an oriented knot, and let $A\in H_{2}(X;\ZZ)$ be a fixed 2-dimensional homology class. We define the \emph{topological $(X,A)$-genus of $K$}, denoted $g_{X,A}^{\topp}(K)$, to be the minimal genus over all properly embedded oriented \emph{locally flat} surfaces $F\subset X\setminus B^{4}$ such that $\del F=K\subset S^{3}$ and $[F]=A$.
\end{definition}

We then have the following lower bound for $g_{X,A}^{\topp}(K)$ coming from the $G$-signature theorem, which in turn gives a lower bound for the smooth $(X,A)$-genus via the inequality $g_{X,A}^{\topp}(K)\le g_{X,A}(K)$:

\begin{theorem}[\cite{CN:slice}, \cite{Gilmer81}, \cite{Viro75}]
Let $X$ be a closed topological 4-manifold with $H_{1}(X;\ZZ)=0$, let $A\in H_{2}(X;\ZZ)$ be such that $2|A$, and let $K\subset S^{3}$ be a knot. Then
\begin{equation}
\label{eq:G_signature}
	g_{X,A}^{\topp}(K)\geq \tfrac{1}{2}\big|\sigma(K)+\sigma(X)-\tfrac{1}{2}A^{2}\big|-\tfrac{1}{2}b_{2}(X).
\end{equation}
\end{theorem}

We also have the following two inequalities, coming from Manolescu's and Konno-Miyazawa-Taniguchi's relative 10/8ths inequalities applied to double-branched covers:

\begin{theorem}[\cite{Man14},\cite{KMT}]
Let $X$ be a closed oriented 4-manifold with $b_{1}(X)=0$, $b_{2}^{+}(X)\neq 0$, and let $A\in H_{2}(X;\ZZ)$ be such that $2|A$ and $A/2\equiv w_{2}(X)\pmod{2}$. Then for any knot $K\subset S^{3}$ we have the following two inequalities:
\begin{align}
	&g_{X,A}(K)\geq -2b_{2}^{+}(X)-\tfrac{1}{4}\sigma(X)+\tfrac{5}{16}A^{2}-\tfrac{5}{8}\sigma(K)-\kappa(K)+1, \label{eq:Man_knot}\\
	&g_{X,A}(K)\geq -b_{2}^{+}(X)-\tfrac{1}{8}\sigma(X)+\tfrac{9}{32}A^{2}-\tfrac{9}{16}\sigma(K)-\kappa_{\KMT}(K). \label{eq:KMT_knot}
\end{align}
\end{theorem}

In Table \ref{table:genus_bounds} located in Appendix \ref{sec:tables_appendix}, we compare the lower bounds given by the above three inequalities and the lower bound coming from Theorem \ref{theorem:genus_bounds} with the upper bound (\ref{eq:upper_bound}), for the torus knots $K=T(3,5)$, $T(3,7)$, $T(3,11)$, and $T(3,13)$, and for the pairs $(X,A)$ featured in Theorem \ref{theorem:X_A_genus_sharp}. Note that $T(3,5)\in\SWFM_{2}^{\#,\C}$, whereas $T(3,7)$, $T(3,11)$, $T(3,13)\in\LSWFS_{2}^{\<j\mu\>}\setminus\SWFM_{2}^{\#,\C}$.

\bigskip

%% file: number_theory.tex
\section{A Bit of Number Theory}
\label{sec:number_theory}

In this Appendix, we prove Proposition \ref{prop:monomials}, which we restate here as Proposition \ref{prop:monomials_appendix} for the convenience of the reader. It turns out that main ingredient in the proof of Proposition \ref{prop:monomials_appendix} is the following lemma concerning cyclotomic units, whose proof takes up the majority of this section:

\begin{lemma}
\label{lemma:appendix_equals_one}
Let $p^{r}$, $r\geq 1$ be a prime power, let $\omega_{p^{r}}=e^{2\pi i/p^{r}}\in\CC$, and let $c_{0},\dots,c_{p^{r}-1}\in\ZZ$ be such that $\sum_{k=0}^{p^{r}-1}c_{k}=0$. Then
\begin{equation}
    \prod_{k=0}^{p^{r}-1}(1+\omega_{p^{r}}^{k})^{c_{k}}=1
\end{equation}
if and only if:
\begin{enumerate}
    \item \emph{ If $p$ odd}:
    \begin{align}
        &c_{0}=0, \label{eq:appendix_lemma_p_r_1}\\
        &\sum_{k=1}^{\frac{p^{r}-1}{2}}k(c_{k}-c_{-k})\equiv 0\pmod{2p^{r}}, \label{eq:appendix_lemma_p_r_2}\\
        &\sum_{s=0}^{r-1}c_{kp^{s}}+c_{-kp^{s}}=\sum_{s=0}^{r-1}c_{kp^{s}/2}+c_{-kp^{s}/2}\text{ for all }k=2,\dots,\frac{p^{r}-1}{2}\text{ with }(k,p)=1, \label{eq:appendix_lemma_p_r_3}
    \end{align}
    where we use the cyclic indexing notation $c_{k+ap^{r}}:=c_{k}$ for all $a\in\ZZ$, and $\frac{c}{2}$ denotes the unique element of $\ZZ/{p^{r}}$ such that $2\cdot\frac{c}{2}\equiv c\pmod{p^{r}}$.
    
    \bigskip
    
    \item \emph{ If $p=2$}:
    \begin{align}
    &c_{2^{r-1}}=0, \label{eq:appendix_lemma_2_r_1}\\
    &\sum_{k=1}^{2^{r-1}-1}k(c_{k}-c_{-k})\equiv 0\pmod{2^{r+1}}, \label{eq:appendix_lemma_2_r_2}\\
    &\sum_{s=0}^{r-2}(c_{k2^{s}+2^{r-1}}+c_{-k2^{s}-2^{r-1}})=-2c_{0}\text{ for all }k=1,\dots,2^{r-1}-1\text{ odd}. \label{eq:appendix_lemma_2_r_3}
    \end{align}
\end{enumerate}
\end{lemma}

In order to prove Lemma \ref{lemma:appendix_equals_one}, we will make use of the following lemma:

\begin{lemma}
\label{lemma:appendix_r_plus}
Let $p^{r}$, and $c_{0},\dots,c_{p^{r}-1}\in\ZZ$ be as in the above lemma. Furthermore, suppose that if $p=2$, then $c{2^{r-1}}=0$. Then $\prod_{k=0}^{2^{r}-1}(1+\omega_{2^{r}}^{k})^{c_{k}}\in\RR_{+}$ if and only if 
\begin{equation}
\label{eq:appendix_lemma_r_plus}
    \sum_{k=1}^{\lfloor\frac{p^{r}-1}{2}\rfloor}k(c_{k}-c_{-k})\equiv 0\pmod{2p^{r}}.
\end{equation}
\end{lemma}

\begin{proof}
Note that
\begin{align*}
    \prod_{k=0}^{p^{r}-1}(1+\omega_{p^{r}}^{k})^{c_{k}}&=2^{c_{0}}\prod_{k=1}^{p^{r-1}-1}\omega_{2p^{r}}^{k(c_{k}-c_{-k})}(\omega_{2p^{r}}^{-k}+\omega_{2p^{r}}^{k})^{c_{k}+c_{-k}} \\
    &=2^{c_{0}}\omega_{2p^{r}}^{\sum_{k=1}^{\frac{p^{r}-1}{2}}k(c_{k}-c_{-k})}\prod_{k=1}^{\frac{p^{r}-1}{2}}(\omega_{2p^{r}}^{-k}+\omega_{2p^{r}}^{k})^{c_{k}+c_{-k}}.
\end{align*}
Now for each $1\le k\le \lfloor\frac{p^{r}-1}{2}\rfloor$, we have that
\[\omega_{2p^{r}}^{-k}+\omega_{2p^{r}}^{k}=2\cos\big(\frac{k\pi}{p^{r}}\big)>0,\]
since
\[0<\frac{k\pi}{p^{r}}<\frac{\pi}{2}\]
for all such $k$. Therefore if $\prod_{k=0}^{p^{r}-1}(1+\omega_{p^{r}}^{k})^{c_{k}}\in\RR_{+}$, we must have that
\[\omega_{2p^{r}}^{\sum_{k=1}^{\frac{p^{r}-1}{2}}k(c_{k}-c_{-k})}=1,\]
or equivalently
\begin{equation*}
    \sum_{k=1}^{\frac{p^{r}-1}{2}}k(c_{k}-c_{-k})\equiv 0\pmod{2p^{r}}.
\end{equation*}
\end{proof}

\begin{proof}[Proof of Lemma \ref{lemma:appendix_equals_one}.]
Recall that the cyclotomic units $C(\omega_{p^{r}})\subset\ZZ[\omega_{p^{r}}]^{\times}$ are defined to be
\[C(\omega_{p^{r}}):=V(\omega_{p^{r}})\cap\ZZ[\omega_{p^{r}}]^{\times},\]
where $V(\omega_{p^{r}})\subset\QQ(\omega_{p^{r}})^{\times}$ is the multiplicative subgroup generated by $\pm\omega_{p^{r}}^{a}$ and expressions of the form $1-\omega_{p^{r}}^{a}$. In (\cite{Wash}, Lemma 8.1) it was shown that the set of units
\[\varepsilon_{p^{r},k}:=\omega_{p^{r}}^{(1-k)/2}\frac{1-\omega_{p^{r}}^{k}}{1-\omega_{p^{r}}}\in\RR,\qquad 2\le k\le\lfloor\frac{p^{r}-1}{2}\rfloor,\qquad(k,p)=1,\]
constitute a set of fundamental units of $C(\omega_{p^{r}})$, i.e., any cyclotomic unit $\alpha\in C(\omega_{p^{r}})$ has a unique presentation of the form
\[\alpha=\pm\omega_{p^{r}}^{a_{0}}\prod_{\substack{k=2 \\ (k,p)=1}}^{\lfloor\frac{p^{r}-1}{2}\rfloor}\varepsilon_{p^{r},k}^{a_{k}}\]
for some $\{a_{k}\}\subset\ZZ$. Equivalently, the $\varepsilon_{p^{r},k}$ form a basis of $C(\omega_{p^{r}})/\text{tors.}\cong\ZZ^{(\phi(p^{r})-3)/2}$ as a $\ZZ$-module, where $\phi$ denotes the Euler totient function. 

For convenience, we define $\varepsilon_{p^{r},1}:=1$ and $\varepsilon_{p^{r},p^{r}-k}:=\varepsilon_{p^{r},k}$ for $2\le k\le\lfloor\frac{p^{r}-1}{2}\rfloor$, $(k,p)=1$. Note that the $\varepsilon_{p^{r}-k}$ are only well-defined up to sign, as fixing a sign requires fixing a choice of square root $\omega_{p^{r}}^{(1-k)/2}$ of $\omega_{p^{r}}^{1-k}$. While it is possible for us to choose a consistent set of signs, we will opt not to do so for the time being. Instead we will first prove the statements up to sign, and then fix the sign using Lemma \ref{lemma:appendix_r_plus}.

First, let $p$ be odd and $0\le s\le r-1$. Then for any $1\le k\le p^{r-s}-1$ with $(k,p)=1$, we have that
\begin{align*}
    1+\omega_{p^{r}}^{kp^{s}}&=\frac{1-\omega_{p^{r}}^{2kp^{s}}}{1-\omega_{p^{r}}^{kp^{s}}}=\frac{\prod_{j=0}^{p^{s}-1}(1-\omega_{p^{r}}^{2k+jp^{r-s}})}{\prod_{j=0}^{p^{s}-1}(1-\omega_{p^{r}}^{k+jp^{r-s}})}=\frac{\prod_{j=0}^{p^{s}-1}(1-\omega_{p^{r}}^{2k+2jp^{r-s}})}{\prod_{j=0}^{p^{s}-1}(1-\omega_{p^{r}}^{k+jp^{r-s}})} \\
    &=\prod_{j=0}^{p^{s}-1}\frac{(1-\omega_{p^{r}}^{2k+2jp^{r-s}})}{(1-\omega_{p^{r}}^{k+jp^{r-s}})}=\prod_{j=0}^{p^{s}-1}\Bigg(\frac{(1-\omega_{p^{r}})}{(1-\omega_{p^{r}}^{k+jp^{r-s}})}\Bigg)\Bigg(\frac{(1-\omega_{p^{r}}^{2k+2jp^{r-s}})}{(1-\omega_{p^{r}})}\Bigg) \\
    &=\pm\prod_{j=0}^{p^{s}-1}\omega_{p^{r}}^{(k+jp^{r-s})/2}\Bigg(\omega_{p^{r}}^{(k+jp^{r-s}-1)/2}\frac{(1-\omega_{p^{r}})}{(1-\omega_{p^{r}}^{k+jp^{r-s}})}\Bigg)\Bigg(\omega_{p^{r}}^{(1-2k+2jp^{r-s})/2}\frac{(1-\omega_{p^{r}}^{2k+2jp^{r-s}})}{(1-\omega_{p^{r}})}\Bigg) \\
    &=\pm\prod_{j=0}^{p^{s}-1}\omega_{p^{r}}^{(k+jp^{r-s})/2}\varepsilon_{p^{r},k+jp^{r-s}}^{-1}\varepsilon_{p^{r},2k+2jp^{r-s}}.
\end{align*}
Hence
\begin{align*}
    \prod_{k=1}^{p^{r}-1}(1+\omega_{p^{r}}^{k})^{c_{k}} &= 2^{c_{0}}\prod_{s=0}^{r-1}\prod_{\substack{k=1 \\ (k,p)=1}}^{p^{r-s}-1}(1+\omega_{p^{r}}^{kp^{s}})^{c_{kp^{s}}} \\
    &=\pm 2^{c_{0}}\prod_{s=0}^{r-1}\prod_{\substack{k=1 \\ (k,p)=1}}^{p^{r-s}-1}\prod_{j=0}^{p^{s}-1}\omega_{p^{r}}^{c_{kp^{s}}(k+jp^{r-s})/2}\varepsilon_{p^{r},k+jp^{r-s}}^{-c_{kp^{s}}}\varepsilon_{p^{r},2k+2jp^{r-s}}^{c_{kp^{s}}} \\
    &=\pm2^{c_{0}}\prod_{\substack{k=1 \\ (k,p)=1}}^{p^{r}-1}\omega_{p^{r}}^{\sum_{s=0}^{r-1}kc_{kp^{s}}/2}\varepsilon_{p^{r},k}^{\sum_{s=0}^{r-1}(c_{kp^{s}/2}-c_{kp^{s}})} \\
    &=\pm 2^{c_{0}}\prod_{\substack{k=1 \\ (k,p)=1}}^{\frac{p^{r}-1}{2}}\Big(\omega_{p^{r}}^{\sum_{s=0}^{r-1}kc_{kp^{s}}/2}\varepsilon_{p^{r},k}^{\sum_{s=0}^{r-1}(c_{kp^{s}/2}-c_{kp^{s}})}\Big)\Big(\omega_{p^{r}}^{-\sum_{s=0}^{r-1}kc_{-kp^{s}}/2}\varepsilon_{p^{r},p^{r}-k}^{\sum_{s=0}^{r-1}(c_{-kp^{s}/2}-c_{-kp^{s}})}\Big) \\
    &=\pm 2^{c_{0}}\prod_{\substack{k=1 \\ (k,p)=1}}^{\frac{p^{r}-1}{2}}\omega_{p^{r}}^{\sum_{s=0}^{r-1}k(c_{kp^{s}}-c_{-kp^{s}})/2}\varepsilon_{p^{r},k}^{\sum_{s=0}^{r-1}((c_{kp^{s}/2}+c_{-kp^{s}/2})-(c_{kp^{s}}+c_{-kp^{s}}))}\\
    &=\pm 2^{c_{0}}\omega_{p^{r}}^{\sum_{\substack{k=1 \\ (k,p)=1}}^{\frac{p^{r}-1}{2}}\sum_{s=0}^{r-1}k(c_{kp^{s}}-c_{-kp^{s}})/2}\prod_{\substack{k=2 \\ (k,p)=1}}^{\frac{p^{r}-1}{2}}\varepsilon_{p^{r},k}^{\sum_{s=0}^{r-1}((c_{kp^{s}/2}+c_{-kp^{s}/2})-(c_{kp^{s}}+c_{-kp^{s}}))}.
\end{align*}
From this, we see that $\prod_{k=1}^{p^{r}-1}(1+\omega_{p^{r}}^{k})^{c_{k}}=\pm 1$ if and only if
\begin{align}
    &c_{0}=0, \label{eq:appendix_p_r_1}\\
    &\sum_{\substack{k=1 \\ (k,p)=1}}^{\frac{p^{r}-1}{2}}\sum_{s=0}^{r-1}k(c_{kp^{s}}-c_{-kp^{s}})\equiv 0\pmod{2p^{r}} \label{eq:appendix_p_r_2}\\
    &\sum_{s=0}^{r-1}((c_{kp^{s}/2}+c_{-kp^{s}/2})-(c_{kp^{s}}+c_{-kp^{s}}))=0\text{ for all } k=2,\dots,\frac{p^{r}-1}{2}\text{ with }(k,p)=1. \label{eq:appendix_p_r_3}
\end{align}
Now for Equation \ref{eq:appendix_p_r_3}, note that
\begin{align*}
    \sum_{\substack{k=1 \\ (k,p)=1}}^{\frac{p^{r}-1}{2}}\sum_{s=0}^{r-1}k(c_{kp^{s}}-c_{-kp^{s}})&\equiv \sum_{s=0}^{r-1}\sum_{\substack{k=1 \\ (k,p)=1}}^{\frac{p^{r-s}-1}{2}}p^{s}k(c_{kp^{s}}-c_{-kp^{s}}) \\
    &\equiv \sum_{k=1}^{\frac{p^{r}-1}{2}}p^{s}\frac{k}{p^{s}}(c_{k}-c_{-k}) \\
    &\equiv \sum_{k=1}^{\frac{p^{r}-1}{2}}k(c_{k}-c_{-k})\pmod{p^{r}}.
\end{align*}
Therefore by Lemma \ref{lemma:appendix_r_plus}, $\prod_{k=0}^{p^{r}-1}(1+\omega_{p^{r}}^{k})^{c_{k}}=1$ if and only if
\begin{align}
    &c_{0}=0, \label{eq:appendix_p_r_5}\\
    &\sum_{k=1}^{\frac{p^{r}-1}{2}}k(c_{k}-c_{-k})\equiv 0\pmod{2p^{r}}, \label{eq:appendix_p_r_6}\\
    &\sum_{s=0}^{r-1}c_{kp^{s}}+c_{-kp^{s}}=\sum_{s=0}^{r-1}c_{kp^{s}/2}+c_{-kp^{s}/2}\text{ for all }k=2,\dots,\frac{p^{r}-1}{2}\text{ with }(k,p)=1, \label{eq:appendix_p_r_7}
\end{align}
as desired.

\bigskip

Now suppose $p=2$. In the case where $r=1$, we have that $\omega_{2}=-1$, and so:
\[\prod_{k=0}^{1}(1+\omega_{2}^{k})^{c_{k}}=(1+\omega_{2}^{0})^{c_{0}}(1+\omega_{2}^{1})^{c_{1}}=2^{c_{0}}\cdot 0^{c_{1}},\]
which is equal to $1$ if and only if $c_{0}=c_{1}=0$.

Now suppose $r\geq 2$. Note that $\prod_{k=1}^{2^{r}-1}(1+\omega_{2^{r}}^{k})^{c_{k}}$ is only non-zero and well-defined if and only if $c_{2^{r-1}}=0$, since
\[(1+\omega_{2^{r}}^{2^{r-1}})^{c_{2^{r-1}}}=(1+(-1))^{c_{2^{r-1}}}=0^{c_{2^{r-1}}},\]
so we will henceforth assume $c_{2^{r-1}}=0$. 

Next, let $0\le s\le r-2$. Then for any $1\le k\le 2^{r-s}-1$ with $(k,2)=1$, we have that
\begin{align*}
    1+\omega_{2^{r}}^{k2^{s}}&=\frac{1-\omega_{2^{r}}^{k2^{s+1}}}{1-\omega_{2^{r}}^{k2^{s}}}=\frac{\prod_{j=0}^{2^{s+1}-1}(1-\omega_{2^{r}}^{k+j2^{r-s-1}})}{\prod_{j=0}^{2^{s}-1}(1-\omega_{p^{r}}^{k+j2^{r-s}})}=\frac{\prod_{j=0}^{2^{s+1}-1}(1-\omega_{2^{r}}^{k+j2^{r-s-1}})}{\prod_{\substack{j=0 \\ j\text{ even}}}^{2^{s+1}-1}(1-\omega_{p^{r}}^{k+j2^{r-s-1}})}\\
    &=\prod_{\substack{j=1 \\ (j,2)=1}}^{2^{s+1}-1}(1-\omega_{2^{r}}^{k+j2^{r-s-1}})=\prod_{\substack{j=0 \\ (j,2)=1}}^{2^{s+1}-1}(1-\omega_{2^{r}})\frac{(1-\omega_{2^{r}}^{k+j2^{r-s-1}})}{(1-\omega_{2^{r}})} \\
    &=\prod_{\substack{j=1 \\ (j,2)=1}}^{2^{s+1}-1}\omega_{2^{r+1}}^{1-2^{r-1}}(\omega_{2^{r+1}}^{2^{r-1}-1}+\omega_{2^{r+1}}^{1-2^{r-1}})\frac{(1-\omega_{2^{r}}^{k+j2^{r-s-1}})}{(1-\omega_{2^{r}})} \\
    &=\prod_{\substack{j=1 \\ (j,2)=1}}^{2^{s+1}-1}\omega_{2^{r}}^{(k+j2^{r-s-1}-1)/2}\omega_{2^{r+1}}^{1-2^{r-1}}(\omega_{2^{r+1}}^{2^{r-1}-1}+\omega_{2^{r+1}}^{1-2^{r-1}})\Bigg(\omega_{2^{r}}^{(1-k-j2^{r-s-1})/2}\frac{(1-\omega_{2^{r}}^{k+j2^{r-s-1}})}{(1-\omega_{2^{r}})}\Bigg) \\
    &=\prod_{\substack{j=1 \\ (j,2)=1}}^{2^{s+1}-1}\omega_{2^{r+1}}^{k+j2^{r-s-1}-2^{r-1}}(\omega_{2^{r+1}}^{2^{r-1}-1}+\omega_{2^{r+1}}^{1-2^{r-1}})\varepsilon_{2^{r},k+j2^{r-s-1}}.
\end{align*}
Hence
\begin{align*}
    \prod_{k=1}^{2^{r}-1}(1+\omega_{2^{r}}^{k})^{c_{k}} &= 2^{c_{0}}\prod_{s=0}^{r-2}\prod_{\substack{k=1 \\ (k,2)=1}}^{2^{r-s}-1}(1+\omega_{2^{r}}^{k2^{s}})^{c_{k2^{s}}} \\
    &=2^{c_{0}}\prod_{s=0}^{r-2}\prod_{\substack{k=1 \\ (k,2)=1}}^{2^{r-s}-1}\prod_{\substack{j=1 \\ (j,2)=1}}^{2^{s+1}-1}\omega_{2^{r+1}}^{c_{k2^{s}}(k+j2^{r-s-1}-2^{r-1})}(\omega_{2^{r+1}}^{2^{r-1}-1}+\omega_{2^{r+1}}^{1-2^{r-1}})^{c_{k2^{s}}}\varepsilon_{2^{r},k+j2^{r-s-1}}^{c_{k2^{s}}}.
\end{align*}
Note that for each $s=0,\dots,r-2$, and every odd integer $1\le k'\le 2^{r}-1$, there exists unique $1\le k\le 2^{r-s}-1$ and $1\le j\le 2^{s+1}-1$ with $(k,2)=(j,2)=1$ such that
\[k'\equiv k+j2^{r-s-1}\pmod{2^{r}}.\]
Furthermore, for each such $k'$ we have that
\[k'2^{s}\equiv (k+j2^{r-s-1})2^{s}\equiv k2^{s}+2^{r-1}\pmod{2^{r}}.\]
For each such $k'$, we have that
\[k'\equiv k+j2^{r-s-1}\pmod{2^{r+1}}\]
if $2^{r-s-1}\le k'\le 2^{r}-1$, and
\[k'\equiv k+j2^{r-s-1}-2^{r}\pmod{2^{r+1}}\]
if $1\le k'\le 2^{r-s-1}-1$. Therefore

\begin{align*}
    2^{c_{0}}\prod_{s=0}^{r-2}&\prod_{\substack{k=1 \\ (k,2)=1}}^{2^{r-s}-1}\prod_{\substack{j=1 \\ (j,2)=1}}^{2^{s+1}-1}\omega_{2^{r+1}}^{c_{k2^{s}}(k+j2^{r-s-1}-2^{r-1})}(\omega_{2^{r+1}}^{2^{r-1}-1}+\omega_{2^{r+1}}^{1-2^{r-1}})^{c_{k2^{s}}}\varepsilon_{2^{r},k+j2^{r-s-1}}^{c_{k2^{s}}} \\
    &=\pm 2^{c_{0}}\prod_{s=0}^{r-2}\prod_{\substack{k=1 \\ (k,2)=1}}^{2^{r}-1}\omega_{2^{r+1}}^{c_{k2^{s}+2^{r-1}}(k-2^{r-1})}(\omega_{2^{r+1}}^{2^{r-1}-1}+\omega_{2^{r+1}}^{1-2^{r-1}})^{c_{k2^{s}+2^{r-1}}}\varepsilon_{2^{r},k}^{c_{k2^{s}+2^{r-1}}} \\
    &=\pm 2^{c_{0}}\prod_{\substack{k=1 \\ (k,2)=1}}^{2^{r}-1}\omega_{2^{r+1}}^{\sum_{s=0}^{r-2}c_{k2^{s}+2^{r-1}}(k-2^{r-1})}(\omega_{2^{r+1}}^{2^{r-1}-1}+\omega_{2^{r+1}}^{1-2^{r-1}})^{\sum_{s=0}^{r-2}c_{k2^{s}+2^{r-1}}}\varepsilon_{2^{r},k}^{\sum_{s=0}^{r-2}c_{k2^{s}+2^{r-1}}} \\
    &=\pm 2^{c_{0}}\prod_{\substack{k=1 \\ (k,2)=1}}^{2^{r-1}-1}\Bigg[\Big(\omega_{2^{r+1}}^{\sum_{s=0}^{r-2}c_{k2^{s}+2^{r-1}}(k-2^{r-1})}(\omega_{2^{r+1}}^{2^{r-1}-1}+\omega_{2^{r+1}}^{1-2^{r-1}})^{\sum_{s=0}^{r-2}c_{k2^{s}+2^{r-1}}}\varepsilon_{2^{r},k}^{\sum_{s=0}^{r-2}c_{k2^{s}+2^{r-1}}}\Big) \\
    &\qquad\cdot\Big(\omega_{2^{r+1}}^{\sum_{s=0}^{r-2}c_{-k2^{s}-2^{r-1}}(-k+2^{r-1})}(\omega_{2^{r+1}}^{2^{r-1}-1}+\omega_{2^{r+1}}^{1-2^{r-1}})^{\sum_{s=0}^{r-2}c_{-k2^{s}-2^{r-1}}}\varepsilon_{2^{r},2^{r}-k}^{\sum_{s=0}^{r-2}c_{-k2^{s}-2^{r-1}}}\Big)\Bigg] \\
    &=\pm 2^{c_{0}}\prod_{\substack{k=1 \\ (k,2)=1}}^{2^{r-1}-1}\Bigg[\omega_{2^{r+1}}^{\sum_{s=0}^{r-2}(k-2^{r-1})(c_{k2^{s}+2^{r-1}}-c_{-k2^{s}-2^{r-1}})} \\
    &\qquad\cdot(\omega_{2^{r+1}}^{2^{r-1}-1}+\omega_{2^{r+1}}^{1-2^{r-1}})^{\sum_{s=0}^{r-2}(c_{k2^{s}+2^{r-1}}+c_{-k2^{s}+2^{r-1}})}\varepsilon_{2^{r},k}^{\sum_{s=0}^{r-2}(c_{k2^{s}+2^{r-1}}+c_{-k2^{s}-2^{r-1}})}\Bigg] \\
    &=\pm 2^{c_{0}}\omega_{2^{r+1}}^{A}(\omega_{2^{r+1}}^{2^{r-1}-1}+\omega_{2^{r+1}}^{1-2^{r-1}})^{B}\prod_{\substack{k=1 \\ (k,2)=1}}^{2^{r-1}-1}\varepsilon_{2^{r},k}^{D_{k}}, \\
\end{align*}
where
\begin{align*}
    &A:=\sum_{\substack{k=1 \\ (k,2)=1}}^{2^{r-1}-1}\sum_{s=0}^{r-2}(k-2^{r-1})(c_{k2^{s}+2^{r-1}}-c_{-k2^{s}-2^{r-1}}) \\
    &B:=\sum_{\substack{k=1 \\ (k,2)=1}}^{2^{r-1}-1}\sum_{s=0}^{r-2}(c_{k2^{s}+2^{r-1}}+c_{-k2^{s}+2^{r-1}}) \\
    &D_{k};=\sum_{s=0}^{r-2}(c_{k2^{s}+2^{r-1}}+c_{-k2^{s}-2^{r-1}})\text{ for all }k=1,\dots,2^{r-1}-1\text{ odd}.
\end{align*}
Therefore $\prod_{k=1}^{2^{r}-1}(1+\omega_{2^{r}}^{k})^{c_{k}}=\pm 1$ if and only if
\begin{equation}
\label{eq:appendix_proof_1}
    \omega_{2^{r+1}}^{A}\prod_{\substack{k=1 \\ (k,2)=1}}^{2^{r-1}-1}\varepsilon_{2^{r},k}^{D_{k}}=\pm 2^{-c_{0}}(\omega_{2^{r+1}}^{2^{r-1}-1}+\omega_{2^{r+1}}^{1-2^{r-1}})^{-B}.
\end{equation}
Note that the left-hand side of Equation \ref{eq:appendix_proof_1} lies in $\QQ(\omega_{2^{r}})^{\times}$. The right-hand side of the above equation lies in $\QQ(\omega_{2^{r}})\subset\QQ(\omega_{2^{r+1}})$ if and only if $B$ is even. Assuming this, we see that the right-hand side is a unit if and only if
\begin{align*}
    \pm 1&=N_{\QQ(\omega_{2^{r}})/\QQ}(\pm 2^{-c_{0}}(\omega_{2^{r+1}}^{2^{r-1}-1}+\omega_{2^{r+1}}^{1-2^{r-1}})^{-B}) \\
    &=\pm N_{\QQ(\omega_{2^{r}})/\QQ}(2)^{-c_{0}}N_{\QQ(\omega_{2^{r}})/\QQ}\big((\omega_{2^{r+1}}^{2^{r-1}-1}+\omega_{2^{r+1}}^{1-2^{r-1}})^{2}\big)^{-B/2} \\
    &=\pm N_{\QQ(\omega_{2^{r}})/\QQ}(2)^{-c_{0}}N_{\QQ(\omega_{2^{r}})/\QQ}\big((\omega_{2^{r+1}}^{2^{r-1}-1})^{2}(1-\omega_{2^{r}})^{2}\big)^{-B/2} \\
    &=\pm N_{\QQ(\omega_{2^{r}})/\QQ}(2)^{-c_{0}}N_{\QQ(\omega_{2^{r}})/\QQ}(\omega_{2^{r}}^{2^{r-1}-1})N_{\QQ(\omega_{2^{r}})/\QQ}\big((1-\omega_{2^{r}})^{2}\big)^{-B/2} \\
    &=\pm N_{\QQ(\omega_{2^{r}})/\QQ}(2)^{-c_{0}}N_{\QQ(\omega_{2^{r}})/\QQ}(1-\omega_{2^{r}})^{-B} \\
    &=\pm 2^{-c_{0}[\QQ(\omega_{2^{r}}):\QQ]}2^{-B} \\
    &=\pm 2^{-2^{r-1}c_{0}-B},
\end{align*}
i.e., if and only if $B=-2^{r-1}c_{0}$. Asssuming this, Equation \ref{eq:appendix_proof_1} can be written as
\[\omega_{2^{r+1}}^{A}\prod_{\substack{k=1 \\ (k,2)=1}}^{2^{r-1}-1}\varepsilon_{2^{r},k}^{D_{k}}=\pm 2^{-c_{0}}(\omega_{2^{r+1}}^{2^{r-1}-1}+\omega_{2^{r+1}}^{1-2^{r-1}})^{2^{r-1}c_{0}}.\]
Using the substitution
\[\omega_{2^{r+1}}^{2^{r-1}-1}+\omega_{2^{r+1}}^{1-2^{r-1}}=\omega_{2^{r+1}}^{2^{r-1}-1}(1-\omega_{2^{r}}),\]
we see that
\begin{align*}
    \omega_{2^{r+1}}^{A}\prod_{\substack{k=1 \\ (k,2)=1}}^{2^{r-1}-1}\varepsilon_{2^{r},k}^{D_{k}}&=2^{-c_{0}}(\omega_{2^{r+1}}^{2^{r-1}-1}+\omega_{2^{r+1}}^{1-2^{r-1}})^{2^{r-1}c_{0}} \\
    &=\pm 2^{-c_{0}}\omega_{2^{r+1}}^{(2^{r-1}-1)2^{r-1}c_{0}}(1-\omega_{2^{r}})^{2^{r-1}c_{0}} \\
    &=\pm 2^{-c_{0}}\omega_{2^{r+1}}^{(2^{r-1}-1)2^{r-1}c_{0}}\prod_{\substack{k=1 \\ (k,2)=1}}^{2^{r}-1}(1-\omega_{2^{r}})^{c_{0}} \\
    &=\pm 2^{-c_{0}}\omega_{2^{r}}^{(2^{r-1}-1)2^{r-2}c_{0}}\prod_{\substack{k=1 \\ (k,2)=1}}^{2^{r}-1}(1-\omega_{2^{r}}^{k})^{c_{0}}\frac{(1-\omega_{2^{r}})^{c_{0}}}{(1-\omega_{2^{r}}^{k})^{c_{0}}} \\
    &=\pm 2^{-c_{0}}\omega_{2^{r}}^{(2^{r-1}-1)2^{r-2}c_{0}}\prod_{\substack{k=1 \\ (k,2)=1}}^{2^{r}-1}\omega_{2^{r}}^{c_{0}(1-k)/2}(1-\omega_{2^{r}}^{k})^{c_{0}}\Big(\omega_{2^{r}}^{(k-1)/2}\frac{(1-\omega_{2^{r}})}{(1-\omega_{2^{r}}^{k})}\Big)^{c_{0}} \\
    &=\pm 2^{-c_{0}}\omega_{2^{r}}^{(2^{r-1}-1)2^{r-2}c_{0}}\prod_{\substack{k=1 \\ (k,2)=1}}^{2^{r}-1}\omega_{2^{r}}^{c_{0}(1-k)/2}(1-\omega_{2^{r}}^{k})^{c_{0}}\varepsilon_{2^{r},k}^{-c_{0}}, \\
    &=\pm 2^{-c_{0}}\Big(\omega_{2^{r}}^{(2^{r-1}-1)2^{r-2}c_{0}}\prod_{\substack{k=1 \\ (k,2)=1}}^{2^{r}-1}\omega_{2^{r}}^{c_{0}(1-k)/2}\Big)\Big(\prod_{\substack{k=1 \\ (k,2)=1}}^{2^{r}-1}(1-\omega_{2^{r}}^{k})\Big)^{c_{0}}\Big(\prod_{\substack{k=1 \\ (k,2)=1}}^{2^{r}-1}\varepsilon_{2^{r},k}^{-c_{0}}\Big) \\
    &=\pm 2^{-c_{0}}\Big(\omega_{2^{r}}^{(2^{r-1}-1)2^{r-2}c_{0}}\omega_{2^{r}}^{2^{r-2}(1-2^{r-1})c_{0}}\Big)\Big(\prod_{\substack{k=1 \\ (k,2)=1}}^{2^{r}-1}(1-\omega_{2^{r}}^{k})\Big)^{c_{0}}\Big(\prod_{\substack{k=1 \\ (k,2)=1}}^{2^{r}-1}\varepsilon_{2^{r},k}^{-c_{0}}\Big) \\
    &=\pm 2^{-c_{0}}(1)(2^{c_{0}})\Big(\prod_{\substack{k=1 \\ (k,2)=1}}^{2^{r-1}-1}\varepsilon_{2^{r},k}^{-2c_{0}}\Big) \\
    &=\pm \prod_{\substack{k=1 \\ (k,2)=1}}^{2^{r-1}-1}\varepsilon_{2^{r},k}^{-2c_{0}},
\end{align*}
and therefore
\[\omega_{2^{r+1}}^{A}\prod_{\substack{k=1 \\ (k,2)=1}}^{2^{r-1}-1}\varepsilon_{2^{r},k}^{D_{k}+2c_{0}}=\pm 1.\]
It follows that $\prod_{k=0}^{2^{r}-1}(1+\omega_{2^{r}}^{k})^{c_{k}}=\pm 1$ if and only if $c_{2^{r-1}}=0$, $A\equiv 0\pmod{2^{r}}$, $B=-2^{r-1}c_{0}$, and $D_{k}=-2c_{0}$ for all $k\geq 3$, or in other words:
\begin{align}
    &c_{2^{r-1}}=0 \label{eq:appendix_2_r_1}\\
    &\sum_{\substack{k=1 \\ (k,2)=1}}^{2^{r-1}-1}\sum_{s=0}^{r-2}(k-2^{r-1})(c_{k2^{s}+2^{r-1}}-c_{-k2^{s}-2^{r-1}})\equiv 0\pmod{2^{r}}, \label{eq:appendix_2_r_2}\\
    &\sum_{\substack{k=1 \\ (k,2)=1}}^{2^{r-1}-1}\sum_{s=0}^{r-2}(c_{k2^{s}+2^{r-1}}+c_{-k2^{s}+2^{r-1}})=-2^{r-1}c_{0} \label{eq:appendix_2_r_3}\\
    &\sum_{s=0}^{r-2}(c_{k2^{s}+2^{r-1}}+c_{-k2^{s}-2^{r-1}})=-2c_{0}\text{ for all }k=3,\dots,2^{r-1}-1\text{ odd}. \label{eq:appendix_2_r_4}
\end{align}
Note that these three equations are equivalent to the following two equations:
\begin{align}
    &\sum_{\substack{k=1 \\ (k,2)=1}}^{2^{r-1}-1}\sum_{s=0}^{r-2}k(c_{k2^{s}+2^{r-1}}-c_{-k2^{s}-2^{r-1}})\equiv 0\pmod{2^{r}}, \label{eq:appendix_2_r_5}\\
    &\sum_{s=0}^{r-2}(c_{k2^{s}+2^{r-1}}+c_{-k2^{s}-2^{r-1}})=-2c_{0}\text{ for all }k=1,\dots,2^{r-1}-1\text{ odd}. \label{eq:appendix_2_r_6}
\end{align}
Next, observe that we can rewrite Equation \ref{eq:appendix_2_r_5} as follows:
\begin{align*}
    0&\equiv\sum_{\substack{k=1 \\ (k,2)=1}}^{2^{r-1}-1}\sum_{s=0}^{r-2}k(c_{k2^{s}+2^{r-1}}-c_{-k2^{s}-2^{r-1}}) \\
    &\equiv\sum_{s=0}^{r-2}\sum_{\substack{k=1 \\ (k,2)=1}}^{2^{r-s-1}-1}2^{s}k(c_{k2^{s}+2^{r-1}}-c_{-k2^{s}-2^{r-1}}) \\
    &\equiv\sum_{k=1}^{2^{r-1}-1}2^{s}\frac{k}{2^{s}}(c_{k+2^{r-1}}-c_{-k-2^{r-1}}) \\
    &\equiv\sum_{k=1}^{2^{r-1}-1}k(c_{k+2^{r-1}}-c_{-k-2^{r-1}}) \\
    &\equiv\sum_{k=2^{r-1}+1}^{2^{r}-1}(k-2^{r-1})(c_{k}-c_{-k}) \\
    &\equiv\sum_{k=1}^{2^{r-1}-1}k(c_{-k}-c_{k})\pmod{2^{r}}.
\end{align*}
Multiplying both sides by $-1$, we obtain
\begin{equation}
\label{eq:appendix_2_r_7}
    \sum_{k=1}^{2^{r-1}-1}k(c_{k}-c_{-k})\equiv 0\pmod{2^{r}}.
\end{equation}
Therefore by Lemma \ref{lemma:appendix_r_plus}, $\prod_{k=0}^{2^{r}-1}(1+\omega_{2^{r}}^{k})^{c_{k}}=1$ if and only if
\begin{align}
    &c_{2^{r-1}}=0 \label{eq:appendix_2_r_8}\\
    &\sum_{k=1}^{2^{r-1}-1}k(c_{k}-c_{-k})\equiv 0\pmod{2^{r+1}} \label{eq:appendix_2_r_9}\\
    &\sum_{s=0}^{r-2}(c_{k2^{s}+2^{r-1}}+c_{-k2^{s}-2^{r-1}})=-2c_{0}\text{ for all }k=1,\dots,2^{r-1}-1\text{ odd}, \label{eq:appendix_2_r_10}
\end{align}
as desired.
\end{proof}

\begin{proposition}
\label{prop:monomials_appendix}
Let $m=p^{r}$ be a prime power, and let $\vec{a},\vec{b}\in\NN^{m}$ with $a_{0},b_{0}\geq 1$. Then $\mbf{w}^{\mbf{a}}=\mbf{w}^{\mbf{b}}\in w_{0}W_{p^{r}}$ if and only if:
\begin{enumerate}
    \item \emph{if $p$ odd:}
    \begin{equation}
        \sum_{k=0}^{p^{r}-1}a_{k}=\sum_{k=0}^{p^{r}-1}b_{k},
    \end{equation}
    and for each $t\in\{0,\dots,r-1\}$, we have that:
    \begin{align}
        &\sum_{\ell=0}^{p^{r-t-1}-1}a_{\ell p^{t+1}}=\sum_{\ell=0}^{p^{r-t-1}-1}b_{\ell p^{t+1}},\\
        \begin{split}
        &\sum_{k=1}^{\frac{p^{t+1}-1}{2}}\sum_{\ell=0}^{p^{r-t-1}-1}k(a_{k+\ell p^{t+1}}-a_{-k-\ell p^{t+1}}) \\
        &\qquad\qquad\equiv\sum_{k=1}^{\frac{p^{t+1}-1}{2}}\sum_{\ell=0}^{p^{r-t-1}-1}k(b_{k+\ell p^{t+1}}-b_{-k-\ell p^{t+1}})\pmod{2p^{t+1}}, \text{ and}
        \end{split} \\
        \begin{split}
        &\sum_{\ell=0}^{p^{r-t-1}-1}\sum_{s=0}^{t}(a_{kp^{s}+p^{t+1}}+a_{-kp^{s}-p^{t+1}}-a_{(kp^{s}+p^{t+1})/2}-a_{(-kp^{s}-p^{t+1})/2}) \\
        &\qquad\qquad=\sum_{\ell=0}^{p^{r-t-1}-1}\sum_{s=0}^{t}(b_{kp^{s}+p^{t+1}}+b_{-kp^{s}-p^{t+1}}-b_{(kp^{s}+p^{t+1})/2}-b_{(-kp^{s}-p^{t+1})/2}) \\
        &\qquad\qquad\qquad\qquad\text{ for all }k=2,\dots,\frac{p^{t+1}-1}{2}\text{ with }(k,p)=1.
        \end{split}
    \end{align}
    Here we use the cyclic indexing convention that if $k$ is odd, then $k/2:=2^{-1}k$ where $2^{-1}\in\ZZ_{p^{r}}^{\times}$ is the unique inverse of $2$ in $\ZZ_{p^{r}}^{\times}$. 
    \item \emph{if $p=2$:}
    \begin{align}
        &\sum_{k=0}^{2^{r}-1}a_{k}=\sum_{k=0}^{2^{r}-1}b_{k}, \\
        &\sum_{k=0}^{2^{r-t-1}-1}a_{(2k+1)2^{t}}=0\iff\sum_{k=0}^{2^{r-t-1}-1}b_{(2k+1)2^{t}}=0\text{ for each }t=0,\dots,r-1,
    \end{align}
    and for each $t\in\{0,\dots,r-1\}$ such that
    \begin{equation}
        \sum_{k=0}^{2^{r-t-1}-1}a_{(2k+1)2^{t}}=\sum_{k=0}^{2^{r-t-1}-1}b_{(2k+1)2^{t}}=0,
    \end{equation}
    we have that:
    \begin{align}
        \begin{split}
        &\sum_{k=1}^{2^{t}-1}\sum_{\ell=0}^{2^{r-t-1}-1}k(a_{k+\ell 2^{t+1}}-a_{-k-\ell 2^{t+1}})\\
        &\qquad\qquad\equiv\sum_{k=1}^{2^{r-1}-1}\sum_{\ell=0}^{2^{r-t-1}-1}k(b_{k+\ell 2^{t+1}}-b_{-k-\ell 2^{t+1}})\pmod{2^{t+2}},\text{ and}
        \end{split} \\
        \begin{split}
        &\sum_{\ell=0}^{2^{r-t-1}-1}2a_{\ell 2^{t+1}}+\Big(\sum_{s=0}^{t-1}a_{k2^{s}+(2\ell+1)2^{t}}+a_{-k2^{s}-(2\ell+1)2^{t}}\Big) \\
        &\qquad\qquad=\sum_{\ell=0}^{2^{r-t-1}-1}2b_{\ell 2^{t+1}}+\Big(\sum_{s=0}^{t-1}b_{k2^{s}+(2\ell+1)2^{t}}+b_{-k2^{s}-(2\ell+1)2^{t}}\Big) \\
        &\qquad\qquad\qquad\qquad\text{ for all }k=1,\dots,2^{t}-1\text{ odd}.
        \end{split}
    \end{align}
\end{enumerate}
\end{proposition}

\begin{proof}
It suffices to consider the case where $\ast=\ev$. Since any element of $R(G^{\ev}_{p^{r}})$ is determined by its character $\chi:G^{\ev}_{p^{r}}\to\CC$, it suffices to look at the traces at all elements $g\in G^{\ev}_{p^{r}}$. 

Let $\mbf{w}^{\mbf{a}}\in w_{0}W_{p^{r}}$. First note that $S^{1}$ acts trivially on the $w_{i}$, hence for any $\phi\in S^{1}$, we have that
\[\tr_{\phi g}(\mbf{w}^{\mbf{a}})=\tr_{g}(\mbf{w}^{\mbf{a}})\]
for all $g\in G^{\ev}_{p^{r}}$. So it suffices to look at the traces at elements lying in the subgroup of $G^{\ev}_{p^{r}}$ generated by $j$ and $\gamma$. Note that
\[\tr_{\gamma^{k}}(w_{0})=\tr_{\gamma^{k}}(1-\wt{c})=1-1=0\]
for all $k$. So by our assumption that $a_{0}\geq 1$, it follows that
\[\tr_{\gamma^{k}}(\mbf{w}^{\mbf{a}})=\tr_{\gamma^{k}}(\mbf{w}^{\mbf{b}})=0\]
for all $k$. Next since $\tr_{j}(w_{k})=2$ for all $k$, we see that
\[\tr_{j}(\mbf{w}^{\mbf{a}})=2^{\sum_{k=0}^{p^{r}-1}a_{k}}.\]
Finally, note that
\[\tr_{j\gamma^{\ell}}(w_{k})=\tr_{j\gamma^{\ell}}(1-\wt{c}\zeta^{k})=1+\omega_{p^{r}}^{k\ell},\]
and so
\[\tr_{j\gamma^{\ell}}(\mbf{w}^{\mbf{a}})=\prod_{k=0}^{p^{r}-1}(1+\omega_{p^{r}}^{k\ell})^{a_{k}}\]
for each $\ell=1,\dots,p^{r}-1$.

Finally, note that for any $0\le t\le r-1$ and for any $\ell,\ell'\in\{1,\dots,p^{r}-1\}$ such that $(\ell,p^{r})=(\ell',p^{r})=p^{t}$, we have that $\prod_{k=0}^{p^{r}-1}(1+\omega_{p^{r}}^{k\ell})^{a_{k}}$ and $\prod_{k=0}^{p^{r}-1}(1+\omega_{p^{r}}^{k\ell'})^{a_{k}}$ are conjugate under the action of $\text{Gal}(\QQ(\omega_{p^{r}}0/\QQ)$. It follows that the traces of $\mbf{w}^{\mbf{a}}$ at $j\gamma^{\ell}$ for all $1\le\ell\le p^{r}-1$ are completely determined by the traces at $j\gamma^{p^{t}}$ for all $t=0,\dots,r-1$.

Furthermore, observe that
\[\prod_{k=0}^{p^{r}-1}(1+\omega_{p^{r}}^{kp^{t}})^{a_{k}}=\prod_{k=0}^{p^{r-t}-1}(1+\omega_{p^{r-t}}^{k})^{\sum_{\ell=0}^{p^{t}-1}a_{k+\ell p^{r-t}}}.\]

It therefore follows that if $\mbf{w}^{\mbf{a}},\mbf{w}^{\mbf{b}}\in W^{*}_{p^{r}}$, then $\mbf{w}^{\mbf{a}}=\mbf{w}^{\mbf{b}}$ in $G^{\ev}_{p^{r}}$ if and only if
\begin{align*}
    &\sum_{k=0}^{p^{r}-1}a_{k}=\sum_{k=0}^{p^{r}-1}b_{k} \\
    &\prod_{k=0}^{p^{r-t}-1}(1+\omega_{p^{r-t}}^{k})^{\sum_{\ell=0}^{p^{t}-1}a_{k+\ell p^{r-t}}}=\prod_{k=0}^{p^{r-t}-1}(1+\omega_{p^{r-t}}^{k})^{\sum_{\ell=0}^{p^{t}-1}b_{k+\ell p^{r-t}}}\text{ for all }t=0,\dots,r-1.
\end{align*}

Now note that
\[\prod_{k=0}^{p^{r}-1}(1+\omega_{p^{r}}^{kp^{r-t-1}})^{a_{k}}=\prod_{k=0}^{p^{t+1}-1}(1+\omega_{p^{t+1}}^{k})^{\sum_{\ell=0}^{p^{r-t-1}-1}a_{k+\ell p^{t+1}}}\neq 0\]
for any odd $p$ and any $t=0,\dots,r-1$, and for $p=2$ we have that
\[\prod_{k=0}^{2^{r}-1}(1+\omega_{2^{r}}^{k2^{r-t-1}})^{a_{k}}=\prod_{k=0}^{2^{t+1}-1}(1+\omega_{2^{t+1}}^{k})^{\sum_{\ell=0}^{2^{r-t-1}-1}a_{k+\ell 2^{t+1}}}=0\]
if and only if
\[\sum_{\ell=0}^{2^{r-t-1}-1}a_{2^{t}+\ell 2^{t+1}}=\sum_{k=0}^{2^{r-t-1}-1}a_{(2k+1)2^{t}}=0,\]
since this is the exponent of the factor $(1+\omega_{2^{t+1}}^{2^{t}})=0$ in the above product. So if we define
\[T(\mbf{w}^{\mbf{a}}):=\twopartdef{\{0,\dots,r-1\}}{p\text{ odd},}{\{t\in\{0,\dots,r-1\}\;|\;\sum_{k=0}^{2^{r-t-1}-1}a_{(2k+1)2^{t}}=0\}}{p=2.}\]
then 
\begin{equation}
\label{eq:appendix_prop_proof_1}
    \prod_{k=0}^{p^{t+1}-1}(1+\omega_{p^{t+1}}^{k})^{\sum_{\ell=0}^{p^{r-t-1}-1}a_{k+\ell p^{t+1}}}=\prod_{k=0}^{p^{t+1}-1}(1+\omega_{p^{t+1}}^{k})^{\sum_{\ell=0}^{p^{r-t-1}-1}b_{k+\ell p^{t+1}}}
\end{equation}
for all $t=0,\dots,r-1$ only if 
\begin{equation}
\label{eq:appendix_proof_t}
    T(\mbf{w}^{\mbf{a}})=T(\mbf{w}^{\mbf{b}})\text{ and }T(\mbf{w}^{\mbf{a}})^{c}=T(\mbf{w}^{\mbf{b}})^{c}.
\end{equation}
(Note that this condition is vacuous if $p$ is odd.) Assuming Equation \ref{eq:appendix_proof_t}, for each $t\in T(\mbf{w}^{\mbf{a}})$ we have that both sides of Equation \ref{eq:appendix_prop_proof_1} are non-zero. Hence for such $t$, Equation \ref{eq:appendix_prop_proof_1} holds if and only if
\[\prod_{k=0}^{p^{t+1}-1}(1+\omega_{p^{t+1}}^{k})^{\sum_{\ell=0}^{p^{r-t-1}-1}(a_{k+\ell p^{t+1}}-b_{k+\ell p^{t+1}})}=1.\]
Now for each $t\in T(\mbf{w}^{\mbf{a}})$, and each $k=0,\dots,p^{t+1}-1$ define
\begin{equation}
\label{appendix_prop_proof_c_tk}
    c_{t,k}:=\sum_{\ell=0}^{p^{r-t-1}-1}(a_{k+\ell p^{t+1}}-b_{k+\ell p^{t+1}}).
\end{equation}
Then by Lemma \ref{lemma:appendix_equals_one}, for each $t\in T(\mbf{w}^{\mbf{a}})=T(\mbf{w}^{\mbf{b}})$ we have that:
\begin{align*}
    &c_{t,0}=0,\\
    &\sum_{k=1}^{\frac{p^{t+1}-1}{2}}k(c_{t,k}-c_{t,-k})\equiv 0\pmod{2p^{t+1}},\\
    &\sum_{s=0}^{t}c_{t,kp^{s}}+c_{t,-kp^{s}}=\sum_{s=0}^{t}c_{t,kp^{s}/2}+c_{t,-kp^{s}/2}\text{ for all }k=2,\dots,\frac{p^{t+1}-1}{2}\text{ with }(k,p)=1
\end{align*}
if $p$ is odd, and:
\begin{align*}
    &\sum_{k=1}^{2^{t}-1}k(c_{t,k}-c_{t,-k})\equiv 0\pmod{2^{t+2}}, \\
    &\sum_{s=0}^{t-1}(c_{t,k2^{s}+2^{t}}+c_{t,-k2^{s}-2^{t}})=-2c_{t,0}\text{ for all }k=1,\dots,2^{t}-1\text{ odd}.
\end{align*}
if $p=2$. Using the substitution given by Equation \ref{appendix_prop_proof_c_tk} gives us the desired result.
\end{proof}

\newpage

%% file: tables.tex
\section{Tables}
\label{sec:tables_appendix}
\begingroup
\renewcommand{\arraystretch}{1.5} 
\begin{center}
\begin{table}[htp]
\begin{tabular}{|c|c|c|c|c|c|}\hline
\label{table:knots}
    $K$ & $\wt{\kappa}=\kappa$ & $\kappa_{\KMT}$ & $\LSWFS_{2}$ & $\SWFS_{2}^{\<j\mu\>}$ & $\LSWFS_{2}^{\<j\mu\>}$ \\ \hline\hline
    $T(3,12n-1)$ & $2$ & $0$ & N & Y & Y \\ \hline
    $k(2,3,12n-1)$ & $2$ & $1$ & N & N & N\\ \hlinewd{1.5pt}
    $\ol{T(3,12n-1)}$ & $0$ & $0$ & N & Y & Y\\ \hline
    $\ol{k(2,3,12n-1)}$ & $0$ & $0$ & N & N & N\\ \hlinewd{1.5pt}
    $T(3,12n-5)$ & $1$ & $-\tfrac{1}{2}$ & N & Y & Y\\ \hline
    $k(2,3,12n-5)$ & $1$ & $\tfrac{1}{2}$ & N & N & N\\ \hlinewd{1.5pt}
    $\ol{T(3,12n-5)}$ & $1$ & $\tfrac{1}{2}$ & N & Y & Y \\ \hline
    $\ol{k(2,3,12n-5)}$ & $1$ & $\tfrac{1}{2}$ & N & N & N\\ \hlinewd{1.5pt}
    $T(3,12n+1)$ & $0$ & $0$ & Y & Y & Y\\ \hline
    $k(2,3,12n+1)$ & $0$ & $0$ & Y & N & Y\\ \hlinewd{1.5pt}
    $\ol{T(3,12n+1)}$ & $0$ & $0$ & Y & Y & Y \\ \hline
    $\ol{k(2,3,12n+1)}$ & $0$ & $0$ & Y & N & Y\\ \hlinewd{1.5pt}
    $T(3,12n+5)$ & $1$ & $\tfrac{1}{2}$ & Y & Y & Y \\ \hline
    $k(2,3,12n+5)$ & $1$ & $\tfrac{1}{2}$ & Y & N & Y\\ \hlinewd{1.5pt}
    $\ol{T(3,12n+5)}$ & $-1$ & $-\tfrac{1}{2}$ & Y & Y & Y \\ \hline
    $\ol{k(2,3,12n+5)}$ & $-1$ & $-\tfrac{1}{2}$ & Y & N & Y\\ \hline
\end{tabular}
\bigskip
\caption{Columns 2 and 3 record the various $\kappa$-invariants associated to the families of torus and pretzel knots appearing in Column 1. Here, the knot $\ol{K}$ denotes the mirror of $K$. Columns 4,5,6 record whether the corresponding family of knots lie in the classes $\LSWFS_{2}$, $\SWFS_{2}^{\<j\mu\>}$, and $\LSWFS_{2}^{\<j\mu\>}$, respectively. (See Section \ref{subsec:knot_invariants} for more information.)} 
\end{table}
\end{center}
\endgroup

\begingroup
\renewcommand{\arraystretch}{1.1} 
\begin{center}
\begin{table}
\begin{tabular}{|c|c|c|c|c|c|c|c|}\hline
\label{table:genus_bounds}
   $X$ & $A$ & $K$ & $G$-Sig. & Man. & KMT & (\ref{eq:genus_bounds}) & U.B. \\ \hline
    \multirow{4}{*}{$\#^{2}(S^{2}\times S^{2})$} & \multirow{4}{*}{((4,4),(4,4))} & $\mathbf{T(3,5)}$ & \textbf{18} & \textbf{21} & \textbf{20} & \textbf{(22)*} & \textbf{22} \\ \cline{3-8}
    & & $T(3,7)$ & 18 & 21 & 21 & 23* & 24 \\ \cline{3-8}
    & & $T(3,11)$ & 22 & 25 & 25 & 27* & 28 \\ \cline{3-8}
    & & $T(3,13)$ & 22 & 27 & 25 & 28* & 30 \\ \hline
    \multirow{8}{*}{$\CC P^{2}\#\CC P^{2}$} & \multirow{4}{*}{(6,2)} & $\mathbf{T(3,5)}$ & \textbf{12} & \textbf{13} & \textbf{13} & \textbf{(14)*} & \textbf{14} \\ \cline{3-8}
    & & $T(3,7)$ & 12 & 13 & 14 & 15* & 16 \\ \cline{3-8}
    & & $T(3,11)$ & 16 & 17 & 18 & 19* & 20 \\ \cline{3-8}
    & & $T(3,13)$ & 16 & 19 & 18 & 20* & 22 \\ \cline{2-8}
    & \multirow{4}{*}{(6,6)} & $\mathbf{T(3,5)}$ & \textbf{20} & \textbf{23} & \textbf{20} & \textbf{(24)*} & \textbf{24} \\ \cline{3-8}
    & & $T(3,7)$ & 20 & 23 & 21 & 25* & 26 \\ \cline{3-8}
    & & $T(3,11)$ & 24 & 27 & 27 & 29* & 30 \\ \cline{3-8}
    & & $T(3,13)$ & 24 & 29 & 27 & 30* & 32 \\ \hline
    \multirow{8}{*}{$S^{2}\times S^{2}\#\CC P^{2}$} & \multirow{4}{*}{((4,4),2)} & $\mathbf{T(3,5)}$ & \textbf{11} & \textbf{12} & \textbf{12} & \textbf{(13)*} & \textbf{13} \\ \cline{3-8}
    & & $T(3,7)$ & 11 & 12 & 13 & 14* & 15 \\ \cline{3-8}
    & & $T(3,11)$ & 15 & 16 & 17 & 18* & 19 \\ \cline{3-8}
    & & $T(3,13)$ & 15 & 18 & 17 & 19* & 21 \\ \cline{2-8}
    & \multirow{4}{*}{((4,4),6)} & $\mathbf{T(3,5)}$ & \textbf{19} & \textbf{22} & \textbf{21} & \textbf{(23)*} & \textbf{23} \\ \cline{3-8}
    & & $T(3,7)$ & 19 & 22 & 22 & 24* & 25 \\ \cline{3-8}
    & & $T(3,11)$ & 23 & 26 & 26 & 28* & 29 \\ \cline{3-8}
    & & $T(3,13)$ & 23 & 28 & 26 & 29* & 31 \\ \hline
    \multirow{4}{*}{$hK3$} & \multirow{4}{*}{0} & $\mathbf{T(3,5)}$ & \textbf{1} & \textbf{3} & \textbf{3} & \textbf{(4)*} & \textbf{4} \\ \cline{3-8}
    & & $T(3,7)$ & 1 & 3 & 4 & 4 & 6 \\ \cline{3-8}
    & & $T(3,11)$ & 5 & 7 & 8 & 8 & 10 \\ \cline{3-8}
    & & $T(3,13)$ & 5 & 9 & 8 & 10* & 12 \\ \hline
\end{tabular}
\bigskip
\caption{This table gives a list of lower and upper bounds for the relaative $(X,A)$-genus of the torus knots $T(3,5)$, $T(3,7)$, $T(3,11)$, and $T(3,13)$, for pairs $(X,A)$ featured in Theorem \ref{theorem:X_A_genus_sharp}. Columns 4-7 list the lower bounds from the $G$-signature theorem (\ref{eq:G_signature}), \cite{Man14}, \cite{KMT}, and our equivariant relative 10/8ths inequality, respectively, and Column 8 gives the upper bound $g(X,A)+g_{4}(K)$ for $g_{X,A}(K)$. Parentheses in Column 7 denote entries which coincide with the upper bound given by Column 8, and asterisks in Column 7 denote entries which give strictly better lower bounds than the bounds from Columns 4,5,6. The rows corresponding to $T(3,5)$ are bold-faced to denote that $T(3,5)$ belongs to the class of knots $\SWFM_{2}^{\#,\C}$ featured in Theorem \ref{theorem:X_A_genus_sharp}.} 
\end{table}
\end{center}
\endgroup